\documentclass[a4paper, 
headsepline=off, DIV=12, titlepage=false, 11pt]{scrartcl}
\setkomafont{author}{\sffamily\scshape}

\title{\vspace{-1cm}\semiLARGE
On Euler systems and Nekov\'{a}\v{r}--Selmer complexes}
\date{}
\author{Dominik Bullach \and David Burns}
 
 \usepackage[backend=bibtex, style=numeric, sorting=nyt, maxbibnames=9, maxalphanames=5, minalphanames=1
 ]{biblatex}
\renewbibmacro{in:}{%
  \ifentrytype{article}{}{\printtext{\bibstring{in}\intitlepunct}}}
  \DeclareFieldFormat{postnote}{#1}%
  \DeclareFieldFormat{journaltitle}{#1\isdot}
  \DeclareFieldFormat[article, inproceedings, unpublished, phdthesis]{title}{\textit{#1}}

\input{preamble}

\begin{document}

\maketitle

\begin{abstract}\vspace{-0.5cm}
    We develop a theory of Euler and Kolyvagin systems relative to the Nekov\'{a}\v{r}--Selmer complexes of $p$-adic representations over local complete Gorenstein rings. This theory is both finer and requires fewer hypotheses than those of Mazur and Rubin \cite{MazurRubin04, MazurRubin2} over discrete valuation rings and of Sakamoto et al.\@ \cite{bss, bss2} over Gorenstein rings. In particular, given appropriate Euler systems, it allows one to study Selmer groups defined relative to Greenberg local conditions. As initial applications, we prove new cases of Kato's generalised Iwasawa main conjecture for both $\ZZ_p(a)$ and the $p$-adic Tate modules of rational elliptic curves, new cases of the Quillen--Lichtenbaum Conjecture, and a  strengthening of existing results on the Birch--Swinnerton-Dyer Conjecture for CM elliptic curves.\end{abstract}

\vspace{-0.1cm}
\tableofcontents

\section{Introduction}

{\sffamily \bfseries Background and main results} The notion of `Euler system' was introduced by Kolyvagin in \cite{kolyvagin90} as a means of axiomatising parallel earlier work of his on the Selmer groups of modular elliptic curves \cite{kolyvagin} and of Thaine on the class groups of real abelian fields \cite{Thaine}. The general theory was then further developed by Rubin \cite{Rubin-euler}, by Kato \cite{Kato-Euler}, by Perrin-Riou \cite{PerrinRiou98} and by Mazur and Rubin \cite{MazurRubin04, MazurRubin2}, and has by now become well-established as an effective method for studying the Selmer groups of $p$-adic Galois representations. In particular, it has led to  spectacular advances concerning the Birch and Swinnerton-Dyer conjecture and other important cases of the `Tamagawa number conjecture' of Bloch and Kato \cite{bloch-kato} concerning motivic $L$-functions  -- for examples, see \cite{Rubin-euler} and \cite{Kato04} and, more recently, \cite{LeiLoefflerZerbes14}, \cite{KingsLoefflerZerbes17}, \cite{LoefflerZerbesSkinner22a}, and \cite{SangiovanniSkinner}.\\   
In this article, we pay special attention to `refined' special value conjectures for motives with extra symmetries. Such conjectures range from very concrete examples such as the `refined conjecture of BSD-type' of Mazur and Tate \cite{mt} for the modular symbols of rational elliptic curves  to the unifying, and highly abstract, `generalised Iwasawa main conjecture' of Kato  \cite{Kato93a, Kato93b} (or, equivalently,  the commutative-coefficient case of the subsequently formulated `equivariant Tamagawa number conjecture' \cite{BurnsFlach01})  that is stated in terms of the Knudsen--Mumford determinants of complexes arising from the $p$-adic realisations of motives. We further recall that, even granting the existence of an appropriate Euler system, evidence for these refined conjectures remains, in general, both rather limited and also often conditional on difficult-to-verify hypotheses such as the vanishing of Iwasawa $\mu$-invariants or order-of-vanishing conditions on $p$-adic $L$-functions.\\
With these things in mind, one of the main aims of the present  article is to render Euler systems directly applicable to the study of refined special value conjectures by systematically incorporating complexes into the foundations of the theory. In the course of doing so, we shall also 
extend the arguments of Kings--Loeffler--Zerbes in \cite[\S\,12]{KingsLoefflerZerbes17} 
that allow one to study Greenberg--Selmer groups if an appropriate Euler system is given.
 This latter aspect is of interest since recently discovered families of Euler systems, ranging from those for Rankin--Selberg convolutions of modular forms \cite{LeiLoefflerZerbes14, KingsLoefflerZerbes17} to those for Asai (twisted tensor) representations of Hilbert modular forms \cite{LeiLoefflerZerbes18}, for $\mathrm{GSp}_4$ \cite{LoefflerZerbesSkinner22a}, for $\mathrm{GSp}_4 \times \mathrm{GL}_2$ \cite{HsuJinSakamoto20} and for $\mathrm{GU}(2, 1)$ \cite{LoefflerZerbesSkinner22b}  satisfy non-trivial Greenberg local conditions and can only be used to bound the associated Greenberg--Selmer  groups if there is a theory encompassing this wider class of local conditions. In this way, the range of examples to which our theory can be applied extends well beyond the classical setting of Euler systems and Selmer groups relative to relaxed local conditions, and this is a theme that we aim to explore further elsewhere. 
\\
Broadly speaking, the main achievement of this article will therefore be the development of a working theory of Euler systems that are valued in (exterior power biduals of) cohomology groups of the Selmer complexes (relative to rather general local conditions)  introduced by Nekov\'{a}\v{r} in \cite{SelmerComplexes} rather than in Galois cohomology groups themselves. 
Unsurprisingly, the idea of incorporating Selmer complexes into the theory of Euler systems originates with Nekov\'a\v{r} himself, with the explicit question of \cite[\S\,0.19.3]{SelmerComplexes} being motivated by the observation that formulating Iwasawa main conjectures in the setting of Selmer complexes, as pioneered in \cite{Kato93a}, can explain the trivial zeros of $p$-adic $L$-functions. There are, however, two key obstacles that need to be overcome in order for this to be properly achieved in the context that we consider.\\
Firstly, previous approaches rely heavily on the assumption that the Kolyvagin systems being considered do not become   trivial  upon reduction to the residual representation. For representations over discrete valuation rings, a conjecture of Kolyvagin suggests that this is indeed a mild restriction for the Kolyvagin systems related (via an Euler system) to $L$-values (cf.\@ \cite{BurungaleCastellaGrossiSkinner23} and the references therein). However, in our more general setting the presence of `trivial zeros mod $p$' will often force residual triviality even for Kolyvagin systems arising from $L$-series, and so we must overcome the associated technical difficulties. At the same time, this possibility of residual  triviality also means that the `core-rank' of Mazur and Rubin of a given Selmer structure can be strictly negative, and in any such case one cannot expect the existence of the `core vertices' that are pivotal to their approach.
To resolve this problem, we combine the Cebotarev density theorem with Artin--Verdier duality  to prove the existence of a weaker, relative, version of core vertex that we refer to as a `relative core vertex'.\\
Secondly, when working over a general Gorenstein ring the relevant Selmer groups may not be free at a relative core vertex or even, under our weaker hypotheses, at a core vertex (should one exist) and this notably complicates the necessary analysis. To deal with this problem, we are therefore forced to keep careful track of certain  `error terms' that occur naturally when one attempts to extend the recent work  of Sakamoto et al.\@ \cite{bss, bss2}, or the related work of Kataoka \cite{Kataoka2} and of Kataoka and Sano \cite{KataokaSano}, to the setting of Nekov\'{a}\v{r}--Selmer complexes, with the ultimate aim of providing `bounds' for these error terms in the limit. \\ 
Whilst the main advances in the general theory that we obtain therefore appear somewhat technical in nature (see, for example, Theorems \ref{new strategy main result}, \ref{kolyvagin derivative thm} and \ref{core vertices and injectivity replacement}), their advantages over previous results extend even beyond the aspect of much wider applicability mentioned earlier. For instance, the results we obtain  are significantly finer since their conclusions directly concern the determinants of Selmer complexes rather than either the Fitting or characteristic ideals of
Selmer groups. This aspect allows us to completely  remove any need to assume  `$\mu=0$'-type hypotheses, and also to weaken any hypotheses related to the existence of `trivial zeroes modulo $p$', both of which have hitherto been key obstacles to obtaining unconditional results on the conjectures formulated by Kato in  \cite{Kato93a, Kato93b}. Moreover, our theory is developed relative to general local complete Gorenstein rings with finite residue fields of characteristic $p$, and so is applicable to the study of Kato's conjectures  in both the `Galois-equivariant case' (instances of which we discuss below) and also relative to deformation rings, as studied, for example, by Fouquet and Wan in \cite{FouquetWan} and by Fouquet in \cite{Fouquet24, fouquet2025}.
\\ 
For all of these reasons, the theory presented here can be combined with known techniques to obtain, even in classical settings, stronger versions of a range of well-known results. \medskip \\
{\sffamily \bfseries Selected arithmetic consequences}
To illustrate the latter point concretely, we shall first apply our techniques to   Kato's Euler system of zeta elements to prove results such as the following. This result is a direct consequence of (the more general) Corollary \ref{etnc result 2} and Remark~\ref{curve evidence} and, before stating it, we recall that Kato's generalised Iwasawa main conjecture asserts an equality of lattices and so is equivalent to the validity of two inclusions.

\begin{thmintro} Let $E$ be a rational elliptic curve and $K$ a finite abelian extension of $\Q$ for which the Hasse--Weil $L$-function $L(E / K, s)$ does not vanish at $s=1$. Fix a prime $p > 3$ such that 
\begin{itemize}[label=$\circ$]
    \item the action of $\gal{\Q^c}{\Q}$ on the $p$-adic Tate module of $E$ contains $\mathrm{SL}_2 (\Z_p)$, and 
    \item either $E$ has potentially good reduction at $p$ or the field 
    $K ( \sqrt{- 2 (A / B)})$ contains no root of unity of order $p$, where $y^2 = x^3  + Ax + B$ is a Weierstra{\ss} equation for $E$ with $A, B \in \Q^\times$.
\end{itemize}
Then the pair $(h^1 (E / K) (1), \Z_p [\gal{K}{\Q}])$ validates
one inclusion in Kato's generalised Iwasawa main conjecture. Further, this case of Kato's conjecture is fully valid if, in addition, $K/\Q$ is a $p$-extension and $E$ has good ordinary reduction at $p$.
\end{thmintro}

We note that this result, in particular, avoids any hypotheses relating to trivial zeros modulo $p$ (such as $p$ being `non-anomalous' in the terminology of Mazur \cite{Mazur72}) that have been used in previous work in this area (see, for example, the discussion in \cite[\S\,6.4.1]{bss2}) and also applies in many situations where $E$ has additive reduction at $p$. For this reason, it even leads to concrete new evidence in support of the precise conjecture of Birch and Swinnerton-Dyer  (see Corollary \ref{bsd evidence}). In addition, in the companion article \cite{BullachHonnor} of Honnor and the first author, the above result plays a pivotal role in the verification of a large part of the conjectures of  Mazur and Tate \cite{mt} on modular symbols mentioned earlier. 

As a further application, we also consider Tate motives over a number field $k$. In this setting, our approach leads to evidence for Kato's conjecture that is conditional only on a conjectural integrality property of the Rubin--Stark Euler system (see Theorems \ref{G_m main result} and \ref{RS implies etnc}). In particular, if $k$ is an imaginary quadratic field, then the latter system is the (integral) Euler system of elliptic units and this allows us to prove the following unconditional result (which follows immediately from Theorem~\ref{etnc imaginary quadratic fields} and Corollary~\ref{etnc imaginary quadratic twist main result}). 

\begin{thmintro} Let $K$ be a finite abelian extension of an imaginary quadratic field $k$.
Then for every prime $p > 3$ and integer $j \geq 1$, Kato's generalised Iwasawa main conjecture is valid for the pair $(h^0(\mathrm{Spec}\,K)(1 - j), \Z_p [\gal{K}{\Q}])$.
\end{thmintro}

This theorem completes earlier partial results of Bley \cite{Bley04, Bley06}, of Johnston-Leung and Kings \cite{JoKi11}, and of Hofer and the first author \cite{BullachHofer}. It thereby also implies a variety of more explicit conjectures for abelian extensions of imaginary quadratic fields, ranging from the `Rubin--Stark conjecture' from \cite{Rub96} to the `refined class number formula' conjecture of Mazur--Rubin \cite{MazurRubin} and Sano \cite{Sano} and several conjectures in classical Galois module theory (see \cite{Burns01} for more details regarding these connections). 

At the same time, the above result allows us to verify the Quillen--Lichtenbaum Conjecture for a new family of explicit examples (see Remark \ref{lichtenbaum cor}) and also answers the open question recorded by Burungale and Tian in \cite[Rem.~2.2]{BurungaleTian26} and so can be expected to have interesting consequences in the Iwasawa theory of CM modular forms (cf.\@ \cite[Rem.~2.7]{BurungaleTian26}). Indeed, as a first application in the latter direction, we shall here combine the above result 
 with the general strategy described by 
Kato in \cite[Ch.\@ I, \S\,3.3]{Kato93b} (see also \cite{bbs}) to obtain the following refinement of the main result of Burungale and Flach \cite{BurungaleFlach} relating to the Birch and Swinnerton-Dyer Conjecture for CM elliptic curves (see Corollary \ref{CM elliptic curves main result} and Remark  \ref{bf remark}). 

\begin{thmintro} Let $E$ be an elliptic curve over a number field $F$ that contains an imaginary quadratic field $k$. Assume that $E$ has complex multiplication by the ring of integers of $k$ and that the extension $F ( E_\mathrm{tors}) / k$ is abelian, and write $B$ for the Weil restriction of $E$ to $k$. Then, for every finite abelian extension $K$ of $k$ with $L (B / K, 1) \neq 0$, and every prime $p > 3$, Kato's generalised Iwasawa main conjecture is valid for the pair $(h^1 (B / K) (1), (\mathrm{End}_k (B) \otimes_\Z \Z_p)[\gal{K}{k}])$.  \end{thmintro}

At this point, it seems worth remarking that, whilst all of the above results relate to  classical rank-one Euler systems, their proofs involve (a refined version  of) the higher-rank theory of Stark systems introduced by Mazur and Rubin \cite{MazurRubin2} and Sano \cite{Sano14} and, as far as we are aware, there is no approach to these results that can avoid these higher--rank aspects.  
For example, the (rank-one) proof strategies used by Rubin \cite{Rubin-euler} and by Mazur--Rubin \cite{MazurRubin04} rely in an essential way on algebraic results such as the structure theorem for Iwasawa modules that are simply not available for the more general coefficient rings that we are forced to consider. \\ 
Finally we remark that, as mentioned earlier, Euler systems related to special values of $L$-series are now known to exist in a variety of other important settings   and in each of these cases our theory can be expected to have applications.\medskip\\
{\sffamily \bfseries Overview of contents}
For clarity of exposition, we have divided this article in two parts. In Part I, we extend  previous work of Mazur and Rubin \cite{Rubin-euler, MazurRubin04, MazurRubin2} and of Sakamoto et al.\@ \cite{bss, bss2} in order to develop a general theory of Euler, Kolyvagin and Stark systems for Nekov\'a\v{r}--Selmer complexes relative to a morphism $\cR \to R$ of local complete Gorenstein rings. In Part~II (comprising \S\,\ref{relaxed Kato section} to \S\,\ref{Gm section}), we  then illustrate this general theory by presenting some concrete arithmetic applications, including each of those discussed above. In fact, a reader who is mainly interested to know how our techniques apply to arithmetic problems may prefer to simply read \S\,\ref{statement of main result section}, for a discussion of key concepts and hypotheses and the statement of Theorem \ref{new strategy main result}, and then directly pass to Part II.  

In a little more detail, then, the main contents of this article are as follows. In \S\,\ref{algebraic preliminaries section} we review the families of rings that arise in the development of our theory and then establish necessary technical results concerning Matlis duals, exterior biduals, Fitting ideals, and resolutions and determinants of complexes. In \S\,\ref{ss, m and c} we review Selmer structures both in the sense of Nekov\'a\v{r}, which we refer to as 
Nekov\'a\v{r}(--Selmer) structures, and of Mazur and Rubin, which we refer to as Mazur--Rubin(--Selmer) structures, and establish some useful relations between the cohomology groups of the Selmer complex of a Nekov\'a\v{r} structure and the Selmer groups of an associated Mazur--Rubin 
structure and its dual. We also study dual Nekov\'a\v{r} structures, describe relations between Selmer complexes that follow from the Artin--Verdier duality theorem, and construct a family of perfect Selmer complexes that plays a vital role in our theory. In \S\,\ref{statement of main result section}, we introduce a notion of Euler systems relative to Nekov\'{a}\v{r} structures, discuss the hypotheses under which our theory can be developed and then state our main technical result (Theorem \ref{new strategy main result}) concerning relations between the values of an Euler system for a given Nekov\'{a}\v{r} structure and the determinant of its associated Selmer complex. The following two sections then  comprise the technical heart of our article as we extend previous arguments relating to Kolyvagin systems for $p$-adic representations over finite self-injective rings. Firstly, in \S\,\ref{koly sys I section}, as a replacement for the `modified Selmer structures' of Mazur and Rubin that play a key role in previous approaches we study the 
Selmer complexes of a natural family of `modified Nekov\'{a}\v{r} structures', and also prove the existence of a 
`Kolyvagin derivative homomorphism' in the setting of Euler and Kolyvagin systems relative to Nekov\'{a}\v{r} structures. Here it is perhaps worth pointing out again that our contribution is new even in the setting of classical rank-one systems, since the weaker hypotheses and finer Selmer groups that we consider force us to use Selmer complexes in the construction of the Kolyvagin derivative (see \S\,\ref{rank one proof}). Then, in 
\S\,\ref{koly 2 section}, we prove the existence in our setting of relative core vertices and investigate the extent to which they can be used to control the value at the unit modulus of a Kolyvagin system relative to a Nekov\'{a}\v{r} structure.  In \S\,\ref{rss section} we combine the main results of \S\,\ref{koly sys I section} and \S\,\ref{koly 2 section} with an algebraic construction of Stark systems and a delicate limit argument in order to prove Theorem \ref{new strategy main result}. Partly with future applications in mind, in an appendix to Part I we also present an axiomatic treatment of the theory of algebraic Stark systems. Then, as a first step towards arithmetic applications, in \S\,\ref{relaxed Kato section} we explain how to apply Theorem \ref{new strategy main result} for certain `modified, relaxed' Nekov\'a\v{r} structures in order to study the general case of Kato's generalised Iwasawa main conjecture. Finally, in \S\,\ref{elliptic section} and \S\,\ref{Gm section} we present some concrete applications of this strategy and, in particular, prove all of the results stated above by respectively using the Euler systems of Kato's zeta elements and of elliptic units. \medskip \\
{\sffamily \bfseries Acknowledgements}
The first author would like to thank Matthew Honnor, David Loeffler, Wendelin Lutz, Gautier Ponsinet, Ryotaro Sakamoto, Takamichi Sano, Alex Torzewski, and Christian Wuthrich for helpful comments and discussions. The first author also gratefully acknowledges the financial support of the Physical Sciences Research Council [EP/W524335/1] through a Doctoral Prize Fellowship at University College London, 
and of the Spanish Ministry of Science and
Innovation through a
Juan de la Cierva Fellowship at Instituto Ciencias Matem\'aticas (ICMAT) as well as through project [PID2022-142024NB-I00] funded by [MCIU/\-AEI/\-10.13039/\-501100011033].

The second author is very grateful to Ryotaro Sakamoto and Takamichi Sano for many fruitful discussions relating to Euler and Kolyvagin systems, and to Masato Kurihara for his help and encouragement. 

\filbreak
\part{The general theory}

In this first part of the article we develop a general theory of Euler and Kolyvagin systems for Nekov\'a\v{r}--Selmer complexes over local complete Gorenstein rings. Arithmetic consequences of this theory will be discussed in the second part of the article (starting in \S\,\ref{relaxed Kato section}).

\section{Algebraic preliminaries} \label{algebraic preliminaries section}

This section establishes some general algebraic results that are essential to the theory developed in later sections.  
 Throughout, we fix a commutative Noetherian ring $R$. For each natural number $n$, we write $\Spec^n (R)$ and $\Spec^{\le n}(R)$ for the sets of prime ideals of $R$ that are of height $n$ and of height at most $n$ respectively.

We endow the linear dual $M^\ast \coloneqq \Hom_R (M, R)$ of an $R$-module $M$ with the structure of an $R$-module by setting $(x \cdot f) (m) \coloneqq x \cdot f(m)$. 

We frequently use (without explicit comment) that inverse limits over systems indexed by the natural numbers are exact on the category of finitely generated $\ZZ_p$-modules.

Throughout the article we also use the following convenient notations: we set $\N_0 \coloneqq \N\cup\{0\}$ and, for $n\in \N$, write $[n]$ for the ordered set $\{i \in \N: i\le n\}$. 

\subsection{Hypotheses on rings} 

\subsubsection{$G_n$-rings}

We first introduce the categories of rings for which most of our theory will be developed. 

\begin{definition} Let $R$ be a commutative Noetherian ring.
Then, for each $n \in \N_0$, one says   
\begin{itemize}[label=$\circ$]
\item $R$ has property $(G_n)$ if $R_\p$ is Gorenstein for all $\p \in \mathrm{Spec}^{\leq n}(R)$; 
\item $R$ has property $(S_n)$ if $\mathrm{depth}(R_\p) \geq \mathrm{inf} (n, \mathrm{ht} (\p) )$ for all $\p\in \mathrm{Spec}(R)$; 
\index{Serre's condition $(S_n)$}
\item $R$ is a `$G_n$-ring' if it has both of the properties $(G_{n - 1})$ and $(S_n)$. 
\end{itemize}
\end{definition}

\begin{rk} \label{ring-conditions-rk}
\begin{romanliste}
\item Condition (S$_n$) is often referred to as `Serre's condition'.  
\item $G_n$-rings were first studied by Ischebeck \cite{Ischebeck69} and by Reiten and Fossum \cite{ReitenFossum72}. $G_2$-rings are also studied by Vasconcelos \cite{Vasconcelos68,Vasconcelos70} who refers to them as `quasi-normal rings'. 
\item A ring is Cohen--Macaulay, respectively Gorenstein, if and only if it satisfies $(S_n)$, respectively $(G_n)$, for all $n \in \N_0$. Since every Gorenstein ring is Cohen--Macaulay, this also shows that Gorenstein rings satisfy both $(S_n)$ and $(G_n)$ for all $n \in \N_0$. 
 \end{romanliste}\end{rk}

Our first result shows that the class of $G_n$-rings is closed under a variety of natural operations. In particular, it implies  (equivariant) Iwasawa algebras and group rings of finite abelian groups with coefficients in $\Z$ or $\Z_p$ are $G_n$-rings (in fact, Gorenstein rings).      

\begin{lem} \label{G_n ring examples}
If $R$ is a $G_n$-ring for some $n \in \N_0$, then the following claims are valid:
\begin{romanliste}
\item The polynomial ring $R[t]$ for an indeterminate $t$ is a $G_n$-ring.
\item The ring of formal power series $R \llbracket t \rrbracket$ is a $G_n$-ring.
\item The group ring $R[ G ]$ for a finite abelian group $G$ is a $G_n$-ring.
\end{romanliste}
\end{lem}

\begin{proof}
Claim (i) is proved in \cite[Cor.\@ (a) after Prop.\@ 1]{ReitenFossum72} and (ii) is  \cite[Cor.\@ after Prop.\@ 3]{ReitenFossum72}. To prove (iii) we apply \cite[Prop.\@ 1\,(ii)]{ReitenFossum72}: to see that the fibres of $R \to R[G]$ are $n$-Gorenstein it suffices to note that, for any field $F$, the ring $F[G]$ is Gorenstein  (this follows from \cite[Cor.\@ 8]{EilenbergNakayama} which shows that $\mathrm{inj dim}_{F [G]} (F [G]) =\mathrm{inj dim}_{F} (F) = 0$).  
\end{proof}

We recall that an $R$-module $M$ is said to be `pseudo-null'\index{pseudo-null} if $M_\p$ vanishes for all $\p\in \Spec^{\leq 1}(R)$. The following observations about this notion will be useful. 

\begin{lem}[Sakamoto] \label{ryotaro-useful-lem}
Fix an $S_2$-ring $R$ and an $R$-module $M$. 
\begin{romanliste}
\item If $M$ is  pseudo-null, then $M^\ast$ and $\Ext^1_R (M, R)$ both vanish.
\item If $R$ is also a $G_0$-ring, and $N_1$ and $N_2$ are finitely generated submodules of $M$ with $(N_1)_\p \subseteq (N_2)_\p$ for all $\p \in \Spec^{\leq 1}(R)$, then the natural map  $(N_1)^{\ast \ast} \to (N_2)^{\ast \ast}$ is injective. 
\end{romanliste}
\end{lem}

\begin{proof}
Claim (i) is proved in \cite[Lem.\@ B.11]{Sakamoto20} and (ii) follows from the proof of \cite[Lem.\@ C.13]{Sakamoto20}. Since (ii) is however slightly more general than the statement given in loc.\@ cit., we reproduce the argument here. By assumption, the module $(N_1 + N_2) / N_2$ is pseudo-null and so, by (i), the natural map $(N_1 + N_2)^\ast \to N_2^\ast$ is bijective. Hence, by \cite[Lem.\@ C.1]{Sakamoto20} (with $r = 1$), the natural map  $(N_1)^{\ast \ast} \to (N_1 + N_2)^{\ast \ast} \cong (N_2)^{\ast \ast}$ is injective, as claimed. 
\end{proof}

\subsubsection{Complete Gorenstein rings}\label{cgr section} 

In this section we fix a prime number $p$ and Noetherian ring $R$ that satisfy the following condition.
\begin{equation} \label{ring condition}
    \text{$R$ is a local complete Gorenstein ring with finite residue field of characteristic $p$.}
\end{equation}
We write $d$ for the dimension of $R$, $\m$ for its maximal ideal and $\mathbb{k}$ for its residue field $R/\m$. We also fix a decreasing filtration 
\begin{equation}\label{filtration}
\a_0 \coloneqq \m \supseteq 
\a_1 \supseteq \a_2 \supseteq \dots \end{equation}
of ideals of $R$ such that, for every $n$, the ring $R_n \coloneqq R / \a_n$ is Gorenstein of dimension zero (or equivalently, finite and self-injective) and the natural map $R \to {\varprojlim}_n (R / \a_n)$ is bijective. 

\begin{remark}\label{ryotaro example} The condition (\ref{ring condition}) ensures $R$ is Cohen--Macaulay and so one can fix a regular sequence $\{x_i\}_{i\in [d]}$ in $\m$ with $\sqrt{(x_1, \dots, x_d)} = \m$. One then obtains a filtration (\ref{filtration}) of the required sort by setting $\a_n \coloneqq (x_1^n, \dots, x_d^n)$ for all $n \geq 1$. Indeed, $\{x_i^n\}_{i \in [d]}$ is a regular sequence, $\sqrt{\a_n} = \m$, the quotient $R/\a_n$ is Gorenstein (cf.\@ \cite[Prop.\@ 3.1.19\,(b)]{BrunsHerzog}) of dimension zero and the map $R\to {\varprojlim}_n R/\a_n$ is bijective since $R$ is complete, Noetherian, and local. Further, this filtration is universal in the sense that if $\{ \a'_n \}_{n \in \N}$ is any other filtration as in (\ref{filtration}), then for every $n$ there exists $m \in \N$ with $\a_m \subseteq \a'_n$. (This is true since $R / \a'_n$ being zero-dimensional implies $R / \p$ is zero-dimensional for any associated prime $\p \subseteq R$ of $R / \a'_n$ and so $R /\p$ is field; it follows that $\m$ is the only associated prime of $R / \a'_n$ and so $\a'_n$ is $\m$-primary, as required.) The following special instances of the filtration (\ref{filtration}) will also be of importance to us.  

\begin{romanliste}
\item $R$ is finite, local and self-injective: for every $n \geq 1$ one can set $\mathfrak{a}_n = (0)$ so $R_n = R$. 
\item $R$ is a $\ZZ_p$-order that is local and Gorenstein: one can set $\a_n = (p^n)$ for $n \geq 1$, so $R_n = R/p^n R$. 
\item $d > 1$ and $R$ is the completed group ring $\ZZ_p\llbracket G \rrbracket$ of an abelian pro-$p$ compact $p$-adic Lie group $G$ of rank $d-1$. Fix a base $\{U_n\}_{n \ge 1}$ of open neighbourhoods of the identity of $G$ with $U_{n+1}\subset U_n$ for all $n$. Then, for each $n \geq 1$, one can take $\a_n$ to be the kernel of the canonical (surjective) projection map from $R=\ZZ_p\llbracket G \rrbracket$ to $R_n \coloneqq (\ZZ_p/p^n \Z_p)[G/U_n]$. 
\end{romanliste}
\end{remark}

We fix an injective envelope $E_R (\mathbb{k})$ of the $R$-module $\mathbb{k}$ (as defined in \cite[Def.\@ 3.2.3]{BrunsHerzog}) and define the  `Matlis dual' of an $R$-module $M$ by setting 
\[ M^\vee \coloneqq \Hom_R (M, E_R (\mathbb{k}))\] 
(regarded as an $R$-module in the obvious way). We recall that $E_R(\mathbb{k})$, and hence also $M^\vee$ for any given module $M$, is unique up to isomorphism. In addition, the injectivity $E_R (\mathbb{k})$ implies that the assignment $M\mapsto M^\vee$ induces a contravariant exact functor on $R$-modules and, if $M$ is either a finitely generated or Artinian $R$-module, then Matlis duality asserts that the natural map $M \to M^{\vee \vee}$ is bijective (cf.\@ \cite[Th.\@ 3.2.13]{BrunsHerzog}).
 
 Finally, we record a useful description of $E_R(\mathbb{k})$ in terms of the filtration (\ref{filtration}) (this result is essentially well-known but, for lack of a good reference, we include a proof).

\begin{lem} \label{inj env lem}
    The following claims are valid.
    \begin{romanliste}
    \item For every $n$ one can take $E_{R_n} (\mathbb{k}) = R_n$ and also fix an isomorphism of $R_n$-modules $R_n \cong \Hom_\Lambda ( R_n, E_R(\mathbb{k})) \cong E_R(\mathbb{k})[\a_n]$. 
     \item The isomorphisms in (i) induce an identification $T_n^\vee \cong T_n^\ast$ and also an injective homomorphism $R_n \cong E_R(\mathbb{k})[\a_n]\subseteq E_R(\mathbb{k})[\a_{n+1}]\cong R_{n+1}$ of $R$-modules. The module $E_R(\mathbb{k})$ is isomorphic to the inductive limit ${\varinjlim}_{n \in \N} R_n$ with respect to the latter morphisms.  
     \item Let $M$ be a finitely generated $R_{n + 1}$-module. Then any choice of injection $R_n  \hookrightarrow R_{n + 1}$ as in (ii) induces an isomorphism $(M [\a_n])^\ast \cong M^\ast / \a_n M^\ast$. 
    \end{romanliste} 
\end{lem}

\begin{proof}
     At the outset we recall that, for any ring $\cR$, an extension $N \subseteq M$ of $\cR$-modules is called `essential' if $U \cap N \neq (0)$ for every non-zero $\cR$-submodule $U \subseteq M$. By \cite[Prop.\@ 3.2.2]{BrunsHerzog} one can characterise injective $\cR$-modules $I$ as those $\cR$-modules that have no essential extensions $I \subseteq M$ with $I \not= M$, and $E_\cR (N)$ is then defined to be a maximal essential extension of $N$.
     
    The first assertion of (i) is proved in \cite[Th.\@ 3.2.10]{BrunsHerzog} and to prove the rest of (i) it is enough to construct an isomorphism of $R$-modules $E_{R_n} (\mathbb{k}) \cong  E_R(\mathbb{k})[\a_n]$. To prove this, we note that $\mathbb{k} \subseteq E_\Lambda (\mathbb{k}) [\a_n]$ is an essential extension of $R_n$-modules (since 
     $\mathbb{k} \subseteq E_R (\mathbb{k})$ is an essential extension of $R$-modules) and so $E_R(\mathbb{k}) [\a_n]$ identifies with a submodule of $E_{R_n} (\mathbb{k})$. In particular, $E_R(\mathbb{k}) [\a_n]$ is finite and hence it is enough to show $|E_R(\mathbb{k}) [\a_n]| \ge |E_{R_n} (\mathbb{k})|$. But this is clear since $\mathbb{k} \subseteq E_{R_n} (\mathbb{k})$ is an essential extension of $R$-modules so that $E_{R_n} (\mathbb{k})$ identifies with a submodule of $E_{R} (\mathbb{k})[\a_n]$. This proves (i). \ 
     
The first isomorphism in (ii) is obtained as the composite map
\begin{align*}
    T_n^\vee  = \Hom_R ( T_n, E_R (\mathbb{k})) =&\, \Hom_R ( T_n \otimes_{R} R_n, E_R (\mathbb{k})) \\
    \cong&\, \Hom_{R_n} ( T_n, \Hom_R (R_n, E_R (\mathbb{k}))) \cong \Hom_{R_n} ( T_n, R_n) = T_n^\ast,
\end{align*}
where the first isomorphism is induced by Tensor-Hom adjunction and the second  by the isomorphism in (i). The second assertion of (ii) is clear and to prove the final assertion we note that ${\varinjlim}_{n \in \N} R_n$ is an essential extension of $\mathbb{k}$. Indeed, let $x \in {\varinjlim}_{n \in \N} R_n$, then $x \in R_n$ for some $n$ and since $R_n$ is an essential extension of $\mathbb{k}$, it follows that $(R_n x) \cap \mathbb{k} \neq 0$, as required. Now let $E_R(\mathbb{k})$ be a maximal essential extension of $\mathbb{k}$ that contains ${\varinjlim}_{n \in \N} R_n$. Then any element 
$m$ of $E_R (\mathbb{k})$ is annihilated by some ideal $\a_n$: indeed, it is enough to show $\Ann_R (m)$ is $\m$-primary, and this follows from the inclusion of sets of associated primes $\mathrm{Ass}_R ( R / \Ann_R (m)) \subseteq \mathrm{Ass}_R ( E_R (\mathbb{k}) = \{ \m \}$, where the last equality is by \cite[Lem.\@ 3.2.7\,(a)]{BrunsHerzog}. Now, one has $E_R (\mathbb{k}) [\a_n] = R_n$ by (i) and so we deduce that $m$ belongs to ${\varinjlim}_{n \in \N} R_n$. It follows that ${\varinjlim}_{n \in \N} R_n$ is a maximal essential extension of $\mathbb{k}$, as required to prove (ii). \\
To prove (iii), we may fix an injective presentation $0 \to M \to R_{n + 1}^{\oplus l} \xrightarrow{\cdot A} R_{n + 1}^{\oplus m}$ for suitable integers $l, m \geq 0$ and an $(m \times l)$-matrix $A = (a_{ij})_{ij}$ because $R_{n + 1}$ is injective as an $R_{n+1}$-module. Since taking $\a_n$-torsion defines a left-exact functor, dualising leads to an exact sequence
\begin{cdiagram}
    ( R_{n + 1} [\a_n])^{\oplus m} \arrow{r}{\cdot A^\mathrm{t}} & (R_{n + 1} [\a_n])^{\oplus l} \arrow{r} & (M [\a_n])^\ast \arrow{r} & 0,
\end{cdiagram}%
where $A^\mathrm{t} = (a_{ji})_{ij}$ is the transpose of $A$. Write $\sigma \: R_n \hookrightarrow R_{n + 1}$ for the fixed embedding, which restricts to an isomorphism $\hat\sigma \: R_n \cong R_{n +1} [\a_n]$
because $R_n$ and $R_{n +1} [\a_n] \cong \Hom_{R_{n + 1}} ( R_n, R_{n + 1})$ have the same length as $R_{n + 1}$-modules. The $R_{n + 1}$-linearity of $\sigma$ then implies that the diagram
\begin{cdiagram}[column sep=small, row sep=small]
    ( R_{n + 1} [\a_n])^{\oplus m} \arrow{r}{\cdot A^\mathrm{t}} & (R_{n + 1} [\a_n])^{\oplus l} \arrow{d}[left]{(\hat\sigma^{-1})^{\oplus l}}[right]{\simeq}\\ 
    R_n^{\oplus m} \arrow{r}{\cdot A^\mathrm{t}} \arrow{u}[left]{\sigma^{\oplus m}}[right]{\simeq} & R_n^{\oplus l} 
\end{cdiagram}%
is commutative. 
Since the cokernel of $R_n^{\oplus m} \xrightarrow{\cdot A^\mathrm{t}} R_n^{\oplus l}$ is canonically isomorphic to $M^\ast / \a_n M^\ast$, we deduce that the map  $(\hat\sigma^{-1})^{ \oplus l}$ induces an isomorphism $(M [\a_n])^\ast \xrightarrow{\simeq} M^\ast / \a_n M^\ast$, as required. 
\end{proof}

In connection with Lemma \ref{inj env lem} the following general observation will be useful.

\begin{lem} \label{ryotaro's reduction trick}
Let $R$ be a Noetherian local ring with maximal ideal $\cM$. For every non-zero element $x$ in an $R$-module $F$ there is an element $r \in R$ such that $r \cdot x$ is both non-zero and belongs to $F [\cM] \coloneqq \{ y \in F \mid a \cdot m = 0 \text{ for all } a \in \cM \}$. 
\end{lem}

\begin{proof}
Since $R$ is Artinian, the descending chain of $R$-modules $R x \supseteq \cM \cdot x \supseteq \dots \supseteq \cM^l \cdot x \supseteq \dots $ becomes stationary and so there exists $i \in \N_0$ such that $\cM^l \cdot x = \cM^{l + 1} \cdot x$. By Nakayama's lemma, this implies $\cM^l \cdot x = 0$.
Let $l$ be the smallest non-negative integer with this property. Then, since $x \neq 0$ (by assumption), $l \geq 1$ and, by the minimality of $l$, $\cM^{l - 1} \cdot x$ is non-zero and contained in $F [\cM]$. We may thus take $r$ to be any element of $\cM^{l - 1} \setminus \cM^l$.
\end{proof}

\begin{remark}\label{switch} Fix an element $r$ as in Lemma \ref{ryotaro's reduction trick} with $F = R$ and $x=1$. If $R$ is a zero-dimensional Gorenstein local ring, the ideal $R[\mathcal{M}]$ is a one-dimensional $\mathbb{k}$-vector space by Lemma \ref{inj env lem}\,(i) and so, in particular, principal. The assignment $1\mapsto r$ therefore induces an isomorphism of $R$-modules $\mathbb{k} = R/\mathcal{M} \cong R[\mathcal{M}]$. 
\end{remark}

For a general result on base change over inverse limit rings, see Lemma \ref{Tor lemma} below.

\subsection{Fitting ideals} 
In this section we fix a commutative unital ring $\cR$ and a finitely-presented $\cR$-module $M$. We then fix (as we may) a resolution of $M$ of the form
\[
\cR^{\oplus m} \xrightarrow{\varphi} \cR^{\oplus n} \to M \to 0
\]
in which $m \ge n$. For $i\in \N_0$, the $i$-th `Fitting ideal' $\Fitt^i_\cR (M)$ of $M$ is then defined to be the ideal of $\cR$ that is generated by the set of all $(n - i) \times (n - i)$-minors of any matrix representing $\varphi$ (where we use the convention that the determinant of an empty matrix is $1$). 

It is easily checked that the above definition is independent of the chosen resolution, that $\Fitt^{i}_\cR (M)\subseteq \Fitt^{j}_\cR (M)$ if $i \le j$ and that $\Fitt^{i}_\cR (M) = \cR$ if $ i \ge n$ (cf.\@ \cite[\S\,3.1, Th.~1, Th.~2]{Northcott}).  
In the following result we record several further properties of these ideals, most of which are well-known, that we shall rely on throughout this article. (A more detailed discussion, and further properties, of Fitting ideals can be found in  \cite[\S\,3.1]{Northcott}, \cite[App.]{MazurWiles} and \cite[\S\,20.2]{eisenbud-comm-algebra}.)

\begin{lem}\label{standard fitting props}
    The following claims are valid for all finitely-presented $\cR$-modules $M$ and $N$ and all $i$ and $j$ in $\N_0$.
    \begin{itemize}
         \item[(i)] If $f \: M \to N$ is a surjective map of $\cR$-modules, then $\Fitt^i_\cR (M) \subseteq \Fitt^i_\cR (N)$.
        \item[(ii)] If $0 \to N \to M \to M / N \to 0$ is a short exact sequence of $\cR$-modules, then one has $\Fitt^i_\cR (N) \cdot \Fitt^j_\cR (M / N) \subseteq \Fitt^{i + j}_\cR (M)$. 
        \item[(iii)] One has $\Fitt^i_\cR (M \oplus N) = \sum_{a = 0}^{a=i} \Fitt^a_\cR (M) \cdot \Fitt^{i-a}_\cR (N)$. In particular, if $N$ is a non-zero free $\cR$-module of rank $r$, then $\Fitt^i_\cR (M) = \Fitt_\cR^{i+r}(M\oplus N)$. 
        \item[(iv)] If $f \: \cR \to \cR'$ is a morphism of rings, then $ M \otimes_\cR \cR'$ is a finitely-presented $\cR'$-module and $\Fitt^i_{\cR'} ( M \otimes_\cR \cR')$ is equal to the ideal of $\cR'$   generated by $f(\Fitt^i_\cR (M))$. 
        \item[(v)] (Buchsbaum--Eisenbud) If $j < i$, then $\Fitt^j_\cR (M)$ annihilates $\exprod^i_\cR M$. In particular, $M$ is itself annihilated by $\Fitt^0_\cR (M)$.  
    \end{itemize}
\end{lem}

\begin{proof} We fix a free presentation $\cR^{\oplus m} \xrightarrow{\varphi} \cR^{\oplus n} \xrightarrow{\pi} M \to 0$ and note that we have an exact sequence $0 \to \ker \pi \to \ker (f \circ \pi) \to \ker f \to 0$. We can therefore construct a free presentation $\cR^{\oplus (m + m')} \xrightarrow{\varphi'} \cR^{\oplus n} \xrightarrow{f \circ \pi} N \to 0$ in which $\varphi'$ restricts to $\varphi$ on $\cR^{\oplus m}$. This implies (i) and the proof of (ii) is similar (see \cite[\S\,3.1, Ex.~2, solution on p.\@ 90]{Northcott} for the details). The first equality in (iii) is proved in \cite[\S\,3.1, Ex.~3, solution on p.\@ 92]{Northcott} and the second assertion follows as a consequence upon replacing $i$ by $i+r$ in the first equality and noting that $\Fitt_\cR^j(N)$ is equal to $(0)$ if $j < r$ and to $\cR$ if $j \ge r$. Claim (iv) follows directly from the definitions upon noting $(-) \otimes_\cR \cR'$ is a right-exact functor and hence that tensoring a free presentation of $M$ with $\cR'$ gives a free presentation of $M \otimes_\cR \cR'$.  

The second assertion of (v) is classical and proved in \cite[\S\,3.1, Th.~5]{Northcott}. It is also an immediate consequence of the first assertion with $i = 1$ and $j = 0$. To prove the first assertion of (v), we note that the general result of \cite[Cor.\@ 1.3]{BuchsbaumEisenbud77} proves $\Fitt^j_\cR (M)$ annihilates $\exprod^{j + 1}_\cR M$. (Note the different convention for Fitting ideals: in loc.\@ cit.\@ the authors write $F_k (M)$ for $\Fitt^{k - 1}_\cR (M)$ in our notation.) The claimed result then follows from the existence of a surjective map $(\exprod^{j + 1}_\cR M) \otimes_\cR (\exprod^{i - j - 1}_\cR M) \twoheadrightarrow \exprod^i_\cR M$. Since the cited result is perhaps not so well-known, for the convenience of the reader we include its proof. Fix a free presentation $\cR^{\oplus m} \stackrel{\varphi}{\to} \cR^{\oplus n} \to M \to 0$ of $M$ with $n \geq j + 1$. For every $s \geq 0$ we define the map
\[
\varphi^{(s)} \: \big( \exprod^s_\cR \cR^{\oplus m} \big) \otimes_\cR \exprod^j_\cR \cR^{\oplus n} \to \exprod^{j + s}_\cR \cR^{\oplus n}, \quad 
(x_1 \wedge \dots \wedge x_s) \otimes y \mapsto  \varphi (x_1) \wedge \dots \wedge \varphi (x_s) \wedge y.
\]
The starting point is then the exact sequence
\begin{cdiagram}
    \cR^{\oplus m} \otimes_\cR \exprod^j_\cR \cR^{\oplus n} \arrow{r}{\varphi^{(1)}} & \exprod^{j + 1}_\cR \cR^{\oplus n} \arrow{r} & \exprod^{i + 1}_\cR M \arrow{r} & 0,
\end{cdiagram}%
which reduces us to proving that, for every $\lambda \in \Fitt^j_\cR (M)$ and $b = b_1 \wedge \dots \wedge b_{s + 1} \in \exprod^{j + 1}_\cR \cR^{\oplus n} $, one has $\lambda b \in \im(\varphi^{(1)})$.\\
Let $e_1, \dots, e_{n - j - 1} \in \cR^{\oplus n}$ be basis vectors such that $\{ b_1, \dots, b_{j + 1}, e_1, \dots, e_{n - (j + 1)} \}$ is a linearly independent set. Then, setting $e \coloneqq e_1 \wedge \dots \wedge e_{n - j - 1}$, one has $e^\ast (e \wedge b) = b$. Now, any choice of isomorphism 
$\exprod^{n}_\cR \cR^{\oplus n} \cong \cR $ induces an identification
\[
\im \Big (
\big( \exprod^{n - j}_\cR \cR^{\oplus m} \big) \otimes_\cR \exprod^{j}_\cR \cR^{\oplus n} \xrightarrow{\varphi^{(n - j)}} \exprod^{n}_\cR \cR^{\oplus n} \cong \cR \Big) \cong  \Fitt^j_\cR (M)
\]
and so, in particular, $\Fitt^j_\cR (M)$ annihilates $\coker(\varphi^{(n - j)})$. That is, $\lambda \cdot (e \wedge b)$ can be written in the form $\sum_{l\in [k]} \varphi ( x_{l1}) \wedge \dots \wedge \varphi (x_{l (n - j)}) \wedge y_l$ for suitable $x_{lp} \in \cR^{\oplus m}$ and $y_i \in \exprod^j_\cR \cR^{\oplus n}$ with $p \in [k]$ (for some $k$) and $l \in [n - j]$. It follows that
\begin{align*}
\lambda b & = e^\ast ( \lambda \cdot (e \wedge b)) =  e^\ast \big ( {\sum}_{l \in [k]} \varphi ( x_{l1}) \wedge \dots \wedge \varphi (x_{l (n - j)}) \wedge y \big)\\
& = {\sum}_{l \in [k]}{\sum}_{p\in [n-j]} \pm \varphi (x_{lp}) \wedge e^\ast \big( \varphi (x_{l1} \wedge \dots \wedge \widehat{ \varphi ( x_{lp})} \wedge \dots \wedge \varphi ( x_{lj}) \wedge y_l \big)
\end{align*}
belongs to $\im(\varphi^{(1)})$, as required.\end{proof}

Finally, we establish a technical result about inverse limits of Fitting ideals that we will need in later arguments. We note that claim (i) of this result extends the result \cite[Th.\@ 2.1]{GreitherKurihara2008} of Greither and Kurihara and that claim (ii) develops an idea of Popescu and Yin \cite{popescu2024fittingidealsprojectivelimits}.

\begin{prop} \label{Greither--Kurihara Lemma}
Let $(\cR_i, \rho_i)_{i \in \N}$ be a projective system of rings, and set $\cR \coloneqq {\varprojlim}_{i \in \N} \cR_i$. Let $(M_i, f_i)_{i\in \N}$ be a projective system of $\cR_i$-modules with surjective transition maps $f_i$, and write $M$ for the associated $\cR$-module ${\varprojlim}_{i \in \N} M_i$. Assume that $M$ is finitely presented and that, in addition, either of the following conditions are satisfied: 
\begin{romanliste}
    \item Each ring $\cR_i$ (resp.\@ module $M_i$) is a compact Hausdorff space and the transition maps $\rho_i$ (resp.\@ $f_i$) are continuous, and there  
     exists a natural number $s$ such that, for all $i \in \N$, the kernel of  $
    M \otimes_\cR \cR_i \to M_i$ is generated by $s$ elements over $\cR_i$. 
    \item $\cR$ satisfies condition (\ref{ring condition}), and $\cR_i \coloneqq \cR / \a_i$ as in \S\,\ref{cgr section}. 
\end{romanliste}
Then, for every integer $r \geq 0$, one has $\Fitt^r_\cR (M) = 
{\varprojlim}_{i \in \N} \Fitt^r_{\cR_i} (M_i)$, where the limit is taken with respect to the maps $\Fitt^r_{\cR_{i + 1}} (M_{i + 1}) \to \Fitt^r_{\cR_i} (M_i)$ induced by the restriction of $\rho_i$. 
\end{prop}

\begin{proof} We first assume that the conditions in (i) are satisfied. 
Since each map $f_i \: M_{i + 1} \to M_i$ induces a surjection $M_{i + 1} \otimes_{\cR_{i + 1}} \cR_i \to M_i$, applying parts (iv) and (i) of Lemma \ref{standard fitting props} implies that there is an inclusion 
\[
\rho_i ( \Fitt^r_{\cR_{i + 1}} ( M_{i + 1})) = \Fitt^r_{\cR_i} (M_{i + 1} \otimes_{\cR_{i + 1}} \cR_i) \subseteq \Fitt^r_{\cR_i} (M_i).
\]
In a similar way, since the second assertion of (i) implies that the natural map $M \otimes_\cR \cR_i \to M_i$ is surjective, one has  $\theta_i ( \Fitt^r_{\cR} ( M)) = \Fitt^r_{\cR_i} ( M \otimes_\cR \cR_i) \subseteq 
\Fitt^r_{\cR_i} (M_i)$, with $\theta_i$ the natural projection map $\cR \to \cR_i$. By taking the limit over $i$, we therefore obtain an inclusion $\Fitt^r_{\cR} ( M) \subseteq {\varprojlim}_{i \in \N} \Fitt^r_{\cR_i} (M_i)$, and so we must prove the reverse inclusion.

To do this, we use the fact $M$ is finitely presented to fix an exact sequence of $\cR$-modules 
\begin{cdiagram}
    \cR^{t_1} \arrow{r} & \cR^{t_2} \arrow{r} & M \arrow{r} & 0
\end{cdiagram}%
for suitable natural numbers $t_1$ and $t_2$. Now, each of the composite maps $\cR^{t_2} \to M \to M_i$, where the second arrow is the natural projection, factors through $\cR^{t_2} \to \cR^{t_2}_i$.
The assumption that the transition maps $f_j \: M_{j + 1} \to M_j$ are surjective ($j \geq i$) implies that $M \to M_i$ is surjective, hence we obtain a surjective map $\cR^{t_2}_i \to M_i$. 
Setting $K_i \coloneqq \ker ( \cR^{t_2}_i \to M_i)$, we obtain an exact commutative diagram 
\begin{equation} \label{diagram bounded number of generators}
\begin{tikzcd}[row sep=small]
    \cR_i^{t_1} \arrow{d} \arrow{r} & \cR^{t_2}_i \arrow[equals]{d} \arrow[twoheadrightarrow]{r} & M \otimes_\cR \cR_i \arrow{d}  \\
     K_i \arrow[hookrightarrow]{r} & \cR^{t_2}_i\arrow[twoheadrightarrow]{r} & M_i.
\end{tikzcd}%
\end{equation}

The Snake Lemma, applied to the above diagram, then combines with condition (iii) to imply that $K_i$ can be generated by at most $t_1 + s$ elements. \\
Since $M_i$ is Hausdorff, $K_i$ is a closed subspace of a compact Hausdorff space, hence is itself compact Hausdorff.  
The associated first derived limit therefore vanishes and so, passing to the limit of the bottom sequences in (\ref{diagram bounded number of generators}), we obtain the exact sequence
\begin{cdiagram}
    0 \arrow{r} & K \coloneqq {\varprojlim}_{i \in \N} K_i \arrow{r} & \cR^{t_2} \arrow{r} & M \arrow{r} & 0.
\end{cdiagram}%
Using this sequence, $\Fitt^r_\cR (M)$ can then be computed as follows: 
For every subset $J \subseteq [t_2]$ of cardinality $r$, we now write $\pi_J \: \cR^{t_2} \to \cR^r$ for the projection $(x_1, \dots, x_{t_2}) \mapsto (x_j)_{j \in J}$. For each $v = (v_1, \dots, v_r) \in K^r$, we then write $\pi_J (v)$ for the $(r \times r)$-matrix with columns $\pi_J (v_1), \dots, \pi_J (v_r)$. Then $\Fitt^r_\cR (M)$ is the ideal of $\cR$ generated by all $\det ( \pi_J (v))$ with $v$ ranging over $K^r$ and $J$ ranging over all subsets of $[t_2]$ of cardinality $r$. \\
We now give a similar description of $\Fitt_{\cR_i} (M_i)$ for every $i \in \N$. We write $\pi_{J, i} \: \cR_i^{t_2} \to \cR_i^r$ for the projection onto coordinates in $J$, and, for every $v = (v_1, \dots, v_r) \in K_i^r$, write $\pi_{J, i} (v)$ for the matrix with columns $\pi_{J, i} (v_1), \dots, \pi_{J, i} (v_r)$. 
By definition, $\Fitt^r_{\cR_i} (M_i)$ is then the ideal of $\cR_i$ generated by all determinants of such matrices $\pi_{J, i} (v)$.\\
 Since $K_i$ can be generated by $t_2 + s$ elements and the determinant is multilinear, any matrix of the form above can be written as a sum of $N \coloneqq (t_2 + s)^{t_2 - r}$ determinants with columns from $\pi_J (K_i)$. 
If we set $\cW_i \coloneqq \bigoplus_{J} (K_i^{r})^{N}$, then we therefore have that $\Fitt^r_{\cR_i} (M_i)$ is equal to the image of the map
\[\Phi_i \: 
\cW_i = {\bigoplus}_{J} (K_i^{r})^{N} \to \cR_i, \quad 
(v_{J, j})_{J, j} \mapsto {\sum}_{J} {\sum}_{j \in [N]} \det (\pi_J (v_{J, j})),
\]
where $J$ runs over all subsets of $\{1, \dots, t_2\}$ of cardinality $t_2 - r$. \\
We now claim that ${\varprojlim}_{i \in \N} (\im \Phi_i) 
$ is equal to the image of the map $\Phi \: {\varprojlim}_{i \in \N} \cW_i \to \cR$ induced by $(\Phi_i)_{i \in \N}$. To show this, we suppose to be given an element $x = (x_i)_{i \in \N} $ of ${\varprojlim}_{i \in \N} (\im \Phi_i)$.
For every $i \in \N$, we write $U_i \coloneqq \Phi_i^{-1} (x_i)$ for the full preimage of $x_i$ under $\Phi_i$ and note that each of the transition maps $\tilde \rho_{i, j} \: \cW_j \to \cW_i$ takes $U_j$ to $U_i$ for $j \geq i$. Since $\cW_j$ is finitely generated over the finite ring $\cR_j$, it is again compact Hausdorff. As the preimage of a closed set under a continuous map, each  $U_j$ is then a closed subset of a compact Hausdorff space, hence itself compact Hausdorff. Now, $\tilde \rho_{i, j}$ is continuous and so $\tilde \rho_{i, j} ( U_j)$ is a compact subset of $\cW_i$, which implies that it must be closed. In addition, the ascending chain given by the $\tilde \rho_{i, j} ( U_j)$ has the finite intersection property because each preimage $U_j$ is non-empty. 
It follows that the intersection $
V_i \coloneqq {\bigcap}_{j \geq i} \tilde \rho_{i, j} ( U_j)$ 
is non-empty because $\cW_i$ is compact. We now inductively construct a preimage $(w_i)_{i \in \N} \in {\varprojlim}_{i \in \N} \cW_i$ of $x$ with each $w_i$ in $V_i$. For the induction base, we take $w_1$ to be any element of $V_1$. Now fix $i \geq 1$ and suppose that $w_i$ is already constructed. We then define $U_j' \coloneqq U_j \cap \tilde \rho_{i, j}^{-1} (w_i)$ for all $j > 1$ and note that each $U_j'$ is non-empty because $w_i$ belongs to $V_i$. Similar to before, it follows that the descending chain $(\tilde \rho_{i + 1, j} (U'_j))_{j \geq i + 1}$ has the finite intersection property, and hence that $
V'_{i + 1} \coloneqq {\bigcap}_{j \geq i + 1} \tilde \rho_{i + 1, j} (U'_j)$ is non-empty. Any element $w_{i + 1}$ of $V'_{i + 1}$ will then both belong to $V_i$ and have the property that $\tilde \rho_{i + 1, i} (w_{i + 1}) = w_i$, as required. This concludes the induction step. \\ 
We have thereby proved that any element $x$ of ${\varprojlim}_{i \in \N} (\im \Phi_i) = {\varprojlim}_{i \in \N} \Fitt^r_{\cR_i} (M_i)$ can be lifted to en element of
${\varprojlim}_{ i\in \N} \cW_i = {\varprojlim}_{ i\in \N} (\bigoplus_J (K_i^r)^N) = \bigoplus_J (K^r)^N$. It follows that $x$ is a finite sum of determinants of matrices of the form $\pi_J (v)$ with $v \in K^r$ and hence, by the explicit description given above, belongs to $\Fitt^r_\cR (M)$. This proves the claimed equality under condition (i). \\
In the remainder of the argument, we therefore assume the validity of condition (ii) and we argue by induction on $d  = \dim \cR$. Since each projection  map $M \to M_n$ is surjective, each $\cR_n$-module $M_n$ is finite and hence compact Hausdorff. The above argument therefore combines with Remark \ref{greither kurihara lemma rk} below to prove the claim if $d \leq 1$. In the following we thus assume $d > 1$. \\
  Following Remark \ref{ryotaro example}, we obtain a filtration (\ref{filtration}) of $\cR$ by setting $\fb_n = (x_1^n, \dots, x_d^n)$, and we now first justify that we may assume $\a_n = \fb_n$. 
  To do this we recall, as noted in Remark~\ref{ryotaro example}, one can define $f \: \N \to \N$ such that $\fb_{f (n)} \subseteq \a_n$ for all $n \in \N$. In particular, each $M_n$ is a $\cR / \fb_{f (n)}$-module and it suffices to prove that ${\varprojlim}_{n \in \N} \Fitt^r_{\cR / \fb_{f (n)}} (M_n) = {\varprojlim}_{n \in \N} \Fitt^r_{\cR_n} (M_n) $. Since the surjection $\cR / \fb_{f (n)} \twoheadrightarrow \cR_n = \cR / \a_n$ maps $\Fitt^r_{\cR / \fb_{f (n)}} (M_n)$ onto $\Fitt^r_{\cR_n} (M_n)$ by Lemma~\ref{standard fitting props}\,(iv), we have an exact sequence
  \begin{equation} \label{change of filtration}
  \begin{tikzcd}
      0 \arrow{r} & \a_n / \fb_{f (n)} \arrow{r} & \Fitt^r_{\cR / \fb_{f (n)}} (M_n) + (\a_n / \fb_{f (n)} ) \arrow{r} & \Fitt^r_{\cR_n} (M_n) \arrow{r} & 0,
  \end{tikzcd}%
  \end{equation}
  in which all involved modules are finite. Now, taking the limit of the surjections $\a_n \twoheadrightarrow \a_n / \fb_{f (n)}$ gives a surjection ${\varprojlim}_{n \in \N} \a_n \twoheadrightarrow {\varprojlim}_{n \in \N} ( \a_n / \fb_{f (n)})$, and this shows the vanishing of ${\varprojlim}_{n \in \N} ( \a_n / \fb_{f (n)})$ because  ${\varprojlim}_{n \in \N} \a_n = \bigcap_{n \in \N} \a_n$ vanishes by the assumption that the natural map $\cR \to {\varprojlim}_{n \in \N} \cR / \a_n$ is bijective. In a similar fashion, one shows that the natural map ${\varprojlim}_{n \in \N} \Fitt^r_{\cR / \fb_{f (n)}} (M_n) \to {\varprojlim}_{n \in \N} ( \Fitt^r_{\cR / \fb_{f (n)}} (M_n) + (\a_n / \fb_{f (n)} ))$ is an isomorphism. Passing to the limit over the exact sequences (\ref{change of filtration}) now shows that the natural map ${\varprojlim}_{n \in \N} \Fitt^r_{\cR / \fb_{f (n)}} (M_n) \to {\varprojlim}_{n \in \N} \Fitt^r_{\cR_n} (M_n)$ is bijective, as claimed.\\
  In the remainder of this argument we therefore assume that $\a_n = \fb_n = (x_1^n, \dots, x_d^n)$, and we also set $\a'_n \coloneqq (x_1^n, \dots, x_{d - 1}^n)$. Now, one has
    \begin{align*}
    \Fitt^r_\cR (M) & = {\varprojlim}_{n \in \N} \Fitt^r_{\cR / (x_d^n)} ( M \otimes_\cR ( \cR / (x_d^n))) \\
    & = {\varprojlim}_{n \in \N} {\varprojlim}_{m \in \N} \Fitt^r_{\cR / (\a'_m, x_d^n)} ( M_m \otimes_\cR ( \cR / (x_d^n))).
    \end{align*}
    Here the first equality follows from the above argument (under condition (i)) and the second equality holds by our induction hypothesis. Since $M_m \cong {\varprojlim}_{n \in \N} (M_m \otimes_\cR( \cR / (x_d^n)))$ for all $m \in \N$, we may use the induction hypothesis again to deduce that
    \[
    {\varprojlim}_{n \in \N} {\varprojlim}_{m \in \N} \Fitt^r_{\cR / (\a'_m, x_d^n)} ( M_m \otimes_\cR ( \cR / (x_d^n))) = {\varprojlim}_{m \in \N} \Fitt^r_{\cR / \a'_m} ( M_m).
    \]
    Now, the image of $\Fitt^r_{\cR / \a'_m} ( M_m)$ in $\cR_m$ is equal to $\Fitt^r_{\cR_m} (M_m)$ by Lemma \ref{standard fitting props}\,(iv), so we have an exact sequence
    \begin{equation} \label{exact sequence to compute limit of Fitt}
    \begin{tikzcd}
        0 \arrow{r} & \a_m / \a'_m \arrow{r} & \Fitt^r_{\cR / \a'_m} ( M_m) + (\a_m / \a'_m)  \arrow{r} & \Fitt^r_{\cR_m} (M_m) \arrow{r} & 0.
   \end{tikzcd}%
   \end{equation}
    Moreover, $\a_m / \a'_m$ is generated by $x_d^m$ and so there exists a surjective map $x_d^m \cR \twoheadrightarrow \a_m / \a'_m$. Taking the limit over $m$ then gives, because the involved modules are finitely generated $\cR$-modules and hence compact Hausdorff, a surjection 
    ${\varprojlim}_{m \in \N} (x^m_d \cR) \twoheadrightarrow {\varprojlim}_{m \in \N } ( \a_m / \a_m')$.
    In addition, the Krull intersection theorem implies that the limit ${\varprojlim}_{m \in \N} (x^m_d \cR) = \bigcap_{m \in \N} (x^m_d \cR)$ vanishes, and so the same is also true for ${\varprojlim}_{m \in \N } ( \a_m / \a_m')$. A similar argument shows that ${\varprojlim}_{m \in \N} \big( ( \Fitt^r_{\cR / \a'_m} ( M_m) + (\a_m / \a'_m)) / \Fitt^r_{\cR / \a'_m} ( M_m) \big)$ vanishes. 
    Taking the limit (over $m$) of the exact sequence (\ref{exact sequence to compute limit of Fitt}) (in which all involved modules are compact Hausdorff) shows that
    \[
    {\varprojlim}_{m \in \N} \Fitt^r_{\cR_m} (M_m) = {\varprojlim}_{m \in \N} \big( \Fitt^r_{\cR / \a'_m} ( M_m) + (\a_m / \a'_m) \big) = {\varprojlim}_{m \in \N}  \Fitt^r_{\cR / \a'_m} ( M_m),
    \]
    as required to complete the proof. 
\end{proof}

\begin{rk} \label{greither kurihara lemma rk}
    The second condition in Proposition \ref{Greither--Kurihara Lemma}\,(i) is automatically valid if $\cR$ is local Noetherian and $\Ann_\cR (M)$ contains an $\cR$-regular sequence of length $\dim (\cR) - 1$. (In particular, the condition is satisfied if either $\dim \cR \leq 1$ or if both $\dim (\cR) = 2$ and $M$ is $\cR$-torsion.)    To justify this, we note that the diagram (\ref{diagram bounded number of generators}) implies it is enough to prove the minimal number of $\cR_i$-generators of $K_i$ is bounded independently of $i$. If $\{x_i\}_{i \in [d-1]}$ is an $\cR$-regular sequence in $\Ann_\cR (M)$, then $(x_1, \dots, x_{d - 1}) \cdot \cR^{t_2}_i\subseteq K_i$. It therefore suffices to bound the minimal number of $\cR_i$-generators of $K_i / (x_1, \dots, x_{d - 1}) \cR^{t_2}_i$. The latter is now a submodule of $\cR^{t_2}_i / (x_1, \dots, x_{d - 1}) \cR^{t_2}_i$, and hence a subquotient of $(\cR / (x_1, \dots, x_{d - 1}))^{t_2}$. Since $\{x_i\}_{i \in [d-1]}$ is an $\cR$-regular sequence, $\cR / (x_1, \dots, x_{d - 1})$ is a local Noetherian ring of Krull dimension one and so the claim follows from the fact that the minimal number of generators of any ideal of such a ring can be bounded by a constant that only depends on the ring (see, for example, \cite{shalev}).
\end{rk}

\subsection{Exterior power biduals}

For any $r \in \N_0$, we write $\exprod^r_R M$ for the $r$-fold exterior power of an $R$-module $M$. The `$r$-th exterior power bidual' of $M$ is then defined by setting 
\[ \bidual^r_R M \coloneqq \left ( \exprod^r_R M^\ast \right )^\ast.\] 
This construction is motivated by the approach of Rubin in \cite[\S\,1.2]{Rub96} and is extensively studied in the literature (cf.\@ \cite[App.\@ A]{sbA},  \cite[App.\@ B]{Sakamoto20}). In this section, we establish several properties of these modules that will be useful in subsequent arguments. 

\subsubsection{Rank reduction} \label{rank reduction section}

If $s\in \N_0$ is such that $s \le r$, and $f$ is an element of $\exprod^s_R M^\ast$, then the assignment $a \mapsto [g \mapsto a ( f \wedge g )]$ induces a `rank reduction' map 
$\bidual^r_R M \to \bidual^{r - s}_R M$ which, by abuse of notation, we  continue to denote by $f$. 

\begin{lem} \label{biduals lemma 1}
Let $R$ be a $G_2$-ring and suppose to be given, for some $ s \in \mathbb{N}$, an exact sequence
\begin{cdiagram}
0 \arrow{r} & N \arrow{r} & M \arrow{rr}{(f_i)_{i\in [s]}} & & R^{\oplus s}
\end{cdiagram}
of finitely generated $R$-modules. Then the following claims are valid. 
\begin{itemize}
\item[(i)] For any $r\in \N$, there exists an exact sequence
\begin{equation} \label{ryotaro exact sequence}
\begin{tikzcd}
0 \arrow{r} & \bidual^r_R N \arrow{r} & \bidual^r_R M \arrow{rr}{(f_i)_{i\in [s]}} & & {\bigoplus}_{i\in [s]} \bidual^{r - 1}_R M.
\end{tikzcd}%
\end{equation}%
Here the first map is induced by the inclusion $N \to M$, and the second is the diagonal map induced by the maps $f_i \: \bidual^r_R M  \to \bidual^{r - 1}_R M$ for each $i \in [s]$. 
\item[(ii)] If $s \leq r$, then the map $\wedge_{i \in [s]} f_i \: \bidual^r_R M \to \bidual^{r - s}_R M$ factors through the map $\bidual^{r - s}_R N \hookrightarrow \bidual^{r - s}_R M$ in (\ref{ryotaro exact sequence}). In particular, there exists an induced map
\begin{equation} \label{rank reduction map}
{\wedge}_{i \in [s]} f_i \: \bidual^r_R M \to \bidual^{r - s}_R N.
\end{equation}
\end{itemize}
\end{lem}

\begin{proof}
Claim (i) requires a slight variation of the proof of \cite[Lem.\@ B.12]{Sakamoto20} and so, for clarification, we shall provide full details. Thus, by dualising the tautological exact sequence $0 \to N \to M \to M / N \to 0$, and setting $Y \coloneqq \im \{ M^\ast \to N^\ast \}$, one obtains an  exact sequence
\begin{equation} \label{Y exact sequence}
\begin{tikzcd}
 0 \arrow{r} & (M / N)^\ast \arrow{r} & M^\ast \arrow{r} & Y \arrow{r} & 0.
\end{tikzcd}%
\end{equation}%
This in turn  implies the existence of an exact sequence
\begin{cdiagram}
 (M / N)^\ast \otimes_R \exprod^{r - 1}_R M^\ast \arrow{r} & 
 \exprod^r_R M^\ast \arrow{r} & \exprod^r_R Y \arrow{r} & 0
\end{cdiagram}%
(see \cite[Lem.\@ 2.5]{bss} for details) and, upon dualising again, this gives an exact sequence 
\begin{equation} \label{another intermediate step exact sequence}
\begin{tikzcd}
 0 \arrow{r} & \big( \exprod^r_R Y \big)^\ast \arrow{r} & \bidual^r_R M \arrow{r} & \big ( (M / N)^\ast \otimes_R \exprod^{r - 1}_R M^\ast \big)^\ast.
\end{tikzcd}%
\end{equation}%
We now first claim that $\big( \exprod^r_R Y \big)^\ast \cong \bidual^r_R N$. To this end, we set $Z \coloneqq \coker ( M \to R^{\oplus s})$
and note that the cokernel of the inclusion $Y \subseteq N^\ast$ identifies with a submodule of $\Ext^1_R ( M / N, R) \cong \Ext^2_R (Z, R)$, and so is pseudo-null since for all $\p\in \rm{Spec}^1(R)$ the localisation $R_\p$ has injective dimension at most one by condition $(G_1)$ on $R$. It follows that both kernel and cokernel of the natural map
$\exprod^r_R Y \to \exprod^r_R N^\ast$ are pseudo-null, which, by Lemma \ref{ryotaro-useful-lem}\,(i), implies that the dual map $\bidual^r_R N \to \big( \exprod^r_R Y \big)^\ast$ is an isomorphism, as claimed.
\\
We now set $C \coloneqq \coker \bigl( (R^{ s})^\ast \to (M / N)^\ast \bigr)$ and observe that 
applying the right-exact functor $( -) \otimes_R \exprod^{r - 1}_R M^\ast$ to the exact sequence $ (R^{ s})^\ast \to (M/N)^\ast \to C \to 0$ gives an exact sequence
\begin{equation} \label{intermediate step exact sequence}
\begin{tikzcd}
(R^{ s})^\ast \otimes_R \exprod^{r - 1}_R M^\ast 
\arrow{r} & ( M / N)^\ast \otimes_R \exprod^{r - 1}_R M^\ast  \arrow{r} & C \otimes_R \exprod^{r - 1}_R M^\ast  \arrow{r} & 0.
\end{tikzcd}%
\end{equation}%
To investigate the last term in this sequence, we note that $C$ naturally identifies with a submodule of $\Ext^1_R ( Z, R)$. It follows from condition ($G_0$) on $R$ that the latter, and hence also $C$, is a torsion $R$-module. By dualising (\ref{intermediate step exact sequence}) we therefore obtain an injection 
\[
\big ( ( M / N)^\ast \otimes_R \exprod^{r - 1}_R M^\ast \big)^\ast \hookrightarrow 
\big ( (R^{ s})^\ast \otimes_R \exprod^{r - 1}_R M^\ast  \big)^\ast.
\]
Now, $(R^{s})^\ast$ is a free $R$-module and so tensoring with $(R^{s})^\ast$ commutes with taking duals. This combines with the isomorphism $(R^{s})^{\ast \ast} \cong R$ to give an isomorphism
\[
\big ( (R^{s})^\ast \otimes_R \exprod^{r - 1}_R M^\ast  \big)^\ast 
\cong R^{s} \otimes_R \bidual^{r - 1}_R M. 
\]
By composing the last two displayed maps, we may then deduce the exact sequence (\ref{ryotaro exact sequence}) claimed in (i) from the exact sequence (\ref{another intermediate step exact sequence}) upon recalling the isomorphism $\big( \exprod^r_R Y \big)^\ast \cong \bidual^r_R N$.\\
By the exactness of (\ref{ryotaro exact sequence}) it suffices to prove, for each $i \in [s]$, that the composite homomorphism  $f_i \circ ( \exprod_{j \in [s]} f_j)$ vanishes   in order to verify (ii). For this, it is  enough to note that, for each $a\in \bidual^r_R M$, the element $
(f_i \circ ( {\wedge}_{j \in [s]} f_j)) ( a) 
=  \big ( ( {\wedge}_{j \in [s]} f_j) \wedge f_i \big) (a) 
$ vanishes since $f_i \wedge f_i = 0$. 
\end{proof}

\begin{rk}
It is known that $R$ is a $G_1$-ring if and only if, for each finitely generated $R$-module $M$, the kernel of the canonical map $M \to M^{\ast \ast}$ is the $R$-torsion submodule of $M$. Indeed, since 
 ($S_1$) is equivalent to asserting each prime in $\Spec^{ \geq 1} (R)$ contains a nonzero divisor, and hence that $R$ has no embedded primes, the claimed equivalence follows directly from \cite[Th.\@ A.1]{Vasconcelos68}. 
\end{rk}

The next result extends \cite[Prop.\@ 2.4]{bss}. Before stating it, we recall that an $R$-module $M$ is called `torsion-less' if the natural map $M \to M^{\ast \ast}$ is injective (cf.\@  Bass \cite[\S\,3]{Bass62}).

\begin{cor}\label{prop 2.4 extension}
    Let $R$ be a $G_2$-ring and $r \geq 1$ an integer. If $N$ is an $R$-submodule of a finitely generated $R$-module $M$ with the property that $M/ N$ is torsion-less, then one has
    \[
    \bidual^r_R N = \big \{ a \in \bidual^r_R M \mid f(a) \in N^{\ast \ast} \text{ for all } f \in \exprod^{r - 1}_R M^\ast \big\}.
    \]
\end{cor}

\begin{proof}
Choose a surjection $R^{\oplus s} \twoheadrightarrow (M / N)^\ast$. Since $M / N$ is assumed to be torsion-less, taking duals gives an injection $M / N \hookrightarrow (M / N)^{\ast \ast} \hookrightarrow R^{\oplus s}$. 
We may therefore apply the argument of Lemma \ref{biduals lemma 1}\,(a) to the exact sequence $0 \to N \to M \to R^{\oplus s}$ to deduce the exact sequence (see the exact sequence (\ref{another intermediate step exact sequence}))
\begin{equation} \label{useful exact sequence 2}
\begin{tikzcd}
    0 \arrow{r} & \bidual^r_R N \arrow{r} & \bidual^r_R M \arrow{r} & \big( ( M / N)^\ast \otimes_R \exprod^{r - 1}_R M^\ast \big)^\ast.
\end{tikzcd}%
\end{equation}
In order to show that an element $a$ of $\bidual^r_R M$ belongs to $\bidual^r_R N$ it is therefore sufficient to prove that $a (m \wedge f) = 0 $ for all $m \in (M / N)^\ast$ and $f \in \exprod^{r - 1}_R M^\ast$. By definition, $f (a)$ is the map $x \mapsto a(x \wedge f)$, so $a (m \wedge f) = f(a) (m)$. Now suppose that $f (a)$ belongs to $N^{\ast \ast}$, then the exact sequence (\ref{useful exact sequence 2}) for $r = 1$ shows that $f (a)$ belongs to the kernel of the natural map $M^{\ast \ast} \to (M / N)^{\ast \ast}$ and so $f(a)(m)$ must vanish, as required.
\end{proof}

\begin{lem} \label{ann im comparison} \label{lemma-image-Fitt}
Let $R$ be a $G_2$-ring, $M$ a finitely generated $R$-module, and $r$ a non-negative integer. Then, for every element $a$ of $\bidual^r_R M$, the following claims are valid.
\begin{romanliste}
    \item If $R$ is a self-injective ring, then $\im (a) = \Ann_R ( \Ann_R ( a))$.
    \item If $\Ann_R (m) = (0)$, then $\im (m)^{\ast \ast} = \Fitt^0_{R} ( \Ext^1_R  ( (\bidual^r_R M)/(Rm), R  ))^{\ast \ast}$.
\end{romanliste} 
\end{lem}

\begin{proof}
By definition, we  have $\im (a) = \{ a ( f) \mid f \in \exprod^r_R M^\ast \}$. If $R$ is self-injective, then the map
\[
\exprod^r_R M^\ast \to \big( \bidual^r_R M\big)^\ast, \quad f \mapsto \{ m \mapsto m (f) \}
\]
 is an isomorphism, and so $\im (a) = \{ f (a) \mid f \in (\bidual^r_R M)^\ast \}$. Now, $R$ being self-injective also implies that any $\varphi \in (Ra)^\ast$ can be lifted to an element of $(\bidual_R^r M)^\ast$, whence $\im (a) = \{ f (a) \mid f \in (Ra)^\ast \}$. Since $R a \cong R / \Ann_R (a)$, claim (i) now follows from the isomorphism
 \[
 \Hom_R ( R / \Ann_R (a), R) \stackrel{\simeq}{\longrightarrow} \Ann_R (\Ann_R (a)), 
 \quad f \mapsto f (1).
 \]
 Turning to claim (ii), Lemma \ref{ryotaro-useful-lem}\,(ii) allows us to verify the claimed equality locally at primes of height at most one. Since $R$ is a $G_2$-ring, we therefore may, and will,  assume $R$ is a Gorenstein ring of dimension at most one. 
In this case, then, the module $\Ext^1_R ( \bidual^r_R M, R)$ vanishes since the exterior bidual $\bidual^r_R M$ is reflexive. Upon dualising the tautological exact sequence
\begin{cdiagram}
0 \arrow{r} & Ra \arrow{r} & \bidual^r_R M \arrow{r} & (\textstyle \bidual^r_R M)/(Ra) \arrow{r} & 0
\end{cdiagram}%
we therefore obtain an exact commutative diagram
\begin{equation*} \label{diagram-image-fitt}
\begin{tikzcd}
& \big( \bidual^r_R M \big)^\ast \arrow[twoheadrightarrow]{d}[left]{f \mapsto f(a)} \arrow{r} & (Ra)^\ast \arrow{d}[left]{f \mapsto f(a)} \arrow{r} & \Ext^1_R \big ( (\textstyle \bidual^r_R M)/(Ra), R \big ) \arrow{r} \arrow[dashed]{d} & 0 \\
0 \arrow{r} & \im (a) \arrow{r} & R \arrow{r} & R/\im (a) \arrow{r} & 0.
\end{tikzcd}
\end{equation*}
Here the second vertical map is bijective as $a$ is assumed to generate a free module of rank one, and the first vertical map is surjective since the present hypotheses on $R$ imply the natural map $\exprod^r_R M^\ast \to (\exprod^r_R M^\ast)^{\ast \ast} = (\bidual^r_R M)^\ast$ is surjective (cf.\@ \cite[Prop.~(5.4.9)\,(iii)]{NSW}). Upon applying the Snake Lemma to the diagram, one therefore finds that the dotted vertical map (that is induced by the commutativity of the diagram) is an isomorphism. This isomorphism leads directly to the claimed description of $\im (a)$.
\end{proof}

\begin{rk} It is known that a Noetherian ring $R$ is self-injective if and only if, for every ideal $I$ of $R$, one has $\Ann_R ( \Ann_R (I)) = I$  (cf.\@ \cite[Th.\@ 15.1]{Lam99}).  
\end{rk}

\subsubsection{Reduced rings}

In this subsection we assume to be given a reduced Noetherian ring $\cR$. In this case the total quotient ring  $\cQ$ of $\cR$ is a finite product of fields and so is a semisimple, semilocal ring. This fact is used in the following result to interpret exterior power biduals in terms of the lattices considered by Rubin in \cite[\S\,1.2]{Rub96}. 

\begin{lemma} \label{biduals-reduced-rings}
If $M$ is a finitely generated $\cR$-module, then, for any integer $r \geq 0$, the map $\xi^r_M$ induces an isomorphism
\begin{equation*} \label{xi-map}
\big \{ a \in \cQ \otimes_\cR \exprod^r_\cR M \mid f (a) \in \cR \text{ for all } f \in \exprod^r_\cR M^\ast \big \} \stackrel{\simeq}{\longrightarrow} \bidual^r_\cR M. 
\end{equation*} 
\end{lemma}

\begin{proof} This argument closely follows that of \cite[Prop.\@ A.8]{sbA} (which deals only with Gorenstein orders).
 Firstly, by applying the functor $\Hom_\cR ( \exprod^r_\cR M^\ast, -)$ to the tautological exact sequence
$0 \to \cR \to \cQ \to \cQ/\cR \to 0$ one deduces that $\bidual^r_\cR M$ identifies with the kernel of the natural map $\Hom_\cR \big ( \exprod^r_\cR M^\ast, \cQ \big ) \to \Hom_\cR \big ( \exprod^r_\cR M^\ast, \cQ/\cR \big)$. In addition, since $\cQ$ is semi-simple, the finitely generated $\cQ$-module $\cQ \otimes_\cR M$ is projective and so there exists a natural composite isomorphism 
\[
 \cQ \otimes_\cR \exprod^r_\cR M \cong \exprod^r_\cQ ( \cQ \otimes_\cR M) 
\cong \bidual^r_\cQ ( \cQ \otimes_\cR M) \cong \Hom_\cR ( \exprod^r_\cR M^\ast, \cQ ),
\]
where the first isomorphism is clear, the second is  $\xi^r_{\cQ \otimes_\cR M}$ and the third is induced by tensor-hom adjunction. 
 The claimed equality is therefore true since an element $a$ belongs to the kernel of the natural map $\cQ \otimes_\cR \exprod^r_\cR M \to  \Hom_\cR \big ( \exprod^r_\cR M^\ast, \cQ/\cR \big )$ if and only if one has $ f (a) \in \cR$ for every $f \in \exprod^r_\cR M^\ast$.  
\end{proof}

\subsection{Complexes, resolutions, and determinants}\label{crd section} 

In the sequel, for a (commutative) noetherian ring $\Lambda$, we write $D(\Lambda)$ for the derived category of $\Lambda$-modules. We also write $D^-(\Lambda)$ and $D^{\rm{perf}}(\Lambda)$ for the full triangulated subcategories of $D(\Lambda)$ comprising complexes that are respectively bounded above and `perfect' (that is, isomorphic in $D(\Lambda)$ to a bounded complex of finitely generated projective $\Lambda$-modules). 

Each object $C$ of $D^{\rm{perf}}(\Lambda)$ has an associated Euler characteristic $\bm{\chi}_\Lambda(C)$ in the Grothendieck group $K_0(\Lambda)$ of finitely generated projective $\Lambda$-modules. If $\Lambda$ is local, then a finitely generated projective module is free and so the $\Lambda$-rank map $P \mapsto \mathrm{rk}_\Lambda(P)$ induces an isomorphism of $K_0(\Lambda)$ with $\ZZ$. In such cases, we regard $\bm{\chi}_\Lambda(C)$ as an integer (so that the class of $C$ in $K_0(\Lambda)$ is equal to $\bm{\chi}_\Lambda(C)\cdot [\Lambda]$).  

We write $D^{\rm{perf},0}(\Lambda)$ for the full triangulated subcategory of $D^{\rm{perf}}(\Lambda)$ represented by complexes $C$ for which $\bm{\chi}_\Lambda(C) = 0$. 
Let $a$ and $b$ be integers with $a\le b$ and let $D^\bullet(\Lambda)$ denote either $D(\Lambda)$, $D^{\rm{perf}}(\Lambda)$ or $D^{\rm{perf},0}(\Lambda)$. Then we write $D^\bullet_{[a,b]}(\Lambda)$, $D^\bullet_{[a,\cdot]}(\Lambda)$ and $D^\bullet_{[\cdot,b]}(\Lambda)$ for the subcategories of $D^\bullet(\Lambda)$ that respectively comprise complexes $C$ for which $H^i(C) = (0)$ if either $i < a$ or $i >b$, $H^i(C)=(0)$ if $i < a$ and $H^i(C)=(0)$ if $i > b$.

In the rest of this section, we fix a ring $R \cong \varprojlim_{n \in \N} R_n$ as in \S\,\ref{cgr section} (so that, in particular, $R$ is Noetherian and local) and describe two useful constructions of complexes in $D^{\rm{perf}}(R)$. 

\subsubsection{Limit complexes}

The following result is essentially well-known. 

\begin{lemma}\label{limit result} We fix integers $a$ and $b$ with $a<b$ and, for every natural number $n$, assume to be given data of the following form
\begin{liste}
\item an object $C_n$ of $D^{\mathrm{perf}}_{[a,b]}(R_n)$; 
\item an isomorphism $\theta_n \: R_n\otimes^{\mathbb{L}}_{R_{n+1}}C_{n+1} \cong C_n$ in $D(R_n)$. 
\end{liste}
Then there exists a bounded complex $\varprojlim_nC_n$ of finitely generated free $R$-modules that is unique up to isomorphism in $D^{\mathrm{perf}}(R)$ and has the following properties: 
\begin{romanliste}
\item $(\varprojlim_nC_n)^i = (0)$ unless $a\le i\le b$;
\item for every $m$, there exists an isomorphism $R_m\otimes^{\mathbb{L}}_{R} (\varprojlim_nC_n) \cong C_m$ in $D(R_m)$;
\item for every $i$, the induced composite map 
\[ H^i({\varprojlim}_nC_n) \to {\varprojlim}_nH^i(R_n\otimes^{\mathbb{L}}_{R} ({\varprojlim}_mC_m)) \cong {\varprojlim}_n H^i(C_n)\]
is bijective. 
\end{romanliste}
Further, if $R_m\otimes^{\mathbb{L}}_{\Lambda}C$ belongs to $D^{\mathrm{perf},0}(R_m)$ for any given $m$, then ${\varprojlim}_nC_n$ belongs to $D^{\mathrm{perf},0}(R)$.  
\end{lemma}

\begin{proof} The conditions (a) and (b) imply that the general argument of \cite[XV, p.~472 to the end]{sga5} applies to realise each $C_n$ by a bounded complex $\hat C_n^\bullet$ of finitely generated $R_n$-modules with $\hat C^i_n = (0)$ if either $i < a$ or $i > b$, $\hat C^i_n$ free for all $i \not= a$ and $\hat C_n^a$ of finite projective dimension, and each isomorphism $\theta_n$ by an isomorphism of complexes $\hat\theta_n \: R_n\otimes_{R_{n+1}}\hat C^\bullet_{n+1} \cong \hat C^\bullet_n$ in such a way that the following condition is satsified. If $(\tilde C_n^\bullet,\tilde\theta_n)$ is any other family of complexes and isomorphisms  that satisfy the same conditions, then there exists a family $(\phi_n \: \hat C_n^\bullet \to \tilde C_n^\bullet)_n$ of isomorphisms of complexes of $R_n$-modules with the property that, for every $n$, there exists a commutative diagram of morphisms of complexes of $R_n$-modules  
\begin{equation}\label{compat diagram}
\xymatrixrowsep{5mm}
\xymatrix{
 \hat C_{n+1}^\bullet \ar[d] \ar[r]^{\phi_{n+1}} & \tilde C_{n+1}^\bullet \ar[d] \\
 \hat C_{n}^\bullet  \ar[r]^{\phi_{n}} & \tilde C_{n}^\bullet}\end{equation}
in which the first and second vertical morphism are respectively induced by $\hat\theta_n$ and $\tilde\theta_n$. 

We next claim that each $R_n$-module $\hat C_n^a$ is projective, and hence free. To see this, we note that for each $\p\in \mathrm{Spec}(R_n)$, one has $\mathrm{depth}_{R_{n,\p}} (R_{n,\p}) = \dim (R_{n,\p}) = 0$ (as $R_n$ is Cohen--Macaulay), so $\mathrm{pd}_{R_{n,\p}} (\hat C^a_{n,\p}) = 0$ (by the Auslander--Buchsbaum formula) and hence $\hat C^a_{n,\p}$ is a free $R_{n,\p}$-module. It follows that $\hat C^a_n$ is a locally free $R_n$-module and hence projective, as claimed. 

The limit $C \coloneqq {\varprojlim}_n \hat C^\bullet_n$ with respect to the morphisms $\hat C^\bullet_{n+1}\to R_n\otimes_{R_{n+1}}\hat C^\bullet_{n+1} \cong \hat C_n^\bullet$ induced by $\hat\theta_n$ is then a bounded complex of $R$-modules that is unique up to isomorphism in $D(R)$ (as a consequence of the diagrams (\ref{compat diagram})) and for which there exists for every $n$ a natural isomorphism $R_n\otimes_R C \cong \hat C^\bullet_n$ of complexes of $R_n$-modules. 

We now claim that each term $C^i = {\varprojlim}_n \hat C_n^i$ is a finitely generated free $R$-module, and hence that $C$ belongs to $D^{\mathrm{perf}}(R)$. To see this, we note each $R_n$-module $\hat C_n^i$ is finitely generated and free and that each projection map $\hat C_{n+1}^i \to R_n\otimes_{R_{n+1}}\hat C_{n+1}^i \cong \hat C_n^i$ is surjective and has kernel $(\a_n / \a_{n + 1})\hat C^i_{n + 1}$. Thus, since $\a_n / \a_{n + 1}$ is contained in the Jacobson radical of $R_{n + 1}$, Nakayama's Lemma implies that each limit $C^i = {\varprojlim}_n \hat C_n^i$ is finitely generated and free over $R = {\varprojlim}_nR_n$ with ${\rm{rk}}_R(C^i) = {\rm{rk}}_{R_m}(\hat C_m)$ for any choice of $m$.

Next we note that each group $H^i(C_n)$ is finite and hence, since inverse systems of finite groups satisfy the Mittag--Leffler condition, the first derived limit ${\varprojlim}_{n}^{(1)} H^i (C_n)$ vanishes. From the exact sequences $0 \to {\varprojlim}_{n}^{(1)} H^{i - 1} (C_n) \to H^i(C) \to {\varprojlim}_nH^i(\hat C_n) \to 0$ we then deduce that the natural maps $H^i(C) \to {\varprojlim}_nH^i(\hat C_n) = {\varprojlim}_n H^i(C_n)$ are bijective, as claimed in (c).\\ 
It thus only remains to show that $\bm{\chi}_R(C) = 0$ if $\bm{\chi}_{R_m}(R_m\otimes^{\mathbb{L}}_{\Lambda}C) = 0$ for any given $m$. But this is true, since, by the argument above, each $R$-module $C^i$ is free of rank ${\rm{rk}}_{R_m}(\hat C_m^i)$, whilst, as $R_m$ is Cohen--Macaulay, the Bass Cancellation Theorem \cite[Th.\@ I.2.3\,(b)]{K-book} combines with the vanishing of $\bm{\chi}_{R_m}(R_m\otimes_RC) = \bm{\chi}_{R_m}(\hat C_m^\bullet)$ to imply that $\sum_{i\in \ZZ}(-1)^{i}{\rm{rk}}_{R_m}(\hat C_m^i)=0$.  
\end{proof} 

We next record relations between the cohomology modules of complexes $C_n$ as in Lemma \ref{limit result}.

\begin{lem} \label{how the cohomology base changes lemma}
    Fix integers $a$ and $b$ with $a<b$ and, for every natural number $n$, assume the pair $(C_n, \theta_n)$ satisfies (a) and (b) in Lemma \ref{limit result}. Then the following claims are valid.
    \begin{romanliste}
        \item Any choice of $R_{n + 1}$-linear injection $j_n \: R_n \hookrightarrow R_{n + 1}$ of rings as in Lemma \ref{inj env lem} induces an isomorphism $H^a (C_{n + 1}) [\a_{n + 1}] \cong H^a (C_n)$.
        \item Each map $\theta_{n + 1}$ induces an isomorphism $H^b (C_{n + 1}) \otimes_{R_{n + 1}} R_n \cong H^b (C_n)$.
    \end{romanliste}
\end{lem}

\begin{proof}
    We have seen in the proof of Lemma \ref{limit result} that $C_{n + 1}$ can be represented by a bounded complex $\hat C^\bullet_{n + 1}$ of finitely generated free $R_{n + 1}$-modules $\hat C^i_{n + 1}$ that are zero if $i \not \in [a, b]$ and with differentials $\partial^i_{n +1} \: \hat C_{n + 1}^i \to \hat C_{n + 1}^{i + 1}$. In particular, $H^a (C_{n + 1})$ identifies with the kernel of $\partial^a_{n + 1}$ and so we may construct a commutative diagram of the form
    \begin{cdiagram}
        0 \arrow{r} & H^a (C_{n + 1}) [\a_{n + 1}] \arrow{r}  & \hat C^a_{n + 1} [\a_{n + 1}] \arrow{r}{\partial^a_{n +1}}  & \hat C^{a + 1}_{n + 1}  \\ 
        0 \arrow{r} & H^a (C_{n + 1} \otimes^\mathbb{L}_{R_{n + 1}} R_n) \arrow[dashed]{u} \arrow{r} & C^a_{n + 1} \otimes_{R_{n + 1}} R_n \arrow{r}{\partial^a_{n +1}}  \arrow{u}{\simeq} & C^{a + 1}_{n + 1} \otimes_{R_{n + 1}} R_n \arrow{u}{\simeq},
    \end{cdiagram}%
    where the first row is exact because $(-) [\a_{n + 1}]$ is a left-exact functor, the two vertical isomorphisms are induced by our choice of identification $j_n \: R_n \cong R_{n + 1} [\a_{n + 1}]$, and the second square commutes because $j_n$ is $R_{n + 1}$-linear. Since $\theta_n$ identifies $C_n$ with $\hat C^\bullet_{n + 1} \otimes^\mathbb{L}_{R_{n + 1}} R_n$, the dashed arrow in the above diagram gives the isomorphism required to establish claim (i). \\
    As for claim (ii), we note that $H^b (C_{n + 1})$ identifies with the cokernel of $\partial^b_{n + 1}$ because $\hat C_{n + 1}^i $ vanishes for $i > b$. 
    Now, $C_n \cong \hat C^\bullet_{n + 1} \otimes^\mathbb{L}_{R_{n + 1}} R_n$ via $\theta_n$ and so 
     $H^1 (C_n)$ can similarly be identified with the cokernel of the map $\hat C_{n + 1}^{b - 1} \otimes_{R_{n + 1}} R_n \to \hat C_{n + 1}^b \otimes_{R_{n + 1}} R_n$ induced by $\partial^b_{n + 1}$. Given this,  (ii) follows upon noting  these two cokernels agree since $( - ) \otimes_{R_{n + 1}} R_n$ is a right-exact functor.
\end{proof}

We finally prove a useful technical result concerning Tor-groups.

\begin{lem} \label{Tor lemma}
Let $ f\: R \to S$ be a morphism of rings satisfying (\ref{ring condition}), and assume that $f$ arises as the limit of commutative diagrams of finite rings of the form
\begin{cdiagram}[row sep=small]
    R_{n + 1} \arrow{r}{f_{n + 1}} \arrow{d}
    & S_{n + 1} \arrow{d}
    \\
    R_n \arrow{r}{f_n} & S_n.
\end{cdiagram}%
    Let $(M_n)_{n \in \N}$ be a projective system of $R_n$-modules and set $M \coloneqq \varprojlim_{n \in \N} M_n$.
    \begin{romanliste}
    \item For every $n \in \N$ such that $M_{n + 1}$ is finitely presented as an $R_{n + 1}$-module, the maps $f_n$ and $M_{n + 1} \to M_n$ induce a natural map 
    \[ \Tor_1^{R_{n + 1}} ( M_{n + 1}, S_{n + 1}) \to \Tor_1^{R_n} (M_n, S_n),\] 
    where we regard  each $S_n$ as an $R_n$-module via $f_n$. Further, if $f_{n+1}$ and $f_n$ are surjective and the natural map $M_{n+1}\otimes_{R_{n+1}}R_n \to M_n$ is bijective, then the cokernel of the above map is isomorphic to the cokernel of the natural map
    \begin{equation} \label{2nd Tor base change map}
    \Tor_1^{R_{n + 1}} (M_{n + 1}, R_n) \to \Tor_1^{S_{n + 1}} ( M_{n + 1} \otimes_{R_{n + 1}} S_{n + 1}, S_n). 
    \end{equation}
    \item If all transition maps $M_{n + 1} \to M_n$ are surjective, and $M$ is finitely presented as an $R$-module, then there exists a canonical isomorphism 
    \[ {\varprojlim}_{n \in \N }\Tor_1^{R_n} (M_n, S_n) \cong \Tor_1^R (M, S),\] 
where the limit is defined with respect to the maps in (i).     
    \end{romanliste}
\end{lem}

\begin{proof}
At the outset we note that if $\cR \to \cS$ is a morphism of commutative Noetherian rings and $C$ is a perfect complex of $\cR$-modules, then one has a spectral sequence
\begin{equation*} \label{Tor spectral sequence}
    E_2^{i,j} = \Tor_{-i}^{\cR} ( H^j (C), \cS) \Rightarrow E^{i + j} = H^{i + j} ( C \otimes^\mathbb{L}_\cR \cS).
\end{equation*}
If $C$ is acyclic outside degrees zero and one, this spectral sequence degenerates on its second page to give an exact sequence
\begin{equation} \label{exact sequence from Tor spectral sequence}
    \begin{tikzcd}[column sep=small]
        \Tor_2^\cR ( H^1 (C), \cS) \arrow{r} & H^0 ( C) \otimes_\cR \cS \arrow{r} & H^0 ( C \otimes^\mathbb{L}_\cR \cS) \arrow{r} & \Tor_1^\cR ( H^1 (C), \cS) \arrow{r} & 0.
    \end{tikzcd}
\end{equation}
Choose an $R_{n + 1}$-free presentation $P^0_{n + 1} \to P^1_{n + 1} \to M_{n + 1} \to 0$ of $M_{n + 1}$, which we regard as a perfect complex $P_{n + 1} = [P^0_{n + 1} \to P^1_{n + 1}]$ in $D(R_{n + 1})$, where $P^0_{n + 1}$ is placed in degree zero. We then obtain a commutative diagram of the form
\begin{equation*} \label{Tor diagram}
\begin{tikzcd}[column sep=small, row sep=small]
    H^0 ( P_{n + 1} )\otimes_{R_{n+1}}S_{n+1} \arrow{r} \arrow{d} & H^0 ( P_{n + 1} \otimes^\mathbb{L}_{R_{n + 1}} S_{n + 1}) \arrow[twoheadrightarrow]{r} \arrow{d} & \Tor_1^{R_{n + 1}} ( M_{n + 1}, S_{n + 1}) \arrow[dashed]{d} \\
    H^0 ( P_{n + 1} \otimes^\mathbb{L}_{R_{n + 1}} R_{n})\otimes_{R_n}S_n \arrow{r} 
    & H^0 ( P_{n + 1} \otimes^\mathbb{L}_{R_{n + 1}} S_{n}) \arrow[twoheadrightarrow]{r} 
    & \Tor_1^{R_{n}} ( M_{n + 1} \otimes_{R_{n + 1}} R_n, S_{n}).
\end{tikzcd}%
\end{equation*}
Here the first two vertical arrows are the natural maps, the first row is (\ref{exact sequence from Tor spectral sequence}) applied with $C = P_{n + 1}$, $\cR = R_{n + 1}$, and $\cS = S_{n + 1}$, and the second row is (\ref{exact sequence from Tor spectral sequence}) applied with $C = P_{n + 1} \otimes^\mathbb{L}_{R_{n + 1}} R_{n}$, $\cR = R_{n}$, and $\cS = S_{n}$. Since the first square commutes, the dashed arrow now exists due to exactness of the rows, and we define the map in claim (i) to be the composite of this dashed arrow with the natural map $\Tor_1^{R_{n}} ( M_{n + 1} \otimes_{R_{n + 1}} R_n, S_{n}) \to \Tor_1^{R_n} (M_n, S_n)$ induced by $M_{n + 1} \otimes_{R_{n + 1}} R_n \to M_n$. 

To complete the proof of (i), we assume that the maps $f_{n+1}$ and $f_n$ are surjective and the map $M_{n + 1} \otimes_{R_{n + 1}} R_n \to M_n$ is bijective.  In this case, the above diagram simplifies to give an exact commutative diagram 
\begin{equation*}\begin{tikzcd}[column sep=small, row sep=small]
    H^0 ( P_{n + 1} ) \arrow{r} \arrow{d} & H^0 ( P_{n + 1} \otimes^\mathbb{L}_{R_{n + 1}} S_{n + 1}) \arrow[twoheadrightarrow]{r} \arrow{d} & \Tor_1^{R_{n + 1}} ( M_{n + 1}, S_{n + 1}) \arrow{d} \\
    H^0 ( P_{n + 1} \otimes^\mathbb{L}_{R_{n + 1}} R_{n}) \arrow{r} 
    \arrow[twoheadrightarrow]{d} 
    & H^0 ( P_{n + 1} \otimes^\mathbb{L}_{R_{n + 1}} S_{n}) \arrow[twoheadrightarrow]{r} \arrow[twoheadrightarrow]{d} 
    & \Tor_1^{R_{n}} ( M_{n}, S_{n})\\
    \Tor_1^{R_{n + 1}} (M_{n + 1}, R_n) \arrow{r}{\alpha} &\Tor_1^{S_{n + 1}} ( M_{n + 1} \otimes_{R_{n + 1}} S_{n + 1}, S_n).
\end{tikzcd}%
\end{equation*}
Here the exactness of the first two columns follows from an application of (\ref{exact sequence from Tor spectral sequence}) with $C = P_{n + 1}$, $\cR = R_{n + 1}$, and $\cS = S_{n}$, respectively  $C = P_{n+1}\otimes^\mathbb{L}_R S_{n+ 1}$, $\cR = S_{n + 1}$, and $\cS = S_{n}$, and $\alpha$ is the natural map in (\ref{2nd Tor base change map}). The final assertion of (i) is therefore obtained by applying the Snake Lemma to this diagram.  
\\
To prove (ii), we fix an $R$-free presentation $P^0 \to P^1 \to M \to 0$ and regard it as a complex $P = [P^0 \to P^1]$ in $D(R)$. Taking $P_{n + 1} = P \otimes^\mathbb{L}_R R_{n + 1}$ and $M_{n + 1} = M \otimes_R R_{n + 1}$ in  the first diagram above  and passing to the limit over $n$
leads to an exact sequence (because all appearing modules are assumed to be finite) that forms the first row of the diagram
\begin{cdiagram}[column sep=small]
{{\varprojlim}}_{n \in \N} H^0 ( P \otimes^\mathbb{L}_R R_{n}) \arrow{r} \arrow{d}{\simeq} & {{\varprojlim}}_{n \in \N} H^0 ( P \otimes^\mathbb{L}_R S_{n}) \arrow[twoheadrightarrow]{r} \arrow{d}{\simeq} & {{\varprojlim}}_{n \in \N} \Tor_1^{R_{n}} ( M_{n}, S_{n})  \arrow[dashed]{d} \arrow{r} & 0\\
H^0 ( P) \arrow{r}  & H^0 ( P \otimes^\mathbb{L}_R S) \arrow[twoheadrightarrow]{r} & \Tor_1^{R} ( M, S) \arrow{r} & 0.
\end{cdiagram}%
Here the second row is (\ref{exact sequence from Tor spectral sequence}) applied with $C = P$, $\cR = R$, and $\cS = S$, and the isomorphisms are by Lemma \ref{limit result}. As a consequence, we deduce that the dashed arrow is an isomorphism.\\
For every $n \in \N$, we now write $K_n$ for the kernel of the map $M \otimes_R R_n \to M_n$ (which is surjective because the system $(M_n)_{n \in \N}$ is assumed to have surjective transition maps). Since all appearing modules in the exact sequence
\begin{cdiagram}
    \Tor_1^{R_n} (K_n, S_n) \arrow{r} & \Tor_1^{R_n} (M \otimes_R R_n, S_n) \arrow{r} & \Tor_1^{R_n} (M_n, S_n) \arrow{r} & K_n \otimes_{R_n} S_n
\end{cdiagram}%
are finite, passing to the limit over $n$ yields an exact sequence. To prove the isomorphism claimed in (ii), it is therefore enough to show $\varprojlim_{n \in \N}\Tor_1^{R_n} (K_n, S_n)$ and $\varprojlim_{n \in \N} (K_n \otimes_{R_n} S_n)$ both vanish. To do this, we note that the limit over the maps $M \otimes_R R_n \to M_n$ is an isomorphism (because $M$ is finitely presented and each $R_n$ is finite), and hence that $\varprojlim_{n \in \N} K_n$ vanishes. We now fix $m \in \N$, write $g_n$ for the map $K_n \to K_m$ for every $n \geq m$, and consider the group of universal norms $U_m \coloneqq \bigcap_{n \geq m} g_n (K_n)$ in $K_m$. Since each $K_n$ is finite, $U_m$ coincides with the image of $\varprojlim_{n \in \N} K_n \to K_m$ (cf.\@ \cite[Lem.\@ 3.10]{BullachDaoud}), which implies that $U_m = (0)$. By finiteness of $K_m$, we can therefore fix an integer $N$ with $N \geq m$ for which $g_N$ is the zero map. It follows that also the maps 
$\Tor_1^{R_N} (K_N, S_N) \to \Tor_1^{R_m} (K_m, S_m)$ and $K_N \otimes_{R_N} S_N \to K_m \otimes_{R_m} S_m$ induced by $g_N$ are both zero. As $m$ was chosen arbitrarily, this proves the required vanishing of both $\varprojlim_{n \in \N}\Tor_1^{R_n} (K_n, S_n)$ and $\varprojlim_{n \in \N} (K_n \otimes_{R_n} S_n)$, thereby concluding the proof of (ii).
\end{proof}

\subsubsection{Quadratic resolutions}

The class of complexes considered in the next result plays an important role in later arguments.

\begin{lem}\label{finite level reps} 
Let $C$ be an object of $D^{\mathrm{perf}}_{[m, 1]}(R)$ for some $m \le 1$ such that, for some $m ' \in \N_0$ and $n' \in \N$, the complex $R_{n'}\otimes^{\mathbb{L}}_{R}C$ belongs to $D^{\mathrm{perf}}_{[0,m']}(R_{n'})$. Then the following claims are valid. 
\begin{romanliste}
\item $R_n\otimes^{\mathbb{L}}_{R}C$ belongs to $D^{\mathrm{perf}}_{[0,1]}(R_n)$ for all $n \in \N$, and the integer $a\coloneqq \bm{\chi}_{R_n}(R_n\otimes^{\mathbb{L}}_{R}C)$ is independent of $n$.
\item $C$ is isomorphic in $D^{\mathrm{perf}}(R)$ to a complex of finitely-generated free $R$-modules $P_0 \stackrel{\phi}{\to} P_1$, in which $P_0$ occurs in degree $0$ and one has $\bm{\chi}_R(C) = \mathrm{rk}_R(P_0)-\mathrm{rk}_R(P_1) = a$. 
\item Let $Y$ be a free $R$-module quotient of $H^1 (C)$. Then the composite map
\[ \pi \: P_1 \to \mathrm{cok}(\phi) \cong H^1 (C) \to Y\] 
induced by (ii) induces an isomorphism of $R$-modules $P_1 \cong  \ker (\pi)\oplus Y$. In particular, the $R$-module $\ker(\pi)$ is free of rank $\mathrm{rk}_R(P_1) - \mathrm{rk}_R(Y)$.
\end{romanliste}
\end{lem}

\begin{proof} 
Since $C$ belongs to $D^\mathrm{perf} (R)$, it is clear that $C_n \coloneqq C \otimes_R^\mathbb{L} R$ is in $D^\mathrm{perf} (R_n)$ for every $n$. The same argument as for Lemma \ref{how the cohomology base changes lemma}\,(ii) moreover shows (under the given hypotheses) that $H^b (C_n) \cong H^b (C) \otimes_R R_n$ for $b \in \N$, and hence that $C_n$ belongs to $D^\mathrm{perf}_{[m_n, 1]} (R_n)$ for some $m_n \le 0$. In addition, the assumption that $C_{n'}$ belongs to $D^\mathrm{perf}_{[0, m']} (R_{n'})$ combines with the results of Lemmas \ref{how the cohomology base changes lemma}\,(i) and \ref{ryotaro's reduction trick} to imply that, for each $a < 0$ and every $n \in \N$, the module $H^a (C_n)$ vanishes, as required to prove $C_n$ belongs to $D^\mathrm{perf}_{[0, 1]} (R_n)$, as claimed in (i). The second assertion of (i) being clear, we now define a complex 
\[ \tilde C\coloneqq \begin{cases}  C\oplus R^{\oplus a}[-1], &\text{if $a \ge 0$,}\\
C\oplus R^{\oplus (-a)}[0], &\text{if $a < 0$.}\end{cases}\] 
This complex belongs to $D^{\mathrm{perf}}_{[m, 1]}(R)$ and (i) implies that, for all $n$, $R_n\otimes^{\mathbb{L}}_{R}\tilde C$ belongs to $D^{\mathrm{perf}}_{[0,1]}(R_n)$. It is also clear that the validity of (ii) and (iii) for $\tilde C$ implies their validity for $C$.  Hence, after replacing $C$ by $\tilde C$ if necessary, in the rest of the argument we will assume $a=0$. 

We first consider the case that $R$ is finite and self-injective. Then, after fixing a surjective map of $R$-modules $\tau \: P \to H^1(C)$, with $P$ finitely generated and free, the argument of \cite[Prop.~A.11\,(i)]{sbA} proves the existence of a finitely generated $R$-module $P'$ of finite projective dimension such that $C$ is isomorphic to $P'\xrightarrow{\phi} P$ and the induced map $P \to {\rm{cok}}(\phi) \cong H^1(C)$ coincides with $\tau$. Since the vanishing of $\bm{\chi}_{R}(C)$ combines with the Bass cancellation theorem \cite[Th.~I.2.3\,(b)]{K-book} to imply $P'$ is isomorphic to $P$ (just as in the argument of Lemma \ref{limit result}), the claimed result is therefore true in this case. 

To deal now with the general case, we fix a surjective map of $R$-modules $\tau\: P \to H^1(C)$, with $P$ finitely generated and free and note that, as $Y$ is free, the composite $\tau'$ of $\tau$ with the projection $H^1(C) \to Y$ induces an isomorphism $P \cong \ker(\tau') \oplus Y$. For each $n$, we set $P_n \coloneqq  R_n\otimes_RP$ and $C_n \coloneqq  R_n\otimes_R^{\mathbb{L}}C$ and note $H^1(C_n) \cong R_n\otimes_R H^1(C)$ since $H^i(C) = (0)$ for all $i > 1$. With $\tau_n$ denoting the surjective map $P_n \to H^1(C_n)$ induced by $\tau$, we can then combine the argument in the last paragraph with that of Lemma \ref{limit result} to fix a resolution $P_n\xrightarrow{\phi_n}P_n$ of $C_n$ so that the induced map $P_n \to {\rm{cok}}(\phi_n) \cong H^1(C_n)$ is $\tau_n$ and the isomorphism $R_n\otimes_{R_{n+1}}^{\mathbb{L}}C_{n+1} \cong C_n$ in $D(R_n)$ is induced by the identification $R_n\otimes_{R_{n+1}}P_{n+1} = R_n\otimes_R P = P_n$. Setting 
$\phi \coloneqq  {\varprojlim}_n\phi_n$, one can then check the complex $P \xrightarrow{\phi} P$ is isomorphic in $D(R)$ to $C$ in such a way that the induced map $P \to {\rm {cok}}(\phi) \cong H^1(C)$ coincides with $\tau = {\varprojlim}_n\tau_n$. Given this, it is easily verified that the complex $P\xrightarrow{\phi} P$ has all of the required properties. 
\end{proof}

In the next two results, we record some useful consequences of the quadratic resolutions constructed in Lemma \ref{finite level reps}\,(i). In both of these results, we fix data $C$ and $Y$ as in the last result, we write $X$ for the kernel of the (given) surjective homomorphism $H^1(C) \to Y$ and we set 
\[ r = r_{Y,C} \coloneqq {\mathrm{rk}}_R(Y) + \bm{\chi}_R(C).\]

\begin{lem} \label{comparing Fitting ideals of H0 and H1 lemma}
   Fix $i\in \N_0$ such that $i+ r\ge 0$. Then one has 
\[\Fitt^{i + r}_R ( H^1 (\RHom_R (C, R) [-1])) = \Fitt^i_R (X)\]
and hence, if $R$ is self-injective, also $\Fitt^{i + r}_R ( H^0 (C)^\ast) = \Fitt^i_R (X)$. 
\end{lem}

\begin{proof} By using Lemma \ref{standard fitting props}\,(iii), it is easily checked that the stated claims are valid if and only if they are valid after replacing $C$ by the complex $\tilde C$ used in the proof of Lemma \ref{finite level reps}. In particular, since $\bm{\chi}_R(\tilde C) = 0$, in the sequel we shall assume $\bm{\chi}_R(C) = 0$, and hence that $r = {\mathrm{rk}}_R(Y)$.

Now, since $Y$ is free, there exists a (non-canonical) isomorphism of $R$-modules $H^1(C) \cong X \oplus Y$ and so, by using Lemma \ref{standard fitting props}\,(iii) again, one finds that  $\Fitt^i_R (X) = \Fitt^{i + r}_R (H^1 (C))$. It therefore suffices to prove that $\Fitt^j_R ( H^1 (\RHom_R (C, R)[-1])) = \Fitt^j_R (H^1 (C))$ for each $j\ge 0$. 

To do this we note that, since $\bm{\chi}_R(C) = 0$, Lemma \ref{finite level reps}\,(ii) implies that we can fix a resolution of $C$ of the form $P \xrightarrow{\phi} P$, with the first term placed in degree $0$. We set $n \coloneqq \mathrm{rk}_R (P)$ and write $A$ for the $(n \times n)$-matrix representing $\phi$ with respect to a fixed choice of basis $\mathfrak{B}$ for $P$. Then $\Fitt^j_R (H^1 (C))$ can be computed as the ideal of $R$ that is generated by the collection of  $(n - j) \times (n - j)$-minors of $A$. 
    We write $b^\ast \: P \to R$ for the `dual' of $b \in \mathfrak{B}$ and define a basis of $P^\ast$ by $\mathfrak{B}^\ast \coloneqq \{ b^\ast \mid b \in \mathfrak{B}\}$. With respect to this choice of basis, the dual $\phi^\ast \: P^\ast \to P^\ast$ is represented by the transpose $A^\mathrm{t}$ of $A$. The complex $\RHom_R (C, R) [-1]$ is represented by $P^\ast \xrightarrow{\phi^\ast} P^\ast$ so that $\coker (\phi^\ast) = H^1 ( \RHom_R (C, R) [-1])$, hence we deduce that $\Fitt^j_R ( H^1 ( \RHom_R (C, R) [-1]))$ is equal to the ideal of $R$ generated by the $(n - j) \times (n - j)$-minors of $\iota (A^\mathrm{t})$. Since the sets of minors of $A$ and $A^\mathrm{t}$ are in bijective correspondence, this proves $\Fitt^j_R ( H^1 (\RHom_R (C, R) [-1]))$ is equal to $\Fitt^j_R (H^1 (C))$, and hence completes the proof of the first claim. 
   
The second claim is then true since if $R$ is self-injective, then  $\Hom_R (- , R)$ is an exact functor and so $H^1 (\RHom_R (C, R)[-1]) = H^0 (C)^\ast$.
\end{proof} 

\begin{lemma}\label{det projection lemma} Set  $r_Y \coloneqq \mathrm{rk}_R(Y)$. Then, if $r>0$, the following claims are valid. 
\begin{itemize}
\item[(i)] Each choice of ordered basis $b_\bullet$ of the (free) $R$-module $Y$ gives rise to a canonical homomorphism of $R$-modules $\vartheta_{C,b_\bullet} \: \Det_R (C) \to \bidual^{r}_R H^0 (C).$
\item[(ii)] Let $R\to R'$ be a surjective homomorphism of rings, with $R'$ as in \S\,\ref{cgr section}. Then the object $C' \coloneqq R'\otimes^{\mathbb{L}}_RC$ of $D^{\mathrm{perf}}(R')$ satisfies the conditions of Lemma \ref{finite level reps} with $R$ replaced by $R'$. In addition, $Y' \coloneqq  R'\otimes_RY$ is a free $R'$-module quotient of $H^1(C') \cong R'\otimes_RH^1(C)$ of rank $r_Y$, the image $b_\bullet'$ of the basis $b_\bullet$ in (i) is an ordered $R'$-basis of $Y'$ and there exists a canonical commutative diagram of $R$-module homomorphisms 
\begin{cdiagram} \Det_R(C) \arrow{rr}{\vartheta_{C,b_\bullet}} \arrow{d}{{\pi_{C,R'}}} & & \bidual^{r}_R H^0(C) \arrow{d}{{\bigcap}_R^{r}\pi_{C,R'}} \\ 
    \Det_{R'} (C') \arrow{rr}{\vartheta_{C',b_\bullet'}} & & \bidual^{r}_{R'} H^0(C').
\end{cdiagram}%
Here $\pi_{C,R'}$ is the canonical composite $\Det_{R} (C) \to R'\otimes_R\Det_{R} (C) \cong \Det_{R'} (C')$ and the map ${\bigcap}_R^r\pi_{C,R'}$ is defined in the course of the proof below.  
\end{itemize}
\end{lemma}

\begin{proof}   
Set $a \coloneqq \bm{\chi}_R(C)$. Then, since $r = r_Y+a > 0$, the argument of Lemma \ref{finite level reps}\,(ii) implies we can fix a natural number $n\geq r$ such that $C$ has a resolution $P^\bullet$ of the form $R^{\oplus n} \xrightarrow{\phi} R^{\oplus (n - r)} \oplus Y$. Then $\mathrm{rk}_R(R^{\oplus (n - r)} \oplus Y) = n-a$ and we consider the composite homomorphism of $R$-modules 
\begin{align*}
\vartheta_{P^\bullet, b_\bullet} \: \Det_R (P^\bullet) & = \big( \exprod^n_R R^{\oplus n} \big) \otimes_R \big ( \exprod_R^{n - a} (R^{\oplus (n-r)}\oplus Y)\big)^\ast \\
& \xrightarrow{\simeq} \big( \exprod^n_R R^{\oplus n} \big) \otimes_R \big ( \exprod_R^{n - r} R^{\oplus (n - r)} \big)^\ast \otimes_R \big( \exprod^{r_Y}_R Y \big)^\ast \\
& \xrightarrow{\simeq} \big( \exprod^n_R R^{\oplus n} \big) \otimes_R  \exprod_R^{n - r} (R^{\oplus (n - r)})^\ast  \\
& \xrightarrow{\vartheta_\phi} \bidual^{r}_R H^0 (C).
\end{align*}
Here the first isomorphism is clear and the second is induced by the isomorphism $\big( \exprod^{r_Y}_R Y \big)^\ast \to R$ that evaluates elements on ${\bigwedge}_{i=1}^{i=r_Y}b_i$ and the canonical isomorphism  $\big ( \exprod_R^{n -r} R^{\oplus (n - r)} \big)^\ast \cong   \exprod_R^{n - r} (R^{\oplus (n - r)})^\ast$. In addition, the map $\vartheta_\phi$ is defined by the condition that, for all $a \in \exprod^n_R R^{\oplus n}$ and $f_i \in (R^{\oplus (n - r)})^\ast$, one has  
\[\vartheta_\phi\bigl(a \otimes_R  {\wedge}_{i\in [n-r]}f_i\bigr) = (-1)^{r(n - r)} \cdot ({\wedge}_{i\in [n-r]}(f_i \circ \phi) \big) (a)\in  \exprod^{r}_R R^{\oplus n}. \]
In particular, writing $\{ \tilde b_i\}_{i \in [n- r]}$ for the standard (ordered) basis of $R^{\oplus (n- r)}$, the inclusion $\im(\vartheta_\phi) \subseteq \bidual^{r}_R H^0 (C)$ follows by applying Lemma \ref{biduals lemma 1}\,(b) to the exact sequence
\[ 0 \to H^0 (C) \to R^n \xrightarrow{(\tilde b_i^\ast \circ \phi)_{i}} {\bigoplus}_{i \in [n-r]} R.\]
Here the sequence is induced by the fixed identification $H^0(C) \cong \ker(\phi)$ after noting that $\im(\phi)$ is contained in the direct summand $R^{\oplus (n-r)}$ of $R^{\oplus (n-r)}\oplus Y$.   

To prove that this construction only depends on the pair $(C,b_\bullet)$, and hence complete the proof of (i), we now fix an alternative representative $\widetilde P^\bullet$ for $C$ of the form $R^{\oplus n'} \xrightarrow{\phi'} R^{\oplus (n' - r)} \oplus Y$ and use it to construct a map $\vartheta_{\widetilde P^\bullet,b_\bullet}$ just as above. We can then fix a quasi-isomorphism $\theta^\bullet \: \widetilde P^\bullet \to P^\bullet$ of complexes of $R$-modules with the property that $H^0(\theta^\bullet)$ and $H^1(\theta^\bullet)$ induce the identity maps on $H^0(C)$ and $H^1(C)$. Then, if necessary after replacing $\widetilde P^\bullet$ by its direct sum with a complex $R^{\oplus m}\xrightarrow{{\rm{id}}} R^{\oplus m}$ for a suitable natural number $m$, we can assume that the maps $\theta^0$ and $\theta^1$ are both surjective (without changing $\Det_R(\widetilde P^\bullet)$). Writing $\ker(\theta^\bullet)$ for the complex $\ker(\theta^0)\xrightarrow{\phi}\ker(\theta^1)$, we thereby obtain a short exact sequence of complexes of $R$-modules $0 \to \ker(\theta^\bullet) \to \widetilde P^\bullet \xrightarrow{\theta^\bullet} P^\bullet \to 0$. This sequence implies firstly that $\ker(\theta^\bullet)$ is acyclic and then, by taking determinants, induces an isomorphism of $R$-modules
\[ \iota \: \Det_R (\widetilde P^\bullet) \cong \Det_R (\ker(\theta^\bullet)) \otimes _R \Det_R (P^\bullet) \cong \Det_R (P^\bullet) \]
that is independent of the choice of $\theta^\bullet$. By an easy diagram chase, one then checks that $\vartheta_{\widetilde P^\bullet, b_\bullet} = \vartheta_{P^\bullet, b_\bullet}\circ \iota.$ This  equality implies that the above construction gives a well-defined map $\Det_R (C) \to \bidual^r_R H^0 (C)$ that only depends on $C$ and the choice of basis $b_\bullet$, as required.  

To prove (ii), we note that $C'$ is represented by the complex $R'\otimes_RP^\bullet$. Given this representative, all assertions of (ii) are clear except for the commutative diagram. To construct this diagram, we note Lemma \ref{biduals lemma 1}\,(a) implies the existence of an exact commutative diagram of $R$-modules
\begin{cdiagram}[column sep=small, row sep=small]
    0 \arrow{r} & \bidual^{r}_{R} H^0(C) \arrow[dashed]{d} \arrow{r} & \exprod^{r}_{R} P \arrow{r} \arrow{d} & P \otimes_{R} \exprod^{r - 1}_{R} P \arrow{d} \\ 
    0 \arrow{r} & \bidual^{r}_{R'} H^0(C') \arrow{r} & \exprod^{r}_{R'}P' \arrow{r} & P' \otimes_{R'} \exprod^{r - 1}_{R'} P'.
\end{cdiagram}%
Here we set $P = R^n \cong R^{n-r}\oplus Y$ and $P'\coloneqq  R'\otimes_RP$ and the solid vertical arrows are the natural projection maps. In particular, since the square in this diagram commutes, there exists a dashed arrow that makes the entire diagram commutative. It is easily checked that this dashed arrow is independent of the choice of representative $P^\bullet$ of $C$ and we denote it by ${\bigcap}_R^{r}\pi_{C,R'}$. Given this explicit construction of ${\bigcap}_R^{r}\pi_{C,R'}$, and those of the maps $\vartheta_{P^\bullet,b_\bullet}$ and $\vartheta_{R'\otimes_RP^\bullet,b_\bullet'}$ in (i), it is then a straightforward exercise to check that the diagram given in (ii) commutes, as required. \end{proof}

\begin{remark}\label{basis independence} For data $C$, $b_\bullet$ and $r_Y$ ($= |b_\bullet|$) as in Lemma \ref{det projection lemma}, the map $\vartheta_{C, b_\bullet}$ has the following additional properties.
\begin{romanliste}
    \item For any ordered basis $\tilde b_\bullet$ of $Y$, define a matrix $U = U_{\tilde b_\bullet,b_\bullet}$ in ${\mathrm{GL}}_{r_Y}(R)$ by the condition $\tilde b_\bullet = U \cdot  b_\bullet$. Then ${\bigwedge}_{i\in [r_Y]}\tilde b_i = {\mathrm{det}}(U)\cdot {\bigwedge}_{i\in [r_Y]}b_i$ and hence, by directly comparing the explicit construction of $\vartheta_{C,\tilde b_\bullet}$ and $\vartheta_{C,b_\bullet}$ in Lemma \ref{det projection lemma}\,(i), one verifies that  $\vartheta_{C,\tilde b_\bullet} = {\mathrm{det}}(U)\cdot \vartheta_{C,b_\bullet}$.
\item If both $R$ is reduced and $\bm{\chi}_R(C) = 0$, then an alternative description of the map $\vartheta_{C,b_\bullet}$ is presented in \cite[Prop.\@ A.11]{sbA}. 
\end{romanliste}
\end{remark}

\section{Selmer structures, complexes, and modules}\label{ss, m and c}

\subsection{Galois cohomology}

In this preliminary section we recall the definition, and basic properties, of some relevant Galois cohomology complexes.

\subsubsection{Definitions and conventions}

Let $R$ be a local complete Gorenstein ring with finite residue field of characteristic $p$ as in condition (\ref{ring condition}). (The constructions in this section can more generally be made for a ring that satisfies the weaker condition ($\ast$) in \cite[\S\,1.4]{fukaya-kato} but since it is sufficient for our purposes to work under assumption (\ref{ring condition}), we shall do so in order to streamline exposition.)\\
Let $( \a_n)_{n \in \N}$ be the descending filtration of finite-index ideals of $R$ from (\ref{filtration}), which has the property that the natural map $R \to \varprojlim_{n \in \N} ( R / \a_n)$ is an isomorphism. 
The discrete topology on the finite rings $R / \a_n$ then naturally induces a topology on $R$. More generally, for any finitely presented $R$-module $M$ we have an isomorphism $M \to \varprojlim_{n \in \N} ( M / \a_n M)$ that allows us to endow $M$ with the inverse limit topology induced by the discrete topology on the finite modules $M / \a_n M$. \\
Let $G$ be a topological group and let $M$ be an $R [G]$-module. 
Following Nekov\'a\v{r} \cite[Def.\@ 3.3.4]{SelmerComplexes} we call an $R[G]$-module $M$ `ind-admissible' if it is equal to the union $\bigcup_{N \in \mathscr{S} (M)} N$ 
with $\mathscr{S} (M)$ the set of all $R [G]$-submodules $N \subseteq M$ which are finitely presented as an $R$-module and on which the action of $G$ is continuous.
For every ind-admissible $R[G]$-module $M$, we then define the (inhomogenous) continuous cochains of degree $i \geq 0$ as
\[
\mathscr{C}^i (G, M) \coloneqq {\varinjlim}_{N \in \mathscr{S} (M)} \mathrm{Maps}_\mathrm{cont} (G^{\oplus i}, N),
\]
and write $\mathscr{C}^\bullet (G, M)$ for the associated complex of (inhomogenous) continuous cochains (see, for example, \cite[Ch.\@ I, \S\,2]{NSW} for the definition of the differential for this complex). We write $\RGamma (G, M)$ for the object of $D (R)$ defined by the complex $\mathscr{C}^\bullet (G, M)$.\\
If $F$ denotes a perfect field with algebraic closure $F^\mathrm{c}$ and absolute Galois group $G_F \coloneqq \gal{F^\mathrm{c}}{F}$, then we set 
\[ \RGamma (F, M) \coloneqq \RGamma (G_F, M).
\]
We also fix a number field $k$ and an algebraic closure $k^\mathrm{c}$ of $k$. If $F$ is an extension of $k$ in $k^\mathrm{c}$, we write $\Pi_F$ for the set of all places of $F$ and $\Pi_F^\infty$ and $\Pi_F^\ell$ for each rational prime $\ell$ for the subsets of $\Pi_F$ comprising places that are respectively archimedean and $\ell$-adic. We set
 $\Pi_F^\RR\coloneqq \{\q \in \Pi_F^\infty: F_\q = \RR\}$ and  $\Pi_F^\CC\coloneqq \{\q \in \Pi_F^\infty: F_\q = \CC\}$ (so $\Pi_F^\infty = \Pi_F^\RR\cup \Pi_F^\CC$), and will often work under the following additional assumption on the pair $(k, p)$.
 \begin{equation} \label{p=2 condition}
    \text{If $p = 2$, then $\Pi_k^\R = \emptyset$.}
\end{equation}
Denote by $S_\ram (M) \subseteq \Pi_F$ the subset of all places for which (a choice of) the inertia subgroup acts non-trivially on the $G_F$-module $M$, and set
\[
S (M) \coloneqq \Pi_F^\infty \cup \Pi_F^p \cup S_\ram (M).
\]
If $S \subseteq \Pi_F$ is a subset that contains $S_\ram (M)$, then $M$ is naturally acted upon by $G_{F, S} \coloneqq \gal{F_S}{F}$ with $F_S$ the maximal extension of $F$ unramified outside $S_F$, and we define 
\[
\RGamma (\cO_{F, S} , M) \coloneqq \RGamma (G_{F, S}, M).
\]
This notation is motivated by the fact that $\RGamma (F, M)$ and $\RGamma (\cO_{F, S} , M)$ coincide with the \'etale cohomology complexes $\RGamma ( (\Spec F)_\et, M)$ and $\RGamma ( (\Spec \cO_{F, S})_\et, M)$ if $M$ defines an \'etale sheaf on the \'etale site of $\Spec F$ (resp.\@ of $\Spec \cO_{F, S}$).\\
In each degree $i$, we set $H^i (\Lambda, M) \coloneqq H^i ( \RGamma (\Lambda, M))$ if $\Lambda$ denotes either $F$ or $\cO_{F, S}$.\\
We furthermore define
\[
\mathscr{C}^\bullet_\mathrm{c} (G_{F, S}, M) \coloneqq \mathrm{cone} \big ( \mathscr{C}^\bullet (G_{F, S}, M)
\xrightarrow{\iota} {\bigoplus}_{v \in S} \mathscr{C}^\bullet (G_{F_v}, M) \big ) [-1],
\]
where $\iota$ denotes the natural restriction map, 
and write $\RGamma_\mathrm{c} (\cO_{F, S}, M)$ for the corresponding object of $D (R)$ (which coincides with the complex of compact-support \'etale cohomology if $M$ defines an \'etale sheaf on $\Spec \cO_{F, S}$). 
In particular, in $D(R)$ one has an exact triangle 
\begin{equation}\label{compact def} 
\RGamma_\mathrm{c} (\mathcal{O}_{F, S},M) \to \RGamma(\mathcal{O}_{F, S},M) \xrightarrow{ \iota} \textstyle \bigoplus_{\q }\RGamma(F_\q,M)\to .\end{equation}
In each degree $i$, we then set 
\[ H^i_\mathrm{c} (\cO_{F, S}, M) \coloneqq H^i ( \RGamma (\cO_{F, S}, M)).\]
For a place $\q \in \Pi_F\setminus \Pi_F^\infty$,  we write $\kappa_\q$ for its residue field and define a complex in $D(R)$ 
\begin{equation*} \RGamma_f(F_\q,M)\coloneqq \RGamma(\kappa_\q, M^{\Gal(F_\q^c/F_\q^{\rm{nr}})}) ,\end{equation*}
where $F_\q^{\rm{nr}}$ denotes the maximal unramified extension of $F_\q$ in $F_\q^\mathrm{c}$. In each degree $i$, we set 
\begin{equation}\label{RGamma_f def}H^i_f(F_\q,M)\coloneqq H^i(\mathrm{R}\Gamma_f(F_\q,M)) = \begin{cases} H^0(F_\q,M), &\text{if $i=0$,}\\
\ker(H^1(F_\q,M) \to  H^1(F_\q^{\rm{nr}}, M)), &\text{if $i=1$},\\
                                                                          (0), &\text{if $i \in \{0,1\}$}.\end{cases}\end{equation}

We also note the existence of a natural composite `inflation' morphism in $D(R)$ 
\[ \iota_{M,\q} \: \RGamma_f(F_\q,M) \to  \RGamma(F_\q, M^{\Gal(F_\q^\mathrm{c}/F_\q^{\rm{nr}})})\to \RGamma(F_\q, M).\]

\subsubsection{Some basic properties}

The case that $A$ is a projective $R$-module is of particular interest to us and we make much use of the following general result.

For a finitely generated projective $R$-module $P$ we write $[P]_R$ for its class in $K_0(R)$. We also fix an injective hull $E_R (\mathbb{k})$ of $\mathbb{k}$ and write $(- )^\vee \coloneqq \Hom_R (- , E_R (\mathbb{k}))$ for the Matlis dual functor. 

\begin{lemma}\label{flach result} 
Assume $R$ satisfies condition (\ref{ring condition}), let $F$ be a number field, and let $A$ be a finite-rank free $R$-module endowed with a continuous action of $G_F$ such that $S(A)$ is finite. 
If $S \subseteq \Pi_F$ is a finite set that contains $S (A)$, then the following claims are valid. 
\begin{romanliste}
\item $\mathrm{R}\Gamma_\mathrm{c} (\cO_{F, S}, A)$ belongs to $D^{\mathrm{perf},0}_{[1,3]}(R)$. 
\item If $\q \in \Pi_F\setminus (\Pi_F^\infty\cup \Pi_F^p)$, then $\mathrm{R}\Gamma(F_\q, A)$ belongs to $D^{\mathrm{perf},0}_{[0,2]}(R)$.
\item If $\q \in \Pi_F^p$, then $\mathrm{R}\Gamma(F_\q, A)$ belongs to $D^{\mathrm{perf}}_{[0,2]}(R)$ and $\bm{\chi}_R\bigl(\mathrm{R}\Gamma(F_\q, A)\bigr) = -[F_\q:\QQ_p]\cdot[A]_R$. 
\item If $\q \in \Pi_F\setminus S(A)$, then $\mathrm{R}\Gamma_f(F_\q, A)$ belongs to $D^{\mathrm{perf},0}_{[0,1]}(R)$.
\item If (\ref{p=2 condition}) is valid, then $\mathrm{R}\Gamma(\cO_{F, S}, A)$ belongs to $D^{\mathrm{perf}}_{[0,2]}(R)$ and $\mathrm{R}\Gamma(F_\q, A)$  belongs to $D^{\mathrm{perf}}_{[0,0]}(R)$ for each $\q\in \Pi_F^\infty$.
\end{romanliste}
In the remainder of the result, we let $\Phi$ denote any one of the functors that sends $A$ to $\mathrm{R}\Gamma_\mathrm{c} (\cO_{F, S}, A)$, to $\mathrm{R}\Gamma(F_\q, A)$ for $\q\in \Pi_F\setminus \Pi_F^\infty$, to $\mathrm{R}\Gamma_f(F_\q, A)$ for $\q\in \Pi_F\setminus S(A)$, or, if (\ref{p=2 condition}) is valid, either to $\mathrm{R}\Gamma(\cO_{F, S}, A)$ or to $\mathrm{R}\Gamma(F_\q, A)$ for $\q\in \Pi_F^\infty$.

\begin{itemize}
\item[(vi)] For any morphism $R\to R'$ of rings satisfying (\ref{ring condition}), there exists a natural isomorphism $\Phi(A) \otimes_R^\mathbb{L} R' \cong \Phi(A\otimes_R R')$ in $D(R')$. 
\item[(vii)] For any morphism $R\to R'$ of finite rings satisfying (\ref{ring condition}), there exists a natural isomorphism $\RHom_R (R',\Phi(A^\vee))  \cong \Phi(\Hom_R (R', A^\vee))$ in $D(R')$.
\item[(viii)] Lemma \ref{limit result} gives rise to a well-defined object $\varprojlim_{n\in \N}\Phi(A/\a_nA)$ of $D(R)$ for which there exists a natural isomorphism $\Phi(A) \to \varprojlim_{n\in \N}\Phi(A/\a_nA)$ in $D(R)$. 
\end{itemize}
\end{lemma}

\begin{proof} 
Claims (i) -- (vi) follow directly from the general results established by Flach in \cite[\S\,4 and \S\,5]{Flach00} and by Fukaya and Kato in \cite[Prop.\@ 1.6.5]{fukaya-kato}.\\
We next prove claim (vii).
Under the conditions of (vii), the functor $\Phi$ satisfies the assumptions of \cite[Prop.\@ 3.1]{Flach00} with the algebra $\mathscr{B}$ in loc.\@ cit.\@ taken to be $\Z_p [\Aut_{\Z_p} (A) ]$. It follows that there exists a bounded complex $P^\bullet$ of finitely generated projective $\mathscr{B}$-modules such that $\Phi(A) \cong \RHom_{\mathscr{B}} (P^\bullet, A)$ in $D(R)$. By construction, $\RHom_{\mathscr{B}} (P^\bullet, A)$ is a complex of finitely generated projective (hence also injective because $R$ is self-injective) $R$-modules and so by definition $\RHom_R (R', \Phi(A)) = \RHom_R (R',\RHom_{\mathscr{B}} (P^\bullet, A))$ in $D(R')$. To prove (vii), it is therefore enough to construct an isomorphism in $D(R')$ 
\begin{equation} \label{exchanging homs iso}
\RHom_R (R',\RHom_{\mathscr{B}} (P^\bullet, A)) \to \RHom_\mathscr{B} ( P^\bullet, \Hom_R (R', A)).
\end{equation}
For this, we note that for any finitely generated projective $\mathscr{B}$-module $P$, there exists a natural isomorphism of $R'$-modules 
\[
\Hom_R (R', \Hom_\mathscr{B} (P, A)) \xrightarrow{\simeq} \Hom_\mathscr{B} ( P, \Hom_R (R', A)),
\quad f \mapsto ( a \mapsto (b \mapsto f (b)(a))).
\]
The map (\ref{exchanging homs iso}) is then induced by these isomorphisms with $P$ taken to be each individual term of the complex $P^\bullet$.\\
Finally, to prove (viii), we note that the family of complexes $(\Phi(A/\a_nA))_n$ satisfies the conditions (a) and (b) of Lemma~\ref{limit result}. Condition (a) is satisfied as a consequence of the respective result of (i) -- (v) and the existence of the required isomorphisms in  (b) follows from (vi) with $(R\to R',A)$ taken to be the pairs $(R/\a_{n+1} \to R/\a_n, A/\a_{n+1}A)$. The construction of Lemma~\ref{limit result} therefore gives a complex $\varprojlim_{n\in \N}\Phi(A/\a_nA)$ in $D(R)$. In addition, since $\Phi(A)$ belongs to $D^{\mathrm{perf}}(R)$, we may fix a bounded complex $P^\bullet$ of finitely generated projective $R$-modules that is isomorphic to $\Phi(A)$ in $D(R)$. Then, for each $n$, the isomorphism in (vi) applied to $R\to R/\a_n$ implies that $(R/\a_n)\otimes_RP^\bullet$ is isomorphic to $\Phi(A/\a_nA)$ in $D(R/\a_n)$. One therefore obtains a composite isomorphism in $D(R)$ 
\[ \Phi(A) \cong P^\bullet \cong {\varprojlim}_{n\in \N} ((R/\a_n)\otimes_RP^\bullet) \cong {\varprojlim}_{n\in \N} \Phi(A/\a_nA),\]
in which the second isomorphism results from the fact each $R$-module $P^i$ is finitely generated projective and the third is a consequence of the uniqueness assertion in Lemma \ref{limit result}. 
\end{proof} 

To recall the relevant duality theorems for Galois cohomology, we write $\mathrm{R}\Gamma_{\rm{Tate}}(F_\q,A)$ for the standard complex computing Tate cohomology of $A$ over $G_{F_\q}$ for $\q \in \Pi_F^\RR$,
and then define an object $\widetilde{\RGamma}_\mathrm{c}  (\cO_{F, S}, A)$ of $D(R)$ by the triangle
\[
\widetilde{\RGamma}_\mathrm{c}  (\cO_{F, S}, A) \to \RGamma (\cO_{F, S}, A) \xrightarrow{\lambda_S (A)} {\bigoplus}_{\q \in S} \RGamma_\mathrm{Tate} (F_\q, A) \to \cdot
\]
with $\lambda_S (A)$ the natural localisation morphism.

\begin{proposition}\label{local Tate duality thm}
If $R$ satisfies (\ref{ring condition}) and $A$ is a finitely generated free $R$-module, then the following claims are valid.  
\begin{itemize}
\item[(i)] For each $\q \in \Pi_F\setminus \Pi_F^\infty$, there are canonical isomorphisms in $D(R)$
\begin{equation*}\label{R linear local duality}
    \RHom_R (\RGamma (F_\q, A^\ast (1)), R) [-2] \cong \RGamma (F_\q, A) \cong \RHom_R ( \RGamma (F_\q, A^\vee (1), R)), E_R (\mathbb{k})) [-2].\end{equation*}
    \item[(ii)] There exist canonical isomorphisms in $D(R)$ 
\[{\hskip-0.3truein \RHom_R ( \widetilde{\RGamma}_\mathrm{c}  (\cO_{F, S}, A^\ast (1)), R) [-3]\cong \RGamma (\cO_{F, S}, A) \cong 
 \RHom_R ( \widetilde{\RGamma}_\mathrm{c}  (\cO_{F, S}, A^\vee (1)), E_R (\mathbb{k})) [-3].}\]
\end{itemize}
\end{proposition}

\begin{proof}
    Since $R$ is assumed to be Gorenstein, it is its own dualising module (cf.\@ \cite[Th.\@ 3.3.7\,(a)]{BrunsHerzog}). Given this fact, the isomorphisms in (i) and (ii) are respectively proved in \cite[Th.~5.2.6 and \S\,5.7.5]{SelmerComplexes} (see also \cite[\S\,5.1, Lem.\@ 12\,b)]{BurnsFlach01} for the first isomorphism in (ii)).
\end{proof}

\subsection{Selmer structures of Nekov\'a\v{r} and of Mazur--Rubin}

Throughout this section, we fix a ring $R$ satisfying (\ref{ring condition}) and  an ind-admissible $R[G_k]$-module $A$ for which $S(A)$ is finite. 

\subsubsection{Nekov\'{a}\v{r}--Selmer structures}

\begin{definition}
A `Nekov\'{a}\v{r}(--Selmer) structure' $\mathscr{F}$ on $A$ comprises  
\begin{itemize}
    \item a finite set $S (\mathscr{F})$ of places of $k$ that contains $S(A)$, and 
    \item for each $\q\in \Pi_k$, a morphism in $D(R)$ 
    \[ \mathrm{R}\Gamma_{\mathscr{F}}(k_\q,A)\xrightarrow{\theta_{\mathscr{F},\q}} \mathrm{R}\Gamma(k_\q,A),\] 
    with 
    \[ \bigl(\RGamma_{\mathscr{F}}(k_\q,A), \theta_{\mathscr{F},\q}\bigr) \coloneqq \bigl(\mathrm{R}\Gamma_f(k_\q,A),\iota_{A,\q}\bigr) \,\,\,\text{ for all }\,\,\,\q \in \Pi_k\setminus S(\mathscr{F}).\] 
We write $\RGamma_{\!/ \mathscr{F}} (k_\q, A)$ for the mapping cone of $\theta_{\mathscr{F}, \q}$ and in each degree $i$ set 
    \[ H^i_\mathscr{F}(k_\q,A) \coloneqq H^i\bigl(\RGamma_\mathscr{F}(k_\q,A)\bigr)\,\,\,\text{and}\,\,\, H^i_{\!/\mathscr{F}}(k_\q,A) \coloneqq H^i\bigl(\RGamma_{\!/\mathscr{F}}(k_\q,A)\bigr).\]
\end{itemize}
\end{definition}

\begin{definition}\label{selmer complex def} The `Selmer complex' associated to a Nekov\'{a}\v{r} structure $\mathscr{F}$ on $A$ is the object $\RGamma_\mathscr{F}(k,A)$  of $D(R)$ that is defined (up to isomorphism) via the existence of an exact triangle   
\begin{equation}\label{mapping fibre} \RGamma_\mathscr{F}(k,A) \to \RGamma(\mathcal{O}_{k,S(\mathscr{F})},A) \oplus {\bigoplus}_{\q \in S(\mathscr{F})}\RGamma_\mathscr{F}(k_\q,A) \xrightarrow{ (\iota_{S (\mathscr{F})},(\theta_{\mathscr{F},\q})_\q)}{\bigoplus}_{\q \in S(\mathscr{F})}\RGamma(k_\q,A)\to \cdot\end{equation}
in which $\iota_\mathscr{F}$ is the diagonal localisation map. In each degree $i$ we set 
\[ H^i_\mathscr{F}(k,A) \coloneqq H^i\bigl(\RGamma_\mathscr{F}(k,A)\bigr).\] 
We also say that $\mathscr{F}$ is `perfect' if $\RGamma_\mathscr{F}(k,A)$ belongs to $D^{\mathrm{perf}}(R)$. 
\end{definition}

\begin{rk}
    The definition of $\RGamma_\mathscr{F}(k,A)$ given in Definition \ref{selmer complex def} specifies it only up to non-unique isomorphism in $D(R)$. However, in all examples that are relevant to this article, the structure $\mathscr{F}$ is given by data on the level of complexes so that the associated Selmer complex can be explicitly defined via a mapping cone construction. 
\end{rk}

The Octahedral axiom implies that the exact triangle (\ref{mapping fibre}) is equivalent to an exact triangle 
\begin{equation}\label{mapping fibre2} \RGamma_\mathscr{F}(k,A) \to \RGamma(\mathcal{O}_{k,S(\mathscr{F})},A) \xrightarrow{ \iota'_{S (\mathscr{F})}} {\bigoplus}_{\q \in S(\mathscr{F})}\RGamma_{\!/\mathscr{F}}(k_\q,A)\to \cdot\end{equation}
in $D(R)$ in which the morphism $\iota'_{S(\mathscr{F})}$ is induced by $\iota_{S (\mathscr{F})}$. The following notion allows us to extend this observation. 

\begin{definition} Let $\mathscr{F}$ and $\mathscr{F}'$ be Nekov\'{a}\v{r} structures on $A$. Then we say that $\mathscr{F}'$ `is a refinement of' (or, more simply, `refines') $\mathscr{F}$, to be written as $\mathscr{F}'\le \mathscr{F}$, if for each $\q \in \Pi_k$ there exists a specified morphism 
\[ j_{\mathscr{F}',\mathscr{F},\q} \: \mathrm{R}\Gamma_{\mathscr{F}'}(k_\q,A) \to \mathrm{R}\Gamma_\mathscr{F}(k_\q,A)\]
in $D(R)$ with the following properties:  $\theta_{\mathscr{F}',\q} =  \theta_{\mathscr{F},\q}\circ j_{\mathscr{F}',\mathscr{F},\q}$ and, if $\q \notin S(\mathscr{F}')\cup S(\mathscr{F})$, then $j_{\mathscr{F}',\mathscr{F},\q} $ is the identity morphism $\mathrm{R}\Gamma_f(k_\q,A) \to \mathrm{R}\Gamma_f(k_\q,A)$. If $\mathscr{F}'$ is a refinement of $\mathscr{F}$, then for each $\q \in \Pi_k$ we write $\RGamma_{\mathscr{F}/\mathscr{F}'}(k_\q,A)$ for the mapping cone of $j_{\mathscr{F}',\mathscr{F},\q}$.
\end{definition}

\begin{lemma}\label{basic properties} The following claims are valid for any Nekov\'{a}\v{r} structure $\mathscr{F}$ on $A$.
\begin{romanliste}
\item If $S$ is any finite subset of $\Pi_k$ with $S(\mathscr{F})\subseteq S$, then there exists an exact triangle in $D(R)$  
\begin{equation*} \RGamma_\mathscr{F}(k,A) \to \RGamma(\mathcal{O}_{k,S},A) \oplus {\bigoplus}_{\q \in S}\RGamma_\mathscr{F}(k_\q,A) \xrightarrow{ (\iota_S,(\theta_{\mathscr{F},\q})_\q)}{\bigoplus}_{\q \in S}\RGamma(k_\q,A)\to\cdot \end{equation*}
in which $\iota_S$ is the diagonal localisation map

\item Let $\mathscr{F}'$ be a refinement of $\mathscr{F}$. Then there exists a canonical exact triangle in $D(R)$ 
\[ \RGamma_{\mathscr{F}'}(k,A) \to \RGamma_\mathscr{F}(k,A) \to {\bigoplus}_{\q \in S(\mathscr{F}')\cup S(\mathscr{F})}\RGamma_{\mathscr{F}/\mathscr{F}'}(k_\q,A)\to \cdot\]
\item If $A$ is a finite-rank free $R$-module, then $\mathscr{F}$ is perfect if and only if, for every $\q$ in $S(\mathscr{F})$, the complex  $\RGamma_\mathscr{F}(k_\q,A)$ belongs to $D^{\mathrm{perf}}(R)$.  If this is the case, then in $K_0(R)$ one has 
\[ \bm{\chi}_R\bigl(\RGamma_{\mathscr{F}}(k,A)\bigr) = {\sum}_{\q \in S(\mathscr{F})}\bm{\chi}_R\bigl(\RGamma_\mathscr{F}(k_\q,A)\bigr).\]

\item If either $i< 0$ or $i > 3$, then $H^i_\mathscr{F}(k,A) = {\bigoplus}_{\q\in S(\mathscr{F})}H^i_\mathscr{F}(k_\q,A)$. The same equality is valid for $i =3$ if $A$ is a finitely generated free $R$-module and there exists a place $\q_0 \in S(\mathscr{F})\setminus \Pi_k^\RR$ for which $H^2(\theta_{\mathscr{F},\q_0})$ is bijective.
\end{romanliste}
\end{lemma}

\begin{proof}For a finite subset $\Sigma$ of $\Pi_k$ we write $C_{\Sigma}(A), C^f_{\Sigma}(A)$ and $C^\mathscr{F}_{\Sigma}(A)$ for the respective direct sums over $\q \in \Sigma$ of the complexes $\mathrm{R}\Gamma(k_\fq,A), \mathrm{R}\Gamma_f(k_\fq,A)$ and $\mathrm{R}\Gamma_\mathscr{F}(k_\fq,A)$. We then recall that, for any finite subsets $S'$ and $S$ with $S(A)\subseteq S'\subseteq S$, there exists a canonical `inflation-restriction' exact triangle in $D(R)$  
\begin{equation}\label{inf res triangle} \RGamma(\mathcal{O}_{k,S'},A) \xrightarrow{\iota_{S',S}(A)} \RGamma(\mathcal{O}_{k,S},A)\oplus C^f_{S\setminus S'}(A)  \xrightarrow{(\iota_{S',S}', (\iota_{\q,A})_\q)} C_{S\setminus S'}(A) \to \end{equation}
in which $\iota_{S',S}'$ is the canonical localisation morphism.  

To prove (i) we assume $S' = S(\mathscr{F})$ and consider the following diagram in $D(R)$ 
\[\xymatrix{ \RGamma(\mathcal{O}_{k,S'},A)\oplus C_{S'}^\mathscr{F}(A) \ar[r] \ar[d]^{(\iota_\mathscr{F},(\theta_{\mathscr{F},\q})_\q)}& \bigl( \RGamma(\mathcal{O}_{k,S},A)\oplus C^f_{S\setminus S'}(A)\bigr) \oplus C_{S'}^\mathscr{F}(A) \ar[r]   \ar[d]^{((\iota_S,(\iota_{\q,A})_\q),(\theta_{\mathscr{F},\q})_\q)} & C_{S\setminus S'}(A)\ar@{=}[d] \ar[r] &.\\ C_{S'}(A) \ar[r] & C_S(A) \ar[r] & C_{S\setminus S'}(A) \ar[r] &.}
\]
Here the upper row is the exact triangle induced by (\ref{inf res triangle}) and the lower row is the canonical exact triangle. In addition, because $C^f_{S\setminus S'}(A) = C^\mathscr{F}_{S\setminus S'}(A)$ (since $S' = S(\mathscr{F})$), the central vertical map agrees with $(\iota_S,(\theta_{\mathscr{F},\q}))_\q)$ where $\q$ runs over $S$, and the diagram commutes. The Octahedral axiom therefore implies that the mapping fibres of the first and second vertical morphisms are canonically isomorphic in $D(R)$ and this implies (i). 

To prove (ii) we set $S \coloneqq S(\mathscr{F}')\cup S(\mathscr{F})$ and consider the following diagram in $D(R)$
\[\xymatrix{ \RGamma_{\mathscr{F}'}(k,A) \ar[r]  & \RGamma(\mathcal{O}_{k,S},A) \oplus C_{S}^{\mathscr{F}'}(A) 
\ar[rr]^{\hskip 0.5truein(\iota_S,(\theta_{\mathscr{F}',\q})_\q)}   \ar[d]^{({\id},(j_{\mathscr{F}',\mathscr{F},\q})_\q)} & & C_{S}(A) \ar@{=}[d] \ar[r] &.\\
\RGamma_{\mathscr{F}}(k,A) \ar[r] & \RGamma(\mathcal{O}_{k,S},A) \oplus C_{S}^{\mathscr{F}}(A) \ar[rr]^{\hskip 0.5truein(\iota_S,(\theta_{\mathscr{F},\q})_\q)}  & & C_{S}(A)\ar[r] &.}
\]
Here the two rows are the exact triangles obtained from (i) and the square commutes by choice of the morphisms $j_{\mathscr{F}',\mathscr{F},\q}$. In particular, the Octahedral axiom implies both that the diagram can be completed to give a morphism of exact triangles and hence that there exists an exact triangle as stated in (ii). 

We note next that the exact triangle (\ref{compact def}) defining $\mathrm{R}\Gamma_c(\mathcal{O}_{k,S},A)$ combines with the exact triangle in (i) (with $S = S(\mathscr{F})$) to imply, via the Octahedral axiom, the existence of an exact triangle in $D(R)$
\begin{equation}\label{cs versus Selmer}\RGamma_\mathrm{c} (\mathcal{O}_{k,S},A) \to \RGamma_{\mathscr{F}}(k,A) \to {\bigoplus}_{\q \in S}\RGamma_\mathscr{F}(k_\q,A)\to .\end{equation}
Given this triangle,  both (iii) and the first assertion of (iv) follow directly from the result of Lemma \ref{flach result}\,(i). To prove the second assertion of (iv) we note that the long exact cohomology sequence of the last displayed triangle gives an exact sequence 
\[ {\bigoplus}_{\q \in S}H^2_\mathscr{F}(k_\q,A) \xrightarrow{\gamma} H^3_c(\mathcal{O}_{k,S},A) \to H^3_\mathscr{F}(k,A) \to {\bigoplus}_{\q\in S}H^3_\mathscr{F}(k_\q,A)\to H^4_c(\mathcal{O}_{k,S},A)= (0).\] 
It is therefore enough to show that the stated hypothesis implies $\gamma$ is surjective. This is true, since the compatibility of local and global duality (as respectively expressed by the second isomorphisms in Proposition \ref{local Tate duality thm}\, (i) and (ii)) identifies the composite map 
\[ H^2(k_{\q_0},A)\xrightarrow{H^2(\theta_{\mathscr{F},\q_0})^{-1}} H^2_\mathscr{F}(k_{\q_0},A) \subseteq {\bigoplus}_{\q \in S}H^2_\mathscr{F}(k_\q,A) \xrightarrow{\gamma} H^3_c(\mathcal{O}_{k,S},A)\]
with the Matlis-dual of the inclusion map $H^0(\mathcal{O}_{k,S},A^\vee(1)) \to H^0(k_{\q_0}, A^\vee(1))$. In particular, since the latter map is injective, its Matlis duals is surjective, as claimed.  
\end{proof} 

\begin{example}\label{nss} (Greenberg--Nekov\'{a}\v{r} structures) Fix a finite subset $S$ of $\Pi_k$ with $S(A)\subseteq S$ and assume to be given, for each $\q \in S$, an ind-admissible $R[G_{k_\q}]$-module
$A_\q$ for which there exists a morphism $j_\q \: A_\q \to A$ of $R [G_{k_\q}]$-modules. Then one obtains a Nekov\'{a}\v{r} structure $\mathscr{F} = \mathscr{F}( (A_\q,j_\q)_{\q \in S})$ with $S(\mathscr{F}) \coloneqq S$ by setting 
\[ \bigl(\mathrm{R}\Gamma_{\mathscr{F}}(k_\q,A), \theta_{\cF, \q} \bigr) \coloneqq \bigl(\mathrm{R}\Gamma(k_\q,A_\q), \mathrm{R}\Gamma (k_\q, j_\q)\bigr)\,\,\,\text{for all} \,\,\,\q \in S.\]
This approach is discussed more fully in  \cite[\S\,7.8]{SelmerComplexes}, and the following  examples will play a key role in subsequent arguments.

\begin{romanliste}

\item The `relaxed' Nekov\'{a}\v{r} structure $\mathscr{F}_{\mathrm{rel}} = \mathscr{F}_{\mathrm{rel}}(A,S)$ for $A$ and $S$. One has $S(\mathscr{F}_{\mathrm{rel}}) = S$ and, for each $\q\in  S$, sets 
\[ \bigl( \RGamma_{\mathscr{F}_{\mathrm{rel}}}(k_\q,A), \theta_{\mathscr{F}_{\mathrm{rel}},\q}\bigr)\coloneqq \bigl( \RGamma(k_\q,A), \mathrm{id}\bigr).\]
Hence $\mathrm{R}\Gamma_{\mathscr{F}_{\mathrm{rel}}}(k,A) = \RGamma(\mathcal{O}_{k,S},A)$ and so Lemma \ref{flach result}\,(v) implies that $\mathscr{F}_{\mathrm{rel}}$ is perfect if $A$ is a free $R$-module of finite rank and (\ref{p=2 condition}) is satisfied. 
\item The `canonical' Nekov\'{a}\v{r} structure $\mathscr{F}_{\rm{can}} = \mathscr{F}_{\rm{can}}(A,S)$ for $A$ and $S$ refines $\mathscr{F}_{\mathrm{rel}}(A,S)$. One has $S(\mathscr{F}_{\rm{can}}) = S$ and for  $\q \in S$ sets 
\[\bigl( \RGamma_{\mathscr{F}_{\mathrm{can}}}(k_\q,A), \theta_{\mathscr{F}_{\mathrm{can}},\q}\bigr) \coloneqq \begin{cases} \bigl(\mathrm{R}\Gamma_f(k_\q,A), \iota_{A,\q}\bigr), &\text{if $\q\in S\setminus (\Pi_k^\infty \cup \Pi_k^p)$}\\
\bigl(\RGamma(k_\q,A), {\mathrm{id}}\bigr),&\text{if $\q\in \Pi_k^\infty \cup \Pi_k^p$}.\end{cases}\]
In particular, if $A$ is a free $R$-module of finite rank, then Lemma \ref{basic properties}\,(iii) combines with Lemma \ref{flach result}\,(ii) and (iv) to imply $\mathscr{F}_{\rm{can}}$ is perfect if and only if, for each $\q \in \Pi_k^\infty$, the $G_{k_\q}$-module $A$ is cohomologically-trivial and the $R$-module $H^0(k_\q,A)$ is projective and, for each $\q\in S(A) \setminus (\Pi_k^\infty \cup \Pi_k^p)$, the complex $\mathrm{R}\Gamma_f(k_\q,A)$ belongs to $D^\mathrm{perf}(R)$.
\item The `strict' Nekov\'{a}\v{r} structure $\mathscr{F}_{\mathrm{str}} = \mathscr{F}_{\!\mathrm{str}}(A,S)$ for $A$ and $S$ refines $\mathscr{F}_{\mathrm{can}}(A,S)$, and hence also $\mathscr{F}_{\!\mathrm{rel}}(A,S)$. One has $S(\mathscr{F}_{\mathrm{str}}) = S$ and for  $\q \in S$ sets 
\[ \bigl( \RGamma_{\!\mathscr{F}_{\mathrm{str}}}(k_\q,A), \theta_{\!\mathscr{F}_{\mathrm{str}},\q}\bigr)\coloneqq \bigl(0,0\bigr).\]
One therefore has $\mathrm{R}\Gamma_{\mathscr{F}_{\mathrm{str}}}(k,A) =  \RGamma_\mathrm{c} (\mathcal{O}_{k,S},A)$ and so Lemma \ref{flach result}\,(i) implies $\mathscr{F}_{\mathrm{str}}$ is perfect if $A$ is a free $R$-module of finite rank. 
\end{romanliste}
\end{example}

 The approach of Example \ref{nss} does not encompass all Nekov\'{a}\v{r} structures  of arithmetic interest. For instance, if $S$ is a finite subset of $\Pi_k$ with $S(A)\subseteq S$, then the specification for each $\q \in S$ of a projective $R$-submodule $X_\q$ of $H^1(k_\q,A)$ determines a perfect Nekov\'{a}\v{r} structure $\mathscr{F} = \mathscr{F}((X_\q)_\q)$ with $S(\mathscr{F}) = S$ and, for each $\q \in S$, $\RGamma_\mathscr{F}(k_\q,A) = X_\q[-1]$ and $\theta_{\mathscr{F},\q}$ the unique morphism $X_\q[-1]\to \RGamma(k_\q,A)$ in $D(R)$ for which $H^1(\theta_{\mathscr{F},\q})$ is the inclusion $X_\q \subseteq H^1(k_\q,A)$. In particular, this construction incorporates the `perfect Selmer structures' introduced in \cite[\S\,2]{bmc} in order to study refined versions of the Birch and Swinnerton-Dyer Conjecture for abelian varieties over number fields. 

There are also  several natural ways in which a given Nekov\'{a}\v{r} structure $\mathscr{F}$ on $A$ gives rise to further structures. In \S\,\ref{duality realtions section} we will discuss natural notions of `dual' Nekov\'{a}\v{r} structure. To end this section, we record several other constructions that are also important for our approach. 

\begin{example}\label{lesc} ($\Sigma$-modifications) 
 Let $\Sigma$ be a finite subset of $\Pi_k$. Then the `$\Sigma$-modification' $\mathscr{F}_\Sigma$ and `$\Sigma$-comodification' $\mathscr{F}^\Sigma$ of $\mathscr{F}$ are the Nekov\'{a}\v{r} structures on $A$ that are specified as follows: $S(\mathscr{F}_\Sigma) = S(\mathscr{F}^\Sigma) \coloneqq S(\mathscr{F}) \cup \Sigma$; for $\q \in S(\mathscr{F})$, one has 
 \[ (\mathrm{R}\Gamma_{\mathscr{F}_\Sigma}(k_\q,A), \theta_{\mathscr{F}_\Sigma,\q}) =  (\mathrm{R}\Gamma_{\mathscr{F}^\Sigma}(k_\q,A), \theta_{\mathscr{F}^\Sigma,\q}) \coloneqq (\mathrm{R}\Gamma_{\mathscr{F}}(k_\q,A),\theta_{\mathscr{F},\q});\]
 for $\q \in \Sigma\setminus S(\mathscr{F})$, one has 
 \[ (\mathrm{R}\Gamma_{\mathscr{F}_\Sigma}(k_\q,A),\theta_{\mathscr{F}_\Sigma,\q}) \coloneqq (0,0) \quad\text{and}\quad (\mathrm{R}\Gamma_{\mathscr{F}^\Sigma}(k_\q,A),\theta_{\mathscr{F}^\Sigma,\q}) \coloneqq (\mathrm{R}\Gamma(k_\q,A),\mathrm{id}).\]
If $\Sigma = \emptyset$, then 
$\mathscr{F}_\Sigma = \mathscr{F}^\Sigma = \mathscr{F}$. In general, $\mathscr{F}_\Sigma$ refines $\mathscr{F}$ and Lemma \ref{basic properties}\,(ii) gives an  exact triangle in $D(R)$   
\begin{equation}\label{sigma triangle} \RGamma_{\mathscr{F}_\Sigma}(k,A) \to \RGamma_\mathscr{F}(k,A) \to {\bigoplus}_{\q \in \Sigma\setminus S(\mathscr{F})}\mathrm{R}\Gamma_f(k_\q,A)\to \cdot\end{equation}
In addition, by comparing the exact triangles in Lemma \ref{basic properties}\,(i) for $\mathscr{F}$ and $\mathscr{F}^\Sigma$ (and with $S$ taken to be $S(\mathscr{F})\cup \Sigma$ in both cases), one finds that local Tate duality (Theorem \ref{local Tate duality thm}\,(i)) for each $\q \in \Sigma\setminus S(\mathscr{F})$ gives rise to an exact triangle in $D(R)$ 
\[ \RGamma_{\mathscr{F}}(k,A) \to \RGamma_{\mathscr{F}^\Sigma}(k,A) \to {\bigoplus}_{\q \in \Sigma\setminus S(\mathscr{F})}\mathrm{R}\Gamma_f(k_\q,A^\vee(1))^\vee[-2]\to \cdot\]
These two displayed exact triangles combine with Lemma \ref{flach result}\,(ii) to imply that $\mathscr{F}$ is perfect if and only if $\mathscr{F}_\Sigma$ and $\mathscr{F}^\Sigma$ are both perfect. As a concrete  example, if $\mathscr{F}$ is the relaxed structure $\mathscr{F}_{\mathrm{rel}} = \mathscr{F}_{\mathrm{rel}}(A,S)$ from Example \ref{nss}\,(i), then one has $ \RGamma_{\mathscr{F}_{\rm{rel}}}(k,A) = \RGamma(\mathcal{O}_{k,S},A)$. In particular, if $\Sigma \cap S = \emptyset$, then the exact triangle (\ref{sigma triangle}) implies that $ \RGamma_{\mathscr{F}_{\rm{rel},\Sigma}}(k,A)$ coincides with the `$\Sigma$-modified \'etale cohomology' complexes $\mathrm{R}\Gamma_{\Sigma}(\cO_{k,S},A)$ used in \cite{bss} and \cite{sbA}. In the case $A = \ZZ_p(1)$ such constructions were first used by Gross \cite{gross88} and Rubin \cite{Rub96} in the context of Stark's conjectures (see also \cite[Ch.\@ IX, \S\,5]{Rubin-euler}).
\end{example}

\begin{example}\label{remark selmer} (Induced structures) \ 

\begin{romanliste}
\item If $\iota \:A' \to A$ is a homomorphism of continuous $R[G_{k}]$-modules, then $\mathscr{F}$ induces a Nekov\'{a}\v{r} structure $\iota^\ast(\mathscr{F})$ on $A'$ as follows. One has $S(\iota^\ast(\mathscr{F})) \coloneqq S(\mathscr{F})$; for $\q\in S(\mathscr{F})$ one defines $\RGamma_{\iota^\ast(\mathscr{F})}(k_\q,A')$ via the exact triangle in $D(R)$ 
\[ \mathrm{R}\Gamma_{\iota^\ast(\mathscr{F})}(k_\q,A') \to \mathrm{R}\Gamma(k_\q,A')\oplus \mathrm{R}\Gamma_\mathscr{F}(k_\q,A) \xrightarrow{(\mathrm{R}\Gamma(k_\q,\iota),
\theta_{\mathscr{F},\q})}  \mathrm{R}\Gamma(k_\q,A) \to ,\]
and takes $\theta_{\iota^\ast(\mathscr{F}),\q}$ to be the morphism induced by the first morphism in this triangle. When the map $\iota$ is clear from context, we will write $\mathscr{F}_{A'}$ in place of $\iota^\ast(\mathscr{F})$.

\item If $j \: A \to A'$ is a homomorphism of continuous $R[G_{k}]$-modules, then $\mathscr{F}$ induces a Nekov\'{a}\v{r} structure $j_\ast(\mathscr{F})$ on $A'$ as follows. One has $S(j_\ast(\mathscr{F})) \coloneqq S(\mathscr{F})$; for $\q\in S(\mathscr{F})$, one defines $\mathrm{R}\Gamma_{j_\ast(\mathscr{F})}(k_\q,A')$ to be $\mathrm{R}\Gamma_{\mathscr{F}}(k_\q,A)$ and $\theta_{j_\ast(\mathscr{F}),\q}$ to be the composite $\mathrm{R}\Gamma(k_\q,j)\circ \theta_{\mathscr{F},\q}$. When the map $\iota$ is clear from context, we will write $\mathscr{F}_{A'}$ in place of $j_\ast(\mathscr{F})$.

\item Fix a morphism of rings $R\to R'$ satisfying (\ref{ring condition}). Assume $A$ is a finitely generated free $R$-module and that condition (\ref{p=2 condition}) is valid. Then, for every place $\q \in \Pi_k$, there exists a natural isomorphism $\RGamma(k_\q, A) \otimes_R^\mathbb{L} R' \cong \RGamma(k_\q, A \otimes_R R')$ in $D(R')$ (cf.\@ Lemma \ref{flach result}\,(vi)) and so one can specify a Nekov\'{a}\v{r} structure $\mathscr{F}\otimes_RR'$ on $A\otimes_RR'$ as follows. One has $S(\mathscr{F}\otimes_RR') \coloneqq S(\mathscr{F})$; for $\q\in S(\mathscr{F})$, one sets 
\[ \mathrm{R}\Gamma_{\mathscr{F}\otimes_RR'}(k_\q,A\otimes_RR') \coloneqq \mathrm{R}\Gamma_{\mathscr{F}}(k_\q,A)\otimes_R^\mathbb{L} R'\] 
and takes $\theta_{\mathscr{F}\otimes_RR',\q}$ to be the following composite morphism in $D(R')$ 
\[ \mathrm{R}\Gamma_{\mathscr{F}\otimes_RR'}(k_\q,A\otimes_RR') \xrightarrow{\theta_{\mathscr{F},\q}\otimes_R^\mathbb{L} R'}\RGamma(k_\q, A) \otimes_R^\mathbb{L} R' \cong \RGamma(k_\q, A \otimes_R R').\]
\item Assume the same conditions as in (iii) and fix finite subsets $S$ and $\Sigma$ of $\Pi_k$ with $S(A)\subseteq S$ and $S\cap \Sigma = \emptyset$. Then, for the structures defined in Examples \ref{nss} and \ref{lesc}, Lemma \ref{flach result}\,(vi) implies that the corresponding induced structures $\mathscr{F}_{\mathrm{rel}}(A,S)_\Sigma\otimes_RR', \mathscr{F}_{\mathrm{rel}}(A,S)^\Sigma\otimes_RR'$, $\mathscr{F}_{\mathrm{str}}(A,S)_\Sigma\otimes_RR'$ and $\mathscr{F}_{\mathrm{str}}(A,S)^\Sigma\otimes_RR'$ respectively identify with $\mathscr{F}_{\mathrm{rel}}(A\otimes_RR',S)_\Sigma,$ $ \mathscr{F}_{\mathrm{rel}}(A\otimes_RR',S)^\Sigma, \mathscr{F}_{\mathrm{str}}(A\otimes_RR',S)_\Sigma$ and $\mathscr{F}_{\mathrm{str}}(A\otimes_RR',S)^\Sigma$. 
\end{romanliste}
\end{example}

For the induced structures defined in the last example, one has the following useful `control theorems' for Selmer complexes.

\begin{prop} \label{Selmer complex control thm}
    Assume  condition (\ref{p=2 condition}) is valid, and let $\mathscr{F}$ be a Nekov\'a\v{r} structure on a finitely generated free $R$-module $A$. Then the following claims are valid. 
    \begin{romanliste}
        \item For every morphism $R \to R'$ of rings satisfying (\ref{ring condition}) one has a natural isomorphism 
        $\RGamma_{\mathscr{F}} (k, A) \otimes^\mathbb{L}_R R' \cong \RGamma_{\mathscr{F}\otimes_R R'} (k, A \otimes_R R')
        $ in $D (R')$. 
        \item Let $R$ be a finite ring satisfying (\ref{ring condition}) and $I \subseteq R$ an ideal. Then one has a natural isomorphism $\RHom_R ( R / I, \RGamma_\mathscr{F} (k, A)) \cong \RGamma_{\mathscr{F}_{A [I]}} (k, A [I])
        $ in $D (R/I)$. 
    \item Assume $\mathscr{F}$ is perfect and, for $n \in \N$, set $A_n \coloneqq A\otimes_R(R/\a_n)$ and $\mathscr{F}_n\coloneqq \mathscr{F}\otimes_R(R/\a_n)$. Then Lemma \ref{limit result} defines an object $\varprojlim_{n \in \N} \RGamma_{\mathscr{F}_n} (k, A_n)$ of $D^{\mathrm{perf}}(R)$ that is naturally isomorphic to $\RGamma_\mathscr{F} (k, A)$. Hence, in each degree $i$, one has $H^i_\mathscr{F} (k, A) = \varprojlim_{n \in \N} H^i_{\mathscr{F}_n} (k, A_n)$. 
    \end{romanliste}
\end{prop}

\begin{proof} The mapping cone of a morphism in $D(R')$ is unique up to ismorphism. Hence, given the explicit definition of each complex $\RGamma_{\mathscr{F}\otimes_R R'} (k_\q, A \otimes_R R')$ and morphism $\theta_{\mathscr{F}_{A\otimes_RR'},\q}$,  the isomorphism in (i) results directly by comparing the triangles (\ref{cs versus Selmer}) for the structures $\mathscr{F}$ and $\mathscr{F}_{A \otimes_R R'}$ and taking account of the isomorphisms in Lemma \ref{flach result}\,(vi). 

The isomorphism in (ii) is obtained by a similar application of Lemma \ref{flach result}\,(vii) with $R' = R / I$.

To prove (iii), the fact $\mathscr{F}$ is perfect allows us to fix a bounded complex $P^\bullet$ of finitely generated projective $R$-modules that is isomorphic in $D(R)$ to $\RGamma_\mathscr{F} (k, A)$. We then fix integers $a$ and $b$ with $a< b$ and $P^i=(0)$ if either $i < a$ or $i > b$. Then by applying the isomorphism in (i) to the morphism $R \to R/\a_n$, respectively $R/\a_{n+1} \to R/\a_n$, we deduce  $\RGamma_{\mathscr{F}_n} (k, A/\a_nA)$ belongs to $D^{\mathrm{perf}}_{[a,b]}(R/\a_n)$, respectively that there exists an isomorphism in $D(R/\a_n)$
\[ \RGamma_{\mathscr{F}_{n+1}} (k, A_{n+1})\otimes^{\mathbb{L}}_{R/\a_{n+1}}R/\a_n \cong \RGamma_{\mathscr{F}_n} (k, A_n).\] The hypotheses of Lemma \ref{limit result} are therefore satisfied in this case so that  $\varprojlim_{n \in \N} \RGamma_{\mathscr{F}_n} (k, A_n)$ is well-defined and isomorphic in $D(R)$ to $\RGamma_\mathscr{F} (k, A)\cong P^\bullet$ by the uniqueness assertion of the latter result. This  isomorphism then induces an identification $H^i_\mathscr{F} (k, A) = \varprojlim_{n \in \N} H^i_{\mathscr{F}_n} (k, A_n)$ in each degree $i$ since all modules $P^j/\a_n$ and $H^j_{\mathscr{F}_n} (k, A_n)$ are finite and inverse limits are exact on the category of finite abelian groups.  
\end{proof}

\begin{rk} The isomorphisms in Proposition \ref{Selmer complex control thm}\,(i) and (ii) respectively give rise to  convergent spectral sequences
\begin{align*}
 \qquad  E_2^{i, j} & = \Tor_{-i}^R ( H^j_\mathscr{F} (k, A), R')  && \mkern-27mu\Rightarrow \quad E^{i + j} = H^{i + j}_{\mathscr{F}_{A \otimes_R R'}} (k, A\otimes_R R'), && \qquad \\
 \qquad   E_2^{i, j}& = \Ext^i_{R} (R / I, H^j_{\mathscr{F}} (k, A)) && \mkern-27mu\Rightarrow \quad E^{i + j}  = H^{i + j}_{\mathscr{F}_{A [I]}} (k, A [I]). && \qquad
        \end{align*}
    In particular, if $H^0_{\mathscr{F}} (k, A)$ vanishes, then the latter spectral sequence induces an isomorphism $H^1_{\mathscr{F}} (k, A) [I] \cong H^1_{\mathscr{F}_{A [I]}} (k, A [I])$. This observation is the natural analogue for Nekov\'a\v{r} structures of the result of \cite[Cor.\@ 3.8]{bss}.
\end{rk}

\subsubsection{Mazur--Rubin--Selmer structures}

\begin{definition}
A `Mazur--Rubin(--Selmer) structure' $\cF$ on $A$ comprises  
\begin{itemize}
    \item a finite set $S (\cF)$ of places of $k$ that contains $S(A)$, and 
    \item for each $\q\in \Pi_k$, an $R$-submodule 
    \[ H^1_\cF (k_\q, A)\subseteq H^1 (k_\q, A)\] 
of $H^1 (k_\q, A)$, with 
\[ H^1_\cF (k_\q, A) \coloneqq H^1_f (k_\q, A) \,\,\text{ for all }\,\, \q \in \Pi_k\setminus S(\cF),\]
where the group $H^1_f (k_\q, A)$ is as defined in (\ref{RGamma_f def}).
\end{itemize}
A  Mazur--Rubin structure $\cF'$ on $A$ `refines' $\cF$, written $\cF'\le \cF$, if $S(\cF)\subseteq S(\cF')$ and for each $\q \in \Pi_k$ one has 
$H^1_{\cF'} (k_\q, A) \subseteq H^1_{\cF} (k_\q, A)$. 
\end{definition}

The link between these structures and Nekov\'{a}\v{r} structures is as follows.

\begin{lemma}\label{comparison selmer structures0} The following claims are valid.
\begin{romanliste}
\item A Nekov\'{a}\v{r} structure $\mathscr{F}$ on $A$ induces a canonical Mazur--Rubin structure $h(\mathscr{F})$ on $A$. If $\mathscr{F}'$ is a Nekov\'{a}\v{r} structure that refines $\mathscr{F}$, then $h(\mathscr{F}')$ refines $h(\mathscr{F}).$
\item Let $\cF$ be a Mazur--Rubin structure on $A$. Then there exists a Nekov\'{a}\v{r} structure $\mathscr{F}$ on $A$ for which   $h(\mathscr{F})=\cF$, $S(\mathscr{F}) = S(\cF)$ and $\mathrm{R}\Gamma_\mathscr{F}(k_\q,A) = H^1_\cF(k_\q,A)[-1]$ for all $\q \in S(\cF)$. These conditions determine $\mathscr{F}$ uniquely if and only if, for every $\q$ in $S(\cF)$, the group  $\mathrm{Ext}^1_{R}(H^1_\cF(k_\q,A),H^0(k_\q,A))$ vanishes.
\end{romanliste}
\end{lemma}

\begin{proof} Fix a Nekov\'{a}\v{r} structure $\mathscr{F}$ on $A$. Then one obtains a well-defined Mazur--Rubin structure $h(\mathscr{F})$ by setting 
\[ S(h(\mathscr{F})) = S(\mathscr{F}) \quad\text{and}\quad  H^1_{h(\mathscr{F})}(k_\q,A) \coloneqq \mathrm{im}(H^1(\theta_{\mathscr{F},\q}))\,\,\text{for all}\,\,\q \in \Pi_k.\] 
With this definition, it is also immediately clear that $h(\mathscr{F}') \le h(\mathscr{F})$ if $\mathscr{F}'\le \mathscr{F}$, as required to prove (i). \\
To prove (ii), we use for each $\q$ in $S(\cF)$ the convergent cohomological spectral sequence 
\begin{multline*} E_2^{p,q} = {\prod}_{t\in\bz}
{\rm{Ext}}^p_{R}(H^t(H^1_\cF(k_\q,A)[-1])),
 H^{q+t}(\mathrm{R}\Gamma(k_\q,A)))\\ \Rightarrow H^{p+q}(\RHom_{R}(H^1_\cF(k_\q,A)[-1],\mathrm{R}\Gamma(k_\q,A)))\end{multline*}
from \cite[{III, 4.6.10}]{verdier}. Since $H^1_\cF(k_\q,A)[-1]$ is acyclic outside degree $1$, this spectral sequence converges to give a short exact sequence 
\begin{multline*}
0 \to \mathrm{Ext}^1_{R}(H^1_\cF(k_\q,A),H^0(k_\q,A)) \to \Hom _{D(R)}(H^1_\cF(k_\q,A)[-1],\mathrm{R}\Gamma(k_\q,A))\\ \xrightarrow{\phi \mapsto H^1(\phi)} \Hom_R(H^1_\cF(k_\q,A),H^1(k_\q,A))\to 0.\end{multline*} 
We may therefore choose a morphism $\phi_\q \: H^1_\cF(k_\q,A)[-1]\to \mathrm{R}\Gamma(k_\q,A)$ in $D(R)$ for which $H^1(\phi_\q)$ is the inclusion $H^1_\cF(k_\q,A)\to H^1(k_\q,A)$. In particular, if we define $\mathscr{F}$ to be the Nekov\'{a}\v{r} structure with $S(\mathscr{F}) = S(\cF)$ and, for each $\q \in S(\cF)$, both $\RGamma_\mathscr{F}(k_\q,A) = H^1_\cF(k_\q,A)[-1]$ and $\theta_{\mathscr{F},\q} = \theta_\q$, then it is clear $h(\mathscr{F}) = \cF$. It is also clear from the displayed exact sequence that, given $\cF$, a Nekov\'{a}\v{r} structure with these conditions is unique if and only if $\mathrm{Ext}^1_{R}(H^1_\cF(k_\q,A),H^0(k_\q,A))$ vanishes for all $\q$ in $S(\cF)$.  
\end{proof} 

\begin{example}\label{ss def} Fix a finite subset $S$ of $\Pi_k$ with $S(A) \subseteq S$.  
\begin{itemize}

\item[(i)] The relaxed Selmer structure $\cF_{\rm{rel}} = \cF_{\rm{rel}}(A,S)$ for $A$ and $S$ defined in \cite[Exam.\@ 2.4]{bss2} is equal to 
 $h(\mathscr{F}_{\rm{rel}}(A,S))$. In particular, one has $S(\cF_{\rm{rel}})= S$ and $H_{\cF_{{\rm{rel}}}}^1(k_\q,A) = H^1(k_\q,A)$  for all $\q\in  S$.

\item[(ii)] The canonical Selmer structure $\cF_{\rm{can}} = \cF_{\rm{can}}(A,S)$ for $A$ and $S$ defined in \cite[Def.~3.2.1]{MazurRubin04} is equal to $h(\mathscr{F}_{\rm{can}}(A,S))$. In particular, one has $S(\cF_{\rm{can}}) = S$ and
\[ H_{\cF_{\rm{can}}}^1(k_\q,A) = \begin{cases}
H^1_f(k_\q,A), &\text{if $\q \in S\setminus (\Pi_k^\infty \cup \Pi_k^p)$,}\\
H^1(k_\q,A), &\text{if $\q \in \Pi_k^\infty \cup \Pi_k^p$.}\\
\end{cases}
\] 

\item[(iii)] If $R$ is a $\ZZ_p$-order, then the unramified Selmer structure $\cF_{\rm{ur}} = \cF_{{\rm{ur}}}(A,S)$ for $A$ and $S$ is a refinement of $\cF_{\rm{can}}(A)$  defined in \cite[Def.~5.1]{MazurRubin}. One has $S(\cF_{\rm{ur}}) =S$ and for $\q\in S$  sets 
\[ H_{\cF_{\rm{ur}}}^1(k_\q,A) \coloneqq \begin{cases} H_{\cF_{\rm{can}}}^1(k_\q,A), &\text{if $\q\in 
S\setminus \Pi_k^p$,}\\
\bigl(\bigcap_L {\rm{Cor}}_{L/k_\q}(H^1(L,A))\bigr)^{\rm{sat}}, &\text{if $\q\in \Pi_k^p$,}\end{cases}\]
where in the intersection $L$ runs over all finite unramified extensions of $k_\q$, and we write $X^{\rm{sat}}$ for the saturation of a subgroup $X$ of $H^1(k_\q,A)$. 
  In particular, if $R$ is a discrete valuation ring, then \cite[Cor.~5.3]{MazurRubin} shows that $\cF_{\rm{ur}}(A)$ coincides with $\cF_{\rm{can}}(A)$ if and only if $H^{0}(k_{\fq}, A^{\vee}(1))$ is finite for every $\fq\in \Pi_k^p$. 

\item[(iv)] Let $A'$ be a submodule, respectively quotient, of the $R[G_{k}]$-module $A$. Then, by \cite[Exam.~1.3.3 and 2.17]{MazurRubin04},  a Mazur--Rubin structure  $\cF$ on $A$ induces the Mazur--Rubin structure $\cF_{A'}$ on $A'$ with $S(\cF_{A'}) \coloneqq S(\cF)$ and,  for each $\q\in S(\cF)$, 
\[
H^1_{\cF_{A'}}(k_\q, A') \coloneqq \ker(H^1(k_\q,A') \to  H^1_{\!/\cF}(k_\q,A),\]
respectively 
\[H^1_{\cF_{A'}}(k_\q, A') \coloneqq\im (H^1_{\cF}(k_\q,A) \to H^1(k_\q,A')).
\]
If $\mathscr{F}$ is a Nekov\'{a}\v{r} structure on $A$ such that $h(\mathscr{F}) = \cF$, then, in the respective notation of Example \ref{remark selmer}\,(i) and (ii),  one has in both cases $\cF_{A'} = h( \mathscr{F}_{A'})$. 
\end{itemize}
\end{example}

\begin{remark}\label{remark selmer2} If $A'$ is a quotient of the $R[G_k]$-module $A$ then, for the following sorts of reasons, care is needed in the use of induced structures in the sense of Example \ref{ss def}\,(iv).

\begin{itemize}
\item[(i)] Whilst, in some cases, the induced structure $\cF_{\rm{can}}(A)_{A'}$ coincides with 
$\cF_{\rm{can}}(A')$, it is in general strictly finer (cf.\@ \cite[Lem.\@ 3.5]{Rubin-euler}, \cite[Prop.\@ 6.2.6]{MazurRubin04}).

\item[(ii)] Let $R \to R'$ be a surjective morphism of rings satisfying (\ref{ring condition}), and write $A'$ for the corresponding quotient $A\otimes_RR'$ of $A$. Assume $A$ is a free $R$-module, let $\mathscr{F}$ be a Nekov\'{a}\v{r} structure on $A$ and recall  the  Nekov\'{a}\v{r} structure $\mathscr{F}\otimes_RR'$ on $A'$ defined (under condition (\ref{p=2 condition})) in Example \ref{remark selmer}\,(iii). Then one has $h(\mathscr{F})_{A'}\le  h(\mathscr{F}\otimes_RR')$ and, in general, this refinement is strict. For example, in contrast to the situation for  Nekov\'{a}\v{r} structures themselves (cf.\@ Example \ref{nss}\,(iv)), the latter refinement is usually strict in the case that $\scrF$ is a relaxed  Nekov\'{a}\v{r} structure.  
\end{itemize}
\end{remark} 

To each Mazur--Rubin structure one associates a Selmer module as follows.

\begin{definition}\label{mr selmer def} Let $\cF$ be a Mazur--Rubin structure on $A$. For $\q \in \Pi_k$, we write $H^1_{ / \cF} (k_\q, A)$ for the quotient module $H^1 (k_\q, A) / H^1_\cF (k_\q, A)$. The Selmer module $H^1_{\cF} (k, A)$ of $\cF$ is then defined to be the kernel of the natural localisation map 
\[
 H^1 (\mathcal{O}_{k,S(\cF)}, A)\xrightarrow{\lambda(\cF)} \bigoplus_{\q\in S(\cF)} H^1_{ / \cF} (k_\q, A).
\]
\end{definition}

If  $S$ is any finite subset of $\Pi_k$ with $S(\cF)\subseteq S$, then the  argument of Lemma \ref{basic properties}\,(i) implies $H^1_{\cF} (k, A)$ is also equal to the kernel of the localisation  map 
\[ \lambda_S(\cF) \: 
 H^1 (\mathcal{O}_{k,S}, A)\to \bigoplus_{\q\in S} H^1_{ / \cF} (k_\q, A).\]

To  describe the link between the Selmer complex of a Nekov\'{a}\v{r} structure $\mathscr{F}$ on $A$ and the Selmer module of the Mazur--Rubin structure $h(\mathscr{F})$, we use, for each finite subset $S$ of $\Pi_k$, the canonical diagonal map
\begin{equation}\label{lambda 0 def} \lambda^0_S(\mathscr{F})\: H^0(k,A) \to \bigoplus_{\q \in S} \frac{H^0(k_\q,A)}{\im(H^0(\theta_{\mathscr{F},\q}))}.\end{equation}
(This is a natural analogue in degree zero of the map $\lambda_S( h(\mathscr{F}))$.)

\filbreak
\begin{lemma}\label{comparison selmer structures} For each Nekov\'{a}\v{r} structure $\mathscr{F}$ on $A$ the following claims are valid. 
\begin{itemize}
\item[(i)]  There exist canonical short exact sequences of $R$-modules  
\begin{cdiagram}[row sep=tiny, column sep=small]
0 \arrow{r} & {\bigoplus}_{\q \in S(\mathscr{F})}\ker(H^0(\theta_{\mathscr{F},\q})) \arrow{r} &  H^0_\mathscr{F}(k,A)\arrow{r} & \ker(\lambda^0_{S(\mathscr{F})}(\mathscr{F}))
\arrow{r} & 0\\
0\arrow{r} & H^1_{h(\mathscr{F})}(k,A)^\vee \arrow{r} & H^1_{\mathscr{F}}(k,A)^\vee \arrow{r} &  \ker(\lambda^0_{S(\mathscr{F})}(\mathscr{F})^\vee)\arrow{r} &  0.
\end{cdiagram}%
\item[(ii)] If $S$ and $S'$ are finite subsets of $\Pi_k$ with $S'\subseteq S$, then there exists a canonical exact sequence of $R$-modules 
\[ 0\to \ker(\lambda^0_{S\setminus S'}(\mathscr{F})^\vee)\to \ker(\lambda^0_{S}(\mathscr{F})^\vee) \xrightarrow{\alpha} \bigoplus_{\q \in S'} \left(\frac{H^0(k_\q,A)}{\im(H^0(\theta_{\mathscr{F},\q}))}\right)^\vee \to \mathrm{cok}(\lambda^0_{S\setminus S'}(\mathscr{F})^\vee) \to 0.\]
\end{itemize}
\end{lemma}

\begin{proof} To prove (i) we abbreviate $S(\mathscr{F})$ to $S$. We then note that $\mathscr{F}$ refines the relaxed structure $\mathscr{F}_{\mathrm{rel}}$ defined in Example \ref{nss}\,(i), and hence that the long exact cohomology sequence of the exact triangle of Lemma \ref{basic properties}\,(ii) for the pair $(\mathscr{F},\mathscr{F}_{\mathrm{rel}})$ gives an exact commutative diagram 
\begin{cdiagram}[column sep=small]
H^{i-1}(\cO_{k,S},A)  \arrow{r} & \bigoplus H^{i-1}(C_\q) \arrow{r} & H^i_\mathscr{F}(k,A) \ar[r] & H^{i}(\cO_{k,S},A) \ar[r] & \bigoplus H^{i}(C_\q) \\
H^{i-1}(\cO_{k,S},A) \arrow[equals]{u} \ar["\lambda^{i-1}"]{r} & \bigoplus\frac{H^{i-1}(k_\q,A)}{\im(H^{i-1}(\theta_{\mathscr{F},\q}))} \ar["(\beta^{i-1}_\q)_\q"]{u} & & H^{i}(\cO_{k,S},A) \arrow[equals]{u} \ar["\lambda^i"]{r}  & \bigoplus \frac{H^{i}(k_\q,A)}{\im(H^i(\theta_{\mathscr{F},\q}))} \ar["(\beta^i_\q)_\q"]{u}
\end{cdiagram}%
Here all direct sums run over $\q \in S$ and for such $\q$ we set $C_\q \coloneqq \RGamma_{\mathscr{F}_{\mathrm{rel}}/\mathscr{F}}(k_\q,A) = \RGamma_{\!/\mathscr{F}}(k_\q,A)$ and write $\beta^i_\q$ for the injective map induced by the long exact cohomology sequence of the tautological exact triangle 
$\RGamma_{\mathscr{F}}(k_\q,A)\to \RGamma(k_\q,A) \to C_\q\to \cdot$. In addition, $\lambda^i$ denotes the natural localisation map so that $\lambda^0 = \lambda^0_S(\mathscr{F})$ and $\lambda^1 = \lambda_S(h(\mathscr{F}))$. 

Now $H^{-1}(\cO_{k,S},A) = (0)$  and, for each $\q$, also $H^{-1}(C_\q) = \ker(H^0(\theta_{\mathscr{F},\q}))$ since 
$H^{-1}(k_\q,A) = (0)$. Given these facts, the exactness of the above diagram with $i=0$, respectively $i=1$, directly gives the first exact sequence in (i), respectively a short exact sequence 
\begin{equation}\label{nekovar mr compare} 0 \to \mathrm{cok}(\lambda^0_S(\mathscr{F})) \to H^1_\mathscr{F}(k,A) \to H^1_{h(\mathscr{F})}(k,A) \to 0.\end{equation}
The second exact sequence in (i) is then obtained as the Matlis dual of the latter sequence. 

The exact sequence in (ii) is directly obtained by applying the Snake Lemma to the obvious exact commutative diagram
\begin{equation*}\xymatrix@R=1em{ 
\ker(\lambda^0_{S\setminus S'}(\mathscr{F})^\vee) \ar@{^{(}->}[r] \ar@{^{(}->}[d] & \left(\bigoplus_{\q \in S\setminus S'} \frac{H^0(k_\q,A)}{\im(H^0(\theta_{\mathscr{F},\q}))}\right)^\vee  \ar@{->>}[rr]^{\hskip0.3truein\lambda^0_{S\setminus S'}(\mathscr{F})^\vee} \ar@{^{(}->}[d] & &\im(\lambda^0_{S\setminus S'}(\mathscr{F})^\vee)  \ar@{^{(}->}[d]\\
\ker(\lambda^0_{S}(\mathscr{F})^\vee) \ar@{^{(}->}[r] & \left(\bigoplus_{\q \in S} \frac{H^0(k_\q,A)}{\im(H^0(\theta_{\mathscr{F},\q}))}\right)^\vee  \ar@{->>}[rr]^{\hskip0.4truein\lambda^0_{S}(\mathscr{F})^\vee} \ar@{->>}[d] & &H^0(k,A)^\vee \\
&\left(\bigoplus_{\q \in S'} \frac{H^0(k_\q,A)}{\im(H^0(\theta_{\mathscr{F},\q}))}\right)^\vee .
& & \hskip1.95truein \qed \qedhere}
\end{equation*} 
\end{proof}

Let  $\cF$ be a Mazur--Rubin structure on $A$ and, for each $n$, write $\cF_n$ for the induced structure $\cF_{A_n}$ on $A_n = A\otimes_R(R/\a_n)$. Then the explicit definition of induced structure (in Example~\ref{ss def}\,(iv)) implies  the existence of a natural projection map $H^1_{\cF_{n+1}}(k, A_{n+1}) \to H^1_{\cF_n}(k, A_n)$. In a similar way, there is a natural diagonal projection map from $H^1_{\cF} (k, A)$ to the corresponding inverse limit ${\varprojlim}_{n \in \N} H^1_{\cF_n} (k, A_n)$. The following result is the analogue of Proposition \ref{Selmer complex control thm} in this setting (and does not require $A$ to be a free $R$-module). 

\begin{lem} \label{limit of selmer group lemma}
  Assume $A$ is finitely generated as an $R$-module. Then, for any Mazur--Rubin structure $\cF$ on $A$, the natural map $H^1_{\cF} (k, A) \to {\varprojlim}_{n \in \N} H^1_{\cF_n} (k, A_n)$ is bijective.
\end{lem}

\begin{proof} 
Consider the exact commutative diagram
\begin{cdiagram}[column sep=small, row sep=small]
    0 \arrow{r} & H^1_{\cF} (k, A) \arrow{r} \arrow{d} & H^1 (k, A) \arrow{d}{\simeq} \arrow{r} & {\bigoplus}_{v\in \Pi_k} H^1_{\!/\cF} (k_v, A) \arrow{d}{\alpha} \\ 
    0 \arrow{r} & {\varprojlim}_n H^1_{\cF_n} (k, A_n) \arrow{r} & {\varprojlim}_n H^1 (k, A_n) \arrow{r} &  {\varprojlim}_n {\bigoplus}_{v\in \Pi_k} H^1_{\!/\cF_n}(k_v, A_n).& 
\end{cdiagram}%
Here the middle vertical map is bijective by \cite[Prop.\@ B.2.3]{Rubin-euler} since each $A_n$ is finite. It therefore suffices to prove that $\alpha$ is injective and hence, since  $\alpha$ factors as a natural composite 
\[\bigoplus_{v \in \Pi_k} H^1_{\!/\cF} (k_v, A) \xrightarrow{(\beta_v)_v} \bigoplus_{v \in \Pi_k} {\varprojlim}_n  H^1_{\!/\cF_{n}} (k_v, A_n)\hookrightarrow {\varprojlim}_n \bigoplus_{v \in \Pi_k} H^1_{\!/\cF_{n}} (k_v, A_n),\]
 it suffices to show  that each map $\beta_v$ is bijective. To do this, we consider the natural exact commutative diagram
\begin{cdiagram}[column sep=small, row sep=small]
0 \arrow{r} & H^1_\cF (k_v, A) \arrow{d}{\gamma_v} \arrow{r} & H^1 (k_v, A) \arrow{r} \arrow{d}{\simeq} & 
H^1_{\!/\cF} (k_v, A) \arrow{r} \arrow{d}{\beta_v} & 0 \\ 
0 \arrow{r} & {\varprojlim}_n H^1_{\cF_{n}} (k_v, A_n) \arrow{r} & {\varprojlim}_n H^1 (k_v, A_n) \arrow{r} & {\varprojlim}_n H^1_{\!/\cF_{n}} (k_v, A_n)  \arrow{r} & 0.
\end{cdiagram}%
Here the second vertical map is again bijective because each $A_n$ is finite and the exactness of the bottom row is obtained by applying the functor ${\varprojlim}_n ( - )$ to the underlying (tautological) exact sequence of finite groups. The bijectivity of $\beta_v$ will therefore follow as a consequence of the Snake Lemma provided that $\gamma_v$ is surjective. This is in turn true  since each map $H^1_\cF (k_v, A) \to H^1_{\cF_n} (k_v, A_n)$ is surjective (by definition of $\cF_{n}$) and inverse limits are exact on the category of finite groups. 
\end{proof}

\subsection{Duality results}\label{duality realtions section}

If $A$ is finitely generated over $R$, then, for each $\q\in \Pi_k$, local Tate duality induces a canonical isomorphism of $R$-modules $H^1(k_\q, A) \simeq H^1(k_\q, A^\vee(1))^\vee$. Following 
 \cite[\S1.3]{MazurRubin04}, for any Mazur--Rubin structure $\cF$ on $A$, one uses these isomorphisms to define a `dual' Mazur--Rubin structure $\cF^\vee$ on $A^\vee(1)$ by setting $S(\cF^\vee)\coloneqq S(\cF)$ and, for each $\q\in \Pi_k$,   
\[ H^1_{\cF^\vee}(k_\q, A^\vee(1)) \coloneqq \ker \big (H^1(k_\q, A^\vee(1)) \simeq H^1(k_\q, A)^\vee\to H^1_{\cF}(k_\q,A)^\vee \big ) \cong H^1_{\!/\cF}(k_\q, A)^\vee . \]

In this section, we discuss the natural analogues of this construction for Nekov\'{a}\v{r} structures.

\subsubsection{Dual Nekov\'{a}\v{r}--Selmer structures}

The following terminology will be convenient. 

\begin{definition} A Nekov\'{a}\v{r} structure $\mathscr{F}$ on $A$ is `$\infty$-relaxed', respectively `$\infty$-strict', if for every $\q\in \Pi_k^\infty$ one has  $\mathrm{R}\Gamma_{\mathscr{F}}(k_\q,A) = \mathrm{R}\Gamma(k_\q,A)$ and $\theta_{\mathscr{F},\q}$ is the identity morphism, respectively $\mathrm{R}\Gamma_{\mathscr{F}}(k_\q,A)$ is the zero complex.\end{definition}

\begin{example} The structures $\mathscr{F}_{\mathrm{rel}}(A,S)_\Sigma$, $\mathscr{F}_{\mathrm{rel}}(A,S)^\Sigma$, $\mathscr{F}_{\mathrm{can}}(A,S)_\Sigma$ and $\mathscr{F}_{\mathrm{can}}(A,S)_\Sigma$ from  Example \ref{nss} are $\infty$-relaxed, whilst $\mathscr{F}_{\mathrm{str}}(A,S)_\Sigma$ and $\mathscr{F}_{\mathrm{str}}(A,S)^\Sigma$ are $\infty$-strict.\end{example}

We now introduce the notions of dual Nekov\'{a}\v{r} structure that are useful in our theory. 

\begin{definition}\label{dualisable structures} 
For any Nekov\'{a}\v{r} structure $\mathscr{F}$ on a finite-rank free $R$-module $A$, we specify  `dual' $\infty$-strict Nekov\'{a}\v{r} structures $\mathscr{F}^\ast$ on $A^\ast(1)$ and $\mathscr{F}^\vee$ on $A^\vee (1)$ as follows. 
\begin{liste}
 \item  $S(\mathscr{F}^\ast) = S(\mathscr{F})$; for $\q \in \Pi_k^\infty$ one has $\mathrm{R}\Gamma_{\mathscr{F}^\ast}(k_\q,A^\ast(1)) \coloneqq 0$; for $\q \in S(\mathscr{F})\setminus \Pi_k^\infty$, we set 
\[ \mathrm{R}\Gamma_{\mathscr{F}^\ast}(k_\q,A^\ast(1)) \coloneqq \RHom_R(\mathrm{R}\Gamma_{\!/\mathscr{F}}(k_\q,A),R[-2])\]
and take $\theta_{\mathscr{F}^\ast,\q}$ to be the composite morphism 
\[ \mathrm{R}\Gamma_{\mathscr{F}^\ast}(k_\q,A) \xrightarrow{\phi_\q^\ast[-2]} \RHom_R(\mathrm{R}\Gamma(k_\q,A),R[-2]) \cong \mathrm{R}\Gamma(k_\q,A^\ast(1))\]
where $\phi_\q$ is the natural morphism $\mathrm{R}\Gamma(k_\q,A) \to \mathrm{R}\Gamma_{\!/\mathscr{F}}(k_\q,A)$ and the isomorphism is the first isomorphism in Proposition \ref{local Tate duality thm}\,(i). 
\item $S(\mathscr{F}^\vee) = S(\mathscr{F})$;  for $\q \in \Pi_k^\infty$ one has $\mathrm{R}\Gamma_{\mathscr{F}^\ast}(k_\q,A^\vee (1)) \coloneqq 0$; for $\q \in S(\mathscr{F})\setminus \Pi_k^\infty$, we set 
\[ \mathrm{R}\Gamma_{\mathscr{F}^\vee}(k_\q,A^\vee (1)) \coloneqq \RHom_R(\mathrm{R}\Gamma_{\!/\mathscr{F}}(k_\q,A),E_R (\mathbb{k}) [-2])\]
and take $\theta_{\mathscr{F}^\vee,\q}$ to be the composite morphism 
\[ \mathrm{R}\Gamma_{\mathscr{F}^\vee}(k_\q,A) \xrightarrow{\phi_\q^\vee [-2]} \RHom_R(\mathrm{R}\Gamma(k_\q,A),E_R (\mathbb{k}) [-2]) \cong \mathrm{R}\Gamma(k_\q,A^\vee(1))\]
where $\phi_\q$ is the natural morphism $\mathrm{R}\Gamma(k_\q,A) \to \mathrm{R}\Gamma_{\!/\mathscr{F}}(k_\q,A)$ and the isomorphism is the second isomorphism in Proposition \ref{local Tate duality thm}\,(i). 
\end{liste}
\end{definition} 

\begin{example}\label{dual comments} For any finite subset $\Sigma$ of $\Pi_k$ with $\Sigma \cap S(\mathscr{F}) = \emptyset$,  one has $(\mathscr{F}_{\rm{rel}}(A,S)_{\Sigma})^\ast = \mathscr{F}_{\rm{str}}(A^\ast(1),S)^\Sigma$ and $(\mathscr{F}_{\rm{rel}}(A,S)_{\Sigma})^\vee = \mathscr{F}_{\rm{str}}(A^\vee (1),S)^\Sigma$. 
\end{example} 

The following result establishes some important compatibilities between the notions of dual and induced Nekov\'a\v{r} and Mazur--Rubin structures. 

\begin{lem} \label{some basic duality properties of Nekovar structures}
For every Nekov\'a\v{r} structure $\mathscr{F}$ on a finite-rank free $R$-module $A$ the following claims are valid.
\begin{romanliste}
\item For a morphism $R \to R'$ of rings satisfying (\ref{ring condition}), the isomorphism $(A \otimes_R R')^\ast (1) \cong A^\ast (1) \otimes_R R'$ induces an equality $(\mathscr{F}_{A \otimes_R R'})^\ast = (\mathscr{F}^\ast)_{A^\ast (1) \otimes_R R'}$ of Nekov\'a\v{r} structures.
    \item For every $n \leq m$, the isomorphism $A_n^\ast (1) \cong A_m^\vee [\a_n]$ induces an equality $(\mathscr{F}_{A_n})^\ast = (\mathscr{F}^\ast_{A_m})_{A_m^\vee [\a_n] (1)}$ of Nekov\'a\v{r} structures.
    \item One has an equality $h ( \mathscr{F}^\vee) = h (\mathscr{F})^\vee$ of Mazur--Rubin structures  on $A^\vee(1)$.
\end{romanliste}
\end{lem}

\begin{proof}
The proofs of (i) and (ii) proceed along very similar lines and so we restrict ourselves to the proof of (ii), leaving details for  (i) to the attentive reader. For this, we define functors
\[ F_1 (-)  \coloneqq \RHom_{R_n} ( - \otimes^\mathbb{L}_R R_n, R_n) \quad\text{and}\quad F_2 (-)  \coloneqq \RHom_R (R_n, \RHom_R ( -, E_R (\mathbb{k}))).\]
Then it follows from derived Tensor-Hom adjunction \cite[Th.\@ 10.8.7]{Weibel-homologicalAlgebra} that there is natural isomorphism of functors  $F_1 \cong F_2$, hence for every $\q \in \Pi_k$ we have a commutative diagram
\begin{cdiagram}
    F_1 ( \RGamma_{\!/ \mathscr{F}} (k_\q, A_m)) \arrow{rr}{F_1 (\theta_{\mathscr{F}, \q})} \arrow{d}{\simeq} & & F_1 ( \RGamma (k_\q, A_m)) \arrow{d}{\simeq} \\
    F_2 ( \RGamma_{\!/ \mathscr{F}} (k_\q, A_m)) \arrow{rr}{F_1 (\theta_{\mathscr{F}, \q})} & & F_2 ( \RGamma (k_\q, A_m))
\end{cdiagram}%
In addition, 
we have an isomorphism (using Lemma \ref{flach result}\,(vi) for the first and Theorem \ref{local Tate duality thm}\,(i) for the second isomorphism)
\[
t_1 \: F_1 ( \RGamma (k_\q, A_m)) \xrightarrow{\simeq} \RHom_{R_n} ( \RGamma (k_\q, A_n), R_n) \xrightarrow{\simeq} \RGamma (k_\q, A_n^\vee (1))
\]
and also an isomorphism (using Theorem \ref{local Tate duality thm}\,(i) for the first and Lemma \ref{flach result}\,(vii) for the second isomorphism)
\[
t_2 \: F_2 ( \RGamma (k_\q, A_m)) \xrightarrow{\simeq} \RHom_R (R_n, \RGamma (k_\q, A_m^\vee (1))) \xrightarrow{\simeq} \RGamma (k_\q, A_n^\vee (1)).
\]
By definition, one then has $\theta_{(\mathscr{F}_{A_n})^\ast, \q} = t_1 \circ F_1 ( \theta_{\mathscr{F}, \q})$ and $\theta_{(\mathscr{F}^\vee)_{A_m^\vee [\a_n] (1)}, \q} = t_2 \circ F_2 ( \theta_{\mathscr{F}, \q})$. \\
Recall that the isomorphisms of local Tate duality in Theorem \ref{local Tate duality thm}\,(i) are induced by cup products and are therefore functorial. 
It follows that we have a commutative diagram (cf.\@ \cite[proof of Prop.\@ 5.2.4]{SelmerComplexes})
\begin{cdiagram}
    F_1 ( \RGamma (k_\q, A_m)) \arrow{d}{\simeq} \arrow{r}{t_1} & \RGamma (k_\q, A_n^\vee (1)) \arrow[equals]{d} \\ 
    F_2 ( \RGamma (k_\q, A_m))  \arrow{r}{t_2} & \RGamma (k_\q, A_n^\vee (1)),
\end{cdiagram}%
which together with the previous discussion proves the claimed equality $(\mathscr{F}_{A_n})^\ast = (\mathscr{F}^\vee)_{A_m^\vee [\a_n] (1)}$ of Nekov\'a\v{r} structures.\\
To prove (iii), we recall that, for $\q \in \Pi_k$, the group $H^1_{h (\mathscr{F})} (k_\q, A^\vee (1))$ is defined to be the image of $H^1 ( \theta_{\mathscr{F}^\vee, \q})$. Now, there is the commutative diagram
    \begin{cdiagram}
        \RGamma_{\mathscr{F}^\vee} (k_\q, A^\vee (1)) \arrow[equals]{r} \arrow{d}{\theta_{\mathscr{F}^\ast, \q}} & \RHom_R ( \RGamma_{\!/ \mathscr{F}} (k_\q, A^\ast (1)), E_R (\mathbb{k})) [-2] \arrow{d}{ \theta_{\mathscr{F}, \q}^\vee} \\ 
        \RGamma (k_q, A^\vee (1)) & \RHom_R ( \RGamma (k_q, A), E_R (\mathbb{k})) [-2] \arrow{l}[above]{\simeq}.
    \end{cdiagram}%
    Here the equality in the upper row  follows directly from the definition of $\mathscr{F}^\vee$, and the lower isomorphism is the second in Proposition \ref{local Tate duality thm}\,(i). Then,  since $E_R (\mathbb{k})$ is self-injective so that the functor $(-)^\vee = \Hom_R ( - , E_R (\mathbb{k}))$ is exact (and hence commutes with cohomology), upon taking cohomology one obtains a commutative diagram
    \begin{cdiagram}
        H^1_{\mathscr{F}^\vee} (k_\q, A^\vee (1))  \arrow[equals]{r} \arrow{d}{H^1(\theta_{\mathscr{F}, \q})} & H^1_{\!/ \mathscr{F}^\vee} (k_\q, A^\vee (1))^\vee \arrow{d}{H^1 ( \theta_{\mathscr{F}, \q})^\vee} \\ 
        H^1 (k_\q, A^\vee (1)) & H^1 (k_\q, A)^\vee \arrow{l}[above]{\simeq},
    \end{cdiagram}%
    where the lower isomorphism is induced by local Tate duality. We deduce that
    \begin{align*}
        H^1_{h (\mathscr{F})} (k_\q, A^\vee (1)) & \coloneqq \im ( H^1 ( \theta_{\mathscr{F}^\vee, \q})) \\ 
        & = \im \big (  H^1_{\!/ \mathscr{F}} (k_\q, A^\vee (1))^\ast \xrightarrow{H^1 ( \theta_{\mathscr{F}, \q})^\vee} H^1 (k_\q, A)^\vee  \cong H^1 (k_\q, A^\vee (1)) \big) \\ 
        & \cong \ker \big ( \im (\theta_{\mathscr{F}, \q})^\vee \to H^1 (k_\q, A)^\vee  \cong H^1 (k_\q, A^\vee (1)) \big) \\ 
        & = \ker \big ( H^1_{h (\mathscr{F})} (k_\q, A)^\vee \to H^1 (k_\q, A)^\vee  \cong H^1 (k_\q, A^\vee (1)) \big) \\ 
        & \eqqcolon H^1_{h (\mathscr{F})^\vee} (k_\q, A^\vee (1)),
    \end{align*}
    where the isomorphism is obtained by applying the exact functor $(-)^\vee$ to the exact sequence
    \[
    H^1_{\mathscr{F}} (k_\q, A) \xrightarrow{H^1 (\theta_{\mathscr{F}, \q})} 
    H^1 (k_\q, A) \to H^1_{\!/ \mathscr{F}} (k_\q, A).
    \]
    This proves the claimed equality $h ( \mathscr{F}^\vee) = h (\mathscr{F})^\vee$ of Mazur--Rubin structures.
\end{proof}

\begin{remark}\label{dual limit selmer complex} Assume $A$ (and hence also $A^\ast(1)$) is a finite-rank free $R$-module and that $\mathscr{F}^\ast$ is perfect. Then, by combining the identifications $(\mathscr{F}^\ast)_{A_n^\ast (1)} = (\mathscr{F}_{A_n})^\ast$ in Lemma \ref{some basic duality properties of Nekovar structures}\,(i) with the result of  Proposition \ref{Selmer complex control thm}\,(iii) for $\mathscr{F}^\ast$, one obtains a natural isomorphism in $D(R)$
\[ \RGamma_{\mathscr{F}^\ast} (k, A^\ast(1)) \cong {\varprojlim}_n \RGamma_{(\mathscr{F}_{A_n})^\ast} (k, A^\ast_n (1))\] 
 and,  in each degree $i$,  an induced identification $H^i_{\mathscr{F}^\ast} (k, A^\ast(1))= \varprojlim_{n \in \N} H^i_{(\mathscr{F}_{A_n})^\ast} (k, A_n^\ast(1))$. \end{remark}

\begin{remark} \label{compatibility of MR structures under local duality rk}
Let $\cF$ be a Mazur--Rubin structure on a finitely generated $R$-module $A$. Then, for each $n$,  there exists an equality of Mazur--Rubin structures $(\cF_{A_n})^\ast = (\cF^\vee)_{A^\vee [\a_n] (1)}$ on $A_n^\ast (1) \cong A^\vee [\a_n]$ that is analogous to Lemma \ref{some basic duality properties of Nekovar structures}\,(ii) (cf.\@ \cite[Ex.~1.3.3]{MazurRubin04}). This fact  combines with the fixed identification
\[ A_n^\ast \cong A_n^\vee = A_{n + 1}^\vee [\a_n] \hookrightarrow A_{n + 1}^\vee = A_{n + 1}^\ast\] 
(cf.\@ Lemma 
\ref{inj env lem}) to induce both a morphism $H^1_{(\cF_{A_n})^\ast} (k, A_n^\ast (1)) \to H^1_{(\cF_{A_{n + 1}})^\ast} (k, A_{n + 1}^\ast (1))$ and also a diagonal map $H^1_{\cF^\vee} (k, A^\vee (1)) \to \textstyle \varinjlim_{n \in \N} H^1_{(\cF_{A_n})^\ast} (k, A_n^\ast (1))$, where the limit is taken with respect to the above morphisms. By a similar argument to that in Lemma \ref{limit of selmer group lemma} one can prove that this diagonal map is bijective. \end{remark}

We will also use the following technical result regarding the $\a_n$-torsion submodules of dual Selmer modules.  

\begin{lem}\label{dual selmer group commutes with M torsion}
Assume that $A$ is $R$-free of finite rank and such that  $\Abar^\vee (1)^{G_k}=(0)$. Then, for any Mazur--Rubin structure $\cF$ on $A$, and every $n \in \N_0$, the natural map $A_n^\vee (1) \hookrightarrow A^\vee (1)$ induces an isomorphism
    $H^{1}_{(\cF_{A_n})^\vee} (k, A_n^\vee (1)) \stackrel{\simeq}{\longrightarrow} H^{1}_{\cF^\vee} (k, A^\vee (1)) [\a_n]
    $.
\end{lem}

\begin{proof} This is proved in \cite[Cor.\@ 3.8]{bss} (see also \cite[Lem.\@ 3.5.3]{MazurRubin04}).
\end{proof}

\subsubsection{Consequences of Artin--Verdier duality}

In this section, we discuss relations between the Selmer complexes that are respectively associated to a Nekov\'a\v{r} structure $\mathscr{F}$ and to its duals $\mathscr{F}^\ast$ and $\mathscr{F}^\vee$. 

In order to state the first result, for each $\q \in \Pi_k^\RR$ we write  $\tau_\q(A)$ for the canonical morphism $\mathrm{R}\Gamma(k_\q,A) \to \mathrm{R}\Gamma_{\rm{Tate}}(k_\q,A)$. We then define an object of $D(R)$ by setting 
\[ \Delta_\infty(k,A) \coloneqq \bigoplus_{\q \in \Pi_k^\CC}A \oplus \bigoplus_{\q\in \Pi_k^\RR}\mathrm{Cone}(\tau_\q(A))[-1].\]

\begin{prop}\label{av prop} Artin--Verdier duality induces canonical morphisms in $D(R)$ 
\begin{align*} \mu_{\mathscr{F},\mathscr{F}^\ast}\: & \RHom_R (\Delta_\infty(k,A^\ast(1)),R)[-3] \to \textstyle \bigoplus_{\q\in \Pi_k^\infty}\mathrm{R}\Gamma_{\!/\mathscr{F}}(k_\q,A)\\
 \mu_{\mathscr{F},\mathscr{F}^\vee}\: & \RHom_R (\Delta_\infty(k,A^\vee(1)),E_R (\mathbb{k}))[-3] \to \textstyle \bigoplus_{\q\in \Pi_k^\infty}\mathrm{R}\Gamma_{\!/\mathscr{F}}(k_\q,A).
 \end{align*}
In addition, setting 
\[ \mathrm{R}\Gamma^{\mathrm{AV}}_{\mathscr{F},\mathscr{F}^\ast}(k,A) \coloneqq \mathrm{Cone}(\mu_{\mathscr{F},\mathscr{F}^\ast})\quad\text{and}\quad \mathrm{R}\Gamma^{\mathrm{AV}}_{\mathscr{F},\mathscr{F}^\vee}(k,A) \coloneqq \mathrm{Cone}(\mu_{\mathscr{F},\mathscr{F}^\vee}),\] 
there exist canonical exact triangles in $D(R)$  
\begin{align*}
\mathrm{R}\Gamma_\mathscr{F}(k,A) & \xrightarrow{\delta_{\mathscr{F},\mathscr{F}^\ast}} \RHom_R (\mathrm{R}\Gamma_{\mathscr{F}^\ast}(k,A^\ast(1)), R)[-3] \to \mathrm{R}\Gamma^{\mathrm{AV}}_{\mathscr{F},\mathscr{F}^\ast}(k,A)\to \cdot\\
\mathrm{R}\Gamma_\mathscr{F}(k,A) & \xrightarrow{\delta_{\mathscr{F},\mathscr{F}^\vee}} \RHom_R (\mathrm{R}\Gamma_{\mathscr{F}^\vee}(k,A^\vee(1)), E_R (\mathbb{k}))[-3] \to \mathrm{R}\Gamma^{\mathrm{AV}}_{\mathscr{F},\mathscr{F}^\vee}(k,A)\to \cdot .
\end{align*}
\end{prop}
\begin{proof} To prove the stated results for $\mathscr{F}^\ast$,  we set $B \coloneqq A^\ast(1)$ and consider the following commutative diagram in $D(R)$
\begin{equation}\label{duality diagram 0}\xymatrix{ 
& & \Delta_\infty(k,B)\ar[d]\\
\mathrm{R}\Gamma_c(\cO_{k,S}, B) \ar[r] \ar[d]_{\mu} &  \mathrm{R}\Gamma(\cO_{k,S},B) \ar[r] \ar@{=}[d] & {\bigoplus}_{\q \in S}\mathrm{R}\Gamma(k_\q,B) \ar[r] \ar[d]^{ (\tau_\q(B))_\q} 
& \cdot\\
\widetilde{\mathrm{R}\Gamma}_c(\cO_{k,S},B) \ar[r] &  \mathrm{R}\Gamma(\cO_{k,S},B) \ar[r]^{\hskip-0.2truein \lambda_S(B)} & {\bigoplus}_{\q \in S}\mathrm{R}\Gamma_{\rm{Tate}}(k_\q,B)\ar[d] \ar[r] 
& \cdot\\
& & \cdot}\end{equation} 
Here the upper triangle is the appropriate case of the exact triangle (\ref{compact def}) and, for each place $\q\in S\setminus \Pi_k^\RR$, we set $\mathrm{R}\Gamma_{\rm{Tate}}(k_\q,B)\coloneqq \mathrm{R}\Gamma(k_\q,B)$ and write $\tau_\q(A)$ for the identity morphism $\mathrm{R}\Gamma(k_\q,B) \to \mathrm{R}\Gamma(k_\q,B)$. In addition, $\lambda_S(B)$ denotes the natural localisation morphism and $\widetilde{\mathrm{R}\Gamma}_c(\cO_{k,S},B) $  its mapping fibre and the vertical exact triangle is that which follows directly from the definition of $\Delta_\infty(k,B)$. In particular, from the (obvious) commutativity of the central square, one deduces the existence of a morphism $\mu$ that completes the diagram to give a morphism of exact triangles, and thereby implies the existence of an exact triangle that forms the first column of the following commutative diagram in $D(R)$ 
\begin{equation}\label{duality diagram11} 
\xymatrix{ \mathrm{R}\Gamma_c(\cO_{k,S}, B) \ar[r] \ar[d]_{\mu} &  \mathrm{R}\Gamma_{\mathscr{F}^\ast}(k,B) \ar[r] \ar[d]_{\mu_{\mathscr{F}^\ast}} & {\bigoplus}_{\q \in S\setminus \Pi_k^\infty}\mathrm{R}\Gamma_{\mathscr{F}^\ast}(k_\q,B) \ar[r] \ar@{=}[d] 
& \cdot\\
\widetilde{\mathrm{R}\Gamma}_c(\cO_{k,S},B) \ar[r]\ar[d] &  \widetilde{\mathrm{R}\Gamma}_{\mathscr{F}^\ast}(k,B) \ar[r] \ar[d]_{\mu'_{\mathscr{F}^\ast}}&  {\bigoplus}_{\q \in S\setminus \Pi_k^\infty}\mathrm{R}\Gamma_{\mathscr{F}^\ast}(k_\q,B) \ar[r] &\cdot\\
\Delta_\infty(k,B)\ar[d] \ar@{=}[r]
& \Delta_\infty(k,B)\ar[d] \\
\cdot & \cdot}\end{equation}
Here the complex $\widetilde{\mathrm{R}\Gamma}_{\mathscr{F}^\ast}(k,B)$ is defined via the same exact triangle as  $ \mathrm{R}\Gamma_{\mathscr{F}^\ast}(k,B)$ after replacing $\mathrm{R}\Gamma (k_\q,B)$ for each $\q \in \Pi_k^\infty$ by $\mathrm{R}\Gamma_{\mathrm{Tate}}(k_\q,B)$ and the morphism $\mu_{\mathscr{F}^\ast}$ is then induced by the obvious analogue of (\ref{duality diagram 0}). In  addition, the first row is the exact triangle of (\ref{cs versus Selmer}) (with $A$ and $\mathscr{F}$ replaced by $B$ and $\mathscr{F}^\ast$) and the second row is the natural analogue of this exact triangle. In particular, since the first square in the diagram commutes there exists a morphism $\mu_{\mathscr{F}^\ast}$ which makes the lower square commute and the second column an exact triangle in $D(R)$. 

Now, if we apply the exact functor $X \mapsto X^\ast[-3]$ to the second row of (\ref{duality diagram11}) and then substitute the first isomorphism in Proposition \ref{local Tate duality thm}\,(ii) and recall the explicit definition of the complexes $\mathrm{R}\Gamma_{\mathscr{F}^\ast}(k_\q,B)$ for $\q \in S\setminus \Pi_k^\infty$, we obtain an exact triangle in $D(R)$
\[\widetilde{\mathrm{R}\Gamma}_{\mathscr{F}^\ast}(k,B)^\ast[-3] \to \mathrm{R}\Gamma(\cO_{k,S},A)\xrightarrow{\lambda} 
  {\bigoplus}_{\q \in S\setminus \Pi_k^\infty}\mathrm{R}\Gamma_{\!/\mathscr{F}}(k_\q,A) \to\cdot \] 
 in which $\lambda$ is the natural localisation map. By the argument of Lemma \ref{basic properties}\,(ii), this in turn induces a canonical exact triangle in $D(R)$ that forms the central row of the following diagram 

\[\label{duality diagram1} 
\xymatrix{ &   \Delta_\infty(k,B)^\ast[-3] \ar[r]^{\hskip-.2truein \theta'\circ  \mu'} \ar[d]_{\mu'} & {\bigoplus}_{\q \in \Pi_k^\infty}\mathrm{R}\Gamma_{\!/\mathscr{F}}(k_\q,A)\ar@{=}[d]\\
\mathrm{R}\Gamma_{\mathscr{F}}(k,A) \ar[d]_{\mu\circ\theta} \ar[r]^{\hskip-0.3truein \theta} & \widetilde{\mathrm{R}\Gamma}_{\mathscr{F}^\ast}(k,B)^\ast[-3] \ar[d]_{\mu} \ar[r]^{\,\,\,\theta'} & {\bigoplus}_{\q \in \Pi_k^\infty}\mathrm{R}\Gamma_{\!/\mathscr{F}}(k_\q,A) \ar[r] & \cdot\\
\mathrm{R}\Gamma_{\mathscr{F}^\ast}(k_\q,B)^\ast[-3] \ar@{=}[r] & \mathrm{R}\Gamma_{\mathscr{F}^\ast}(k_\q,B)^\ast[-3] \ar[d]\\
&.}\]
Here the central column is the exact triangle obtained by applying the functor $X\mapsto X^\ast[-3]$ to the second column in (\ref{duality diagram11}) so $\mu' = (\mu_{\mathscr{F}^\ast}')^\ast[-3]$ and $\mu = (\mu_{\mathscr{F}^\ast})^\ast[-3]$. In particular, if we respectively define $\mu_{\mathscr{F},\mathscr{F}^\ast}$ and $\delta_{\mathscr{F},\mathscr{F}^\ast}$ to be the morphisms $\theta'\circ  (\mu_{\mathscr{F}^\ast}')^\ast[-3]$ and $(\mu_{\mathscr{F}^\ast})^\ast[-3]\circ\theta$, then the Octahedral axiom combines with the commutativity of the above diagram to imply that  the mapping fibres of $\mu_{\mathscr{F},\mathscr{F}^\ast}$ and $\delta_{\mathscr{F},\mathscr{F}^\ast}$ are isomorphic (in $D(R)$). This last fact leads directly to an exact triangle of the required form and so proves all claimed  results for $\mathscr{F}^\ast$.

In addition, after making obvious changes, the same argument derives the analogous claims  for $\mathscr{F}^\vee$  from  the second isomorphism in Proposition \ref{local Tate duality thm}\,(ii). Since this is a routine matter, we leave details to an interested reader. \end{proof}

\begin{remark}\label{infty relaxed rem} If $\mathscr{F}$ is $\infty$-relaxed, then $\mathrm{R}\Gamma^{\mathrm{AV}}_{\mathscr{F},\mathscr{F}^\vee}(k,A) = \RHom_R(\Delta_\infty(k,A^\vee(1)),E_R (\mathbb{k}))[-2]$ and, by explicit computation, one checks that this is  represented by the complex 
\[ \bigl(\bigoplus_{\q\in \Pi_k^\CC}A(-1)[-2]\bigr) \oplus \bigl(\bigoplus_{\q\in \Pi_k^\RR}[ A(-1)\xrightarrow{1-c_\q} A(-1) \xrightarrow{1+c_\q} A(-1) \xrightarrow{1-c_\q} A(-1) \xrightarrow{1+c_\q}\cdots ]\bigr).\]
Here, for each $\q \in \Pi_k^\RR$, we write $c_\q$ for the non-trivial element of $G_{k_\q}$ and the first term of the displayed complex occurs in degree $2$. If $\mathscr{F}$ is $\infty$-relaxed and either $R$ is self-injective or $A$ is a free $R$-module, then $\mathrm{R}\Gamma^{\mathrm{AV}}_{\mathscr{F},\mathscr{F}^\ast}(k,A)$ also coincides with the above complex.  \end{remark} 

Let $\cF_1$ and $\cF_2$ be Mazur--Rubin structures on $A$ with $\cF_1\le \cF_2$. Then Mazur and Rubin \cite[Th.\@ 2.3.4]{MazurRubin04} have shown that global duality gives rise to a canonical exact sequence of $R$-modules  
\begin{equation}\label{mr duality}  H^1_{\cF_1} (k, A) \hookrightarrow H^1_{\cF_2} (k, A) \to \bigoplus_{\q\in S(\cF_1)} \frac{H^1_{\cF_2} (k_\q, A)}{H^1_{\cF_1} (k_\q, A)}
 \to H^{1}_{\cF_1^\ast} (k, A^\vee (1))^\vee \twoheadrightarrow H^{1}_{\cF_2^\ast} (k, A^\vee (1))^\vee.\end{equation}

We next use Proposition \ref{av prop} to establish an analogue of this result for Nekov\'{a}\v{r} structures.

\begin{thm} \label{global duality thm} Let $\mathscr{F}_1$ and $\mathscr{F}_2$ be $\infty$-relaxed Nekov\'{a}\v{r} structures on $A$ with $\mathscr{F}_1\le \mathscr{F}_2$ and set $B \coloneqq A^\vee(1)$. For $i\in 
\{1,2\}$ write  $ \tilde H^1_{\mathscr{F}^\vee_i} (k_\q, B)^\vee$  for the kernel of the morphism 
\[ H^1_{\mathscr{F}^\vee_i} (k_\q, B)^\vee \to 
H^2\bigl(\mathrm{R}\Gamma^{\mathrm{AV}}_{\mathscr{F}_i}(k,A)\bigr) \]
induced by the exact triangle in Proposition \ref{av prop}. Then the following claims are valid.

\begin{itemize}
    \item[(i)] There exists a canonical long exact sequence of $R$-modules 
\begin{multline*} \bigoplus H^0\bigl(\mathrm{R}\Gamma_{\mathscr{F}_2/\mathscr{F}_1}(k_\q,A)\bigr)\to H^1_{\mathscr{F}_1} (k, A) \to H^1_{\mathscr{F}_2} (k, A)\to \bigoplus H^1\bigl(\mathrm{R}\Gamma_{\mathscr{F}_2/\mathscr{F}_1}(k_\q,A)\bigr)\\ \to \tilde H^1_{\mathscr{F}^\vee_1} (k_\q, B)^\vee \to \tilde H^{1}_{\mathscr{F}_2^\vee} (k,B)^\vee \to \bigoplus H^0\bigl(\mathrm{R}\Gamma_{\mathscr{F}_1^\vee/\mathscr{F}^\vee_2}(k_\q,B)\bigr)^\vee\end{multline*}
in which all direct sums are taken over $\q\in S(\mathscr{F}_1)$.
\item[(ii)] Set $\cF_1 \coloneqq h(\mathscr{F}_1)$ and  $\cF_2 \coloneqq h(\mathscr{F}_2)$ and assume the following hypotheses:
\begin{itemize}
\item[(a)] the maps $H^1(\theta_{\mathscr{F}_1,\q} )$, $H^1(\theta_{\mathscr{F}_2,\q} )$ and $H^2(j_{\mathscr{F}_1,\mathscr{F}_2,\q})$ are injective for every  $\q\in S(\mathscr{F}_1)$:
\item[(b)] $\Pi_k^\RR = \emptyset$ if $p=2$. 
\end{itemize}
Then there exists a canonical exact commutative diagram of $R$-modules
\begin{cdiagram}[column sep=tiny]
H^1_{\mathscr{F}_1} (k, A) \arrow["\alpha_1"]{r} \arrow[twoheadrightarrow, "\beta_1"]{d} & H^1_{\mathscr{F}_2} (k, A) \arrow{r} \arrow[twoheadrightarrow, "\beta_2"]{d} & \bigoplus H^1\bigl(\mathrm{R}\Gamma_{\mathscr{F}_2/\mathscr{F}_1}(k_\q,A)\bigr) \arrow{r} & \tilde H^1_{\mathscr{F}^\vee_1} (k_\q, B)^\vee \arrow["\alpha_2"]{r} & \tilde H^{1}_{\mathscr{F}_2^\vee} (k, B)^\vee\\
H^1_{\cF_1} (k, A) \arrow[hookrightarrow]{r} & H^1_{\cF_2} (k, A) \arrow{r} & \bigoplus \frac{H^1_{\cF_2} (k_\q, A)}{H^1_{\cF_1} (k_\q, A)} \arrow[equals, "\beta_3"]{u} \arrow{r} & H^1_{\cF^\vee_1} (k_\q, B)^\vee \arrow[twoheadrightarrow]{r} \arrow[hookrightarrow, "\beta_4"]{u} & H^{1}_{\cF_2^\vee} (k, B)^\vee . \arrow[hookrightarrow, "\beta_5"]{u}
\end{cdiagram}%
Here both direct sums are taken over $\q\in S(\mathscr{F}_1)\cup S(\mathscr{F}_2) = S(\cF_1)\cup S(\cF_2)$, the first row is induced by the exact sequence in (i), the second row is the relevant case of (\ref{mr duality}) and all vertical maps are described in the course of the argument below. In addition, the map $\alpha_1$ is injective (resp.\@ $\alpha_2$ is surjective) if for each $\q \in S(\mathscr{F}_1)\cup S(\mathscr{F}_2)$ the map $H^0(j_{\mathscr{F}_1,\mathscr{F}_2,\q})$ is surjective (resp.\@ $H^2(j_{\mathscr{F}_1,\mathscr{F}_2,\q})$ is surjective and $H^3(j_{\mathscr{F}_1,\mathscr{F}_2,\q})$ is injective).
\end{itemize}
\end{thm}

\begin{proof}
 We set $S \coloneqq S(\mathscr{F}_1)\cup S(\mathscr{F}_2)$. We then recall that, since $\mathscr{F}_1\le \mathscr{F}_2$, for each $\q\in S$ we are given a morphism $j_{\mathscr{F}_1,\mathscr{F}_2,\q} \: \RGamma_{\mathscr{F}_1}(k_\q,A) \to \RGamma_{\mathscr{F}_2}(k_\q,A)$ in $D(R)$ and $\mathrm{R}\Gamma_{\mathscr{F}_2/\mathscr{F}_1}(k_\q,A)$ denotes its mapping cone. 

Now, since $\mathscr{F}_1$ and $\mathscr{F}_2$ are both $\infty$-relaxed, the morphism $j_{\mathscr{F}_1,\mathscr{F}_2,\q}$ is an isomorphism for each $\q \in \Pi_k^\infty$. In addition, if we set $B \coloneqq A^\vee(1)$, then for each $\q \in S\setminus \Pi_k^\infty$, the definition of the dual condition $(\RGamma_{\mathscr{F}^\vee}(k_\q,B),\theta_{\mathscr{F}_i^\vee,\q})$ implies that $j_{\mathscr{F}_1,\mathscr{F}_2,\q}$ induces a canonical morphism 
\[ j_{\mathscr{F}^\vee_2,\mathscr{F}^\vee_1,\q} \: \mathrm{R}\Gamma_{\mathscr{F}^\vee_2}(k_\q,B) \to \mathrm{R}\Gamma_{\mathscr{F}^\vee_1}(k_\q,B)\]
whose mapping cone is isomorphic to $\mathrm{R}\Gamma_{\mathscr{F}_2/\mathscr{F}_1}(k_\q,A)^\vee[-2]$. In particular, if we use these morphisms to regard $\mathscr{F}^\vee_2$ as a refinement of $\mathscr{F}_1^\vee$, then we obtain a commutative diagram of exact triangles in $D(R)$ of the form 
\begin{equation*}\xymatrix@R=1em@C=1.5em{
\mathrm{R}\Gamma_{\mathscr{F}_1}(k,A) \ar[d] \ar[r] & \mathrm{R}\Gamma_{\mathscr{F}_1^\vee}(k,A^\vee(1))^\vee[-3] \ar[d] \ar[r] &  \RGamma^{\mathrm{AV}}_{\mathscr{F}_1,\mathscr{F}_1^\vee}(k,A)\ar@{=}[d] \ar[r] &\cdot\\
\RGamma_{\mathscr{F}_2}(k,A) \ar[d] \ar[r] & \RGamma_{\mathscr{F}_2^\vee}(k,A^\vee(1))^\vee[-3] \ar[d] \ar[r] &  \RGamma^{\mathrm{AV}}_{\mathscr{F}_2,\mathscr{F}_2^\vee}(k,A) \ar[r] &\cdot\\
\bigoplus_{\q\in S}\mathrm{R}\Gamma_{\mathscr{F}_2/\mathscr{F}_1}(k_\q,A) \ar[r]^{\hskip-0.2truein\cong}_{\hskip-0.2truein\delta}\ar[d]& \bigoplus_{\q\in S}\mathrm{R}\Gamma_{\mathscr{F}^\vee_1/\mathscr{F}^\vee_2}(k_\q,A)^\vee[-2]\ar[d]\\
\cdot &\cdot}\end{equation*} 
Here the first and second rows are the respective exact triangles from Proposition \ref{av prop}  and the equality occurs since $\mathscr{F}_1$ and $\mathscr{F}_2$ are $\infty$-relaxed (cf.\@ Remark \ref{infty relaxed rem}). In addition, the first column is the exact triangle in Lemma \ref{basic properties}\,(ii) with $(\mathscr{F}',\mathscr{F}) = (\mathscr{F}_1,\mathscr{F}_2)$ and the second is the image under the exact functor $X \mapsto X^\ast[-3]$ of the corresponding exact triangle with $(\mathscr{F}',\mathscr{F}) = (\mathscr{F}_2^\vee,\mathscr{F}_1^\vee)$. Finally, $\delta$ is the direct sum over $\q$ of the local duality isomorphisms described above. The exact sequence in (i) is now obtained by combining the long exact cohomology sequences of the first two columns. \\
In regard to (ii), we first note that the assumed injectivity, for each $\q \in S$, of both $H^1(\theta_{\mathscr{F}_1,\q})$ and $H^1(\theta_{\mathscr{F}_2,\q})$ implies that $H^1(j_{\mathscr{F}_1,\mathscr{F}_2,\q})$ is injective, and then combines with the assumed injectivity of $H^2(j_{\mathscr{F}_1,\mathscr{F}_2,\q})$ to induce an identification $H^1\bigl(\mathrm{R}\Gamma_{\mathscr{F}_2/\mathscr{F}_1}(k_\q,A)\bigr)\cong H^1_{\cF_2} (k_\q, A)/H^1_{\cF_1} (k_\q, A)$. We take the map $\beta_3$ in the claimed diagram to be the direct sum over $\q \in S$ of these identifications. We next define $\beta_1$ and $\beta_2$ to be the Matlis duals of the first map in the second exact sequence of Lemma \ref{comparison selmer structures}\,(i) with $\mathscr{F}$ taken to be $\mathscr{F}_1$ and $\mathscr{F}_2$ respectively. We also note that, for both $i =1$ and $i=2$, the same exact sequence in Lemma \ref{comparison selmer structures}\,(i) with $\mathscr{F}$ taken to be $\mathscr{F}_i^\vee$ induces an exact sequence 
\[ 0 \to H^1_{\cF_i^\vee}(k,B)^\vee \xrightarrow{\gamma_i} H^1_{\mathscr{F}_i^\vee}(k,B)^\vee \to \bigoplus_{\q \in \Pi_k^\infty} H^0(k_\q,B)^\vee\oplus \bigoplus_{\q \in S(\mathscr{F}_i)\setminus \Pi_k^\infty} \left(\frac{H^0(k_\q,B)}{\im(H^0(\theta_{\mathscr{F}_i^\vee,\q}))}\right)^\vee\]
in which in the direct sum we have also used the fact that $\mathscr{F}_i^\vee$ is $\infty$-strict so that 
$H^0(\theta_{\mathscr{F}_i^\vee,\q})$ is the zero map for each $\q \in \Pi_k^\infty$. In addition, under condition (c), there are natural isomorphisms  
\begin{align*} {\bigoplus}_{\q \in \Pi_k^\infty} H^0(k_\q,B)^\vee \cong&\, {\bigoplus}_{\q \in \Pi_k^\CC}B^\vee \oplus {\bigoplus}_{\q \in \Pi_k^\RR} \bigl((1+c_\q)B\bigr)^\vee\\
\cong&\, {\bigoplus}_{\q \in \Pi_k^\CC}A \oplus  {\bigoplus}_{\q \in \Pi_k^\RR} (1-c_\q)A\\
= &\, H^2\bigl(\mathrm{R}\Gamma^{\mathrm{AV}}_\mathscr{F}(k,A)\bigr),\end{align*}
where the equality follows from Remark \ref{infty relaxed rem}. It follows that the map $\gamma_i$ factors through the submodule $\tilde H^1_{\mathscr{F}^\vee_1} (k_\q, B)^\vee$ of $H^1_{\mathscr{F}^\vee_1} (k_\q, B)^\vee$ and so, in the claimed diagram, we can take $\beta_4$ and $\beta_5$ to be the maps that are respectively induced by $\gamma_1$ and $\gamma_2$. With these specifications of the maps $\beta_i$, it is then a straightforward exercise to check that the claimed diagram commutes, as required. \\
Finally, we note that the assumed surjectivity of $H^0(j_{\mathscr{F}_1,\mathscr{F}_2,\q})$ and injectivity of $H^1(j_{\mathscr{F}_1,\mathscr{F}_2,\q})$ for all $\q \in S$ would combine to imply $H^0\bigl(\mathrm{R}\Gamma_{\mathscr{F}_2/\mathscr{F}_1}(k_\q,A)\bigr) =(0)$ for such $\q$ and hence that $\alpha_1$ is injective as a consequence of the exact sequence in (i). In a similar way, the  surjectivity of $H^2(j_{\mathscr{F}_1,\mathscr{F}_2,\q})$ and injectivity of $H^3(j_{\mathscr{F}_1,\mathscr{F}_2,\q})$ for all $\q \in S$ would imply $H^2\bigl(\mathrm{R}\Gamma_{\mathscr{F}_2/\mathscr{F}_1}(k_\q,A)\bigr) =(0)$ for such $\q$ and hence that $\alpha_2$ is surjective. This verifies the final assertion of (ii). 
\end{proof}

\subsection{Perfect Selmer complexes}\label{psc section}

Throughout this subsection,  we assume $(k,p)$ satisfies (\ref{p=2 condition}). We also fix a Nekov\'a\v{r} structure $\mathscr{F}$ on a finitely generated free $R$-module $A$ and assume the following hypothesis to be valid. 

\filbreak
\begin{hypothesis}\label{fF hyp} The following conditions are satisfied: 
\begin{itemize}
\item[(a)] $\mathscr{F}$ is $\infty$-relaxed;
\item[(b)] For every $\q \in S(\mathscr{F})\setminus \Pi_k^\infty$, the following conditions are satisfied:
\begin{itemize}
    \item[(i)] $\RGamma_\mathscr{F}(k_\q, A)$ belongs to $D_{[0,2]}^{\mathrm{perf}}(R)$;
    \item[(ii)] $H^0(\theta_{\mathscr{F},\q})$ is injective;
\end{itemize}
\item[(c)] For some $\q_0\in S(\mathscr{F})\setminus \Pi_k^\infty$, the map $H^0(\theta_{\mathscr{F},\q_0})$ is $0$.
\item[(d)] For some $\q_1\in S(\mathscr{F})\setminus \Pi_k^\infty$, the map $H^2(\theta_{\mathscr{F},\q_1})$ is bijective. 
\end{itemize}
\end{hypothesis}

For each finite subset $S$ of $\Pi_k$ we also use the map $\lambda^0_S(\mathscr{F}^\vee)$ from (\ref{lambda 0 def}) to define  an $R$-module  
\[ X_S(\mathscr{F}) \coloneqq \ker(\lambda^0_S(\mathscr{F}^\vee)^\vee).\]
We then abbreviate $X_{S(\mathscr{F})}(\mathscr{F})$ to $X(\mathscr{F})$.

The following result constructs a family of complexes that plays a key role in our theory.  

\begin{prop} \label{construction complex} Assume $(k,p)$ satisfies (\ref{p=2 condition}), that $A$ is a free $R$-module of finite rank, and that $\mathscr{F}$ satisfies Hypothesis \ref{fF hyp}. Then the following claims are valid. 
\begin{romanliste}
\item The complex 
\[ C(\mathscr{F}) \coloneqq \RHom_R ( \RGamma_{\mathscr{F}^\ast} (k, A^\ast (1)), R) [-2]
    \]
is a well-defined object of $D^{\mathrm{perf}}(R)$ such that, in $K_0(R)$, one has 
\[ \bm{\chi}_R\bigl(C(\mathscr{F})\bigr) = {\sum}_{\q\in S(\mathscr{F})\setminus \Pi_k^\infty}\bm{\chi}_R\bigl(\RGamma_{\mathscr{F}^\ast}(k_\q,A^\ast(1))\bigr).\] 
\item $C (\mathscr{F})$ is canonically isomorphic to $\RHom_R ( \RGamma_{\mathscr{F}^\vee} (k, A^\vee (1)), E_R (\mathbb{k})) [-2]$ in $D(R)$. In particular, in each degree $i$, one has 
    $H^i(C(\mathscr{F})) = H^{2-i}_{\mathscr{F}^\vee} (k, A^\vee (1))^\vee$.
    
 \item Fix a morphism $R\to R'$ of rings satisfying (\ref{ring condition}). Then there exists a natural isomorphism $C(\mathscr{F}) \otimes_R^\mathbb{L} R' \cong C(\mathscr{F}\otimes_R R')$ in $D^{\mathrm{perf}}(R')$, where $\mathscr{F}\otimes_R R'$ is the Nekov\'a\v{r} structure  on $A\otimes_RR'$ defined in Example \ref{remark selmer}\,(iii). 
    \item  Lemma \ref{limit result} defines an object $\varprojlim_{n \in \N} C(\mathscr{F}_{A_n})$ of $D^{\mathrm{perf}}(R)$ that is naturally isomorphic to $C(\mathscr{F})$. In particular, in each degree $i$, one has 
    $H^i(C(\mathscr{F})) = {\varprojlim}_{n \in \N} H^{2-i}_{(\mathscr{F}_{A_n})^\ast} (k, A_n^\ast (1))^\ast$. 
    \item $C(\mathscr{F})$ is acyclic outside degrees $0$ and $1$. There exists a canonical identification  
$H^0 (C(\mathscr{F}))$ $ = H^1_{\mathscr{F}} (k, A)$ and a canonical exact commutative diagram 
\begin{equation}\label{new diagram}\xymatrix{ 
X_{S(\mathscr{F})\setminus \Pi_k^\infty}(\mathscr{F})\\
H^{2}_{\mathscr{F}}(k,A) \ar@{^{(}->}[r] \ar@{->>}[u]^{\alpha_1} &  H^1 (C(\mathscr{F})) \ar@{->>}[r] \ar@{=}[d] & {\bigoplus}_{\q \in \Pi_k^\infty} H^0(k_\q,A^\vee(1))^\vee \\
H^1_{h(\mathscr{F})^\vee} (k, A^\vee (1))^\vee \ar@{^{(}->}[r] \ar@{^{(}->}[u]^{\alpha_2} &  H^1 (C(\mathscr{F})) \ar@{->>}[r] &  X(\mathscr{F}). \ar@{->>}[u]^{\alpha_3}}\end{equation} 
    \item Let $\Sigma$ be a finite subset of $\Pi_K$. Then, for the $\Sigma$-comodification $\mathscr{F}^\Sigma$ of $\mathscr{F}$ defined in Example \ref{lesc}, 
there exists a canonical exact triangle in  $D^{\mathrm{perf}} (R)$
    \begin{equation} \label{triangle changing S}
    C (\mathscr{F}) \to C(\mathscr{F}^\Sigma) \xrightarrow{(\rho_\q)_\q} {\bigoplus}_{\q \in \Sigma \setminus S(\mathscr{F})} \mathrm{R}\Gamma_f (k_\q, A^\ast(1))^\ast[-1] \to \cdot
        \end{equation}
  \end{romanliste}
\end{prop}

\begin{proof} We set $B \coloneqq A^\ast(1)$. Since $\mathscr{F}^\ast$ is $\infty$-strict and $S(\mathscr{F}^\ast) = S(\mathscr{F})$, Lemma \ref{basic properties}\,(iii) implies $\mathscr{F}^\ast$  is perfect if and only if, for every $\q\in S(\mathscr{F})\setminus \Pi_k^\infty$, the complex $\RGamma_{\mathscr{F}^\ast}(k_\q,B)$ belongs to $D^{\mathrm{perf}}(R)$. Since the latter condition is satisfied as a consequence of Hypothesis \ref{fF hyp}\,(b)(i), it follows that $\RGamma_{\mathscr{F}^\ast} (k, B)$, and hence also $C(\mathscr{F})$, belongs to $D^{\mathrm{perf}}(R)$. For similar reasons, one also has 
\[ \bm{\chi}_R(C(\mathscr{F})) = \bm{\chi}_R(\RGamma_{\mathscr{F}^\ast}(k_\q,B)) = {\sum}_{\q \in S(\mathscr{F})\setminus \Pi_k^\infty}\bm{\chi}_R(\RGamma_{\mathscr{F}^\ast}(k_\q,B),\]
where the  second equality follows from the formula in Lemma \ref{basic properties}\,(iii) with $(A,\mathscr{F})$ replaced by $(B,\mathscr{F}^\ast)$. This proves (i).  

To prove (ii), we use the following diagram in $D(R)$
\begin{cdiagram}[row sep=small, column sep=small]
    C(\mathscr{F}) \arrow{r} \arrow[dashed]{d} & \mathrm{R}\Gamma^{\mathrm{AV}}_{\mathscr{F},\mathscr{F}^\ast}(k,A)[1] \arrow{r}{\theta_1} \arrow[equals]{d} & \mathrm{R}\Gamma_\mathscr{F}(k,A)[2] \arrow{r} \arrow[equals]{d} & \cdot  \\ 
   \RHom_R (\mathrm{R}\Gamma_{\mathscr{F}^\vee}(k,A^\vee(1)), E_R (\mathbb{k}))[-2] \arrow[r] & \mathrm{R}\Gamma^{\mathrm{AV}}_{\mathscr{F},\mathscr{F}^\vee}(k,A)[1] \arrow{r}{\theta_2} & \mathrm{R}\Gamma_\mathscr{F}(k,A)[2] \arrow{r} & \cdot
\end{cdiagram}%
Here the upper and lower rows are equivalent to the respective exact triangles in Proposition~\ref{av prop}~(ii) and the left hand equality follows from Remark \ref{infty relaxed rem}. In addition, an analysis of the construction of these triangles shows that the morphisms $\theta_1$ and $\theta_2$ coincide as they are both induced by the morphism ${\bigoplus}_{\q\in \Pi_k^\infty}\mathrm{R}\Gamma_{\!/\mathscr{F}}(k_\q,A) \to \RGamma_\mathscr{F}(k,A)[1]$ that occurs in the exact triangle (\ref{mapping fibre2}). The Octahedral axiom therefore implies the existence of a dashed arrow that makes the above diagram into a morphisms of exact triangles and hence is an isomorphism in $D(R)$. This proves the first assertion of (ii) and then the second assertion follows immediately from the fact that Matlis duality is an exact functor. 

To prove (iii) we note that the argument of Lemma \ref{some basic duality properties of Nekovar structures}\,(i) implies an equality $\mathscr{F}^\ast\otimes_RR'= (\mathscr{F}\otimes_RR')^\ast$ of  Nekov\'a\v{r} structures on $A^\ast(1)\otimes_RR' = (A\otimes_RR')^*(1)$. After taking this into account, the exact triangle (\ref{cs versus Selmer}) with $\mathscr{F}$ replaced by $\mathscr{F}^\ast$ combines with Lemma \ref{flach result}\,(vi) (with $\mathscr{F}(-)$ taken to be $\RGamma_\mathrm{c} (\mathcal{O}_{k,S},-)$) and the explicit definition of the local conditions for $\mathscr{F}^\ast\otimes_RR'$ to imply the existence of a natural isomorphism in $D(R')$
\[ \RGamma_{\mathscr{F}^\ast} (k, A^\ast (1))\otimes_R^\mathbb{L} R' \cong \RGamma_{(\mathscr{F}\otimes_RR')^\ast} (k, (A\otimes_RR')^\ast (1)). \]
Upon applying the exact functor $\RHom_{R'}(-, R') [-2]$ to this isomorphism one obtains the claimed isomorphism in (ii).

To prove (iv) we first apply (iii) with $R \to R'$ taken to be $R \to R/\a_n$ to deduce that for each $n$ there exists a natural isomorphism $C(\mathscr{F}) \otimes_R^\mathbb{L} R/\a_n \cong C(\mathscr{F}_{A_n})$ in $D^{\mathrm{perf}}(R/\a_n)$. Given this family of isomorphisms the first assertion of (iii) is proved by mimicking the argument of Proposition \ref{Selmer complex control thm} after replacing $\RGamma_\mathscr{F}(k,A)$ by $C(\mathscr{F})$. From the ismorphism 
$C(\mathscr{F}) \cong \varprojlim_{n \in \N} C(\mathscr{F}_{A_n})$ one then derives, in each degree $i$, an identification
\[ H^i (C(\mathscr{F})) = H^i({\varprojlim}_{n \in \N} C(\mathscr{F}_{A_n})) = {\varprojlim}_{n \in \N} H^i (C(\mathscr{F}_{A_n})) = {\varprojlim}_{n \in \N} H^{2-i}_{(\mathscr{F}_{A_n})^\ast} (k, A_n^\ast (1))^\ast.
    \]
Here the second equality is valid since inverse limits are exact on the category of finite abelian groups and the third  since $R_n$ is self-injective and so taking duals commutes with taking cohomology. This proves (iv). 

Turning to the proof of (v), we note that, since $\Pi_k^\RR = \emptyset$ if $p=2$, Hypothesis \ref{fF hyp}\,(a) implies $H^i_{\mathscr{F}}(k_\q,A) = (0)$ for all $i > 2$ and all $\q \in \Pi_k^\infty$. This fact combines with Remark \ref{infty relaxed rem} to imply $H^i(\mathrm{R}\Gamma^{\mathrm{AV}}_\mathscr{F}(k,A)) = (0)$ for $i \not= 2$ and also combines with  Hypothesis \ref{fF hyp}\,(b) and (d) and Lemma \ref{basic properties}\,(iv) to imply that $H^i_\mathscr{F}(k,A)=(0)$ for $i < 0$ and $i > 2$. These observations in turn combine with the long cohomology exact sequence of the exact triangle in Proposition \ref{av prop} to imply  $H^i(C(\mathscr{F})) = (0)$ for $i < -1$ and $i >1$, to identify $H^i (C(\mathscr{F}))$ with $H^{i+1}_\mathscr{F} (k, A)$ for $i \in \{-1,0\}$ and to give a short exact sequence that forms the central row of (\ref{new diagram}). In addition,  since Hypothesis \ref{fF hyp}\,(c) implies injectivity of $\lambda^0_{S(\mathscr{F})}(\mathscr{F})$, Hypothesis \ref{fF hyp}\,(b)\,(ii) combines with the first exact sequence in Lemma \ref{comparison selmer structures}\,(ii) to imply that $H^{0}_\mathscr{F}(k,A) = (0)$. To complete the proof of (ii), it is therefore enough to construct the diagram (\ref{new diagram}). To do this, we note that $H^1 (C_\mathscr{F} (A)) \cong H_{\mathscr{F}^\vee}^1 (k, A^\vee (1))^\vee$ by (iv) and the fact that taking Matlis duals is exact. Given this identification, we obtain the exact sequence that forms the lower row of (\ref{new diagram}) by first taking Matlis duals in 
the second exact sequence in Lemma \ref{comparison selmer structures}\,(i) with $(A,\mathscr{F})$ taken to be $(A^\vee (1),\mathscr{F}^\vee)$, and then recalling that $h(\mathscr{F}^\vee) = h (\mathscr{F})^\vee$ by Lemma \ref{some basic duality properties of Nekovar structures}\,(iii).\\
     We now take the map $\alpha_3$ in (\ref{new diagram}) to be the map $\alpha$ in the exact sequence of Lemma \ref{comparison selmer structures}\,(iii) with $(\mathscr{F},S)$ replaced by $(\mathscr{F}^\vee,S(\mathscr{F}^\vee))$. Then, with this definition, the commutativity of the second square in (\ref{new diagram}) is clear and this has two consequences: firstly, the map $\alpha_3$ is surjective (as can also be seen directly from Lemma \ref{comparison selmer structures}\,(iii) since the conditions \ref{fF hyp}\,(b)(iii) and (d) combine to imply $H^0(\theta_{\mathscr{F}^\vee,\q_1})$ is the zero map) and there exists an injective morphism $\alpha_2$ that makes the first square of (\ref{new diagram}) commute. The existence of a surjective morphism $\alpha_1$ that makes the first column of (\ref{new diagram}) a short exact sequence now follows by applying the Snake Lemma to the lower two rows of the diagram and taking account of the exact sequence in Lemma \ref{comparison selmer structures}\,(ii) with $(\mathscr{F},S,S')$ taken to be $(\mathscr{F}^\vee,S(\mathscr{F}^\vee),\Pi_k^\infty)$. This proves (v). 

To prove (vi) we note that $(\mathscr{F}^\Sigma)^\ast$ coincides with the Nekov\'a\v{r} structure $(\mathscr{F}^\ast)_\Sigma$ on $B$. Given this, the claimed exact triangle is directly obtained by applying the exact functor $X \mapsto X^\ast[-2]$ to the exact triangle (\ref{sigma triangle}) with $(A,\mathscr{F})$ replaced by $(B,\mathscr{F}^\ast)$. 
\end{proof}

The following result establishes some properties of the modules $X_S (\mathscr{F})$ that will be needed in later arguments.

\begin{lem} \label{limit Fitt X} Let $\mathscr{F}$ be a Nekov\'a\v{r} structure  on a finitely generated $R$-module $A$ and $S$ a finite subset of $\Pi_k$. Then the $R$-module 
\[ X_S (\mathscr{F})\coloneqq \ker (\lambda^0_S (\mathscr{F}^\vee)^\vee)\] 
is finitely generated.
Further, if $(k,p)$ satisfies (\ref{p=2 condition}), $A$ is a free $R$-module and $\RGamma_\mathscr{F} (k_\q, A)$ belongs to $D^\mathrm{perf}_{[0, 2]} (R)$ for all $\q \in S$, then the following claims are also valid.
\begin{romanliste}
    \item Every morphism $R \to R'$ of rings satisfying condition (\ref{ring condition}) induces a surjective map of $R'$-modules $X_S (\mathscr{F}) \otimes_R R' \twoheadrightarrow X_S (\mathscr{F} \otimes_R R')$.
    \item The maps from (i) combine to give an isomorphism  $X_S (\mathscr{F}) \cong {\varprojlim}_{i \in \N} X_S (\mathscr{F} \otimes_R R_i)$ of $R$-modules in which all of the transition morphisms in the inverse limit are surjective.
\end{romanliste}
\end{lem}

\begin{proof} The definition of $\lambda^0_S (\mathscr{F}^\vee)$ directly implies that 
\[X_S (\mathscr{F}) \subseteq \bigoplus_{\q \in S} \bigl( H^0 (k_\q, A^\vee (1)) / \im (H^0 (\theta_{\mathscr{F}^\vee, \q}))\bigr)^\vee \subseteq \bigoplus_{\q \in S}  H^0 (k_\q, A^\vee (1))^\vee.\]
In addition, for each $\q \in S$, the Matlis dual of the inclusion $H^0 (k_\q, A^\vee (1)) \subseteq A^\vee (1)$ is a surjective map $A ( - 1) \twoheadrightarrow H^0 (k_v, A^\vee (1))^\vee$. The latter map implies that each $R$-module $H^0 (k_\q, A^\vee (1))^\vee$ is finitely generated and hence, since $S$ is finite (and $R$ is Noetherian), the displayed  inclusions imply  $X_S (\mathscr{F})$ is also finitely generated over $R$, as claimed.

In the rest of the argument we assume all of the stated hypotheses for (i) and (ii). We then fix $\q \in S$ and note that the definition of $\mathscr{F}^\vee$ combines with local duality to give an isomorphism 
\begin{align}\label{missing duality iso}
( \coker H^0 (\theta_{\mathscr{F}^\vee,\q}))^\vee &\cong \ker \big ( H^0 (k_\q, A^\vee (1))^\vee \xrightarrow{H^0 (\theta_{\mathscr{F}^\vee, \q})^\vee} H^0_{\mathscr{F}^\vee} (k_\q, A^\vee (1))^\vee \big)\notag\\
& \cong \ker \big ( H^2 (k_\q, A) \to H^2_{\!/\mathscr{F}} (k_\q, A) \big)\notag\\
& = \im \big ( H^2_\mathscr{F} (k_\q, A) \xrightarrow{H^2 (\theta_{\mathscr{F},\q})}H^2 (k_\q, A) \big).
\end{align}
This isomorphism combines with the analogous isomorphism for $\mathscr{F}' \coloneqq \mathscr{F}\otimes_RR'$ to induce an exact commutative diagram
\begin{cdiagram}[column sep=small, row sep=small]
H^1_{\!/ \mathscr{F}} (k_\q, A) \otimes_R R' \arrow{r} \arrow{d} & 
    H^2_\mathscr{F} (k_\q, A) \otimes_R R' \arrow{d}{\cong} \arrow{rrrr}{H^2 (\theta_{\mathscr{F},\q})} & & & & ( \coker H^0 (\theta_{\mathscr{F}^\vee,\q}))^\vee \otimes_R R' \arrow{r} \arrow{d} & 0 \\ 
   H^1_{\!/ \mathscr{F}'} (k_\q, A \otimes_R R') \arrow{r} & H^2_{\mathscr{F}'}(k_\q, A \otimes_R R') \arrow{rrrr}{H^2 (\theta_{\mathscr{F}',\q})} & & & & ( \coker H^0 (\theta_{(\mathscr{F}')^\vee,\q}))^\vee \arrow{r} & 0.
\end{cdiagram}%
Here the bijectivity of the central vertical map follows from (the argument of) Lemma \ref{how the cohomology base changes lemma}\,(ii) and the fact $\RGamma_\mathscr{F} (k_\q, A)$ belongs to $D^\mathrm{perf}_{[0, 2]} (R)$, and so the third vertical map is surjective. 

In addition, global duality induces (via Proposition \ref{local Tate duality thm}\,(ii)) an isomorphism  
\begin{equation}\label{second missing duality iso}H^0 (k, A^\vee (1))^\vee \cong H^3 (\widetilde{\RGamma}_\mathrm{c}  (k, A))\end{equation}
and since $\widetilde{\RGamma}_\mathrm{c}  (k, A)$ belongs to $D^\mathrm{perf}_{[0, 3]} (R)$ (by condition (\ref{p=2 condition})), we can  similarly deduce that the map $H^0 (k, A^\vee (1))^\vee \otimes_R R' \to H^0 (k, (A\otimes_R R')^\vee (1))^\vee$ is bijective.\\
Now, from the respective definitions of the modules $X_S (\mathscr{F})$ and $X_S (\mathscr{F}')$, one obtains an exact commutative diagram of the form
\begin{cdiagram}[column sep=small, row sep=small]
    & X_S (\mathscr{F}) \otimes_R R' \arrow{r} \arrow{d} & \big( \textstyle \bigoplus_{\q \in S}  (\coker H^0 (\theta_{\mathscr{F}^\vee,\q}))^\vee \big) \otimes_R R' \arrow{r} \arrow{d} & H^0 (k, A^\vee (1))^\vee \otimes_R R' \arrow{r} \arrow{d} & 0 \\ 
    0 \arrow{r} & X_S (\mathscr{F}') \arrow{r} & \textstyle \bigoplus_{\q \in S}  (\coker H^0 (\theta_{(\mathscr{F}')^\vee,\q}))^\vee \arrow{r} & H^0 (k, (A \otimes_R R')^\vee (1))^\vee \arrow{r} & 0
\end{cdiagram}%
in which the second and third vertical maps are as above, and the first vertical map is induced by commutativity of the second square. In particular, since we have seen that the second and third vertical maps are respectively surjective and bijective, the Snake Lemma implies that the first vertical arrow is surjective, as required to prove (i). \\
As for (ii), 
by taking $R' = R_i$ in the above diagram and passing to the limit (over $i$), we obtain a diagram of the form
\begin{cdiagram}[column sep=tiny, row sep=small]
   0 \arrow{r} & X_S (\mathscr{F}) \arrow{r} \arrow{d} &   \bigoplus_{\q \in S}  (\coker H^0 (\theta_{\mathscr{F}^\vee,\q}))^\vee \arrow{r} \arrow{d} & H^0 (k, A^\vee (1))^\vee \arrow{r} \arrow{d} & 0 \\ 
    0 \arrow{r} & \varprojlim_i X_S (\mathscr{F} \otimes_R R_i) \arrow{r} &  \bigoplus_{\q \in S}  \varprojlim_i (\coker H^0 (\theta_{(\mathscr{F} \otimes_R R_i)^\vee,\q}))^\vee \arrow{r} & \varprojlim_i H^0 (k, (A \otimes_R R_i)^\vee (1))^\vee \arrow{r} & 0,
\end{cdiagram}%
in which the upper row is the defining exact sequence for $X_S (\mathscr{F})$.
We are therefore reduced to proving that the second and third vertical maps in this diagram are bijective. 
As in the discussion of  (i), we are then further reduced (via duality isomorphisms of the form (\ref{missing duality iso}) and (\ref{second missing duality iso}) with $\mathscr{F}$ and $A$ replaced by each $\mathscr{F}\otimes_RR_i$ and $A\otimes_RR_i$) to showing that there are natural isomorphisms
\[ \varprojlim_{i \in \N} H^2_{\mathscr{F} \otimes_R R_i} (k_\q, A \otimes_R R_i) \cong 
H^2_\mathscr{F} (k_\q, A), \quad \varprojlim_{i \in \N } H^2 (k_\q, A \otimes_R R_i) \cong 
H^2 (k_\q, A)\]
for each $\q \in S$, and also  
\[ \varprojlim_{i \in \N} H^3 (\widetilde{\RGamma}_\mathrm{c}  (k, A \otimes_R R_i)) \cong H^3 ( \widetilde{\RGamma}_\mathrm{c}  (k, A)).\]
These isomorphisms are all in turn derived by applying Lemma \ref{limit result} to the respective complexes $\RGamma_\mathscr{F} (k_\q, A)$, $\RGamma (k_\q, A)$, and $\widetilde{\RGamma}_\mathrm{c}  (k, A)$. 
\end{proof}

\section{Euler systems relative to Nekov\'a\v{r} structures}\label{statement of main result section}

\subsection{The definition of higher-rank Euler systems}
\label{higher-rank euler systems definitions sections}

We now fix a prime number $p$ and a number field $k$ such that condition (\ref{p=2 condition}) is satisfied. We also fix a ring $\cR$ that satisfies condition (\ref{ring condition}) and a finite-rank free $\cR$-module $\cT$ that carries an $\cR$-linear continuous action of $G_k$. We assume $S_\ram (\cT)$ is finite, and choose a finite subset $S_0$ of $\Pi_k$ with 
\[
\Pi_k^\infty \cup \Pi_k^p \subseteq S_0.
\]
(In applications one often takes $S_0 = \Pi_k^\infty \cup \Pi_k^p \cup S_\ram (\cT)$.)
For $\q \in \Pi_k\setminus S_0$, we fix a choice $\cI_\q$ of inertia subgroup in $G_k$ of $\q$ and set
\[
\Eul_\q (X) \coloneqq \det ( 1 - \Frob_\q^{-1} X \mid H^0 ( \cI_\q, \cT^\ast (1)) ) \in \cR [X].
\]
We also fix an abelian pro-$p$ extension $\cK$ of $k$ in which all places in $\Pi_k^\infty$ split completely, and denote the collection of finite extensions of $k$ in $\cK$ by $\Omega$. For each extension $K$ of $k$  in $\cK$ we abbreviate $S_\ram(K/k)$ to $S_\ram(K)$ and then set
\[
S_0 (K) \coloneqq S_0 \cup S_\ram (K) 
\quad \text{ and } \quad 
\cG_K \coloneqq \gal{K}{k}. 
\]
We observe that the ring $\cR [\cG_K]$ also satisfies condition (\ref{ring condition}) (with the necessary Gorenstein property  following, for example, from Lemma \ref{G_n ring examples}).  
We consider the  $\cR[\cG_K]\times \cR\llbracket G_k\rrbracket$-module 
\[ \cT_{K} \coloneqq \cT \otimes_\cR \cR  [\cG_K ]\] 
upon which $\cG_K$ acts via right multiplication and $\sigma \in G_k$ by $\sigma \cdot (a \otimes x) \coloneqq (\sigma a) \otimes (x \overline{\sigma}^{-1})$ where $\overline{\sigma}$ is the image of  $\sigma$ in $\cG_K$.

In the sequel we assume to be given a family of Nekov\'a\v{r} structures that satisfies the following hypothesis.

\begin{hyp} \label{system of fF hyp}
     $\fF \coloneqq (\mathscr{F}_{\!K})_{K \in \Omega}$ is a family of Nekov\'a\v{r} structures parametrised by fields in $\Omega$ that has  both of the following properties.
    \begin{romanliste}
        \item Each $\mathscr{F}_{\!K}$ is a Nekov\'a\v{r} structure on $\cT_K$ that satisfies Hypothesis \ref{fF hyp} (with $R = \cR[\cG_K]$). 
        \item For $K$ and $L$ in $\Omega$ with $K \subseteq L$, the Nekov\'a\v{r} structure $\mathscr{F}_{\!L} \otimes_{\cR [\cG_L]} \cR [\cG_K]$ induced by $\mathscr{F}_{\!L}$ on $\cT_K$ is the modification $\mathscr{F}^L_{\!K} \coloneqq \mathscr{F}_{\!K}^{S_{\mathrm{ram}}(L)}$ of $\mathscr{F}_{\!K}$ defined in Example \ref{lesc}. 
    \end{romanliste}
For $K$ in $\Omega$, we set $H^1_{\fF}(K,\cT) \coloneqq H^1_{\mathscr{F}_K} (k, \cT_K)$. 
\end{hyp}

\begin{bsp}
    Fix a finite subset $S'$ of $\Pi_k \setminus \Pi_k^\infty$ with $\Pi^p_k \cup S_\ram (\cT) \subseteq S'$ and assume to be given, for every $\q \in S'$, an $\cR[G_{k_\q}]$-submodule
$\cT_\q$ of $\cT$ that is free as an $\cR$-module. 
As in Example~\ref{nss} we define a Greenberg--Nekov\'a\v{r} structure $\mathscr{F}_K \coloneqq \mathscr{F} ( (\cT_{\q, K}, j_{\q, K})_{\q \in S (K)})$ by taking 
\begin{itemize}
    \item $S (K) \coloneqq S' \cup \Pi_k^\infty \cup S_\ram (K )$,
    \item $\cT_{\q, K} \coloneqq \cT_\q \otimes_\cR \cR [\cG_K]$ if $\q \in S'$ and $\cT_{\q, K} \coloneqq \cT_K$ if $\q \in S (K) \setminus S'$,
    \item $j_{\q, K}$  to be the inclusion $\cT_{\q, K} \hookrightarrow \cT_K$ for all $\q \in S (K)$.
\end{itemize}
 The associated family $\fF \coloneqq (\mathscr{F}_K)_{K \in \Omega}$ then satisfies Hypothesis \ref{fF hyp} if there exist places $\q_0$ and $\q_1$ in $S'$ with $\cT_{\q_0} = (0)$ and $\cT_{\q_1} = \cT$.
\end{bsp}

For each $K$ and $L$ in $\Omega$ with $K\subseteq L$, the long exact cohomology sequence of the exact triangle (\ref{triangle changing S}) (with $\mathscr{F}$ taken to be $\mathscr{F}_K$ and $\Sigma$ to be $S_\ram(L)$) combines with the explicit descriptions of cohomology given in Proposition \ref{construction complex}\,(v) to give an injective map of $\cR [\cG_K]$-modules 
\begin{equation}\label{intermediate step injection} H^1_{\fF} (K, \cT)=H^1_{\mathscr{F}_{K}} (k, \cT_K)\to H^1_{\mathscr{F}^L_{K}} (k, \cT_K).\end{equation}
These maps have two important consequences. Firstly, since Hypothesis \ref{system of fF hyp}\,(ii) combines with Proposition \ref{construction complex}\,(iii) and Lemma \ref{how the cohomology base changes lemma}\,(i) to identify $H^1_{\mathscr{F}^L_{K}} (k, \cT)$ with a submodule of  $H^1_{\fF} (L, \cT)$, the map (\ref{intermediate step injection}) induces a  canonical injective homomorphism 
\begin{equation}\label{intermediate step injection2} H^1_{\fF} (K, \cT)\hookrightarrow H^1_{\fF} (L, \cT).\end{equation}
We use this homomorphism to identify $H^1_{\fF} (K, \cT)$ with a submodule of $H^1_{\fF} (L, \cT)$. 

In addition,  Lemma \ref{biduals lemma 1}\,(i) implies that, for each $a \in \N_0$, the map (\ref{intermediate step injection}) induces a canonical injective map of $\cR [\cG_K]$-modules 
\[ \bidual^a_{\cR [\cG_K]}H^1_{\fF} (K, \cT) \hookrightarrow \bidual^a_{\cR [\cG_K]} H^1_{\mathscr{F}^L_K} (k, \cT_K),\] 
and, via these maps, we regard  $\bidual^a_{\cR [\cG_K]}H^1_{\fF} (K, \cT)$ as a submodule of $\bidual^r_{\cR [\cG_K]} H^1_{\mathscr{F}^L_K} (k, \cT_K)$. 
We will then also use the isomorphisms that are constructed in the next result. In the proof of this result, and also  often in the sequel, we will use,  for any finite group $\Delta$, the trace element 
\begin{equation}\label{norm def} N_\Delta \coloneqq {\sum}_{\delta \in \Delta} \delta \in \Z [\Delta].\end{equation} 

\begin{lem} \label{construction of injection lemma}
Fix a family $\fF$ of Nekov\'a\v{r} structures  as in Hypothesis \ref{system of fF hyp} and fields $K$ and $L$ in $\Omega$ with $K \subseteq L$. Then, for each $a \in \N $, there exists a natural isomorphism of $\cR [\cG_K]$-modules
\[
\nu^a_{L / K} \: \bidual^a_{\cR [\cG_K]} H^1_{\mathscr{F}^L_K} (k,\cT_K) \xrightarrow{\simeq} \big( \bidual^a_{\cR [\cG_L]} H^1_{\fF} (L,\cT) \big)^{\gal{L}{K}}.
\]
\end{lem}

\begin{proof} Hypothesis \ref{system of fF hyp}\,(i) implies that the complex $C ({\mathscr{F}_L})$ constructed in Proposition \ref{construction complex} satisfies the assumptions of Lemma \ref{finite level reps}. It follows that $C (\mathscr{F}_L)$ has a resolution $P^\bullet$ in $D(\cR [\cG_L])$ of the form $P^0\xrightarrow{\phi}P^1$, in which $P^0$ and $P^1$ are finitely generated free $\cR [\cG_L]$-modules and $P^0$ is placed in degree $0$. Set $H \coloneqq \gal{L}{K}$ and $P^i_{H} \coloneqq P^i \otimes_{\cR [\cG_L]} \cR [\cG_K]$ for $i \in \{0,1\}$. Then, by  taking $H$-invariants of the exact sequence given by Lemma \ref{biduals lemma 1}\,(i), we obtain the upper row in the following exact commutative diagram
\begin{cdiagram}[column sep=small]
    0 \arrow{r} & \big( \bidual^a_{\cR [\cG_L]} H^1_{\fF} (L, \cT) \big)^H \arrow{r}  & \big( \exprod^a_{\cR [\cG_L]} P^0 \big)^H \arrow{r}{\phi} & \big(P^1 \otimes_{\cR [\cG_L]} \exprod^{a - 1}_{\cR [\cG_L]} P^0 \big)^H  \\ 
    0 \arrow{r} & \bidual^a_{\cR [\cG_K]} H^1_{\mathscr{F}^L_K} (k, \cT_K) \arrow{r} \arrow[dashed]{u}{\simeq} & \exprod^a_{\cR [\cG_K]} P^0_{H} \arrow{r}{\phi} \arrow{u}{\simeq} & P^1_{H} \otimes_{\cR [\cG_K]} \exprod^{a - 1}_{\cR [\cG_K]} P^0_{H}. \arrow{u}{\simeq}
\end{cdiagram}%
To describe the lower row we note that Proposition \ref{construction complex}\,(iii) and Hypothesis \ref{system of fF hyp}\,(ii) combine to imply $C (\mathscr{F}^L_K)= C (\mathscr{F}_L \otimes_{\cR [\cG_L]} \cR [\cG_K])$ is isomorphic to $C (\mathscr{F}_L) \otimes^\mathbb{L}_{\cR [\cG_L]} \cR [\cG_K]$ and hence to  $P^\bullet \otimes_{\cR [\cG_L]} \cR [\cG_K]$. The lower row is therefore obtained by applying Lemma \ref{biduals lemma 1}\,(i) to the exact sequence obtained from the latter resolution. Now, for any finitely generated free $\cR [\cG_L]$-module $M$, the assignment $m \mapsto \NN_{H}(m)$ induces an isomorphism $M_H \xrightarrow{\simeq} M^H$. The first and second solid vertical isomorphisms are then obtained by applying this observation with $M$ taken to be $\exprod^a_{\cR [\cG_L]} P^0$ and $P^1 \otimes_{\cR [\cG_L]} \exprod^{a - 1}_{\cR [\cG_L]} P^0$ respectively. Since the square involving these isomorphisms clearly commutes, there exists an induced dashed map as in the diagram and the Snake Lemma implies this map is bijective. We can therefore take the latter map to be the required isomorphism $\nu^a_{L / K}$. 
\end{proof}

We can now specify an appropriate  notion of Euler system for our theory.

\begin{definition}
 Let $\fF$ be a family  of Nekov\'a\v{r} structures as in Hypothesis \ref{system of fF hyp}. Then, for each $a\in \N$,  an `Euler system' of rank $a$ for $\fF$ is an element  
\[
c = (c_K)_{K} \in {\prod}_{K \in \Omega} \bidual^a_{\cR [\cG_K]} H^1_{\fF} (K, \cT)
\]
with the property that, for all fields $K$ and $L$ in $\Omega$ with $K \subseteq L$, one has 
\[
\NN_{\gal{L}{K}}(c_L) = \nu^a_{L / K} \big( ({{\prod}}_{v \in S_0 (L) \setminus S_0 (K)} \Eul_v ( \Frob_v^{-1}) ) \cdot c_K \big).
\]
The collection of all such elements is naturally a module over $\cR \llbracket \cG_\cK \rrbracket$ that we denote by $\ES^a_{S_0}(\fF)$. If $S_0 = \Pi_k^\infty \cup \Pi_k^p \cup S_\ram (\cT)$, then we abbreviate $\ES^a_{S_0}(\fF)$ to $\ES^a(\fF)$.\end{definition}
  
\begin{example}\label{relaxed es remark}
Assume $\cR$ is reduced and $\Z$-torsion free, and write $\cQ$ for its total quotient ring. 
Then, for $K\in \Omega$, the ring  $\cR [\cG_K]$ is also reduced, with total quotient ring $\cQ [\cG_K]$, and we write $\mathscr{F}_{\!\mathrm{rel},K}$ for the relaxed Nekov\'a\v{r} structure on the $\cR [\cG_K]$-module $\cT_K$. Let $\fF_{\mathrm{rel}} = \fF_{\mathrm{rel}}(\cT)$ denote the family $\{\mathscr{F}_{\!\mathrm{rel},K}\}_{K\in \Omega}$ and, for $K\in \Omega$, set $\mathcal{O}'_K \coloneqq \cO_{K, S(K)}$. 
Then for each $L \in \Omega$ with $K \subseteq L$, an explicit computation shows that  
    \[
    H^1_{\fF_{\mathrm{rel}}} (K, \cT) = H^1 (\mathcal{O}_K', \cT) \qquad 
    H^1_{(\mathscr{F}_{\!\mathrm{rel}, K})^L} (k, \cT_K) = H^1 ( \cO_{K, S(L)}, \cT) \qquad 
    H^1_{\fF_{\mathrm{rel}}} (L, \cT) = H^1 (\mathcal{O}_L', \cT).
    \]
In addition, setting $H \coloneqq \gal{L}{K}$, the observation of \cite[Rem.\@ 6.11]{bss} implies the existence for every $a\in \N$ of a commutative diagram of $\cQ[\cG_L]$-modules 
    \begin{cdiagram}[column sep=tiny, row sep=small]
 \exprod^a_{\cQ [\cG_L]} H^1 (\mathcal{O}_L', \cQ \otimes_\cR \cT) \arrow{rr}{\wedge^a \cores_{L / K}} \arrow{d}{\simeq} & &\exprod^a_{\cQ [\cG_K]} H^1 (\mathcal{O}_{K,S(L)}, \cQ \otimes_\cR \cT) \arrow{d}{\simeq} \\ 
\cQ \otimes_\cR \bidual^a_{\cR [\cG_L]} H^1 (\mathcal{O}_L', \cT)  \arrow{dr}{\cdot \NN_{H}}
& &
\cQ \otimes_\cR \bidual^a_{\cR [\cG_K]} H^1 (\mathcal{O}_{K,S(L)}, \cT) \arrow{ld}[above]{\nu^a_{L / K}}[below]{\hskip0.4truein\simeq} \\
& \cQ \otimes_\cR\big( \bidual^a_{\cR [\cG_L]} H^1 (\mathcal{O}_L', \cT) \big)^{ H} &  
    \end{cdiagram}%
in which the horizontal map is induced by corestriction $H^1 (\mathcal{O}_L' , \cT) \to H^1 (\mathcal{O}_{K,S(L)}, \cT)$ and both vertical isomorphisms are as in Lemma \ref{biduals-reduced-rings}. This diagram implies that elements of $\ES^a (\fF_{\mathrm{rel}}(\cT))$ coincide precisely with the notion of Euler systems of rank $a$ for the pair $(\cT,\Omega)$, as defined in \cite[Def.\@ 6.1]{bss}. 
\end{example}

\subsection{Hypotheses on Galois representations}
\label{general set up section}

In this subsection, we assume to be given a pair of local complete Gorenstein rings 
\[ \cR = {{\varprojlim}}_{i \in \N} \cR_i\quad\text{and}\quad R = {{\varprojlim}}_{i \in \N} R_i\]
that arise as the respective inverse limits (as in \S\,\ref{cgr section}) of families of rings $\cR_i$ and $R_i$ that are local, Gorenstein and of dimension zero. In particular, all transition morphisms $\cR_{i+1}\to \cR_i$ and $R_{i+1}\to R_i$ are assumed to be surjective. 

We write $\cM$ and $M$ for the maximal ideals of $\cR$ and $R$, and assume the residue fields 
\[ \mathbb{K} \coloneqq \cR / \cM \quad \text{and}\quad \mathbb{k} \coloneqq R / M\] 
are both finite and of characteristic $p$. We also adopt the convention that $\cR_0 \coloneqq \mathbb{K}$ and $R_0 \coloneqq \mathbb{k}$.\\
Then the maximal ideals $\cM_i$ and $M_i$ of each $\cR_i$ and $R_i$ are the respective images of $\cM$ and $M$ under the canonical projections $\cR \to \cR_i$ and $R \to R_i$. In addition, the corresponding residue fields $\cR_i / \cM_i$ and $R_i / M_i$ respectively identify with $\mathbb{K}$ and $\mathbb{k}$ and so are both independent of $i$. 

We also assume to be given a morphism 
\[ \varrho \: \cR \to R\]
of local rings such that, for every natural number $i$, the diagram
\begin{equation}\label{ring diagram}\xymatrix{
    \cR_{i + 1} \ar[r]^{\varrho_{i + 1}} \ar[d] & R_{i + 1} \ar[d]\\ 
    \cR_i \ar[r]^{\varrho_i} & R_i}
\end{equation}
commutes. Here $\varrho_i$ and $\varrho_{i + 1}$ denote the maps induced by $\varrho$, and the vertical arrows are the given transition maps that occur in the respective inverse limits. 

In addition, we assume to be given a finitely generated free 
$\cR$-module $\cT$ that is endowed with an $\cR$-linear continuous action of $G_k$ that is unramified outside a finite set of places $S_\ram (\cT)$. 
We fix a finite set $S_0 \subseteq \Pi_k$ containing $\Pi_k^\infty \cup \Pi_k^p$, and
we set
\[
T \coloneqq \cT \otimes_\cR R
\qquad 
S \coloneqq S_\ram (\cT) \cup S_0
\]
and, for every $i \in \N_0$, also 
\begin{align*}
\cT_i \coloneqq  \cT \otimes_\cR \cR_i 
\qquad 
T_i  \coloneqq  T \otimes_R R_i \qquad 
\cTbar  \coloneqq \cT_0 = \cT \otimes_\cR \mathbb{K}
\qquad 
 \Tbar \coloneqq T_0 = T \otimes_R \mathbb{k}.
\end{align*}
We write $\mathscr{F}$ for the Nekov\'a\v{r} structure $\mathscr{F}_k$ on $\cT$ that was fixed in \S\ref{higher-rank euler systems definitions sections}, and use the associated  structures
\[ \mathscr{F}_i \coloneqq \mathscr{F} \otimes_\cR \cR_i,\quad \cF_i \coloneqq h (\mathscr{F}_i)\quad\text{and}\quad F_i \coloneqq h ( \mathscr{F}_i \otimes_{\cR_i} R_i)\]
on $\cT_i$ and $T_i$. The Mazur--Rubin structures that are induced (via the procedure described in Remark \ref{ss def}\,(iv)) by $\cF_i$ on $\cTbar$ and by $F_i$ on $\Tbar$  will then be denoted by $\overline{\cF_i}$ and $\overline{F_i}$, and we caution the reader that these structures are, in general,  respectively finer than 
$\cF_0$ and $F_0$.

We fix an ascending chain of number fields
\begin{equation} \label{minimal fields}
k (\cT_0) \subseteq k (\cT_1) \subseteq \dots \subseteq k (\cT_i) \subseteq \dots 
\end{equation}
with the property that $G_{k(\cT_i)}$ acts trivially on $\cT_i$ for all  $i \in \N_0$. 
For  $i \in \N_0$, we  set  
\[ l (i) \coloneqq \max \{ |\cR_i|, |R_i|\},\,\,\, \,k_i \coloneqq k ( \mu_{l (i)}, ( \bigO_k^\times)^{1 / l(i)}) k(1) \quad\text{and}\quad k_i (\cT_i) \coloneqq k(\cT_i) k_i\] 
and then define fields  
\[ k_\infty\coloneqq \bigcup_{i \in \N} k_i \quad\text{and}\quad k (\cT)_\infty\coloneqq \bigcup_{i \in \N} k_i (\cT_i).\]
Given $i \in \N_0$ and an element $\tau \in G_{k_\infty}$, we define a subset of $\Pi_k$ by setting   
\[ \cQ_i = \cQ (\tau, \cT_i) \coloneqq \{ v \in \Pi_k\setminus S : \Frob_v\,\, \text{ is conjugate to }\, \tau \, \text{ in }\, \cG_{k_{i} (\cT_i)}\},\]
and write $\cN_i \coloneqq \cN (\cQ_i)$ for the set of square-free products of primes in $\cQ_i$. We observe that there are  decreasing filtrations
\[
\cQ_0 \supseteq \cQ_1 \supseteq \cQ_2 \supseteq \dots 
\quad \text{ and } \quad 
\cN_0 \supseteq \cN_1 \supseteq \cN_2 \supseteq \dots ,
\]
and for each modulus $\fn \in \cN_i$, we set 
\[ V(\fn) \coloneqq \{\q \in \cQ_i: \q \mid \fn\} = \{\q \in \cQ_0 : \q \mid \fn\}\,\,\text{ and }\,\, \nu(\fn) \coloneqq |V(\fn)|.\] 
For each natural number $j$ with $j \geq i$, the Tate--Shafarevich group  
\[ \sha_{\overline{F_i}, j} (k,\Tbar) = \sha_{\overline{F_i}} (k, \Tbar, \cQ_j) \coloneqq 
\ker \Bigl( H^1_{\overline{F_i}} (k, \Tbar) \to \prod_{v \in \cQ_j} H^1 (k_v, \Tbar) \Bigr )\]
is a submodule of $H^1_{\overline{F_i} (\fn)} (k, \Tbar)$, where the modified Mazur--Rubin structure $\overline{F_i} (\fn)$ is as defined in \cite[Exam.~2.1.8]{MazurRubin04} (see also \cite[\S\,3.1.3]{bss}). Further, by
using global duality sequences of the form (\ref{mr duality}), one can show that the `core-rank' invariant  
\begin{equation}
    \label{definition core rank}
\bm{\chi} (\overline{F_i},j) \coloneqq \dim_\mathbb{k} \big( H^1_{\overline{F_i} (\fn)} (k, \Tbar) / \sha_{\overline{F_i}, j} (k,\Tbar) \big)-  \dim_\mathbb{k} \big( H^{1}_{\overline{F_i}^\ast (\fn)} (k, \Tbar^\ast (1)) / \sha_{\overline{F_i}^\ast, j} (k,\Tbar^\ast (1))\big)
\end{equation}
is independent of $\fn$ (for details see \S\,\ref{existence core section}). \\
Before stating the technical hypotheses that our subsequent arguments require, we make one further observation. Specifically, we note that the maps in Lemma \ref{limit Fitt X} combine with the general result of Lemma \ref{Tor lemma} to give natural maps
\[
\Tor_1^{\cR_{j'}} (X (\mathscr{F}_{j'}), R_{j'}) \to \Tor_1^{\cR_j} (X (\mathscr{F}_j), R_j)
\]
for integers $j$ and $j'$ with $j '\ge j$, and also a canonical isomorphism of $\cR$-modules 
\begin{equation}\label{tor isomorphism}
 \Tor_1^{\cR} ( X (\mathscr{F}), R) \cong {\varprojlim}_{j \in \N} \Tor_1^{\cR_j} (X (\mathscr{F}_j), R_j),
\end{equation}
where the limit is taken with respect to the above maps for $j'\ge j$. In particular, for each  $i\in \N$, this fact gives rise to the following invariants $J_i$ and $j(i)$ of the Nekov\'a\v{r} structure $\mathscr{F}$.

\begin{definition}\label{J def}
    For each $i\in \N$ we set
    \[
    J_i = J_i(\mathscr{F})\coloneqq \im \bigl( \Tor_1^{\cR} ( X (\mathscr{F}), R) \to \Tor_1^{\cR_i} (X (\mathscr{F}_i), R_i) \bigr),
    \]
    where the arrow denotes the map induced by (\ref{tor isomorphism}). Then, since the group $\Tor_1^{\cR_i} (X (\mathscr{F}_i),R_i)$ is finite, there exists $m \in \N$ with $m \ge i$ for which $J_i$ is equal to the image of the natural map $\Tor_1^{\cR_{m}} (X (\mathscr{F}_{m}), R_{m}) \to \Tor_1^{\cR_i} (X (\mathscr{F}_i), R_i)$. We define $j(i)$ to be the least possible value of such an $m$ subject to the condition that the assignment $i \mapsto j(i)$ is an increasing function of $i$.
\end{definition}

We can now finally state the hypotheses on $\cT$ and $\mathscr{F}_k$ that will be used in our arguments.

\begin{hyps} \label{new strategy hyps}
    The following conditions are satisfied.
    \begin{romanliste}
    \item The $\mathbb{k}[G_k]$-module $\Tbar$ and  $\mathbb{K}[G_k]$-module $\cTbar$ are both irreducible.
    \item There exists an element $\tau$ of $G_{k_\infty}$ such that $\dim_\mathbb{K} (\cTbar / (\tau - 1) \cTbar ) = 1$.
    \item[(ii${}^\ast$)] If $p = 2$, then $\dim_\mathbb{K} (\cTbar) = 1$. 
    \item $H^1 (k (\cT)_\infty / k, \Tbar^\ast (1)) = (0)$.
    \item If $p \in \{2,3\}$, then the $\Z_p [G_k]$-modules $\Tbar \oplus \cTbar$ and $\Tbar^\ast(1)\oplus \overline{\cT}^\ast(1)$ have no nonzero isomorphic subquotients. 
    \item For every  $i \in \N_0$, the Nekov\'a\v{r} structure $\mathscr{F}_i$ satisfies Hypothesis \ref{fF hyp}  (with $A = \cT_i$ and $R = \cR_i$).
    \item For every $i \in \N_0$, one has $\bm{\chi} ( \overline{F_i},j(i)) > 0$.
    \item $\gal{k (\cT)_\infty}{k}$ is a compact $p$-adic analytic group.  
    \end{romanliste}
\end{hyps}

\begin{rk}\label{vanishing remark}
We clarify aspects of Hypotheses \ref{new strategy hyps}.
\begin{romanliste}
\item It is clear that if $\cT$ satisfies Hypothesis \ref{new strategy hyps}\,(i) and (ii), respectively (ii)$^\ast$, then so does $\cT^\ast(1)$ (and with respect to the same element $\tau$ in (ii)). 
\item The argument of \cite[Lem.\@ 3.8]{bss2} implies that, under Hypothesis \ref{new strategy hyps}\,(iii), the modules  $H^1 (k_j (\cT_j) / k, \cT_i^\ast (1))$ and $H^1 (k_j (\cT_j) / k, T_i^\ast (1))$ vanish for every $(i,j) \in \N^{\oplus 2}$ with $j \geq i$.
\item The Jordan--H\"older Theorem implies that Hypothesis~\ref{new strategy hyps}~(iv) is equivalent to the following condition: if both $B_1$ and $B_2$ denote either $\Tbar$ or $\overline{\cT}$, then the $\Z_p [G_k]$-modules $B_1$ and $B_2^\ast(1)$ 
have no nonzero isomorphic subquotients. For more details, see  the proof of Proposition \ref{cebotarev prop tweak}\,(ii) below.
\item Sakamoto \cite{Sakamoto24} has developed a theory of Kolyvagin systems in a case where $p = 3$ and Hypothesis \ref{new strategy hyps}\,(iv) fails. It seems reasonable to believe that the approach of Sakamoto would also allow a corresponding weakening of Hypothesis \ref{new strategy hyps}\,(iv) in our set-up.
    \item Hypothesis \ref{new strategy hyps}\,(vi) is weaker than the `cartesian' condition that originates with Mazur and Rubin \cite{MazurRubin04} and is assumed by Sakamoto et al.\@ throughout \cite{bss}. In particular, subsequent analysis (in \S\,\ref{verifying the hypotheses section} and \S\,\ref{proof of iwasawa theory result}) will show that, in cases relevant to the study of Kato's `generalised Iwasawa main conjecture', Hypothesis \ref{new strategy hyps}\,(vi) is satisfied under a variety of  wide-ranging, and natural, conditions.
 \item Let $T'$ be a continuous $\Z_p\llbracket G_k\rrbracket$-module that is free of rank $t$ over $\Z_p$, with $k(T')$ the fixed field of $k^\mathrm{c}$ under the kernel  of the induced homomorphism $G_k \to \mathrm{Aut}_{\Z_p}(T')$. Then  $\gal{k(T')}{k}$ is isomorphic to a closed subgroup of $\mathrm{GL}_t(\Z_p)$ and so is a compact $p$-adic analytic group. In particular, if $\cR = \Z_p\llbracket \gal{L}{k}\rrbracket$ for any compact $p$-adic analytic extension $L/k$ and  $\cT = T'\otimes_{\ZZ_p}\cR$, then $k(\cT)_\infty$ is the composite of the extensions $k(T')$, $L$ and $k_\infty$ of $k$ and so Hypothesis \ref{new strategy hyps}\,(vii) is valid.  In addition, Hypothesis \ref{new strategy hyps}\,(vii) can be omitted in any case in which certain Tate--Shafarevich groups are known to vanish (see Remark \ref{omitting 7 remark}).    
    \end{romanliste}
\end{rk}

At this point, we fix an abelian extension $\cK$ of $k$ in $k^c$ and write $\Omega \coloneqq \Omega (\cK)$ for the set of finite extensions of $k$ in $\cK$.  Then, in the sequel, we will often assume the following hypothesis.

\begin{hyps} \label{more new hyps}
The following conditions are satisfied.
\begin{romanliste}
\item $\cK$ is a pro-$p$ extension that contains $k (\q)$ for every prime $\q \in  \cQ_0$ and also a $\Z_p$-power extension of $k$ in which no finite place splits completely. 
\item For all $a \in \N_0$ and primes  
$\q \in \cQ_0$, the endomorphism $\Frob_\q^{p^a} - 1$ is injective on $\cT$. 
    \end{romanliste}
\end{hyps}

\begin{rk}\label{hyps cons}
    As discussed by Rubin in \cite[\S\,9.1]{Rubin-euler}, in many situations there exist alternative conditions that can replace
    Hypothesis \ref{more new hyps}\,(i). It is also useful to note that Hypothesis \ref{more new hyps}\,(ii) implies  $\cT$ can have no element of finite order. 
\end{rk}

\subsection{Statement of the main result}\label{somr section}

We fix a family $\fF = (\mathscr{F}_K)_{K \in \Omega}$ of Nekov\'a\v{r} structures satisfying Hypothesis \ref{system of fF hyp} and, following Proposition \ref{construction complex}\,(i), define an integer     
\[ \bm{\chi}_{\fF} \coloneqq \bm{\chi}_\cR(C(\mathscr{F}_k)) \in K_0(\cR) \xrightarrow[\cong]{\mathrm{rk}_\cR} \ZZ.\]

We also assume to be given a non-zero free $\cR$-module quotient $Y$ of ${\bigoplus}_{\q\in \Pi_k^\infty}H^0(k_\q,\cT^\vee(1))^\vee$. Such an $\cR$-module is necessarily finitely generated and we set   
\begin{equation}\label{r def} r_Y \coloneqq \mathrm{rk}_\cR(Y) \quad\text{and}\quad  r_{\fF,Y} \coloneqq r_Y + \bm{\chi}_{\fF}.  \end{equation}
We fix an $\cR$-basis $b_\bullet$ of $Y$ and use the surjective map in the central row of diagram (\ref{new diagram}) to regard $Y$ as a quotient of $H^1(C(\mathscr{F}_k))$. We recall (from  Proposition \ref{construction complex}\,(i) and (v)) that $C(\mathscr{F}_k)$ belongs to $D^{\mathrm{perf}}_{[0,1]}(\cR)$ and $H^0(C(\mathscr{F}_k)) = H^1_\fF (k,\cT)$. Hence, if $r_{\fF,Y} >0$, then the  general construction of Lemma \ref{det projection lemma}\,(i) gives a map of $\cR$-modules 
\begin{equation}\label{vartheta Y def}
\vartheta_{\mathscr{F}_k,Y} \coloneqq \vartheta_{C(\mathscr{F}_k), b_\bullet} \:  \Det_\cR (C (\mathscr{F}_k)) \to \bidual^{r_{\fF,Y}}_\cR H^1_\fF (k,\cT)\end{equation}
that depends only on $C (\mathscr{F}_k)$ and $b_\bullet$.  (Note also that Remark \ref{basis independence}\,(i) implies a change of $b_\bullet$ will not affect the validity of any of the results stated subsequently and so, for clarity of exposition, we prefer to write $\vartheta_{\mathscr{F}_k,Y}$ in place of the more precise $\vartheta_{\mathscr{F}_k,b_\bullet}$).

We finally note the $\cR$-module $X(\mathscr{F}_k)$ is finitely-presented (by Lemma \ref{limit Fitt X} and our assumption  $\cR$ is Noetherian) and so, for each $a \in \N_0$, the Fitting ideal
 $\Fitt^{a}_\cR (X(\mathscr{F}_k))$ is well-defined.\\ 
We can now state our main result concerning  Euler systems for Nekov\'a\v{r} structures. 
\begin{thm}\label{new strategy main result} 
Let $\fF = (\mathscr{F}_K)_{K \in \Omega}$ be a family of Nekov\'a\v{r} structures satisfying Hypothesis~\ref{system of fF hyp}.
Assume that Hypotheses \ref{new strategy hyps} and \ref{more new hyps} are valid, and that the (non-zero) free $\cR$-module $Y$ is such that $r_{\fF,Y} >0$. Then, for every pair of elements $x$ and $y$ of $\Fitt^{r_Y}_\cR (X(\mathscr{F}_k))$, there exists a natural number $N$ that depends only on $\cTbar$ and has the following property. For every Euler system $c = (c_K)_K$ in $\ES_{S_0}^{r_{\fF,Y}} (\fF)$ and every prime ideal $\p$ of $\cR$ for which both
\begin{romanliste}
\item the map $\varrho_\fp \: \cR_\p \to R_\p$ induced by $\varrho$ is nonzero and surjective, and
    \item $\Fitt^0_{R} ( \Tor_1^\cR (X(\mathscr{F}_k), R))_\p = R_\p$,
\end{romanliste}
there is a containment 
\[ \,xy^N \cdot (c_k)_\fp \in  y^N\cdot \vartheta_{\mathscr{F}_k,Y} ( \Det_\cR (C(\mathscr{F}_k)))_\fp.\]
\end{thm}

\begin{rk} \label{rk on avoiding more new hyps}
  The only role that Hypothesis \ref{more new hyps} plays in the proof of the above result is in the construction of an appropriate `Kolyvagin derivative homomorphism' in Theorem \ref{kolyvagin derivative thm}. In particular, in any situation in which suitable Kolyvagin systems are already known to exist one can avoid assuming Hypothesis \ref{more new hyps} in proving the displayed containment in Theorem \ref{new strategy main result} (see also Remark \ref{rk on avoiding more new hyps details} for more details in this regard).
\end{rk}

\begin{remark} The assumed existence of a non-zero free $\cR$-module quotient $Y$ of the direct sum ${\bigoplus}_{\q\in \Pi_k^\infty}H^0(k_\q,\cT^\vee(1))^\vee$ for which one has $r_{\fF,Y} > 0$ constitutes  a restriction on the  Nekov\'{a}\v{r} structure $\mathscr{F}_k$. For example, if $\mathscr{F}_k$ is a Greenberg--Nekov\'{a}\v{r} structure $\mathscr{F}( (\cT_\q,j_\q)_{\q \in S(\fF_k)})$ (as defined in Example \ref{nss} with $A = \cT$) for which each sub-representation $\cT_\q$ is a free $\cR$-module, then Proposition  \ref{construction complex}\,(i) combines with Lemma \ref{flach result}\, (ii) and (iii) to imply that 
\[ \bm{\chi}_{\fF} = - {\sum}_{\q\in \Pi_k^p}[K_\q:\QQ_p]\cdot (\mathrm{rk}_{\cR}(\cT)- \mathrm{rk}_{\cR}(\cT_\q)).\]
Hence, in this case, there exists an $\cR$-module $Y$ of the required form with $r_{\fF,Y} > 0$ if and only if the subrepresentations $\cT_\q$ for each $\q\in \Pi_k^p$ satisfy the following condition
\[ {\sum}_{\q\in \Pi_k^p}[K_\q:\QQ_p]\cdot\mathrm{rk}_{\cR}(\cT_\q) > [K:\QQ]\cdot \mathrm{rk}_{\cR}(\cT)- {\sum}_{\q\in \Pi_k^\infty}\mathrm{rk}_\cR(H^0(k_\q,\cT^\vee(1))^\vee).
\]
\end{remark} 

\begin{remark} 
Though technical in nature, Theorem \ref{new strategy main result} has several significant advantages over the main results of the existing theory of Euler, Kolyvagin and Stark systems (in arbitrary rank), as  developed by Mazur and Rubin in \cite{MazurRubin04, MazurRubin2} and by Sakamoto et al.\@ in \cite{bss, bss2}. Firstly, it is finer since its conclusion directly concerns the determinants of Selmer complexes rather than the Fitting ideals of Selmer groups. Secondly, it is more general both in dealing with Euler systems relative to a wide class of Nekov\'{a}\v{r} structures (rather than only to the relaxed Nekov\'{a}\v{r} structure) and also with representations over arbitrary local complete Gorenstein rings with finite residue fields of characteristic $p$. Thirdly, it is more widely applicable in arithmetic settings since several of the  assumptions in Hypothesis~\ref{new strategy hyps} are weaker than the corresponding conditions that are imposed in both \cite{MazurRubin04, MazurRubin2} and \cite{bss, bss2} (see, for example, Remark \ref{vanishing remark}\,(iv)). \end{remark}
\smallskip

After a lengthy series of preliminary results, some of which are possibly of independent interest, the proof of Theorem \ref{new strategy main result} will finally be obtained in \S\,\ref{rss section}.  However, a reader who is more interested to understand how Theorem \ref{new strategy main result} can be applied in arithmetic settings rather than in the details of its proof, may prefer at this point to pass directly to the second part of the article (that starts in \S\,\ref{relaxed Kato section}).

\section{Kolyvagin systems I: modified Selmer complexes and derivatives} \label{koly sys I section}

The theory of Kolyvagin systems was introduced by Mazur and Rubin in \cite{MazurRubin04} and then further developed by Sakamoto et al in \cite{bss}, as a means of axiomatising Kolyvagin's construction of `derivative classes' in \cite{kolyvagin}. In this, and the following, section we shall develop a version of this theory relative to the Selmer complexes that arise from a family $\fF$ of Nekov\'a\v{r} structures satisfying Hypothesis \ref{system of fF hyp}. These sections constitute the technical heart of our general approach. In them we fix a natural number $i$ and, for clarity of exposition, we abbreviate the notation already introduced in \S\,\ref{general set up section} in the following way: 
\begin{equation}\label{conv notation}\begin{cases}\mathscr{F}\coloneqq \mathscr{F}_k,\,\, \cQ \coloneqq \cQ_i,\,\cN \coloneqq \cN_i,\\
\bLambda \coloneqq \cR_i, \,\,\cA \coloneqq \cT_i,\,\,\cB \coloneqq \cA^\ast(1),\,\,\tilde{\mathscr{F}}\coloneqq \mathscr{F}_i,\, \tilde{\cF} \coloneqq h(\tilde{\mathscr{F}}),\\
\Lambda \coloneqq R_i, \,\,A \coloneqq T_i,\, B \coloneqq A^\ast(1), \,\, \tilde F \coloneqq h (\tilde{\mathscr{F}} \otimes_{\bLambda} \Lambda),\,\, \\
\overline{A} \coloneqq \cA\otimes_{\bLambda}\mathbb{k} = A\otimes_{\Lambda}\mathbb{k} = \overline{T},\,\,\overline{B} \coloneqq \overline{A}^\ast(1), \,\,  \overline{F} \coloneqq \tilde F_{\overline{A}}\end{cases}\end{equation} 
where in the last case $\tilde F_{\overline{A}}$ denotes the Mazur--Rubin structure on $\overline{A}$ induced by $\tilde F$ (cf.\@ Example~\ref{ss def}\,(iv)). In particular with this notation, we note that 
\[ \overline{\cT} = \cA\otimes_{\bLambda}\mathbb{K}\quad\text{and}\quad \overline{T}= \cA\otimes_{\bLambda}\mathbb{k} = A\otimes_{\Lambda}\mathbb{k}.\]

\subsection{Comparison maps} \label{comparison maps section} In this subsection, we construct several maps that will play a key role in the sequel. 

For $\q$ in $\cQ$, we write $k (\q)$ for the maximal $p$-extension of $k$ in the ray class field modulo $\q$ of $k$ and note that, by \cite[Lem.\@ 4.1.2]{Rubin-euler} (or \cite[Lem.\@ 2.1]{scarcity}), the group  
\[ G_\q \coloneqq \gal{k(\q)}{k(1)}\] 
is  cyclic of order divisible by $l (i)$. Following Rubin \cite[Def.\@ 4.4.1]{Rubin-euler}, we then specify a generator $\sigma_\q$ of $G_\q$
 as follows. We fix a topological generator $\varpi$ of ${\varprojlim}_{n \in \N} \mu_{p^n} (\overline{\Q}) \cong \Z_p (1)$ and an embedding $\iota_\q \: \overline{\Q} \hookrightarrow \overline{k_\q}$. This embedding induces an identification of
 $G_\q$ with $\gal{k (\q)_\mathfrak{Q}}{k_\q}$, where $\mathfrak{Q}$ is the place of $k (\q)$ above $\q$ specified by $\iota_\q$. In particular, since the local reciprocity map $\mathrm{rec}_\q$ identifies $\gal{k (\q)_\mathfrak{Q}}{k_\q}$ with a quotient of $\mu_{k_\q} \otimes_\Z \Z_p$, we specify the generator $\sigma_\q$ to be the image of $\iota_\q(\varpi)$ under the composite map  
 \[
 {{\varprojlim}}_{n \in \N} \mu_{p^n} ( \overline{k_\q}) \twoheadrightarrow \mu_{k_\q} \otimes_\Z \Z_p \stackrel{\mathrm{rec}_\q}{\twoheadrightarrow} \gal{k (\q)_\mathfrak{Q}}{k_\q} \cong G_\q. 
 \]

Similarly, for any $\fn \in \cN$, we set 
\begin{equation}\label{group def} k(\fn)\coloneqq {\prod}_{\q \in V(\fn)} k(\q)\quad \text{and}\quad G_\fn \coloneqq \gal{k(\fn)}{k(1)}\cong {\bigotimes}_{\q \in V(\fn)} G_\q.\end{equation}
Next we note that, with $k_\q^\mathrm{nr}$ the maximal unramified extension of $k_\q$ inside $k_\q^\mathrm{c}$, the natural `valuation' map $\ord_\q$ induces an identification $k_\q^{\mathrm{nr}, \times} / (k_\q^{\mathrm{nr}, \times})^{l (i)} = \Z/l(i)\Z$. Hence, upon tensoring the canonical composite isomorphism 
\[ H^1 ( k_\q^\mathrm{nr}, \mu_{l (i)}) \cong k_\q^{\mathrm{nr}, \times} / (k_\q^{\mathrm{nr}, \times})^{l(i)} = \Z/l(i)\Z\] 
with $\cA$, we obtain an identification of $H^1 (k_\q^\mathrm{nr}, \cA (1))$ with $\cA$. This  combines with the 
generator $\varpi$ of $\Z_p (1) = H^0(k_\q^\mathrm{nr},\Z_p(1))$ fixed above to give an isomorphism
\[
\partial_\q \: H^1 ( k_\q^\mathrm{nr}, \cA)^{G_{\kappa_\q}} \stackrel{\cup \iota_\q (\varpi)}{\cong} 
H^1 ( k_\q^\mathrm{nr}, \cA (1) )^{G_{\kappa_\q}} \stackrel{\simeq}{\longrightarrow} 
\cA^{\tau = 1}. 
\]
The inflation-restriction sequence implies the existence of a canonical short exact sequence  
\begin{equation}\label{inf-res ses} 
0 \to H^1_f (k_\q, \cA)  \to H^1 ( k_\q, \cA) \xrightarrow{\partial_\q\circ {\rm{res}}_\q}  \cA^{\tau=1}\to 0,\end{equation}
in which ${\mathrm{res}}_\q$ denotes the (surjective) restriction map $H^1(k_\q,\cA) \to H^1 ( k_\q^\mathrm{nr}, \cA)^{G_{\kappa_\q}}$. One checks that $\partial_\q\circ {\rm{res}}_\q$ is explicitly given by evaluating a cocycle at $\sigma_\q$, and hence agrees with the map used by Mazur and Rubin in \cite[Lem.\@ 1.2.1]{MazurRubin04}.
 
The next result relies on the validity of Hypothesis \ref{new strategy hyps}\,(ii) and is essentially well-known (cf.\@ \cite[Lem.~1.2.3]{MazurRubin04}). However, since it is important for us, we shall include a proof. 

\begin{lemma}\label{cfactor} There exists a canonical isomorphism of $\cR_i$-modules $\cA / (\tau - 1)  \to \cA^{\tau = 1}$. \end{lemma} 

\begin{proof} Each endomorphism $f$ of the (free) $\bLambda$-module $\cA$ gives rise to a `cofactor map' $c_f \: \cA \to \cA$ that is uniquely characterised by the commutativity of the diagram
\begin{cdiagram}
    \cA \arrow{r}{\simeq} \arrow{d}{c_f} & \Hom_{\bLambda} \big ( \exprod^{\mathrm{rk}_{\bLambda} (\cA) - 1}_{\bLambda} \cA, \exprod^{\mathrm{rk}_{\bLambda} (\cA)}_{\bLambda} \cA \big) \arrow{d}{h \mapsto h \circ ( \exprod f )} \\ 
    \cA \arrow{r}{\simeq} &  \Hom_{\bLambda} \big ( \exprod^{\mathrm{rk}_{\bLambda} (\cA) - 1}_{\bLambda} \cA, \exprod_{\bLambda}^{\mathrm{rk}_{\bLambda} (\cA)} \cA \big),
\end{cdiagram}%
in which both horizontal arrows send $a\in \cA$ to the map  $x \mapsto a \wedge x$. (Thus, if we fix a ${\bLambda}$-basis of $\cA$, then $c_f$ is represented by the adjugate of the matrix representing $f$ in this basis.) Then, as $f \circ c_f$ and $c_f \circ f$ both coincide with multiplication by $\det (f)$, there is a well-defined map
\[
\cA / f (\cA) \stackrel{c_f}{\longrightarrow} \ker \big (\cA / \det (f) \cA \stackrel{f}{\longrightarrow} \cA / \det (f) \cA \big ). 
\]
We now take $f$ to be multiplication by $1 - \tau$. Then $\det (f) = 0$ (as Hypothesis \ref{new strategy hyps}\,(ii) implies $\ker (f)$ contains a non-zero divisor) and so the above construction gives a map of ${\bLambda}$-modules $c_f \: \cA / (\tau - 1) \cA \to  \cA^{\tau = 1}$. It is enough to prove this map is bijective, or equivalently (as $\cA$ is finite) surjective. Then, as Hypothesis \ref{new strategy hyps}\,(ii) implies the natural map $\mathbb{k}\otimes_{\bLambda} \cA^{\tau = 1} \to \overline{A}^{\tau = 1}$ is bijective, and $\mathbb{k}\otimes_{\bLambda}c_f$ identifies with the cofactor map of the reduction $\overline{f} \: \overline{A} \to \overline{A}$ of $f$, Nakayama's lemma reduces us to proving $c_{\overline{f}}$ is bijective. But Hypothesis \ref{new strategy hyps}\,(ii) implies the corank of $\overline{f}$ is one and hence that its adjugate matrix, and so also $c_{\overline{f}}$, is nontrivial. Since Hypothesis \ref{new strategy hyps}\,(ii) also implies the $\mathbb{k}$-dimension of $\ker (\overline{f}) = \overline{A}^{\tau = 1}$ is one, the obvious inclusion  $\im (c_{\overline{f}}) \subseteq \ker (\overline{f})$ must be an equality, as required to prove the claim. \end{proof}  

We now fix an isomorphism of $\cR_i$-modules $\cA/(\tau-1) \cong \bLambda$ as in Hypothesis \ref{new strategy hyps}\,(ii), and define a composite map of $\bLambda$-modules  
\begin{multline}\label{v_q def} 
v_\q \: H^1 (k, \cA) \otimes_\Z G_\q \stackrel{\loc_\q}{\longrightarrow}
H^1 (k_\q, \cA) \otimes_\Z G_\q \stackrel{\res_\q}{\twoheadrightarrow} H^1 ( k_\q^\mathrm{nr}, \cA)^{G_{\kappa_\q}} \otimes_\Z G_\q \\\stackrel{\partial_\q}{\longrightarrow} \cA^{\tau = 1} \cong \cA/(\tau-1) \cong \bLambda, 
\end{multline}
in which the penultimate map is the isomorphism from Lemma \ref{cfactor}. 

We next observe that, since the inertia subgroup at $\q$ acts trivially on $\cT_i$ and $\Frob_\q$ acts as $\tau$, there exists a well-defined `evaluation' map
\begin{equation}\label{evaluation def}
H^1 (k_\q, \cA) \to \cA / (\tau - 1), \quad x \mapsto x (\Frob_\q) + (\tau - 1) \cT_i. 
\end{equation}

This map is surjective and its kernel coincides with the `transverse' cohomology group 
\[ H^1_\tr (k_\q, \cA) \coloneqq H^1 (G_\q, \cA^{G_{k(\q)}}),\] 
regarded as a submodule of $H^1 (k_\q, \cA)$ via the inflation map (cf.\@ the discussion of \cite[\S\,1.2]{MazurRubin}). The `finite-singular comparison map' is then defined to be the composite map of $\cR_i$-modules 
\begin{equation}\label{fs def}
\psi_\q^\fs \: H^1 (k, \cA) \stackrel{\loc_\q}{\longrightarrow}
H^1 (k_\q, \cA) \twoheadrightarrow  \cA / (\tau - 1)  \cong \bLambda 
\end{equation}
where the final map is the same isomorphism as fixed in (\ref{v_q def}). 

\subsection{Modified Selmer complexes}\label{msc section} 

In this subsection we refine the `modified Selmer structures' that are used in \cite{MazurRubin04} by defining a corresponding family of `modified Nekov\'a\v{r} structures' and describing relations between their associated Selmer complexes. 

We start by analysing the complex $\RGamma(k_\q,\cA)$ for each prime $\q \in \cQ$. 

\begin{lemma}\label{explicit cohomology lemma} For each prime $\q \in \cQ$, the following claims are valid.

\begin{itemize}
\item[(i)] $H^0(k_\q,\cA), H^1(k_\q,\cA)$ and $H^2(k_\q,\cA)$ are free $\bLambda$-modules of ranks $1, 2$ and $1$ respectively.
\item[(ii)] There exists a canonical isomorphism in $D^{\mathrm{perf}}(\bLambda)$ 
\[  H^0(k_\q,\cA)[0]\oplus  H^1(k_\q,\cA)[-1] \oplus  H^2(k_\q,\cA)[-2] \cong \RGamma(k_\q,\cA).\]
\item[(iii)] For $i \in \{0,1,2\}$ fix a projective $\bLambda$-submodule $X_i$ of $H^i(k_\q,\cA)$ and write $Y_i$ for the quotient $H^i(k_\q,\cA)/X_i$. Then the isomorphism in (ii) induces a canonical morphism in $D^{\mathrm{perf}}(\bLambda)$ 
\[ \phi_{(X_0,X_1,X_2)} \: X_0[0]\oplus  X_1[-1] \oplus X_2[-2] \to \RGamma(k_\q,\cA).\]
Further, local Tate duality identifies each module $Y_i^\ast$ with a submodule of $H^{2-i}(k_\q,\cB)$ and also gives a canonical isomorphism in $D^{\mathrm{perf}}(\bLambda)$ 
\[ \mathrm{Cone}\bigl(\phi_{(X_0,X_1,X_2)}\bigr)^\ast[-2] \cong Y_2^\ast[0]\oplus  Y_1^\ast [-1] \oplus  Y_0^\ast[-2].\]
\end{itemize}
\end{lemma}

\begin{proof} For $\q\in \cQ$, one has $H^0(k_\q,\cA)= \cA^{\tau=1}$ and $H^1_f(k_\q,\cA) = \cA/(\tau-1)$ and so both $\bLambda$-modules are free of rank one (see (\ref{v_q def})). From the inflation-restriction sequence (\ref{inf-res ses}) it then follows that the $\bLambda$-module $H^1(k_\q,\cA)$ is free of rank $2$. These observations combine with Proposition \ref{flach result}\,(ii) to imply that the $\bLambda$-module $H^2(k_\q,\cA)$ has finite projective dimension and its class in $K_0(\bLambda)$ is equal to $[H^1(k_\q,\cA)]-[H^0(k_\q,\cA)] = [\bLambda]$. Since $\bLambda$ is both local and self-injective, it therefore follows that $H^2(k_\q,\cA)$ is isomorphic to $\bLambda$, as required to prove (i).

Fix $i \in \{0,1,2\}$. Then, since $H^i(k_\q,\cA)$ is a projective $\bLambda$-module, there exists a unique morphism $\theta_i \: H^i(k_\q,\cA)[-i] \to \RGamma(k_\q,\cA)$ in $D(\bLambda)$ for which $H^i(\theta_i)$ is the identity map on $H^i(k_\q,\cA)$. The direct sum morphism 
\[ \theta_0 \oplus \theta_1\oplus \theta_2\: H^0(k_\q,\cA)[0]\oplus  H^1(k_\q,\cA)[-1] \oplus  H^2(k_\q,\cA)[-2]\to \RGamma(k_\q,\cA)\]
is then an isomorphism in $D(\bLambda)$, thereby proving (ii). 

For  $X_0, X_1, X_2$ as in (iii), the isomorphism in (ii) gives rise to an exact triangle in $D^{\mathrm{perf}}(\bLambda)$  
\[ X_0[0]\oplus  X_1[-1] \oplus X_2[-2] \xrightarrow{\phi_{(X_0,X_1,X_2)}} \RGamma(k_\q,\cA) \to Y_0[0]\oplus  Y_1[-1] \oplus Y_2[-2] \to \cdot.\]
Claim (iii) follows directly from this triangle.\end{proof} 

For any collection $\a, \fb, \fn$ of pairwise coprime moduli in $\cN$ for which $V(\a\fb\fn)\cap S(\tilde{\mathscr{F}}) = \emptyset$, we now define a `modified Nekov\'a\v{r} structure' $\tilde{\mathscr{F}}^\fb_\a (\fn)$ on $\cA$ in the following way: we set 
\[
S (\tilde{\mathscr{F}}^\fb_\a (\fn)) \coloneqq S (\tilde{\mathscr{F}}) \cup V(\a \fb \fn) \]
and then assign the condition at each $\q \in \Pi_k$ to be  
\[  \bigl( \RGamma_{\!\tilde{\mathscr{F}}^\fb_\a (\fn)}(k_\q,\cA), \theta_{\!\tilde{\mathscr{F}}^\fb_\a (\fn),\q}\bigr) \coloneqq 
\begin{cases}
    \bigl( \RGamma_{\!\tilde{\mathscr{F}}}(k_\q,\cA), \theta_{\!\tilde{\mathscr{F}},\q}\bigr) \quad & \text{ if } \q \notin V(\a\fb\fn), \\
    (H^0(k_\q,\cA)[0],\phi_{\cA.\q}) & \text{ if } \q \in V(\a), \\
    \bigl( \RGamma(k_\q,\cA), \mathrm{id}\bigr) & \text{ if } \q \in V(\fb), \\
     (H^0(k_\q,\cA)[0]\oplus H^1_\tr (k_\q, \cA)[-1],\phi_{\cA,\q}) & \text{ if } \q \in V(\fn).
\end{cases}
\]
Here we set  $\phi_{\cA,\q}\coloneqq \phi_{(H^0(k_\q,\cA),(0),(0))}$ for each $\q \in V(\a)$ and $\phi_{\cA,\q} \coloneqq \phi_{(H^0(k_\q,\cA),H^1_\tr (k_\q, \cA),(0))}$ for each $\q \in V(\n)$. 
 
\begin{remark}\label{modified nekovar = modified mr remark} If any of the moduli $\a, \fb$ and $\fn$ are equal to the empty product of primes, then, for convenience, we omit them  from the notation $\tilde{\mathscr{F}}^\fb_\a (\fn)$. In particular, with this convention, the structure $\tilde{\mathscr{F}}^\fb$ coincides 
 with the modification  $\tilde{\mathscr{F}}^{V(\fb)}$ of $\tilde{\mathscr{F}}$ defined in Example \ref{lesc}. Further, for all  moduli $\a, \fb$ and $\fn$ as above, an explicit check shows that the induced Mazur--Rubin structure $h(\tilde{\mathscr{F}}^\fb_\a (\fn))$ agrees with the modification $\tilde{\cF}^\fb_\a (\fn)$ of $\tilde{\cF} = h(\tilde{\mathscr{F}})$ that is introduced by Mazur and Rubin in \cite[Exam. 2.1.8]{MazurRubin04} and used extensively in \cite{bss}. \end{remark} 

The following consequence of Proposition \ref{construction complex} will play a key role in our approach. 

\begin{prop} \label{what we need from Selmer complexes}
 Assume $(k,p)$ satisfies (\ref{p=2 condition}), that $\cA$ is a finitely generated free $\bLambda$-module and that $\tilde{\mathscr{F}}$ satisfies Hypothesis \ref{fF hyp}. Then, as $\a, \fb$ and $\fn$ range over all pairwise coprime moduli in $\cN$ for which $V(\a\fb\fn)\cap S(\tilde{\mathscr{F}}) = \emptyset$, the following claims are valid. 
\begin{itemize} 
    \item[(i)] $C(\tilde{\mathscr{F}}^\fb_\a (\fn))$ belongs to $D^{\mathrm{perf}}(\bLambda)$ and is such that, in $K_0(\bLambda)$, one has 
    \begin{align*}\bm{\chi}_\bLambda\bigl(C(\tilde{\mathscr{F}}^\fb_\a (\fn))\bigr)  =&\, \bm{\chi}_\bLambda\bigl(C(\tilde{\mathscr{F}})\bigr) - \nu(\a)\cdot [\bLambda] \\
    =&\, -\nu(\a)\cdot [\bLambda] + {\sum}_{\q \in S(\tilde{\mathscr{F}})\setminus \Pi_k^\infty}\bm{\chi}_\bLambda(\RGamma_{\tilde{\mathscr{F}}^\ast}(k_\q,B)). \end{align*} 
  \item[(ii)] For any morphism $\bLambda\to \bLambda'$ of rings satisfying (\ref{ring condition}), there exists a natural isomorphism $C(\tilde{\mathscr{F}}^\fb_\a (\fn))\otimes^\mathbb{L}_\bLambda \bLambda' \cong C((\tilde{\mathscr{F}}\otimes_\bLambda\bLambda')^\fb_\a (\fn))
    $ in $D^\mathrm{perf} (\bLambda')$.
    \item[(iii)] $C(\tilde{\mathscr{F}}^\fb_\a (\fn))$ is acyclic outside degrees $0$ and $1$. In addition, $H^0(C(\tilde{\mathscr{F}}^\fb_\a (\fn))) = H^1_{\!\tilde{\mathscr{F}}^\fb_\a (\fn)}(k,\cA)$ and there exists a canonical short exact sequence of $\bLambda$-modules 
\[0 \to H^{1}_{(\tilde{\cF}^\ast)^\a_\fb(\fn)}(k,\cB)^\ast \to  
        H^1 (C(\tilde{\mathscr{F}}^\fb_\a (\fn))) \to X(\tilde{\mathscr{F}})\oplus {\bigoplus}_{\q \in V(\fb)} H^0 (k_\q, \cB)^\ast \to 0.\]
\item[(iv)] For all divisors $\n$ of $\m$ there are canonical exact triangles in $D^{\mathrm{perf}}(\bLambda)$ 
\begin{equation}\label{global duality triangle1} 
 C(\tilde{\mathscr{F}}^\n) \to C(\tilde{\mathscr{F}}^\m) \to {\bigoplus}_{\q\in V(\m/\n)}\bigl(H^1_{\!/f}(k_\q,\cA)[0]\oplus H^2(k_\q,\cA)[-1]\bigr) \to \cdot \end{equation}
\begin{equation}\label{global duality triangle2} C(\tilde{\mathscr{F}}(\m)) \to C(\tilde{\mathscr{F}}^\m) \to {\bigoplus}_{\q\in V(\m)} \bigl(H^1_{\!/\tr} (k_\q, \cA) [0] \oplus H^2(k_\q,\cA)[-1]\bigr) \to \cdot\end{equation}
\begin{equation}\label{global duality triangle3} C(\tilde{\mathscr{F}}_\fn (\m/\fn)) \to   C(\tilde{\mathscr{F}}(\m)) \to {\bigoplus}_{\q \in V(\fn)} H^1_{\tr}(k_\q,\cA)[0] \to \cdot\end{equation}
\begin{equation}\label{global duality triangle4} C(\tilde{\mathscr{F}}_\fa (\m)) \to    C(\tilde{\mathscr{F}}(\m)) \to {\bigoplus}_{\q \in V(\fa)} H^1_f(k_\fq,\cA)[0]  \to \cdot . \end{equation}
\item[(v)] Let $\Phi$ denote either the Nekov\'a\v{r} structure $\tilde{\mathscr{F}}$ or Mazur--Rubin structure $\tilde{\cF}= h(\tilde{\mathscr{F}})$ on $\cA$. Then the cohomology sequences of the exact triangles in (iv) combine with the descriptions of cohomology in (iii) to induce the following exact sequences of $\bLambda$-modules 
\begin{align} 
&H^1_{\Phi^\mathfrak{n}} (k, \cA) \hookrightarrow H^1_{\Phi^\m} (k, \cA) \xrightarrow{(\hat{v}_\q)_{\q\in V(\m/\n)}} 
    \bLambda^{\nu ( \m / \mathfrak{n})}
    \to H^{1}_{(\tilde{\cF}^\ast)_\fn} (k, \cB)^\ast \twoheadrightarrow H^{1}_{(\tilde{\cF}^\ast)_\m} (k, \cB)^\ast \label{global duality sequence1}\\
 &H^1_{\Phi(\m)} (k, \cA) \hookrightarrow H^1_{\Phi^\m} (k, \cA) \xrightarrow{(\hat{\psi}_\q^\mathrm{fs})_{\q\in V(\m)}} 
    \bLambda^{\nu (\m)} \to H^{1}_{\tilde{\cF}^\ast (\m)} (k, \cB)^\ast \twoheadrightarrow H^{1}_{\tilde{\cF}^\ast_\m} (k,\cB)^\ast \label{global duality sequence2}\\ 
 &H^1_{\Phi_\mathfrak{n} (\m/\n)}(k,\cA)\! \hookrightarrow\! H^1_{\Phi (\m)}(k,\cA) \xrightarrow{(\check{v}_\q)_{\q\in V(\n)}} 
    \bLambda^{\nu(\mathfrak{n})}\! \to \! H^{1}_{(\tilde{\cF}^\ast)^\fn (\m/\n)} (k, \cB)^\ast\! \twoheadrightarrow\! H^{1}_{\tilde{\cF}^\ast (\m)} (k,\cB)^\ast\label{global duality sequence3}\\
&H^1_{\Phi_\a (\m)}(k,\cA) \hookrightarrow H^1_{\Phi (\m)}(k,\cA) \xrightarrow{(\check{\psi}_\q^{\mathrm{fs}})_{\q\in V(\a)}} 
    \bLambda^{\nu(\a)} \to H^{1}_{(\tilde{\cF}^\ast)^\a (\m)} (k, \cB)^\ast \twoheadrightarrow \! H^{1}_{\tilde{\cF}^\ast (\m)} (k, \cB)^\ast.\label{global duality sequence4}
    \end{align}
Here, if $\Phi = \tilde{\mathscr{F}}$, then we write $\hat{v}_\q$ and $\hat{\psi}_\q^{\mathrm{fs}}$, respectively $\check v_\q$ and $\check{\psi}_\q^{\mathrm{fs}}$, for the composites of $v_\q$ and $\psi_\q^{\mathrm{fs}}$ with the canonical homomorphisms  $H^1_{\!\tilde{\mathscr{F}}^\m}(k,\cA) \to H^1_{\tilde{\cF}^\m}(k,\cA)$, respectively  $H^1_{\!\tilde{\mathscr{F}}(\m)}(k,\cA) \to H^1_{\tilde{\cF}(\m)}(k,\cA)$. On the other hand, if $\Phi =\tilde{\cF}$, then both $\hat{v}_\q$ and $\check{v}_\q$ denote $v_\q$ and both $\hat{\psi}_\q^{\mathrm{fs}}$ and $\check{\psi}_\q^{\mathrm{fs}}$  denote $\psi_\q^{\mathrm{fs}}$.
\end{itemize}
\end{prop}

\begin{proof} We write $\mathscr{D}$ for the dual Nekov\'a\v{r} structure $\tilde{\mathscr{F}}_\a^\fb(\fn)^\ast$ on $\cB$. Then $S(\mathscr{D}) = S(\tilde{\mathscr{F}})\cup V(\a\fb\n)$ and Lemma \ref{explicit cohomology lemma}\,(iii) implies that the local condition for $\mathscr{D}$ at  $\q \in \Pi_k$ is  
\begin{equation}\label{dual local conditions}  \bigl( \RGamma_{\!\mathscr{D}}(k_\q,\cA), \theta_{\!\mathscr{D},\q}\bigr) \coloneqq 
\begin{cases}
    \bigl( \RGamma_{\!\tilde{\mathscr{F}}^\ast}(k_\q,\cB), \theta_{\!\tilde{\mathscr{F}}^\ast,\q}\bigr) \quad & \text{ if } \q \notin V(\a\fb\fn), \\
     (H^0(k_\q,\cB)[0]\oplus H^1 (k_\q, \cB)[-1],\theta_{\!\mathscr{D},\q}) & \text{ if } \q \in V(\a), \\
    \bigl( 0,0\bigr) & \text{ if } \q \in V(\fb), \\
     (H^0(k_\q,\cB)[0]\oplus H^1_\tr (k_\q, \cB)[-1],\theta_{\mathscr{D},\q}) & \text{ if } \q \in V(\fn),
\end{cases}
\end{equation}
with 
\[ \theta_{\!\mathscr{D},\q}\coloneqq \begin{cases}  \phi_{(H^0(k_\q,\cB),H^1 (k_\q, \cB),(0))}, &\text{if $\q \in V(\a)$,}\\
 \phi_{(H^0(k_\q,\cB),H^0_\tr(k_\q,\cB),(0))}, &\text{if $\q \in V(\n)$}.\end{cases}\] 

The fact $\tilde{\mathscr{F}}$ satisfies Hypothesis \ref{fF hyp} implies that the same is true for $\tilde{\mathscr{F}}_\a^\fb(\n)$. In particular, the result of Proposition \ref{construction complex}\,(i) for $\tilde{\mathscr{F}}_\a^\fb(\n)$ directly implies that $C(\tilde{\mathscr{F}}_\a^\fb(\n))$ belongs to $D^{\mathrm{perf}}(\bLambda)$ and that its Euler characteristic in $K_0(\bLambda)$ is equal to
\begin{align*}
\textstyle \sum_{\q\in S(\tilde{\mathscr{F}}_\a^\fb(\n))\setminus \Pi_k^\infty}\bm{\chi}_R\bigl(\RGamma_{\mathscr{D}}(k_\q,\cB)\bigr)
& = \textstyle \sum_{\q\in S(\tilde{\mathscr{F}})\setminus \Pi_k^\infty}\bm{\chi}_\bLambda\bigl(\RGamma_{\!\tilde{\mathscr{F}}^\ast}(k_\q,\cB)\bigr)
\\ & \qquad \quad + \textstyle \sum_{\q\in V(\fa)}\bm{\chi}_\bLambda\bigl(H^0(k_\q,B)[0]\oplus H^1 (k_\q, \cB)[-1]\bigr) \\
& \qquad \quad + \textstyle \sum_{\q\in V(\fn)}\bm{\chi}_\bLambda\bigl(H^0(k_\q,B)[0]\oplus H^1_\tr (k_\q,\cB)[-1]\bigr)\\
& = \textstyle \sum_{\q \in S(\tilde{\mathscr{F}})\setminus \Pi_k^\infty}\bm{\chi}_\bLambda(\RGamma_{\!\tilde{\mathscr{F}}^\ast}(k_\q,\cB)) -  \textstyle \sum_{\q\in V(\fa)}[\bLambda] \\
& = \bm{\chi}_\bLambda\bigl(C(\tilde{\mathscr{F}})\bigr) -\nu(\a)\cdot [\bLambda].
\end{align*}
Here the first equality follows directly from the explicit description (\ref{dual local conditions}), the second from Lemma \ref{explicit cohomology lemma}\,(i) and the fact $H^1_{\tr}(k_\q,\cB)$ is a free $\bLambda$-module of rank $1$ and the last follows directly from Proposition \ref{construction complex}\,(i). This proves (i). 

To prove (ii) we recall from Proposition \ref{flach result}\,(vi) that, for each $\q \in V(\a\fb\n)$, there exists a natural isomorphism $\RGamma(k_\q,\cA)\otimes_\bLambda \bLambda'\cong \RGamma(k_\q,\cA\otimes_\bLambda \bLambda')$ in $D(\bLambda')$. In conjunction with Lemma \ref{explicit cohomology lemma}\,(ii), this isomorphism induces, in each degree $i$, an isomorphism of $\bLambda'$-modules $H^i(k_\q,\cA)\otimes_\bLambda\bLambda' \cong H^i(k_\q,\cA\otimes_\bLambda\bLambda')$. Further, for $\q \in V(\n)$, the isomorphism $H^1(k_\q,\cA)\otimes_\bLambda\bLambda' \cong H^1(k_\q,\cA\otimes_\bLambda\bLambda')$ combines with the canonical isomorphisms $\bigl( \cA/(\tau-1)\cA)\otimes_\bLambda\bLambda' \cong (\cA\otimes_\bLambda\bLambda')/(\tau-1)$ to imply  the image under $-\otimes_\bLambda\bLambda'$ of the map (\ref{evaluation def}) for $\cA$ is equal to the corresponding map for $\cA\otimes_\bLambda\bLambda'$ and hence that there exists a canonical isomorphism $H^1_{\rm{tr}}(k_\q,\cA)\otimes_\bLambda\bLambda' \cong H^1_{\rm{tr}}(k_\q,\cA\otimes_\bLambda\bLambda')$ with respect to which $\phi_{\cA,\q}\otimes_\bLambda\bLambda'$ identifies with $\phi_{\cA\otimes_\bLambda\bLambda',\q}$.  
Taken together, these observations imply the induced structure $\tilde{\mathscr{F}}_\a^\fb (\fn)\otimes_\bLambda\bLambda'$ identifies with $(\tilde{\mathscr{F}}\otimes_\bLambda\bLambda')_\a^\fb (\fn)$ and so the isomorphism in (ii) follows from Proposition \ref{construction complex}\,(ii) for $\tilde{\mathscr{F}}_\a^\fb(\n)$. 

Regarding (iii), we note Proposition \ref{construction complex}\,(v) for $\tilde{\mathscr{F}}^\fb_\a(\n)$ directly implies that $C(\tilde{\mathscr{F}}^\fb_\a (\fn))$ is acyclic outside degrees $0$ and $1$ and $H^0(C(\tilde{\mathscr{F}}^\fb_\a (\fn))) = H^1_{\!\tilde{\mathscr{F}}^\fb_\a (\fn)}(k,\cA)$. From Remark \ref{modified nekovar = modified mr remark}, we also know that the dual Mazur--Rubin structure $h(\tilde{\mathscr{F}}^\fb_\a(\n))^\ast$ on $B$ is equal to $\tilde{\cF}^\fb_\a(\n)^\ast = (\tilde{\cF}^\ast)^\a_\fb(\fn)$. In addition, since $\tilde{\mathscr{F}}$ is assumed to validate Hypothesis \ref{fF hyp}\,(b)\,(i) and (d), one has 
 $H^0(\RGamma_{\mathscr{F}^\ast}(k_{\q_1},B)) = (0)$ and so the map $\lambda^0_{S(\tilde{\mathscr{F}})\cup V(\a)}(\mathscr{D})$ from (\ref{lambda 0 def}) is injective. From Lemma \ref{comparison selmer structures}\,(ii) with $\mathscr{F}, S$ and $S'$ taken to be $\mathscr{D}, S(\mathscr{D}) = S(\tilde{\mathscr{F}}) \cup V(\a\fb\n)$ and $V(\fb\n)$, one therefore obtains a short exact sequence of $\bLambda$-modules 
\[ 0 \to  X_{S(\tilde{\mathscr{F}})\cup V(\a)}(\tilde{\mathscr{F}}^\fb_\a(\fn)) \to  X_{S(\tilde{\mathscr{F}}_\a^\fb(\n))}(\tilde{\mathscr{F}}_\a^\fb(\n)) \to \bigoplus_{\q \in V(\fb\n)} \left(\frac{H^0(k_\q,B)}{\im(H^0(\theta_{\mathscr{D},\q}))}\right)^\ast \to 0.\]
Further, from (\ref{dual local conditions}) one checks that $H^0(\theta_{\!\mathscr{D},\q})$ is the zero map for $\q \in V(\fb)$ and is surjective for $\q \in V(\fa\n)$,  and so this sequence is equivalent to an exact sequence
\[ 0 \to  X(\tilde{\mathscr{F}}) \to  X(\tilde{\mathscr{F}}_\a^\fb(\fn)) \to \bigoplus_{\q \in V(\fb)} H^0(k_\q,\cB)^\ast \to 0.\]
This sequence splits since, for each $\q \in V(\mathfrak{b})$, the $\bLambda$-module $H^0(k_\q,\cB)^\ast$ is free. Given these observations, the exact sequence in (iii) is now obtained from the lower row of the diagram (\ref{new diagram}) with $\mathscr{F}$ taken to be $\tilde{\mathscr{F}}_\a^\fb(\fn)$. 

The exact triangles in (iv) are all derived from Lemma \ref{basic properties}\,(iii). In the first case, the descriptions (\ref{dual local conditions}) imply $(\tilde{\mathscr{F}}^\m)^\ast \le (\tilde{\mathscr{F}}^\n)^\ast$ and also, for $\q \in V(\m/\n)$, that $\RGamma_{(\tilde{\mathscr{F}}^\n)^\ast/(\tilde{\mathscr{F}}^\m)^\ast}(k_\q,\cB)$ is isomorphic to $H^0(k_\q,\cB)[0]\oplus H^1_{f}(k_\q,\cB)[-1].$ In this case, therefore, Lemma \ref{basic properties}\,(iii) gives an exact triangle 
\[ \RGamma_{(\tilde{\mathscr{F}}^\m)^\ast}(k,\cB) \to \RGamma_{(\tilde{\mathscr{F}}^\n)^\ast}(k,\cB) \to \bigoplus_{\q \in V(\m/\n)}\bigl(H^0(k_\q,\cB)[0]\oplus H^1_{f}(k_\q,\cB)[-1]\bigr)\to \cdot\]
and (\ref{global duality triangle1}) is directly derived from the image of this triangle under the exact functor $(-)^\ast[-2]$. 
 
In a similar way, (\ref{dual local conditions}) implies $(\tilde{\mathscr{F}}^\m)^\ast \le \tilde{\mathscr{F}}(\m)^\ast$ and that, for each $\q \in V(\m)$, there exists a natural isomorphism in $D^{\mathrm{perf}}(\bLambda)$ 
\begin{align*}
\RGamma_{\!\tilde{\mathscr{F}}(\m)^\ast/(\tilde{\mathscr{F}}^\m)^\ast}(k_\q,\cB)^\ast[-2] & \cong \bigl(H^0(k_\q,\cB)[0]\oplus H^1_{\tr}(k_\q,\cB)[-1]\bigr)^\ast[-2]\\ 
& \cong H^1_{\!/\tr} (k_\q, \cA) [-1] \oplus H^2(k_\q,\cA)[-2].
\end{align*}
Given these facts, the exact triangle  (\ref{global duality triangle2}) is directly derived from the image under $(-)^\ast[-2]$ of the exact triangle in Lemma \ref{basic properties}\,(iii) with $(\mathscr{F}',\mathscr{F})$ taken to be $((\tilde{\mathscr{F}}^\m)^\ast,\tilde{\mathscr{F}}(\m)^\ast)$. 

We next use (\ref{dual local conditions}) to check that $\tilde{\mathscr{F}}(\m)^\ast \le \tilde{\mathscr{F}}_\n(\m/\n)^\ast$, with $\RGamma_{\!\tilde{\mathscr{F}}_\n(\m/\n)^\ast/\tilde{\mathscr{F}}(\m)^\ast}(k_\q,\cB)$ acyclic for each $\q \notin V(\n)$ and naturally isomorphic to $H^1_{\!/\tr}(k_\q,\cB)[-1]$ for each $\q \in V(\n)$. Given these facts, the exact triangle (\ref{global duality triangle3}) is derived from the image under $(-)^\ast[-2]$ of the triangle in Lemma \ref{basic properties}\,(iii) with $(\mathscr{F}',\mathscr{F})$ taken to be $(\tilde{\mathscr{F}}(\m)^\ast,\tilde{\mathscr{F}}_\n(\m/\n)^\ast)$.

Finally, to complete the proof of (iv), we note that (\ref{dual local conditions}) implies $\tilde{\mathscr{F}}(\m)^\ast \le \tilde{\mathscr{F}}_\a(\m)^\ast$ and that $\RGamma_{\!\tilde{\mathscr{F}}_\a(\m)^\ast/\tilde{\mathscr{F}}(\m)^\ast}(k_\q,\cB)$ is acyclic for each $\q \notin V(\a)$ and identifies with $H^1_{\!/f}(k_\q,B)[-1]$ for each $\q \in V(\a)$. These observations imply that the exact triangle  (\ref{global duality triangle4}) is directly derived from the image under $(-)^\ast[-2]$ of the exact triangle in Lemma \ref{basic properties}\,(iii) with $(\mathscr{F}',\mathscr{F})$ taken to be $(\tilde{\mathscr{F}}(\m)^\ast,\tilde{\mathscr{F}}_\a(\m)^\ast)$.\\
Each of the exact sequences in (v) is derived from the long exact cohomology sequence of the corresponding exact triangle in (iv). 
 For example, the long exact cohomology sequence of (\ref{global duality triangle1}) combines with the descriptions of cohomology in (iii) to give the following variant of the exact commutative diagram in Theorem \ref{global duality thm}\,(ii)
\begin{equation*}
\xymatrixcolsep{5mm}\xymatrix{ 
H^1_{\!\tilde{\mathscr{F}}^\n} (k, \cA) \ar@{^{(}->}[r] \ar@{->>}[d] & H^1_{\!\tilde{\mathscr{F}}^\m}(k, \cA) \ar[r] \ar@{->>}[d] & 
\bigoplus_{\q\in V(\m/\n)} H^1_{\!/f}(k_\q,\cA) \ar[d]_{\cong}^{(\alpha_\q)_\q}\\
H^1_{\tilde{\cF}^\n} (k, \cA) \ar@{^{(}->}[r] & H^1_{\tilde{\cF}^\m}(k, \cA) \ar[r]^{(v_\q)_\q} & \bLambda^{\nu(\m/\n)} \ar[r] & H^1_{(\tilde{\cF}^\ast)_\n}(k,\cB)^\ast \ar@{->>}[r] & H^1_{(\tilde{\cF}^\ast)_\m}(k,\cB)^\ast.}\end{equation*} 
Here the first two vertical maps are the canonical projections from (\ref{nekovar mr compare}) and each $\alpha_\q$ denotes the isomorphism $H^1_{\!/f}(k_\q,\cA) \cong \bLambda$ induced by the composite map
\begin{equation}\label{explaining map} H^1(k_\q,\cA) \xrightarrow{\mathrm{res}} H^1(k_\q^{\mathrm{nr}},\cA)^{G_{\kappa_\q}} \xrightarrow{\partial_\q}\cA^{\tau=1} \cong \cA/(\tau-1) \cong \bLambda\end{equation}
that occurs in the definition (\ref{v_q def}) of $v_\q$. Taking account of the latter isomorphisms, the lower row is the exact sequence of (\ref{mr duality}) with $(\cF_1,\cF_2)$ taken to be $(\tilde{\cF}^\n,\tilde{\cF}^\m)$ and directly gives the sequence (\ref{global duality sequence1}) in the case $\Phi = \tilde{\cF}$. In addition, since the commutativity of the second square implies that the second map in the upper row is  $(\hat{v}_\q)_{\q\in V(\m/\n)}$, the two rows combine to give the sequence (\ref{global duality sequence1}) in the case $\Phi = \tilde{\mathscr{F}}$. 

In a similar way, the long exact cohomology sequence of (\ref{global duality triangle2}) combines with the descriptions of cohomology in (iii) and the appropriate case of (\ref{mr duality}) to give an exact commutative diagram 
\begin{equation*}
\xymatrixcolsep{5mm}\xymatrix{ 
H^1_{\!\tilde{\mathscr{F}}(\m)} (k, \cA) \ar@{^{(}->}[r] \ar@{->>}[d] & H^1_{\!\tilde{\mathscr{F}}^\m}(k, \cA) \ar[r] \ar@{->>}[d] & 
\bigoplus_{\q\in V(\m)} H^1_{\!/\tr}(k_\q,\cA) \ar[d]_{\cong}^{(\beta_\q)_\q}\\
H^1_{\tilde{\cF}(\m)} (k, \cA) \ar@{^{(}->}[r] & H^1_{\tilde{\cF}^\m}(k, \cA) \ar[r]^{(\psi_\q^{\mathrm{fs}})_\q} & \bLambda^{\nu(\m)} \ar[r] & H^1_{\tilde{\cF}^\ast(\m)}(k,\cB)^\ast \ar@{->>}[r] & H^1_{(\tilde{\cF}^\ast)_\m}(k,\cB)^\ast.}\end{equation*} 
Here, for each $\q$, we write $\beta_\q$ for the isomorphism $H^1_{\!/\tr}(k_\q,\cA) \cong \bLambda$ induced by the composite map $H^1(k_\q,\cA) \to \cA/(\tau-1) \cong \bLambda$ that occurs in the definition of $\psi_\q^{\mathrm{fs}}$. It follows that the second map in the upper row is $(\hat{\psi}_\q^{\mathrm{fs}})_{\q\in V(\m)}$ and so the exact sequence (\ref{global duality sequence2}) with $\Phi = \tilde{\cF}$, respectively $\Phi = \tilde{\mathscr{F}}$, follows directly from the lower row of the above diagram, respectively from a comparison of the upper and lower rows of the diagram. 

The derivations of (\ref{global duality sequence3}) and (\ref{global duality sequence4}) from the respective exact triangles (\ref{global duality triangle3}) and (\ref{global duality triangle4}) follow along precisely similar lines. For brevity, we therefore leave details of these derivations to the reader except to note that, for each $\q\in V(\n)$, the isomorphism (\ref{explaining map}) induces an isomorphism of $H^1_{\tr}(k_\q,\cA) \cong H^1_{\!/f}(k_\q,\cA)$ with $\bLambda$ and, for each $\q\in V(\a)$, the map $\psi_\q^{\mathrm{fs}}$ relies on the isomorphism of $H^1_f(k_\q,\cA) = \cA/(\tau-1)$ with $\bLambda$ that was fixed earlier. 
\end{proof}

The $\cR$-module $Y$ fixed at the beginning of \S\,\ref{somr section} gives rise  to an $\bLambda$-module $\tilde{Y} \coloneqq Y\otimes_\cR \bLambda$ that is non-zero, free and identifies, via  (\ref{new diagram}), with a quotient of $X(\tilde{\mathscr{F}})$. In the next result we shall use variants of the exact triangles (\ref{global duality triangle1}) and (\ref{global duality triangle2}) that are respectively constructed via the following diagrams in $D(\bLambda)$
\begin{equation}\label{modify diagram1}
\xymatrixcolsep{5mm}
\xymatrixrowsep{7mm}
\xymatrix{ (\tilde{Y}\oplus M(\n))[-1] \ar[r] & (\tilde{Y}\oplus M(\m))[-1] \ar[r] & M(\m/\n)[-1] \ar[r] & \cdot\\
C(\tilde{\mathscr{F}}^\n)\ar[u]^{\theta_\n} \ar[r] & C(\tilde{\mathscr{F}}^\m)\ar[r]\ar[u]^{\theta_\m} & {\bigoplus}_{\q\in V(\m/\n)}C_\q \ar[r] \ar[u]^{\theta_{\m,\n}} & \cdot\\
C_{\tilde{Y}}(\tilde{\mathscr{F}}^\n) \ar[r]^{\hskip 0.1truein \rho_1} \ar[u] & C_{\tilde{Y}}(\tilde{\mathscr{F}}^\m) \ar[r]^{\hskip -0.4truein\rho'_1} \ar[u]  & {\bigoplus}_{\q\in V(\m/\n)}H^1_{\!/f}(k_\q,\cA)[0] \ar[r]  \ar[u] & \cdot} \end{equation}
\begin{equation}\label{modify diagram}
\xymatrixcolsep{5mm}
\xymatrixrowsep{7mm}
\xymatrix{ \tilde{Y}[-1] \ar[r] & (\tilde{Y}\oplus M(\m))[-1] \ar[r] & M(\m)[-1]\ar[r] & \cdot\\
C(\tilde{\mathscr{F}}(\m))\ar[u]^{\theta_\m'} \ar[r] & C(\tilde{\mathscr{F}}^\m)\ar[r]\ar[u]^{\theta_\m} & {\bigoplus}_{\q\in V(\m)}C_\q' \ar[r] \ar[u]^{\theta''_{\m}}& \cdot\\
C_{\tilde{Y}}(\tilde{\mathscr{F}}(\m)) \ar[r]^{\hskip 0.1truein \rho_2} \ar[u] & C_{\tilde{Y}}(\tilde{\mathscr{F}}^\m) \ar[r]^{\hskip -0.4truein\rho'_2} \ar[u]  & {\bigoplus}_{\q\in V(\m)}H^1_{\!/{\mathrm{tr}}}(k_\q,\cA)[0] \ar[r] \ar@{=}[u] & \cdot} \end{equation}

Here, for each modulus $\a$ in $\cN$, we write $M(\a)$ for the free $\bLambda$-module ${\bigoplus}_{\q\in V(\a)}H^0(k_\q,\cB)^\ast$ so that the upper rows of the respective diagrams are the exact triangles induced by the obvious short exact sequence $0\to M(\n) \to M(\m) \to M(\m/\n) \to 0$ and $0\to \tilde{Y} \to \tilde{Y} \oplus M(\m) \to M(\m) \to 0$. In addition, we set 
\[ C_\q \coloneqq H^1_{\!/f}(k_\q,\cA)[0]\oplus H^2(k_\q,\cA)[-1],\,\, \,\,\text{respectively}\,\,C'_\q \coloneqq H^1_{\!/\tr}(k_\q,\cA)[0]\oplus H^2(k_\q,\cA)[-1],\]
for $\q \in V(\m/\n)$, respectively $\q \in V(\m)$, so that the central rows of the two diagrams are respectively the exact triangles of  (\ref{global duality triangle1}) and (\ref{global duality triangle2}). The morphisms $\theta'_\m$ and $\theta_\m$ are the unique morphisms in $D(\bLambda)$ for which $H^1(\theta'_\m)$ and $H^1(\theta_\m)$ are the surjective maps induced by the descriptions of $H^1(C(\tilde{\mathscr{F}}(\m)))$ and $H^1(C(\tilde{\mathscr{F}}^\m))$ in Proposition \ref{what we need from Selmer complexes}\,(ii) (these morphisms exist and are unique since $C(\tilde{\mathscr{F}}(\m))$ and $C(\tilde{\mathscr{F}}^\m)$ are both acyclic in degrees greater than one) and we write $C_{\tilde{Y}}(\tilde{\mathscr{F}}(\m))$ and $C_{\tilde{Y}}(\tilde{\mathscr{F}}^\m)$ for their respective mapping fibres. Finally, we write $\theta_{\m,\n}$ and $\theta''_{\m}$ for the unique morphisms in $D(\bLambda)$ for which $H^1(\theta_{\m,\n})$ and $H^1(\theta''_{\m})$ are the direct sums over $V(\m/\n)$ and $V(\m)$ of the local duality isomorphisms $H^2(k_\q,\cA) \cong H^0(k_\q,\cB)^\ast$ so that the respective mapping fibres identify with ${\bigoplus}_{\q\in V(\m/\n)}H^1_{\!/f}(k_\q,\cA)[0]$ and ${\bigoplus}_{\q\in V(\m)}H^1_{\!/\tr}(k_\q,\cA)[0]$ and the right hand columns of the diagrams are the associated exact triangles. At this point, we note that the upper squares of both diagrams commute and so the Octahedral axiom implies the existence of the indicated morphisms $\rho_1, \rho_1', \rho_2$ and $\rho_2'$ that make the respective lower rows exact triangles and the whole diagrams commutative in $D(\bLambda)$. 

\begin{proposition}\label{new modified selmer} Recall the integer $r = r_{\fF,Y}$ defined in (\ref{r def}). Then for all moduli $\m$ and $\n$ in $\cN$ with $\n \mid \m$, the following claims are valid.  
\begin{itemize}
\item[(i)] $C_{\tilde{Y}}(\tilde{\mathscr{F}}(\m))$ and $C_{\tilde{Y}}(\tilde{\mathscr{F}}^\m)$ belong to $D^{\mathrm{perf}}(\bLambda)$ and are such that, in $K_0(\bLambda)$, one has 
\[ \bm{\chi}_\bLambda(C_{\tilde{Y}}(\tilde{\mathscr{F}}(\m))) = r\cdot [\bLambda]\quad\text{and}\quad \bm{\chi}_\bLambda(C_{\tilde{Y}}(\tilde{\mathscr{F}}^\m)) = (r+ \nu(\m))\cdot [\bLambda]. \]
\item[(ii)] $C_{\tilde{Y}}(\tilde{\mathscr{F}}(\m))$ and $C_{\tilde{Y}}(\tilde{\mathscr{F}}^\m)$ are acyclic outside degrees zero and one. There are identifications  $H^0(C_{\tilde{Y}}(\tilde{\mathscr{F}}(\m)) = H^1_{\!\tilde{\mathscr{F}}(\m)}(k,\cA)$ and $H^0(C_{\tilde{Y}}(\tilde{\mathscr{F}}^\m)) = H^1_{\!\tilde{\mathscr{F}}^\m}(k,\cA)$ and  canonical short exact sequences
\begin{align*} 0 \to H^{1}_{\!\tilde{\cF}^\ast(\m)}(k, \cB)^\ast \to  
        &\,H^1 (C_{\tilde{Y}}(\tilde{\mathscr{F}}(\m))) \to \ker(X(\tilde{\mathscr{F}})\to \tilde{Y}) \to 0\\
    0 \to H^{1}_{(\tilde{\cF}^\ast)_\m}(k, \cB)^\ast \to  
        &\,H^1 (C_{\tilde{Y}}(\tilde{\mathscr{F}}^\m)) \to \ker(X(\tilde{\mathscr{F}})\to \tilde{Y}) \to 0.\end{align*}
\item[(iii)] If $r> 0$, then there exists a commutative diagram of $\bLambda$-modules 
\[\xymatrix{ \Det_\bLambda (C_{\tilde{Y}}(\tilde{\mathscr{F}}^\n)) \ar[rr]^{\vartheta_{\!\tilde{\mathscr{F}}^\n,\tilde{Y}}} & & \bidual^{r+\nu(\n)}_\bLambda H^1_{\!\tilde{\mathscr{F}}^\n}(k,\cA)\\
\Det_\bLambda (C_{\tilde{Y}}(\tilde{\mathscr{F}}^\m)) \ar[d]_{\varphi_2} \ar[u]^{\varphi_1} \ar[rr]^{\vartheta_{\!\tilde{\mathscr{F}}^\m,\tilde{Y}}} && \bidual^{r + \nu (\m)}_\bLambda H^1_{\!\tilde{\mathscr{F}}^\m}(k,\cA) \ar[u]_{\bigwedge_{\q \in V(\m/\n)} \hat{v}_\q}\ar[d]^{\bigwedge_{\q \in V(\m)} \hat{\psi}_\q^\mathrm{fs}}\\
\Det_\bLambda (C_{\tilde{Y}}(\tilde{\mathscr{F}}(\m))) \ar[rr]^{\vartheta_{\!\tilde{\mathscr{F}}(\m),\tilde{Y}}} & & \bidual^{r}_\bLambda H^1_{\!\tilde{\mathscr{F}}(\m)}(k,\cA).}\]
The maps $\vartheta_{\tilde{\mathscr{F}}^\m,\tilde{Y}}$ and $\vartheta_{\tilde{\mathscr{F}}(\m),\tilde{Y}}$ are defined in the course of the proof below and have cokernels that are respectively annihilated by $\mathrm{Fitt}_\bLambda^0(H^1 (C_{\tilde{Y}}(\tilde{\mathscr{F}}^\m)))$ and $\mathrm{Fitt}_\bLambda^0(H^1 (C_{\tilde{Y}}(\tilde{\mathscr{F}}(\m))))$. The maps $\varphi_1$ and $\varphi_2$ are isomorphisms and arise as follows: $\varphi_1$ is induced by the lower row of (\ref{modify diagram1}) and the identification $\Det_\bLambda(H^1_{\!/f}(k_\fq,\cA)[0]) \cong \Det_\bLambda(\bLambda[0]) = \bLambda$ for each $\q\in V(\m/\n)$ given by the isomorphism $H^1_{\!/f}(k_\fq,\cA)\cong \bLambda$ induced by (\ref{explaining map}); $\varphi_2$ is induced by the lower row of (\ref{modify diagram}) and the identification $\Det_\bLambda(H^1_{\!/{\mathrm{tr}}}(k_\fq,\cA)[0]) \cong \Det_\bLambda(\bLambda[0]) = \bLambda$ for each $\q\in V(\m)$ given by the isomorphism $H^1_{\!/{\mathrm{tr}}}(k_\fq,\cA)\cong \bLambda$ induced by (\ref{evaluation def}). 
\end{itemize}
\end{proposition}

\begin{proof} The $\bLambda$-module $M(\m)$ is free of rank $ \nu(\m)$. Hence, by combining the relevant cases of  Proposition \ref{what we need from Selmer complexes}\,(i) with the exact triangles given by the first two columns of (\ref{modify diagram}), one deduces $C_{\tilde{Y}}(\tilde{\mathscr{F}}(\m))$ and $C_{\tilde{Y}}(\tilde{\mathscr{F}}^\m)$ belong to $D^{\mathrm{perf}}(\bLambda)$, and also that, in $K_0(\bLambda)$, one has both 
\begin{align*} \bm{\chi}_\bLambda\bigl(C_{\tilde{Y}}(\tilde{\mathscr{F}}(\m))\bigr) =&\, \bm{\chi}_\bLambda\bigl(C(\tilde{\mathscr{F}}(\m))\bigr) - \bm{\chi}_\bLambda({\tilde{Y}}[-1])\\
                                              =&\,\bm{\chi}_\bLambda\bigl(C(\tilde{\mathscr{F}})\bigr) + [\tilde{Y}]\\
                                              =&\, r\cdot [\bLambda]\end{align*}
and 
\begin{align*} \bm{\chi}_\bLambda\bigl(C_{\tilde{Y}}(\tilde{\mathscr{F}}^\m)\bigr) =&\, \bm{\chi}_\bLambda\bigl(C(\tilde{\mathscr{F}}^\m)\bigr) - \bm{\chi}_\bLambda(M(\tilde{Y},\m)[-1])\\
                                              =&\, \bm{\chi}_\bLambda\bigl(C(\tilde{\mathscr{F}})\bigr) + [\tilde{Y}] + \nu(\m)\cdot [\bLambda]\\
                                              =&\, (r+ \nu(\m))\cdot [\bLambda].\end{align*}
This proves (i). In addition, all assertions in (ii) are  obtained by combining the relevant cases of Proposition \ref{what we need from Selmer complexes}\,(iii) with an analysis of the long exact cohomology sequences of the first two columns of (\ref{modify diagram}). 

We note now that the results of (i) and (ii) combine with the argument of Lemma \ref{finite level reps} to imply the hypotheses of Proposition \ref{eagon-northcott-prop}(iv) are satisfied by both of the following sets of data

\begin{itemize}
\item[$\circ$] $C^\bullet, D^\bullet, r$ and $n$ are respectively taken to be $C_{\tilde{Y}}(\tilde{\mathscr{F}}^\m)$, $C_{\tilde{Y}}(\tilde{\mathscr{F}}^\n)$, $r+\nu(\n)$ and $\nu(\m/\n)$ and the exact triangle (\ref{triangle 1}) is the lower row of (\ref{modify diagram1});
\item[$\circ$] $C^\bullet, D^\bullet, r$ and $n$ are respectively taken to be $C_{\tilde{Y}}(\tilde{\mathscr{F}}(\m))$, $C_{\tilde{Y}}(\tilde{\mathscr{F}}^\m)$, $r$ and $\nu(\m)$ and the exact triangle (\ref{triangle 1}) is the lower row of (\ref{modify diagram}).
\end{itemize}

By applying Lemma \ref{finite level reps} to the first, respectively second, set of data we obtain the upper, respectively, lower square in the commutative diagram of (iii), with the assertions concerning annihilation of cokernels following directly from Proposition \ref{eagon-northcott-prop}\,(i). Here we also use the fact that the argument of Proposition \ref{what we need from Selmer complexes}\,(v) explicitly describes the map $H^0(\rho_1')$ induced by the lower row of (\ref{modify diagram1}) in terms of the maps $\hat{v}_\q$ for $\q$ in $V(\m/\n)$ and the  map $H^0(\rho_2')$ induced by the lower row of (\ref{modify diagram}) in terms of the maps $\hat{\psi}_\q^{\mathrm{fs}}$ for $\q$ in $V(\m)$. \end{proof}

\subsection{The Kolyvagin derivative homomorphism} \label{koly section}

In this subsection, we fix a family of Nekov\'a\v{r} structures 
\[ \fF = (\mathscr{F}_K)_{K \in \Omega}\] 
that satisfies Hypothesis \ref{system of fF hyp} and also natural numbers $t$ and $i$. We also continue to use the notation fixed in (\ref{conv notation}) (so that $\tilde{\mathscr{F}}$ denotes $\mathscr{F}_{k,i}$ etc.). 

We shall first define a notion of Kolyvagin system of rank $t$ for the Nekov\'a\v{r} structure $\tilde{\mathscr{F}}$ and then refine arguments from \cite{bss} and \cite{Kato-Euler} in order to prove that Euler systems in $\ES^t (\fF)$ give rise to such Kolyvagin systems via a natural `derivative homomorphism' construction. 

\subsubsection{Kolyvagin systems for Nekov\'a\v{r} structures} 

Fix a modulus $\fn \in \cN$ and a prime $\q \in \cQ \setminus V(\fn)$. Then, by applying Lemma \ref{biduals lemma 1}\,(ii) to the exact sequence (\ref{global duality sequence3}) (with $\m$ and $\fn$ replaced by $\fn$ and $\q$), one obtains a map of $\bLambda$-modules
\[
\check{v}_\q \: \bidual^t_{\bLambda} H^1_{\!\tilde{\mathscr{F}} (\fn)} (k, \cA) \to \bidual^{t - 1}_{\bLambda} H^1_{\!\tilde{\mathscr{F}}_{\q} (\fn)} (k, \cA).
\]
In the same way, by applying Lemma \ref{biduals lemma 1}\,(ii) to (\ref{global duality sequence4}) (with $\a$ and $\m$ replaced by $\q$ and $\n$),
one obtains 
a map of $\Lambda$-modules  
\[ \check{\psi}_\q^\fs \:\bidual^t_{\bLambda} H^1_{\!\tilde{\mathscr{F}}( \fn)} (k, \cA) 
\to  \bidual^{t - 1}_{\bLambda} H^1_{\!\tilde{\mathscr{F}}_{\q} (\fn)} (k, \cA).  \]
We can now give a definition of Kolyvagin system that is appropriate for our theory. 

\begin{definition}\label{koly def}
   A `Kolyvagin system' of rank $t$ for the Nekov\'a\v{r} structure $\tilde{\mathscr{F}}$ is a family
    \[
    (\kappa_\fn)_{\fn} \in \prod_{\fn \in \cN}  \bidual^t_{\bLambda} H^1_{\tilde{\mathscr{F}}(\fn)} (k,\cA)   
    \]
    with the property that, for every $\fn \in \cN$ and $\q \in \cQ \setminus V(\fn)$, the `finite-singular relation'
    \[
    \check{v}_\q(\kappa_{\fn \q}) = \check{\psi}_\q^\fs (\kappa_{\fn}) 
    \]
    is valid in $\bidual^{t - 1}_{\bLambda} H^1_{\tilde{\mathscr{F}}_{\q} (\fn)} (k,\cA)$. The collection  of all such families is naturally a $\bLambda$-module that we denote by $\KS^t(\tilde{\mathscr{F}})$.
\end{definition}

To prepare for the statement of our main result concerning these systems, we note that, for each field $K \in \Omega$, Hypothesis \ref{system of fF hyp} implies the Nekov\'a\v{r} structure $\mathscr{F}_K$ satisfies Hypotheses \ref{fF hyp} (with $R = \cR [\cG_K]$). It follows that  Lemma \ref{det projection lemma}\,(ii) combines with Proposition \ref{what we need from Selmer complexes}\,(iii) and (v) to imply the existence of a natural `projection' map of $\cR [\cG_K]$-modules 
\begin{equation}\label{proj map def}
\pi_{K, \tilde{\mathscr{F}}}^t \: \bidual^t_{\cR [\cG_K]} H^1_\fF (K, \cT) 
\to \bidual^t_{\bLambda [\cG_K]} H^1_{\tilde{\mathscr{F}}} (K, \cA).
\end{equation}
For each modulus $\fn \in \cN$, we use the group $G_\fn$ defined in (\ref{group def}) and also fix a  pre-image $\NN_{\fn}$ under the (surjective) projection map $\Z [\cG_{k (\fn)}]\to \Z[\cG_{k (1)}]$ of the trace element  $\NN_{\cG_{k (1)}}$ defined in (\ref{norm def}). Then, by using the generator $\sigma_\q$ of $G_\q$ for $\q \in \cQ_i$ fixed at the beginning of \S\,\ref{koly sys I section}, we  define `derivative operators'
\[
D_\q \coloneqq \sum_{j\in [|G_\q| - 1]} j \sigma_\q^j \in \Z [G_\q],
\quad  
D_\fn \coloneqq \prod_{\q \in V(\fn)} D_\q \in \Z [G_\fn]
\quad \text{ and } \quad 
D'_\fn \coloneqq D_{\fn}\cdot\NN_{\fn} \in \Z [\cG_{k (\fn)}].
\]
The following result is well-known. 

\begin{lem} \label{bss lem 6.12} Fix $c = (c_K)_{K \in \Omega}$ of $\ES^t_{S_0} (\fF)$. Then, for every modulus $\fn \in \cN$, the element 
\[ D'_\fn \cdot \pi^t_{k(\fn), \tilde{\mathscr{F}}} ( c_{k (\fn)}) \in \bidual^t_{\bLambda [\cG_{k (\fn)}]} H^1_{\tilde{\mathscr{F}}} (k (\fn), \cA)\]
is fixed by $\cG_{k (\fn)}$ and also independent of the choice of lift $\NN_{\fn}$ of $\NN_{\cG_{k (1)}}$.
\end{lem}

\begin{proof}
   Both claims are true if, for all $\n$, one has $(\sigma - 1) D_\fn (\pi^t_{K, \tilde{\mathscr{F}}} ( c_{k (\fn)})) = 0$ for every $\sigma \in G_\fn$. Hence, since $\a_i$ annihilates $\bidual^t_{\bLambda[\cG_{k (\fn)}]} H^1_{\tilde{\mathscr{F}}} (k (\fn), \cA)$, it is enough to show that every element 
   $(\sigma - 1) D_\fn (c_{k (\fn)})$ belongs to $\a_i\cdot \bidual^t_{\cR[\cG_K]} H^1_\mathscr{F}(k(\fn), \cT) $. This is proved by the argument of  \cite[Lem.\@ 6.12]{bss}, which we now briefly explain for the convenience of the reader. 
   
   We use induction on $\nu (\fn)$ and may assume both that $\nu (\fn) > 1$ (since there is nothing to prove if $\fn = 1$) and also $\sigma = \sigma_\q$ for some $\q \in V (\fn)$ (since $G_\fn$ is generated by the set $\{ \sigma_\q : \q \in V (\fn)\}$). Then, since $(\sigma_\q  - 1) D_\q = |G_\q| - \NN_{G_\q}$, one has 
   \begin{align*}
       (\sigma_\q  - 1) D'_\fn c_{k (\fn)} & = ( |G_\q| - \NN_{G_\q}) D'_{\fn / \q} c_{k (\fn)}\\
       & = 
       |G_\q| D'_{\fn / \q} c_{k (\fn)} - \nu^t_{k (\fn \q) / k (\fn)} ( D'_{\fn / \q} \cdot\Eul_\q (\Frob_\q^{-1}) c_{k (\fn / \q)}),
   \end{align*}
   where the second equality holds by the Euler system norm relations. Now, since $\q \in \cQ$, the element $\Eul_\q (\Frob_\q^{-1})$ belongs to the augmentation ideal of $\bLambda [G_{\fn / \q}]$ and so the induction hypothesis implies $D'_{\fn / \q}\cdot \Eul_\q (\Frob_\q^{-1}) c_{k (\fn/\q)}$ belongs to $\a_i\cdot \bidual^t_{\bLambda [\cG_{k (\fn / \q)}]} H^1_{\tilde{\mathscr{F}}} (k (\fn / \q), \cA)$. In addition, the containment $\q \in \cQ$ also implies that $|\bLambda|$ divides $|G_\q|$. Since $|\bLambda|$ belongs to $\a_i$, this shows that $(\sigma_\q  - 1) D'_\fn c_{k (\fn)}$ is contained in  $\a_i \cdot\bidual^t_{\cR[\cG_K]} H^1_\mathscr{F} (k(\fn), \cT)$, as required.
\end{proof}

For each  modulus $\n \in \cN$ the Nekov\'a\v{r} structure
$\tilde{\mathscr{F}}^{k(\fn)}$ agrees with $\tilde{\mathscr{F}}^{\fn}$ (cf.\@ Remark \ref{modified nekovar = modified mr remark}). As a consequence, the assumed validity of Hypothesis \ref{system of fF hyp} combines with the argument of Lemma~\ref{construction of injection lemma} to imply the existence  of a canonical isomorphism of $\bLambda$-modules
\[\nu^t_{k (\fn) / k,\tilde{\mathscr{F}}} \:  \bidual^t_{\bLambda} H^1_{\tilde{\mathscr{F}}^{\fn}} (k, \cA) \xrightarrow{\simeq} 
\big( \bidual^t_{\bLambda [\cG_{k (\fn)}]} H^1_{\tilde{\mathscr{F}}} (k (\fn),\cA) \big)^{\cG_{k (\fn)}}.\] 

For each prime $\mathfrak{l} \in \cQ\setminus V (\fn)$, we may regard $\Eul_\mathfrak{l} (\Frob_\mathfrak{l}^{-1})$ as an element of $\cR [G_\fn]$. Then, writing  $I_\fn$ for the augmentation ideal of $\Z [G_\fn]$, the definition of $\cQ$ implies that the image of $\Eul_\mathfrak{l} (\Frob_\mathfrak{l}^{-1})$ in $\bLambda [G_\fn]$ belongs to  $\bLambda \otimes_\Z I_\fn$ and so defines an element of $\bLambda \otimes_\Z (I_\fn / I_\fn^2)$. In particular, for every $\q \neq \mathfrak{l}$, we may define an element $x_\mathfrak{l}^{(\q)}$ of $\bLambda$ by means of the equality
\begin{equation} \label{x fl definition}
\Eul_\mathfrak{l} (\Frob_\mathfrak{l}^{-1}) = x_\mathfrak{l}^{(\q)} \otimes (\sigma_\q - 1) 
\quad \text{ in } \quad \bLambda \otimes_\Z (I_\q / I_\q^2).
\end{equation}
We also write $\mathfrak{S} (\fn)$ for the group $\mathrm{Per} (V (\fn))$ of permutations of the set $V(\fn)$.

\begin{definition} \label{kolyvagin derivative def}
   Fix an element $c = (c_K)_{K \in \Omega}$ of $\ES^t_{S_0} (\fF)$. Then, for each modulus $\fn \in \cN$, 
   Lemma \ref{bss lem 6.12} allows us to define
\[
\kappa' (c)_\fn \coloneqq (\nu^t_{k (\fn) / k, \tilde{\mathscr{F}}})^{-1} ( \pi^t_{k(\fn), \tilde{\mathscr{F}}} (D'_\fn (c_{k(\fn)}))) \quad \in \bidual^t_{\bLambda} H^1_{\tilde{\mathscr{F}}^\fn} (k,\cA). 
\]
We then set 
    \[
    \kappa (c)_\fn \coloneqq {\sum}_{\tau \in \mathfrak{S} (\fn)} \mathrm{sgn} (\tau)  
    \cdot \big( {{\prod}}_{ \q \in V(\fn / \mathfrak{d}_\tau)} x^{(\q)}_{\tau (\q)} \big) \cdot \kappa'(c)_{\mathfrak{d}_\tau}
    \quad 
    \in  \bidual^a_{\bLambda}  H^1_{\tilde{\mathscr{F}}^\fn} (k, \cA),
    \]
    where $\mathfrak{d}_\tau$ denotes the product over all (prime ideals) $\q\in V(\n)$ with $\tau (\q) = \q$, and define the `Kolyvagin derivative' of $c$ to be the family  
    \[
    \kappa (c) \coloneqq (\kappa (c)_\fn)_{\fn \in \cN}
    \in 
    {{\prod}}_{\fn \in \cN} \bidual^t_{\bLambda}  H^1_{\tilde{\mathscr{F}}^\fn} (k, \cA).
    \]
\end{definition}

 We can now finally state the main result of \S\,\ref{koly section}.  

\begin{thm} \label{kolyvagin derivative thm}
Assume the family $\fF= (\mathscr{F}_K)_{K \in \Omega}$ of Nekov\'a\v{r} structures satisfies Hypothesis~\ref{system of fF hyp} and that  $\cK$ and $\cT$ satisfy Hypothesis \ref{more new hyps}. Fix an Euler system  $c$ in $\ES^t_{S_0} (\fF)$. Then, for each modulus $\fn \in \cN$ and prime $\q\in V(\fn)$, one has
    \[
    \kappa (c)_\fn \in  \bidual^t_{\bLambda} H^1_{\tilde{\mathscr{F}}(\fn)} (k,\cA)
    \]
and   
\[ \check{v}_\q ( \kappa (c)_\fn) = \check{\psi}_\q^\fs ( \kappa (c)_{\fn / \q}) \in \bidual^{t - 1}_{\Lambda} H^1_{\tilde{\mathscr{F}}_{ \q} (\fn/\q)} (k, \cA).\]
In particular, the assignment $c \mapsto \kappa(c)$ defines a well-defined homomorphism of $\cR$-modules
\[ \ES^t_{S_0} (\fF) \to \KS^t (\tilde{\mathscr{F}}).\] 
\end{thm}
\medskip

After some preliminary steps, the proof of this result will be completed in \S\,\ref{rank one proof}.

\subsubsection{The reduction of Theorem \ref{kolyvagin derivative thm} to rank one} 

In the special case that $\fF$ is the family $\fF_{\mathrm{rel}}(\cT)$ of relaxed Nekov\'a\v{r} structures discussed in Example \ref{relaxed es remark}, $\cR$ is equal to $\cO [\cG_K]$ for a finite extension $\cO$ of $\Z_p$ and a field $K \in \cK$ and the filtration $(\a_n)_n$ is $(p^n \cR)_n$, then the result of  Theorem \ref{kolyvagin derivative thm} coincides with \cite[Th.\@ 6.15]{bss}. We recall that the latter result is proved by first adapting a technique of Rubin \cite{Rub96} and Perrin-Riou \cite{PerrinRiou98} to reduce to the case $t = 1$, and then deriving this special case from an argument that is used by Mazur and Rubin to prove \cite[Th.\@ 3.2.4]{MazurRubin04}. We shall use the same strategy to prove Theorem \ref{kolyvagin derivative thm} for general Nekov\'a\v{r} structures, though in the case $t=1$ we shall also now rely on a modified version of arguments of Kato \cite[\S\,2]{Kato-Euler} that are compatible with our use of Selmer complexes. 

As a first step, we must therefore establish a version of the key technical result \cite[Lem.\@ 6.22]{bss} that is appropriate for Nekov\'a\v{r} structures.

\begin{prop} \label{bss lem 6.22}
 Fix a modulus $\fn$ in $\cN$. Then, for each element $\Phi \in \exprod^{t - 1}_{\bLambda} H^1_{\tilde{\mathscr{F}}^{\fn}} (k, \cA)^\ast$, there exists a family of elements  
    \[
    \Psi = (\Psi_K)_{K \in \Omega}  \in {\prod}_{K \in \Omega} \exprod^{t - 1}_{\cR [\cG_K]} H^1_{\mathscr{F}} (K, \cT)^\ast
    \]
    with both of the following properties.
    \begin{romanliste}
        \item For all $K$ and $L$ in $\Omega$ with $K \subseteq L$, and all $y \in \bidual^t_{\cR [\cG_K]} H^1_\mathscr{F} (K, \cT)$, one has
        \[
        (\Psi_L \circ \nu^t_{L / K}) (y) = (\nu^1_{L / K} \circ \Psi_K) (y) \in  H^1_\mathscr{F} (L, \cT),
        \]
         where the maps $\nu^t_{L / K}$ and $\nu^1_{L / K}$ are as defined in Lemma \ref{construction of injection lemma}.
        \item For each divisor $\m$ of $\fn$, we use the first map in (\ref{global duality sequence1}) to regard $H^1_{\tilde{\mathscr{F}}^{\m}} (k, \cA)$ as a submodule of $H^1_{\tilde{\mathscr{F}}^{\fn}} (k, \cA)$. Then, in $H^1_{\tilde{\mathscr{F}}^{\fn}} (k,\cA)$, one has 
    \[
    \Phi (
    (\nu^t_{k (\m) / k, \tilde{\mathscr{F}}})^{-1} (D'_\m \cdot \pi^t_{k (\m), \tilde{\mathscr{F}}} (c_{k(\m)})))
    = (\nu^1_{k (\m) / k,\tilde{\mathscr{F}}})^{-1} (D'_\m\cdot \pi^1_{k (\m),\tilde{\mathscr{F}}} ( \Psi_{k (\m)} ( c_{k (\m)}))).
    \]
    
    \end{romanliste}
\end{prop}

    \begin{proof} For any $K$ and $L$ in $\Omega$ with $K \subseteq L$, the canonical injective map (\ref{intermediate step injection2}) induces a restriction map $\exprod^{t - 1}_{\cR [\cG_L]} H^1_{\mathscr{F}} (L, \cT)^\ast \to \exprod^{t - 1}_{\cR [\cG_K]} H^1_{\mathscr{F}} (K, \cT)^\ast$. Via these maps, the construction of $\Psi$ is reduced to the consideration of a cofinal subset of $\Omega$. 
    
We therefore fix a cofinal subset $\{ F_n\}_{n \in \N}$ of $\Omega$ that is totally ordered with respect to inclusion, order-isomorphic to $\N$ and such that $F_1 = k (\fn)$ and, for each $n$, we set 
\[ \mathcal{S}_n \coloneqq \cR [\cG_{F_n}]\quad\text{and}\quad C_n \coloneqq C (\mathscr{F}_{F_n}).\] 
Then, to construct a family of the required sort, it is enough to inductively construct elements $\Psi_{F_n}$ of $\exprod^{t - 1}_{\mathcal{S}_n} H^1_{\mathscr{F}}(F_n, \cT)^\ast$ that have the claimed properties. To do this, we first inductively construct a family of finitely generated free $\mathcal{S}_n$-modules $P_n$ such that  $C_n$ has a resolution $P_n \xrightarrow{d_n} P_n$ (with the first term placed in degree zero) and a family of elements $f_n \in \exprod^{t - 1}_{\mathcal{S}_n} P_n^\ast$ with certain compatibility properties. \\
To construct the required objects for $n=1$, we note Lemma \ref{how the cohomology base changes lemma}\,(i) identifies $H^1_{\tilde{\mathscr{F}}^{\fn}} (k, \cA)$ with the 
$G_\fn$-invariants of $H^1_{\tilde{\mathscr{F}}} (k (\fn),\cA)$. Hence, upon taking duals over the self-injective ring $\bLambda [G_\fn]$, it follows that the restriction map
\begin{equation}\label{composite map}
    \exprod^{t - 1}_{\bLambda [G_\fn]} H^1_{\tilde{\mathscr{F}}} (k (\fn),\cA)^\ast \to \exprod^{t - 1}_{\bLambda[G_{\fn}]} \bigl(H^1_{\tilde{\mathscr{F}}} (k (\fn), \cA)^{G_\fn}\bigr)^\ast \cong \exprod^{t- 1}_{\bLambda} H^1_{\tilde{\mathscr{F}}^{\fn}} (k,\cA)^\ast
    \end{equation}
is surjective and so we can fix a pre-image $\widetilde{\Phi}$ of $\Phi$ under this map. Now, since $\mathscr{F}_{F_1}$ validates Hypothesis \ref{fF hyp}, Lemma \ref{finite level reps} implies $C_1$ has a resolution $P_1 \xrightarrow{d_1} P_1$, where $P_1$ is a finitely generated free $\mathcal{S}_1$-module (and the first term is placed in degree zero). Hence, from Proposition \ref{construction complex}\,(iii), it follows that  $C(\mathscr{F}_{F_1} \otimes_\cR \bLambda) \cong C_1 \otimes_\cR^\mathbb{L} \bLambda$ is isomorphic to $P_{1,i} \xrightarrow{d_{1,i}} P_{1,i}$ with $P_{1,i} \coloneqq P_1 \otimes_\cR \bLambda$ and $d_{1,i} \coloneqq d_1 \otimes_\cR \mathrm{1}_\bLambda$. In particular, since the natural composite map
    \[
    \exprod^{t - 1}_{\mathcal{S}_1} P_1^\ast \to \big( \exprod^{t - 1}_{\mathcal{S}_1} P_1^\ast \big) \otimes_\cR \bLambda \cong \exprod^{t - 1}_{\bLambda [G_\fn]} P_{1,i}^\ast \to \exprod^{t - 1}_{\bLambda [G_\fn]} \ker(d_{1,i})^\ast\cong 
    \exprod^{t - 1}_{\bLambda [G_\fn]} H^1_{\tilde{\mathscr{F}}} (k (\fn),\cA)^\ast
    \]
    is surjective (where the isomorphism follows from the fact $P_1$ is $\mathcal{S}_1$-free, and the third map is surjective since $\Lambda [G_\fn]$ is self-injective), we can fix a pre-image $f_1$ of $\widetilde \Phi$ under this map. 

Having constructed $P_1$ and $f_1$, we now pass to the inductive step. For this, we assume to be given a resolution $P_{n} \xrightarrow{d_{n}} P_{n}$ of $C_{n}$ (in which $P_n$ is a finitely generated free $\mathcal{S}_n$-module) and an element $f_n$ of $\exprod^{t - 1}_{\mathcal{S}_{n}} P_{n}^\ast$ and use them  to construct similar objects for $n+1$ in place of $n$.  

To do this, we set $\tilde C_{n} \coloneqq C (\mathscr{F}_{F_{n + 1}} \otimes_{\mathcal{S}_{n+1}} \mathcal{S}_n)$ and write $\Pi_n$ for the set of places of $F_n$ above those in $S(F_{n+1}) \setminus S(F_{n})$. Then the relevant cases of Hypothesis \ref{system of fF hyp}\,(ii) and Proposition \ref{construction complex}\,(vi) combine to imply the existence of a canonical exact triangle in $D^{\mathrm{perf}}(\mathcal{S}_n)$ 
\begin{equation}\label{limit change} 
C_n \xrightarrow{\alpha_n} \tilde C_n\to \bigoplus_{\q \in \Pi_n}  \mathrm{R}\Gamma_f (F_{n,\q}, \cT^\ast(1))^\ast[-1] \to \cdot.
\end{equation}
In addition, for  $\q\in \Pi_n$, the complex $\mathrm{R}\Gamma_f (F_{n,\q}, \cT^\ast(1))^\ast$ has a resolution $Q(\q) \to Q(\q)$, with $Q(\q)$ a finitely generated free $\mathcal{S}_n$-module (and the first term  placed in degree zero). Hence, setting $Q_n \coloneqq {\bigoplus}_{\q \in \Pi_n}Q(\q)$, a standard mapping cone construction combines with the fixed resolution of $C_n$ to imply $\tilde C_{n}$ is isomorphic in $D^{\mathrm{perf}}(\mathcal{S}_n)$ to a complex $P_{n}\oplus  Q_n \xrightarrow{d_n'} P_{n}\oplus Q_n$ (with the first term placed in degree zero) in such a way that $\alpha_n$ is induced by a commutative diagram of $\mathcal{S}_n$-modules 
\begin{cdiagram}[column sep=small]
P_n \arrow{r}{d_n} \arrow[hookrightarrow]{d}{\alpha^0_n} & P_n \arrow[hookrightarrow]{d}{\alpha^1_n} \\
    P_n \oplus Q_n \arrow{r}{d_n'} & P_n \oplus Q_n,
\end{cdiagram}%
in which $\alpha^0_n$ and $\alpha^1_n$ are the natural inclusion maps. In particular, 
since restriction through $\alpha_n^0$ induces a surjective map
$\kappa: \exprod^{t - 1}_{\mathcal{S}_n} (P_n \oplus Q_n)^\ast \to \exprod^{t - 1}_{\mathcal{S}_n} P_n^\ast$, we can fix  an element $\widetilde f_n$ of $\exprod^{t - 1}_{\mathcal{S}_n} (P_n \oplus Q_n)^\ast$ with $\kappa(\widetilde f_n) = f_n$. 

Next we note that Proposition \ref{construction complex}\,(iv) induces an isomorphism in $D^{\mathrm{perf}}(\mathcal{S}_n)$
\begin{equation}\label{qi descent} C_{n+1} \otimes^\mathbb{L}_{\mathcal{S}_{n+1}} \mathcal{S}_n \cong \tilde C_{n},\end{equation}
and hence also an isomorphism of $\mathcal{S}_n$-modules $H^1(C_{n+1})\otimes_{\mathcal{S}_{n + 1}} \mathcal{S}_n \cong H^1(\tilde C_n)$. 

We now fix a finitely generated free $\mathcal{S}_{n + 1}$-module $P_{n + 1}$ for which there exists an isomorphism of $\mathcal{S}_n$-modules from $(P_{n+1})_n\coloneqq P_{n+1}\otimes_{\mathcal{S}_{n+1}}\mathcal{S}_n$ to $P_n\oplus Q_n$, and write $\iota_n$ for the induced surjective map of $S_{n+1}$-modules $P_{n + 1} \twoheadrightarrow P_n \oplus Q_n$. The  resolution of $\tilde C_{n}$ fixed above induces a surjective map of $\mathcal{S}_n$-modules $j_n \: P_n \oplus Q_n \twoheadrightarrow H^1 (\tilde C_n)$. Then, since $P_{n+1}$ is a free $\mathcal{S}_{n+1}$-module, we can fix a commutative diagram of $\mathcal{S}_{n+1}$-modules 
\begin{cdiagram}[row sep=small]
P_{n+1} \arrow[twoheadrightarrow]{r}{j_{n + 1}} \arrow[twoheadrightarrow]{d}{\iota_n} & H^1(C_{n+1}) \arrow[twoheadrightarrow]{d} \\
P_n\oplus Q_n \arrow[twoheadrightarrow]{r}{j_{n}} & H^1(\tilde C_{n})
\end{cdiagram}%
in which the right hand vertical map is induced by (\ref{qi descent}).
Note here that  Nakayama's Lemma implies any map $j_{n + 1}$ in such a diagram is surjective since $j_n\circ \iota_n$ is surjective and the  kernel of the projection $\mathcal{S}_{n+1} \to \mathcal{S}_n$ lies in the Jacobson radical of $\mathcal{S}_{n+1}$. 
The argument of Lemma \ref{finite level reps}\,(i) now implies that $C_{n+1}$ has a resolution  $P_{n+1}\xrightarrow{d_{n+1}} P_{n+1}$, in which the first term is placed in degree zero and the isomorphism $\cok(\phi_{n+1}) \cong H^1(C_{n+1})$ is induced by $j_{n+1}$, and the descent isomorphism (\ref{qi descent}) is induced by a commutative diagram of  $\mathcal{S}_{n+1}$-modules 
\begin{cdiagram}[row sep=small, column sep=small]
P_{n+1} \arrow{r}{d_{n+1}} \arrow{d}{i'_n}  & P_{n+1}\arrow{d}{i_n}\\
P_n\oplus Q_n \arrow{r}{d_n'} & P_n\oplus Q_n.
\end{cdiagram}%
In particular, since this morphism of complexes induces a quasi-isomorphism between the lower complex and the image under $-\otimes_{\mathcal{S}_{n+1}}\mathcal{S}_n$ of the upper complex, and the map $i_{n}\otimes_{\mathcal{S}_{n+1}}\mathcal{S}_n$ is bijective, the map $i_n'$ must induce an isomorphism $(P_{n+1})_n \cong P_n\oplus Q_n$. The map $i'_n$ therefore induces a surjective map
\[
\exprod^{t - 1}_{\mathcal{S}_{n + 1}} P_{n + 1}^\ast \twoheadrightarrow \big ( \exprod^{t - 1}_{\mathcal{S}_{n + 1}} P_{n + 1}^\ast \big) \otimes_{\mathcal{S}_{n + 1}} \mathcal{S}_n
\cong \exprod^{t - 1}_{\mathcal{S}_n} (P_n \oplus Q_n)^\ast 
\]
and we fix a pre-image $f_{n + 1}$ of $\widetilde{f}_n$ under this map. Having inductively constructed modules $P_n$ and elements $f_n$, we now claim that the properties (i) and (ii) are satisfied by the elements 
\[ \Psi_{F_n} \coloneqq \rho_n (f_n) \in \exprod^{t - 1}_{\mathcal{S}_n} H^1_\mathscr{F} (F_n, \cT)^\ast.\]
Here $\rho_n$ denotes the restriction map $\exprod^{t - 1}_{\mathcal{S}_n} P_n^\ast \to \exprod^{t - 1}_{\mathcal{S}_n} \ker(d_n)^\ast\cong \exprod^{t - 1}_{\mathcal{S}_n} H^1_\mathscr{F} (F_n, \cT)^\ast$, in which the last isomorphism is induced by Proposition \ref{construction complex}\,(v). \\
It is enough to verify (i) with $L/K$  taken to be $F_{n+1}/F_{n}$. The key observation for this case is that, for each $y \in \exprod^{t - 1}_{\mathcal{S}_{n + 1}} P_{n + 1}$, the composite map
\[
\exprod^{t - 1}_{\mathcal{S}_n} (P_{n + 1})_n \cong  
\exprod^{t - 1}_{\mathcal{S}_n} (P_{n + 1})^{\gal{F_{n + 1}}{F_n}}
\cong 
\big( \exprod^{t - 1}_{\mathcal{S}_{n + 1}} P_{n + 1} \big)^{\gal{F_{n + 1}}{F_n}} \hookrightarrow
\exprod^{t - 1}_{\mathcal{S}_{n + 1}} P_{n + 1} 
\]
sends $\NN_{\gal{F_{n + 1}}{F_n}}^{t - 1}y$ to $\NN_{\gal{F_{n + 1}}{F_n}}y$. It follows that 
\[
\NN_{\gal{F_{n + 1}}{F_n}} \cdot f_{n} ( \NN_{\gal{F_{n + 1}}{F_n}}^{t} y) = 
f_{n + 1} ( \NN_{\gal{F_{n + 1}}{F_n}} y),
\]
and this implies the equality in  (i) since the assignment $\NN_{\gal{F_{n + 1}}{F_n}}^{t} y \mapsto \NN_{\gal{F_{n + 1}}{F_n}} y$ restricts to give the map $\nu^t_{F_{n + 1} / F_n}$ on $\bidual^t_{\mathcal{S}_n} H^1_{\mathscr{F}} (F_n, \cT)$. \\
To verify (ii) we fix a divisor $\m$ of $\n$. We regard $H^1_{\tilde{\mathscr{F}}} (k (\m),\cA)$ as a submodule of $H^1_{\tilde{\mathscr{F}}} (k (\n), \cA)$ (just as in (\ref{intermediate step injection2})) and write $\widetilde \Phi_\m$ for the image of $\widetilde \Phi$ under the induced restriction map 
\[ \exprod^{t - 1}_{\Lambda [G_\fn]} H^1_{\tilde{\mathscr{F}}} (k (\fn), \cA)^\ast\to \exprod^{t - 1}_{\Lambda [G_\m]} H^1_{\tilde{\mathscr{F}}} (k (\m), \cA)^\ast.\] 
Then, since $\Psi_{k (\m)}$ is the restriction of $\Psi_{k (\n)} = \rho_n(f_1)$ and $\widetilde \Phi$ is a pre-image of $\Phi$ under the  restriction map (\ref{composite map}), our explicit choice of $f_1$ implies that 
\[
D'_\m \pi^1_{k(\m), \tilde{\mathscr{F}}} (\Psi_{k (\m)} (c_{k (\m)})) = D'_\m \widetilde \Phi_\m (\pi^1_{k(\m), \tilde{\mathscr{F}}} ( c_{k(\m)})) = \widetilde \Phi_\m (D'_\m\pi^1_{k(\m), \tilde{\mathscr{F}}} ( c_{k(\m)})).
\]
Hence one has  
\begin{align*}
&(\nu^1_{k (\m), \tilde{\mathscr{F}}})^{-1} ( D'_\m \pi^1_{k(\m), \tilde{\mathscr{F}}} (\Psi_{k (\m)} (c_{k (\m)})))\\
=\, &(\nu^1_{k (\m), \tilde{\mathscr{F}}})^{-1} (  \widetilde \Phi_\m (D'_\m\pi^1_{k(\m),\tilde{\mathscr{F}}} ( c_{k(\m)}))) \\
 =\, &(\nu^1_{k (\m), \tilde{\mathscr{F}}})^{-1} (  (\widetilde \Phi_\m \circ \nu^t_{k (\m), \tilde{\mathscr{F}}} \circ (\nu^t_{k (\m),\tilde{\mathscr{F}}})^{-1}) (D'_\m\pi^1_{k(\m), \tilde{\mathscr{F}}} ( c_{k(\m)}))) \\ 
 =\, &\Phi ( (\nu^t_{k(\m) / k,\tilde{\mathscr{F}}})^{-1} (D'_\m\pi^1_{k(\m), \tilde{\mathscr{F}}} ( c_{k(\m)}))),
\end{align*}
where the last equality uses the fact $ \widetilde \Phi_\m \circ \nu^t_{k (\m) / k,\tilde{\mathscr{F}}} = \nu^1_{k (\m) / k,\tilde{\mathscr{F}}} \circ \Phi$ on 
 ${\bigcap}^t_{\Lambda} H^1_{\tilde{\mathscr{F}}^{\m}} (k, \cA)$.  
This verifies the required property (ii).  
\end{proof}

A collection $\Psi = (\Psi_K)_{K \in \Omega}$ of maps that satisfy the relations in Proposition \ref{bss lem 6.22}\,(i) is often referred to as a `Perrin-Riou functional' (see, for example, \cite[Def.\@ 8.1.2]{LeiLoefflerZerbes14}). The key observation, made independently by Rubin in \cite[\S\,6]{Rub96} and by Perrin-Riou in \cite[\S\,1.2.3]{PerrinRiou98}, is that any such collection gives rise to a well-defined map of $\cR \llbracket \cG_\cK \rrbracket$-modules
\[ \ES^t_{S_0} (\fF_{\mathrm{rel}}(\cT)) \to \ES^1_{S_0} (\fF_{\mathrm{rel}}(\cT)), \quad (c_K)_{K \in \Omega} \mapsto ( \Psi_K (c_K))_{K \in \Omega}\] 
where $\fF_{\mathrm{rel}}(\cT)$ is as defined  in Example \ref{relaxed es remark}. By precisely the same argument, one verifies that this result remains valid after replacing $\fF_{\mathrm{rel}}(\cT)$ by any family of Nekov\'a\v{r} structures $\fF$ that  satisfies  Hypothesis \ref{system of fF hyp}. This observation then combines  with Proposition \ref{bss lem 6.22} to prove the following reduction result for Theorem \ref{kolyvagin derivative thm}. 

\begin{lemma} \label{reduction to rank one}
To prove Theorem \ref{kolyvagin derivative thm} it is enough to consider the case $t=1$.
\end{lemma}

\begin{proof} We first show that, for each $\fn \in \cN$, the element $\kappa (c)_\fn$ belongs to $\bidual_{\bLambda}^t H^1_{\tilde{\mathscr{F}}(\fn)}(K,\cA)$. For this,  Corollary \ref{prop 2.4 extension} implies it suffices to show that, for every $\Phi \in {\bigwedge}_{\bLambda}^{t-1}H^1_{\tilde{\mathscr{F}}^\fn}(k,\cA)^\ast$ one has 
\begin{equation} \label{trans}
\Phi(\kappa(c)_\fn)=\sum_{\tau \in \mathfrak{S} (\fn)} \mathrm{sgn} (\tau)  
    \cdot \big( \prod_{ \q \in V(\fn / \mathfrak{d}_\tau)} x^{(\q)}_{\tau (\q)} \big) \cdot \Phi (\kappa'(c)_{\mathfrak{d}_\tau}) 
    \quad \in H^1_{\tilde{\mathscr{F}}(\fn)}(k,\cA).
\end{equation}
However, by Proposition \ref{bss lem 6.22}, there exists an element $\Psi \in \varprojlim_F {\bigwedge}_{\cR[\cG_F]}^{t-1}H^1_{\mathscr{F}} (F,\cT)^\ast$ with the property that, for every divisor $\fd$ of $\fn$, one has
\[
\Phi(\kappa'(c_\fd))=\kappa'(\Psi(c)_\fd).
\]
 In particular, since $\Psi(c)$ belongs to $\ES^1_{S_0} (\fF)$, the containment (\ref{trans}) is valid for all $t$ if it is valid for $t=1$. \\
We show next that, for every $\fn \in \cN$ and $\fq \in V(\fn)$, one has 
\[
v_\fq(\kappa(c)_\fn)=\psi_\fq^{\rm fs}(\kappa(c)_{\fn/\fq}).
\]
For this we note that, since $\bidual_{\bLambda}^t H^1_{\tilde{\mathscr{F}}(\fn)}(k,\cA)$ is defined to be $( \bigwedge_{\bLambda}^t H^1_{\tilde{\mathscr{F}}(\fn)}(k,\cA)^\ast)^\ast$, the element  $\kappa(c)_\fn$ of $\bidual_{\bLambda}^t H^1_{\tilde{\mathscr{F}}(\fn)}(k,\cA)$ defines a map
$\bigwedge_{\bLambda}^t H^1_{\tilde{\mathscr{F}}(\fn)}(k,\cA)^\ast \to \bLambda$. Given this, $v_\fq(\kappa(c)_\fn)$ is the map
\[
{{\bigwedge}}_{\bLambda}^{t-1} H^1_{\tilde{\mathscr{F}}(\fn)}(k,\cA)^\ast \to \bLambda, \quad \Phi \mapsto \kappa(c)_\fn(v_\fq \wedge \Phi).
\]
In particular, if we identify ${\bigcap}_{\bLambda}^1 H^1_{\tilde{\mathscr{F}} (\fn)}(k,\cA)$ with $H^1_{\tilde{\mathscr{F}}(\fn)}(k,\cA)$, and thereby regard $\Phi(\kappa(c)_\fn)$ as an element of $H^1_{\tilde{\mathscr{F}}(\fn)}(k,\cA)$, then we have
\[
\kappa(c)_\fn(v_\fq \wedge \Phi)=(-1)^{t-1} \cdot v_\fq(\Phi(\kappa(c)_\fn)).
\]
Similarly, one finds that $\psi_\fq^\mathrm{fs}(\kappa(c)_{\fn/\fq}) \in {{\bigcap}}_{\bLambda}^{t-1}H^1_{\tilde{\mathscr{F}}(\fn)}(k,\cA)$ identifies with the map
\[
{{\bigwedge}}_{\bLambda}^{t-1}H^1_{\tilde{\mathscr{F}}(\fn)}(k,\cA)^\ast \to \Lambda, \quad \Phi \mapsto \kappa(c)_{\fn/\fq}(\psi_\fq^\mathrm{fs} \wedge \Phi)=(-1)^{t-1}\cdot \psi_\fq^\mathrm{fs}(\Phi(\kappa(c)_{\fn/\fq})).
\]
To verify the claimed equality $v_\fq(\kappa(c)_\fn)=\psi_\fq^\mathrm{fs}(\kappa(c)_{\fn/\fq})$, it is thus enough to show that, for every $\Phi \in {{\bigwedge}}_{\bLambda}^{t-1}H^1_{\tilde{\mathscr{F}}(\fn)}(k,\cA)^\ast$, one has an equality
\begin{eqnarray} \label{fs}
v_\fq(\Phi(\kappa(c)_\fn))=\psi_\fq^{\rm fs}(\Phi(\kappa(c)_{\fn/\fq}))
\end{eqnarray}
Moreover, by using Proposition \ref{bss lem 6.22}, we can choose an element $\Psi$ of $\varprojlim_F {\bigwedge}_{\cR[\cG_F]}^{a-1}H^1_{\mathscr{F}}(F,\cT)^\ast$ with both $\Phi(\kappa(c)_\fn)=\kappa(\Psi(c))_\fn$ and $\Phi(\kappa(c)_{\fn/\fq})=\kappa(\Psi(c))_{\fn/\fq}$. In particular, since $\Psi(c)$ belongs to $\ES^1_{S_0} (\fF)$, the required equality (\ref{fs}) will therefore follow if Theorem \ref{kolyvagin derivative thm}  is known to be valid in the case $t=1$. This proves the claimed result. 
\end{proof}

\subsubsection{The proof of Theorem \ref{kolyvagin derivative thm} in rank one}\label{rank one proof}

In this section we shall prove (in Propositions \ref{finite-singular relations prop} and \ref{finite-singular projections prop}) the result of Theorem \ref{kolyvagin derivative thm} in the special case $t = 1$. In view of Lemma \ref{reduction to rank one}, the results presented here will therefore also complete the proof of Theorem \ref{kolyvagin derivative thm} for arbitrary $t$.

In the first result, we adapt an approach of Kato (from \cite[\S\,2]{Kato-Euler}), to show that Euler systems for Nekov\'a\v{r} structures satisfy the `congruence condition' discussed in \cite[\S\,4.8]{Rubin-euler}. For this, for each $\q \in \cQ$ and $K \in \Omega$, we write 
\[
\loc_{K, \q} \: H^1_{\fF} (K, \cT) \to \textstyle \bigoplus_{v \in \{\q\}_K} H^1 (K_v, \cT) \cong H^1 (k_\q, \cT [\cG_K])
\]
for the natural localisation map. We will also often use the fact that, for every  $\n \in \cN$ and $\q \in V (\fn)$, the extension $k (\fn) / k (\fn / \q)$ is totally ramified at all places above $\q$, and hence that for every $i$ there is a natural isomorphism $H^i_f (k_\q, \cT [\cG_{k (\fn)}]) \cong H^i_f (k_\q, \cT [\cG_{k (\fn / \q)}])$.

\begin{prop} \label{the congruence condition}
    Assume $\fF$ satisfies Hypothesis \ref{system of fF hyp} and that $\cK$ and $\cT$ satisfy Hypothesis \ref{more new hyps}. Then, for each $\fn \in \cN$, $\q \in V (\fn)$, and  $c \in \ES^1_{S_0} (\fF)$, the following claims are valid.
    \begin{romanliste}
        \item $\Eul_\q (\Frob_\q^{-1}) - \Eul_\q (\NN \q \cdot\Frob_\q^{-1}) \in [k (\q) : k(1)]\cdot \cR [G_{\fn / \q}]$.
        \item $\loc_{k (\fn), \q} (c_{k (\fn)})$ and $\loc_{k (\fn/ \q), \q} (c_{k (\fn/ \q)})$ belong to $H^1_f (k_\q, \cT [\cG_{k (\fn)}]) \cong H^1_f (k_\q, \cT [\cG_{k ({\fn / \q})}])$.
        \item One has an equality
    \[
    \loc_{k (\fn), \q} ( c_{k (\fn)}) = \left ( \frac{\Eul_\q (\Frob_\q^{-1}) - \Eul_\q (\NN \q \cdot\Frob_\q^{-1})}{[k (\q) : k(1)]} \right) \cdot \loc_{k (\fn / \q), \q} (c_{k (\fn/ \q)}).
    \]
    \end{romanliste}
\end{prop}

\begin{proof} Claim (i) is true since $ \Eul_\q (\Frob_\q^{-1}) - \Eul_\q (\NN \q \cdot\Frob_\q^{-1})$ belongs to $(\NN \q - 1)\cdot\cR [G_{\fn / \q}]$ and $\NN \q - 1$ is divisible by $[k (\q) : k(1)]$. \\
    We next note that $\im(\loc_{k (\fn / \q), \q})\subseteq H^1_f (k_\q, \cT [\cG_{k (\fn / \q)}])$ because $\q\notin S (\mathscr{F}_{k (\fn / \q)})$. In particular,  since $c_{k (\fn / \q)}\in H^1_{\fF} (k (\fn / \q), \cT)$, one has $\loc_{k (\fn / \q), \q} (c_{k (\fn / \q)})\in H^1_f (k_\q, \cT [\cG_{k ({\fn / \q})}])$. To prove  (ii) it is therefore sufficient to prove $ \loc_{k (\fn), \q} (c_{k (\fn)})$ belongs to $H^1_f (k_\q, \cT [\cG_{k (\fn)}])$.\\ 
    To do this, we  use the fact that, by the assumed validity of Hypothesis \ref{more new hyps}\,(ii), the field $\cK$ contains a $\Z_p$-extension $k_\infty$ of $k$ in which no finite place splits completely.
     We then set $k (\fn)_\infty \coloneqq k_\infty \cdot k (\n)$ and $k (\fn)_n \coloneqq k_n \cdot k(\fn)$ for $n \in \N$, where $k_n$ denotes the $n$-th layer of $k_n / k$. We further set  $\mathcal{S} \coloneqq \cR \llbracket \cG_{k (\fn)_\infty} \rrbracket$ and regard the free $\mathcal{S}$-module $\cT \otimes_\cR \mathcal{S}$ as a $G_k$-module via the action  $\sigma \cdot (a \otimes b) \coloneqq (\sigma a) \otimes (b \overline{\sigma}^{-1})$ for  $\sigma \in G_k$, where $\overline{\sigma}$ denotes the image of $\sigma$ in $\cG_{k_\infty}$. Then Shapiro's lemma induces a canonical isomorphism 
    \[
    H^1 (k_\q, \cT \otimes_\cR \mathcal{S}) \cong {\varprojlim}_{n \in \N} H^1 (k_\q, \cT [\cG_{k (\fn)_n}]),
    \]
    where the transition morphisms in the limit are the natural corestriction maps. In particular, the Euler system distribution relations imply that the family 
\[ d_\fn \coloneqq (\loc_{k (\fn)_n, \q} (c_{k (\fn)_n}))_{n \in \N}\] 
defines an element of $H^1 (k_\q, \cT \otimes_\cR \mathcal{S})$. Hence, if we can prove that the natural composite map
    \[
    {\varprojlim}_{n \in \N} H^1_f (k_\q, \cT [\cG_{k (\fn)_n}]) \cong H^1_f (k_\q, \cT \otimes_\cR \mathcal{S}) \to H^1 (k_\q, \cT \otimes_\cR \mathcal{S})
    \]
    is bijective, then one has $\loc_{k (\fn), \q} (c_{k (\fn)})\in H^1_f (k_\q, \cT [\cG_{k (\fn)}])$, as claimed. Further, if $\cR$ is finitely generated over $\Z_p$, then the bijectivity of the displayed map is proved by Rubin in \cite[Prop.\@ B.3.4]{Rubin-euler}, and so we now adapt the argument of loc.\@ cit.\@ to our more general setting.\\
  For this, we write $F$ for the completion of $k (\n)$ at a place above $\q$ and  $F_n$ for the $n$-th layer of the (unique) unramified $\Z_p$-extension of $F$. The inertia subgroup $\mathcal{I} \coloneqq \mathcal{I}_{F_n} \subseteq G_{F_n}$ is then independent of $n$ because $F_n / F$ is unramified. 
    Since $\mathcal{I}$ acts trivially on $\cT$, the relevant case of the inflation-restriction sequence gives an exact sequence
    \begin{equation} \label{local inflation-restriction}
    0 \to H^1_f (F_n, \cT) \to H^1 (F_n, \cT) \to H^1 (\mathcal{I}, \cT)^{G_{F_n} / \mathcal{I}} = \Hom_\mathrm{cont} ( \mathcal{I}^{(p)}, \cT)^{G_{F_n}},
    \end{equation}
    where $\mathcal{I}^{(p)}$ denotes the pro-$p$ completion of $\mathcal{I}$. Local class field theory implies that $\mathcal{I}^{(p)}$ is a finitely generated $\Z_p$-module and so $X \coloneqq \Hom_\mathrm{cont} ( \mathcal{I}^{(p)}, \cT)$ is a finitely generated $\cR$-module. It follows that $X$ is Noetherian (since $\cR$ is) and $p$-torsion free (since $\cT$ is, by Remark \ref{hyps cons}). In particular, the ascending chain $( X_n)_{n \in \N}$  with $X_n \coloneqq \Hom_\mathrm{cont} ( \mathcal{I}^{(p)}, \cT)^{G_{F_n}}$ becomes stationary and so, for some $m \in \N$, one has $X_n = X_m$ for all $n \geq m$. For each $n \ge m$, the norm map $X_{n + 1} \to X_n$ is therefore induced by multiplication by $p$, and so $\varprojlim_{n \in \N} X_n = (0)$. Upon taking the limit (over $n$) of the exact sequence (\ref{local inflation-restriction}), one therefore obtains an isomorphism $\varprojlim_{n \in \N} H^1_f (F_n, \cT) \cong \varprojlim_{n \in \N} H^1 (F_n, \cT)$. Since (by assumption) the set $\Sigma_\q$ of places $k(\fn)_\infty$ above $\q$ is finite, this in turn induces an isomorphism
    \[
    H^1_f (k_\q, \cT \otimes_\cR \mathcal{S}) \cong \varprojlim_{n \in \N} \bigoplus_{v \in \Sigma_\q} H^1_f (k (\fn)_{n, v}, \cT) \to \varprojlim_{n \in \N} \bigoplus_{v\in \Sigma_\q} H^1 (k (\fn)_{n, v}, \cT)
     \cong  H^1 (k_\q, \cT \otimes_\cR \mathcal{S}),
    \]
    as required to complete the proof of (ii). \\
Before proving (iii), we claim $H^1_f (k_\q, \cT \otimes_\cR \mathcal{S})$ is torsion free. To see this, for every $v \in \Sigma_\q$ we
set $\mathcal{S}_v \coloneqq \cR \llbracket \gal{k (\fn)_{\infty,v}}{k_\q}\rrbracket$  and note 
Shapiro's Lemma gives a direct sum decomposition
\[ 
H^1_f ( k_\q, \cT \otimes_\cR \mathcal{S}) \cong \bigoplus_{v \in \Sigma_\q} H^1_f (k_\q,\cT\otimes_\cR\mathcal{S}_v) \cong \bigoplus_{v \in \Sigma_\q}  H^1 ( \mathbb{F}_v, \cT ), 
\]
in which $\mathbb{F}_v$ denotes the residue field of $k ( \fn)_{\infty, v}$. Now, since every finite extension of $\mathbb{F}_v$ has degree coprime to $p$, $\Frob_v$ must act trivially on $\cT$. It follows that there are isomorphisms 
\[ H^1 ( \mathbb{F}_v, \cT ) \cong \cT / (1 - \Frob_v) \cong \cT\] 
and hence that $H^1_f (k_\q, \cT \otimes_\cR \mathcal{S})$ is torsion free since $\cT$ is (by Remark \ref{hyps cons}).\\
Turning now to (iii), we observe that (ii) implies (iii) is valid if, in $H^1_f (k_\q, \cT \otimes_\cR \mathcal{S})$, one has 
    \begin{equation*} \label{eqn in Iwasawa theory 1}
    d_\fn = \left ( \frac{\Eul_\q (\Frob_\q^{-1}) - \Eul_\q (\NN \q \cdot\Frob_\q^{-1})}{[k (\q) : k(1)]} \right) \cdot d_{\fn / \q}
    \end{equation*}
where $d_\fn$ and $d_{\fn/\q}$ are as defined above. In addition, since $H^1_f (k_\q, \cT \otimes_\cR \mathcal{S})$ is torsion free, this equality will follow if we can prove that
\begin{equation} \label{eqn in Iwasawa theory 2}
   [k (\q) : k(1)] \cdot  d_\fn =  \left ( \Eul_\q (\Frob_\q^{-1}) - \Eul_\q (\NN \q \cdot\Frob_\q^{-1}) \right) \cdot d_{\fn / \q}.
    \end{equation}
    To do this, we note first that the trace element $\NN_{G_\q}$ acts as multiplication by $[k (\q) : k (1)]$ on $H^1_f (k_\q, \cT [\cG_{k (\fn)_n}]) = H^1 (\mathbb{F}_\q, \cT [\cG_{k (\fn)_n}])$, and hence that
    \[
    [k (\q) : k(1)] \cdot \loc_{k (\fn)_n, \q}  (c_{k (\fn)_n}) = \loc_{k (\fn)_n, \q} ( \NN_{G_\q} c_{k (\fn)_n}) = \Eul_\q (\Frob_\q^{-1}) \cdot \loc_{k (\fn / \q)_n, \q} ( c_{k (\fn / \q)_n}).
    \]
    In addition, the exact sequence  
    \[
    0 \to \cT  [\cG_{k (\fn/ \q)_n}] \xrightarrow{1 - \Frob_v} \cT  [\cG_{k (\fn / \q)_n}] \to H^1 ( \mathbb{F}_\q, \cT [\cG_{k (\fn / \q)_n}]) \to 0
    \]
   implies that $H^1 ( \mathbb{F}_\q, \cT [\cG_{k (\fn / \q)_n}])$ is annihilated by the element 
    \[ {\det}_{\cR [\cG_{k (\fn / \q)_n}]} ( 1 - \Frob_v \mid \cT [\cG_{k (\fn / \q)_n}] ) = \Eul_\q ( \NN \q \cdot\Frob_\q^{-1}) \in \cR [\cG_{k (\fn / \q)_n}].\] 
    These facts combine to prove the equality (\ref{eqn in Iwasawa theory 2}), and hence also  (iii).
\end{proof}

In order to prove the relevant `finite-singular relations', we recall certain maps introduced by Kato in \cite[\S\,4]{Kato-Euler}. To do this, we fix $\fn \in \cN$ and $\q \in \cQ$, and define a composite homomorphism 
\begin{multline*} \Psi_{\fn, \q}^\fs: H^1_f (k_\q, \cT [\cG_{k (\fn)}]) \to H^1_f (k_\q, \cA [\cG_{k (\fn)}])\\ \xrightarrow{(\ref{evaluation def})} (\cA [G_{k (\fn)}])^{\mathcal{I}_{k (\fn), \q}} / (\Frob_\q - 1) \xrightarrow{c_{1 - \Frob_\q}} H^0_f ( k_\q, \cA [G_{k (\fn)}]).\end{multline*}
Here the first map is induced by the projection $\cT \to \cA$, we write 
$\mathcal{I}_{k (\fn), \q} \subseteq \cG_{k ( \fn)}$ for the inertia subgroup at $\q$ (so $\mathcal{I}_{k (\fn), \q} = G_{k(\q)}$ if $\q \in V(\fn)$ and $\mathcal{I}_{k (\fn), \q}$ is trivial otherwise), and the `cofactor map' $c_{1 - \Frob_\q}$ (as defined in the proof of Lemma \ref{cfactor}) is induced by multiplication by $1 - \Frob_\q$ on $\cA [\cG_{k(\fn)} / \mathcal{I}_{k(\fn), \q}]$.\\
We also have a well-defined `evaluation' map  
\[
\mathrm{ev}_{\sigma_\q} \: H^1 (k_\q, \cA [\cG_{k (\fn)}]) \to 
H^1 ( k_\q^\mathrm{nr}, \cA [\cG_{k (\fn)}])^{G_{\kappa_\q}} \to 
\big(\cA [\cG_{k (\fn)}] / (\sigma_\q - 1)\big)^{G_{\kappa_\q}}  , \quad [\xi] \mapsto \xi ( \sigma_\q).
\]
If $\q \not \in V (\fn)$, we thereby obtain from $\mathrm{ev}_{\sigma_\q}$ a map 
\[
V_{\fn, \q} \:  H^1 (k_\q, \cA [\cG_{k (\fn)}]) \xrightarrow{\mathrm{ev}_{\sigma_\q}} H^0_f (k_\q, \cA [\cG_{k(\fn)}]) \cong H^0_f (k_\q, \cA [\cG_{k(\fn \q)}]),
\]
where the last isomorphism is a consequence of the $k (\fn \q) / k (\fn)$ being totally ramified at $\q$.\\
The above maps $\Psi^\fs_{\fn, \q}$ and $V_{\fn, \q}$  are then related to the maps  $\check{\psi}_\q^\fs$ and $\check{v}_\q$ that occur in the definition of Kolyvagin systems relative to Nekov\'a\v{r} structures in the following way.

\begin{lem} \label{fs relations lemma 1}
Fix  $c \in \ES^1_{S_0} (\fF)$, $\fn \in \cN$ and $\q \in \cQ$ and write $h$ for the isomorphism of $\cR$-modules $\cA / (\tau - 1) \cong \bLambda$  fixed just before (\ref{v_q def}).  Then one has 
    \[
    (\Psi^\fs_{\fn, \q} \circ \loc_{k (\fn), \q}) ( D'_\fn c_{k (\fn)}) = (\iota_{\fn, \q} \circ h^{-1} \circ \check{\psi}^\fs_\q) (\kappa' (c)_\fn). 
    \]
    In addition, if $\q \in V(\fn)$, and 
    we set
    \[
    x_{\fn, \q} \coloneqq (\nu_{k (\fn) / k (\fn/ \q), \widetilde{\mathscr{F}}}^1)^{-1} 
    ( \pi^1_{k (\fn),  \widetilde{\mathscr{F}}}
    (D'_\fn c_{k (\fn)})) \in H^1_{\widetilde{\mathscr{F}}} (k, \cA [\cG_{k ( \fn / \q)}]),
    \]
    then one has 
    \[
    (V_{\fn / \q, \q} \circ \loc_{k (\fn / \q), \q}) 
    ( x_{\fn, \q}) 
    = (\iota_{\fn, \q} \circ c_{1 - \tau}^{-1} \circ h^{-1} \circ \check{v}_\q) (\kappa' (c)_\fn).
    \] 
\end{lem}

\begin{proof} We recall that the maps 
$\check{v}_\q$ and $\check{\psi}^\fs_\q$ (as defined in Proposition \ref{what we need from Selmer complexes}\,(v)) 
factor through the respective maps $v_\q$ and $\psi^\fs$ from (\ref{v_q def}) and (\ref{fs def}), and hence that it suffices to compare the latter maps to  $V_{\fn / \q, \q}$ and $\Psi^\fs_{\fn, \q}$. Then, since the cofactor map behaves functorially, the commutative diagrams 
\begin{cdiagram}[column sep=small]
H^1 (k_\q, \cA [\cG_{k (\fn)}]) \arrow{r}{(\ref{evaluation def})} & (\cA [\cG_{k (\fn)}])^{\cI_{k (\fn),\q}} / (\Frob_\q - 1) & 
H^1 (k_\q, \cA [\cG_{k (\fn / \q)}]) \arrow{r}{\mathrm{ev}_\q} & ( \cA [\cG_{k (\fn / \q)}] )^{G_{\kappa_\q}} \\ 
H^1 (k_\q, \cA) \arrow{r}{(\ref{evaluation def})} 
\arrow{u}{\iota_{\fn, \q}^\ast} 
& \cA / (\Frob_\q - 1) \arrow{u}{\iota_{\fn, \q}} 
&  
H^1 (k_\q, \cA) \arrow{r}{\partial_\q \circ \res_\q} 
\arrow{u}{\iota_{\fn / \q, \q}^\ast} 
& \cA^{G_{\kappa_\q}}  \arrow{u}{\iota_{\fn / \q, \q}} 
\end{cdiagram}%
reduce both claims to an explicit comparison of the definitions of all involved maps (which we leave to the reader).
\end{proof}

In the sequel, for each $K\in \Omega$ we write 
\[ \delta_{K, \q} \: H^0 (k_\q, \cA [\cG_K]) \to H^1 (k_\q, \a_i \cdot\cT [\cG_K])\]
for the connecting homomorphism that arises  from the tautological short  exact sequence 
\begin{equation}\label{taut ses} 0 \to \a_i\cdot \cT [\cG_K] \to \cT [\cG_K] \to \cA[\cG_K] \to 0.\end{equation}
In particular, for each $\fn \in \cN$ and $\q \in \cQ$, the following result explicitly computes the images under the respective composite maps $\delta_{k (\fn), \q} \circ \Psi^\fs_{\fn, \q}$ and $\delta_{k(\fn), \q} \circ V_{\fn, \q}$ of the elements that arise in the finite-singular relations. 

\begin{prop} \label{fs relations lemma 2}
    For each $c \in \ES^1_{S_0} (\mathfrak{F})$, $\fn \in \cN$ and $\q \in \cQ$, the following claims are valid.
    \begin{romanliste}
    \item For each element  $a$ of the ideal $\langle \a_i, I_\fn\rangle$ of  $\cR [G_\fn]$ generated by $\a_i \cup I_\fn$, there exists a unique element $a \cdot (D'_\fn c_{k (\fn)})$ of $H^1_f (k_\q, \a_i \cT [\cG_{k (\fn)}])$ that satisfies  
    \begin{equation*} \label{characterising property}
    (a \cdot (D_\fn' c_{k (\fn)})) ( \Frob_\q) = a( (D_\fn' c_{k (\fn)}) ( \Frob_\q))\end{equation*}
    in $(\a_i \cT [\cG_{k ( \fn)}])^{\cI_{k (\fn), \q}} / (\Frob_\q - 1)$.
        \item If $\q \not \in V (\fn)$, then $\Eul_\q ( \NN_\q \cdot\Frob_\q)\in \langle \a_i, I_\fn\rangle$, and in $H^1_f (k_\q, \a_i \cT [\cG_{k (\fn)}])$ one has
        \[
        (\delta_{k (\fn), \q} \circ \Psi^\fs_{\fn, \q} \circ \loc_{k (\fn), \q}) ( D'_\fn c_{k (\fn)}) (y) = - \Eul_\q ( \NN \q\cdot \Frob_\q) \cdot (D_\fn' c_{k (\fn)}),
        \]
        where the right-hand side is the element defined in (i) with $a = - \Eul_\q ( \NN_\q\cdot \Frob_\q)$. 
        \item If $\q \in V(\fn)$, then in $H^1_f (k_\q, \a_i \cT [\cG_{k (\fn)}])$ one has
        \[
        (\delta_{k (\fn), \q} \circ V_{\fn, \q}) (x_{\fn, \q}) = (\sigma_\q - 1) D'_\fn c_{k (\fn)},
        \]
        where the right-hand side is the element defined in (i) with $a = \sigma_\q-1 \in I_\fn$.
        \item Assume $\fF$ satisfies Hypothesis \ref{system of fF hyp} and that $\cK$ and $\cT$ satisfy Hypothesis \ref{more new hyps}. Then, for every system $c \in \ES^1_{S_0} (\fF)$, every modulus $\fn \in \cN$, and every prime $\q \in V(\fn)$ one has
        \[
        (\delta_{k (\fn / \q), \q} \circ \Psi^\fs_{\fn / \q, \q} \circ \loc_{k (\fn / \q), \q}) ( D'_{\fn / \q} c_{k(\fn / \q)}) = (\delta_{k (\fn), \q} \circ V_{\fn, \q}) (x_{\fn, \q})
        \]
        in $H^1_f (k_\q, \a_i \cdot\cT [\cG_{k (\fn / \q)}]) \cong H^1_f (k_\q, \a_i \cdot\cT [\cG_{k (\fn)}])$.
    \end{romanliste}
\end{prop}

\begin{proof} To prove (i), we write $k (\fn)_\infty \coloneqq k_\infty \cdot k(\fn)$ with $k_\infty$ the cyclotomic $\Z_p$-extension of $k$. Then, since $k_\infty \cap k (\fn)\subseteq k ( 1)$, there is an injection $G_\fn \hookrightarrow \cG_{k (\fn)_\infty}$ that allows us to view $a$ as an element of $\cR \llbracket \cG_{k (\fn)_\infty} \rrbracket$. Letting $K$ be a finite extension of $k (\fn)$ contained in $k (\fn)_\infty$, the element $a$ therefore acts on $c_K$.\\
Now Proposition \ref{construction complex}\,(iii) gives a canonical isomorphism $C_{\widetilde{\mathscr{F}}} ( \cA [\cG_K]) \cong (\cR/\a_i) \otimes_\cR^\mathbb{L} C_\mathscr{F} ( \cT [\cG_K])$. Upon tensoring the short exact sequence $0 \to \a_i \to \cR \to (\cR/\a_i) \to 0$ with $C_\mathscr{F} ( \cT [\cG_K])$, one thus obtains a canonical exact triangle
\[
\a_i \otimes_\cR^\mathbb{L} C_\mathscr{F} ( \cT [\cG_K]) \to C_\mathscr{F} ( \cT [\cG_K]) \to C_{\widetilde{\mathscr{F}}} ( \cA [\cG_K]) \to
\]
and hence, by Proposition \ref{construction complex}\,(v), an induced exact sequence
\begin{equation} \label{a useful exact sequence}
0 \to H^0 (  \a_i \otimes_\cR^\mathbb{L} C_\mathscr{F} ( \cT [\cG_K])) \to H^1_\mathscr{F} ( k, \cT [\cG_K]) \xrightarrow{\pi^1_{K, {\widetilde{\mathscr{F}}}}} H^1_{\widetilde{\mathscr{F}}} (k, \cA [\cG_K]).
\end{equation}
By the argument of Lemma \ref{bss lem 6.12}, the element $\pi^1_{K, {\widetilde{\mathscr{F}}}} (D'_\fn c_{K})$ is fixed by $G_\fn$ and so is annihilated by $a$. From the last displayed exact sequence, we can thus deduce $a D'_\fn c_K\in H^0 (  \a_i \otimes_\cR^\mathbb{L} C_\mathscr{F} ( \cT [\cG_K]))$. 

To proceed, we use the natural composite morphism 
\[
\xi_K: \a_i \otimes_\cR^\mathbb{L} C_\mathscr{F} ( \cT [\cG_K]) \to \a_i \otimes_\cR^\mathbb{L} \mathrm{R} \Gamma (k_\q, \cT [\cG_K]) \to \mathrm{R} \Gamma (k_\q, \a_i \cT [\cG_K]),
\]
in which the first morphism is induced by the localisation morphism $C_\mathscr{F} ( \cT [\cG_K]) \to \mathrm{R} \Gamma (k_\q,\cT [\cG_K])$ and the second by commutativity of the following diagram of exact triangles 
\begin{cdiagram}[column sep=small]
    \a_i \otimes^\mathbb{L}_R \mathrm{R} \Gamma (k_\q, \cT [\cG_K]) \arrow{r} \arrow[dashed]{d} & \mathrm{R} \Gamma (k_\q, \cT [\cG_K]) \arrow{r} \arrow[equals]{d} & (R / \a_i) \otimes^\mathbb{L}_R \mathrm{R} \Gamma (k_\q, \cT [\cG_K]) \arrow{r} \arrow{d}{\simeq} & \phantom{X} \\
    \mathrm{R} \Gamma (k_\q, \a_i \cT [\cG_K]) \arrow{r} & \mathrm{R} \Gamma (k_\q, \cT [\cG_K]) \arrow{r} & \mathrm{R} \Gamma (k_\q, \cA [\cG_K] ) \arrow{r} & \phantom{X}
\end{cdiagram}%
Here the third vertical map is the canonical isomorphism from Lemma \ref{flach result}\,(vi) and the lower row the triangle associated to the short exact sequence $0 \to \a_i \cT [\cG_K] \to \cT [\cG_K] \to \cA [\cG_K] \to 0$. In particular, there exists a natural commutative diagram
\begin{cdiagram}[column sep=small, row sep=small]
    H^0 (  \a_i \otimes_\cR^\mathbb{L} C_\mathscr{F} ( \cT [\cG_K])) \arrow{r} \arrow{d}{H^0(\xi_K)} &  H^1_\mathscr{F} ( k, \cT [\cG_K]) \arrow{d}{\loc_{k (\fn), \q}} \\
    H^1 (k_\q, \a_i \cT [\cG_K]) \arrow{r} & H^1 (k_\q,  \cT [\cG_K]).
\end{cdiagram}%
Hence, if the element 
\[ a\cdot D'_\fn c_{k(\fn)} \coloneqq H^0(\xi_{k(\fn)})(a D'_\fn c_{k(\fn)})\] 
belongs to $H^1_f (k_\q, \a_i \cT [\cG_{k(\fn)}])$, then it satisfies all of the properties specified in (i). On the other hand, the required containment $a\cdot D'_\fn c_{k(\fn)} \in H^1_f (k_\q, \a_i \cT [\cG_{k(\fn)}])$ is a consequence of the equality $\varprojlim_K H^1_f (k_\q, \a_i \cT [\cG_K]) = \varprojlim_K H^1 (k_\q, \a_i \cT [\cG_K])$, with $K$ ranging over the subfields of a $\Z_p$-extension of $k$ contained in $\cK$ in which no finite place splits completely (and which exists by Hypothesis \ref{more new hyps}\,(i)), that was shown in the proof of Proposition \ref{the congruence condition}\,(ii). Indeed, given this, it is enough to note that, since $c$ is an Euler system and the second arrow in (\ref{a useful exact sequence}) is injective, the family $(a\cdot D'_\fn c_{K (\fn)})_K $ belongs to $\varprojlim_K H^1 (k_\q, \a_i \cT [\cG_K])$ because $(c_{K ( \fn)})_K$ is norm compatible. This proves (i).  \\
We next recall that the definition of $\cQ$ implies $\q$ splits completely in the fields $k (1)$ and $k ( \mu_{|\cR_i|})$, and also that $\Frob_\q$ restricts to $\tau$ on the field $k (\cA)$. This implies  $\NN \q - 1 \in |\cR_i| \Z \subseteq \a_i$ and hence, since $\Frob_\q - 1 \in I_\fn$, that 
\[
\Eul_\q ( \NN_\q \cdot\Frob_\q) \equiv \Eul_\q (1) \equiv {\det}_{\cR_i} ( 1 - \tau^{-1} \mid \cA^\ast (1)) \equiv 0 \mod \langle I_\fn, \a_i \rangle. 
\]
This proves the first assertion of (ii) and, in particular, implies (i) can be applied with $a$ taken to be $-\Eul_\q ( \NN_\q \cdot\Frob_\q)$.
 To prove the rest of (ii), we now fix $y \in H^1_f (k_\q, \cT [\cG_{k (\fn)}])$ and note  
    \[ \Psi^\fs_{\fn, \q} (y) = c_{1 - \Frob_\q} ( \overline{y} (\Frob_\q)),\] 
    with $\overline{y} \in H^1_f (k_\q, \cA [\cG_{k (\n)}])$ denoting the reduction of $y$ modulo $\a_i$. In addition, one has  
    \begin{align*}
    (\Frob_\q - 1) c_{1 - \Frob_\q} & = - (1 - \Frob_\q) c_{1 - \Frob_\q} \\
    & = - {\det}_{\bLambda [\cG_{k (\fn)}]} ( 1 - \Frob_\q \mid \cA[\cG_{k (\fn)}]) \\
    & = - \Eul_\q (\NN\q \cdot \Frob_\q) 
    \quad \in \bLambda [\cG_{k (\fn)}].
    \end{align*}
   Now, for any $\alpha \in H^0 (k_\q, \cA [\cG_{k (\fn)}])$ the definition of the map $\delta_{k (\fn), \q}$ implies $\delta_{k (\fn), \q} (\alpha) ( g)$ is equal to the class of the cocyle $G_{k_\q} \to \a_i\cdot \cT [\cG_{k (\fn)}], g \mapsto ( g - 1) \tilde \alpha$, where $\tilde \alpha \in \cT [\cG_{k (\fn)}]$ is a lift of $z$. In particular, since $c_{1 - \Frob_\q} (y (\Frob_\q))$ is a lift of $c_{1 - \Frob_\q} (\overline{y} ( \Frob_\q))$ and belongs to $(\cT [\cG_{k (\fn)}])^{\cI_{k (\fn), \q}}$ as $y\in H^1_f (k_\q, \cT [\cG_{k (\fn)}])$, one has $(\delta_{k (\fn), \q} \circ \Psi^\fs_{\fn, \q}) (y) (\Frob_\q) \in H^1_f (k_\q, \a_i\cdot \cT [\cG_{k (\fn)}])$. Moreover, we may compute that
    \begin{align*}
    (\delta_{k (\fn), \q} \circ \Psi^\fs_{\fn, \q}) (y) (\Frob_\q) & = (\Frob_\q - 1) \cdot c_{\Frob_\q - 1} \cdot y (\Frob_\q) \\
    & =  - \Eul_\q (\NN\q \cdot \Frob_\q)  \cdot y (\Frob_\q) \\
    & = ( - \Eul_\q (\NN\q \cdot \Frob_\q)  \cdot y) (\Frob_\q).
    \end{align*}
    Applying this with $y = \loc_{k (\fn), \q} ( D'_\fn c_{k (\fn)})$ then implies (ii) since an element of $H^1_f (k_\q, \cA [G_{k(\fn)}])$ is uniquely determined by its value on $\Frob_\q$.\\
    To prove (iii), we take $y \in H^1_f (k_\q, \cA [\cG_{k(\fn / \q)}])$ and similarly compute that
    \begin{align*}
        (\delta_{k (\fn), \q} \circ V_{\fn, \q}) (y) ( \Frob_\q) & = (\Frob_\q - 1) \cdot \widetilde{ y (\sigma_\q)},
    \end{align*}
    where $\widetilde{ y (\sigma_\q)} \in \cT [\cG_{k (\fn / \q)}]$ is a lift of $y (\sigma_\q)$. 
    Note  $(\delta_{k (\fn), \q} \circ V_{\fn, \q}) (y)$ belongs to $H^1_f (k_\q, \a_i \cT [\cG_{\fn}])$. Indeed, this is true since we may view $(V_{\fn, \q}) (y)$ as an element of $H^0_f (k_\q, \cA [\cG_{k (\fn) / \q}]$ and because $\delta_{\fn / \q, \q}$ maps into $H^1_f (k_\q, \a_i \cT [\cG_{\fn / \q}])$ as the representation $\cA [\cG_{\fn / \q}]$ is unramified at $\q$. This shows that $(\delta_{k (\fn), \q} \circ V_{\fn, \q}) (y) $ is uniquely determined by its value on $\Frob_\q$. 

    We now apply this with $y = x_{\fn, \q}$ so that, denoting by $\widetilde{\sigma_\q}$ a lift of $\sigma_\q$ to $G_k$, we may take $\widetilde{y ( \sigma_\q)}$ to be $(D'_\fn c_{k ( \fn)}) ( \widetilde{ \sigma_\q})$.
    Given the explicit action of $G_\q$ on $H^1 (k_\q, \a_i \cdot\cT [G_\fn])$, repeated use of the cocycle property then implies that
    \begin{align*}
        ((\sigma_\q  - 1) D'_\fn c_{k ( \fn)}) (\Frob_\q)
        & = \widetilde{\sigma_\q} (D'_\fn c_{k ( \fn)}) ( \widetilde{\sigma_\q}^{-1} \Frob_\q \widetilde{\sigma_\q}) - (D'_\fn c_{k ( \fn)}) (\Frob_\q) \\
   & = \big( (D'_\fn c_{k ( \fn)}) (\Frob_\q \widetilde{\sigma_\q}) + \widetilde{\sigma_\q} (D'_\fn c_{k ( \fn)}) (\widetilde{\sigma_\q}^{-1}) \big) - (D'_\fn c_{k ( \fn)}) (\Frob_\q) \\
    & = \big( \Frob_\q (D'_\fn c_{k ( \fn)}) (\widetilde{\sigma_\q}) + (D'_\fn c_{k ( \fn)}) (\Frob_\q) \big) - (D'_\fn c_{k ( \fn)}) (\widetilde{\sigma_\q}) \\
    & \qquad \qquad - (D'_\fn c_{k ( \fn)}) (\Frob_\q) \\
    & = (\Frob_\q - 1) (D'_\fn c_{k ( \fn)}) (\widetilde{\sigma_\q}).
    \end{align*}
    We have therefore proved that 
    \[ (\delta_{k (\fn), \q} \circ V_{\fn, \q}) (x_{\fn, \q}) ( \Frob_\q) = ((\sigma_\q  - 1) D'_\fn c_{k ( \fn)}) (\Frob_\q) (\Frob_\q),\] 
    and this implies (iii) since $(\delta_{k (\fn), \q} \circ V_{\fn, \q}) (x_{\fn, \q})$ and  $(\sigma_\q  - 1) D'_\fn c_{k ( \fn)}$ belong to $H^1_f (k_\q, \a_i \cT [\cG_{k (\fn)}])$.\\
    To prove (iv), we first use (i) to regard $|G_\q|D_\fn' c_{k (\fn)}$ as an element of $H^1_f (k_\q, \a_i \cT [\cG_{k (\fn)}])$ and  $(\Eul_\q ( \Frob_\q^{-1}) - \Eul_\q (\NN \q \cdot \Frob_\q^{-1})) D'_{\fn / \q} c_{k(\fn / \q)}$ as an element of  $H^1_f (k_\q, \a_i \cT [\cG_{k (\fn / \q)}])$. From 
  Proposition \ref{the congruence condition}\,(iii), we then obtain an equality 
    \begin{equation} \label{adapted congruence} 
    |G_\q| D'_\fn c_{k(\fn)} =  (\Eul_\q ( \Frob_\q^{-1}) - \Eul_\q (\NN \q \cdot \Frob_\q^{-1})) D'_{\fn / \q} c_{k(\fn / \q)} 
    \end{equation}
in $H^1_f (k_\q, \a_i \cT [\cG_{k ( \fn})]) \cong H^1_f (k_\q, \a_i \cT [\cG_{k ( \fn / \q})])$.\\
We also note that the identity $(\sigma_\q - 1) D_\q =  |G_\q| - \NN_{G_\q}$ implies an equality
\begin{equation} \label{intermediate calculation}
(\sigma_\q - 1) D'_{\fn} c_{k (\fn)}  = 
( | G_\q| - \NN_{G_\q}) D'_{\fn / \q} c_{k (\fn)} = 
|G_\q| D'_{\fn / \q} c_{k (\fn)} - \Eul_\q (\Frob_\q^{-1}) D'_{\fn / \q} c_{k (\fn / \q)} 
\end{equation}
in $H^1_{\widetilde{\mathscr{F}}} (k, \cT [\cG_{k ( \fn)}])$. Since (i) can be applied to both summands on the right, these equalities also hold in $H^1_f (k_\q, \a_i \cT [\cG_{k(\fn)}])$. 
We may then compute that 
    \begin{align*}
     &\hskip0.2truein (\delta_{k (\fn), \q} \circ V_{\fn, \q}) ( D'_\fn c_{k (\fn)})\\
    &  = (\sigma_\q - 1) D'_{\fn} c_{k (\fn)} \\
        & =
        ( | G_\q| - \NN_{G_\q}) D'_{\fn / \q} c_{k (\fn)} \\ 
        & =(\Eul_\q ( \Frob_\q^{-1}) - \Eul_\q (\NN \q\cdot \Frob_\q^{-1})) D'_{\fn / \q} c_{k (\fn / \q)}  - \Eul_\q (\Frob_\q^{-1}) D'_{\fn / \q} c_{k (\fn / \q)}\\ 
        & =
         - \Eul_\q (\NN \q \cdot\Frob_\q^{-1}) D'_{\fn / \q} c_{k (\fn / \q)} \\
        & = (\delta_{k (\fn / \q), \q} \circ \Psi^\mathrm{fs}_{\fn, \q}) ( D'_{\fn / \q} c_{k (\fn / \q)}),
    \end{align*}
where the first equality is by (iii), the second by (\ref{intermediate calculation}), the third by (\ref{adapted congruence}), and the last by  (ii). This proves the equality in (iv).
\end{proof}

The above computations can now be combined to prove the `finite-singular relations' that occur in Theorem \ref{kolyvagin derivative thm} in the case $t=1$.

\begin{prop} \label{finite-singular relations prop}
    Assume $\fF$ satisfies Hypothesis \ref{system of fF hyp} and that $\cK$ and $\cT$ satisfy Hypothesis \ref{more new hyps}. Then, for every $c \in \ES^1_{S_0} (\fF)$, $\fn \in \cN$, and $\q \in  V(\fn)$ one has
    \[
    \check{v}_\q (\kappa'(c)_\fn) 
    = \check{\psi}_\q^\mathrm{fs}  (\kappa' (c)_{\fn / \q}) \in \bLambda.\]
\end{prop}

\begin{proof}
Since the maps $h$ and $c_{1-\tau}$ are bijective and $\iota_{\fn, \q}$ is injective, the claimed equality is valid if, in $H^0_f (k_\q, \cA [\cG_{k (\fn / \q)}]) \cong H^0_f (k_\q, \cA [\cG_{k (\fn)}])$, one has 
\[
(\iota_{\fn / \q, \q} \circ h^{-1} \circ \check{\psi}^\fs_\q) (\kappa' (c)_{\fn / \q}) = (\iota_{\fn, \q} \circ c_{1 - \tau}^{-1} \circ h^{-1} \circ \check{v}_\q) (\kappa' (c)_\fn).
\]
Since Hypothesis \ref{more new hyps}\,(ii) implies the connecting maps $\delta_{k (\fn / \q), \q}$ and $\delta_{k (\fn), \q}$ are injective, it thus suffices to show that, in $H^1_f (k_\q, \a_i\cdot\cT [\cG_{k (\fn / \q)}]) \cong H^1_f (k_\q, \a_i \cdot\cT [\cG_{k (\fn)}])$, one has 
\[
(\delta_{k (\fn / \q), \q} \circ \iota_{\fn / \q, \q} \circ h^{-1} \circ \check{\psi}^\fs_\q) (\kappa' (c)_{\fn / \q}) = (\delta_{k (\fn), \q} \circ \iota_{\fn, \q} \circ c_{1 - \tau}^{-1} \circ h^{-1} \circ \check{v}_\q) (\kappa' (c)_\fn).
\]
It is then enough to note that the latter equality can be directly verified by combining Lemma \ref{fs relations lemma 1} with Proposition \ref{fs relations lemma 2}\,(iv).
\end{proof}

At this point, to complete the proof of Theorem \ref{kolyvagin derivative thm} it only remains to show that the explicit linear combination of elements $\kappa' (c)_\mathfrak{d} c_{k (\mathfrak{d})}$ that occurs in Definition \ref{kolyvagin derivative def} belongs to the  group $H^1_{\mathscr{F} (\fn)} (k, \cA)$, and for this we adapt an argument of Mazur and Rubin in \cite[App.\@ A]{MazurRubin04}. This means that the key technical result we have to prove  is  the following.

\begin{lem}[{\cite[Th.~A.4]{MazurRubin04}}] \label{formula for finite-singular projection}
    Assume $\fF$ satisfies Hypothesis \ref{system of fF hyp} and that $\cK$ and $\cT$ satisfy Hypothesis \ref{more new hyps}. Then, for every $c \in \ES^1_{S_0} (\fF)$, $\fn \in \cN$, and $\q \in  V(\fn)$ one has
    \[
    \check{\psi}_\q^\fs ( \kappa' (c)_\fn) = - \sum_{\tau \in \mathfrak{S}_{\q} (\fn)} \mathrm{sgn} (\tau) \cdot 
    \big( \prod_{ \mathfrak{l} \in V(\fn / \mathfrak{d}_\tau)} x^{(\mathfrak{l})}_{\tau (\mathfrak{l})} \big) \cdot 
    \check{\psi}_\q^\fs ( \kappa' (c)_{\mathfrak{d}_\tau} ),
    \]
    where $\mathfrak{S}_\q (\fn)$ denotes the subset of $\mathfrak{S}(\fn)$ comprising cycles $\tau$ for which $\tau(\q)\not=\q$ and $x^{(\mathfrak{l})}_{\tau (\mathfrak{l})} $ are the elements defined in (\ref{x fl definition}). 
\end{lem}

\begin{proof}
Let $\mathfrak{S}_1 (\fn)$ denote the collection of all cycles in $\mathrm{Per} ( V (\fn))$. Then for each $\tau \in \mathfrak{S}_1 (\fn)$ one has $\q \notin V(\mathfrak{d}_\tau)$ and so Proposition \ref{finite-singular relations prop} shows that $\check{\psi}_\q^\fs ( \kappa' (c)_{\mathfrak{d}_\tau}) = \check{v}_\q ( \kappa' (c)_{\q \mathfrak{d}_\tau})$.
Since the maps $h$ and $c_{1 - \tau}$ are isomorphisms, it is therefore enough to compute $(\iota_{\fn, \q} \circ h^{-1} \circ \check{\psi}^\fs_\q) (\kappa_\fn)$ in terms of $(\iota_{\fn, \q} \circ c_{1 - \tau}^{-1} \circ h^{-1} \circ \check{v}_\q) (\kappa_{\mathfrak{d}_\tau})$. 
Since Hypothesis \ref{more new hyps}\,(ii) implies that the connecting homomorphism $\delta_{k (\fn), \q}$ is injective, we may carry out this computation after applying $\delta_{k (\fn), \q}$. Using Lemma \ref{fs relations lemma 1} and Proposition \ref{fs relations lemma 2}\,(ii) and (iii) we are therefore reduced to proving that, in $H^1_f (k_\q, \a_i \cT [\cG_{k (\fn)}])$, one has  
    \begin{equation} \label{something else that needs proving}
     \Eul_\q (\Frob_\q) \cdot (D'_\fn c_{k (\fn)}) 
    = \textstyle \sum_{\tau \in \mathfrak{S}_\q (\fn)} \mathrm{sgn} (\tau)
    \cdot \big( {{\prod}}_{ \mathfrak{l} \in V(\fn / \mathfrak{d}_\tau)} x^{(\mathfrak{l})}_{\tau (\mathfrak{l})} \big) 
    \cdot (\sigma_\q - 1) \cdot (D'_{\q \mathfrak{d}_\tau} c_{k (\q \mathfrak{d}_\tau)}).
    \end{equation}
To do this, we fix $\m \in \cN$ with $\m \mid \fn$, and $\p \in V (\m)$.
If $u$ is an element of $I_\m + \a_i \cR [\cG_{k ( \m)}]$, then $u$ annihilates $(\sigma_\p - 1) (D'_\m c_{k (\m)})$. Indeed, this follows from the fact that $u$ annihilates $\kappa' (c)_\m$ (by Lemma \ref{bss lem 6.12}) and Lemma \ref{fs relations lemma 1} and Proposition \ref{fs relations lemma 2}\,(iii). 
For every $\mathfrak{l} \in V (\m)$ we have
\[
\Eul_\mathfrak{l} ( \Frob_\mathfrak{l}^{-1}) \equiv \textstyle \sum_{\p \in V (\m /\mathfrak{l})} x_\mathfrak{l}^{(\p)} (\sigma_\p - 1) \mod (I_\m + \a_i \cR [\cG_{k ( \m)}])
\]
by definition of the elements $x_\mathfrak{l}^{(\mathfrak{p})}$ (cf.\@ (\ref{x fl definition})),
and so it follows that
\begin{align} \nonumber
           \Eul_\mathfrak{l} (\Frob_\mathfrak{l}^{-1}) \cdot (D'_\m c_{k (\m)}) & = \textstyle \sum_{\p \in V(\m / \mathfrak{l})} x^{(\p)}_\mathfrak{l} (\sigma_\p - 1) D'_\m c_{k (\m)} \\
          & = - \textstyle \sum_{\p \in V(\m / \mathfrak{l})}  x^{(\p)}_\mathfrak{l} \Eul_\p ( \Frob_\p^{-1}) \cdot D'_{\m / \p} c_{k (\m / \p)},
          \label{tool for expansion}
    \end{align}
where the second equality follows from Proposition \ref{fs relations lemma 2}\, (ii), (iii) and (iv). \\
We can now use (\ref{tool for expansion}) (with $\m  = \fn$ and $\mathfrak{l} = \q$) to express $- \Eul_\q (\Frob_\q) \cdot (D'_\fn c_{k (\fn)})$ as a sum of terms of the form $x^{(\p_1)}_\q \Eul_{\p_1} ( \Frob_{\p_1}^{-1}) \cdot D'_{\m / \p_1} c_{k (\fn / \p_1)}$. 
If $\p_1 \neq \q$, then we can apply (\ref{tool for expansion}) (with $\m = \m / \p$ and $\mathfrak{l} = \p$) to the corresponding summand in order to write it as a sum of terms of the form $x^{(\p_2)}_{\p_1} \Eul_{\p_2} ( \Frob_{\p_2}^{-1}) \cdot D'_{\m / (\p_1 \p_2)} c_{k (\fn / (\p_1 \p_2))}$. Continuing this process until $\p_n = \q$ for some $n$ produces a cycle $ \tau \coloneqq (\q \, \p_{n - 1} \, \dots \, \p_1) \in \mathfrak{S}_\q (\fn)$ such that $\mathfrak{d}_\tau \coloneqq \fn \cdot (\prod_{j = 1}^n \p_j)^{-1}$ and $\mathrm{sgn} (\tau) = (-1)^{n - 1}$, and the resulting summand is
\[
- \mathrm{sgn} (\tau) \cdot \big ({\prod}_{\mathfrak{l} \in V (\n / \mathfrak{d}_\tau)} x^{(\mathfrak{l})}_{\tau (\mathfrak{l})} \big ) \cdot (\sigma_\q - 1) (D'_{\mathfrak{d}_\tau} c_{k (\mathfrak{d}_\tau)}).
\]
This proves (\ref{something else that needs proving}), thereby concluding the proof of the lemma.
\end{proof}

The next result now finally completes the proof of Theorem \ref{kolyvagin derivative thm}.  

\begin{prop} \label{finite-singular projections prop}
    Assume $\fF$ satisfies Hypothesis \ref{system of fF hyp} and that $\cK$ and $\cT$ satisfy Hypothesis \ref{more new hyps}. Then, for every system $c \in \ES^1_{S_0} (\fF)$ and modulus $\fn \in \cN$, one has $\,\kappa (c)_\fn \in H^1_{\tilde{\mathscr{F}} (\fn)} (k,\cA)$.
\end{prop}

\begin{proof}
If $\m$ divides $\fn$, then $\kappa' (c)_\m$ belongs to $H^1_{\tilde{\mathscr{F}}^\m} (k, \cA) \subseteq H^1_{\tilde{\mathscr{F}}^\n} (k, \cA)$ and so the same is true for $\kappa (c)_\fn$. To prove that $\kappa (c)_\fn$ belongs to $H^1_{\tilde{\mathscr{F}} (\fn)} (k, \cA)$ it is then enough, by the exact sequence (\ref{global duality sequence2}), to prove that  $\check{\psi}_\q^\fs (\kappa (c)_\fn)=0$ for all $\q \in V(\fn)$. 
To verify this, we write 
\[ U_\q (\fn) \coloneqq \{ \tau \in \mathfrak{S} (\fn) \mid \tau (\q) = \q \}\] 
for the stabiliser of $\q$ in $\mathfrak{S} (\fn)$ and note that every  $\tau \in \mathfrak{S} (\fn) \setminus U_\q (\fn)$ can be written as $\tau = \sigma \circ\rho$ with $\sigma \in U_\q (\fn)$ and $\rho \in \mathfrak{S}_\q (\fn / \mathfrak{d}_\sigma)$. Given this, we can rearrange the terms in the definition of $\kappa (c)_\fn$ as
\begin{align*}
    \kappa (c)_\fn & = {\sum}_{\tau \in \mathfrak{S} (\fn)} \mathrm{sgn} (\tau)  
    \cdot \big( {{\prod}}_{ \q \in V(\fn / \mathfrak{d}_\tau)} x^{(\q)}_{\tau (\q)} \big) \cdot \kappa'(c)_{\mathfrak{d}_\tau} \\ 
    & = \Big ( {\sum}_{\tau \in U_\q (\fn)} \mathrm{sgn} (\tau)  
    \cdot \big( {{\prod}}_{ \q \in V(\fn / \mathfrak{d}_\tau)} x^{(\q)}_{\tau (\q)} \big) \cdot \kappa'(c)_{\mathfrak{d}_\tau} \Big) \\
    & \qquad \quad + \Big( {\sum}_{\sigma \in U_\q (\fn), \rho \in \mathfrak{S}_\q (\fn / \mathfrak{d}_{\sigma})} \mathrm{sgn} (\sigma\circ \rho)  
    \cdot \big( {{\prod}}_{ \q \in V(\fn / \mathfrak{d}_{\sigma \rho})} x^{(\q)}_{(\sigma \rho) (\q)} \big) \cdot \kappa'(c)_{\mathfrak{d}_{\sigma \rho}} \Big)\\
    & = {\sum}_{\tau \in U_\q (\fn)} \mathrm{sgn} (\tau) \cdot \big( {{\prod}}_{ \q \in V(\fn / \mathfrak{d}_\tau)} x^{(\q)}_{\tau (\q)} \big) \cdot \lambda_\tau
\end{align*}
where we set
\[
\lambda_\tau \coloneqq \kappa'(c)_{\mathfrak{d}_\tau} + {\sum}_{\rho \in \mathfrak{S}_\q (\fn)} \mathrm{sgn} (\rho) \cdot \big( {{\prod}}_{ \q \in V( \mathfrak{d}_\tau / \mathfrak{d}_\rho)} x^{(\q)}_{\rho (\q)} \big) \cdot \kappa' (c)_{\mathfrak{d}_\rho}.
\]
In particular, since  Lemma \ref{formula for finite-singular projection} implies that  $\check{\psi}_\q^\fs (\lambda_\tau) = 0$ for every $\tau \in U_\q (\fn)$, this calculation implies  $\psi^\fs_\q ( \kappa (c)_\fn)=0$, as required.
\end{proof}

\section{Kolyvagin systems II: controlling  values at $1$}\label{koly 2 section}

Throughout this section we fix  $i\in \N$ and continue to use the notation specified in (\ref{conv notation}).

\subsection{Tate--Shafarevich modules and relative core vertices}

The notion of  `core vertex' plays a key role in the theory of Kolyvagin systems developed by Mazur and Rubin in \cite{MazurRubin04}. In this subsection, we use the Cebotarev density theorem to prove the existence (under Hypotheses~\ref{new strategy hyps}) of an appropriate analogue of this notion for  our theory.  

\subsubsection{Tate--Shafarevich modules and cohomological invariants}
\label{Selmer and sha section}

For each $j \in \N$ with  $j \geq i$, and each Mazur--Rubin structure $\cF'$ on $\cA$, we define $\bLambda$-modules 
\[
\sha_{\cF', j} (\cA) = \sha_{\cF'} (k, \cA, \cQ_j) \coloneqq \ker \Bigl( H^1_{\cF'} (k, \cA) \to \prod_{\q \in \cQ_j} H^1 (k_\q, \cA) \Bigr )\]
and
\[ 
\fX_{\cF', j} (\cA)  \coloneqq H^1_{\cF'} (k,\cA) \cap H^1 (\cG_{k_j (\cT_j)}, \cA)
 = \ker \bigl( H^1_{\cF'} (k, \cA) \to H^1 (k_j (\cT_j), \cA) \bigr). 
\]

In a similar way, for each $j \in \N$ with  $j \geq i$, and each  Nekov\'a\v{r} structure $\mathscr{F}'$ on $\cA$, the exact sequence (\ref{nekovar mr compare}) implies the existence of a localisation map $H^1_{\mathscr{F}'} (k, \cA) \to H^1 (k_\q, \cA)$  for each $\q \in \cQ_j$, and so  we can define a $\bLambda$-module  
\[
\sha_{\mathscr{F}', j} (\cA)  = \sha_{\mathscr{F}'} (k, \cA, \cQ_j) \coloneqq \ker \bigl( H^1_{\mathscr{F}'} (k, \cA) \to \prod_{\q \in \cQ_j} H^1 (k_\q, \cA) \bigr).
\]

In this way, we obtain increasing filtrations of Tate-Shafarevich modules 
\[\sha_{\cF', i} (\cA) \subseteq \sha_{\cF', i+ 1} (\cA) \subseteq \dots \subseteq H^1_{\cF'} (k, \cA)\]
and 
\[ \sha_{\mathscr{F}', i} (\cA) \subseteq \sha_{\mathscr{F}', i + 1} (\cA) \subseteq \dots \subseteq H^1_{\mathscr{F}'} (k, \cA).
\]
In the following result concerning these modules we write $\langle \tau \rangle$ for the subgroup of $\gal{k_j(\cT_j)}{k_j}$ generated by $\tau$.

\begin{lem} \label{general neukirch interpretation lemma} Assume Hypothesis \ref{new strategy hyps}\,(i) and (ii). Then, for each $j \in \N$ with $j \ge i$,  the following claims are valid.

\begin{itemize} 
\item[(i)] $\sha_{\tilde\cF, j} (\cA)\subseteq \fX_{\tilde\cF, j} (\cA)$, with equality if $H^1 ( \langle \tau \rangle, \cA)$ vanishes.
\item[(ii)] If $x \in \fX_{\tilde\cF, j} (\cA)$ is such that $\loc_v (x) = 0$ for some $v \in \cQ_j$, then $x\in \sha_{\tilde\cF, j} (\cA)$. 
\item[(iii)] There exists a canonical isomorphism of $\bLambda$-modules
\[ H^1_{\tilde{\mathscr{F}}} (k, \cA) / \sha_{\tilde{\mathscr{F}}, j} (\cA) \xrightarrow{\simeq} H^1_{\tilde{\cF}} (k,\cA) / \sha_{\tilde{\cF}, j} (\cA).\]
\item[(iv)] If Hypothesis \ref{new strategy hyps}\,(iii) is also valid, then $\sha_{\overline{F}^\ast, j} (\overline{B})$ vanishes.
\end{itemize}
\end{lem}

\begin{proof} To prove the inclusion in (i),  we need to show that any element $\xi$ of $\sha_{\tilde\cF, j} (\cA)$ is trivial upon restriction to $G_{k_j (\cT_j)}$. Since $G_{k_j (\cT_j)}$ acts trivially on $\cA$ (since $j \geq i$), the restriction $f \coloneqq \res (\xi)$ of $\xi$ to $G_{k_j (\cT_j)}$ is a $\cG_{k_j (\cT_j)}$-equivariant homomorphism $f \: G_{k_j (\cT_j)} \to \cA$. In particular, the fixed field of the kernel of $f$ is  a finite extension $K$ of $k_j (\cT_j)$ that is Galois over $k$.\\
Now let $\sigma$ be an element in $\gal{K}{k_j (\cT_j)}$. Then, since $\tau \sigma$ agrees with $\tau$ when restricted to $k_j (\cT_j)$, Cebotarev's density theorem allows us to choose places $v_1$ and $v_2$ in $\cQ_j$ such that the projections to $\cG_K$ of $\Frob_{v_1}$ and $\Frob_{v_2}$ are respectively equal to $\tau \sigma$ and $\tau$. By assumption, $\xi$ is trivial when restricted to the decomposition groups of $v_1$ and $v_2$, so there are elements $a_1$ and $a_2$ of $\cA$ such that, for $n \in \{1,2\}$, one has $\xi ( \rho) = (\rho - 1) \cdot a_n$ if $\rho$ belongs to the decomposition group of $v_n$. It follows that
\begin{align*}
\xi ( \Frob_{v_2}^{-1}\cdot \Frob_{v_1}) & = \Frob_{v_2}^{-1} \cdot \xi ( \Frob_{v_1}) + \xi ( \Frob_{v_2}^{-1})\\
& = \Frob_{v_2}^{-1}\cdot \Frob_{v_1} a_1 - \Frob_{v_2}^{-1} a_1 + \Frob_{v_2}^{-1} a_2 - a_2\\
& = a_1 - \tau^{-1} a_1 + \tau^{-1} a_2 - a_2 \\
& = (\tau^{-1} - 1) \cdot (a_2 - a_1).
\end{align*}
Now, since $\Frob_{v_2}^{-1}\cdot \Frob_{v_1}\in G_{k_j (\cT_j)}$, one has 
\[ \xi ( \Frob_{v_2}^{-1}\cdot \Frob_{v_1}) = f ( \Frob_{v_2}^{-1} \cdot\Frob_{v_1}) = f( \tau^{-1} \tau \sigma) = f (\sigma).\]
The above calculation therefore implies that $\im(f)$ is contained in $(\tau  - 1)\cA$. Hypothesis \ref{new strategy hyps}\,(i) and (ii) then combines with Lemma \ref{injectivity lemma} below (applied with $L = k_j (\cT_j)$) to imply that $f$ is trivial, as required to prove the inclusion in (i). \\
Next we note that, if $H^1 (\langle \tau \rangle,\cA)$ vanishes, then the inflation-restriction sequence shows every element $x$ of $\fX_{\cF, j} (\cA)$ is the inflation of an element of $H^1 ( k_j (\cT_j)^{\langle \tau \rangle} / k, \cA)$. Since (the explicit definition of $\cQ_j$ ensures) every place $v$ in $\cQ_j$ splits completely in $k_j (\cT_j)^{\langle \tau \rangle}$, it follows that $\loc_v (x)$ must be trivial for every such $v$. This implies $x\in \sha_{\cF, j} (\cA)$, as required to complete the proof of (i). \\
Turning to (ii) we note that each $x$ in $\fX_{\cF, j} (\cA)$ is trivial when restricted to $G_{k_j (\cT_j)}$ and so is unramified at every place that is unramified in $k_j (\cT_j) / k$. In particular, $x$ is unramified at every place $v$ in $\cQ_j$ and so $\loc_v (x)\in H^1_f (k_v, \cA)$. Since, for each such $v$, there is an isomorphism
\[
H^1_f (k_v, \cA) \stackrel{\simeq}{\longrightarrow} \cA / (\tau - 1)\cA, \quad [y] \mapsto y (\Frob_v)\,\,\,\, (\mathrm{mod}\,\,\, (\tau - 1)\cA ),
\]
it follows that $\res_v (x)$ vanishes if and only if the element 
\[
x (\Frob_v) = x (\tau \tau^{-1} \Frob_v) = x (\tau) + \tau^{-1} x ( \tau^{-1} \cdot\Frob_v) = x (\tau)
\]
belongs to $(\tau - 1)\cA$. Since the latter condition does not depend on $v$, we see that $x$ belongs to $\sha_{\cF, j} (\cA)$ if $\res_v (x)$ vanishes for any given $v \in \cQ_j$, as claimed in (ii). \\
To prove (iii), we note that, for $\q\in S(\tilde{\mathscr{F}})$, the natural localisation map $H^1_{\tilde{\mathscr{F}}} (k, \cA) \to H^1 (k_\q, \cA)$ factors through the map $H^1_{\tilde{\mathscr{F}}} (k, \cA) \to H^1 (\cO_{k, S(\tilde{\mathscr{F}})}, \cA)$ induced by the exact triangle (\ref{mapping fibre2}). Claim (iii) therefore follows from the fact  the exact sequence (\ref{nekovar mr compare}) implies that the image of  the latter map is equal to  $H^1_{\tilde{\cF}} (k, \cA)$. 

To prove (iv), we note  Hypothesis \ref{new strategy hyps}\,(iii) implies that $H^1 (\cG_{k_j (\cT_j)}, \overline{B})$, and hence also its subgroup $\fX_{\overline{F}^\ast(1), j} (\overline{B})$, vanishes 
(cf.\@ Remark \ref{vanishing remark}\,(i) and (ii)). In this case, therefore, the vanishing of $\sha_{\overline{F}^\ast(1), j} (\overline{B})$  follows from the same argument that proves the inclusion in (i), after replacing $\tilde\cF$ and $\cA$ by $\overline{F}^\ast$ and $\overline{B}$.  
\end{proof}

\begin{lem} \label{injectivity lemma}
Assume Hypotheses \ref{new strategy hyps}\,(i) and (ii). Then, for any finite Galois extension $L$ of $k$ containing $k_i (\cA)$, the following natural map is injective
    \[
H^1 (L, \cA)^{\cG_L} = \Hom_{\cG_L}(G_L, \cA)  \to  \Hom (G_{L}, \cA/ (\tau - 1)\cA ). 
    \]
\end{lem}

\begin{proof}
    Let $x \: G_{L} \to \cA$ be a non-trivial homomorphism in $\Hom_{\cG_L}(G_{L}, \cA)$ that belongs to the kernel of the above map. That is, $\im (x)$ is contained in $(\tau - 1) \cA$. Then, since $x$ is non-trivial, Lemma \ref{ryotaro's reduction trick} implies the existence of a non-zero element $\lambda$ of $\Lambda$ such that $\lambda \cdot x$ is a non-zero element of the $\cM_i$-torsion submodule of $\Hom_{\cG_L} (G_{L}, \cA)$. In particular $\im (\lambda x) = \lambda \cdot \im (x)$ is contained in $\cA [\cM_i]$.  
    In addition, following Remark \ref{switch}, there exists an isomorphism $\cA \otimes_{\bLambda} \mathbb{K} \cong \cA [\cM_i]$ and so Hypothesis \ref{new strategy hyps}\,(i) implies $\cA[\cM_i]$ has no proper non-trivial $G_k$-stable submodules. If $\lambda \cdot \im (x)$ is non-zero, it must therefore span $\cA [\cM_i]$ over $\bLambda$. Since $\im (x)$, and hence $\lambda \cdot \im (x)$, is contained in the $\bLambda$-module $(\tau - 1)\cA$, it follows that $(\tau - 1)\cA$ contains $\cA[\cM_i]$. We therefore obtain a surjective map of $\Lambda$-modules of the form 
    \[ \theta \: \cA / (\cA [\cM_i]) \twoheadrightarrow \cA/ (1 - \tau) \cA.\] 
If $\cA = \cA\otimes_{\bLambda}\mathbb{K}$, then the domain of $\theta$ vanishes, whilst its codomain is isomorphic to $\mathbb{K}$ and this is a contradiction. On the other hand, if $i \geq 1$, then the codomain  of $\theta$ is isomorphic to $\bLambda$, whilst Lemma \ref{ryotaro's reduction trick} implies that every element of its domain is annihilated by a non-zero element of $\bLambda$. The surjectivity of $\theta$ therefore implies that $1 \in \bLambda$ is annihilated by a non-zero element of $\bLambda$ and this is a contradiction. In all cases, therefore, the map $x$ must be trivial, as required. 
\end{proof}

For $j \in \N$ with $j \ge i$, and any pairwise coprime moduli $\a$, $\fb$ and $\n$ in $\cN_j$, one has 
\[ \sha_{\overline{F}, j} (\overline{A})\subseteq H^1_{\overline{F}^\fb_\a (\fn)} (k, \overline{A})\,\,\text{ and }\,\, \sha_{\overline{F}^\ast, j} (\overline{B})\subseteq H^1_{(\overline{F}^\ast)^\a_\fb (\fn)} (k, \overline{B}).\]
In particular if, for $\m \in \cN_j$, we set   
\begin{align*}
 \lambda_{\overline{F}} (\m, j)  & \coloneqq\,  \dim_\mathbb{k} \big( H^1_{\overline{F} (\m)} (k, \overline{A}) / \sha_{\overline{F}, j} (\overline{A}) \big) \\
   \lambda_{\overline{F}}^{\ast} (\m,j)  & \coloneqq\, \dim_\mathbb{k} \big( H^{1}_{\overline{F}^\ast (\m)} (k, \overline{B}) / \sha_{\overline{F}^\ast, j} (\overline{B})\big),
 \end{align*}
then the integer 
\[\bm{\chi} (\overline{F},j) \coloneqq \lambda_{\overline{F}} (\m,j) - \lambda_{\overline{F}}^{\ast} (\m,j)
\]
 provides an appropriate analogue in our theory of the cohomological `core-rank' invariants for Mazur--Rubin structures that are introduced in \cite[\S4.1]{MazurRubin04}. 
 
 The basic properties of these integers are as follows. 

\begin{lemma}\label{core rank relation} Fix $j \in \N$ with $j \ge i$. Then $\bm{\chi} (\overline{F},j)$ is independent of $\m$. Further, for all pairwise coprime moduli $\a$, $\fb$ and $\fn$ in $\cN_j$, the following claims are valid. 
\begin{itemize}
\item[(i)] $\bm{\chi} ( \overline{F}^\fb_\a (\fn),j) = \bm{\chi} (\overline{F},j) + \nu (\fb) - \nu (\a)$.
\item[(ii)] $\bm{\chi} (\overline{F},j) = \bm{\chi} (\overline{F}(\fn),j)$.
\item[(iii)] If Hypothesis \ref{new strategy hyps}\,(i), (ii) and (iii) are satisfied, then $\lambda_{\overline{F}}^{\ast} (\fn,j)  = 
\dim_\mathbb{k}( H^{1}_{\overline{F}^\ast (\fn)} (k, \overline{B}))$. 
\end{itemize}
\end{lemma}

\begin{proof} Both the independence of $\bm{\chi} (\overline{F},j)$ from $\m$ and the equality in (i) follow directly from the argument of \cite[Cor. 3.21]{Sakamoto18}. Specifically, since 
\begin{align*}
 \lambda_{\overline{F}^\fb_\a} (\fn, j)  &= \dim_\mathbb{k} ( H^1_{\overline{F}^\fb_\a (\fn)} (k, \overline{A})) - \dim_{\mathbb{k}}(\sha_{\overline{F}, j} (\overline{A})) \\
   \lambda_{\overline{F}^\fb_\a }^{\ast} (\fn,j)  &= \dim_\mathbb{k} ( H^{1}_{(\overline{F}^\ast)^\fb_\a (\fn)} (k, \overline{B})) - \dim_\mathbb{k}(\sha_{\overline{F}^\ast, j} (\overline{B})),
 \end{align*}
one need only compare dimensions by using appropriate cases of  
 the global duality exact sequences 
(\ref{global duality sequence1})--(\ref{global duality sequence4}) with $\bLambda, \cA$ and both $\Phi$ and $\tilde{\cF}$ taken to be $\mathbb{k}, \overline{A}$ and $\overline{F}$ respectively.

Claim (ii) then follows directly from (i) in the case $\a = \fb = 1$, and (iii) is an immediate consequence of Lemma \ref{general neukirch interpretation lemma}\,(iv). 
\end{proof} 

\begin{remark}\label{lambda independent} Lemma \ref{core rank relation}\,(iii) implies that, if Hypothesis \ref{new strategy hyps} is satisfied, then  
$\lambda_{\overline{F}}^{\ast} (\fn,j)$ is a non-negative integer that is independent of $j$. In such  cases, we will usually abbreviate $\lambda_{\overline{F}}^{\ast} (\fn,j)$ to $\lambda_{\overline{F}}^{\ast} (\fn)$. \end{remark} 

\subsubsection{The Cebotarev density theorem}

In the sequel, for $j \in \N$ with $j \ge i$ we write $\sha_j (\cA)$ for the subset of $H^1 (k,\cA)$ comprising all classes that are locally trivial at every place in $\cQ_j$ and $\fX_j (\cA)$ for $H^1 (\cG_{k_j (\cT_j)}, \cA)$, regarded as a submodule of $H^1 (k, \cA)$ via the inflation map.

Then the following consequence of the Cebotarev density theorem constitutes a refined version of both \cite[Prop.\@ 3.6.1]{MazurRubin04} and \cite[Lem.\@ 3.9]{bss}.
\begin{prop} \label{cebotarev prop tweak}
We assume to be given data of one of the following forms:
\begin{romanliste}
    \item Hypotheses \ref{new strategy hyps}\,(i), (ii) and (ii${}^\ast$) are satisfied; $Z_1 \in \{ \cA,\overline{\cA}, A, \overline{A} \}$ and $D = \{ c_1, c_1^\ast \}$ where, for some $j \in \N$ with  $j \ge i$, one has 
    \[ c_1\in (H^1 (k, Z_1) / \sha_j (Z_1)) \setminus \{0\}\,\,\text{ and }\,\, c_1^\ast\in (H^1 (k, Z_1^\ast (1)) / \sha_j ( Z_1^\ast (1)))\setminus \{0\}.\]
    \item Hypotheses \ref{new strategy hyps}\,(i), (ii), (ii${}^\ast$) and (iv) are satisfied; $Z_1 \in \{A, \overline{A} \}$, $Z_2 \in \{ \cA, \overline{\cA}, A, \overline{A}\}$ and $D = \{ c_1, c_2, c_1^\ast, c_2^\ast \}$ where, for some $j \in \N$ with $j \ge i$ and both $s =1$ and $s =2$, one has 
    \[ c_s \in (H^1 (k, Z_s) / \sha_j ( Z_s))\setminus \{0\} \,\,\text{ and }\,\, c_s^\ast\in (H^1 (k, Z_s^\ast (1)) / \sha_j( Z_s^\ast (1)))\setminus \{0\}.\]  
\end{romanliste}
Then there exists a subset $\cS \subseteq \cQ_j$ of positive density with $\psi^\fs_\q (c) \neq 0$ for all $c \in D$ and $\q \in \cS$. 
\end{prop}

\begin{proof}
 We note, at the outset, we may assume  $c_s\notin \fX_j (Z_s)$ for $s \in \{1, 2 \}$. Indeed, if $c_s \in \fX_j (Z_s) / \sha_j (Z_s)$ is nonzero, then Lemma \ref{general neukirch interpretation lemma}\,(iii) implies $\loc_v (c_s) \neq 0$ for all $v \in \cQ_j$ and so the respective claims are trivial for $c_s$. Similarly, we may assume $c_s^\ast\notin \fX_j (Z_s^\ast (1))$ for  $s \in \{1, 2 \}$. \\
    Next we observe that, for each $\q \in \cQ_j$, there exists a unit $u_\q$ (in $\bLambda$, $\Lambda$, $\mathbb{K}$ or $\mathbb{k}$ as appropriate, and only depending on $\q$) such that   
\[ \psi^\fs_\q (x) = u_\q \cdot x ( \Frob_\q) \in Z_s/( \Frob_\q - 1) Z_s\quad\text{for all}\quad x \in H^1 (k, Z_s),\]
and similarly for $Z_s^\ast (1)$.

Writing $f_{\q} \in H^1(k, Z_1)^\ast$ and $f_{\q}^\ast \in H^1(k, Z_1^\ast (1))^\ast$
for the respective maps $x \mapsto  x(\Frob_\q)$, to prove the claim in case (i) it is thus enough show there exists a subset $\cS \subseteq \cQ_j$ of positive density with the property that $f_{\q} (c_1) \neq 0$ and $f_{\q}^\ast (c_1^\ast) \not= 0$ for every $\q \in \cS$. To do this, we set $\Delta \coloneqq \gal{k_j (\cT_j)}{k}$ and consider the composite map
\[
\phi \: H^1 (k, Z_1) \stackrel{\mathrm{Res}}{\longrightarrow} H^1 (k_j (\cT_j), Z_1)^\Delta = \Hom_\Delta (G_{k_j (\cT_j)}, Z_1) 
\to \Hom (G_{k_j (\cT_j)}, Z_1 / (\tau - 1) Z_1)
\]
where `$\mathrm{Res}$' denotes the natural restriction map and the unlabelled arrow the map induced by the natural projection $Z_1 \to Z_1 / (\tau - 1)Z_1$.\\ 
Now, by assumption, $c_1$ does not belong to $\fX_j (M)$, and hence not to the kernel of the first map in the composite $\phi$. Since Lemma \ref{injectivity lemma} implies the second map in this composite is injective, it follows that   $\phi (c_1) \neq 0$ and so $\ker (\phi (c_1))$ is a proper subgroup of $G_{k_j (\cT_j)}$. Consider the subset $H_1 \coloneqq \phi (c_1)^{-1} ( - c_1 (\tau))$ 
of $G_{k_j (\cT_j)}$. Then, if $g_1, g_2 \in H_1$, one has  $g_1 g_2^{-1}\in \ker(\phi (c_1))$ (since $\phi (c_1)$ is a homomorphism) and so  $H_1$ is a coset of $\ker (\phi (c_1))$.\\ 
In exactly the same way, if  we set $
H_1^\ast \coloneqq \phi' (c_1^\ast)^{-1} ( - c_1^\ast (\tau))$, 
with $\phi'$ the composite map 
\begin{align*}
 H^1 (k, Z_1^\ast (1))  \stackrel{\mathrm{Res}^\ast}{\longrightarrow} &H^1 (k_j (\cT_j), Z_1^\ast (1))^\Delta \\
 = &\Hom_\Delta (G_{k_j (\cT_j)}, Z_1) \to \Hom_\Delta (G_{k_j (\cT_j)}, Z_1^\ast (1) / (\tau - 1)Z_1^\ast (1)),
\end{align*}
then one can deduce $H_1^\ast$ is coset of the proper subgroup $\ker (\phi' (c_1^\ast))$ of $G_{k_j (\cT_j)}$.\\ 
It now follows from the general result of Lemma \ref{group theory lemma 1} below that $H_1 \cup H_1^\ast$ is a proper subset of $G_{k_j (\cT_j)}$. (Here we are use the fact that, if $\tau = \id$, then $H_1 = \ker (\phi (c_1))$ and $H_1^\ast = \ker (\phi' (c_1^\ast))$ are subgroups of $G_{k_j (\cT_j)}$). Fix an element $\gamma \in G_{k_j (\cT_j)} \setminus (H_1 \cup H_1^\ast)$ and write $L \coloneqq L_1 L_1^\ast$ for the finite extension of $k_j (\cT_j)$ 
obtaining by composing the fixed fields $L_1$ and $L_1^\ast$ of the respective  kernels of $\mathrm{Res} (c_1)$ and $\mathrm{Res}^\ast (c_1^\ast)$. 
 Let $\cS \subseteq \cQ_j$ be the subset of primes that are both coprime to $\fn$ and such that the restriction of $\Frob_\q$ to $L$ agrees with $\tau \gamma$. Then, by construction, for every $\q \in \cS$ one has $\tau^{-1}\Frob_\q\in G_{k_j (\cT_j)}$ and so, since the cocycle relation implies 
 \[ c_1 (\tau \tau^{-1} \Frob_\q)
    = c_1(\tau) + \tau c_1 (\tau^{-1} \Frob_\q),\]
we can compute
\begin{align}\label{first explicit comp}
    c_1 ( \Frob_\q) = c_1 (\tau \tau^{-1} \Frob_\q)
    & \equiv c_1 (\tau) + \phi (c_1) ( \tau^{-1} \Frob_\q) \,\,\,\, (\mathrm{mod}\,\,\, (\tau - 1)\cA ) \\ 
    & \equiv c_1 (\tau) + \phi (c_1) (\gamma) \,\,\,\, (\mathrm{mod}\,\,\, (\tau - 1)\cA )\notag\\
    & \not\equiv 0 \,\,\,\, (\mathrm{mod}\,\,\, (\tau - 1)\cA ). \notag
\end{align}
Here the second congruence is valid since $\tau^{-1} \Frob_\q \in \gamma \ker (\phi (c_1))$ and the final assertion since $\gamma \not \in H_1$. The non-vanishing of each  $c_1 ( \Frob_\q)$ implies that $\cS$ has all of the required properties, and hence concludes the proof of the claim in case (i).

In case (ii) it is convenient to argue separately for the cases $p >3$ and $p \in \{2,3\}$.

If $p > 3$, then for $a \in \{1,2\}$ we set $H_a \coloneqq \phi (c_a)^{-1} ( - c_2 (\tau))$ and $H_a^\ast \coloneqq \phi' (c_a^\ast)^{-1} ( - c_a^\ast (\tau))$ and take $L_a$ and $L_a^\ast$ to be the fixed fields of the respective kernels of $\mathrm{Res}(c_a)$ and $\mathrm{Res}^\ast (c_a^\ast)$. Then, since a simple counting argument (using $p > 3$) implies $X \coloneqq H_1 \cup H_2 \cup H_1^\ast \cup H_2^\ast$ is not equal to $G_{k_j (\cT_j)}$, we can therefore again fix an element $\gamma \in G_{k_j (\cT_j)}\setminus X$ and, just as above, take $\cS$ to be the subset of primes that are coprime to $\fn$ and such that the restriction of $\Frob_\q$ to the compositum $L_1 L_2 L_1^\ast L_2^\ast$ agrees with $\tau \gamma$.

In the rest of the argument, we thus assume $p \in \{2,3\}$. In this case, we can first use Lemma \ref{ryotaro's reduction trick} to choose elements $y_1, y_2, y_1^\ast, y_2^\ast$ of $R$ such that each $y_s^\bullet \mathrm{Res}^\bullet (c_s)$ is a nonzero element of 
\[
\Hom_\Delta (G_{k_j (\cT_j)}, Z^\bullet_s) [\a_s] = \Hom_\Delta (G_{k_j (\cT_j)}, Z^\bullet_s [\a_s]) \cong \Hom_\Delta (G_{k_j (\cT_j)}, \overline{Z^\bullet_s}).  
\]
Here the symbol $\bullet$ can either be omitted or is equal to $\ast$, $\a_s$ denotes $\cM_i$, respectively $M_i$, if $Z_s$ is a module over $\bLambda$, respectively $\Lambda$, we set $\overline{Z^\bullet_s} \coloneqq Z^\bullet_s / \a_s Z^\bullet_s$ and
the isomorphism is induced by a choice of isomorphism $\mathbb{k} \cong \Lambda [\a_i]$, respectively \@ $\mathbb{K} \cong \bLambda$.

Given these choices, it is then enough to find a subset $\cS \subseteq \cQ_j$ of positive density such that, for every $v \in \cS$, all of the elements $\loc_v (y_1c_1), \loc_v (y_2 c_2), \loc_v (y_1^\ast c_1^\ast), \loc_v (y_2^\ast c_2^\ast)$ are nonzero.\\
To do this, we write $L_1, L_2, L_1^\ast$ and $L_2^\ast$ for the finite extensions of $k_j (\cT_j)$ that are the fixed fields of the respective kernels of $y_1 \mathrm{Res} (c_1)$, $y_2 \mathrm{Res}  (c_2)$, $y_1^\ast \mathrm{Res}^\ast  (c_1^\ast)$ and  $y_2^\ast \mathrm{Res}^\ast  (c_2^\ast)$. We then set $L' \coloneqq L_1 L_2 \cap L_1^\ast L_2^\ast$ and claim there is an equality
\[ L' = k_j (\cT_j).\]
To verify this, we note that $\gal{L'}{k_j (\cT_j)}$ is a quotient  of the subgroup $\gal{L_1L_2}{k_j (\cT_j)}$ of $\gal{L_1}{k_j (\cT_j)} \times \gal{L_2}{k_j (\cT_j)}$, which we can identify with a submodule of $\overline{A} \oplus \overline{Z_2}$ via the map $(y_1 \mathrm{Res} (c_1), y_2 \mathrm{Res} (c_2))$. This shows that $\gal{L'}{k_j (\cT_j)}$ is isomorphic to a $\Z_p [G_k]$-subquotient of $\Tbar \oplus \overline{Z_2}$. Similarly, $(y_1^\ast \mathrm{Res} (c_1^\ast), y_2^\ast \mathrm{Res} (c_2^\ast))$ induces an isomorphism between $\gal{L'}{k_j (\cT_j)}$ and a $\Z_p [G_k]$-subquotient of the module $\overline{B}\oplus \overline{Z_2}^\ast (1)$. In order to deduce $\gal{L'}{k_j (\cT_j)}$ is trivial, and hence that $L' = k_j (\cT_j)$, it is thus enough to show Hypothesis \ref{new strategy hyps}\,(iv) implies the $\ZZ_p[G_k]$-modules $\overline{A} \oplus \overline{Z_2}$ and $\overline{B} \oplus \overline{Z_2}^\ast (1)$  have no non-zero isomorphic subquotients. In addition, since $\ZZ_p[G_k]$ acts on these modules via a finite (and hence Artinian) quotient ring, the modules have composition series and so, by the Jordan--H\"older Theorem, it is enough to show that they have no isomorphic composition factors (as $\ZZ_p[G_k]$-modules). Now every composition factor of $\overline{A} \oplus \overline{Z_2}$ is isomorphic to a composition factor of either $\overline{A}$ or $\overline{Z_2}\in \{\overline{A}, \overline{\cA}\}$ and hence to a composition factor of $\overline{A}\oplus \overline{\cA}$. Similarly, every composition factor of $B \oplus \overline{Z_2}^\ast (1)$ is isomorphic to a composition factor of $B\oplus \overline{\cA}^\ast(1)$. To prove the claimed equality of fields, it is thus enough to show that no composition factor of $\overline{A} \oplus \overline{\cA}$ is isomorphic to a composition factor of $\overline{B}\oplus \overline{\cA}^\ast(1)$ and this follows directly from Hypothesis \ref{new strategy hyps}\,(iv). 

We now set $H_a^\bullet \coloneqq \phi ( y_a^\bullet c_a^\bullet)^{-1} ( - y_a^\bullet c_a^\bullet)$ for $a \in \{1, 2 \}$. We then use the same argument as in case (i) to choose elements $\gamma, \gamma^\ast \in G_{k_j (\cT_j)}$ with $\gamma \not \in H_1 \cup H_2$ and $\gamma^\ast \not \in H_1^\ast \cup H_2^\ast$. Since $L' = k_j (\cT_j)$, we can then choose an element $\gamma'\in G_{k_j (\cT_j)}$ whose restrictions to $L_1 L_2$ and $L_1^\ast L_2^\ast$ are $\gamma$ and $\gamma^\ast$ respectively. The subset $\cS$ of $\cQ_j$ comprising all primes $\q$ for which $\Frob_\q$ is conjugate to $\tau \gamma'$ in $\gal{L_1 L_2 L_1^\ast L_2^\ast}{k}$ is then easily checked to have all of the required properties.  
\end{proof}

The following general observation was used in the above argument.

\begin{lem} \label{group theory lemma 1}
Let $G$ be an infinite group, and $U_1$ and $U_2$ proper normal subgroups of $G$ of finite index. Let $V_1$ be a coset of $U_1$ in $G$, and $V_2$ a coset of $U_2$ in $G$. Then $G \neq V_1 \cup V_2$ unless $U_1 = U_2$ is a subgroup of index 2 and $V_1$ and $V_2$ are the two cosets in $G / U_1$. 
\end{lem}

\begin{proof}
Suppose $G = V_1 \cup V_2$. By assumption $U_1, U_2 \subsetneq G$, so also must have $V_1, V_2, \subsetneq G$. We can therefore find elements $x, y \in G$ with $x \not \in V_1$ and $y \not \in V_2$. It follows that $x \in V_2$ and $y \in V_1$. If $z \coloneqq x y\in V_1$, then $z y^{-1} = x\in U_1$. Similarly, if $z\in V_2$, then $x^{-1} z = y\in U_2$.\\ 
If $(G : U_1) > 2$ and $(G : U_2) > 2$, then we may choose $x \not \in U_1 \cup V_1$ and $y \not \in U_2 \cup V_2$, which contradicts the previous conclusion. Without loss of generality we may therefore assume that $(G : U_1) = 2$. In particular, we have $G = V_1 \cup w U_1$ with $w$ an element of $G$ with $w \not \in V_1$.
In this case we therefore have $V_1' \coloneqq w U_1 \subseteq V_2$.\\  
Now let $u_1$ be an element of $U_1$. Then $u$ must belong to $V_1$ or $V_1'$. In the first case $V_1 \cap U_1 \neq \emptyset$ and so $V_1 = U_1$. In the second case $V_1' \cap U_1 \neq \emptyset$ and so $V_1' = U_1$. Similarly an element $u_2$ of $U_2$ must belong to $V_1$ or $V_1'$. 
If it belongs to $V_1' \subseteq V_2$, then $U_2 \cap V_2 \neq \emptyset$ and so $V_2 = U_2$. Since $u_2$ was arbitrary, this is the case if \textit{any} element of $U_2$ belongs to $V_1'$. Otherwise we therefore must have $U_2 \cap V_1' = \emptyset$, and hence $U_2 \subseteq V_1$. \\ 
It is thus enough to consider separately the following four cases:
\begin{itemize}[label=$\circ$, leftmargin=*]
    \item $V_1 = U_1$ and $V_2 = U_2$.  
    In this case one has $G = U_1 \cup U_2$ and so an elementary argument in group theory (in fact, the same as at the beginning of this proof) implies $G = U_1$ or $G = U_2$, which contradicts the hypothesis that $U_1$ and $U_2$ are proper subgroups of $G$.
    \item $V_1 = U_1$ and $U_2 \subseteq V_1$.  
    In this case $U_2 \subseteq V_1 = U_1$ and so the argument in the beginning of the proof implies $U_2$ has index 2. Since $U_1$ also has index 2, it follows that $U_1 = U_2$.
    \item $V_1' = U_1$ and $V_2 = U_2$.  
    In this case $U_1 = V_1' \subseteq V_2 = U_2$. Since $(G : U_1) = 2$ and $U_2 \neq G$, it follows that $U_1 = U_2$. 
    \item $V_1' = U_1$ and $U_2 \subseteq V_1$.  
    In this case $U_1 \cap U_2 \subseteq V_1' \cap V_1 = \emptyset$. However, since $G$ is infinite, and both $U_1$ and $U_2$ have finite index in $G$, this is impossible. 
    \qedhere
    \end{itemize}
\end{proof}

\subsubsection{Relative core vertices}\label{existence core section}

In this subsection, we discuss the following analogue in our theory of the key notion of core vertex from \cite[Def.~4.1.8]{MazurRubin04}. 

\begin{definition}
    A `relative core vertex' for  $\tilde {\cF}$ on $\cA$ is a modulus $\fn$ in $\cN$ for which the group $H^1_{\tilde F^\ast (\fn)} (k, B)$ vanishes. 
\end{definition}

\begin{rk} This notion is relative to the given ring homomorphism $\bLambda \to \Lambda$, and is a weakening of the notion of core vertex used by Mazur and Rubin. To justify the latter observation, we let $\fn$ is a core vertex for $\tilde{\cF}$ in the sense of \cite[Def.\@ 4.1.8]{MazurRubin04}. Then the group $H^1_{\tilde{\cF}^\ast (\fn)} (k, \cB)$ vanishes and so \cite[Cor. 3.8]{bss} implies  $H^1_{(\tilde{\cF}_{A})^\ast (\fn)} (k, B)$ also vanishes, 
where $\tilde{\cF}_{A}$ is the Mazur--Rubin structure on $A$ induced by $\tilde{\cF}$ (as in Example \ref{ss def}\,(iv)). In particular, since $\tilde{\cF}_{A} \le \tilde F$ (cf.\@ Remark \ref{remark selmer2}\,(ii)), and hence  $\tilde F^\ast(\fn) \le (\tilde{\cF}_{A})^\ast(\fn)$, the group $H^1_{\tilde F^\ast (\fn)} (k, B)$  vanishes. It follows that any core vertex for $\tilde{\cF}$ in the sense of Mazur and Rubin is a relative core vertex for $\tilde{\cF}$ in the above sense, as claimed.  However, the converse may not be true. In fact, even though our next result  shows Hypotheses \ref{new strategy hyps} guarantees the existence of relative core vertices for $\tilde{\cF}$, it is still possible  $H^1_{\tilde{\cF}^\ast (\fn)} (k, \cB)$ is nonzero for every $\fn \in \cN$.  Fortunately, however, if $\fn$ is a relative core vertex for $\tilde{\cF}$, then in all cases one can usefully `bound' the complexity of the $\bLambda$-module $H^1_{\tilde{\cF}^\ast (\fn)} (k,\cB)$ (see Remark \ref{core vertex remark} below) and this observation will play a key role in our theory.\end{rk}

\begin{rk} \label{core vertices equivalent def rk}
Lemma \ref{ryotaro's reduction trick} implies  $H^1_{\tilde F^\ast (\fn)} (k, B) = (0) \Longleftrightarrow H^1_{\tilde F^\ast (\fn)} (k, B) [M_i] = (0)$, and Lemma \ref{dual selmer group commutes with M torsion} implies $H^1_{\tilde F^\ast (\fn)} (k, B) [M_i]$ is isomorphic to $H^1_{\overline{F}^\ast (\fn)} (k, \overline{B})$. Hence, if Hypothesis \ref{new strategy hyps}\,(i), (ii) and (iii) are satisfied, then Lemma \ref{core rank relation}\,(iii) implies  $\fn$ is a relative core vertex for  $\tilde{\cF}$ if and only if the (non-negative) integer $\lambda_{\overline{F}}^\ast (\fn)$ defined in  Remark \ref{lambda independent} is equal to $0$, or equivalently, for each $j \in\N$ with $j \ge i$ the core-rank of the pair $(\overline{F},j)$ is equal to    
\begin{equation}\label{core equivalence}\bm{\chi}(\overline{F},j) = \dim_\mathbb{k} ( H^1_{\overline{F}(\fn)} (k, \overline{A})) - \dim_{\mathbb{k}}(\sha_{\overline{F}, j} (\overline{A})).\end{equation}
\end{rk}

The next result guarantees the existence of relative core vertices with certain additional properties that will be essential in our later arguments. 

\begin{lem} \label{existence of core vertices replacement}
Assume Hypotheses \ref{new strategy hyps}. Then for every $j \in \N$ with $j \ge i$, there exists a relative core vertex $\fn$ for $\tilde{\cF}$ that belongs to $\cN_j (\subseteq \cN)$ and is such that $\nu (\fn)$ is equal to the integer $\lambda^\ast_{\overline{F}}(1)$ defined in  Remark \ref{lambda independent}. 
\end{lem}

\begin{proof} Set $s \coloneqq \lambda_{\overline{F}}^\ast (1) = \dim_\mathbb{k} ( H^{1}_{\overline{F}^\ast} (k, \overline{B}))$. Then we shall use an induction on $s$ to construct a modulus $\fn$ in $\cN_j$ such that $\nu (\fn) =  \lambda_{\overline{F}}^\ast (1)$ and $H^1_{\overline{F}^\ast (\fn)} (k, \overline{B})$ vanishes. This is enough since, for any such $\n$, Lemma \ref{dual selmer group commutes with M torsion} implies $H^1_{\tilde F^\ast (\fn)} (k, B) [M_i] = (0)$ and hence, by Lemma \ref{ryotaro's reduction trick}, that $H^1_{\tilde F^\ast (\fn)} (k, B) = (0)$ so that $\fn$ is a relative core vertex for 
$\tilde{\cF}$ of the required form. \\
If, firstly, $s = 0$, then the claim is clearly satisfied (with $\fn = 1$) and so we assume $s > 0$. 
In this case, we shall now inductively construct primes $\{\q_a\}_{a \in [s]} \subseteq \cQ_j$ with the property that, for each $b \in [s]$, the ideal $\fn_b \coloneqq {\prod}_{a \in [b]} \q_a$ is such that $\lambda_{\overline{F}}^\ast (\fn_b) = \lambda_{\overline{F}}^\ast (1) - b$. We thus assume that suitable primes $\{\q_a\}_{a \in [b]}$ have been constructed for some $b$ with $0 \leq b < s$, and we set $\fn_b \coloneqq {\prod}_{a\in [b]} \q_a$. 

Then, since Hypothesis \ref{new strategy hyps}\,(vi) implies $\bm{\chi} (\overline{F},j(i)) \geq 0$, where the integer $j(i)$ is defined in Definition \ref{J def}, one has 
\[  0 < \lambda_{\overline{F}}^\ast (1) - b = \lambda_{\overline{F}}^\ast (\fn_b) = \lambda_{\overline{F}}^\ast (\fn_b,j(i)) \leq \lambda_{\overline{F}}(\fn_b,j(i)).\]
The groups $H^1_{\overline{F} (\fn_b)} ( k, \overline{A}) / \sha_{\overline{F},j(i)} (\overline{A})$ and $H^1_{\overline{F}^\ast (\fn_b)} (k,\overline{B})$ are therefore both non-trivial. 
We may therefore apply Proposition \ref{cebotarev prop tweak}\,(i) in order to fix a prime $\q_{b + 1} \in\cQ_j$ such that the localisation maps 
$H^1_{\overline{F} (\fn_b)} (k, \overline{A}) \to H^1 (k_{\q_{b+1}}, \overline{A})$ and $H^1_{\overline{F}^\ast (\fn_b)} (k, \overline{B}) \to H^1 (k_{\q_{b+1}}, \overline{B})$ are both nonzero. Given this choice, the result of \cite[Prop.\@ 5.7]{bss} then implies that 
\[ \lambda_{\overline{F}}^\ast (\fn_{b + 1}) = \lambda_{\overline{F}}^\ast (\fn_b) - 1 = \lambda_{\overline{F}}^\ast (1) - (b + 1),\] 
as required to complete the induction step. This proves the claimed result.
\end{proof}

We show next that, for suitable moduli $\a$, $\fb$ and $\fn$, the ideals $J_i$ in Definition \ref{J def} can be used to bound the complexity of the groups $H^1_{(\tilde{\cF}^\ast)^\fb_\a(\fn)} (k,\cB)= H^1_{(\tilde{\cF}_{\fb}^\a)^\ast (\fn)} (k,\cB)$.  

\begin{lem} \label{what do replacement core vertices do}
    Let $\a, \fb$ and $\fn$ be pairwise coprime moduli  in $\cN_{j(i)}$ for which the Selmer group $H^1_{(\overline{F_{\fb}^\a})^\ast (\fn)} (k, \overline{B})$ vanishes. Then one has $\Fitt^0_{\Lambda} (J_i) \subseteq \Lambda\cdot \varrho_i\bigl(\Fitt^0_{\bLambda} ( H^1_{(\tilde{\cF}_{\fb}^\a)^\ast (\fn)} (k, \cB)^\ast)\bigr).$
\end{lem}

\begin{proof} 
We compare the exact sequence obtained by applying the functor $(-) \otimes_{\bLambda} \Lambda$ to the short exact sequence in Proposition \ref{what we need from Selmer complexes}\,(iii) with the short exact sequence obtained from Proposition \ref{what we need from Selmer complexes}\,(iii) after replacing  $\cA$ and $\tilde{\cF}$ by $A$ and $\tilde F$. In this way, we obtain an exact commutative diagram of the form 
\begin{equation} \label{base change diagram}
    \begin{tikzcd}[column sep=small, row sep=small]
     \Tor_1^{\bLambda} ( X (\tilde{\mathscr{F}}), \Lambda) \arrow{r} &  
       H^1_{(\tilde{\cF}_\fb^\a)^\ast (\fn)} (k, \cB)^\ast \otimes_{\bLambda} \Lambda \arrow{r} \arrow{d} & H^1 (C (\tilde{\mathscr{F}}^\a_\fb (\fn))) \otimes_{\bLambda} \Lambda \arrow{d}{\simeq} \\ 
    0 \arrow{r} & H^{1}_{(\tilde{F}_{\fb}^\a)^\ast (\fn)} (k, B)^\ast \arrow{r} &  H^1 (C( (\tilde{\mathscr{F}}\otimes_{\bLambda}\Lambda)^\a_\fb (\fn))),
    \end{tikzcd}%
    \end{equation}
in which the vertical isomorphism is induced, via Lemma \ref{how the cohomology base changes lemma}\,(ii), by the isomorphism in Proposition \ref{what we need from Selmer complexes}\,(ii). In addition, the assumed vanishing of $H^1_{(\overline{F}_\fb^\a)^\ast (\fn)} (k, \overline{B})$ combines with Lemmas \ref{ryotaro's reduction trick} and \ref{dual selmer group commutes with M torsion} to imply that $H^{1}_{(\tilde{F}_{\fb}^\a)^\ast (\fn)} (k, B)$ also vanishes.  The commutativity of the above diagram therefore gives rise to an exact sequence 
    \begin{equation*} \label{Selmer, sha, and tor exact sequence}
        \Tor_1^{\bLambda} ( X (\tilde{\mathscr{F}}), \Lambda)  \to  H^1_{(\tilde{\cF}_{\fb}^\a)^\ast (\fn)} (k, \cB)^\ast \otimes_{\bLambda} \Lambda \to 0.
    \end{equation*}
Upon comparing this sequence to that obtained by applying the functor  $(-) \otimes_{R_{j(i)}} \Lambda$ to the analogous sequence with $\bLambda, \tilde{\mathscr{F}}, \Lambda$ and $\cB$ replaced by $\cR_{j(i)}, \mathscr{F}_{j(i)}, R_{j(i)}$ and $\cT_{j(i)}^\ast(1)$, one then obtains an  exact commutative diagram
\begin{cdiagram}[column sep=small, row sep=small]
   \Tor_1^{\cR_{j (i)}} (X (\mathscr{F}_{j (i)}), \bLambda) \otimes_{R_{j(i)}} \Lambda \arrow{r} \arrow{d}{\alpha} & H^1_{((\cF_{j(i)})_\fb^\a)^\ast (\fn)} ( k, \cT_{j (i)}^\ast (1))^\ast \otimes_{\cR_{j (i)}} \Lambda \arrow{d}{\beta} \arrow{r}  & 
    0
    \\ 
    \Tor_1^{\bLambda} (X (\tilde{\mathscr{F}}),\Lambda) \arrow{r} & H^1_{(\tilde{\cF}_{\fb}^\a)^\ast (\fn)} ( k, \cB)^\ast \otimes_{\bLambda} \Lambda \arrow{r} & 
    0.
\end{cdiagram}%
Here $\alpha$ is the natural map and, similarly to (\ref{base change diagram}), the map $\beta$ is 
induced by the map 
\[
\beta' \: H^1_{((\cF_{j(i)})_\fb^\a)^\ast (\fn)} ( k, \cT_{j (i)}^\ast (1))^\ast \to H^1_{(\tilde{\cF}_{\fb}^\a)^\ast (\fn)} ( k, \cB)^\ast
\]
that arises as the restriction (via the exact sequence in Proposition \ref{what we need from Selmer complexes}\,(iii)) of the surjective map
 $H^1 (C ((\mathscr{F}_{j (i)})^\a_\fb (\fn))) \to H^1 (C (\tilde{\mathscr{F}}^\a_\fb (\fn)))$ induced by Proposition \ref{what we need from Selmer complexes}\,(ii). In particular, $\beta$ is surjective since $\beta'$ is dual to the injective map
 \[
 H^1_{(\tilde{\cF}_\fb^\a)^\ast (\fn)} ( k, \cB) \cong H^1_{((\cF_{j(i)})_\fb^\a)^\ast (\fn)} ( k, \cT_{j (i)}^\ast (1)) [\a_i] 
 \subseteq H^1_{((\cF_{j(i)})_\fb^\a)^\ast (\fn)} ( k, \cT_{j (i)}^\ast (1))
 \]
 induced by Lemma \ref{dual selmer group commutes with M torsion}.
 Now, by definition of the integer $j(i)$, the image of $\alpha$ in the above diagram is equal to the submodule $J_i$ and so the commutativity of the diagram combines with the surjectivity of $\beta$ to imply the existence of a surjective map  of $\Lambda$-modules of the form $J_i \twoheadrightarrow H^1_{(\tilde{\cF}_{\fb}^\a)^\ast (\fn)} ( k, \cB)^\ast \otimes_{\bLambda}\Lambda$. This surjective map then combines with Lemma \ref{standard fitting props}\,(ii) and (iv) to imply an inclusion
\begin{align*} 
 \Fitt^0_{\Lambda} (J_i) \subseteq &  \Fitt^0_{\Lambda} ( H^{1}_{(\tilde{\cF}_{\fb}^\a)^\ast (\fn)} (k, \cB)^\ast \otimes_{\bLambda} \Lambda) 
=  \Lambda \cdot \varrho_i\big(\Fitt^0_{\bLambda} ( H^{1}_{(\tilde{\cF}_{\fb}^\a)^\ast (\fn)} (k, \cB)^\ast)\big),
\end{align*}
and this proves the claimed result. 
\end{proof}

\begin{rk}\label{core vertex remark} Taking $\a = \fb = 1$ in Lemma \ref{what do replacement core vertices do}, one obtains the following useful fact: if $\fn$ is a relative core vertex for $\tilde{\cF}$ that belongs to  $\cN_{j (i)}$, then there is an inclusion 
\[ \Fitt^0_{\Lambda} (J_i) \subseteq \Lambda\cdot \varrho_i\bigl(\Fitt^0_{\bLambda} ( H^1_{\tilde{\cF}^\ast (\fn)} (k, \cB)^\ast)\bigr).\]\end{rk}

\subsection{Controlling Kolyvagin systems via relative core vertices}

The aim of this subsection is to show that if $\fn$ is a relative core vertex for $\tilde{\cF}$ that belongs to $\cN_{j(i)}$, and $\kappa$ is any Kolyvagin system in $\KS^t (\tilde{\mathscr{F}})$ that vanishes at $\fn$, then the value of $\kappa$ at the modulus $1$ can be explicitly controlled. This result is stated precisely as  Theorem \ref{core vertices and injectivity replacement} and its proof adapts the graph-theoretical analysis of core vertices that forms a key part of the approach of Mazur and Rubin in \cite[\S\,4]{MazurRubin04}. 

\subsubsection{Graphs and paths} 

We use the following variants of the notion of the graph $\cX^0$ of core vertices for Mazur--Rubin structures that are defined in \cite[Def.\@ 4.3.6]{MazurRubin04} (see also \cite[Def.\@ 5.14]{bss}).

\begin{definition} \label{graph definition}
    For each $j\in \N$ with $j \geq i$ we define a graph $\cX^0_j = \cX^0_j (i)$ as follows.
    \begin{romanliste}
        \item The vertices of $\cX^0_j$ are the relative core vertices for $\tilde{\cF}$ that are contained in $\cN_j$.
        \item Vertices $\fn$ and $\fn \q$ as in (i) are joined by an edge in $\cX^0_j$ if and only if the localisation map $H^1_{\overline{F} (\fn)} (k, \overline{A}) \to H^1_f (k_\q, \overline{A})$ is non-zero.
    \end{romanliste}
\end{definition}

\begin{definition} A `path' on $\cX^0_j$ is a finite ordered set 
\[ U = \{ (\fn_1, \q_1), \dots, (\fn_s, \q_s)\} \subset \cN_{j} \times \cQ_{j}\] 
with the property that, for every $a \in [s-1]$ one has either $\fn_{a + 1}= \fn_a \q_a$ or $\fn_{a + 1} = \fn_a/\q_a$. We set $|U| \coloneqq s$ and refer to this as the `length' of  $U$. We also say that moduli $\fn$ and $\fn'$ in $\cN_j$ are connected by $U$ if one has $\fn = \fn_1$ and $\fn' = \fn_s$.  \end{definition}

We shall need an upper bound on the minimum possible length of paths between certain pairs of vertices on $\cX^0_{j(i)}$. 
To prove such a result we shall carefully analyse the arguments of \cite[Cor.\@ 5.16]{bss} in order to determine the length of the paths that are constructed in the latter result. In fact, though our current hypotheses are weaker than those of loc.\@ cit., this analysis doesn't require any essentially new ideas. Nevertheless, since it forms a key part of our argument, for the convenience of the reader we provide a detailed argument in the remainder of this subsection. 

We therefore start by recalling a result of Sakamoto et al.

\begin{lem}[{\cite[Lem.\@ 5.13]{bss}}] \label{bss lem 5.13}
    Fix $j \in \N$ with $j \ge i$, $\fn \in \cN_j$ and $\q \in \cQ_j\setminus V(\fn)$. Then the following claims are valid. 
    \begin{romanliste}
        \item If $\fn$ is a relative core vertex for $\tilde{\cF}$ and the map $H^1_{\overline{F} (\fn)} (k, \overline{A}) \to H^1_f (k_\q, \overline{A})$ is non-zero, then $\fn \q$ is a relative core vertex for $\tilde{\cF}$ and $\fn$ and $\fn \q$ are joined by an edge in $\cX^0_j$.
        \item If $\fn \q$ is a relative core vertex for $\tilde{\cF}$ and the map $H^1_{\overline{F} (\fn \q)} (k, \overline{A}) \to H^1_\tr (k_\q, \overline{A})$ is non-zero, then $\fn$ is a relative core vertex for $\tilde{\cF}$ and $\fn$ and $\fn \q$ are joined by an edge in $\cX^0_j$.
    \end{romanliste}
\end{lem}

Our next two results are then minor refinements of \cite[Lem.\@ 5.14]{bss} and \cite[Lem.\@ 5.15]{bss} respectively in that we also provide an explicit bound on the length of constructed paths.

\begin{lem}\label{bss lem 5.14}
Assume Hypotheses \ref{new strategy hyps} and fix $\fn \in \cN_{j(i)}$ and $\q \in \cQ_{j(i)}$. Then, if $\fn$ and $\fn\q$ are both relative core vertices for $\tilde{\cF}$, there exists a path from $\fn$ to $\fn \q$ in $\cX^0_{j(i)}$ that is of length at most three. 
\end{lem}

\begin{proof} If the localisation map $H^1_{\overline{F} (\fn)} (k, \overline{A}) \to H^1_f (k_\q, \overline{A})$ is nonzero, then Lemma \ref{bss lem 5.13}\,(i) implies that  $\fn$ and $\fn \q$ are joined by a path of length 1. We may therefore assume this localisation map is zero, and hence that 
 $H^1_{\overline{F} (\fn)} (k, \overline{A}) = H^1_{\overline{F}_\q (\fn)} (k, \overline{A}) \subseteq H^1_{\overline{F} (\fn \q)} (k, \overline{A})$. The latter inclusion must therefore be an equality since the fact $\fn$ and $\fn\q$ are relative core vertices combines with (\ref{core equivalence}) and the first assertion of Lemma \ref{core rank relation} to imply that  
\[ \dim_\mathbb{k} ( H^1_{\overline{F}(\fn)} (k, \overline{A})) = \bm{\chi} (\overline{F},j(i)) +  \dim_{\mathbb{k}}(\sha_{\overline{F}, j(i)} (\overline{A})) =  \dim_\mathrm{k}(H^1_{\overline{F} (\fn\q)} (k, \overline{A})).\]  
Since $\bm{\chi} (\overline{F},j(i)) > 0$ (by Hypotheses \ref{new strategy hyps}\,(vi)), the first equality here implies that the 
inclusion $\sha_{\overline{F}, j(i)} (\overline{A})\subseteq H^1_{\overline{F} (\fn)} (k, \overline{A})$ is strict. In addition, since $H^1_{\overline{F} (\fn)} (k, \overline{A}) = H^1_{\overline{F} (\fn \q)} (k, \overline{A})$, the exact sequence (\ref{global duality sequence3}) (with the data $\Phi =\tilde{\cF}, \cA, \bLambda,  \m, \n$ taken to be $\overline{F}, \overline{A}, \mathbb{k}, \fn\q$ and $\q$ respectively) implies  $H^{1}_{(\overline{F}^\ast)^\q (\fn)} (k, \overline{B})$ is non-trivial.
We can therefore apply Proposition \ref{cebotarev prop tweak}\,(i) to deduce the existence of  a prime $\mathfrak{r}$ in $\cQ_{j(i)}\setminus V(\fn \q)$ for which the localisation maps
\[
H^1_{\overline{F} (\fn)} (k, \overline{A}) \to H^1 (k_\mathfrak{r}, \overline{A}) 
\quad \text{ and } \quad 
H^1_{(\overline{F}^\ast)^\q (\fn)} (k, \overline{B}) \to H^1 (k_{\mathfrak{r}},  \overline{B})
\]
are both nonzero. By using the argument of \cite[Lem.\@ 5.14]{bss}, one then concludes the existence of a path in $\cX^0_{j(i)}$ of the form
 $\fn \to \fn \mathfrak{r} \to \fn \mathfrak{r} \q \leftarrow \fn \q$. 
 
At this stage, we have proved that, in all cases, $\fn$ and $\fn \q$ are connected by a path of length at most three, as required.
\end{proof}

\begin{lem}\label{bss lem 5.15}
Assume Hypotheses \ref{new strategy hyps}. Let $\fn_1$ and $\fn_2$ be relative core vertices for $\tilde {\cF}$ in $\cN_{j(i)}$, and fix primes $\q_1 \in V(\fn_1)$ and $\q_2 \in V(\fn_2)$ such that neither $\fn_1 / \q_1$ and $\fn_2 / \q_2$ are relative core vertices for $\tilde {\cF}$. Then there exists a prime $\mathfrak{r}$ in $\cQ_{j(i)}\setminus V(\fn_1 \fn_2)$ with the following properties. 
\begin{romanliste}
\item Both $\fn_1 \mathfrak{r} / \q_1$ and $\fn_2 \mathfrak{r} / \q_2$ are relative core vertices for $\tilde {\cF}$.
\item For $l\in \{1, 2 \}$, there exists a path in $\cX^0_{j(i)}$ between $\fn_l$ and $\fn_l \mathfrak{r} / \q_l$ of length at most four.
\end{romanliste} 
\end{lem}

\begin{proof}
    Proposition \ref{cebotarev prop tweak}\,(ii) combines with Hypothesis \ref{new strategy hyps}\,(iv) to imply the existence of a prime $\mathfrak{r}$ in $\cQ_{j(i)}\setminus V(\fn_1 \fn_2)$ such that the maps 
    \[
H^1_{\overline{F} (\fn_l)} (k, \overline{A}) \to H^1 (k_\mathfrak{r}, \overline{A}) 
\quad \text{ and } \quad 
    H^1_{\overline{F} (\fn_l / \q_l)} (k, \overline{A}) \to H^1 (k_\mathfrak{r}, \overline{A}) 
    \]
are non-zero for both $l \in \{1, 2 \}$. Lemma \ref{bss lem 5.13}\,(i) then implies that the moduli $\fn_l$ and $\fn_l \mathfrak{r}$ are directly connected by an edge in $\cX^0_{j(i)}$. The proof of \cite[Lem.\@ 5.15]{bss} moreover shows that $\fn_l \mathfrak{r} / \q_l$ is a relative core vertex for $\tilde {\cF}$, and so Lemma \ref{bss lem 5.14} implies the existence of a  path in $\cX^0_{j(i)}$ between $\fn_l \mathfrak{r}$ and $\fn_l \mathfrak{r} / \q_l$ of length at most three. In total, therefore, there exists a path between $\fn_l$ and $\fn_l \mathfrak{r} / \q_l$ in $\cX^0_{j(i)}$ that has length at most four.
\end{proof}

We can now prove our main observation concerning path lengths. 

\begin{prop} \label{length of paths between core vertices}
Let $\fn_1$ and $\fn_2$ be relative core vertices for $\tilde {\cF}$ in $\cN_{j(i)}$ for which $\nu (\fn_1) = \nu (\fn_2) = \lambda_{\overline{F}}^\ast (1)$. Then, if  Hypotheses \ref{new strategy hyps} is valid, there exists a path in $\cX^0_{j(i)}$ between $\fn_1$ and $\fn_2$ of length at most $8 \cdot (\lambda_{\overline{F}}^\ast (1) - \nu ( \mathrm{gcd} ( \fn_1, \fn_2)))$.
\end{prop}

\begin{proof} We argue by induction on the non-negative integer 
 \[ \mu_{\overline{F}} (\fn_1, \fn_2) \coloneqq \lambda^\ast_{\overline{F}} (1) - \nu ( \mathrm{gcd} ( \fn_1, \fn_2)).
 \]
Firstly, if $\mu_{\overline{F}}(\fn_1, \fn_2) = 0$, and hence $\nu ( \mathrm{gcd} ( \fn_1, \fn_2)) = \lambda^\ast_{\overline{F}} (1) $, then one must have $\fn_1 = \fn_2$ since, by assumption, $\nu (\fn_1)$ and $\nu(\fn_2)$ are also both equal to $\lambda^\ast_{\overline{F}}(1)$. This proves the induction base. \\ 
We therefore assume that $\mu_{\overline{F}} (\fn_1, \fn_2) > 0$, and hence that $\fn_1 \neq \fn_2$, and we then fix primes $\q_1 \in V(\fn_1 / \mathrm{gcd} (\fn_1, \fn_2))$ and $\q_2 \in V(\fn_2 / \mathrm{gcd} (\fn_1, \fn_2))$. Now, by \cite[Cor.\@ 5.11]{bss}, one knows that any relative core vertex $\m$ for $\tilde{\cF}$ satisfies $\nu (\m) \geq \lambda^\ast_{\overline{F}} (1)$. In particular, since
\[ \nu ( \fn_l / \q_l) = \nu (\fn_l) - 1 = \lambda^\ast_{\overline{F}} (1) - 1\] 
for $l \in \{1, 2\}$, neither $\fn_1 / \q_1$ nor $\fn_2 / \q_2$ can be a relative core vertex for $\tilde{\cF}$. Lemma \ref{bss lem 5.15}  therefore implies the existence of a prime $\mathfrak{r}$ in $\cQ_{j(i)}\setminus V(\fn_1 \fn_2)$ such that, for both $l \in \{1,2\}$, there exists a path in $\cX^0_{j(i)}$ between $\fn_l$ and $\fn_l \mathfrak{r} / \q_l$ of length at most four. In addition, one has
\begin{align*}\mu_{\overline{F}} ( \fn_1 \mathfrak{r} / \q_1, \fn_2 \mathfrak{r} / \q_2) =&\, 
\lambda^\ast_{\overline{F}}(1) - \nu ( \mathrm{gcd}( \fn_1 \mathfrak{r} / \q_1, \fn_2 \mathfrak{r} / \q_2))\\
=&\, \lambda^\ast_{\overline{F}} (1) - \nu ( \mathrm{gcd} ( \fn_1, \fn_2)) - 1\\
=&\, \mu_{\overline{F}} (\fn_1, \fn_2) - 1,
\end{align*}
and so, by the induction hypothesis, the vertices $\fn_1 \mathfrak{r} / \q_1$ and $\fn_2 \mathfrak{r} / \q_2$ are connected in $\cX^0_{j(i)}$ by a path of length at most $8\cdot \mu_{\overline{F}}(\fn_1 \mathfrak{r} / \q_1, \fn_2 \mathfrak{r} / \q_2)$.

By concatenating these three paths, we have therefore constructed a path in $\cX^0_{j(i)}$ from $\fn_1$ to $\fn_2$ (via  $\fn_1 \mathfrak{r} / \q_1$ and $\fn_2 \mathfrak{r} / \q_2$) that has length at most
\[ 
4 + 8\cdot \mu_{\overline{F}} (\fn_1 \mathfrak{r} / \q_1, \fn_2 \mathfrak{r} / \q_2) + 4
= 8 ( \mu_{\overline{F}} (\fn_1 \mathfrak{r} / \q_1, \fn_2 \mathfrak{r} / \q_2) + 1) = 8 \cdot\mu_{\overline{F}} (\fn_1, \fn_2),
\]
 as required. \end{proof} 

\subsubsection{Moving along the graph}

Given two values $\kappa_\fn$ and $\kappa_{\fn \q}$ of a system $\kappa$ in  $\KS^t (\tilde{\mathscr{F}})$, the defining relation of Kolyvagin systems (in Definition \ref{koly def}) relates their images under the maps $\check{\psi}_\q^\fs$ and $\check{v}_\q $ that are respectively induced on biduals by $\psi_\q^\mathrm{fs}$ and $v_\q$. In order to be able to relate the values $\kappa_\fn$ and $\kappa_{\fn \q}$ themselves, it is therefore crucial for us to control the kernels of  $\check{\psi}_\q^\fs$ and $\check{v}_\q $. This is achieved by the following result that  uses the  
 integer $\bm{\chi}_\fF$ fixed at the beginning of \S\,\ref{somr section}. 

\begin{lem} \label{bss lem 5.19 new replacement} Fix a non-negative integer $t$ for which one has $t_\fF \coloneqq t + \bm{\chi}_\fF > 0$. Then, for each modulus $\fn \in \cN_{j(i)}$ and prime $\q \in \mathcal{P}\setminus V(\fn)$, the kernels of both of the maps  
\begin{align*}
\check{\psi}_\q^\fs  & \:  \bidual^{t_\fF}_{\bLambda} H^1_{\tilde{\mathscr{F}}(\fn)} (k, \cA)   %
\; \to  \bidual^{t_\fF - 1}_{\bLambda}  H^1_{\tilde{\mathscr{F}}_{\q} (\fn)} (k, \cA) \\
\check{v}_\q & \:  \bidual^{t_\fF}_{\bLambda}  H^1_{\tilde{\mathscr{F}} (\fn \q)} (k, \cA)   
\to   \bidual^{t_{\fF} - 1}_{\bLambda}  H^1_{\tilde{\mathscr{F}}_{\q} (\fn)} (k, \cA)
\end{align*}
are annihilated by $\Fitt^0_{\bLambda} ( H^1_{(\tilde{\cF}^\ast)^{\q}(\fn)} (k, \cB)^\ast) \cdot \Fitt^{t}_{\bLambda} (X (\tilde{\mathscr{F}}))$.
\end{lem}

\begin{proof}
The  exact sequences (\ref{global duality sequence3}) and (\ref{global duality sequence4}) (with $\Phi = \tilde{\mathscr{F}}$ and $\n = \q, \m = \fn\q$, respectively $\a = \q$ and $\m = \fn$) imply 
 $H^1_{\tilde{\mathscr{F}}_{\q} (\fn)} (k, \cA)$ coincides with the kernel of both  $\check{v}_\q \: H^1_{\tilde{\mathscr{F}}(\fn \q)} (k, \cA) \to \bLambda$ and  $\check{\psi}_\q^\fs \: H^1_{\tilde{\mathscr{F}}(\fn)} (k, \cA)  \to \bLambda$. Lemma~\ref{biduals lemma 1}~(i) therefore implies  $\bidual^{t_\fF}_{\bLambda} H^1_{\tilde{\mathscr{F}}_{\q} (\fn)} (k, \cA)$ is the kernel of both of the displayed maps, and so we must show the stated ideal annihilates the latter module. In addition, by Lemma \ref{standard fitting props}~(v), the module  $\exprod^{t_\fF}_{\bLambda} H^1_{\tilde{\mathscr{F}}_{\q} (\fn)} (k, \cA)^\ast$, and hence also its $\bLambda$-linear dual $\bidual^{t_\fF}_{\bLambda} H^1_{\tilde{\mathscr{F}}_{\q} (\fn)} (k, \cA)$, is annihilated by $\Fitt^{t_\fF - 1}_{\bLambda} ( H^1_{\tilde{\mathscr{F}}_{\q} (\fn)} (k, \cA)^\ast)$
and so we are reduced to proving 
\begin{equation}\label{reduced inclusion} 
\Fitt^0_{\bLambda} ( H^1_{(\tilde{\cF}^\ast)^{\q}(\fn)} (k, \cB)^\ast) \cdot \Fitt^{t}_{\bLambda} (X (\tilde{\mathscr{F}})) 
\subseteq \Fitt^{t_\fF - 1}_{\bLambda} ( H^1_{\tilde{\mathscr{F}}_{\q} (\fn)} (k, \cA)^\ast).
\end{equation}

To do this, we use the complex $C (\tilde{\mathscr{F}}_{\q} (\fn))$ from Proposition \ref{what we need from Selmer complexes}.
In particular, from Proposition \ref{what we need from Selmer complexes}\,(i) and (iii) one has $\bm{\chi}_{\bLambda}(C (\tilde{\mathscr{F}}_{\q} (\fn))) =\bm{\chi}_\fF -1$ and $H^0(C (\tilde{\mathscr{F}}_{\q} (\fn))) = H^1_{\tilde{\mathscr{F}}_{\q}(\fn)} (k, \cA)$. By applying Lemma \ref{comparing Fitting ideals of H0 and H1 lemma} with  $C = C (\tilde{\mathscr{F}}_{\q} (\fn))$ and $Y = (0)$, one therefore has  
\[ \Fitt^{t_\fF - 1}_{\bLambda} ( H^1_{\tilde{\mathscr{F}}_{\q}(\fn)} (k, \cA)^\ast) = \Fitt^{t}_{\bLambda} ( H^1 ( C (\tilde{\mathscr{F}}_{\q} (\fn))). 
\]
To deduce the required equality (\ref{reduced inclusion}), we then need only note that Lemma \ref{standard fitting props}\,(ii) applies to the exact sequence in Proposition \ref{what we need from Selmer complexes}\,(iii) (with $\a$ and $\fb$ taken to be $\q$ and $1$ respectively) to imply that 
\begin{equation*}\label{r-s inclusion}
\Fitt^0_{\bLambda} ( H^1_{(\tilde{\cF}^\ast)^{\q} (\fn)} (k, \cB)^\ast) \cdot 
\Fitt^{t}_{\bLambda} ( X (\tilde{\mathscr{F}})) \subseteq  \Fitt^{t}_{\bLambda} ( H^1 ( C (\tilde{\mathscr{F}}_{\q} (\fn)))).
\qedhere
\end{equation*}
\end{proof}

Lemma \ref{bss lem 5.19 new replacement} has the following concrete consequence concerning the values of Kolyvagin systems at moduli that are connected by paths on the graph $\cX^0_{j(i)}$. 

\begin{lem}\label{induction on paths claim lemma}
    Fix $t \in \N_0$ with $t_\fF \coloneqq t + \bm{\chi}_\fF > 0$. Let $\kappa \in \KS^{t_\fF} (\tilde{\mathscr{F}})$ be a Kolyvagin system of rank $t_\fF$ for $\tilde{\mathscr{F}}$, $\fn$ a relative core vertex for $\tilde{\cF}$ in $\cN_{j(i)}$ with $\kappa_\fn = 0$ and $U$ a path in $\cX^0_{j(i)}$ that connects $\fn$ to $\fn'$. Then $\kappa_{\fn'}$ is annihilated by every element of the ideal 
    \begin{equation*} \label{induction on paths claim}
    \big(\prod_{ (\m,\q) \in U} \Fitt^0_{\bLambda} (H^1_{(\tilde{\cF}^\ast)^\q(\m)} (k, \cB)^\ast) \big) \cdot \Fitt^{t}_{\bLambda} (X (\tilde{\mathscr{F}}))^{|U|}. 
    \end{equation*} 
\end{lem}

\begin{proof}
We will prove the claim by induction on the length $s = |U|$ of $U$. Let us therefore assume that the claim has already been proved for all paths of length at most $s-1$. \\
Writing $U = ( (\fn_1, \q_1), \dots, (\fn_s, \q_s))$, we then see that $\fn_1$ and $\fn_{s - 1}$ are connected by a path of length $s - 1$. By the induction hypothesis it therefore follows that 
\begin{equation} \label{something lives in sha}
\big( {\prod}_{ l \in [s-1]} \Fitt^0_{\bLambda} (H^1_{(\tilde{\cF}^\ast)^{\q_l}(\fn_l)} (k, \cB)^\ast) \big) \cdot \Fitt^{t}_{\bLambda} (X (\tilde{\mathscr{F}}))^{s - 1} \cdot \kappa_{\fn_{s - 1}} = \{0 \}.
\end{equation}
Now, as $\fn' = \fn_s$ and $\fn_{s - 1}$ are connected by a path of length one, we have either $\fn' = \fn_{s - 1} \q_{s - 1}$ or $\fn' = \fn_{s - 1} / \q_{s - 1}$ and we consider these cases separately. 

We first assume $\fn' = \fn_{s - 1} \q_{s - 1}$. In this case, the defining relation of Kolyvagin systems implies 
\[
v_{\q_{s - 1}} ( \kappa_{\fn'}) = \psi_{\q_{s - 1}}^\fs ( \kappa_{\fn' /\q_{s - 1}}) = \psi_{\q_{s - 1}}^\fs (\kappa_{\fn_{s - 1}}).
\]
This equality then combines with (\ref{something lives in sha}) to imply an inclusion 
\[ \big( {\prod}_{ l \in [s-1]} \Fitt^0_{\bLambda} (H^1_{(\tilde{\cF}^\ast)^{\q_l}(\fn_l)} (k, \cB)^\ast) \big) \cdot \Fitt^{t}_{\bLambda} (X (\tilde{\mathscr{F}}))^{s - 1} \cdot \kappa_{\fn'} \subseteq \ker(v_{\q_{s - 1}}).\] 
In particular, since (the second assertion of) Lemma \ref{bss lem 5.19 new replacement} implies that the kernel of $v_{\q_{s - 1}}$ is annihilated by the product of $\Fitt^0_{\bLambda} ( H^1_{(\tilde{\cF}^\ast)^{\q_{s - 1}}(\fn_{s - 1})} (k, \cB)^\ast)$ and $\Fitt^{t}_{\bLambda} (X (\tilde{\mathscr{F}}))$, the claimed equality is clear in this case. 

We now assume $\fn' = \fn_{s- 1} / \q_{s - 1}$. In this case, the relevant Kolyvagin system relation asserts  
\[
\psi_{\q_{s - 1}}^\fs (\kappa_{\fn'}) = v_{\q_{s- 1}} (\kappa_{\fn' \q_{s - 1}}) = v_{\q_{s - 1}} (\kappa_{\fn_{s - 1}}).
\]
Then, just as above, this equality can be combined with (\ref{something lives in sha}) and (the first assertion of) Lemma \ref{bss lem 5.19 new replacement} to deduce the validity of the claimed equality.

This therefore concludes the inductive step, thereby proving the claimed result.
\end{proof}

\subsubsection{Bounding dimensions}

In this subsection, we shall provides  bounds for the dimensions of $\mathbb{K}$-spaces that arise in subsequent arguments.

\begin{lem} \label{when is the Selmer equal to sha}
  Let $\fn$ and $Q$ be coprime moduli in $\cN_{j(i)}$. Then, if the homomorphism 
\[
\big( H^1_{\tilde{\mathscr{F}}(\fn)} (k, \cA) / \sha_{\tilde{\mathscr{F}},j(i)}(\cA) \big) [\cM_i] \xrightarrow{ (\check{\psi}_\q^\fs )_\q} \bigoplus_{\q \in V(Q)}  \bLambda
\]
is injective, one has $H^1_{\tilde{\mathscr{F}}_{Q} (\fn)} (k, \cA) = \sha_{\tilde{\mathscr{F}}, j(i)} (\cA)$. 
\end{lem}

\begin{proof}
    The relevant case of the long exact sequence (\ref{global duality sequence4}) implies that the kernel of the displayed map is equal to 
    $\bigl(H^1_{\tilde{\mathscr{F}}_{Q} (\fn)} (k, \cA) / \sha_{\tilde{\mathscr{F}}, j(i)} (\cA)\bigr)[\mathcal{M}_i]$. The given assumption therefore implies that this module vanishes and hence, by Lemma \ref{ryotaro's reduction trick}, that $H^1_{\tilde{\mathscr{F}}_{Q} (\fn)} (k, \cA) / \sha_{\tilde{\mathscr{F}}, j(i)} (\cA)$ itself vanishes, as claimed.
\end{proof}

For each modulus $\fn$ in $\cN_{j(i)}$, we set   
\begin{align}\label{alpha def}
\alpha_i (\fn) \coloneqq&\, \dim_\mathbb{K} \bigl(\big( H^1_{\tilde{\mathscr{F}} (\fn)} (k, \cA) / \sha_{\tilde{\mathscr{F}}, j (i)} (\cA) \big) [\cM_i]\bigr) \\
=&\, \dim_\mathbb{K}\bigl( \big( H^1_{\tilde{\cF}(\fn)} (k, \cA) / \sha_{\tilde{\cF},j(i)} (\cA) \big) [\cM_i]\bigr),\notag\end{align}
where the equality is a consequence of the isomorphism in Lemma \ref{general neukirch interpretation lemma}\,(iii). This quantity constitutes an upper bound for the minimal possible value of $\nu(Q)$ for moduli $Q$ in $\cN_{j(i)}$ that are coprime to $\fn$ and such that the displayed map in Lemma \ref{when is the Selmer equal to sha} is injective. The following result establishes some crucial properties of these bounds. 

\begin{thm} \label{bound for alpha}
The following claims are valid.
\begin{romanliste}
    \item For each modulus $\fn \in \cN_{j (i)}$ and prime $\q \in \cQ_{j (i)} \setminus V (\fn)$, one has $|\alpha_i (\fn \q) - \alpha_i (\fn)| \leq 1$.
    \item There exists an increasing function $\Phi_{\mathscr{F}_k}: \N_0\to \N_0$ such that, for every $(i,\fn)\in \N\times \cN_{j(i)}$, one has 
    $\alpha_i (\fn) \leq \Phi_{\mathscr{F}_k}(\nu(\fn))$. 
\end{romanliste}
\end{thm}

\begin{proof} 
To prove (i), we fix an $\bLambda$-submodule $W$ of $H^1_{\tilde{\mathscr{F}}_{\q} (\fn)} (k, \cA)$.
Then, for the modulus $\fn' = \fn$, respectively $\fn' = \q \fn$, the exact sequence (\ref{global duality sequence1}) with $(\Phi,\fn,\m)$ taken to be $(\tilde{\mathscr{F}}(\fn), 1,\q)$, respectively the exact sequence (\ref{global duality sequence2}) with $(\Phi,\m)$ taken to be $(\tilde{\mathscr{F}}(\fn),\q)$, gives an exact sequence of $\bLambda$-modules 
\[
0 \to H^1_{\tilde{\mathscr{F}}(\fn')} (k, \cA) / W \to H^1_{\tilde{\mathscr{F}}^\q (\fn)} (k, \cA) / W \to \bLambda.
\]
Applying the functor $(-) [\cM_i]$ to this sequence, one  obtains an exact sequence of $\mathbb{K}$-modules 
\[
0 \to \big( H^1_{\tilde{\mathscr{F}}(\fn')} (k, \cA) / W \big) [\cM_i] \to \big( H^1_{\tilde{\mathscr{F}}^\q (\fn)} (k, \cA) /W \big) [\cM_i] \to \mathbb{K}, 
\]
from which one deduces  that 
\begin{equation} \label{inequality for additive functions}
\big | \dim_\mathbb{K} \bigl(\big( H^1_{\tilde{\mathscr{F}} (\fn \q)} (k, \cA) / W \big) [\cM_i]\bigr)
- \dim_\mathbb{K}\bigl( \big( H^1_{\tilde{\mathscr{F}} (\fn)} (k, \cA) / W \big) [\cM_i]\bigr) \big | \leq 1.
\end{equation}
Claim (i) now follows upon setting $W = \sha_{\tilde{\mathscr{F}}, j (i)} (\cA)$.\\
To prove (ii), we lighten notation by setting $\sha_s \coloneqq \sha_{\mathscr{F}_s, j (s)} (\cT_s)$ for each $s\ge 0$. Then, upon applying the functor $\Hom_{\cR_s} ( \mathbb{K}, - )$ to the tautological short exact sequence 
    \[ 
    0 \to \sha_s \to H^1_{\mathscr{F}_s(\fn)} (k, \cT_s) \to H^1_{\mathscr{F}_s (\fn)} (k, \cT_s) / \sha_s \to 0
    \]
    we obtain an exact sequence of $\mathbb{K}$-modules
\begin{equation}\label{deduced ses} H^1_{\mathscr{F}_s (\fn)} (k, \cT_s) [\cM_s]
    \to \big( H^1_{\mathscr{F}_s (\fn)} (k, \cR_s) /  \sha_s\big) [\cM_s]
    \to \Ext^1_{\cR_s} ( \mathbb{K}, \sha_s).
    \end{equation}
   To prove (ii), it is therefore suffices to provide suitable bounds on the $\mathbb{K}$-dimensions of the outer terms in this exact sequence.\\ 
For the first module, we note that Proposition \ref{what we need from Selmer complexes}\,(i), (ii) and (iii) allow us to apply  Lemma \ref{how the cohomology base changes lemma}\,(i) (with $C_n = C(\mathscr{F}_n(\fn)), a=0$ and $b=1$) to deduce the existence of a natural isomorphism of (finite) $\mathbb{K}$-modules 
\[ H^1_{\mathscr{F}_s(\fn)} (k, \cT_s) [\cM_s] \cong H^1_{\mathscr{F}_0(\fn)} (k, \overline{\cT}).\] 
We write $\beta (\fn)$ for the $\mathbb{K}$-dimension of the latter  module. Then, for each modulus $\m \in \cN_{j(s)}$ and prime $\q \in \cQ_{j(s)}\setminus V(\m)$, we can take $W = (0)$ in (\ref{inequality for additive functions}) (with $\tilde{\mathscr{F}} = \mathscr{F}_i$, $\mathcal{A}$ and $\cM_i$ replaced by $\mathscr{F}_0, \overline{\cT}$ and $\cM_0$ respectively) in order  to deduce an inequality $|\beta (\m \q) - \beta (\m)| \leq 1$. By combining this inequality with an induction on $\nu(\fn)$, one can then prove that  
\begin{equation}\label{first ineq}
\dim_{\mathbb{K}}\bigl(H^1_{\mathscr{F}_s (\fn)} (k, \cT_s) [\cM_s]\bigr) = \beta (\fn) \leq  \nu (\fn) + \beta(1) = 
\nu (\fn) +  \dim_\mathbb{K}\bigl(H^1_{\mathscr{F}_0} (k, \overline{\cT})\bigr). 
\end{equation}
To bound the $\mathbb{K}$-dimension of $\Ext^1_{\cR_s} ( \mathbb{K}, \sha_s)$ we note that the derived Tensor-Hom adjunction isomorphism in $D (\mathbb{K})$
 \[
 \RHom_{\cR_s} ( \mathbb{K}, \RHom_{\cR_s} ( \sha_s^\ast, \cR_s)) 
 \xrightarrow{\simeq} \RHom_{\cR_s} ( \mathbb{K} \otimes^\mathbb{L}_{\cR_s} \sha_s^\ast, \cR_s)
 \]
  (cf.\@ \cite[Th.\@ 10.8.7]{Weibel-homologicalAlgebra}) induces, on cohomology in degree one, an isomorphism of $\mathbb{K}$-modules 
\begin{equation}\label{ext-tor}
\Ext^1_{\cR_s} ( \mathbb{K}, \sha_s) \cong \Tor_1^{\cR_s} ( \mathbb{K}, \sha_s^\ast )^\ast.
\end{equation}
 To study the latter module, we  use the exact commutative diagram of $\cR_s$-modules
\begin{cdiagram}[row sep=small, column sep=small]
0 \arrow{r} & \sha_s \arrow[dashed]{d} \arrow{r} & H^1_{\mathscr{F}_s} (k, \cT_s) \arrow{r} \arrow{d}{\simeq} & {\bigoplus}_{\q \in \cQ_{j (s)}} H^1_f (k_\q, \cT_s) \arrow{d} \\ 
    0 \arrow{r} & \sha_{s + 1}[\a_s] \arrow{r} & H^1_{\mathscr{F}_{s + 1}} (k, \cT_{s + 1}) [\a_s] \arrow{r} & {\bigoplus}_{\q \in \cQ_{j (s + 1)}} H^1_f (k_\q, \cT_{s +1}) [\a_s].
\end{cdiagram}%
Here the rows follow directly from the respective definitions of $\sha_s$ and $\sha_{s + 1}$ and (in the second case) the left exactness of the functor $(-)[\a_s]$. The central vertical map is 
the isomorphism that is induced, via Lemma \ref{how the cohomology base changes lemma}\,(i) (and the result of Proposition \ref{what we need from Selmer complexes}), by the fixed isomorphism $\cR_{s + 1} [\a_s] \cong \cR_s$. Finally, the right hand vertical map is the projection induced by the isomorphisms $H^1_f (k_\q, \cT_{s + 1}) [\a_s] \cong H^1_f (k_\q, \cT_s)$ for each $\q \in \cQ_{j(s + 1)} \subseteq \cQ_{j(s)}$. In particular, the second square in the diagram commutes and so the indicated dashed map exists in order to make the whole diagram commutative. From the commutativity of the diagram, it then follows that the latter map is injective and hence, upon taking the dual of the first square in the diagram and using Lemma \ref{inj env lem}\,(iii), we obtain a commutative diagram of surjective maps of $\cR_{s+1}$-modules 
\begin{equation}\label{transition diagram}\xymatrix{  H^1_{\mathscr{F}_{s+1}} (k, \cT_{s+1})^\ast\otimes_{\cR_{s+1}}\cR_s\ar@{->>}[d]  \ar[r]^{\hskip0.1truein \cong} & (H^1_{\mathscr{F}_{s + 1}} (k, \cT_{s + 1}) [\a_s])^\ast \ar[r]^{\hskip0.3truein \cong} & H^1_{\mathscr{F}_s} (k, \cT_s)^\ast\ar@{->>}[d] \\
\sha_{s + 1}^\ast\otimes_{\cR_{s+1}}\cR_s \ar[r]^{\cong} & \big( \sha_{s + 1}[\a_s] \big)^\ast \ar@{->>}[r] & \sha_{s}^\ast.}\end{equation}
Now, since all modules here are finite, exactness is preserved upon passing to the inverse limit (over $s$) and so one obtains a surjective map of $\cR$-modules %
\[ {\varprojlim}_{s\in \N} H^1_{\mathscr{F}_{s}} (k, \cT_{s})^\ast \twoheadrightarrow M \coloneqq {\varprojlim}_s\sha_{s}^\ast,\]
in which the limits are defined with respect to the maps that are induced by the respective rows of (\ref{transition diagram}). In particular, since the upper row of (\ref{transition diagram})  is an isomorphism, Nakayama's Lemma implies that ${\varprojlim}_{s\in \N} H^1_{\mathscr{F}_{s}} (k, \cT_{s})^\ast$, and hence also $M$, is a finitely generated $\cR$-module. Thus, if we set $M_s \coloneqq M \otimes_\cR \cR_s$ for each $s \in \N$,  then the natural map $M \to {\varprojlim}_{s \in \N} M_s$ is an isomorphism. In addition, since the natural projection map $\kappa_s\: M \to \sha_s^\ast$ is surjective, setting $N_s \coloneqq \ker(\cR_s\otimes_\cR\kappa_s)$ gives a tautological short exact sequence of $\cR_s$-modules
\[ 0 \to N_s \to M_s \xrightarrow{\cR_s\otimes_\cR\kappa_s} \sha_{s}^\ast \to 0.\] 
Since each $N_s$ is finite, exactness is preserved when passing to the inverse limit over $s$ and so ${\varprojlim}_{s \in \N} N_s$ vanishes. Then, as $\mathbb{K}$ is finitely-presented as an $\cR$-module, one also has 
\[
{\varprojlim}_{s \in \N} (  \mathbb{K}\otimes_{\cR_s}N_s ) = {\varprojlim}_{s \in \N} (\mathbb{K}\otimes_{\cR}N_s) \cong \mathbb{K}\otimes_\cR  {\varprojlim}_{s \in \N} N_s   = (0).
\]
In particular, by applying the functor $\mathbb{K}\otimes_{\cR_s} (-)$ to the above short exact sequence and then taking inverse limits over $s$, we  obtain an exact sequence
\begin{cdiagram}[column sep=small]
    {\varprojlim}_{s \in \N} \Tor_1^{\cR_s} (  \mathbb{K},M_s) \arrow{r} & {\varprojlim}_{s \in \N} \Tor_1^{\cR_s} (\mathbb{K},\sha_s^\ast) \arrow{r} & {\varprojlim}_{s \in \N} ( \mathbb{K}\otimes_{\cR_s}N_s ) = (0),
\end{cdiagram}%
and hence, upon taking duals, a composite injective map 
\begin{equation*}\label{ci map}
{\varinjlim}_{s} \Tor_1^{\cR_{s}} ( \mathbb{K}, \sha_s^\ast )^\ast \cong  \big( {\varprojlim}_{s} \Tor_1^{\cR_s} (\mathbb{K},\sha_{s}^\ast) \big)^\vee \hookrightarrow \big ( {\varprojlim}_{s} \Tor_1^{\cR_s} ( \mathbb{K},M_s) \big)^\vee \cong \Tor_1^\cR (\mathbb{K},M)^\vee,
\end{equation*}
where the second isomorphism follows from Lemma \ref{Tor lemma}\,(ii) with $R$ taken to be $\cR$ and each $S_n$ to be $\mathbb{K}$. In addition, from Lemma \ref{new bound} below, the $\mathbb{K}$-dimension of the kernel of the direct limit $\Tor^{\cR_i}_1 (\mathbb{K}, \sha_i^\ast)^\ast \to {\varinjlim}_s\Tor^{\cR_j}_1 (\mathbb{K}, \sha_s^\ast)^\ast$ of the homomorphisms induced by Lemma \ref{Tor lemma}(i)  
is at most $\mathrm{rk}(k (\cT)_\infty/k)\cdot \dim_\mathbb{K}(\overline{\cT})$, where $\mathrm{rk}(k (\cT)_\infty/k)$ is the `rank' of the group $\gal{k (\cT)_\infty}{k}$ as specifed below. Taken together, these observations imply that 
\[ \dim_{\mathbb{K}}\bigl(\Tor_1^{\cR_i} (\mathbb{K},\sha_i^\ast)^\ast\bigr) \le \mathrm{rk}(k (\cT)_\infty/k)\cdot \dim_\mathbb{K}(\overline{\cT}) +  \dim_{\mathbb{K}}\bigl(\Tor_1^\cR (\mathbb{K},M)^\vee\bigr).\]
We now define the function $\Phi_{\mathscr{F}_k}\:\N_0\to \N_0$ by setting 
\[ \Phi_{\mathscr{F}_k}(m) \coloneqq m + \dim_\mathbb{K}\bigl(H^1_{\mathscr{F}_0} (k, \cTbar)\bigr) + \mathrm{rk}(k (\cT)_\infty/k)\cdot \dim_\mathbb{K}(\overline{\cT})+ \dim_{\mathbb{K}}\bigl(\Tor_1^\cR (\mathbb{K},M)^\vee\bigr).\]
Then this function is clearly increasing and, by combining the last displayed inequality with the isomorphism (\ref{ext-tor}), the inequality (\ref{first ineq}) and the exact sequence (\ref{deduced ses}), one checks that $\alpha_i(\fn) \le \Phi_{\mathscr{F}_k}(\nu(\fn))$ for every $\fn \in \cN_{j(i)}$, as required to prove (ii). 
\end{proof}

Before stating the next result we recall that, under Hypothesis \ref{new strategy hyps}\,(vii), $\gal{k(\cT)_\infty}{k}$ is a compact $p$-adic analytic group. In particular, \cite[Th.~9.38\,(ii)]{ddms} implies that the Sylow $p$-subgroups of $\gal{k(\cT)_\infty}{k}$ have a common finite rank (as pro-$p$ groups) and we denote this by $\mathrm{rk}(k (\cT)_\infty/k)$. 

\begin{lem}\label{new bound} Set $\sha_s \coloneqq \sha_{\mathscr{F}_s, j (s)} (\cT_s)$ for each $s \in \N$. Then, for each $i \in \N$, one has 
\[ \dim_{\mathbb{K}}\bigl( \ker \bigl( \Tor^{\cR_i}_1 (\mathbb{K}, \sha_i^\ast)^\ast \xrightarrow{\Theta_i} {\varinjlim}_j \Tor^{\cR_j}_1 (\mathbb{K}, \sha_j^\ast)^\ast \bigr)\bigr) \le \mathrm{rk}(k (\cT)_\infty/k)\cdot \dim_\mathbb{K}(\overline{\cT}),\]
where the map $\Theta_i$ is the direct limit of the homomorphisms induced by Lemma \ref{Tor lemma}.
\end{lem}

\begin{proof} For integers $s$ and $t$ with $t \ge s\ge i$ we set  
\[
\sha_{s,t} \coloneqq \sha_{\mathscr{F}_s, t} (\cT_s), 
\,\,\,\,\widetilde{\sha}_{s,t} \coloneqq \sha_{\cF_s, t} (\cT_s)\quad\text{and}\quad \fX_s\coloneqq \fX_{\cF_s, s} (\cT_s)
\]
(so that one has $\sha_s = \sha_{s,j(s)}$). Then $\sha_{s,s} \subseteq \sha_{s,t}$ and  $\widetilde\sha_{s,s} \subseteq \widetilde\sha_{s,t}$ and, by the same argument as in Lemma \ref{general neukirch interpretation lemma}\,(iii), the canonical surjective map $H^1_{\mathscr{F}_s}(k,\cT_s) \to H^1_{\cF_s}(k,\cT_s)$ induces an isomorphism $
    \sha_{s,t} / \sha_{s,s} \cong \widetilde\sha_{s,t} / \widetilde\sha_{s,s}. $ In addition, Lemma \ref{general neukirch interpretation lemma}\,(i) implies 
    $\widetilde\sha_{s,t}\subseteq H^1 (k_{t} (\cT_{t}) / k, \cT_s)$, and Lemma \ref{general neukirch interpretation lemma}\,(ii) shows that $\widetilde\sha_{s,t}\cap \fX_s = \widetilde\sha_{s,s}$. Hence, setting  $\Delta_{s,t} \coloneqq \gal{k_t (\cT_t)}{k_s (\cT_s)}$, there exists a composite injective map of $\cR_s$-modules 
    \[
    \sha_{s,t} / \sha_{s,s} \cong \widetilde\sha_{s,t} / \widetilde\sha_{s,s}\hookrightarrow H^1 (\Delta_{s,t}, \cT_s) = \Hom ( \Delta_{s, t}, \cT_s)
    \]
    and thus, upon taking duals, a surjective map of $\mathbb{K}$-modules
\begin{equation}\label{induced surj}
\Hom ( \Delta_{s,t}, \cT_s)^\ast \otimes_{\cR_s} \mathbb{K} \twoheadrightarrow ( \sha_{s,t} / \sha_s)^\ast \otimes_{\cR_s} \mathbb{K}.
\end{equation}
On the other hand, for each $s\ge i$, there exists an exact commutative diagram 
    \begin{cdiagram}[row sep=small, column sep=small]
0 \arrow{r} & \sha_{i,j(s)} \arrow[dashed]{d} \arrow{r} & H^1_{\mathscr{F}_i} (k, \cT_i) \arrow{r} \arrow{d}{\simeq} & {\bigoplus}_{\q \in \cQ_{j(s)}} H^1_f (k_\q, \cT_i) \arrow{d}{\simeq} \\ 
    0 \arrow{r} & \sha_{s}[\a_i] \arrow{r} & H^1_{\mathscr{F}_{s}} (k, \cT_{s}) [\a_i] \arrow{r} & {\bigoplus}_{\q \in \cQ_{j(s)}} H^1_f (k_\q, \cT_{s}) [\a_i]
\end{cdiagram}%
 analogous to that following (\ref{ext-tor}), and hence an isomorphism $\sha_{i,j(s)} \cong \sha_{s}[\a_i]$ of $\cR_i$-modules. This isomorphism combines with Lemma \ref{inj env lem}\,(iii) to induce an isomorphism  $\sha_s^\ast \otimes_{\cR_s} \cR_i \cong (\sha_{i,j(s)})^\ast$. The latter  isomorphism then allows us to apply Lemma \ref{Tor lemma}\,(i) with each $S_n$ taken to be $\mathbb{K}$ (so that the target module, and hence cokernel, of the associated map (\ref{2nd Tor base change map}) vanishes) to conclude that the induced map $\Tor^{\cR_s}_1 (\sha_s^\ast, \mathbb{K}) \to \Tor^{\cR_i}_1 ( (\sha_{i,j(s)})^\ast, \mathbb{K})$ is surjective. This in turn gives rise to an exact commutative diagram of the form
\begin{equation*}
\xymatrixcolsep{5mm}
\xymatrixrowsep{5mm}
\xymatrix{ \Tor^{\cR_i}_1 (\mathbb{K}, (\sha_{i,j(s)})^\ast)^\ast \ar@{^{(}->}[d] \ar[r] & \Tor^{\cR_i}_1 (\mathbb{K}, \sha_i^\ast)^\ast \ar[d] \ar[r] &  \mathbb{K}\otimes_{\cR_i} ( \sha_{i,j(s)} / \sha_i)^\ast \\
\Tor^{\cR_s}_1 (\mathbb{K}, \sha_s^\ast)^\ast \ar[r]^{\mathrm{id}} & \Tor^{\cR_s}_1 (\mathbb{K}, \sha_s^\ast)^\ast.}
\end{equation*}
By applying the Snake Lemma to this diagram and then taking direct limits over $s \ge i$, we obtain a composite injective homomorphism 
\begin{equation}\label{comp inj}
    \ker(  \Theta_i) \cong \varinjlim_{s \ge i}\ker \big( \Tor^{\cR_i}_1 (\mathbb{K}, \sha_i^\ast)^\ast \to \Tor^{\cR_s}_1 (\mathbb{K}, \sha_s^\ast)^\ast \big) \hookrightarrow \varinjlim_{s\ge i} \bigl(\mathbb{K}\otimes_{\cR_i}( \sha_{i,j(s)} / \sha_i)^\ast\bigr).
\end{equation}
In addition, the natural isomorphisms for each $s \ge i$    
\begin{align*}
\Hom ( \Delta_{i, j(s)}, \cT_i)^\ast \otimes_{\cR_i} \mathbb{K} 
& \cong \big( \Hom ( \Delta_{i, j(s)}, \cT_i) [\cM_i] \big)^\vee
\cong \Hom ( \Delta_{i, j(s)}, \cTbar)^\vee 
\end{align*}
combine to give an isomorphism 
\begin{align*}
{\varinjlim}_{s\ge i} \Hom ( \Delta_{i, j(s)}, \cT_i)^\ast \otimes_{\cR_i} \mathbb{K}  
& \cong ( {\varprojlim}_{s\ge i} \Hom ( \Delta_{i, j(s)}, \cTbar))^\vee
\cong \Hom_\mathrm{cont} ( \gal{k_\infty (\cT)}{k_i (\cT_i)}, \cTbar)^\vee 
\end{align*}
and hence also, in conjunction with (\ref{induced surj}), an induced surjective map 
\[ 
\Hom_\mathrm{cont} ( \gal{k_\infty (\cT)}{k_i (\cT_i)}, \cTbar)^\vee \twoheadrightarrow {\varinjlim}_{s\ge i} \bigl(\mathbb{K}\otimes_{\cR_i}( \sha_{i,j(s)} / \sha_i)^\ast\bigr). 
\]
Recalling  the injective map (\ref{comp inj}), we can therefore compute that  
\begin{align*}
\dim_\mathbb{K}\bigl(\ker(  \Theta_i)\bigr) \le&\, 
\dim_\mathbb{K}\bigl( {\varinjlim}_{s\ge i} \bigl(\mathbb{K}\otimes_{\cR_i}( \sha_{i,j(s)} / \sha_i)^\ast\bigr)\bigr)\\
\le&\, \dim_{\mathbb{K}}\bigl(\Hom_\mathrm{cont} ( \gal{k_\infty (\cT)}{k_i (\cT_i)}, \cTbar)^\vee\bigr)\\
=&\, \dim_{\mathbb{F}_p}\bigl(\Hom_\mathrm{cont}  ( \gal{k(\cT)_\infty}{k_i (\cT_i)},\mathbb{F}_p)\bigr)\cdot \dim_{\mathbb{K}}(\overline{\cT})\\
=&\, \dim_{\mathbb{F}_p}\bigl(\gal{k(\cT)_\infty}{k_i (\cT_i)}^{\mathrm{ab}}\otimes_{\ZZ_p}\mathbb{F}_p\bigr)\cdot \dim_{\mathbb{K}}(\overline{\cT})\\
=&\, \mathrm{d}\bigl(\gal{k(\cT)_\infty}{k_i(\cT_i)}\bigr)\cdot \dim_{\mathbb{K}}(\overline{\cT})\\
\le&\, \mathrm{rk}(k (\cT)_\infty/k)\cdot\dim_{\mathbb{K}}(\overline{\cT}).\end{align*}
Here we write $\mathrm{d}\bigl(\gal{k(\cT)_\infty}{k_i(\cT_i)}\bigr)$ for the minimal number of topological generators of the pro-$p$ group 
$\gal{k(\cT)_\infty}{k_i(\cT_i)}$, so that the third equality follows from a standard property of the Frattini subgroup 
(cf.\@ \cite[Prop.~1.9, Prop.~1.13]{ddms}). In addition, since $\gal{k(\cT)_\infty}{k_i(\cT_i)}$ is a closed subgroup of a Sylow $p$-subgroup of 
$\gal{k(\cT)_\infty}{k}$, the final inequality follows from \cite[Prop.~3.11, Def.~3.12]{ddms}.\end{proof}

\begin{remark}\label{omitting 7 remark} The proof of Lemma \ref{new bound} is the only point in the proof of Theorem \ref{new strategy main result} in which Hypothesis \ref{new strategy hyps}\,(vii) is used. In particular, if the Tate--Shafarevich group $\sha_{\mathscr{F}_s, j (s)} (\cT_s)$ vanishes for every $s \in \N$, then Lemma \ref{new bound} plays no role in the proof of Theorem \ref{bound for alpha} and so Hypothesis \ref{new strategy hyps}\,(vii) can be omitted from the statement of Theorem \ref{new strategy main result}. \end{remark} 

\subsubsection{Consequences for Kolyvagin systems}

We can now finally prove the key technical result concerning Kolyvagin systems that will be used in the next section to prove Theorem \ref{new strategy main result}. In order to state this result, for each $i \in \N$ we abbreviate the non-negative integer defined in Remark \ref{lambda independent} to  
\[ \lambda_i^\ast(1) \coloneqq \lambda_{\overline{F_i}}^\ast(1),\] 
where $\overline{F_i}$ is the Mazur--Rubin structure on $\overline{T}$ induced by $F_i = h(\mathscr{F}_i\otimes_{\cR_i}R_i)$. For each $i \in \N$ and  $d \in \N_0$ we then define a polynomial $Z_{i}(d;X)$ of degree $d$ in $\ZZ[X]$  by setting
\[ Z_{i}(d;X) \coloneqq 8\lambda_i^\ast(1)X^d + {\sum}_{j =0}^{j= d-1} X^j.\]

\begin{thm} \label{core vertices and injectivity replacement}
Assume Hypotheses \ref{new strategy hyps} and fix $t\in \N_0$ such that $t_\fF \coloneqq t + \bm{\chi}_\fF > 0$. Then, for every $i \in \N$, there exists a relative core vertex $\n_0 = \n_0(i)$ for $\cF_i$ that belongs to $\cN_{j (i)}$ and also has the following property. If $\kappa = (\kappa_{\fn})_{\fn \in \cN_i}$ is any system in $ \KS^{t_\fF} (\mathscr{F}_i)$ for which $\kappa_{\fn_0}=0$, then there exists an ideal $I$ of $\cR_i$ that satisfies both of the following conditions.
\begin{itemize}
    \item[(i)] The annihilator of $\kappa_1$ in $\cR_i$ contains $I \cdot \Fitt^{t}_{\cR_i} (X (\mathscr{F}_i))^{9\lambda^\ast_i (1)}$. 
\item[(ii)]
The $R_i$-module $J_i$ in Definition \ref{J def} is such that  
\[ \Fitt^0_{R_i} ( J_i)^{Z_i(\lambda_i^\ast(1);\,\alpha_i(1)+\lambda_i^\ast(1))} \subseteq R_i\cdot \varrho_i\bigl(I\bigr).\]
\end{itemize}
\end{thm}

\begin{proof} Our argument will show that the required properties are satisfied by any relative core vertex $\fn_0$ for $\cF_i$ that is constructed as in Lemma \ref{existence of core vertices replacement} with $j = j(i)$. Before verifying this, however, it is convenient to prove the existence of an ideal $I_1'$ of $\cR_i$ that satisfies both 
\begin{equation}\label{last step} \Fitt^0_{R_i} (J_i) \subseteq R_i \cdot \varrho_i(I_1') \,\,\,\text{ and }\,\,\,  (I_1'\cdot\Fitt^{t}_{\cR_i} (X (\mathscr{F}_i)))\cdot \bidual^{t_\fF}_{\cR_i} \sha_{\mathscr{F}_i, j(i)} (\cT_i) = (0).\end{equation}
To do this, we first note  Lemma \ref{standard fitting props}\,(v) directly implies that $\bidual^{t_\fF}_{\cR_i} \sha_{\mathscr{F}_i, j(i)} (\cT_i)$ is annihilated by $\Fitt^{t_\fF - 1}_{\cR_i} ( \sha_{\mathscr{F}_i, j(i)} (\cT_i))$. To study the latter ideal, we  use Lemma \ref{existence of core vertices replacement} to fix a modulus $\a$ in $\cN_{j(i)}$ such that 
$H^1_{F_i^\ast (\a)} (k, T_i^\ast (1))$, and hence also (by Remark \ref{core vertices equivalent def rk})  $H^1_{(\overline{F_{i}})^\ast (\a)} (k, \Tbar^\ast (1))$, vanishes. We then use Proposition \ref{cebotarev prop tweak} to fix a prime $\q \in \cQ_{j (i)}\setminus V(\a)$ such that the localisation map $H^1_{\overline{F_i} (\a)} (k, \Tbar) \to H^1_f (k_\q, \Tbar)$ is nonzero (and thus surjective). Given this, the global duality  exact sequence (\ref{global duality sequence4}) (with $\Phi = \tilde\cF = \overline{F_i}$) implies that  $H^1_{(\overline{F_{i, \q}})^\ast (\a)} (k, \Tbar^\ast (1)) = H^1_{(\overline{F_{i}})^\ast (\a)} (k, \Tbar^\ast (1))$. It follows that $H^1_{(\overline{F_{i, \q}})^\ast (\a)} (k, \Tbar^\ast (1))$ vanishes and hence, by Lemma \ref{what do replacement core vertices do}, that the first property in (\ref{last step}) is satisfied by the ideal  
\[ I_1' \coloneqq \Fitt^0_{\cR_i} ( H^1_{(\cF_{i,\q})^\ast (\a)} (k, \cT_i^\ast (1))^\ast).\] 
In addition, since $\sha_{\mathscr{F}_i,j(i)} (\cT_i)$ is, by its very definition, a submodule of $H^1_{\mathscr{F}_{i,\q} (\a)} (k, \cT_i)$, there exists a natural surjective homomorphism $H^1_{\mathscr{F}_{i,\q} (\a)} (k, \cT_i)^\ast \twoheadrightarrow \sha_{\mathscr{F}_i, j(i)} (\cT_i)^\ast$. Taken together with Lemma~\ref{standard fitting props}\,(i), this surjective map implies an inclusion
 \[
 \Fitt^{t_\fF- 1}_{\cR_i} ( H^1_{\mathscr{F}_{i,\q} (\a)} (k, \cT_i)^\ast) \subseteq \Fitt^{t_\fF - 1}_{\cR_i} ( \sha_{\mathscr{F}_i, j(i)} (\cT_i)^\ast).
 \]
 Given this,  the fact that the above ideal $I_1'$ also has the second property in (\ref{last step}) is a direct consequence of the inclusion       
$I_1' \cdot \Fitt^{t}_{\cR_i} (X (\mathscr{F}_i)) \subseteq \Fitt^{t_\fF- 1}_{\cR_i} ( H^1_{\fF_{i,\q} (\a)} (k, \cT_i)^\ast)$ that is proved in  (\ref{reduced inclusion}) (with $\n$ replaced by $\a$). 

Turning now to the proof of the claimed result, we fix a relative core vertex $\fn_0$ for $\cF_i$ that is constructed as in Lemma \ref{existence of core vertices replacement} with $j = j(i)$ (so that $\fn_0 \in \cN_{j(i)}$ and $\nu (\fn_0) = \lambda_i^\ast (1)$). We also assume to be given a Kolyvagin system $\kappa$ in $ \KS^{t_\fF} (\mathscr{F}_i)$ for which $\kappa_{\fn_0} = 0$. It is then enough for us to prove the following: for any modulus $\fn$ in $\cN_{j (i)}$ that satisfies 
\begin{equation} \label{the moduli we are talking about}
    \lambda_i^\ast (\fn) = \lambda_i^\ast (1) - \nu (\fn),
\end{equation}
there exists an ideal $I_\fn$ of $\cR_i$ that satisfies both  
\begin{equation}\label{first claim}\Fitt^0_{R_i} ( J_i)^{Z_i (\lambda_i^\ast(\fn))} 
\subseteq R_i \cdot \varrho_i(I_\fn) 
\,\,\text{ and }\,\, 
I_\fn \cdot \Fitt^{t}_{\cR_i} (X (\mathscr{F}_i))^{\lambda_i^\ast(\fn) +8 \lambda^\ast_i (1)}\cdot \kappa_\fn = \{0\},\end{equation}
where, for brevity, we have set 
\[ Z_i (\lambda_i^\ast(\fn)) \coloneqq Z_i (\lambda_i^\ast(\fn);\,\alpha_i(1)+\lambda_i^\ast(1)).\] 
Indeed, since $\fn = 1$ satisfies (\ref{the moduli we are talking about}), if this claim is valid, then we will obtain an ideal $I$ of $\cR_i$  with all of the required properties by simply setting $I \coloneqq I_1$. \\  
Now, to construct  ideals $I_\fn$ satisfying (\ref{first claim}), we shall use an induction on the non-negative integer $\lambda_i^\ast (\fn)$. We therefore first assume that $\lambda_i^\ast (\fn) = 0$, so that  the exponents $Z_i (\lambda_i^\ast(\fn))$ and $\lambda_i^\ast(\fn) +8 \lambda^\ast_i (1)$  in (\ref{first claim}) are both equal to $8\lambda_i^\ast(1)$. Then, in this case, Remark \ref{core vertices equivalent def rk} implies that $\fn$ is 
a relative core vertex for $\cF_i$ and (\ref{the moduli we are talking about}) implies $\nu (\fn) = \lambda_i^\ast (1)$. Upon combining Proposition \ref{cebotarev prop tweak} with $n=i$ and $j=j(i)$ with the argument of Proposition \ref{length of paths between core vertices} with $\fn_1$ and $\fn_2$ taken to be $\fn$ and $\fn_0$, we can therefore deduce that $\fn$ is connected to $\fn_0$ in $\cX^0_{j(i)}(i)$ by a path $U$ of length at most $8\lambda_i^\ast (1)$. Given this, and the assumed vanishing of $\kappa_{\fn_0}$, the results of Lemmas \ref{what do replacement core vertices do} and \ref{induction on paths claim lemma} directly imply that (\ref{first claim}) is validated by the ideal 
\[ I_\fn \coloneqq \prod_{ (\q, \m) \in U} \Fitt^0_{\cR_i} (H^1_{(\cF_{i, \q})^\ast (\m)} (k, \cT_i^\ast (1))^\ast).\] 
For the inductive step we now assume to be given a modulus $\fn$ in $\cN_{j (i)}$ that satisfies (\ref{the moduli we are talking about}) and is also such that both $\lambda_i^\ast (\fn)> 0$ and ideals $I_\m$ with the properties in (\ref{first claim}) are known to exist for all moduli  $\m$ in $\cN_{j (i)}$ that satisfy both (\ref{the moduli we are talking about}) and $\lambda_i^\ast (\m)  < \lambda_i^\ast (\fn)$.

At this point, we recall the non-negative integer $\alpha_i(\fn)$ defined in (\ref{alpha def}). If $\alpha_i(\fn) = 0$, then Lemma \ref{when is the Selmer equal to sha} implies $H^1_{\mathscr{F}_i (\fn)} (k, \cT_i) = \sha_{\mathscr{F}_i, j(i)} (\cT_i)$. Hence,  in this case, the ideal $I_\fn \coloneqq I_1'$ validates (\ref{first claim}) as a direct consequence of (\ref{last step}) and our assumption that $\lambda^\ast_i ( \fn) > 0$. In the remainder of the argument, we shall therefore assume that $\alpha_i(\fn)>0$.

To deal with this case, we note that $H^1_{\overline{F_i}^\ast (\fn)} (k, \Tbar^\ast (1))$ is non-trivial since  $\lambda_i^\ast (\fn) > 0$. Hence, since $\bm{\chi} (F_i,j(i)) \geq 0$ (by Hypothesis \ref{new strategy hyps}\,(vi)), Lemma \ref{core rank relation}\,(iii) implies $H^1_{\overline{F_i} (\fn)} (k, \Tbar) \neq \sha_{\overline{F_i},j(i)} (\Tbar)$.  In particular, Lemma \ref{when is the Selmer equal to sha} can be combined with Proposition \ref{cebotarev prop tweak}\,(ii) to deduce the existence of a subset $\{\q_l: l \in [\alpha_i (\fn)]\}$ of $\cQ_{j(i)}$ for which the localisation map
\begin{equation}\label{injective localisation}
H^1_{\mathscr{F}_i (\fn)} (k, \cT_i) / \sha_{\mathscr{F}_i, j(i)} (\cT_i) \to \bigoplus_{l \in [\alpha_i (\fn)]} H^1_f (k_{\q_l}, \cT_i)
\end{equation}
is injective and, simultaneously, for every $l \in [\alpha_i (\fn)]$, the localisation maps 
\[
H^1_{\overline{F_i} (\fn)} (k, \Tbar) \to H^1 (k_{\q_l}, \Tbar) 
\quad \text{ and } \quad 
H^{1}_{\overline{F_i}^\ast (\fn)} (k, \Tbar^\ast (1)) \to H^1 (k_{\q_l}, \Tbar^\ast (1))
\]
are both nonzero. For every index $l \in [\alpha_i (\fn)]$, one then has 
\[
\lambda_i^\ast (\fn \q_l) = \lambda_i^\ast (\fn) - 1 = \lambda_i^\ast (1) - \nu (\fn)  - 1= \lambda_i^\ast (1) - \nu (\fn \q_l), 
\]
where the first equality follows from \cite[Prop.\@ 5.7]{bss}, the second from (\ref{the moduli we are talking about}) and the last is clear. It follows that the modulus $\fn \q_l$ satisfies both (\ref{the moduli we are talking about}) and also $\lambda_i^\ast (\fn \q_l) < \lambda_i^\ast (\fn)$ and so the inductive hypothesis implies the existence of an ideal $I_{\fn \q_l}$ of $\cR_i$ that has the properties in (\ref{first claim}) after replacing $\fn$ by $\fn \q_l$. 

In particular, since $\lambda^\ast_i(\fn \q_l) = \lambda_i^\ast(\fn)-1$, we know that the ideal $I_{\fn \q_l}\cdot\Fitt^{t}_{\cR_i} (X (\mathscr{F}_i))^{(\lambda_i^\ast(\fn)+8 \lambda_i^\ast (1))-1}$ annihilates the element $\loc_{\q_l} (\kappa_{\fn \q_l})$ of $\bidual^{t_\fF - 1}_{\cR_i} H^1_{\mathscr{F}_{i, \q_l} (\fn)} (k, \cT_i) \subseteq \bidual^{t_\fF- 1}_{\cR_i} H^1_{\mathscr{F}_i (\fn)} (k, \cT_i)$. From the equality $\psi_{\q_l}^\fs (\kappa_{\fn}) = v_{\q_l} ( \kappa_{\fn \q_l})$ that follows from the defining relation of Kolyvagin systems, it therefore follows that $\psi_{\q_l}^\fs (\kappa_{\fn})$ is also annihilated by $I_{\fn \q_l}\cdot\Fitt^{t}_{\cR_i} (X (\mathscr{F}_i))^{(\lambda_i^\ast(\fn)+8 \lambda_i^\ast (1))-1}$. 

Next we note that injectivity of (\ref{injective localisation}) implies $\sha_{\mathscr{F}_i,j(i)} (\cT_i)$ is the kernel of the diagonal map 
\[
(\psi_{\q_l}^\fs)_{l \in [\alpha_i (\fn)]} \: H^1_{\mathscr{F}_i (\fn)} (k, \cT_i) \to \cR^{\oplus \alpha_i (\fn)}.
\]
From Lemma \ref{biduals lemma 1}\,(i), it then follows that $\bidual^{t_\fF}_{\cR_i} \sha_{\mathscr{F}_i,j(i)} (\cT_i)$ is the kernel of the map 
\[
\bidual^{t_\fF}_{\cR_i} H^1_{\mathscr{F}_i (\fn)} (k, \cT_i)   \xrightarrow{(\psi_{\q_l}^\fs)_{l}}
\bigoplus_{l \in [\alpha_i (\fn)]}  \bidual^{t_{\fF} - 1}_{\cR_i} H^1_{\mathscr{F}_i (\fn)} (k, \cT_i),  
\]
and so the above observations combine to imply an inclusion   
\begin{equation}\label{almost 1} ({\prod}_{l\in [\alpha_i (\fn)]} I_{\fn \q_l}) \cdot \Fitt^{t}_{\cR_i} (X (\mathscr{F}_i))^{(\lambda_i^\ast(\fn)+8 \lambda_i^\ast (1))-1} \cdot \kappa_\fn\subseteq \bidual^{t_\fF}_{\cR_i} \sha_{\mathscr{F}_i,j(i)} (\cT_i).\end{equation}
In addition, there are also inclusions
\begin{align}\nonumber 
   \Fitt^0_{R_i} ( J_i)^{Z_i (\lambda_i^\ast(\fn))} & =
\Fitt^0_{R_i} ( J_i)\cdot\big( \Fitt^0_{R_i} ( J_i)^{Z_i (\lambda_i^\ast(\fn \q_l))} \big)^{\alpha_i (1) + \lambda_i^\ast (1)} \\ \nonumber
& \subseteq 
\Fitt^0_{R_i} ( J_i)\cdot {\prod}_{l\in [\alpha_i (\fn)]}   \Fitt^0_{R_i} ( J_i)^{Z_i (\lambda_i^\ast(\fn \q_l))} \\ 
& \subseteq \Fitt^0_{R_i} ( J_i) \cdot \varrho_i\bigl({\prod}_{l\in [\alpha_i (\fn)]} I_{\fn \q_l}\bigr).
\label{almost 2} 
\end{align}
Here the equality is valid for every $l\in [\alpha_i(\fn)]$ and follows from the fact that $Z_i(d;X) =1 +Z_i(d-1;X)X$ for each $d>0$. In addition, the first inclusion is valid since  
\[ |[\alpha_i (\fn)]| = \alpha_i  (\fn) \leq \alpha_i (1) + \nu (\fn) = \alpha_i (1) + (\lambda_i^\ast (1) - \lambda_i^\ast(\fn)) < \alpha_i (1) + \lambda_i^\ast (1),\]
where the first inequality here is obtained via (repeated applications of) Theorem \ref{bound for alpha}\,(i) and the
equality is an immediate consequence of (\ref{the moduli we are talking about}). Finally, we note that the second inclusion in (\ref{almost 2}) follows from the induction hypothesis. \\
Upon combining the inclusions (\ref{almost 1}) and (\ref{almost 2}) with those of (\ref{last step}), it is now easily deduced that the product ideal
$I_\fn \coloneqq I_1'\cdot {\prod}_{l \in [\alpha_i (\fn)]} I_{\fn \q_l}$ satisfies both of the conditions in (\ref{first claim}), thereby verifying the inductive step. This therefore completes the proof of the claimed result.  \end{proof}

\begin{rk}\label{bound comments} We  make some observations about the exponents that occur in Theorem~\ref{core vertices and injectivity replacement}. We recall first that $\lambda^\ast_i (1)$ is defined to be $\dim_{\mathbb{k}}\bigl(H^1_{\overline{F_i}^\ast} (k, \Tbar^\vee (1))\bigr)$ and hence that  
\begin{equation}\label{N bound} 0 \le 9\lambda_i^\ast(1) \le N \coloneq  9\cdot{\mathrm{dim}}_{\mathbb{k}}\bigl(H^1 (\cO_{k, S (\mathscr{F}_k)}, \Tbar^\vee (1))\bigr).\end{equation}

Since this bound is independent of $i$ it implies, in particular, that the set of polynomials $\{ Z_j(\lambda_j^\ast(1):X): j \in \mathbb{N}\}$ is finite. In addition, the value of any polynomial $Z_j(\lambda_j^\ast(1):X)$ at a non-negative integer is a natural number and so we can set 
\[ Z \coloneqq {\mathrm{max}}\{ Z_j(\lambda_j^\ast(1):m): j \in \mathbb{N}, \, m \in \N_0, m \le \Phi_{\mathscr{F}_k}(N)\},\]
where $\Phi_{\mathscr{F}_k}$ is the $\N_0$-valued function that occurs in Theorem~\ref{bound for alpha}\,(ii). Then the latter result (with $\fn = 1$) combines with (\ref{N bound}) to imply that, for every $i\in \N$, one has 
\begin{equation}\label{Z bound} 0 < Z_i(\lambda_i^\ast(1);\,\alpha_i(1)+\lambda_i^\ast(1)) \le Z.\end{equation}
\end{rk}

\section{Stark systems and the proof of Theorem \ref{new strategy main result}}\label{rss section}

In this section, we fix a family $\fF = (\mathscr{F}_K)_{K \in \Omega}$ of Nekov\'a\v{r} systems as in \S\ref{somr section} and continue to write $\mathscr{F}$ and $\mathscr{F}_i$ for $i \in \N$  in place of $\mathscr{F}_k$ and $\mathscr{F}_{k,i}$.\\
We recall that the notion of `Stark systems'  was independently introduced by Mazur and Rubin in \cite{MazurRubin} and by Sano in \cite{Sano14}. To define a suitable variant for our theory we fix an ordering $\prec$ on $\Pi_k$ and use the following sign convention. Given a modulus $\fn \in \cN_i$, we label the primes in $V(\fn)$ as $\q_1, \dots, \q_{\nu (\fn)}$ in such a way that $\q_1 \prec \q_2 \prec \dots \prec \q_{\nu (\fn)}$. For any  set of elements $\{ a_\q\: \q \in V(\fn) \}$ indexed by $V(\fn)$, we then define the exterior product $\bigwedge_{\q \in V (\fn)} a_\q$ to be $\bigwedge_{i\in [\nu (\fn)]} a_{\q_i}$. Further, if a modulus $\m \in \cN_i$ is divisible by $\fn$, then we define a `sign'  $\mathrm{sgn} (\m, \fn) \in \{ \pm 1\}$ via the equality 
\[
\big( \wedge_{\q \in V(\m / \fn)} \q \big) \wedge \big ( \wedge_{\q \in V( \fn)} \q \big) = 
\mathrm{sgn} (\m, \fn) \cdot \wedge_{\q \in V( \m)} \q
\]
in the exterior algebra $\bigwedge^\ast_\Z \Z [\Pi_k]$.

\begin{definition} Fix $t \in \N_0$ and $i \in \N$. Then the $\cR_i$-module of `Stark systems' of rank $t$ for the Nekov\'a\v{r} structure $\mathscr{F}_i$ on $\cT_i$ is the inverse limit
    \[
    \SS^t (\mathscr{F}_i) \coloneqq {{\varprojlim}}_{\fn \in \cN_i} \bidual^{t + \nu (\fn)}_{\cR_i} H^1_{\mathscr{F}^\fn_i} (k, \cT_i).
    \]
Here the transition morphisms for $\fn \mid \m$ are the maps
    \begin{equation} \label{transition map Stark system}
    v_{\m, \fn} \coloneqq \mathrm{sgn} (\m, \fn) \cdot {\wedge}_{\q \in V(\m/\fn)} \hat{v}_\q \: \bidual^{t + \nu (\m)}_{\cR_i} H^1_{\mathscr{F}^\m_i} (k, \cT_i) \to \bidual^{t + \nu (\fn)}_{\cR_i} H^1_{\mathscr{F}^\fn_i} (k, \cT_i)
    \end{equation}
     that are obtained by applying Lemma \ref{biduals lemma 1}\,(ii) to the exact sequence (\ref{global duality sequence1}) with $\Phi= \mathscr{F}_i$. 
\end{definition}

The reader will find a detailed axiomatic treatment of Stark systems in the next section. For the moment, however, we shall only record an algebraic construction of such systems that will be used the proof of Theorem \ref{new strategy main result}. We observe that this construction is modelled on that of Sano and the second author in \cite[Th.\@ 3.4]{sbA}. 

To state the result, we use the free $\cR$-module $Y$, of rank $r_Y$, that is fixed at the beginning of \S\,\ref{somr section} and, for each $i \in \N$,  define an $\cR_i$-module $Y_i$ via the push-out diagram
\begin{equation*}
\begin{tikzcd}[column sep=small, row sep=small]
{\bigoplus}_{\q \in \Pi_k^\infty} H^0(k_\q,\cT^\vee(1))^\vee \arrow[twoheadrightarrow]{r} \arrow[twoheadrightarrow]{d} &  {\bigoplus}_{\q \in \Pi_k^\infty} H^0(k_\q,\cT_i^\ast(1))^\ast \arrow[dashed]{d}  \\
Y  \arrow[dashed]{r} & Y_i 
\end{tikzcd}
\end{equation*} 
where the horizontal map is the canonical surjective morphism that exists since (\ref{p=2 condition}) is satisfied (see the argument of Lemma \ref{limit Fitt X})). Then $Y_i$ is an $\cR_i$-module quotient of ${\bigoplus}_{\q \in \Pi_k^\infty} H^0(k_\q,\cT_i^\ast(1))^\ast$ that is free of rank $r_Y$ and there exists a natural isomorphism 
\begin{equation}\label{Y proj map}
Y \otimes_\cR \cR_i \xrightarrow{\simeq} Y_i.
\end{equation}
Moreover, via the map $\alpha_3$ in (\ref{new diagram}), we may regard $Y_i$ as a quotient of $X(\mathscr{F}_{i})$. 
We also note that, for all moduli $\fn$ and $\m$ in $\cN_i$ with $\fn \mid \m$, the exact triangle in the lower row of (\ref{modify diagram1})
(in which $\tilde Y, \tilde{\mathscr{F}}$ and $\cA$ correspond to $Y_i, \mathscr{F}_i$ and $\cT_i$) combines with the identification, for each prime $\q \in V(\m/\n)$, of 
$H^1_{\!/f}(k_\q,\cT_i)$ with $\cR_i$ that is used in the construction of the upper square of the diagram in Proposition \ref{new modified selmer}\,(iii) (and earlier in the  derivation of the duality sequence (\ref{global duality sequence1})) to induce an isomorphism of $\cR_i$-modules 
\begin{equation*}\label{transition maps in horizontal systems}
\Det_{\cR_i} ( C_{Y_i}(\mathscr{F}_i^\m)) \cong \Det_{\cR_i} ( C_{Y_i} (\mathscr{F}^\fn_i)) \otimes_{\cR_i} \bigotimes_{\q \in V(\m / \fn)}  \Det_{\cR_i} ( H^1_{\!/f}(k_\q,\cT_i)[0]) \cong  \Det_{\cR_i} ( C_{Y_i}(\mathscr{F}^\fn_i)).
\end{equation*}

We fix an (ordered) $\cR$-basis $b_{\bullet}$ of $Y$ and write $b_{\bullet,i}$ for the  (ordered) $\cR_i$-basis of $Y_i$ given by the image of $b_{\bullet}$ under (\ref{Y proj map}). Finally, to ensure compatibility of our constructions under change of $i$, we assume that  the map $\vartheta_{\mathscr{F}_i^\fn,Y_i} $ that occurs in Proposition \ref{new modified selmer}\,(iii) is defined with respect to the  basis $b_{\bullet,i}$. 

\begin{lemma} \label{algebraic stark systems prop} Set $r\coloneqq r_Y + \bm{\chi}_\cR(C(\mathscr{F}_k))$ and assume that $r >0$. Then, for each  $i\in \N$, the assignment $(a_\fn)_{\fn \in \cN_i} \mapsto (\vartheta_{\mathscr{F}_i^\fn,Y_i} ( a_\fn))_{\fn \in \cN_i}$ induces a well-defined homomorphism of $\cR_i$-modules
    \[
    {\varprojlim}_{\fn \in \cN_i} \Det_{\cR_i} (C_{Y_i}(\mathscr{F}_i^\fn)) \to \SS^r (\mathscr{F}_i), \quad 
    \]
in which the inverse limit is defined with respect to the morphisms specified above. 
\end{lemma}

\begin{proof} This follows directly from the commutativity of the upper square in the  diagram of Proposition \ref{new modified selmer}\,(iii). 
\end{proof}

In order to prove Theorem \ref{new strategy main result}, we now fix an Euler system $c$ in $\ES^r_{S_0} (\fF)$. For each $i \in \N$, we write $\kappa_i = (\kappa_{i, \fn})_{\fn \in \cN_i}$ for the Kolyvagin system in $\KS^r (\mathscr{F}_i)$ that is given by the Kolyvagin derivative of $c$ (cf.\@ Theorem \ref{kolyvagin derivative thm}). We recall, in particular, that 
\begin{equation}\label{ks=es}
\kappa_{i, 1} = \pi^r_{k, \mathscr{F}_i}  (c_k),\end{equation}
where $\pi^r_{k, \mathscr{F}_i}$ is the canonical map $\bidual^r_\cR H^1_\fF (k, \cT) \to \bidual^r_{\cR_i} H^1_{\mathscr{F}_i} (k, \cT_i)$ from (\ref{proj map def}). \\
Now, by applying Proposition \ref{Greither--Kurihara Lemma}\,(ii) to the isomorphism $X (\mathscr{F}_k) \cong {\varprojlim}_i X_S(\mathscr{F}_i)$ constructed in Lemma \ref{limit Fitt X}\,(ii), we deduce that the ideals $(\Fitt^{r_Y}_{\cR_i} ( X (\mathscr{F}_i)))_{i \in \N}$ form a projective system with respect to the natural maps $\cR_{i + 1} \to \cR_i$ and also that their limit is equal to
\[ \Fitt^{r_Y}_\cR ( X (\mathscr{F}_k)) = {\varprojlim}_{i \in \N} \Fitt^{r_Y}_{\cR_i} ( X (\mathscr{F}_i)) \subseteq \cR.\]

In particular, if we fix an element $x$ of $\Fitt^{r_Y}_\cR (X (\mathscr{F}_k))$ as in the statement of Theorem \ref{new strategy main result}, then its image under  the projection map $\cR \to \cR_i$ belongs to $\Fitt^{r_Y}_{\cR_i} ( X (\mathscr{F}_i))$. As a consequence, if we write $\fn_i$ for 
the relative core vertex for $\cF_i = h(\mathscr{F}_i)$ in $\cN_{j (i)}$ that is provided by Lemma \ref{existence of core vertices replacement} and also fix an element $z$ of the ideal $\Fitt^0_{\cR_i} (H^1_{\cF_i^\ast (\fn_i)} (k, \cT_i^\ast (1))^\ast)$, then the exact sequence in Proposition \ref{what we need from Selmer complexes}\,(iii) (with $\a = \fb = 1$) combines with Lemma \ref{standard fitting props}\,(iv) to imply that  
\[
z \cdot x \in  \Fitt^0_{\cR_i} (H^1_{\cF_i^\ast (\fn_i)} (k, \cT_i^\ast (1))^\ast) \cdot 
\Fitt^{r_Y}_\cR (X (\mathscr{F}_k)) \subseteq \Fitt^{r_Y}_\cR ( H^1 (C (\mathscr{F}_i (\fn_i)))).
\]
Now the first displayed sequence in Proposition \ref{new modified selmer}\,(ii) implies that the latter ideal is equal to $\Fitt^{0}_\cR ( H^1 (C_{Y_i}(\mathscr{F}_i (\fn_i))))$. Given this equality, the above containment therefore combines with the assertion in Proposition~\ref{new modified selmer}\,(iii) regarding annihilation of the cokernel of $\vartheta_{\mathscr{F}_i (\fn_i), Y_i}$ to imply  
the existence of an element $\fz'_{\fn_i}$ of $\Det_{\cR_i} ( C_{Y_i}(\mathscr{F}_i (\fn_i)))$ for which one has  
\[
(z \cdot x) \cdot \kappa_{i, \fn_i} = \vartheta_{\mathscr{F}_i (\fn_i), Y_i} (\fz'_{\fn_i}).
\]
From the commutativity of the lower square (and surjectivity of the map $\varphi_2$) in the diagram of Proposition~\ref{new modified selmer}\,(iii), we can then deduce the existence of an 
element $\fz_{\fn_i}$ of $\Det_{\cR_i} ( C_{Y_i}(\mathscr{F}_i^{\fn_i}))$ for which the element $\epsilon'_{\fn_i} \coloneqq \vartheta_{\mathscr{F}_i^{\fn_i}, Y_i} (\fz_{\fn_i})$ of 
 $\bidual^{r + \nu (\fn_i)}_{\cR_i} H^1_{\mathscr{F}_i^{\fn_i}}(k, \cT_i) $ satisfies   
 \[
 ( {\bigwedge}_{\q \in V(\fn_i)} \hat{\psi}_\q^\fs) (\epsilon'_{\fn_i}) = 
 ( {\bigwedge}_{\q \in V(\fn_i)} \hat{\psi}_\q^\fs) ( \vartheta_{\mathscr{F}_i^{\fn_i}, Y_i} (\fz_{\fn_i})) = \vartheta_{\mathscr{F}_i (\fn_i), Y_i} (\fz'_{\fn_i}) = 
 (z\cdot x) \cdot \kappa_{i, \fn_i}.
 \]
Further, since $\fz_{\fn_i}$ belongs to $\Det_{\cR_i} ( C_{Y_i} (\mathscr{F}_i^{\fn_i}))$, Lemma \ref{algebraic stark systems prop}  implies the existence of a Stark system $\epsilon_i = (\epsilon_{i, \fn})_{\fn \in \cN_i}$ in $\SS^r( \mathscr{F}_i)$ for which one has $\epsilon_{i, \fn_i} = \epsilon'_{\fn_i}$. \\
Next we note that a straightforward calculation, as in \cite[\S\,5.2]{bss}, proves the existence of a well-defined `regulator map' of $\cR_i$-modules 
\[
\mathrm{Reg}^r \: \SS^r ( \mathscr{F}_i) \to \KS^r ( \mathscr{F}_i), \quad 
\epsilon = (\epsilon_\fn)_{\fn \in \cN_i} \mapsto \big( ({\bigwedge}_{\q \in V(\fn)} \hat\psi_\q^\fs) (\epsilon_\fn)  \big)_{\fn \in \cN_i}.
\]
Then, by its very construction, the element 
\[ \mathrm{Reg}^r (\epsilon_i) - (x\cdot z) \cdot\kappa_i \in \KS^r ( \mathscr{F}_i)\] 
vanishes when evaluated at $\fn_i$. We now write $I_i$ for the ideal $I$ of $\cR_i$ that is constructed in Theorem \ref{core vertices and injectivity replacement} (with $t =r$). Then, by combining claim (i) of the latter result with the inequality (\ref{N bound}) we deduce that, for any element  $y$ of $\Fitt^{r_Y}_\cR (X (\mathscr{F}_k))$ that is fixed as in the statement of Theorem \ref{new strategy main result}, every element in the $\cR_i$-ideal $y^N\cdot I_i$ annihilates the value 
\[ 
(\mathrm{Reg}^r (\epsilon_i) - (x\cdot z)\cdot \kappa_i)_1 = \epsilon_{i, 1} - (x\cdot z)\cdot\kappa_{i, 1}\]
of $\mathrm{Reg}^r (\epsilon_i) - (x\cdot z)\cdot \kappa_i$ at $1$. For every element $a$ of $I_i$, one therefore has an equality 
\begin{equation}\label{intermediate eq}
a y^N (x\cdot z) \cdot \kappa_{i, 1} = a y^N \cdot \epsilon_{i, 1}.\end{equation}
On the other hand, the construction of Stark systems in Lemma \ref{algebraic stark systems prop} implies that 
\[ \epsilon_{i, 1} = \vartheta_{\mathscr{F}_i,Y_i} ( \fz''),\]
where $\fz''$ is the element of $\Det_{\cR_i} (C_{Y_i} (\mathscr{F}_i))$ that is given by the image of $\fz_{\fn_i}$ under the isomorphism $\varphi_1$ in the diagram of  Proposition~\ref{new modified selmer}\,(iii) with $(\m,\fn)$ taken to be $(\fn_i,1)$. From the equality (\ref{intermediate eq}), we can therefore derive a containment   
\begin{equation*} \label{nearly there}
az \cdot xy^N   \cdot \kappa_{i, 1} = y^N \cdot \vartheta_{\mathscr{F}_i,Y_i} ( a\cdot\fz'') \in 
y^N\cdot \vartheta_{\mathscr{F}_i,Y_i} ( \Det_{\cR_i} (C_{Y_i}(\mathscr{F}_i))).
\end{equation*}
Now, as the elements $a$ and $z$ vary, they generate the ideal 
\[I_i' \coloneqq I_i \cdot \Fitt^0_{\cR_i} ( H^1_{\cF^\ast_i (\fn_i)} (k, \cT_i^\ast (1))^\ast)\]
of $\cR_i$. From the above containment, we can therefore deduce that the following ideal of $\cR_i$ 
\[
\mathfrak{A}_i (x,y) \coloneqq \big \{w \in \cR_i \mid  w\cdot xy^N\cdot \kappa_{i, 1} \in y^N \cdot\vartheta_{\mathscr{F}_i,Y_i} ( \Det_{\cR_i} (C_{Y_i} (\mathscr{F}_i))) \big\}
\]
contains $I_i'$. In particular, since  Remark \ref{core vertex remark} combines with Theorem \ref{core vertices and injectivity replacement}\,(ii) and the inequality (\ref{Z bound}) to imply that $R_i \cdot \varrho_i (I_i')$ contains $\Fitt^0_{R_i} (J_i)^{Z+1}$, there is an inclusion 
\begin{equation}\label{key nclusion 0}
 \Fitt^0_{R_i} ( J_i)^{Z+1} \subseteq R_i \cdot \varrho_i\bigl( \mathfrak{A}_i(x,y)\bigr).
\end{equation}
We next claim that the natural maps $\cR_{i + 1} \to \cR_i$ send $\mathfrak{A}_{i +1} (x, y)$ to $\mathfrak{A}_i (x, y)$. 
To justify this, we  note  the isomorphism $C (\mathscr{F}_{i + 1}) \otimes^\mathbb{L}_{\cR_{i + 1}} \cR_i \cong C (\mathscr{F}_{i})$ in $D^{\mathrm{perf}}(\cR_i)$ given by  Proposition~\ref{what we need from Selmer complexes}\,(ii) and the obvious isomorphism of $\cR_i$-modules $Y_{i+1}\otimes_{\cR_{i+1}}\cR_i \cong Y_i$  combine with the exact triangle given by the first column in (\ref{modify diagram}) (with $\m = 1$) to induce an isomorphism in $D^{\mathrm{perf}}(\cR_i)$
\[ \tau_i\: C_{Y_{i+1}}(\mathscr{F}_{i + 1}) \otimes^\mathbb{L}_{\cR_{i + 1}} \cR_i \cong C_{Y_i} (\mathscr{F}_{i}).\]  
This isomorphism then combines with the explicit construction of the maps $\vartheta_{\mathscr{F}_i,Y_{i}}$ in Proposition~\ref{new modified selmer}\,(iii) (and our use of the compatible family of bases $(b_{\bullet,i})_i$) to give a canonical commutative diagram of $\cR_{i+1}$-modules 
\begin{cdiagram}
    \Det_{\cR_{i +1}} (C_{Y_{i+1}} (\mathscr{F}_{i + 1})) \arrow{rr}{\vartheta_{\mathscr{F}_{i+1},Y_{i+1}}} \arrow{d}{\Det_{\cR_i}(\tau_i)} & &\bidual^r_{\cR_{i+1}} H^1_{\mathscr{F}_{i + 1}} (k, \cT_{i + 1}) \arrow{d}{\pi^r_{i + 1/ i}} \\ 
    \Det_{\cR_i} (C_{Y_i}(\mathscr{F}_i)) \arrow{rr}{\vartheta_{\mathscr{F}_i,Y_i}} & & \bidual^r_{\cR_i} H^1_{\mathscr{F}_i} (k, \cT_i).
\end{cdiagram}%
Here the map $\pi^r_{i+1/i}$ is constructed (via Lemma \ref{det projection lemma}\,(ii)) in the same way as the projection map $\pi^r_{k,\mathscr{F}_i}$ in  (\ref{proj map def}). In particular, from the equality (\ref{ks=es}) one deduces that  
\[
\pi^r_{i + 1/ i} ( \kappa_{i + 1, 1}) = \pi^r_{i + 1/ i} ( \pi^r_{k, \mathscr{F}_{i + 1}} ( c_k)) = \pi^r_{k, \mathscr{F}_i} ( c_k) = \kappa_{i, 1}.
\]
Hence, if $a$ belongs to $\mathfrak{A}_{i+1}(x,y)$, so that 
\[ a\cdot xy^N\cdot \kappa_{i + 1, 1}\in y^N\cdot \vartheta_{\mathscr{F}_{i+1},Y_{i+1}} (\Det_{\cR_{i+1}} (C_{Y_{i+1}}(\mathscr{F}_{i+1}))),\] 
then the commutativity of the above diagram implies that 
\[ a\cdot xy^N\cdot \kappa_{i, 1}\in y^N\cdot \vartheta_{\mathscr{F}_i,Y_i} (\Det_{\cR_i} ( C_{Y_i}(\mathscr{F}_i))).\] 
It follows that the image of $a$ under the projection $\cR_{i + 1} \to \cR_i$ belongs to $\mathfrak{A}_i (x, y)$, as claimed.

In particular, since the fixed $\cR$-basis $b_\bullet$ of $Y$ induces an identification 
\[
\vartheta_{\mathscr{F}_k,Y} ( \Det_\cR (C (\mathscr{F}_k))) = {{\varprojlim}}_{i \in \N} \vartheta_{\mathscr{F}_i,Y_i}(\Det_{\cR_i} (C_{Y_i}(\mathscr{F}_i))),
\]
one obtains an inclusion
\[
{{\varprojlim}}_{i \in \N} \mathfrak{A}_i(x,y) \subseteq \mathfrak{A}(x,y) \coloneqq \big \{r \in \cR \mid r\cdot xy^N\cdot c_k \in y^N \cdot\vartheta_{\mathscr{F}_k,Y}( \Det_\cR (C (\mathscr{F}_k))) \big\}.
\]
Next we note that the compactness of $R$ implies  
\[
\varrho\bigl({{\varprojlim}}_{i\in \N} \mathfrak{A}_i(x,y)\bigr) = {{\varprojlim}}_{i\in \N} \varrho_i\bigl( \mathfrak{A}_i(x,y)\bigr) \subseteq R.
\]
Now, by (\ref{key nclusion 0}), the second limit in this display contains $({\varprojlim}_i \Fitt^0_{R_i} (J_i) )^{Z+1}$. In addition, since the canonical isomorphism (\ref{tor isomorphism}) combines with the definition of the modules $J_i$ to give an isomorphism $\Tor_1^{\cR} (X (\mathscr{F}_k), R) \cong  \varprojlim_{i\in \N} J_i$, the general result of Proposition \ref{Greither--Kurihara Lemma}\,(ii) implies that 
$\Fitt^0_R ( \Tor_1^{\cR} (X (\mathscr{F}_k), R))$ is equal to ${\varprojlim}_{i \in \N} \Fitt^0_{R_i} (J_i)$. These observations therefore combine to prove an inclusion   
\[
\Fitt^0_{R} ( \Tor_1^\cR (X (\mathscr{F}_k), R))^{Z+1}  \subseteq R \cdot \varrho\bigl( \mathfrak{A}(x,y)\bigr). 
\]
At this point, we fix a prime ideal $\p$ of $\cR$ that satisfies the conditions (i) and (ii) in Theorem \ref{new strategy main result}. Then condition  (ii) combines with the last displayed inclusion to imply that $R_{\p}\cdot \varrho (\mathfrak{A}(x,y))_{\p}$ contains the element $1$ of $R_{\p}$. In particular, if $\varrho_\p \: \cR_\p \to R_{\p}$ is surjective, as follows from the condition (i), then $R_{\p}\cdot\varrho (\mathfrak{A}(x,y))_{\p} = \varrho_\p ( \mathfrak{A}(x,y)_\p)$ and so the $\cR_\p$-ideal $\mathfrak{A}(x,y)_\p$ must contain
a preimage $\xi$ of $1$ under $\varrho_\p$. In addition, condition (i) also ensures that the map $\varrho_\p$ is non-zero, so that $\ker (\varrho_\p)\subseteq \p \cR_\p$  and $\varrho_\p (1) = 1$. It follows that $\xi-1$ belongs to $\p \cR_\p$ and so Nakayama's Lemma implies $\xi$ is a unit in $\cR_\p$. This in turn implies that  $\mathfrak{A}(x,y)_\fp = \cR_\fp$, as required to complete the proof of Theorem \ref{new strategy main result}. 
\medskip  \qed 

\begin{rk} \label{rk on avoiding more new hyps details}
    A close inspection of the above proof shows that the only contribution of the Kolyvagin derivative construction in Theorem \ref{kolyvagin derivative thm} is to guarantee the existence of a Kolyvagin system $\kappa_i \in \KS^r (\mathscr{F}_i)$ with the property that $\kappa_{i, 1} = \pi^r_{k, \mathscr{F}_i} (c_k)$ for every $i \in \N$ (cf.\@ (\ref{ks=es})). In particular, if the existence of such Kolyvagin systems is in some way already known, then one need not apply Theorem \ref{kolyvagin derivative thm} in the above argument, and hence can avoid assuming Hypothesis \ref{more new hyps} in Theorem \ref{new strategy main result} (cf.\@ Remark~\ref{rk on avoiding more new hyps}).
\end{rk}

\appendix 

\section{Abstract Stark systems}\label{appendix section}

In this appendix to Part I, we establish several useful general results in the theory of Stark systems. Whilst some of these results have already been applied in previous sections, with future applications in mind we have preferred to adopt a more axiomatic approach here. 

\subsection{Characteristic ideals}

We first recall a notion introduced by Greither \cite[\S\,5.2]{Greither04} and Sakamoto \cite[Def.\@ C.3]{Sakamoto20}. For this, we let $R$ be a $G_2$-ring.

\begin{definition} 
Let $Z$ be a finitely generated $R$-module and fix a surjective homomorphism of $R$-modules of the form $g \: R^{\oplus s} \twoheadrightarrow Z$ with $s \in \N$. Then Lemma~\ref{biduals lemma 1}\,(ii) implies that the tautological short exact sequence
\[ 0 \to \ker (g) \xrightarrow{(f_i)_{i \in [s]}} R^{\oplus s} \xrightarrow{g} Z \to 0 \] 
induces a homomorphism of $R$-modules  $\wedge_{i\in [s]} f_i \: \bidual^s_R \ker(g) \to \bidual^0_R(0) = R$. The `characteristic ideal' $\Char_R (Z) $ of $Z$ is defined to be the image of $\wedge_{i\in [s]} f_i$. 
\end{definition}

\begin{rk} \label{characteristic ideals remark}
    Sakamoto has proved the following general results on characteristic ideals. 
\begin{romanliste}
\item $\Char_R (Z)$ is independent of the choice of $g$ (see \cite[Rem.\@ C.5]{Sakamoto20}). 
\item $\Fitt^0_R (Z) \subseteq \Char_R (Z)$, with equality if $\rm{pd}_R(Z)\le 1$ (cf.\@ \cite[Prop.\@ C.7]{Sakamoto20}). 
\item $\Char_R (Z) = \Ann_R (Z)$ if $R$ is  Gorenstein of dimension zero  (see \cite[Prop.\@ 4.5]{Sakamoto22}). 
\end{romanliste}
\end{rk}

We will also require the following additional properties of characteristic ideals.

\begin{lem} \label{characteristic ideals lemma}
Assume $R$ is a $G_2$-ring, and let $M$ be a finitely generated $R$-module. 
\begin{romanliste}
\item $\Char_R (M) \subseteq \Ann_R (M)^{\ast \ast}$.
\item For any natural numbers $r \geq s > 0$, any exact sequence of $R$-modules of the form 
\begin{cdiagram}
    0 \arrow{r} & N \arrow{r} & M \arrow{rr}{(f_i)_{i \in [s]}} & & R^{\oplus s} \arrow{r} & Z \arrow{r} & 0
\end{cdiagram}%
and any $a$ in the image of the map $\bigwedge_{i \in [s]} f_i \: \bidual^r_R M \to \bidual^{r - s}_R N$ in (\ref{rank reduction map}), one has 
\[
\im (a) \coloneqq \big \{ \varphi ( a) \mid  \varphi \in \exprod^{r - s}_R N^\ast \big \} \subseteq \Char_R (Z).
\]
Moreover, if $M$ contains a free direct summand $F$ of rank $t$, then one has 
\[
\Fitt^{t}_R (M^\ast) \cdot \im (a) \subseteq \Fitt^0_R (Z)^{\ast \ast}. 
\]
\item For any submodule $N$ of $ M$, one has $\Char_R (M) \subseteq \Char_R (N) \cap \Char_R ( M / N)$. 
\end{romanliste}
\end{lem}

\begin{proof}
To prove (i), we fix a surjective map of $R$-modules $f \: R^{n} \twoheadrightarrow M$. We write $\{b_i\}_{i \in [n]}$ for the standard basis of $R^{n}$ and  $b_i^\ast \: R^n \to R$ for each $i\in [n]$ for the dual of $b_i$.
Setting $N \coloneqq 
\ker (f)$, the definition of the characteristic ideal directly implies that every element of $\mathrm{char}_R (M)$ can be written as $x = ( \wedge_{i \in [n]} b_i^\ast) (a)$ for some $a \in \bidual^r_R N$. 
Now, for each $i$ and $l$ in $[n]$,  
we may compute that
\[
b_l^\ast \big ( ( \wedge_{j \neq i} b_j^\ast) (a) \big)
= ( ( \wedge_{j \neq i} b_j^\ast) \wedge b_l^\ast ) (a) = 
(-1)^{n - i} \cdot \delta_{il} \cdot ( \wedge_{i \in [n]} b_i^\ast) (a) = (-1)^{n - i} \cdot \delta_{i l} \cdot x
\]
with the Kronecker symbol $\delta_{il}$. 
This shows that
\[
(-1)^{n - i} \cdot ( \wedge_{j \neq i} b_j^\ast) (a) =  x \cdot b_i\in N^{\ast \ast},\] 
where the containment is true since $\wedge_{i \neq j} b_i^\ast$ defines a map $
    \bidual^n_R N \to N^{\ast \ast}$.\\
To proceed, recall that the cokernel of the natural map $N \to N^{\ast \ast}$ identifies with $\Ext^2_R (A, R)$ for a specific $R$-module $A$ (cf.\@ \cite[Prop.~(5.4.9)\,(iii)]{NSW}) and hence, by assumption $(G_1)$ on $R$, is pseudo-null. It follows that $x b_i\in N_\p$ for every prime ideal $\p \in \Spec^{\leq 1}(R)$. Since the set $\{ f(b_i) \}_{i \in [n]}$  generates $M_\p$, this proves $x\in \Ann_{R_\p} (M_\p) = \Ann_R (M)_\p$.
Noting that reflexive ideals are uniquely determined by their localisations at prime ideals in $\Spec^{\leq 1}(R)$ (see Lemma \ref{ryotaro-useful-lem}\,(ii)), we can therefore deduce that $\Char_R (M) \subseteq \Ann_R (M)^{\ast \ast}$, as claimed. \\
To prove (ii), we fix $b \in \bidual^r_R M$ and set $\theta\coloneqq \wedge_{i\in [s]} f_i$. Then, for $\varphi \in \exprod^{r - s}_R M^\ast$, one has  
\[\varphi ( \theta(b))  = 
( \theta\wedge \varphi ) (b) 
= (-1)^{rs} \cdot ( \varphi \wedge \theta ) (b)  = (-1)^{rs} \cdot  \theta ( \varphi (b)).
\]
In particular, since $\theta$ is the composite $\bidual^{s}_R M \to \bidual^s_R (M / N) \to \exprod^s_R (R^{\oplus s}) \cong R$ (where the first map is induced by the natural projection), and $\Char_R (Z)$ is, by definition, the image of the second map in this composite, we have shown that   
\begin{equation} \label{evaluation-ideal0}
 \big \{ \varphi (  \theta(b)) \mid  \varphi \in \exprod^{r - s}_R M^\ast \big \} \subseteq \Char_R (Z).
\end{equation}
To proceed, we note $\Char_R (Z)$ is a reflexive ideal of $R$ (cf.\@ \cite[Prop.\@ C.10]{Sakamoto20}) and so uniquely determined by its localisations at 
$\p \in \Spec^{\leq 1}(R)$. To prove the first assertion of (ii), it is thus enough to demonstrate that, for each such $\p$, the localisations of the first set in (\ref{evaluation-ideal0}) and of $\im ( \theta(b))$ coincide. To do this, we use the commutative diagram
\begin{cdiagram}[row sep=small, column sep=small]
\exprod^{r - s}_R M^\ast \arrow{r} \arrow{d}{} & \exprod^{r - s}_R N^\ast \arrow{d} \\
\big ( \bidual^{r - s}_R M \big)^\ast \arrow{r} & \big ( \bidual^{r - s}_R N \big)^\ast.
\end{cdiagram}%
Here the upper horizontal map is induced by restriction, the vertical arrows are the respective canonical maps induced by (\ref{rank reduction map}) and the lower horizontal map is induced by the inclusion $\bidual^{r - s}_R N \hookrightarrow \bidual^{r - s}_R M$.  
In particular, since $ \im (\theta )\subseteq \bidual^{r - s}_R N$ (by Lemma \ref{biduals lemma 1}\,(ii)), the above diagram implies that, for $\varphi \in \exprod^{r - s}_R M^\ast$ and $a \in \im (\theta)$, one can compute $\varphi (a)$ as the value at $a$ of the image of $\varphi$ in $\exprod^{r - s}_R N^\ast$. Since, for  $x \in \bidual^{r - s}_R N = \Hom_R(\exprod^{r - s}_R N^\ast,R)$ and $\psi \in \exprod^{r - s}_R N^\ast$, the value $\psi (x) \in R$ is defined to be $x (\psi)$, the above discussion therefore  shows that 
\[  \big \{ \varphi (  \theta(b)) \mid  \varphi \in \exprod^{r - s}_R M^\ast \big \}\subseteq \{ \varphi ( \theta(b)) \mid  \varphi \in \exprod^{r - s}_R N^\ast \}.\]
To deduce the first assertion of (ii), it is therefore enough to show that the cokernel of this inclusion is pseudo-null. Now, for $\p\in \Spec^{\leq 1}(R)$, the group $\Ext^1_R (M / N, \; R)_\p = \Ext^2_{R_\p} (Z_\fp, R_\p)$ vanishes since, by assumption, $R$ satisfies ($G_1$) and so $R_\p$ has injective dimension one. It follows that $\coker( M^\ast \to N^\ast)_\p = \Ext^1_R (M / N, R)_\p$ vanishes, and hence that the $\p$-localisation of the cokernel of $\exprod^{r  - s}_R M^\ast \to \exprod^{r - s}_R N^\ast$ also vanishes. Every $\varphi \in \exprod^{r - s}_{R_\p} N_\p^\ast$ can therefore be lifted to an element $\psi \in \exprod^{r - s}_{R_\p} M^\ast_\p$, and so the claimed result follows from 
\begin{align*}
\{ \varphi ( \theta(b)) \mid \varphi \in \exprod^{r - s}_R N^\ast \}_\p =&\,
\{ \varphi ( \theta(b)) \mid  \varphi \in \exprod^{r-s}_{R_\p} N_\p^\ast\}\\
\subseteq&\, \{ \varphi ( \theta(b)) \mid  \varphi \in \exprod^{r-s}_{R_\p} M_\p^\ast\} = \{ \varphi ( \theta(b)) \mid \varphi \in \exprod^{r - s}_R M^\ast \}_\p.
\end{align*}

We now prove the second assertion of (ii). As $\Fitt^0_R (Z)^{\ast \ast}$ is a reflexive ideal of $R$, localising at primes in $\Spec^{\leq 1} (R)$ as in the discussion above shows it is enough to prove $\Fitt^0_R (Z)$ contains $ \Fitt^r_R (M^\ast) \cdot \{ \varphi ( \theta(b)) \mid \varphi \in \exprod^{r - s}_R M^\ast\}$.

By assumption, $M$ has a free direct summand $F$ of rank at least $s$. Then $F^\ast$ is a free direct summand of $M^\ast$ and the induced decomposition $\exprod^{s}_R M^\ast \cong \bigoplus_{i = 0}^{s} (\exprod^i_R F^\ast) \otimes ( \exprod^{s - i}_R M^\ast / F^\ast)$ implies the kernel of the surjection $\exprod^{s}_R M^\ast \to \exprod^{s}_R F^\ast$ identifies with $\bigoplus_{i = 0}^{s - 1} (\exprod^i_R F^\ast) \otimes  \exprod^{s - i}_R (M^\ast / F^\ast)$. 
Lemma \ref{standard fitting props}\,(v) now implies  $\Fitt^{s - i - 1}_R (M^\ast / F^\ast)$ annihilates $\exprod^{s - i}_R (M^\ast / F^\ast)$ for $i \in \{0, \dots, s - 1\}$. In particular, since $\Fitt^{ s - i - 1}_R (M^\ast / F^\ast)$ contains $\Fitt^0_R ( M^\ast / F^\ast) = \Fitt^{\mathrm{rk}_R( F)}_R (M^\ast)$, we deduce that $\Fitt^{\mathrm{rk}_R(F)}_R (M^\ast)$ annihilates the kernel of $\exprod^{s}_R M^\ast \to \exprod^{s}_R F^\ast$. Upon dualising, this implies $\Fitt^{\mathrm{rk}_R (F)}_R (M^\ast)$ annihilates the cokernel of the natural map $\exprod^{s}_R F = \bidual^r_R F \to \bidual^{s}_R M$.\\
Now let $a = \varphi (b)$ for some $\varphi \in \exprod^{r - s}_R M^\ast$. Then, for any $x$ in $\Fitt^{\mathrm{rk}_R(F)}_R (M^\ast)$, the element $x\cdot \varphi (b)$ of $\bidual^s_R M$ belongs to $\exprod^s_R F$.
It follows that 
\[ x\cdot \varphi (a) = x ( \varphi \circ \theta ) (b)
= (-1)^{(r - s) s} \cdot \theta ( \varphi (x b))\in \im ( \exprod^{s}_R M \stackrel{\theta}{\longrightarrow} R)  = 
\Fitt^0_R (Z),\] 
as required to complete the proof of (ii). \\
As for (iii), we note Sakamoto \cite[Th.\@ C.8]{Sakamoto20} proves  $\Char_R (M) \subseteq \Char_R (N)$. However, for the convenience of the reader, we provide a different argument for this inclusion in the course of this proof. For this, we fix surjections of the form $f \: R^{\oplus n} \to N$ and $g \: R^{\oplus m} \to M / N$. Then, since $R^{\oplus m}$ is projective, we can lift $g$ to a map $\widetilde{g} \: R^{\oplus m} \to M$. The resulting map $f \oplus \widetilde{g} \: R^{\oplus (n + m)} \to M$ is then surjective, and (writing $K_1$, $K_2$, and $K_3$ for the relevant kernels) we obtain an exact commutative diagram  
\begin{cdiagram}
    K_1 \arrow[hookrightarrow]{d}{(\varphi_i)_{i\in [n]}} \arrow[hookrightarrow]{r} & K_3 \arrow[twoheadrightarrow]{r} 
     \arrow[hookrightarrow]{d}{(\varphi_i)_{i\in [n+m]}} & K_2  \arrow[hookrightarrow]{d}{(\varphi_i)_{i\in [n+m]\setminus [n ]}} \\ 
     R^{\oplus n} \arrow[hookrightarrow]{r} \arrow[twoheadrightarrow]{d}{f} & R^{\oplus n} \oplus R^{\oplus m} \arrow[twoheadrightarrow]{r}\arrow[twoheadrightarrow]{d}{f \oplus \widetilde{g}}  & R^{\oplus m}  \arrow[twoheadrightarrow]{d}{g}  \\ 
     N \arrow[hookrightarrow]{r} & M \arrow[twoheadrightarrow]{r} & M / N. 
\end{cdiagram}%
From the above diagram we can deduce the exact sequence
\[
    0 \to K_1 \to K_3 \xrightarrow{(\varphi_i)_{i \in [n+m]\setminus [n]}} R^{\oplus m} \to M / N \to 0. 
\]
Now, if $a \in \bidual^{n + m}_R K_3$, then Lemma \ref{biduals lemma 1}\,(ii) (applied to the above exact sequence) implies that $(\wedge_{i \in [n+m]\setminus [n]} \varphi_i ) (a)\in \bidual^n_R K_1$, and so claim (ii) allows us to conclude that  
\begin{align*}
({\wedge}_{i \in [n+m]} \varphi_i ) ( a) & = (-1)^{n m} \big ( ({\wedge}_{i\in [n]} \varphi_i ) \circ ({\wedge}_{i \in [n+m]\setminus [n]} \varphi_i ) \big) (a) \big)\\  
& \in \im \big ( ({\wedge}_{i \in [n+m]\setminus [n]} \varphi_i ) (a) \big) \cap  ({\wedge}_{i \in [n]} \varphi_i ) \big ( \bidual^n_R K_1 \big) \\
& \subseteq \Char_R ( M / N) \cap  ({\wedge}_{i \in [n]} \varphi_i ) \big ( \bidual^n_R K_1 \big).
\end{align*}
The claimed inclusion is then valid since 
\begin{multline*}
\Char_R (M) = \im \big \{ \bidual^{n + m}_R K_3 \to \exprod^{n + m}_R R^{\oplus (n + m)} \cong R \big \}  
=  ({\wedge}_{i \in [n + m]} \varphi_i )\bigl( \bidual^{n + m}_R K_3 \bigr) \\
 \subseteq \Char_R ( M / N) \cap ({\wedge}_{i \in [n]} \varphi_i ) \bigl(\bidual^{n}_R K_1\bigr) 
= \Char_R ( M / N) \cap \Char_R (N). \qedhere
\end{multline*}
\end{proof}

\subsection{The general formalism} \label{abstract stark systems section}

Let $(\cQ, \prec)$ be a totally ordered set, and regard its power set $\cP \coloneqq \cP ( \cQ)$ as partially ordered with respect to inclusion. 
If $S \in \cP$ is a finite set, then by $\wedge_{v \in S} v$ we mean $v_1 \wedge \dots \wedge v_{|S|}$ using a labelling $S = \{ v_1, \dots, v_{|S|}\}$ with $v_1 \prec \dots \prec v_{|S|}$. If $S' \in \cP$ is a further finite set with $S \subseteq S'$, then we define a `sign'  $\mathrm{sgn} (S', S) \in \{ \pm 1\}$ via the equality 
\[
\big( \wedge_{v \in (S' \setminus S)} v \big) \wedge \big ( \wedge_{v \in S} v \big) = 
\mathrm{sgn} (S', S) \cdot \wedge_{v \in S} v
\]
in the exterior algebra $\bigwedge^\ast_\Z \Z [\cP]$.\\
We now suppose to be given a $G_2$-ring $R$ and an inductive system $\mathfrak{I} \coloneqq (M_S, \iota_{S, S'})_{S \in \cP}$ of finitely generated $R$-modules with the following properties:
\begin{itemize}
\item[(P$_1$)] The maps $\iota_{S, S'} \: M_S \to M_{S'}$ are injective (and so will be suppressed from the notation). 
    \item[(P$_2$)] For any element $v$ of $\cQ$, there exists a specified map of $R$-modules $f_v \: {\varinjlim}_S M_S \to R$, where the limit is over all elements $S$ of $\cP$ that contain $v$. 
    \item[(P$_3$)] For all $S$ and $S'$ in $\cP$ with $S \subseteq S'$, there exists a specified exact sequence of $R$-modules 
\begin{equation*}\label{third condition} \begin{tikzcd}
        0 \arrow{r} & M_S \arrow{r} & M_{S' } \arrow{rr}{\oplus_{v \in S' \setminus S} f_v} & & {\bigoplus}_{v \in S' \setminus S} R.  
    \end{tikzcd}\end{equation*}
\end{itemize}

\begin{definition}
For each $r\in \N_0$, an (abstract)  `Stark system of rank $r$' for the inductive system $\mathfrak{I}$ is an element of the inverse limit $R$-module    
\[ \SS^r (\mathfrak{I}) \coloneqq {\varprojlim}_{S \in \cP} \bidual^{r + |S|}_R M_S
\]
in which the transition morphism for each $S\subseteq S'$ is 
\[ \mathrm{sgn} (S', S) \cdot (\bigwedge_{v \in S' \setminus S} f_v ) \: \bidual^{r + |S'|}_R M_{S'} \to \bidual^{r + |S|}_R M_S,\]
as induced by the exact sequence in (P$_3$) (and Lemma \ref{biduals lemma 1}\,(ii)). Such a system is therefore an element 
$(c_S)_{S \in \cP}$ of ${{\prod}}_{S \in \cP} \bidual^{r + |S|}_R M_S$ with $\mathrm{sgn} (S', S) \cdot ({\bigwedge}_{v \in S' \setminus S} f_v ) ( c_{S'}) = c_S$ for $S \subseteq S'$.
\end{definition}

We now further assume the existence of an inductive system $\mathfrak{P} = \mathfrak{P}(\mathfrak{I}) \coloneqq (Z_S, \rho_{S, S'})_{S \in \cP}$ of (finitely generated) $R$-modules with the following properties: 
\begin{itemize}
\item[(P$_4$)] For  $v \in \cQ$ and $S\in \cP$, there exists a specified map of $R$-modules $g_{S,v} \: R \to Z_{S}$ with $\rho_{S' , S}\circ g_{S, v} =g_{S', v}$ 
for $S \subseteq S'$. 
\item[(P$_5$)] For $S$ and $S'$ in $\cP$ with $S \subseteq S'$ there exists a specified exact sequence of $R$-modules 
\begin{equation*} \label{abstract five term exact sequence}
    \begin{tikzcd}
        0 \arrow{r} & M_S \arrow{r} & M_{S' } \arrow{rr}{\oplus_{v \in S' \setminus S} f_v} & & \bigoplus_{v \in S' \setminus S} R
        \arrow{rr}{\sum_{v \in S' \setminus S} g_{S,v}}& & Z_{S} \arrow{r}{\rho_{S, S'}} & Z_{S'} \arrow{r} & 0. 
    \end{tikzcd}
    \end{equation*}
\end{itemize}

\begin{thm} \label{abstract stark systems thm}
Fix a $G_2$-ring $R$ and inductive systems of $R$-modules $\mathfrak{I}$ and $\mathfrak{P}$ as above. 
Then, for each $r\in \N_0$, the following claims are valid. 
\begin{romanliste}
\item For every $U \in \cP$ there exist elements $S$ of $\cP$ containing $U$ for which  $\Fitt^{r + |S|}_R (M_S^\ast)$ annihilates the kernel of the map 
    \[
    \SS^r(\mathfrak{I}) \to \bidual^{r + |S|}_R M_S, \]
that sends each element $(c_T)_{T \in \cP}$ to $c_S$.
    \item Fix $U \in \cP$ and denote by $\Upsilon_U$ image of the map 
$\sum_{v \in \cQ}  g_{U, v}$ from $\bigoplus_{v \in \cQ} R$ to $ Z_U$. For
 every system $(c_S)_{S \in \cP}\in \SS^r(\mathfrak{I})$ one has 
 \[
 \im (c_U) \subseteq \mathrm{char}_R (\Upsilon_U).
 \]
 In addition, if $S$ is a set as in (i) that contains $U$, $i \geq 0$ is an integer with $|S \setminus U | \geq i$, and $M_S$ contains a free direct summand of rank $t$, then 
 \[
 \Fitt^t_R (M_S^\ast) \cdot \big ( {\sum}_{V \subseteq S \setminus U, |V| = i} \im (c_{U \cup V}) \big) \subseteq \Fitt^i_R (\Upsilon_U).
 \]
\end{romanliste}
\end{thm}

\begin{proof} 
Since $R$ is Noetherian, the ascending chain of ideals 
\[
0 = \ker (\rho_{\emptyset, \emptyset}) \subseteq \cdots \subseteq \ker (\rho_{\emptyset, S}) \subseteq \cdots
\]
must terminate. There must therefore exist an element  $S$ of $\cP$ such that $\ker (\rho_{\emptyset, S}) = \ker (\rho_{\emptyset, S'})$ for all $S' \in \cP$ with $S \subseteq S'$.
 From the exact sequence in property  (P$_5$), it therefore follows that $\im(\sum_{v \in S} g_{\emptyset,v}) = \im(\sum_{v \in S'} g_{\emptyset,v})$ and hence that $\im(\sum_{v \in S} g_{\emptyset,v})$ is equal to the module $\Upsilon_\emptyset$ that occurs in (ii). 
 Note that we can also increase $S$ if necessary to ensure $S$ contains any given element of $\cP$.\\
We now claim that for any subset $U$ of $S$ one also has that 
$\im(\sum_{v \in S \setminus U} g_{U,v}) = \Upsilon_U$.
To justify this, we let $x$ be an element of $\Upsilon_U$ so that, by definition of $\Upsilon_U$, the element $x$ is contained in the image of $\sum_{v \in S'} g_{U, v}$ for some $S' \in \cP$. Let us now consider the diagram
\begin{cdiagram}
    \bigoplus_{v \in S' \cup U} R \arrow{rr}{\sum_{v \in S' \cup U} g_{\emptyset, v}} \arrow{d} & & Z_\emptyset \arrow{d}{\rho_{\emptyset, U}} \\
    {\bigoplus_{v \in S'}} R \arrow{rr}{\sum_{v \in S'} g_{U, v}} & & Z_U,
\end{cdiagram}%
where the vertical arrow on the left is the natural projection map $(y_v)_{v \in S' \cup U} \mapsto (y_v)_{v \in S'}$. To prove that this diagram commutes, we first observe that $\rho_{\emptyset, U} \circ g_{\emptyset, v} = 0$ for all $v \in U$ by property ($\mathrm{P}_5$)
so that the commutativity follows from assumption ($\mathrm{P}_4$). 
By the surjectivity of the arrow on the left, we then conclude that there is an element $y \in Z_\emptyset$ in the image of $\sum_{v \in S' \cup U} g_v$ such that $\rho_{\emptyset, U} (y) = x$. In particular, $y$ is in $\Upsilon_\emptyset = \im ( \sum_{v \in S} g_{\emptyset, v})$ and so $y = \sum_{v \in S} g_{\emptyset, v} (y_v)$ for suitable $(y_v)_{v \in S} \in R^{\oplus |S|}$. 
From $\rho_{\emptyset, U} \circ g_{\emptyset, v} = 0$ for all $v \in U$ and property $(\mathrm{P}_4)$ it now follows that
\[
x = \rho_{\emptyset, U} (y) = \rho_{\emptyset, U} ( \sum_{v \in S} g_{\emptyset, v} (y_v)) = 
 \sum_{v \in S \setminus U} g_{U, v} (y_v),
\]
as required to prove the desired containment $x \in \im (\sum_{v \in S \setminus U} g_{U, v})$.\\
From property $(\mathrm{P}_5)$ we further deduce that $\ker (\rho_{U, S}) = \Upsilon_S$
and that one has an exact sequence 
\begin{equation} \label{4 term exact sequence}
\begin{tikzcd}
 0 \arrow{r} & M_U \arrow{r} & M_{S } \arrow{rr}{\oplus_{v \in S \setminus U} f_v} & & \bigoplus_{v \in S \setminus U} R
        \arrow{rr}{\sum_{v\in S \setminus U} g_{U,v}}& & \Upsilon_U \arrow{r} & 0.
    \end{tikzcd}%
    \end{equation}
By applying Lemma \ref{characteristic ideals lemma}\,(ii) to this sequence and noting $c_U = \mathrm{sgn} (S, U) \cdot (\bigwedge_{v \in S \setminus U} f_v) (c_S)$
it now follows that $\im (c_U) \subseteq \mathrm{char}_R (\Upsilon_U)$, as claimed in (ii), and, in addition, that
\begin{equation} \label{fitting ideal inclusion for stark system}
    \Fitt^t_R (M_S^\ast) \cdot \im (c_U) \subseteq \Fitt^0_R (\Upsilon_U).
\end{equation}
We next let $i \geq 0$ be an integer such that $n \coloneqq |S \setminus U | \geq i $ and aim to prove by induction on $i$ that one has
\begin{equation} \label{Fitt one step up}
\Fitt^{i}_R ( \Upsilon_{U}) = {\sum}_{V \subseteq \cQ \setminus U, |V| = i} \Fitt^0_R ( \Upsilon_{U \cup V}).
\end{equation}
If $i = 0$ the claim is clear, and so we may assume that $i > 0$. 
 Using the free presentation of $\Upsilon_U$ provided by (\ref{4 term exact sequence}), we then see that $\Fitt^{i}_R ( \Upsilon_{U})$ is generated by the $(n - i) \times (n - i )$-minors of $\oplus_{u \in S \setminus U} f_u$. Any such minor is a $( (n - 1) - (i - 1))) \times ( (n  - 1) - (i - 1))$-minor of $\oplus_{u \in S \setminus (U \cup \{ v \}} f_u$ for some $v \in S \setminus U$. It follows that
\begin{align*}
\Fitt^{i}_R ( \Upsilon_{U}) & = {\sum}_{v \in S \setminus U} \Fitt^{i - 1}_R ( \Upsilon_{U \cup \{ v \}}) \\
& = {\sum}_{v \in S \setminus U} {\sum}_{V \subseteq \cQ \setminus (U \cup \{ v \}), |V| = i - 1} \Fitt^{i - 1}_R ( \Upsilon_{U \cup \{ v \} \cup V}) \\
& = {\sum}_{V \subseteq \cQ \setminus U, |V| = i} \Fitt^0_R ( \Upsilon_{U \cup V}),
\end{align*}
where the second equality holds by the induction hypothesis. This concludes the inductive step. \\
Combining (\ref{Fitt one step up}) with (\ref{fitting ideal inclusion for stark system}) we now obtain
\[
\Fitt^t_R (M_S^\ast) \cdot \big ( {\sum}_{V \subseteq (S \setminus U), |V| = i} \im (c_{U \cup V}) \big) \subseteq \Fitt^i_R ( \Upsilon_U),
\]
 as required to prove claim (ii).\\
To finish the proof, we show that any set $S$ chosen as above has the property stated in (i). To do this, we fix $S$ and a Stark system $(c_T)_T$ with $c_S = 0$, and need to show, for every $T \in \cP$, that $c_T$ is annihilated by $\Fitt^{r + |S|}_R (M_S^\ast)$. Now, after replacing $T$ by $T \cup S$, we can assume  $T$ contains $S$, and hence that $\Upsilon \coloneqq \Upsilon_\emptyset$ is equal to $\ker (\rho_{\emptyset, S}) = \ker (\rho_{\emptyset, T})$. From the diagram
\begin{cdiagram}[row sep=small]
    0 \arrow{r} & \Upsilon \arrow[equals]{d} \arrow{r} & Z_\emptyset \arrow{r}{\rho_{\emptyset, S}} \arrow[equals]{d} & Z_S \arrow{r} \arrow{d}{\rho_{S, T}} & 0 \\
    0 \arrow{r} & \Upsilon \arrow{r} & Z_\emptyset \arrow{r}{\rho_{\emptyset, T}} & Z_T \arrow{r} & 0
\end{cdiagram}%
we thus deduce $\rho_{S, T}$ is bijective. In this case, therefore, property (P$_5$) implies the existence of a split-exact sequence
\begin{cdiagram}
    0 \arrow{r} & M_S \arrow{r} & M_T \arrow{rr}{\oplus_{v \in T \setminus S} f_v} & & R^{|T \setminus S|} \arrow{r} & 0,
\end{cdiagram}%
and hence an isomorphism $M_T^\ast \cong M_S^\ast \oplus R^{|T \setminus S|}$. This in turn induces an isomorphism  
\[ \exprod^{r + |T|}_R M_T^\ast \cong {\bigoplus}_{i = 0}^{i=|T \setminus S|} ( \exprod^{r + |T| - i}_R M_S^\ast) \otimes_R (\exprod^i_R R^{|T \setminus S|})\] 
which, upon dualising, gives an exact sequence
\begin{cdiagram}
0 \arrow{r} & {\bigoplus}_{i = r + |S| + 1}^{r + |T|} \bidual^i_R M_S \arrow{r} & 
\bidual^{r + |T|}_R M_T \arrow{rr}{\wedge_{v \in T \setminus S} f_v} & &
\bidual^{r + |S|}_R M_S.
\end{cdiagram}%
In particular, since $c_T\in \ker(\wedge_{v \in T \setminus S} f_v)$, one has $c_T \in \bigoplus_{i = r + |S| + 1}^{r + |T|} \bidual^i_R M_S$. 
Thus, it suffices to note that, by Lemma \ref{standard fitting props}\,(v), the ideal 
 $\Fitt^{r + |S|}_R (M_S^\ast)$ annihilates the $R$-module $\exprod^{i}_R M_S^\ast$ for each $i$ with $r + |S| < i\le r + |T|$, and hence that $\Fitt^{r + |S|}_R (M_S^\ast)$ annihilates $\bigoplus_{i = r + |S| + 1}^{r + |T|} \bidual^i_R M_S$. 
\end{proof}

\begin{bsp}
Let $K / k$ be a finite abelian group with Galois group $G \coloneqq \gal{K}{k}$. Let $(S, V, T)$ be a Rubin datum for $K / k$, and suppose that the Rubin--Stark Conjecture holds for all data $(S \cup S', V \cup S', T)$ as $S'$ ranges over $\cP ( \cQ)$ for a set $\cQ$ of finite places of $k$ that is disjoint from $T$ and comprises only places that split completely in $K / k$. Then the collection $(\varepsilon^{V \cup S'}_{K / k, S \cup S', T})_{S' \in \cP}$ of Rubin--Stark elements is a Stark system in the above sense, with $R  = \Z [G]$,  $M_S = \cO_{K, S, T}^\times$ and $Z_S$ equal to the $(S_K, T_K)$-ray class group $\Cl_{K, S, T}$ of $K$. Hence, in this case, Theorem \ref{abstract stark systems thm}\,(ii)  implies 
$\im (\varepsilon^V_{K / k, S, T})$ is contained in the $\ZZ[G]$-characteristic ideal of the subgroup  of $\Cl_{K, S, T}$ generated by the classes of places $w$ in $\cQ_K$. 
(For a concrete application of this observation, see \cite[Th.\@ 1.1]{BullachMaciasCastillo}.)
\end{bsp}

\subsection{The Eagon--Northcott complex} \label{eagon northcott section}

In this section we review a family of complexes introduced by Eagon and Northcott in \cite{eagon--northcott} (see also \cite[App.~2, \S\,A2H]{eisenbud-syzygies}). These complexes
 generalise the classical construction of a Koszul complex (see \cite[\S\,A2.6]{eisenbud-comm-algebra} or \cite[App.\@ C]{Northcott} for unifying frameworks) and can be used to compute associated Fitting ideals. 
 
 To explain the  construction, 
 we fix a morphism $\phi \: F_0 \to F_1$ of finitely-generated free $R$-modules $F_0$ and $F_1$ of ranks $d$ and $d - r$, respectively.
 For each  $i \in \{ 0, \dots, r - 1 \}$ we write $\Sym^{r - i}_R F_1$ for the $(r - i)$-th symmetric power (over $R$) of $F_1$.  
Then the dual of the multiplication map $\mu_i \: F_1 \otimes_R (\Sym^{r - i - 1}_R F_1) \to \Sym^{r - i}_R F_1$ is a map $
\mu_i^\ast \: ( \Sym^{r - i}_R F_1)^\ast \to F_1^\ast \otimes_R ( \Sym^{r - i - 1}_R F_1)^\ast$ and we use this to define a composite \begin{align*}
\partial_i \: (\Sym^{r - i}_R F_1)^\ast \otimes_R \exprod^{d - i}_R F_0 
& \xrightarrow{\mu_i^\ast \otimes \phi^{(d - i)}}
\big( F_1^\ast \otimes_R ( \Sym^{r - i - 1}_R F_1)^\ast \big) \otimes_R \big(  F_1 \otimes_R \exprod^{d - i - 1}_R F_0 \big) \\
&  \xrightarrow{\phantom{\mu_i^\ast \otimes \phi^{(d - i)}}} 
(\Sym^{r - i - 1}_R F_1)^\ast \otimes_R \exprod^{d - i - 1}_R F_0.
\end{align*}
Here $\phi^{(d - i)} \: \exprod^{d - i}_R F_0 \to F_1 \otimes_R \exprod^{d - i - 1}_R F_0$ is 
defined by means of
\[
\phi^{(d- i)} (m_1 \wedge \dots \wedge m_{d - i}) = {\sum}_{j\in [d-i]} ( - 1)^{j + 1} \phi (m_j) \otimes m_1 \wedge \dots \wedge \widehat{m_j} \wedge \dots \wedge m_{d - i}, 
\]
where we write $\widehat{m_j}$ to mean omission of the element $m_j$, and the second arrow is induced by the canonical evaluation map $F_1^\ast \otimes_R F_1 \to R$. \\
Then it is proved in \cite[Th.\@ A2.10\,(a)]{eisenbud-comm-algebra} that the sequence
\begin{align*}
     (\Sym^{r}_R F_1)^\ast \otimes_R \exprod^{d}_R F_0 & \stackrel{\partial_0}{\longrightarrow} 
    (\Sym^{r - i - 1}_R F_1)^\ast \otimes_R \exprod^{d - i - 1}_R F_0 \stackrel{\partial_1}{\longrightarrow} \dots \\
   \dots  & \longrightarrow F_1^\ast \otimes_R \exprod^{d - r + 1}_R F_0 
   \xrightarrow{\partial_{r - 1}}
    \exprod^{d - r}_R F_0 \xrightarrow{\exprod_R^{d-r} \phi}
    \exprod^{d - r}_R F_1 
\end{align*}
constitutes a complex $C^\bullet_{\mathrm{EN}} (\phi)$ of (finitely generated, free) $R$-modules in which the final term $\exprod^{d-r}_R F_1$ is considered (by convention) as placed in degree zero.\\
The following well-known property of $C^\bullet_{\mathrm{EN}} (\phi)$ will play a crucial role in later arguments.

\begin{prop} \label{eagon northcott property}
    For every morphism $\phi \: F_0 \to F_1$ of free $R$-modules as above, and every integer $i$, the ideal $\Fitt^0_R (\coker (\phi))$ annihilates the $R$-module $H^i ( C^\bullet_\mathrm{EN} (\phi))$. 
\end{prop}

\begin{proof} This result is stated without proof in \cite[Th.\@ A2.59]{eisenbud-syzygies}. For completeness, we therefore present a proof. For this, we follow an  argument used by Buchsbaum and Rim \cite{BuchsbaumRim} to prove the analogous property for the Buchsbaum--Rim complex.

    Let $S \coloneqq R [X_{ij} \mid i \in [d - r], j \in [d]]$ be the polynomial ring in $d (d - r)$ commuting variables and consider the map $\Phi \: S^{\oplus d} \to S^{\oplus (d - r)}$ that is represented by the matrix $(X_{ij})_{ij}$. It then is a result of Northcott \cite{Northcott63} that $\Fitt^0_S ( \coker (\Phi))$ is an ideal of grade $d - (d - r) + 1 = r + 1$. It therefore follows from \cite[Th.\@ A2.10\,(c)]{eisenbud-comm-algebra} that $C^\bullet_\mathrm{EN} (\Phi)$ is acylic. (This is what is referred to as the `generic case' in loc.\@ cit.). In particular, $C^\bullet_\mathrm{EN} (\Phi)$ furnishes a projective resolution of
    \[
    H^0 ( C^\bullet_\mathrm{EN} (\Phi)) = \coker \big ( \exprod^{d - r}_S S^{\oplus d} \xrightarrow{\exprod^{d - r}_S \Phi} \exprod^{d - r}_S S^{\oplus (d - r)} \big ) \cong S / \Fitt^0_S (\coker (\Phi)).
    \]
    Let $\phi$ be represented by a matrix $A = (a_{ij})_{ij}$ and define $\overline{S}$ to be the quotient of $S$ by the ideal generated by $X_{ij} - a_{ij}$ for $ 1 \leq i \leq d - r$ and $1 \leq j \leq d$. Since Eagon--Northcott complexes behave well under base change, the isomorphism $\overline{S} \cong R$ that sends the class of $X_{ij}$ to $a_{ij}$ induces an isomorphism
    $
     C^\bullet_\mathrm{EN} (\phi) \cong C^\bullet_\mathrm{EN} (\Phi) \otimes^\mathbb{L}_S \overline{S}
    $.
    It follows that for every $i \geq 0$ one has
    \[
    H^i ( C^\bullet_\mathrm{EN} (\phi)) = \Tor_i^S ( S / \Fitt^0_S (\coker (\Phi)), \overline{S}).
    \]
    This description shows that $H^i ( C^\bullet_\mathrm{EN} (\phi))$
    is annihilated by the image of $\Fitt^0_S ( \coker (\Phi))$ under $\overline{S} \cong R$. By Lemma \ref{standard fitting props}\,(iv), this image is equal to
    $\Fitt^0_R (\overline{S} \otimes_S \coker (\Phi)) = \Fitt^0_R (\coker (\phi))$, whence the claimed annihilation result follows.
\end{proof}

\subsection{Determinants and biduals}

We next prove a useful technical result that relates the determinant of a perfect complex to an appropriate bidual of its cohomology.  

\begin{prop} \label{eagon-northcott-prop} Let $R$ be a $G_2$-ring and fix $d\in \N$ and $r \in \N_0$ with $r \le d$. Let $C^\bullet$ be a complex of $R$-modules $F_0 \stackrel{\phi}{\to} F_1$, in which the first term is 
a free $R$-module of rank $d$ that
is placed in degree zero and $F_1$ is a free $R$-module of rank $d - r$ that is placed in degree one. Then the following claims are valid.
\begin{itemize}
\item[(i)] The image of the canonical map 
\begin{align*}
\vartheta_{\phi} \: \Det_R (C^\bullet) \coloneqq \big( \exprod^d_R F_0 \big) \otimes_R \big( \exprod^{d - r}_R F_1^\ast \big) & \to \exprod^r_R F_0, \\
 a \otimes ( {\wedge}_{i \in [d-r]}f_i) & \mapsto (-1)^{r(d - r)} \cdot ({\wedge}_{i\in [d-r]}(f_i \circ \phi) \big) (a)
\end{align*}
is contained in $\bidual^r_R H^0 (C^\bullet)$,
and there is also an inclusion of ideals of $R$ 
\[
\Fitt^0_R (H^1 (C^\bullet)) \cdot \Ann_R ( \Ext^1_R ( R / \Fitt^0_R (H^1 (C^\bullet)), R)) \subseteq
\Ann_R \big ( \bidual^r_R H^0 (C^\bullet) / (\im (\vartheta_\phi)) \big).
\]
\item[(ii)] We have an inclusion of ideals of $R$
\[
\Fitt^0_R (H^1 (C^\bullet)) \subseteq \big \{ f (a) \mid a \in \im (\vartheta_{ \phi}), f \in \exprod^r_R H^0 (C^\bullet)^\ast \big \} 
\]
with pseudo-null cokernel.
\item[(iii)] There exists a canonical injective map
\[
\widetilde{\vartheta_\phi} \: \Fitt^0_R (H^1 (C^\bullet))^\ast \otimes_R \Det_R (C^\bullet) \to \bidual^r_R H^0 (C^\bullet),\] 
that sends each $\id_R \otimes a$ to $\vartheta_\phi(a)$.
The cokernel of $\widetilde{\vartheta_\phi}$ is $R$-torsion-free and annihilated by $\Fitt^0_R (H^1 (C^\bullet))$. In particular, $\widetilde{\vartheta_\phi}$ is bijective if $H^1 (C^\bullet)$ is $R$-torsion.
\item[(iv)] Suppose $D^\bullet$ is a perfect complex of $R$-modules that fits into an exact triangle of the form
\begin{equation} \label{triangle 1}
    \begin{tikzcd}
        C^\bullet \arrow{r} & D^\bullet \arrow{r} &  R^{n} [0]  \arrow{r} & \phantom{X}.
    \end{tikzcd}%
\end{equation}%
Then $D^\bullet$ admits a representative $R^{d+n} \stackrel{\psi}{\to} R^{d-r}$, in which the first term is placed in degree zero, and there is a commutative diagram
\begin{cdiagram}[row sep=small]
\Det_R (D^\bullet) \arrow{r}{\vartheta_{\psi}} \arrow{d}{\simeq} & \bidual_R^{r + n} H^0 (D^\bullet) \arrow{d}{(-1)^{rn} \wedge_{i \in [n]} f_i} \\
\Det_R (C^\bullet) \arrow{r}{\vartheta_{\phi}} & \bidual_R^{r} H^0 (C^\bullet).
\end{cdiagram}%
Here the first vertical map is the isomorphism
induced by (\ref{triangle 1}) and the canonical identification of $R$-modules $\Det_R (R^{n})\cong R$. 
In addition, the second vertical map is induced, via Lemma \ref{biduals lemma 1}\,(ii), by the map $(f_i)_{i \in [n]} \: H^0 (D^\bullet) \to R^{n}$ in the long exact cohomology sequence of (\ref{triangle 1}). 
\end{itemize}
\end{prop}

\begin{proof} It is convenient to first prove (ii) and (iii), before proving (i) and then (iv). \\ 
The proof of \cite[Prop.\@ A.2\,(ii)]{sbA}
shows that
\begin{equation} \label{evaluation-ideal}
\left\{ f (a) \mid a \in \im (\vartheta_\phi), \; f \in \exprod^{r}_R F_0^\ast \right\} = \Fitt^{0}_R (H^1 (C^\bullet)),
\end{equation}
and so (ii) follows from an argument that was also used in the proof of Lemma \ref{characteristic ideals lemma}\,(ii) (details can also be found in \cite[Lem.~2.7\,(c)]{BullachDaoud}). 

To prove (iii), we set  
\[ A\coloneqq \ker\bigl( \exprod^{d - r}_R F_0 \xrightarrow{\exprod^{d - r}_R \phi} \exprod^{d - r}_R F_1\bigr) \quad\text{and} \quad B \coloneqq \coker\bigl( F_1^\ast \otimes_R \exprod^{d - r + 1}_R F_0 \xrightarrow{\partial_{r - 1}} \exprod^{d - r}_R F_0\bigr).\]
Then, since the image of  $\exprod^{d - r}_R \phi$ is equal to $\Fitt^0_R ( H^1 ( C^\bullet)) \cdot \exprod^{d - r}_R F_1$, there exists an exact
commutative diagram
\begin{equation} \label{eagon northcott diagram}
\begin{tikzcd}[row sep=small]
 F_1^\ast \otimes_R \exprod^{d - r + 1}_R F_0 \arrow{r}{\partial_{r - 1}} \arrow{d} & \exprod^{d - r}_R F_0 \arrow[equals]{d} \arrow[twoheadrightarrow]{r} & B  \arrow{d}  \\
  A \arrow[twoheadrightarrow]{d} \arrow[hookrightarrow]{r} & \exprod^{d - r}_R F_0 \arrow[twoheadrightarrow]{r} & \Fitt^0_R ( H^1 (C^\bullet)) \otimes_R  \exprod^{d - r}_R F_1 \\
 H^{-1} (C^\bullet_{\mathrm{EN}}(\phi)),
\end{tikzcd}%
\end{equation}
in which $C^\bullet_{\mathrm{EN}}(\phi)$ is the Eagon--Northcott complex from \S\,\ref{eagon northcott section}.
By applying the Snake Lemma to this diagram, one then obtains an exact sequence
\begin{equation} \label{eagon northcott ses 1}
\begin{tikzcd}
 0 \arrow{r} & H^{-1} (C^\bullet_{\mathrm{EN}}(\phi)) \arrow{r} & B \arrow{r} & \Fitt^0_R ( H^1 (C^\bullet)) \otimes_R  \exprod^{d - r}_R F_1 \arrow{r} & 0.
\end{tikzcd}%
\end{equation}
Now, by dualising the upper row of (\ref{eagon northcott diagram}) and tensoring with $\exprod^d_R F_0 \cong R$, we obtain the upper row in the diagram
\begin{equation} \label{a big diagram}
\begin{tikzcd}[column sep=small]
    0 \arrow{r} & B^\ast \otimes_R \exprod^d_R F_0 \arrow{r} \arrow[dashed]{d} & \big( \exprod^{d - r}_R F_0 \big)^\ast \otimes_R \exprod^d_R F_0 \arrow{r} \arrow{d}{\simeq} & F_1 \otimes_R \big( \exprod^{d - r + 1}_R F_0 \big)^\ast \otimes_R \exprod^d_R F_0 \arrow{d}{\simeq} \\ 
    0 \arrow{r} & \bidual^r_R H^0 (C) \arrow{r} & \exprod^r_R F_0 \arrow{r} & F_1 \otimes_R \exprod^{r - 1}_R F_0.
\end{tikzcd}%
\end{equation}
Here the lower row is exact by Lemma \ref{biduals lemma 1}\,(i) and the vertical isomorphisms are the cases $i = r$ and $i = r-1$ of the isomorphism
\begin{equation} \label{det pairing}
\big( \exprod^{d - i}_R F_0 \big)^\ast \otimes_R \exprod^d_R F_0 \xrightarrow{\simeq} \exprod^i_R F_0, \quad 
\theta \otimes x \mapsto \theta (x).
\end{equation}
It follows that the dashed arrow in (\ref{a big diagram}) exists and is an isomorphism, which shows that $B^\ast \otimes_R \exprod^d_R F_0$ is isomorphic to $\bidual^r_R H^0 (C^\bullet)$.
Upon tensoring the dual of (\ref{eagon northcott ses 1}) with $\exprod^d_R F_0$ $(\cong R)$, we therefore obtain a canonical exact sequence
\begin{equation}\label{another ses}\Fitt^0_R ( H^1 (C^\bullet))^\ast \otimes_R  \big ( \exprod^{d - r}_R F_1 \big)^\ast 
 \otimes_R \big ( \exprod^d_R F_0 \big)  \hookrightarrow \bidual^r_R H^0 (C^\bullet) \to H^{-1} (C^\bullet_{\mathrm{EN}}(\phi))^\ast \otimes_R \big ( \exprod^d_R F_0 \big).\end{equation}
In connection with this exact sequence we note that the first map is defined as the composite of the dual map of $\exprod_R^{d - r} \phi$ and the 
isomorphism (\ref{det pairing}) for $i = r$. Hence, if we now define $\widetilde{\vartheta_\phi}$ to be this first map, then an explicit check shows $\widetilde{\vartheta_\phi} ( \id_R \otimes a) = \vartheta_\phi (a)$, as required by (iii). In addition, the above exact sequence identifies $\coker(\widetilde{\vartheta_\phi})$ with a submodule of $H^{-1} (C^\bullet_{\mathrm{EN}}(\phi))^\ast \otimes_R \big ( \exprod^d_R F_0 \big)$ and thereby implies $\coker(\widetilde{\vartheta_\phi})$ is $R$-torsion free, as claimed. Finally, we note Proposition \ref{eagon northcott property} implies $\Fitt^0_R (H^1 (C^\bullet)) = \Fitt^0_R (\coker (\phi))$ annihilates $H^i(C^\bullet_{\mathrm{EN}}(\phi))$ in every degree $i$. In particular, if $H^1 (C^\bullet)$ is $R$-torsion, then $\Fitt^0_R (H^1 (C^\bullet))$ contains a nonzero divisor and so $H^{-1} (C^\bullet_{\mathrm{EN}}(\phi))$ is also $R$-torsion. This completes the proof of (iii).\\
Returning now to consider (i), we note that its first claim is proved by the argument of Lemma~\ref{det projection lemma}. Next we observe that a slight adaptation of (\ref{eagon northcott ses 1}) leads to the exact sequence
\begin{cdiagram}
 0 \arrow{r} &  B/ H^{-1} (C^\bullet_{\mathrm{EN}}(\phi)) \arrow{r} & \exprod^{d - r}_R F_1 \arrow{r} & \big(R / \Fitt^0_R ( H^1 (C^\bullet))\big) \otimes_R \exprod^{d - r}_R F_1 \arrow{r} & 0.
\end{cdiagram}%
Upon dualising this sequence and then tensoring with $\exprod^d_R F_0$ ($\cong R$) we obtain the top row in an exact commutative diagram 
\begin{cdiagram}[column sep=small, row sep=small]
    \big ( \exprod^{d - r}_R F_1 \big)^\ast 
 \otimes_R \big ( \exprod^d_R F_0 \big) \arrow{r} \arrow[equals]{d} & \big( B/ H^{-1} (C^\bullet_{\mathrm{EN}}(\phi)) \big)^\ast \arrow{r} \arrow[hookrightarrow]{d} & \Ext^1_R (  (R / \Fitt^0_R ( H^1 (C^\bullet))), R) \\ 
 \big ( \exprod^{d - r}_R F_1 \big)^\ast 
 \otimes_R \big ( \exprod^d_R F_0 \big) \arrow{r}{\vartheta_\phi} & \bidual^r_R H^0 (C^\bullet) \arrow[twoheadrightarrow]{d} & \\ 
 & H^{-1} (C^\bullet_\mathrm{EN} (\phi))^\ast & 
\end{cdiagram}%
in which the central column is induced by combining (\ref{eagon northcott ses 1}) with (\ref{another ses}). Given this diagram, the displayed inclusion in (i) then follows directly from Proposition \ref{eagon northcott property}.\\
To prove (iv), we write $h = h^\bullet$ for the morphism $R^{n} [0] \to C^\bullet [-1]$ in $D(R)$ that appears in the exact triangle (\ref{triangle 1}). One then has $D^\bullet = \mathrm{cone} (h) [1]$, and so 
$D^\bullet$ is represented by the complex $P_0 \stackrel{\psi}{\to} P_1$, where $P_0 \coloneqq  R^{n} \oplus F_0$ occurs in degree $0$, $P_1 \coloneqq F_1$, and $\psi \: P_0 \to P_1$ is the map defined by sending each $(a,b) \in  R^{n} \oplus F_0$ to $h^1 (a) + \phi (b)$. This verifies the first part of (iv). \\
We next write $\{e_i : i\in [n]\}$ for the standard basis of $R^{n}$ $(\subseteq P_0$), and $e_i^\ast \: R^{n} \to R$ for the dual of $e_i$ for  $i \in [n]$. 
The isomorphism $\Det_R (D^\bullet) \cong \Det_R (C^\bullet)$ induced by the exact triangle (\ref{triangle 1}) can then be explicitly described as 
\begin{align*}
\Det_R (D^\bullet) = \big ( \exprod^{d + n}_R (R^{n} \oplus F_0) \big) \otimes_R \big ( \exprod^{d - r}_R F_0^\ast \big) 
& \to \big ( \exprod^{d}_R F_0  \big) \otimes_R \big ( \exprod^{d - r}_R F_0^\ast \big) = \Det_R (C^\bullet)\\
h \otimes g & \mapsto (-1)^{n \cdot d} \cdot ( {\wedge}_{i \in [n]} e_i^\ast  ) ( h) \otimes g.
\end{align*}
In addition, for each $g  = {\wedge}_{i \in [r]}g_i$, one has 
\begin{align*}
 (-1)^{n d} \cdot  \vartheta_\phi \big ( ( {\wedge}_{i \in [n]} e_i^\ast ) ( a) \otimes g \big) 
& =  (-1)^{n d} \cdot \big( ( {\wedge}_{i \in [d-r]} ( g_i \circ \phi ) ) \circ ( {\wedge}_{i \in [n]} (e_i^\ast \circ h^1) \big) (a) \\
& =  (-1)^{n d} \cdot \big(  ( {\wedge}_{i \in [n]} e_i^\ast ) \wedge ( {\wedge}_{i \in [d-r]} ( g_i \circ \phi ) ) \big) (a) \\
& = (-1)^{n  d  + n  (d - r)} \cdot  \big( ( {\wedge}_{i \in [d-r]} ( g_i \circ \phi ) ) \wedge ( {\wedge}_{i \in [n]} (e_i^\ast \circ h^1) \big) (a) \\
& = (-1)^{ n  r} \cdot \big(  ( {\wedge}_{i \in [n]} e_i^\ast ) \circ ( {\wedge}_{i \in [d-r]} ( g_i \circ \phi ) ) \big) (a) \\
& = (-1)^{r(d + n - r)} \cdot \big ( ( {\wedge}_{i \in [n]} f_i ) \circ \vartheta_\psi \big) (a).
\end{align*}
Here the last equality relies on the fact that $\vartheta_\psi (a)$ belongs to $\bidual^{ r + n}_R H^0 (D^\bullet)$ and hence that its image under ${\wedge}_{i \in [n]} e_i^\ast$ is equal to its image under ${\wedge}_{i \in [n]} f_i$ (since each $e_i$ induces the map $f_i \: H^0 (D^\bullet) \to R^{n}$ on cohomology). This completes the proof of (iv). \end{proof} 

\subsection{Algebraic Stark systems} \label{abstract algebraic stark systems section}

Let $\cQ$ be a set, and regard its power set $\cP \coloneqq \cP ( \cQ)$ as a partially ordered set with respect to inclusion.
Let $R$ be a $G_2$-ring and suppose to be given a family $(C^\bullet_{S})_{S \in \cP}$ of perfect complexes of $R$-modules that satisfies the following hypotheses.  

\begin{hyp} \label{hyp complexes}
Assume the following conditions are valid.
\begin{romanliste}
\item $C^\bullet_\emptyset$ admits a representative of the form $[F_0 \to F_1]$ with $F_0$ and $F_1$ finitely generated free $R$-modules.
(Here the first term is placed in degree zero.)
\item The integer $r \coloneqq \mathrm{rk}_R (F_0) - \mathrm{rk}_R (F_1)$ is non-negative. 
\item For every pair of sets $S, S' \in \cP$ with $S \subseteq S'$ there exists a specified exact triangle 
\begin{equation} \label{triangle 2}
\begin{tikzcd}
    C^\bullet_{S} \arrow{r}{\iota_{S, S'}} & C^\bullet_{S'} \arrow{rr}{\oplus_{v \in S' \setminus S} f_{S', v}} & & \bigoplus_{v \in S' \setminus S} R [0] \arrow{r} & \phantom{X}
\end{tikzcd}
\end{equation}
such that $f_{S', v} \circ \iota_{S, S'} = f_{S, v}$. We will therefore simply write $f_v$ instead of $f_{S', v}$. 
\end{romanliste}
\end{hyp}

\begin{rk} \label{rk how to define stark systems for a family of complexes}
 \
\begin{romanliste}
\item If $(C^\bullet_{S})_{S \in \cP}$ satisfies Hypothesis \ref{hyp complexes}, then we obtain a system $(M_S, \iota'_{S, S'})_{S \in \cP}$ as in \S\,\ref{abstract stark systems section} with $M_S \coloneqq H^0 (C^\bullet_S)$ and $\iota'_{S, S'}$ the map $H^0(\iota_{S,S'}) \: H^0 (C^\bullet_S) \to H^0 (C^\bullet_{S'})$.
In particular, we obtain a module $\SS^r ( \{ C^\bullet_S, \iota_{S, S'} \}_{S \in \cP} )$ of Stark systems  of rank $r$ associated to  $(C^\bullet_{S})_{S \in \cP}$.
\item Hypothesis \ref{hyp complexes} and Proposition \ref{eagon-northcott-prop}\,(iv) combine to imply that $C^\bullet_S$ has a representative $[P_0 \stackrel{\phi_S}{\to} P_1]$ in which $P_0$ and $P_1$ are free $R$-modules of ranks $\mathrm{rk}_R (F_0) + |S|$ and $\mathrm{rk}_R (F_1)$.
\end{romanliste}
\end{rk}

We can now state our main result on Stark systems in this setting.   

\begin{thm} \label{abstract algebraic stark systems thm}
Let $R$ be a $G_2$-ring and $(C^\bullet_{S})_{S \in \cP}$ a family of perfect complexes of $R$-modules that satisfies Hypothesis \ref{hyp complexes}. Write $\Upsilon$ for the image of the map $\sum g_v \: \bigoplus_{v \in \cQ} R \to H^1  (C^\bullet_\emptyset )$, where each $g_v \: R \to H^1 (C^\bullet_\emptyset )$ is the boundary map in the long exact cohomology sequence associated with (\ref{triangle 2}) for $S' = \{ v \}$ and $S  = \emptyset$.
Then the following claims are valid. 
\begin{romanliste}
\item There exist $S \in \cP$ with $ \im (\sum_{v \in S} g_v) = \Upsilon$, and for such $S$, the ideal  $\Fitt^0_{R} ( H^1 (C^\bullet_S)) = \Fitt^0_R ( H^1 (C^\bullet_\emptyset) / \Upsilon )$ annihilates the kernel of the `projection map'
\[
\mathrm{pr}_S \: \SS^r ( \{ C^\bullet_S, \iota_{S, S'} \}_{S \in \cP} ) \to \bidual^{r + |S|}_R H^0 (C^\bullet_S).
\]

\item There exists a well-defined homomorphism of $R$-modules 
\[
F \: {\varprojlim}_{S \in \cP} \Det_R (C^\bullet_S) \to \SS^r ( \{ C^\bullet_S, \iota_{S, S'} \}_{S \in \cP} ), \quad (a_S)_{S \in \cP} \mapsto (\vartheta_{\phi_S} (a_S))_{S \in \cP},
\]
where the transition maps on the left are the isomorphisms $\Det_R (C^\bullet_{S'}) \stackrel{\simeq}{\to} \Det_R ( C^\bullet_S)$ induced by  (\ref{triangle 2}). The kernel of $F$ is annihilated by $\Ann_R (\Fitt^0_R ( H^1 (C^\bullet_\emptyset) / \Upsilon )^\ast)$, and its cokernel is annihilated by the ideal $\Fitt^0_R ( H^1 (C^\bullet_\emptyset) / \Upsilon )^2 \cdot \mathfrak{A}$ with 
\[
\mathfrak{A} \coloneqq \Ann_R ( \Ext^1_R ( R / \Fitt^0_R (H^1 (C^\bullet_\emptyset) / \Upsilon), R)).
\]
\item For every system $c \in \SS^r (\{ C^\bullet_S, \iota_{S, S'} \}_{S \in \cP}  )$, and with $\mathfrak{A}$ as in (ii), there are inclusions 
\begin{align*} 
\im (c_\emptyset) & \subseteq \Char_R ( \Upsilon),\\
 \Fitt^0_R ( H^1 (C^\bullet_\emptyset) / \Upsilon) \cdot \mathfrak{A} \cdot  c_\emptyset & \subseteq \im (\vartheta_{\phi_\emptyset}),
 \\ 
 \Fitt^0_R ( H^1 (C^\bullet_\emptyset) / \Upsilon ) \cdot \mathfrak{A} \cdot \im (c_\emptyset) & \subseteq \Fitt^0_R ( H^1 (C^\bullet_\emptyset))^{\ast \ast}.
 \end{align*}
\end{romanliste}
\end{thm}

\begin{proof}
Theorem \ref{abstract stark systems thm}\,(i) implies both the existence of $S \in \cP$ with $\im (\sum_{v \in S} g_v)= \Upsilon$ and that $\iota (\Fitt^{r + |S|}_R ( H^0 (C^\bullet)^\ast))$ annihilates $\ker(\mathrm{pr}_S)$. To finish the proof of (i), it is thus enough to 
note that 
$\Fitt^0_{R} ( H^1 (C^\bullet_S)) \subseteq \iota (\Fitt^{r + |S|}_R ( H^0 (C^\bullet)^\ast))$ by Lemma \ref{comparing Fitting ideals of H0 and H1 lemma}.\\ 
The first part of (ii) follows from Proposition \ref{eagon-northcott-prop}\,(iv),
and the second part (regarding the kernel of $F$) follows from Proposition \ref{eagon-northcott-prop}\,(iii). To prove the final part we fix a set $S$ with $\Upsilon = \ker (H^1 (C^\bullet_\emptyset) \to H^1 (C^\bullet_S) )$, and hence $H^1 (C^\bullet_S) \cong H^1 (C^\bullet_\emptyset) / \Upsilon$ (and note that any $S'$ that contains $S$ has the same property).
We also fix $c = (c_X)_{X \in \cP} \in \SS^r (\{ C^\bullet_X, \iota_{X, X'} \}_{X \in \cP}  )$ and an element $x\in \Fitt^0 (H^1 (C^\bullet_S)) \cdot \mathfrak{A}$. Then Proposition \ref{eagon-northcott-prop}\,(i) implies there exists $z_S \in \Det_R (C^\bullet_S)$ with $\vartheta_{\phi_S} (z_S) = x c_S$. Since all transition maps in the inverse limit $ {\varprojlim}_{X \in \cP} \Det_R (C^\bullet_{X})$ are isomorphisms, we can then lift $z_S$ to an element $z = (z_{X})_{X}$ of this inverse limit. Now, $F (z)$ is a Stark system with $F(z)_S = \vartheta_{\phi_S} (z_S) = x c_S$, and  hence $F(z) - x c \in \ker(\mathrm{pr}_S)$. By  (i) we therefore have $y x z = y F (z)$ for any $y \in \Fitt^0_R (H^1 (C^\bullet_S))$. This implies that $\rm{cok}(F)$ is annihilated by $\Fitt^0_R (H^1 (C^\bullet_S))^2 \cdot \mathfrak{A}$, as required to complete the proof of (ii).\\
Since $F (z)$ and $c$ are Stark systems, we moreover have that
\begin{align*}
x \cdot c_\emptyset =&\,   \mathrm{sgn} (S, \emptyset) \cdot (\exprod_{v \in S} f_{S, v} ) ( x \cdot c_S)\\
=&\,   \mathrm{sgn} (S, \emptyset) \cdot (\exprod_{v \in S} f_{S, v} ) ( x \cdot F(z)_S)\\
=&\, x \cdot F(z)_\emptyset \\
=&\, x \cdot \vartheta_{\phi} (z_{\emptyset}),
\end{align*}
and this implies $x \cdot c_\emptyset\in \im(\vartheta_\phi)$, as required to prove the second inclusion in (iii). Finally, 
 the first inclusion  in (iii) is a special case of Theorem \ref{abstract stark systems thm}\,(ii) and the third follows directly upon applying Proposition \ref{eagon-northcott-prop}\,(ii) to the second inclusion.
\end{proof}

\unappendix

\part{Arithmetic applications}

In the remainder of the article we shall use the theory developed above to prove significantly stronger versions of a range of existing results concerning special value conjectures. 

\section{Relaxed Nekov\'a\v{r} structures and Kato's Conjecture} \label{relaxed Kato section}

In this section we explain how to apply Theorem \ref{new strategy main result} relative to suitably modified relaxed Nekov\'a\v{r} structures in order to study the `generalised Iwasawa main conjecture' formulated by Kato in \cite{Kato93a, Kato93b}. 

For brevity, in the sequel we shall refer to the latter conjecture as `Kato's Conjecture'.  

\subsection{The general strategy}\label{svc section}

To review Kato's Conjecture, we fix a smooth projective variety $X$ over $k$ for which the motive $M = h^i (X) (j)$ has coefficients in a semisimple commutative $\Q$-algebra $A$. Then, for each prime $\ell$, the $\ell$-adic realisation $V_\ell (M) \coloneqq H^i_{\mathrm{\acute et}} (X_{k^c}, \Q_\ell) (j)$ is a finitely generated module over the semisimple $\QQ_\ell$-algebra $A_\ell \coloneqq A \otimes_\Q \Q_\ell$ that is endowed with a continuous commuting action of $G_k$. It is conjectured that, for every place $\q$ of $k$ outside $\Pi_k^\infty \cup \Pi_k^\ell$, the characteristic polynomial 
\[ \Eul_\q (M, x) \coloneqq \mathrm{det}_{A_\ell} ( 1 - \Frob_\q^{-1} x \mid V_\ell (M)^\ast (1)^{I_\q}) \in A_\ell [x]\] 
belongs to $A [x]$ and is independent of $\ell$. Assuming this, for each finite set of places $S$ of $k$ that contains both $\Pi_k^\infty$ and all places at which $M$ has bad reduction, the $S$-truncated motivic $L$-series of $M$ is defined via the 
$(A \otimes_\Q \C)$-valued infinite product
\[
L_S (M, s) \coloneqq \prod_{\q \not \in S} \Eul_\q (M, \Frob_\q^{-1}\cdot \NN \q^{-s})^{-1}. 
\]
This product converges if the real part of $s$ is large-enough and we assume (as is conjectured) that it has a meromorphic continuation to $s  = 0$. It can then be shown that its leading term  $L^\ast_S (M, 0)$
  belongs to $(\R \otimes_\Q A)^\times$. To study this element one fixes a prime number $p$ together with a Gorenstein $\Z_p$-order $\mathcal{A}_p$ in $A_p$ for which there exists a (full) $G_k$-stable sublattice $T \coloneqq T_p (M)$ of $V_p (M)$ that is free as an $\cA_p$-module (so that  $\Hom_{\ZZ_p}(T,\ZZ_p(1))$ is also free as an $\cA_p$-module). We assume $S$ also contains $\Pi_k^p$ and use the complex 
 \[
 C_S(T) \coloneqq C(\mathscr{F}_{\mathrm{rel}}(T,S)) = \RHom_{\cA_p} ( \RGamma_\mathrm{c} (\cO_{k, S}, T), \cA_p) [-2]
 \]
constructed in Proposition \ref{construction complex}. Here $\mathscr{F}_{\mathrm{rel}}(T,S)$ is the relaxed Nekov\'a\v{r} structure specified in Example \ref{nss}\,(i) so that the second equality follows from Examples \ref{dual comments} and \ref{nss}(iii). 
 
 Then, modulo  standard conjectures, the theory developed by Bloch and Kato \cite{bloch-kato}, as interpreted and extended by Kato \cite{Kato93a, Kato93b} and Fontaine--Perrin-Riou \cite{FontainePerrinRiou94}, gives rise to a canonical
rank-one $A$-module $\Xi (M)$, called the `fundamental line' of $M$, as well as canonical
`period-regulator' isomorphisms of invertible $\R \otimes_\Q A$-modules,  resp.\@ $A_p$-modules
\[ 
\lambda_{M} \: \R \otimes_\Q A \xrightarrow{\simeq} \R \otimes_\Q \Xi (M), 
\quad\text{resp.} \,\,\, 
\vartheta_{M, S} \: \Q_p \otimes_{\Z_p} \Det_{\cA_p} (C_S (T)) \xrightarrow{\simeq} \Q_p \otimes_\Q \Xi (M).
\]

Writing $M^\ast(1)$ for the Kummer dual of $M$, we can now recall Kato's Conjecture  for $M$ relative to $\cA_p$. (This conjecture was first formulated in \cite{Kato93a} and is shown in \cite[\S\,2]{BurnsFlach03} to be equivalent, in the subsequent terminology of \cite{BurnsFlach01}, to the  `equivariant Tamagawa number conjecture' for the pair $(M,\mathcal{A}_p)$.) 

\begin{conj}[{$\TNC (M,\cA_p)$}] \label{etnc statement} 
One has 
\[ \lambda_{M}( L^\ast_S (M^\ast(1), 0))\in \Xi (M)\quad\text{and}\quad 
 \cA_p \cdot \lambda_{M}( L^\ast_S (M^\ast(1), 0)) = \vartheta_{M, S}\bigl(\Det_{\cA_p} (C_S (T))\bigr).
\]
\end{conj}

For each Galois extension $F$ of $k$ we now consider  the free $\cA_p[\cG_F]$-module 
\[ 
T_{F / k} \coloneqq \mathrm{Ind}_{G_k}^{G_F} (T)
= T \otimes_{\Z} \Z [\cG_F]
\]  
with $G_k$-action given by $\sigma \cdot (a \otimes b) \coloneqq (\sigma a) \otimes (x \overline{\sigma}^{-1})$ for $\overline{\sigma} \in \cG_F$ the restriction of $\sigma \in G_k$ to $F$.
For each place $v \in \Pi_k^\infty$ we also fix an extension of $v$ to $k^\mathrm{c}$ and write $G_v$ for the corresponding decomposition subgroup of $G_k$. Given an abelian extension $K$ of $k$, we denote the restriction of $\tau_v$ to $K$ by $\tau_{K, v}$ and define an 
$\cA_p [\langle \tau_{K, v} \rangle]$-module by setting 
\[
\cY_K (T, v) \coloneqq \begin{cases}
    T(-1)^{\tau_v = 1} \quad & \text{ if } \tau_{K, v} = \id_K, \\ 
    T(-1) & \text{ if } \tau_{K, v} \neq \id_K,
\end{cases}
\]
where $M^{\tau_v = \pm 1}$ denotes the maximal submodule of a $G_k$-module $M$ on which $\tau_v$ acts as $\pm 1$.
We may then define a subset of $\Pi_k^\infty$ as
\[
V (K) \coloneqq \{ v \in \Pi_k^\infty \mid \cY_K (T, v) \text{ is a nonzero free } \cA_p[\langle \tau_{K, v} \rangle]\text{-module} \}
\]
and, for every $v \in V (K)$, a natural number
\[
r_K (T, v) \coloneqq \mathrm{rk}_{\cA_p}(T(-1)^{\tau_v = 1}). 
\]
\begin{example}\label{Y example} In the following examples we assume $\cA_p$ is the valuation ring $\cO$ of a finite extension of $\QQ_p$. 
\begin{romanliste}
    \item If $T = \cO (1)$, then $V(K)$ is equal to the set of infinite places of $k$ that split in $K$ and $r_K (T, v) = 1$ for all $v \in V (K)$.  
    \item If $T$ is self-dual (that is, $T \cong T^\ast (1)$ as $G_k$-modules) and $p > 2$, then $V ( K) = \Pi_k^\infty$ and $r_K (T, v) =(\mathrm{rk}_\cO (T))/2$ for all $v \in V (K)$. To justify this, we write $\rho \: G_k \to \mathrm{GL}_{\mathrm{rk}_\cO (T)} (\cO)$ for a realisation of the representation $T$ and note self-duality implies $M \rho (\tau_v) M^{-1}=  - \rho (\tau_v^{-1})^\mathrm{t}$ for some $M \in \mathrm{GL}_{\mathrm{rk}_\cO T} (\cO)$. Now, $\tau_v = \tau_v^{-1}$ since $\tau_v^2 = \id$ and so we deduce that $\rho (\tau_v)$ has vanishing trace. This implies that $1$ and $-1$ appear with equal multiplicity among the eigenvalues of $\rho (\tau_v)$, and hence that $V ( K) = \Pi_k^\infty$. Indeed, if $v \in \Pi_k^\infty$ splits in $K$, then $v \in V ( K)$ if and only if $1$ is an eigenvalue of $\rho (\tau_v)$. If $v$ does not split, on the other hand, then $v \in V ( K)$ if and only if $\mathrm{rk}_\cO (T(-1)^{\tau_v = 1}) = \mathrm{rk}_\cO (T(-1)^{\tau_v = -1})$. 
\end{romanliste}     
\end{example}

Assume the pair $(k,p)$ satisfies (\ref{p=2 condition}). Then, with $H_v$ denoting the fixed field of $\tau_{K, v}$ in $K$, the $\cA_p [\cG_K]$-module 
\begin{equation} \label{definition of Y_K (T)}
\cY_{K} (T) \coloneqq \bigoplus_{v \in V(K)} \mathrm{Ind}^{G_k}_{G_{H_v}} (\cY_K (T, v))
\end{equation}
is  a quotient of $Y_{k, \Pi_k^\infty} (T_{K / k}) \coloneqq \bigoplus_{v \in \Pi_k^\infty} H^0 (k_v, T_{K / k}^\vee (1))^\vee$ that is free of rank  
\begin{equation} \label{r_K (T) definition}
r_K (T) \coloneqq {\sum}_{v \in V (K)} r_K (T, v).
\end{equation} 

We set 
\[ e_{K,S,T}\coloneqq \sum e \in A[\cG_K],\]
where $e$ runs over all primitive idempotents of $A[\cG_K]$ that annihilate each of $H^0 (\cO_{K, S}, V_p (M))$, $H^2 (\cO_{K, S}, V_p (M))$ and the kernel of the projection map $Y_{k, \Pi_k^\infty} (T_{K / k}) \to \cY_{K} (T)$. We also write $\cQ_K$ and $\cQ_{K,S,T}$ for the semisimple $\CC_p$-algebras $(\CC_p\otimes_\QQ A)[\cG_K]$ and $\cQ_Ke_{K,S,T}$. Then the descriptions in  Proposition \ref{construction complex}\,(ii) combine with Proposition \ref{construction complex}\,(i) to imply the following: The $\cQ_{K,S,T}$-module 
\[ e_{K,S,T}(\CC_p\cdot H^1(\cO_{k, S(K)},T_{K/k})) = e_{K,S,T}(\CC_p\cdot H^1(\cO_{K, S(K)},T))\] 
is free of rank $r_K(T)$ and any choice of $\cA_p[\cG_K]$-basis $b_{K} = \{b_{K,i}\}_{i \in [r_K(T)]}$ of $\cY_K(T)$ induces a composite homomorphism of $\cQ_K$-modules
\begin{align} \label{definition of finite level theta map}
    \Theta_{b_{K}} \: \CC_p\cdot\Det_{\cA_p[\cG_K]} (C_{S(K)} (T_{K/ k})) & \stackrel{\simeq}{\longrightarrow} \Det_{\cQ_K} ( \CC_p\cdot H^1(\cO_{k, S(K)}, T_{K/k})))\notag \\
    & \qquad \otimes_{\cQ_K} \Det_{\cQ_K} ( \CC_p\cdot H^1 ( C_{S(K)} (T_{K / k})))^{-1}\notag  \\ 
    & \longrightarrow e_{K,S,T}\bigl(\CC_p\cdot\exprod^{r_K(T)}_{\cA_p[\cG_K]} H^1 (\cO_{K, S(K)}, T)\bigr).
    \end{align}
Here the first map is the canonical `passage-to-cohomology map' and the second  is obtained by multiplying by $e_{K,S,T}$ and then evaluating elements of %
\[ e_{K,S,T}\bigl(\Det_{\cQ_K} ( \CC_p\cdot H^1 ( C_{S(K)} (T_{K / k})))^{-1}\bigr)= e_{K,S,T}\bigl(\CC_p\cdot\exprod^r_{\cA_p[\cG_K]}\cY_{K} (T)\bigr)^{-1}\] 
at $e_{K,S,T}{\bigwedge}_{i \in [r_K(T)]}b_{K,i}$. We write $M_K$ for the motive $M\otimes_\QQ h^0(\rm{Spec}(K))$, regarded as defined over $k$ and with coefficients $A[\cG_K]$. Then, upon fixing an isomorphism $\C_p \cong \C$, we may define   
\[ \eta_{b_K} \coloneqq \Theta_{b_K}\bigl((\vartheta_{M_K, S(K)}^{-1}\circ \lambda_{M_K})( L^\ast_{S(K)} (M_K^\ast(1), 0))\bigr) \in \C_p \cdot \exprod^d_{\cA_p[\cG_K]} H^1 (\cO_{K, S(K)}, T).\]

As in \S\,\ref{higher-rank euler systems definitions sections}, we next fix an abelian pro-$p$ extension $\cK$ of $k$ in which all places in $\Pi_k^\infty$ split completely, write $\Omega = \Omega(\cK)$ for the collection of finite extensions of $k$ inside $\cK$ and consider the family of relaxed Nekov\'a\v{r} structures
\[ \fF_{\mathrm{rel}} = (\mathscr{F}_{\mathrm{rel}}(T_{K/k},S(K)))_{K \in \Omega}.\] 

Then, assuming the rank $r = r_K(T)$ defined above to be independent of the choice of field $K$ in $\Omega$, one can aim to use Theorem \ref{new strategy main result} to study $\TNC(M,\cA_p)$ in the following way. 

\filbreak
\begin{strategy}\label{gen strat}\
\begin{itemize}
\item[(i)] Identify a compatible choice of bases $\{ b_K\}_{K \in \Omega}$ such that the family $\eta_{T,\cA_p} \coloneqq  (\eta_{b_K})_{K\in \Omega}$ defined above belongs to $\ES^r(\fF_{\mathrm{rel}})$. 

\item[(ii)] Apply Theorem  \ref{new strategy main result} to the family $\eta_{T,\cA_p}$ and then derive from the conclusion of this result consequences related to $\TNC(M,\cA_p)$.
\end{itemize}
\end{strategy}

\begin{remark}\label{deligne remark} Accomplishing step (i) requires one to show that, for every field $K$ in $\Omega$, one has $\eta_{b_K}\in {\bigcap}^d_{\cA_p[\cG_K]} H^1 (\cO_{K, S(K)}, T)$ in $\C_p \cdot \exprod^d_{\cA_p[\cG_K]} H^1 (\cO_{K, S(K)}, T)$. This restriction on the element $\eta_{b_K}$ constitutes an explicit `integral' refinement of the Deligne--Beilinson Conjecture for  $M_K$ and has been verified in only a very small number of cases. If $H^0 (K, T / p T) = 0$, 
then in the case
$M = h^0 (\Spec (k))$ and $\cA_p = \Z_p$ the prediction for $\eta_{b_K}$ is equivalent to the `Rubin--Stark Conjecture' for $K/k$ from \cite{Rub96} and, in general, one can show that it is implied by the validity of $\TNC(M_K,\cA_p[\cG_K])$.
\end{remark}

\begin{remark}\label{forgetful etnc remark}  For each $K$ in $\Omega$, the functoriality arguments of \cite[\S\,3.5, Th.~3.1b) and \S\,4.4, Th.~4.1]{BurnsFlach01} show that the validity of  $\TNC(M_K,\cA_p[\cG_K])$ implies that of $\TNC(M'_K,\cA_p)$. Here we write $M_K'$ for $M\otimes h^0 (\Spec (K))$, regarded as defined over $K$ and with coefficients $A$.   \end{remark}
 
\subsection{Statement of the main result} \label{statement iwasawa result section}

In this subsection we shall explain how, as part of the general Strategy \ref{gen strat}\,(ii), the techniques developed in Part I of this article allow one to study an Iwasawa-theoretic variant of Conjecture~\ref{etnc statement}. In  \S\,\ref{elliptic section} and \S\,\ref{Gm section}  the results obtained here will then be used to 
complete Strategy \ref{gen strat}, and thereby
obtain results towards Conjecture \ref{etnc statement} itself, in several interesting new cases.\\
Throughout, we fix a finite extension $\Phi$ of $\Q_p$ in $\Q_p^\mathrm{c}$, write $\cO$ for the valuation ring of $\Phi$, $\varpi$ for a uniformising element of $\mathcal{O}$ and $\kappa$ for the residue field $\cO /(\varpi)$. We also fix a finitely generated free $\cO$-module $T$ that is endowed with an $\cO$-linear continuous action of $G_k$ and  assume this action is unramified outside a finite set $S_\ram (T)$ of places of $k$.\\
As in \S\,\ref{higher-rank euler systems definitions sections}, we then fix a finite set of places 
\[ S_0 \subseteq \Pi_k\quad\text{with}\quad \Pi_k^\infty \cup \Pi_k^p \subseteq S_0\] 
and, for every finite abelian extension $K$ of $k$ set 
\[
S_0 (K) \coloneqq S_0 \cup S_\ram (K / k)
\quad \text{ and } \quad 
S(K) \coloneqq S_0 (K) \cup S_\ram (T).
\]

Given a subset $\Sigma$ of $\Pi_k$ with 
\[ \Sigma \cap (S_0 \cup S_\ram (T)) = \emptyset,\] 
we write $\cK^\Sigma$ for the composite of all abelian extensions of $k$ which are unramified at $\Sigma$, and denote by $\Omega_\Sigma \coloneqq \Omega (\cK^\Sigma)$ the set of all finite subextensions of $\cK^\Sigma / k$. We abbreviate $\Omega_\emptyset$ to $\Omega$. \\
For $K \in \Omega_\Sigma$, we use the $\Sigma$-modified Nekov\'{a}\v{r} structure 
\[ \mathscr{F}_{\mathrm{rel}, \Sigma} (T_{K / k}) \coloneqq \scrF_{\mathrm{rel}} (T_{K /k}, S (K))_\Sigma\] 
(cf. Example~\ref{lesc}) and, for each $i\in \N_0$, we  set 
\[
H^i_\Sigma (\cO_{K, S (K)}, T) \coloneqq H^1_{\mathscr{F}_{\mathrm{rel}, \Sigma} (T_{K / k})} (k, T_{K / k}).
\]
We observe that this construction agrees with previous definitions of $\Sigma$-modified cohomology (as used, for example, in \cite[\S\,2.3]{sbA}), and is motivated by constructions of Gross \cite{gross88} and Rubin \cite{Rub96} in the context of refinements of Stark's conjectures.

\begin{definition} \label{euler systems iwasawa theory def}
    Fix $r\in \N$ and a finite subset $\Sigma$ of $\Pi_k$ with  $\Sigma \cap (S_0 \cup S_\ram (T)) = \emptyset$. Then we define $\ES^r_{\Sigma, S_0} (T)$ to be the set of all elements
    \[
    c = (c_K)_{K} \in \prod_{K \in \Omega_\Sigma} \bidual^r_{\cO [\cG_K]} H^1_\Sigma (\cO_{K, S(K)}, T)
    \]
  that satisfy the following distribution relation: For all $K, L \in \Omega_\Sigma$, with $K \subseteq L$ one has
        \[
        \mathrm{cores}^r_{L / K} (c_L) = \big( {\prod}_{v \in S_0 (L) \setminus S_0 (K)} \Eul_v (\Frob_v^{-1}) \big) \cdot c_K.
        \]
Here 
the equality takes place in $\Phi \otimes_\cO \exprod^r_{\cO [\cG_K]} H^1_\Sigma (\cO_{K, S(L)}, T)$ (by using Lemma \ref{biduals-reduced-rings}) and 
we write $\mathrm{cores}^r_{L / K}$ for the morphism 
\[ \exprod^r \cores_{L / K}: \Phi \otimes_\cO \exprod^r_{\cO [\cG_L]} H^1_\Sigma (\cO_{L, S(L)}, T) \to \Phi \otimes_\cO \exprod^r_{\cO [\cG_K]} H^1_\Sigma (\cO_{K, S(L)}, T)\] 
that is induced
by the natural corestriction map $\cores_{L / K} \: H^1_\Sigma (\cO_{L, S(L)}, T) \to H^1_\Sigma (\cO_{K, S(L)}, T)$.
\end{definition}

\begin {rk} The link between the above notion of Euler system and that defined in \S\,\ref{higher-rank euler systems definitions sections} will be explained in \S\,\ref{twisting lemma section}.\end{rk}

We now fix a finite abelian extension $K$ of $k$ and a finite-rank $\Z_p$-power extension $k_\infty$ of $k$, set 
\[ K_\infty \coloneqq K \cdot k_\infty\,\,\text{ and }\,\, \Lambda_K \coloneqq \cO \llbracket \cG_{K_\infty} \rrbracket.\] 
We write $\Delta_K$ for the finite subgroup $\gal{K_\infty}{k_\infty}$ of $\cG_{K_\infty}$ and fix (as we may) a subgroup $\Gamma_{\!\!K}$ of $\cG_{K_\infty}$ that is mapped bijectively to $\cG_{k_\infty}$ by the restriction map $\cG_{K_\infty} \to \cG_{k_\infty}$. In particular, the group $\Gamma_{\!\!K}$ is isomorphic to $\Z_p^n$ for some $n > 0$ and there exists a direct product decomposition  
\begin{equation} \label{Gamma definition}
    \cG_{K_\infty} = \Delta_K \times \Gamma_{\!\!K}
\end{equation}
as well as a canonical direct product decomposition 
\begin{equation}\label{sylow decomp} \Delta_K = \nabla_{\!\!K} \times \square_K\end{equation} 
in which $\square_K$ and $\nabla_{\!\!K}$ are respectively the $p$-Sylow subgroup of $\Delta_K$ and the maximal subgroup of $\Delta_K$ of order prime to $p$. We also set 
\[ L \coloneqq K^{\square_K}\quad \text{and}\quad F \coloneqq K^{\nabla_{\!\!K}}.\] 
Given a character $\chi \in \widehat{\nabla_{\!\!K}} \coloneqq \Hom_\Z ( \nabla_{\!\!K}, \Phi^{\mathrm{c}, \times})$ for some fixed choice of algebraic closure $\Phi^\mathrm{c}$ of $\Phi$, we write $\cO_\chi \coloneqq \cO [\im (\chi)]$ for the $\cO$-algebra generated by the values of $\chi$.
We note that $\cO_\chi$ is an unramified extension of $\cO$ with residue field $\kappa_\chi \coloneqq \cO_\chi / (\varpi)$.
Moreover, any $\cO [\nabla_{\!\!K}]$-module $M$ decomposes as a direct sum 
\[ M = \bigoplus_{\chi \in \widehat{\nabla_{\!\!K}} / \sim} M_\chi\] 
with $M_\chi \coloneqq M \otimes_{\cO [\nabla_{\!\!K}]} \cO_\chi$ and $\sim$ the equivalence relation on $\widehat{\nabla_{\!\!K}}$ induced by the action of $G_{\Phi}$. Given an element $m \in M$, we denote by $m_\chi$ the image of $m$ in $M_\chi$. Lastly, we endow the tensor product $\cO_\chi$-module 
\[ T (\chi) \coloneqq T \otimes_\cO \cO_\chi,\]
with the `$\chi$-twisted'
$G_k$-action given by $\sigma \cdot a \coloneqq \chi (\sigma)^{-1} \cdot (\sigma a)$ for all $a \in T$. (Note that some authors use a different sign convention and denote this representation as $T (\chi^{-1})$.)\\
We assume $k_\infty$ contains a $\Z_p$-extension of $k$ in which no finite place splits completely (and hence that no finite place splits completely in $k_\infty$ itself). 
We also define Galois extensions 
\[
\mathfrak{E} = \mathfrak{E}(F,k_\infty/k) \coloneqq F \cdot k(1) \cdot k_\infty ( \mu_{p^\infty}, ( \cO_k^\times)^{1 / p^\infty})\,\,\,\text{and}\,\,\, \mathfrak{E}_T^\chi \coloneqq (k^\mathrm{c})^{\ker(\rho_{T(\chi)})}
\]
of $k$, where $\rho_{T(\chi)}$ denotes the canonical homomorphism $G_{\mathfrak{E}} \to \mathrm{Aut} (T(\chi))$ (so that $\mathfrak{E}\subseteq \mathfrak{E}_T^\chi$). \\
Finally, we fix  $r\in \N$ (which is to be specified later on) and consider the following hypothesis.

\begin{hyps} \label{new strategy iwasawa hyps}
    The above data $(T, k_\infty, F, \chi, r)$ satisfies all of the following conditions.
    \begin{romanliste}
    \item The $\kappa_\chi [G_k]$-module $\Tbar (\chi) \coloneqq T (\chi) \otimes_{\cO_\chi} \kappa_\chi$ is irreducible.
    \item There exists an element $\tau$ of $G_\mathfrak{E}$ for which $\dim_{\kappa_\chi} (\Tbar (\chi) / (\tau - 1) \Tbar (\chi)) = 1$.
    \item[(ii${}^\ast$)] If $p = 2$, then $\dim_{\kappa_\chi} (\Tbar (\chi)) = 1$. 
    \item The module $H^1 (\mathfrak{E}_T^\chi / k, (\Tbar (\chi))^\ast (1))$ vanishes.
    \item If $p \in \{2,3\}$, then the $\Z_p [G_k]$-modules $\Tbar (\chi)$ and $(\Tbar (\chi))^\ast (1) = \Tbar^\ast (1) (\chi^{-1})$ have no nonzero isomorphic subquotients.
    \item $H^1 (\mathfrak{E}_T^\chi / k, \Tbar (\chi))$ is a finite-dimensional $\kappa_\chi$-vector space and
    \[
    \dim_{\kappa_\chi}\!\big( H^1 (\mathfrak{E}_T^\chi / k, \Tbar (\chi)) \big)< r + \dim_{\kappa_\chi}\!\big( H^0 (k, \Tbar (\chi)) \big) - {\sum}_{v \in \Sigma} \dim_{\kappa_\chi}\!\big( H^0 (k_v, \Tbar^\ast (1) (\chi^{-1})) \big).\]
    \end{romanliste}
\end{hyps}

\begin{rk} \label{invariants vanishing remark}
 If the $\kappa_\chi [G_k]$-module $\Tbar (\chi)$ is irreducible, then so is $\Tbar^\ast (1) (\chi^{-1})$. In particular, the validity of Hypotheses \ref{new strategy iwasawa hyps}\,(i) and (iii) combine to imply that  $H^0 (k, \Tbar^\ast (1) (\chi^{-1}))=(0)$.  
\end{rk}

We consider the induced representation 
\[
\cT \coloneqq \mathrm{Ind}^{G_k}_{G_{K_\infty}} (T) = T \otimes_\cO \Lambda_K,
\]
upon which $\sigma \in G_k$ acts via $\sigma \cdot (a \otimes b) \coloneqq (\sigma a) \otimes (b \overline{\sigma}^{-1})$ with $\overline{\sigma}$ the image of $\sigma$ in $\Lambda_K$.\\
We write $\mathscr{F}_{\mathrm{rel}, \Sigma} (\cT)$ for the $\Sigma$-modified relaxed Nekov\'a\v{r} structure $\mathscr{F}_{\mathrm{rel}} (\cT, S(K))$ on $\cT$, and denote by 
\[ \cF_{ \mathrm{rel}, \Sigma} (\cT) \coloneqq h ( \mathscr{F}_{\mathrm{rel}, \Sigma} (\cT))\] 
the induced Mazur--Rubin structure.

We will also use the complex
\[
C_{S (K), \Sigma} (\cT) \coloneqq C (\mathscr{F}_{\mathrm{rel}, \Sigma} (\cT))
\]
from Proposition \ref{construction complex}. (If $\Sigma = \emptyset$, then we will drop the subscript $\Sigma$.) In particular, we will assume that each group $H^1_\Sigma (\cO_{K, S(K)}, T)$ is $\cO$-torsion free, and so may apply Lemma \ref{finite level reps} to deduce that the complex $C_{S (K), \Sigma} (\cT)$  admits a resolution of the form $P \to P$ with $P$ a finitely generated free $\Lambda_K$-module and the first term is placed in degree zero. This fact combines with \cite[Lem.\@ B.9]{Sakamoto20} to imply that the natural map
\begin{align*}
\bidual^r_{\Lambda_K} H^1_\Sigma (\cO_{k, S(K)}, \cT) & \to 
{\varprojlim}_{K \subseteq E \subseteq K_\infty} \bidual^r_{\cO [\cG_E]} H^1_\Sigma (\cO_{k, S(K)}, T \otimes_\cO \cO [\cG_E]) 
\\ 
& \xrightarrow{\simeq} {\varprojlim}_{K \subseteq E \subseteq K_\infty} \bidual^r_{\cO [\cG_E]} H^1_\Sigma (\cO_{E, S(K)}, T)
\end{align*}
is an isomorphism. (Here the second map is the isomorphism from Shapiro's lemma.) For every $c \in \ES^r_{\Sigma, S_0} (T)$ we therefore obtain a well-defined element
\[
c_{K_\infty} \coloneqq ( c_E)_E \in \bidual^r_{\Lambda_K} H^1_\Sigma (\cO_{k, S(K)}, \cT).
\]
Given any finite set $U \subseteq \Pi_k$, we use the $\Lambda_K$-module 
\[ Y_U (\cT) \coloneqq \bigoplus_{v \in U} H^0 (k_v, \cT^\vee (1))^\vee.\]
We note in particular that the dual of the diagonal map $H^0 (k, \cT^\vee (1)) \to {\bigoplus}_{v \in U} H^0 (k_v, \cT^\vee (1))$ induces a map $Y_U (\cT) \to H^0 (k, \cT^\vee (1))^\vee$ of $\Lambda_K$-modules, the kernel of which is equal to the module 
\[ X_U (\cT) \coloneqq X_U (\mathscr{F}_{\mathrm{rel}, \Sigma}  (\cT))\] 
defined in \S\,\ref{psc section}. 

\begin{remark}\label{X=Y remark} The above observation implies that, for any non-empty subset $U'$ of $U$, there exists a short exact sequence of $\Lambda_K$-modules 
$0 \to X_{U'} (\cT) \xrightarrow{\subseteq} X_{U} (\cT) \to Y_{U\setminus U'} (\cT) \to 0$, in which the third arrow denotes the natural projection map. In addition, if Hypothesis \ref{new strategy hyps} is valid, then Remark \ref{invariants vanishing remark} implies that the modules $X_U (\cT)$ and $Y_U (\cT)$ coincide.  \end{remark}

We next fix a non-zero $\Lambda_K$-free quotient $Y$ of $Y_{\Pi_k^\infty} (\cT)$,
for example of the form obtained by passing to the limit over the modules $\cY_{EK} (T)$ defined in (\ref{definition of Y_K (T)}).
We note that the central exact sequence in Proposition \ref{construction complex}\,(v) implies that  $Y$ is then also a free quotient of $H^1 (C_{S, \Sigma} (\cT))$. At this point we also specify the natural number $r$ by setting  
\[
r \coloneqq \mathrm{rk}_{\Lambda_K} (Y)\in \N.
\]

By using the identification of rings $\Lambda_K \cong \cO [\Delta_K] \llbracket \Gamma_{\!\!K} \rrbracket$ induced by (\ref{Gamma definition}) to regard $\Phi [\Delta_K]$ as a subring of the total ring of fractions $\cQ_K \coloneqq \cQ (\Lambda_K)$ of $\Lambda_K$, 
we next define idempotents of $\Phi[\Delta_K]$ by setting 
\[ e_K \coloneqq {\sum}_{\psi\in \Xi_K } e_\psi\,\,\text{ and }\,\,  \epsilon_K \coloneqq {\sum}_{\psi\in \Xi_K'} e_\psi. \] 
Here we write $\Xi_K$ and $\Xi_K'$ for the sets of characters $\psi \in \widehat{\Delta_K}$ for which 
the primitive idempotent $e_\psi$ of $\Phi^c[\Delta_K]$ annihilates the kernel of the natural projection maps 
\[ \cQ^\mathrm{c}_K \otimes_{\Lambda_K} H^2_\Sigma (\cO_{k, S (K)}, \cT) \to \cQ^\mathrm{c}_K \otimes_{\Lambda_K} Y,\quad \text{respectively } \,\, \cQ^\mathrm{c}_K \otimes_{\Lambda_K} Y_{\Pi_k^\infty} (\cT) \to \cQ^\mathrm{c}_K \otimes_{\Lambda_K} Y,\]
where we have set $\cQ^\mathrm{c}_K \coloneqq \Phi^\mathrm{c}\otimes_{\Phi}\cQ_K$. 

\begin{rk} These idempotents have the following useful properties. 
\begin{romanliste}
\item Since no place in $\Pi_k\setminus \Pi_k^\infty$  splits completely in $k_\infty$, the idempotent $e_K$ does not depend on $\Sigma$ (cf.\@ \cite[Lem.\@ 4.12]{BullachDaoud}). In addition, if $k_\infty$ is the cyclotomic $\Z_p$-extension of $k$, then the weak Leopoldt conjecture (from \cite[\S\,1.3]{PerrinRiou95}) predicts that $\cQ_K \otimes_{\Lambda_K} H^2_\Sigma (\cO_{k, S(K)}, \cT)$ vanishes and hence $e_K = 1$. 
\item The definition of $\Xi_K'$ implies that $\cQ_K(1 - \epsilon_K) \otimes_{\Lambda_K} Y_{\Pi_k^\infty} (\cT)$ is precisely equal to the kernel of the projection map $\cQ_K \otimes_{\Lambda_K} Y_{\Pi_k^\infty} (\cT) \to \cQ_K \otimes_{\Lambda_K} Y$. 
\end{romanliste}
\end{rk}

We next fix a $\Lambda_K$-basis $b_\bullet = \{ b_i\}_{i \in [r]}$ of $Y$. Then, since $\cQ_K$ is a semisimple ring, we can define a composite morphism  of $\cQ_K$-modules 
\begin{align} \nonumber
    \Theta_{K_\infty,\Sigma, b_\bullet}\:  \cQ_K \otimes_{\Lambda_K} \Det_{\Lambda_K} (C_{S(K), \Sigma} (\cT)) 
    & \xrightarrow{\sim} \Det_{\cQ_K} ( \cQ_K \otimes_{\Lambda_K} H^0 ( C_{S(K), \Sigma} (\cT))) \\
    & \qquad \otimes_{\cQ_K} \Det_{\cQ_K} ( \cQ_K \otimes_{\Lambda_K} H^1 ( C_{S(K), \Sigma} (\cT)))^{-1}\notag\\ 
    & \xrightarrow{} e_K \epsilon_K \cQ_K \otimes_{\Lambda_K} 
    \exprod^r_{\Lambda_K} H^1_\Sigma (\cO_{k, S (K)}, \cT) 
    \label{definition Theta map}
\end{align}
(which we will usually abbreviate to $\Theta_{\Sigma, b_\bullet}$). Here the first map is the canonical `passage-to-cohomology map' and the second is induced by first multiplying by $e_K \epsilon_K$ and then applying the `evaluation at $e_K\bigwedge_{i \in [r]} b_i$' map to elements of  
\[ e_K \epsilon_K \bigl(\Det_{\cQ_K} ( \cQ_K \otimes_{\Lambda_K} H^1 (C_{S (K), \Sigma} (\cT)))^{-1} \bigr)= e_K\big(\cQ_K \otimes_{\Lambda_K} \Det_{\Lambda_K} (Y)^{-1}\bigr).\]

Our next result concerns the image under $\Theta_{\Sigma, b_\bullet}$ of the $\Lambda_K$-submodule 
$\Det_{\Lambda_K} (C_{S(K), \Sigma} (\cT))$ of $\cQ_K \otimes_{\Lambda_K} \Det_{\Lambda_K} (C_{S(K), \Sigma} (\cT)) $. 

\begin{thm} \label{new iwasawa theory euler systems main result}
    Assume that $(k,p)$ verifies condition (\ref{p=2 condition}) and that $K$ is a finite abelian extension of $k$ for which the following conditions are both satisfied.
    \begin{liste}
        \item Every place $v$ in $S_\mathrm{ram} (T^\vee (1)) \setminus S_0$ is finitely decomposed in $k_\infty$ and such that the module $\bigoplus_{w \mid v} H^0 (E_w, T^\ast (1))$ vanishes for every intermediate field $E$ of $K_\infty/k$.
        \item $H^1_\Sigma (\cO_{K, S (K)}, T)$ is $\cO$-torsion free. 
    \end{liste}
Set $\cT \coloneqq T \otimes_\cO \Lambda_K$ and fix a non-zero free $\Lambda_K$-module quotient $Y$ of $Y_{\Pi_k^\infty} (\cT)$ of rank $r$. Also fix a character $\chi \: \nabla_{\!\!K} \to \Phi^{\mathrm{c}, \times}$ for which $(T, k_\infty, F, \chi, r)$ validates Hypotheses \ref{new strategy iwasawa hyps}. Then the following claims are valid for every system $c \in \ES^r_{\Sigma, S_0} (T)$. 
\begin{romanliste}
        \item For each $\Lambda_K$-basis $b_\bullet = \{b_i\}_{i \in [r]}$ of $Y$ one has
    \begin{multline*}
    \hspace{-0.5cm}
    \Fitt^r_{\Lambda_K} ( X_{S_0} (\cT)) \cdot 
    \Big( {\prod}_{v \in S(K) \setminus S_0 (K)} \Eul_v (\Frob_v^{-1}) \Big) \cdot
    (c_{K_\infty})_\chi 
    \; \subseteq \; 
    \Theta_{\Sigma, b_\bullet} ( \Det_{\Lambda_K} (C_{S (K), \Sigma} (\cT)))_\chi.
    \end{multline*}
    \item If $\p$ is a prime of $\cQ_K$ in the support of $\cQ_K \otimes_{\Lambda_K}H^1_{\cF_{ \mathrm{rel}, \Sigma} (\cT)^\vee} (k, \cT^\vee (1))^\vee_\chi$, then $(\epsilon_K c_{K_\infty})_\p = 0$.
    \item For all $f \in H^1_\Sigma (\cO_{k, S (K)}, \cT)^\ast$ one has 
    \[ 
    \epsilon_K \cdot f (c_{K_\infty})_\chi \in \epsilon_K \Fitt^0_{\Lambda_K} ( H^1_{\cF_{ \mathrm{rel}, \Sigma} (\cT)^\vee} (k, \cT^\vee (1))^\vee )^{\ast \ast}_\chi.
    \]
\end{romanliste}
\end{thm}

The proof of this result will occupy the remainder of this section and is given in \S\,\ref{proof of iwasawa theory result} after we have made a series of preliminary observations in \S\,\ref{preliminaries to proof of iwasawa theory result}.\\
We will then discuss consequences of Theorem \ref{new iwasawa theory euler systems main result} concerning Kato's Conjecture (\ref{etnc statement}) in the settings of both elliptic curves and the multiplicative group in \S\,\ref{elliptic section} and \S\,\ref{Gm section}, respectively.

\subsection{Preliminary results} \label{preliminaries to proof of iwasawa theory result}

We first make several general observations that will be used in \S\,\ref{proof of iwasawa theory result} to prove Theorem \ref{new iwasawa theory euler systems main result}. 

\subsubsection{Twisting Euler and Kolyvagin systems} \label{twisting lemma section}

In this section, we assume to be given an $\cO$-algebra $R$ that is endowed with a continuous $\mathcal{O}$-linear action of $G_k$, and write 
\[ \psi \: G_k \to \mathrm{Aut} (R)\]
for the induced homomorphism. We then write 
\[ T (\psi) \coloneqq T \otimes_\cO R\] 
for the $\mathcal{O}[\![G_k]\!]$-module upon which each $\sigma \in G_k$ acts via the rule $\sigma \cdot (x \otimes y) \coloneqq (\sigma x) \otimes ( \sigma^{-1} y)$. We also define an associated Galois extension of $k$ by setting 
\[ k_\psi \coloneqq (k^\mathrm{c})^{\ker (\psi)}.\]  

\begin{remark} Fix a finite abelian extension $K$ of $k$ in $k^c$ and write $\psi_K \: G_k \to \cO [\cG_K]^\times$ for the homomorphism that sends each element of $G_k$ to its image in $\cG_K$. Then the module $T(\psi_K)$ defined above coincides with the induced representation $T_{K / k}$ defined in \S\,\ref{svc section}.\end{remark}

\begin{prop} \label{twisting lemma}
    Assume that the algebra $R$ satisfies condition (\ref{ring condition}), and that $\im (\psi)$ is both abelian and either finite or a finitely generated $\Z_p$-module. Let $\Sigma$ be a finite subset of $\Pi_k$ with $\Sigma \cap S(k_\psi) = \emptyset$ and such that $C_{S(k_\psi), \Sigma} ( T_{k_\psi / k})$ has a representative of the form $P \to P$ in which $P$ is a finite-rank free $\cO \llbracket \cG_{k_\psi} \rrbracket$-module and the first term is placed in degree 0.  Then the following claims are valid. 
    \begin{romanliste}
        \item Fix an abelian pro-$p$ extension $E$ of $k$ in $k^c$ such that $\cG_E$ is a finitely generated $\Z_p$-module, and assume that $\Sigma \cap S(Ek_\psi) = \emptyset$. 
        Then, for every integer $r \geq 0$, there are canonical ($\psi$-semilinear) homomorphisms
            \[ \mathrm{Tw}^r_{E, \psi} \: \bidual^r_{\cO \llbracket \cG_{Ek_\psi}\rrbracket} H^1_\Sigma (\cO_{Ek_\psi, S(Ek_\psi)}, T) \to \bidual^r_{R \llbracket\cG_E\rrbracket} H^1_\Sigma (\cO_{E, S(Ek_\psi)}, T(\psi))
            \]
            and
            \[
            \mathrm{Tw}^\mathrm{det}_{E, \psi} \: \Det_{\cO \llbracket \cG_{Ek_\psi}\rrbracket} ( C_{S(Ek_\psi)} (T_{Ek_\psi / k})) \to \Det_{R \llbracket\cG_E\rrbracket} ( C_{S (Ek_\psi)} (T_{E / k} (\psi)))
            \]
            that have both of the following properties.
            \begin{liste}
               \item Let $Y$ be a free rank $r$ $\cO \llbracket \cG_{Ek_\psi} \rrbracket$-module quotient of $Y_{\Pi_k^\infty} ( T_{Ek_\psi / k})$ with basis $b_\bullet$. Write $b'_\bullet$ for the induced basis of the (free) $R [\cG_E]$-module quotient $Y' \coloneqq Y \otimes_{\cO \llbracket \cG_{Ek_\psi} \rrbracket} R \llbracket\cG_E\rrbracket$ of $Y_{\Pi_k^\infty} ( T_{E / k} (\psi)) \cong Y_{\Pi_k^\infty} ( T_{Ek_\psi / k}) \otimes_{\cO \llbracket \cG_{Ek_\psi} \rrbracket} R \llbracket\cG_E\rrbracket$. Then the diagram  
            \begin{cdiagram}
                \Det_{\cO \llbracket \cG_{Ek_\psi}\rrbracket} ( C_{S(Ek_\psi),\Sigma} (T_{Ek_\psi / k})) \arrow{r}{\vartheta} \arrow{d}{\mathrm{Tw}^\mathrm{det}_{E, \psi}} & 
                \bidual^r_{\cO \llbracket\cG_{Ek_\psi} \rrbracket} H^1_\Sigma (\cO_{Ek_\psi, S(Ek_\psi)}, T) \arrow{d}{\mathrm{Tw}^r_{E, \psi}} \\
                \Det_{R \llbracket \cG_E \rrbracket} ( C_{S (Ek_\psi),\Sigma} (T_{E / k} (\psi))) \arrow{r}{\vartheta'} & \bidual^r_{R \llbracket\cG_E\rrbracket} H^1_\Sigma (\cO_{E, S(Ek_\psi)}, T(\psi))
            \end{cdiagram}%
            commutes. Here we write $\vartheta$ and $\vartheta'$ for the maps obtained by applying Lemma \ref{det projection lemma}\,(i) to the data $(C_{S(Ek_\psi),\Sigma} (T_{Ek_\psi / k}), b_\bullet)$ and $(C_{S (Ek_\psi),\Sigma} (T_{E / k} (\psi)), b_\bullet')$, respectively.
            \item The image under $\mathrm{Tw}^\mathrm{det}_{E, \psi}$ of an $\cO\llbracket\cG_{Ek_\psi} \rrbracket$-basis element of $\Det_{\cO \llbracket \cG_{Ek_\psi}\rrbracket} ( C_{S(Ek_\psi)} (T_{Ek_\psi / k}))$ is an $R \llbracket\cG_E\rrbracket$-basis element of $\Det_{R \llbracket\cG_E\rrbracket} ( C_{S (Ek_\psi)} (T_{E / k} (\psi)))$. 
            \end{liste} 
        \item Write $\cK^{\Sigma, p}$ for the composite of all $p$-power degree extensions of $k$ contained in $\cK^\Sigma$, and let $\Omega (\cK^{\Sigma, p}) \subseteq \Omega_\Sigma$ denote the subset of fields contained in $\cK^{\Sigma, p}$. Then the family of Nekov\'a\v{r} structures 
        \[ \fF_{\mathrm{rel}, \Sigma} (T (\psi)) \coloneqq (\scrF_{\mathrm{rel}, \Sigma} (T(\psi)_{E / k}))_{E \in \Omega (\cK^{\Sigma, p})}\] 
        satisfies Hypothesis \ref{system of fF hyp}. In addition, the following claims are valid.
        \begin{liste}
            \item There exists a $\psi$-semilinear map of $\cO \llbracket G_k \rrbracket$-modules
        \[
        \mathrm{Tw}_\psi^r \: 
        \ES^r_{\Sigma, S_0} (T) \to \ES^r_{S_0} ( \fF_{\mathrm{rel}, \Sigma} (T (\psi))), \quad c \mapsto ( \mathrm{Tw}^r_{E, \psi} ( c_{Ek_\psi}))_{E \in \Omega (\cK^{\Sigma, p})}. 
        \]
        \item Assume that, for every $E \in \Omega_\Sigma$, Hypothesis \ref{more new hyps}\,(ii) holds for $T_{E / k}$. Fix a filtration $(\a_i)_{i \geq 0}$ on $R$ 
        as in (\ref{filtration}), and set $T_i (\psi) \coloneqq T (\psi) /\a_i T (\psi)$. Then, for every Euler system $c \in \ES^r_{\Sigma, S_0} (T)$ and non-negative integer $i$, there exists a Kolyvagin system $(\kappa_\fn)_{\fn} \in \KS^r ( \mathscr{F}_{\mathrm{rel}, \Sigma} (T_i (\psi)))$ such that, writing $\pi^r_{i, T ( \psi)}$ for the map from (\ref{proj map def}), one has 
        \[
        \kappa_1 = \pi^r_{i, T ( \psi)} ( \mathrm{Tw}_\psi^r (c)_k)
        \quad \text{ in } \quad \bidual^r_{R / \a_i} H^1_\Sigma (\cO_{k, S(k_\psi)}, T_i (\psi)).
        \]
        \end{liste}
    \end{romanliste}
\end{prop}

\begin{proof} At the outset we claim that, for every abelian pro-$p$ extension $E$ of $k$ such that $\cG_E$ is finitely generated over $\Z_p$, the given assumptions imply the existence of an isomorphism  
\begin{equation}\label{desired rep} C_{S (E k_\psi), \Sigma} (T_{ E k_\psi / k}) \cong \big [P_E \to P_E\big ]\end{equation} 
in $D(\cO \llbracket \cG_{E k_\psi} \rrbracket)$, in which $P_E$ is a finitely-generated free $\cO \llbracket \cG_{E k_\psi} \rrbracket$-module and the first term of the second complex is placed in degree zero. To construct such an isomorphism, we will apply Lemma \ref{finite level reps} to the complex
$C_{S (E k_\psi), \Sigma} (T_{ E k_\psi / k})$, and so must first verify the assumptions necessary for an application of the latter result. 

At the outset, we note Lemma \ref{flach result} implies  $\mathscr{F}_{\mathrm{rel}, \Sigma} (T_{Ek_\psi / k})$ satisfies Hypothesis \ref{fF hyp} and so Proposition \ref{construction complex}\,(v) implies that $C_{S (E k_\psi), \Sigma} (T_{ E k_\psi / k})$ belongs to the category $D^{\mathrm{perf}, 0}_{[0, 1]} ( \cO \llbracket \cG_{Ek_\psi} \rrbracket)$. To apply Lemma \ref{finite level reps} in this context, it is thus enough for us to show $C_{S (E k_\psi), \Sigma} (T_{ E k_\psi / k}) \otimes^\mathbb{L}_{\cO \llbracket \cG_{E k_\psi} \rrbracket} \kappa$ belongs to $D^{\mathrm{perf}, 0}_{[0, m']} (\kappa)$ for some $m' \ge 0$. To verify this, we note that Proposition \ref{construction complex}\,(iii) implies the existence of an isomorphism in $D(\kappa)$ of the form 
\[
C_{S (E k_\psi), \Sigma} (T_{ E k_\psi / k}) \otimes^\mathbb{L}_{\cO \llbracket \cG_{E k_\psi} \rrbracket} \kappa \cong C_{S (E k_\psi), \Sigma} (T_{k_\psi / k}) \otimes^\mathbb{L}_{\cO \llbracket \cG_{k_\psi} \rrbracket} \kappa. 
\]
In particular, if we can prove that $C_{S (E k_\psi), \Sigma} (T_{k_\psi / k})$ has a representative of the form $P' \to P'$ with $P'$ a finite-rank free $\cO \llbracket \cG_{k_\psi} \rrbracket$-module and the first term placed in degree zero, then the above isomorphism would induce an isomorphism in 
$D^{\mathrm{perf}}(\kappa)$ of the form 
\[ C_{S (E k_\psi), \Sigma} (T_{ E k_\psi / k}) \otimes^\mathbb{L}_{\cO \llbracket \cG_{E k_\psi} \rrbracket} \kappa \cong \big [P' \otimes_{\cO \llbracket \cG_{k_\psi} \rrbracket} \kappa \to P' \otimes_{\cO \llbracket \cG_{k_\psi} \rrbracket} \kappa\big ],\] 
thereby showing that the left hand complex belongs to $D^{\mathrm{perf}, 0}_{[0, 1]} (\kappa)$, as required. 

At this point, we note  Proposition \ref{construction complex}\,(vi) implies $C_{S (E k_\psi), \Sigma} (T_{k_\psi / k})$ is isomorphic to a mapping cone of a morphism of the form
\[ {\bigoplus}_{v \in S (E k_\psi) \setminus S (k_\psi)} \mathrm{R}\Gamma_f ( k_v, T_{k_\psi / k}^\vee (1))^\vee [-2] \to  C_{S (k_\psi), \Sigma} (T_{k_\psi / k})
\]
and also recall that, for each $v \in S (E k_\psi) \setminus S (k_\psi)$, there exists an isomorphism in $D(\cO \llbracket \cG_{k_\psi} \rrbracket)$  
\[ \mathrm{R}\Gamma_f ( k_v, T_{k_\psi / k}^\vee (1))^\vee [-2] \cong \big [T_{k_\psi / k} (-1) \xrightarrow{\Frob_v -1} T_{k_\psi / k} (-1)\big ],\] 
in which the first term of the second complex is placed in degree $1$. In particular, since the given assumptions imply  
$C_{S (k_\psi), \Sigma} (T_{k_\psi / k})$ is isomorphic to a complex $P \to P$ in which $P$ a finitely-generated free $\cO \llbracket \cG_{k_\psi} \rrbracket$-module and the first term is placed in degree zero, a standard mapping cone construction can be combined with the above observations to imply that $C_{S (E k_\psi), \Sigma} (T_{k_\psi / k})$ is isomorphic to a complex of the required form $P' \to P'$. 

Having verified all hypotheses of Lemma \ref{finite level reps}, we can now directly deduce from the latter result the existence of an isomorphism (\ref{desired rep}). By combining this isomorphism with the result of Proposition \ref{construction complex}\,(iii), we can then deduce the existence of an isomorphism in 
$D(R \llbracket \cG_E \rrbracket)$ 
\begin{align*}
C_{S (Ek_\psi), \Sigma} (T (\psi)_{E / k}) & \cong C_{S (Ek_\psi), \Sigma} (T_{Ek_\psi / k}) \otimes^\mathbb{L}_{\cO \llbracket \cG_{E k_\psi} \rrbracket} R \llbracket \cG_E \rrbracket\\
& \cong \big [P_E \otimes_{\cO \llbracket \cG_{E k_\psi} \rrbracket} R \llbracket \cG_E \rrbracket 
\to P_E \otimes_{\cO \llbracket \cG_{E k_\psi} \rrbracket} R \llbracket \cG_E \rrbracket \big],
\end{align*}
in which the first term of the last complex is placed in degree $0$. 

On the other hand, the complexes 
    \begin{align*}
        C^{r}(T_{Ek_\psi / k}) &\coloneqq \left[\exprod_{\cO \llbracket \cG_{E k_\psi} \rrbracket}^{r} P_E \to P_E \otimes_{\cO \llbracket \cG_{E k_\psi} \rrbracket} \exprod_{\cO \llbracket \cG_{E k_\psi} \rrbracket}^{r - 1} P_E\right] \\
         C^{r}(T_{E / k}(\psi_n)) & \coloneqq C^{r}_{S (Ek_\psi), \Sigma} (T_{Ek_\psi / k}) \otimes_{\cO \llbracket \cG_{E k_\psi} \rrbracket}^{\mathbb{L}} R \llbracket \cG_E \rrbracket
    \end{align*}
(in which the first term of $ C^{r}(T_{Ek_\psi / k})$ is placed in degree zero) are both acyclic in degrees less than zero and, moreover, Lemma \ref{biduals lemma 1}\,(i) implies that there are canonical isomorphisms  
    \begin{align*}
        H^0(C^{r}(T_{Ek_\psi / k})) &\cong \bidual_{\cO \llbracket \cG_{E k_\psi} \rrbracket}^{r} H^1 (\cO_{Ek_\psi, S(E k_\psi)}, T)\label{bidual-bottom-cohomology}\\
        H^0( C^{r}(T_{E / k}(\psi_n))) &\cong \bidual_{ R \llbracket \cG_E \rrbracket}^{r} H^1(\cO_{E,S (Ek_\psi)}, T(\psi)). 
    \end{align*}
These facts combine with an explicit analysis of the second page of the universal coefficient spectral sequence  of $\cO \llbracket \cG_{E k_\psi} \rrbracket$-modules 
    \begin{equation}\label{spec seq eq}
        E_2^{i,j} = \Tor_{-i}^{\cO \llbracket \cG_{E k_\psi} \rrbracket}(H^j(C^{r}(T_{Ek_\psi / k})), R \llbracket \cG_E \rrbracket) 
        \quad \Rightarrow \quad 
        H^{i+j}(C^{r}(T_{E / k}(\psi_n))).
    \end{equation} 
to imply the existence of a homomorphism $\mathrm{Tw}^r_{E, \psi}$ of the required sort. The claims (a) and (b) in (i) can then be checked via explicit computations (cf.\@ the argument of \cite[Cor.~4.9]{Tsoi}).\\
Turning now to (ii), we first note Lemma \ref{flach result} implies the claimed validity of Hypothesis \ref{system of fF hyp} for $\fF (T (\psi))$. Part (a) of  (ii) is then proved as in \cite[\S\,2.4 and Ch.\@ 6]{Rubin-euler}. In a little more detail, it is clear that the maps $\mathrm{Tw}^r_{E, \psi}$ combine to give a $\psi$-semilinear map
\[
\ES^r_{\Sigma, S_0} (T) \to \prod_{E \in \Omega (\cK^{\Sigma, p})} H^1_\Sigma (\cO_{E, S(Ek_\psi)}, T (\psi))
\]
of $R \llbracket \cG_{\cK^{\Sigma, p}} \rrbracket$-modules. The key point now is that, for every field $E \in \Omega (\cK^{\Sigma, p})$ and place $v  \in \Pi_k \setminus S (E k_\psi)$, the $\psi$-semilinearity of the map $\mathrm{Tw}^0_{E, \psi} \: \cO \llbracket \cG_{E k_\psi } \rrbracket \to R \llbracket \cG_E \rrbracket$ implies that 
\[ \mathrm{Tw}^0_{E, \psi}(\Eul_v (\Frob_v^{-1}, T)) = \Eul_v (\psi (\Frob_v^{-1}) \Frob_v^{-1}, T) = \Eul_v (\Frob_v^{-1}, T (\psi^{-1})).\]
As for claim (b), let us write $\psi_i \: G_k \to R_i$ for the morphism induced by $\psi$, and $k_i \coloneqq (k^\mathrm{c})^{\ker (\psi_i)}$ for the kernel field of $\psi_i$. 
We may assume also that $\varpi^i \in \a_i$ for every $i$ by taking a subfiltration of $(\a_i)_i$ if necessary, and we  set $T_i \coloneqq T / \varpi^i T$. Theorem \ref{kolyvagin derivative thm} then shows that there is a Kolyvagin derivative homomorphism
\[
\mathcal{D} \: 
\ES^r_{\Sigma, S_0} (T) \to 
\ES^r_{S_0} ( \mathscr{F}_{\mathrm{rel}, \Sigma} (T_{k_i / k})) \to \KS^r ( T_{i, k_i / k})
\]
with the property that $\mathcal{D} (c)_1 = \pi^r_{i, T} (c_{k})$
with the projection map $\pi^r_{i, T}$ 
from (\ref{proj map def})
.
Here we take the field $k (T_{i, k_i / k})$ in (\ref{minimal fields}) that determines the set of Kolyvagin primes $\cQ ( \tau, T_{i, k_i / k})$ appearing in the definition of $\KS^r ( T_{i, k_i / k})$ to be the minimal extension of $k$ such that $G_{k (T_{i, k_i / k})}$ acts trivially on $k (T_{i, k_i / k})$. Then $G_{k (T_{i, k_i / k})}$ also acts trivially on $T_i (\psi_i)$ and so we may take $k( T_i (\psi))$ to be equal to $k (T_{i, k_i / k})$. In particular, since an element $\tau \in G_k$ that validates Hypothesis \ref{new strategy hyps}\,(ii) for $T_{i, k_i / k}$ also does so for $T_i (\psi)$, we have an equality of sets of Kolyvagin primes $\cQ ( \tau, T_{i, k_i / k}) = \cQ (\tau, T_i (\psi_i))$.
\\
We now denote by $\tilde{\mathscr{F}}$ and $\tilde{\mathscr{F}} \otimes R_i$ the $\Sigma$-modified relaxed Nekov\'a\v{r}--Selmer structures on $T_{i, k_i / k}$ and on $T_{i, k_i / k} \otimes_\cO R_i = T_i (\psi)$, respectively. Then, for each modulus $\fn \in \cN ( \cQ ( \tau, T_{i, k_i / k}))$, the method used to prove claim (i) applies to the complexes $C (\tilde{\mathscr{F}} (\fn))$ from Proposition \ref{what we need from Selmer complexes} to prove the existence of `twisting' homomorphisms 
\[
\mathrm{Tw}_{\psi, \fn}^r \: \bidual^r_{(\cO / \varpi^i) [\cG_{k_i}]} H^1_{\tilde{\mathscr{F}} (\fn)} (k, T_{i, k_i / k}) \to \bidual^r_{R_i} H^1_{(\tilde{\mathscr{F}} \otimes R_i) (\fn)} (k, T_{i} (\psi))
\]
that are compatible with the finite-singular relations (cf.\@ \cite[Rk.~3.1.4]{MazurRubin04}) and therefore
combine to define a map
\[
\widetilde{\mathrm{Tw}}^r_{\psi} \: 
\KS^r (T_{i, k_i / k}) \to \KS^r (T_i (\psi)), \quad (\kappa_\fn)_\fn \mapsto ( \mathrm{Tw}_{\psi, \fn}^r (\kappa_\fn))_\fn. 
\]
We now claim that $\kappa \coloneqq \widetilde{ \mathrm{Tw}}^r_{\psi} (\mathcal{D} ( c)))$ 
has the property that
$
\kappa_1 = \pi_i^r ( \mathrm{Tw}^r_\psi (c)_k)$, and hence that $\kappa$ is the Kolyvagin system sought after.\\
To do this, we observe that one has the commutative diagram
\begin{cdiagram}[row sep=small]
    \bidual^r_{\cO \llbracket \cG_{k_\psi} \rrbracket} H^1_\Sigma (\cO_{k_\psi, S (k_\psi)}, T) \arrow{r}{\mathrm{Tw}^r_{\psi, k}} \arrow{d}{\pi^r_{i, T}} & 
    \bidual^r_R H^1_\Sigma (\cO_{k, S(k_\psi)}, T(\psi)) \arrow{d}{\pi^r_{i, T (\psi)}} \\
    \bidual^r_{(\cO /\varpi^i) [\cG_{k_i}]} H^1_\Sigma (\cO_{k_i, S (k_\psi)}, T_i) \arrow{r}{\mathrm{Tw}_{\psi, 1}^r} & \bidual^r_{R_i} H^1_\Sigma (\cO_{k, S(k_\psi)}, T_i (\psi)).
\end{cdiagram}%
It follows that one has
\[
\kappa_1 = 
\widetilde{\mathrm{Tw}}^r_{\psi} (\mathcal{D} ( c)))_1 = \mathrm{Tw}^r_{\psi, 1} ( \mathcal{D} (c)_1) = \mathrm{Tw}^r_{\psi, 1} ( \pi^r_{i, T} (c_k)) = \pi_{i, T ( \psi)}^r ( \mathrm{Tw}^r_{\psi, k} (c_k)),
\]
as required to complete the proof of claim (ii)\,(b).
\end{proof}

\begin{rk} Fix a homomorphism $\chi \: \Delta_K \to \Phi^{\mathrm{c}, \times}$ and write $\cO_\chi$ for the extension of $\cO$ generated by the values of $\chi$. We then obtain an induced  homomorphism  $\psi \: G_k \to \cO_\chi [\cG_F]^\times$
and the construction of Proposition \ref{twisting lemma} defines a morphism $\ES^r_{\Sigma, S_0} (T) \to \ES^r_{S_0} ( \fF_{\mathrm{rel}, \Sigma} (T_{F / k} (\chi)))$. Via this homomorphism, one obtains a precise link between the notion of Euler system specified in Definition \ref{euler systems iwasawa theory def} and that used in \S\,\ref{higher-rank euler systems definitions sections}. \end{rk}

\subsubsection{Torsion subgroups}

The following result clarifies the condition in Theorem \ref{new iwasawa theory euler systems main result}\,(b).  

\begin{lem} \label{torsion lemma new}
    Suppose that Hypotheses \ref{new strategy iwasawa hyps}\,(i)\,(ii) and Hypothesis \ref{more new hyps}\,(ii) are valid for $T$ (with $R  = \cO$). Then, for all finite subsets $\Sigma$ of $\Pi_k$ with $\Sigma \cap S(k) = \emptyset$, all fields $F \in \Omega (\cK^{\Sigma, p})$ and all finite subsets $U$ of $\Pi_k$ with   $S(F)\subseteq U$, the following assertions are equivalent.
    \begin{romanliste}
        \item $H^0_\Sigma (\cO_{k, U}, \Tbar ) \neq (0)$,
        \item $H^1_\Sigma (\cO_{F, U}, T)$ is not $\cO$-torsion free,
        \item $\Sigma = \emptyset$, $\mathrm{rk}_\cO (T)  = 1$, and $G_k$ acts trivially on $\Tbar$.
    \end{romanliste}
\end{lem}

\begin{proof}
    At the outset we define $V \coloneqq \Phi \otimes_\cO T$. The long exact sequence in cohomology arising from the short exact sequence $0 \to T \to V \to V / T \to 0$ then combines with the vanishing of $T^{G_k}$ (that follows from Hypothesis \ref{more new hyps}\,(ii)) to give the exact sequence 
    \begin{cdiagram}
        0 \arrow{r} & H^0_\Sigma (\cO_{F, U}, V / T) \arrow{r} & H^1_\Sigma (\cO_{F, U}, T) \arrow{r} & H^1_\Sigma (\cO_{F, U}, V).
    \end{cdiagram}%
    Since $H^1_\Sigma (\cO_{F, U}, V) = \Phi \otimes_\cO H^1_\Sigma (\cO_{F, U}, T)$ by Proposition \ref{construction complex}\,(iii), we deduce an identification $H^1_\Sigma (\cO_{F, U}, T)_\mathrm{tor} \cong H^0_\Sigma (\cO_{F, U}, V / T)$. The latter is non-trivial if and only if $H^0_\Sigma (\cO_{k, U}, V / T) \neq 0$ because $F$ is a $p$-extension of $k$ (cf.\@ \cite[Cor.\@ 1.6.13]{NSW}). Writing $\varpi$ for a uniformiser of $\cO$, we have an isomorphism $(V / T) [\varpi] \cong \Tbar$, and so $H^0_\Sigma (\cO_{k, U}, V / T) \neq (0)\Longleftrightarrow H^0_\Sigma (\cO_{k, U}, \Tbar ) \neq (0)$. This proves the equivalence of conditions (i) and (ii).\\
    We note next that the triangle (\ref{sigma triangle}) implies $H^0_\Sigma (\cO_{k, U}, \Tbar)$ is equal to the kernel of the diagonal map $H^0 (k, \Tbar) \to \bigoplus_{v \in \Sigma} H^0 (k_v, \Tbar)$. Since the latter map is injective if $\Sigma \neq \emptyset$, we conclude that (i) implies $\Sigma = \emptyset$.
    In addition, since Hypothesis \ref{new strategy iwasawa hyps}\,(ii) implies that $\Tbar$ is an irreducible $\mathbb{k} [G_k]$-module, the module $H^0 (\cO_{k, U}, \Tbar)$ either vanishes or is equal to $\Tbar$. In the latter case, Hypothesis \ref{new strategy iwasawa hyps}\,(i) then implies that $\Tbar  = \Tbar^{\tau = \id}$ is a one-dimensional $\mathbb{k}$-vector space, and hence, by Nakayama's Lemma, that $\mathrm{rk}_\cO (T) = 1$. This shows (i) implies (iii), and since (iii) clearly implies (i), this proves the claimed result.
\end{proof}

\subsubsection{Core ranks}\label{verifying the hypotheses section} 

We next clarify Hypothesis \ref{new strategy hyps}\,(vi) in the case of relaxed Nekov\'a\v{r} structures. To do this, we use the notation specified in (\ref{conv notation}). In particular, we now write $\Lambda$ and $A$ for $R_i$  and $\cT\otimes_{\cR}\Lambda$, so that $A\otimes_{\Lambda}\mathbb{k} = \Tbar$.\\
We fix finite subsets $S$ and $\Sigma$ of $\Pi_k$ with $S (k)\subseteq S$ and $\Sigma \cap S = \emptyset$ and recall that, if the pair $(k,p)$ verifies (\ref{p=2 condition}), then the $\Sigma$-modified  relaxed Nekov\'a\v{r} structure $\scrF_\mathrm{rel} ( A, S)_\Sigma$ on $A$ coincides with both of the induced structures 
$\scrF_\mathrm{rel} ( \cT, S)_\Sigma\otimes_{\cR}\Lambda$ and $\scrF_\mathrm{rel} ( \cT\otimes_{\cR}\cR_i, S)_\Sigma\otimes_{\cR_i}\Lambda$ (cf.\@ Example \ref{remark selmer}\,(iv)).\\
We also recall that a finitely-generated $\Lambda$-module $M$ is said to be `quadratically-presented' if there exists an exact sequence  of $\Lambda$-modules of the form
\[
\Lambda^{\oplus m} \to \Lambda^{\oplus m} \to M \to 0,
\]
with $m \coloneq\dim_\mathbb{k} ( M \otimes_\Lambda \mathbb{k})$.

\begin{lem} \label{Euler characteristic if qadratically presented}  
Fix  finite subsets $S$ and $\Sigma$ of $\Pi_k$ with $S (k)\subseteq S$ and $\Sigma \cap S = \emptyset$. Write $\tilde F$ for the Mazur--Rubin structure $h(\scrF_{\mathrm{rel}}( A, S)_\Sigma)$ on $A$ and $\overline{F}$ for the structure that it induces on $\Tbar$.
Then, if the pair $(k,p)$ verifies (\ref{p=2 condition}), the following claims are valid.
\begin{romanliste}
    \item The core-rank $\bm{\chi}(\overline{F}, j)$ of the pair $(\overline{F}, j)$ (in the sense of (\ref{definition core rank})) is equal to 
    \begin{multline*}
    \dim_{\mathbb{k}}\bigl(H^0 (k, \Tbar)\bigr) + \dim_{\mathbb{k}}\bigl( X_{S} (\Tbar)\bigr) + \dim_{\mathbb{k}} \bigl(\sha_{\overline{F}^\ast, j} (\Tbar^\ast (1))\bigr) - \dim_{\mathbb{k}} \bigl(\sha_{\overline{F}, j} (\Tbar)\bigr)\\ + {\sum}_{v \in \Sigma}\dim_{\mathbb{k}} \bigl( H^0 (k_v, \Tbar^\ast (1))\bigr) 
        + {\sum}_{v \in S\setminus \Pi_k^\infty} \dim_{\mathbb{k}} \bigl(\mathrm{Tor}_1^\Lambda ( H^0 (k_v, A^\ast (1))^\ast, \mathbb{k})\bigr).
    \end{multline*}
\item If the $\Lambda$-module $H^2 (k_v, A)$ is quadratically presented for every $v \in S \setminus \Pi_k^\infty$, then  $\bm{\chi}(\overline{F}, j)$ is at least  
    \begin{multline*}
        {\sum}_{v \in \Pi^\infty_k} \dim_{\mathbb{k}}\bigl(H^0 (k_v, \Tbar^\ast (1))\bigr)  +  \dim_{\mathbb{k}} \bigl(\sha_{\overline{F}^\ast, j} (\Tbar^\ast (1))\bigr)  - \dim_{\mathbb{k}}\bigl(\sha_{\overline{F}, j} (\Tbar)\bigr) \\
        + \dim_{\mathbb{k}}\bigl(H^0 (k, \Tbar)\bigr) - \dim_\mathbb{k} \big( H^0(k,\Tbar^\ast (1))\big)  - {\sum}_{v \in \Sigma} \dim_{\mathbb{k}}\bigl(H^0 (k_v, \Tbar^\ast (1))\bigr). 
      \end{multline*}  
      \end{romanliste}
\end{lem}

\begin{proof}
 We write $F_0$ for the Mazur--Rubin structure $h (\scrF_{\mathrm{rel}}( A, S)_\Sigma\otimes_\Lambda \mathbb{k})$ on $\Tbar$. Then, by using Theorem \ref{global duality thm}\,(ii) to compare $F_0$ with $\overline{F}$,  one finds that  
\begin{align*}
\bm{\chi}(\overline{F}, j) = & \bm{\chi} (F_0, j) - {\sum}_{v \in S(\cF)} \dim_\mathbb{k} \bigl(H^1_{\!/ \overline{F}} (k_v, \Tbar)\bigr) \\
  = &\dim_\mathbb{k}\bigl( H^0 (k, \Tbar)\bigr) + \dim_\mathbb{k}\bigl( X_{S (\cF)} (\Tbar)\bigr) - {\sum}_{v \in S(\cF)} \dim_\mathbb{k}\bigl( H^1_{\!/ \overline{F}} (k_v, \Tbar)\bigr) \\
& \quad \qquad + 
{\sum}_{v \in \Sigma} \dim_\mathbb{k} \bigl(H^0 (k, \Tbar) \bigr)
+ \dim_{\mathbb{k}} \bigl(\sha_{\overline{F}^\ast, j} (\Tbar^\ast (1))\bigr) - \dim_{\mathbb{k}}\bigl( \sha_{\overline{F}, j} (\Tbar)\bigr).
\end{align*}
In this computation we have also used Proposition \ref{construction complex}\,(v) and the equality $\bm{\chi}_\mathbb{k}(C_{S, \Sigma} (\Tbar)) = 0$ that follows from  Proposition \ref{construction complex}\,(i) and Lemma \ref{flach result}. \\
Now, by the definition of induced Mazur--Rubin structure, one has  
\[
H^1_{\!/ \overline{F}} (k_v, \Tbar) \coloneqq \coker \big ( H^1_{\tilde F} (k_v, A) \subseteq H^1 (k_v, A) \to H^1 (k_v, \Tbar)  \big ).
\]
In addition, for $v \in \Sigma$, one has $H^1_{\tilde F} (k_v, A) = 0$, so $H^1_{\!/ \overline{F}} (k_v, \Tbar) = H^1 (k_v, \Tbar)$ and hence 
\begin{align*} \dim_{\mathbb{k}}\bigl(H^1_{\!/ \overline{F}} (k_v, \Tbar)\bigr) = &\, \dim_\mathbb{k}\bigl(H^1 (k_v, \Tbar)\bigr)\\
=&\, \dim_\mathbb{k}\bigl(H^0 (k_v, \Tbar)\bigr) + \dim_\mathbb{k}\bigl(H^2 (k_v, \Tbar)\bigr)\\
=&\, \dim_\mathbb{k}\bigl(H^0 (k_v, \Tbar)\bigr) + \dim_\mathbb{k}\bigl(H^0 (k_v, \Tbar^\ast (1))\bigr).
\end{align*}
It remains to consider the case $v \in S$. In this case $H^1_{\!/ \overline{F}} (k_v, \Tbar)$ is defined to be the cokernel of the natural map $\rho_v \: H^1 (k_v, A) \to H^1 (k_v, \Tbar)$. In addition, the universal coefficient spectral sequence
\[
E_2^{i, j} = \mathrm{Tor}_{-i}^\Lambda ( H^j (k_v, A), \mathbb{k}) \Rightarrow 
E^{i + j} = H^{i + j} (k_v, \Tbar) 
\]
induces an the exact sequence
\[
    \mathrm{Tor}_2^\Lambda (H^2 (k_v, A), \mathbb{k}) \to
    H^1 (k_v, A) \otimes_\Lambda \mathbb{k} \to  H^1 (k_v, \Tbar) \to \mathrm{Tor}_1^\Lambda ( H^2 (k_v, A), \mathbb{k}) \to 0
\]
in which the second arrow is induced by $\rho_v$. In particular, since we are assuming $(k,p)$ verifies (\ref{p=2 condition}), one has $H^1_{\!/ \overline{F}} (k_v, \Tbar) = (0)$ for each  $v \in \Pi_k^\infty$. In addition, for  $v \in S\setminus \Pi_k^\infty$, the above sequence combines with local duality to imply $H^1_{\!/ \overline{F}} (k_v, \Tbar)$ is isomorphic to $\mathrm{Tor}_1^\Lambda ( H^0 (k_v, A^\ast (1))^\ast, \mathbb{k})$, as required to complete the proof of (i).\\  
Before turning to the proof of (ii), we claim that $\dim_\mathbb{k}\bigl(\Tor_1^\Lambda (M, \mathbb{k})\bigr)\leq \dim_\mathbb{k} (M \otimes_\Lambda \mathbb{k})$ for every quadratically presented $\Lambda$-module $M$. To justify this, we fix a presentation 
\[\Lambda^{\oplus n} \xrightarrow{f} \Lambda^{\oplus n} \to M \to 0\] 
with $n \coloneqq \dim_\mathbb{k} (M \otimes_\Lambda \mathbb{k})$. Setting $I \coloneqq \ker (f)$, we thereby obtain an induced exact sequence 
    \[ 
    0 = \Tor_1^\Lambda ( \Lambda^{\oplus n}, \mathbb{k}) \to \Tor_1^\Lambda (M, \mathbb{k}) \to I \otimes_\Lambda \mathbb{k},
    \] 
    from which it follows that $\dim_\mathbb{k}(\Tor_1^\Lambda (M, \mathbb{k})) \leq \dim_\mathbb{k} (I \otimes_\Lambda \mathbb{k}) \leq n$ since $I$ is a quotient of $\Lambda^{\oplus n}$. This proves the claim.\\
   Then, by applying this observation with $M = H^2 (k_v, A) \cong H^0(k_v,A^\ast(1))^\ast$ for each $v \in S \setminus \Pi_k^\infty$, and noting $H^2 (k_v, A) \otimes_\Lambda \mathbb{k}$ identifies with $H^2 (k_v, \Tbar) \cong H^0(k_v,\Tbar^\ast(1))^\ast$, we obtain an inequality
    \[
    {\sum}_{v \in S \setminus \Pi_k^\infty} \big ( \dim_\mathbb{k}\bigl(H^0(k_v,\Tbar^\ast(1))\bigr)  
    - \dim_\mathbb{k}\bigl(\Tor^\Lambda_1 ( H^0(k_v,A^\ast(1))^\ast, \mathbb{k})\bigr)\big) \geq 0.
    \]
From the definition of $X_{S} (\Tbar)$ it therefore follows that   
 \begin{align*} & \phantom{\,=\,} \dim_{\mathbb{k}}\bigl( X_{S} (\Tbar)\bigr)  + \dim_\mathbb{k}\bigl( H^0 (k, \Tbar^\ast (1))\bigr)\\
 & =  {\sum}_{v \in \Pi_k^\infty}  \dim_\mathbb{k}\bigl(H^0(k_v,\Tbar^\ast(1))\bigr)+ {\sum}_{v \in S\setminus \Pi_k^\infty}  \dim_\mathbb{k}\bigl(H^0(k_v,\Tbar^\ast(1))\bigr) \\
 & \ge {\sum}_{v \in \Pi_k^\infty}  \dim_\mathbb{k}\bigl(H^0(k_v,\Tbar^\ast(1))\bigr)+{\sum}_{v \in S\setminus \Pi_k^\infty} \dim_{\mathbb{k}} \bigl(\mathrm{Tor}_1^\Lambda ( H^0 (k_v, A^\ast (1))^\ast, \mathbb{k})\bigr). \end{align*}
The result of (ii) is now  obtained by substituting this inequality into the formula of (i).  
\end{proof}

\begin{remark} \label{Euler characteristic if qadratically presented rk}
The condition of quadratic-presentability is automatic in the following cases.
\begin{romanliste}
\item Let $\Lambda \coloneqq (\Z / p^i \Z) [G]$ with  $i \in \N$ 
and $G$ a finite abelian group. Then any $G$-cohomologically trivial finitely generated $\Lambda$-module is quadratically presented.  
\item If $\Lambda$ is a principal ideal ring, then every $\Lambda$-module is quadratically presented. Indeed, as a local principal ideal ring, $\Lambda$ is isomorphic to a quotient of a discrete valuation ring $\Lambda'$ and since the claim is true over $\Lambda'$, base-changing to $\Lambda$ gives the desired presentation. 
\end{romanliste}
\end{remark}

\subsubsection{Localisation, Euler factors, and Fitting ideals}\label{euler factor section}

We recall (from \cite[\S\,3C1]{BKS2}) that a prime ideal $\fp$ in ${\mathrm{Spec}}^1(\Lambda_K)$ is said to be `regular' if it does not contain the order of the torsion subgroup of $\cG_{K_\infty}$. For such $\p$, the localisation $\Lambda_{K, \p}$ of $\Lambda_K$ is a discrete valuation ring and there exists a character $\chi = \chi_\p \: \Delta_K \to \Phi^{\mathrm{c}, \times}$ and a height-one prime ideal $\wp_\chi$ of the ring 
\[ \Lambda_\chi \coloneqq \cO_\chi \llbracket \Gamma_{\!\!K} \rrbracket\] 
for which there is an identification $\Lambda_{K, \p} = \Lambda_{\chi, \wp_\chi}$. 

If $\p$ is not regular, then it is said to be `singular'. If $\p$ is any such prime, then $p \in \p$,
and there exists a character $\chi \: \nabla_{\!\!K} \to\Phi^{\mathrm{c}, \times}$ together with
 a height-one prime ideal $\wp_\chi$ of the ring 
 \[ (\Lambda_K)_\chi \coloneqq \Lambda_\chi [\square_K]\] 
 such that $\Lambda_{K, \p} = (\Lambda_K)_{\chi, \wp_\chi}$. There is then also the following localisation criterion from \cite[Lem.\@ 6.3]{BurnsGreither} (see also \cite[Lem.\@ 5.6]{Flach04}).

\begin{lem} \label{singular primes and mu vanishing}
    Fix a character $\chi \: \nabla_{\!\!K} \to \Phi^{c, \times}$ and a prime ideal $\wp \in \Spec^1(\Lambda_\chi [\square_K])$ with $p\in \wp$. If $M$ is any  finitely generated torsion $\Lambda_\chi [\square_K]$-module for which the Iwasawa $\mu$-invariant (as a $\Lambda_\chi$-module) is $0$, then the localisation $M_\wp$ vanishes.
\end{lem}

In addition, the following result is useful in the computation of the ideal $\Fitt^0_{\Lambda_K} ( X_{S} (\cT))$.  

\begin{lem} \label{error term iwasawa theory}
For every  $v \in \Pi_k\setminus \Pi_k^\infty$ and  $\p \in \Spec^1(\Lambda_K)$, the following claims are valid.
\begin{romanliste}
    \item Assume $v \not \in \Pi_k^p$ and either $v \not \in S_{\mathrm{ram}}(T^\vee (1))$ or both $v$ is finitely decomposed in $k_\infty$ and 
    $H^0 (k_v, T_{E / k}^\ast (1))=(0)$ for all finite extensions $E$ of $k$ in $K_\infty$. Then, if $p \notin \p$, and we write $\chi$ for the character $\chi_\p$ defined above, one has
\[
\Fitt^0_{\Lambda_K} ( H^0 (k_v, \cT^\vee (1))^\vee)_\p = 
\begin{cases}
    \Eul_v (\Frob_v^{-1}) \Lambda_{K_\chi, \wp_\chi} \quad & \text{ if } v \notin S_{\mathrm{ram}}(K_\chi/k), \\
    \Lambda_{K_\chi, \wp_\chi} & \text{ if } v \in  S_{\mathrm{ram}}(K_\chi/k).
\end{cases}
\] 
    \item If $v$ is finitely decomposed in $k_\infty$ and $p \in \p$, then $H^0 (k_v, \cT^\vee (1)^\vee)_\p=(0)$.  
\end{romanliste}
\end{lem}

\begin{proof}
To prove (i), we assume that $p \notin \p$. In this case there are isomorphisms 
\[
(H^0 (k_v, \cT^\vee (1))^\vee)_\p \cong 
H^2 (k_v, \cT)_\p \cong H^2 (k_v, T_{K_{\chi, \infty} / k})_{\wp_\chi}
\cong H^0 (k_v, T_{K_{\chi, \infty} / k}^\vee (1))^\vee_{\wp_\chi},
\]
and so it suffices to investigate $\Fitt^0_{\Lambda_{K_\chi}} ( H^0 (k_v, T_{K_{\chi, \infty} / k}^\vee (1))^\vee)$. In the remainder of this argument we will therefore assume that $K = K_\chi$. \\
To discuss the case $v \in S_{\rm{ram}}(K/k)$, we write $I$ for the inertia subgroup of $\cG_{K}$ and regard the trace element $\NN_{I}$ (from (\ref{norm def})) as an element of $\Lambda_{K}$ in the natural way. By assumption, $\chi (I) \neq 1$ and so in $\Lambda_{K, \p}$ one has $\NN_I = \chi (\NN_I) = 0$. On the other hand, $\NN_I$ acts as multiplication by $|I| \in \Lambda_{K, \p}^\times$ on $H^0 (k_v, \cT^\vee (1))^\vee = \bigoplus_{w \mid v} H^0 (K_{\infty, w}, T^\vee (1))^\vee$ and so we conclude that, in this case, the localisation of $H^0 (k_v, \cT (1))^\vee$ at $\p$ vanishes. This proves the claim if $v \in S_{\rm{ram}}(K/k)$, and so it remains to consider the case $v \notin  S_{\rm{ram}}(K/k)$.\\
Let us first assume that, in addition, $v \not \in S_\mathrm{ram} (T^\vee (1))$ so that the action of $G_{k_v}$ is unramified on $\cT^\vee (1)$.
Taking Matlis duals of the exact sequence 
\begin{cdiagram}
    0 \arrow{r} & H^0 (k_v, \cT^\vee (1)) \arrow{r} & \cT^\vee (1) \arrow{rr}{\Frob_v - 1} & & \cT^\vee (1),
\end{cdiagram}%
we obtain an identification $H^0 (k_v, \cT^\vee (1))^\vee \cong \coker (\cT (-1) \xrightarrow{\Frob_v - 1} \cT (-1) )$. As such, we obtain the required equality via the computation
    \begin{align*}
    \Fitt^0_{\Lambda_{K}} ( H^0 (k_v, \cT^\vee (1))^\vee) & = \Lambda_{K} \cdot {\det}_{\Lambda _{K}} ( \Frob_v - 1 \mid \cT (-1))
    \\
    & = \Lambda_{K} \cdot {\det}_{\Lambda_{K}} ( \Frob_v^{-1} - 1 \mid \cT^\ast (1)) \\
    & = \Lambda_{K} \cdot \Eul_v (\Frob_v^{-1}).
    \end{align*}
  We finally assume that $v$ is finitely decomposed in $k_\infty$ and that $H^0 (k_v, T^\ast_{ E/ k} (1))$ vanishes for all finite extensions of $E$ of $k$ contained in $K_\infty$. Setting $V \coloneqq T \otimes_\cO \Phi$, this assumption combines with the tautological short  exact sequence $0 \to T^\ast (1) \to V^\ast (1) \to V^\ast (1) / T^\ast (1) \to 0$  to imply, for every such $E$,  exactness of the sequence
   \begin{cdiagram}
       0 \arrow{r} & H^0 ( k_v, V_{E / k}^\ast (1) / T^\ast_{E / k} (1)) \arrow{r} & H^1 (k_v, T_{E / k}^\ast (1)) \arrow{r} & H^1 (k_v, V_{E / k}^\ast (1)).
   \end{cdiagram}%
   These sequences in turn induces  natural identifications 
   \[ H^0 (k_v, T_{E / k}^\vee (1)) = H^0 ( k_v, V_{E / k}^\ast (1) / T^\ast_{E / k} (1)) \cong H^1 (k_v, T_{E / k}^\ast (1))_\tor.\] 
  Further, if we write  $\mathbb{F}_w$ for the residue field of a place $w$ of $E$ above $v$ and $I_w \subseteq G_{E_w}$ for its inertia subgroup, then the  
   isomorphism $H^1 (k_v, T_{E / k}^\ast (1)) \cong \bigoplus_{w \mid v} H^1 (E_w, T^\ast (1))$ from Shapiro's lemma combines with inflation-restriction sequence (for each $w$) to give short exact sequences 
   \begin{cdiagram}[column sep=small]
       0 \arrow{r} & {\bigoplus}_{w \mid v} H^1 (\mathbb{F}_w, T^\ast (1)^{I_w})_\tor \arrow{r} & 
       H^1 (k_v, T_{E / k}^\ast (1))_\tor \arrow{r} & {\bigoplus}_{w \mid v} H^1 (I_w, T^\ast (1))^{G_{E_w}}_\tor.
   \end{cdiagram}%
In addition, the assumption $H^0 (k_v, T^\ast_{ E / k} (1)) = (0)$ also implies that we have the exact sequence
   \begin{equation} \label{resolution unramif cohom}
   \begin{tikzcd}
       0 \arrow{r} & T_{E / k}^\ast (1)^{I_v} \arrow{rr}{\Frob_v - 1} & & T_{E / k}^\ast (1)^{I_v} \arrow{r} & H^1 (\mathbb{F}_v, T_{E / k}^\ast (1)^{I_v}) \arrow{r} & 0
   \end{tikzcd}
   \end{equation}
   which, by comparing $\cO$-ranks, implies that $H^1 (\mathbb{F}_v, T_{E / k}^\ast (1)^{I_v})$ is finite and, in particular, equal to its $\cO$-torsion submodule. \\
   Next we note that since we assume $v\notin \Pi_k^p$ and also that $v$ is finitely decomposed in $k_\infty$, 
    the modules ${\bigoplus}_{w \mid v} H^1 (I_w, T^\ast (1))$ are finitely generated $\cO$-modules and, moreover, stabilise as $E$ ranges over the subfields of $K_\infty$. It follows that $\varprojlim_{k \subseteq E \subseteq K_\infty} {\bigoplus}_{w \mid v} H^1 (I_w, T^\ast (1))^{G_{E_w}}$ vanishes (cf.\@ also the argument of \cite[App.~B, Prop.~3.3]{Rubin-euler}) and so, by 
   passing to the limit over $E$ in the exact sequence (\ref{resolution unramif cohom}), we deduce from the previous discussion the existence of isomorphisms 
   \[
   H^0 (k_v, \cT^\vee (1))^\vee_{\p} \cong ({\varprojlim}_E H^1 (k_v, T^\ast_{E / k} (1))^\vee_\tor)_\p \cong ({\varprojlim}_E H^1 (\mathbb{F}_v, T_{E / k}^\ast (1)^{I_v})^\vee)_{\p}.  
   \]
 In addition, upon taking $\cO$-linear duals of (\ref{resolution unramif cohom}) and then passing to the limit (over $E$) of the resulting exact sequences , we obtain an exact sequence
   \begin{cdiagram}[column sep=small]
       0 \arrow{r} & (\cT^\ast (1)^{I_v})^\ast \arrow{rrr}{\Frob_v - 1} & & &
       (\cT^\ast (1)^{I_v})^\ast \arrow{r} & {\varprojlim}_E H^1 (\mathbb{F}_v, T_{E / k}^\ast (1)^{I_v})^\vee \arrow{r} & 0.
   \end{cdiagram}%
 In particular, since the $\Lambda_{K}$-module $\cT^\ast (1)^{I_v}$ is free of finite rank (since $v$ is assumed to be unramified in $K_{\infty}$), this exact sequence implies that  
   \begin{align*}
       \Fitt^0_{\Lambda_{K}} ( {{\varprojlim}}_E H^1 (\mathbb{F}_v, T_{E / k}^\ast (1)^{I_v})^\vee ) & = \Lambda_{K} \cdot {\det}_{\Lambda_{K}} ( \Frob_v - 1 \mid (\cT^\ast (1)^{I_v})^\ast) \\
       & = \Lambda_{K} \cdot {\det}_{\Lambda_{K}} ( \Frob_v^{-1} - 1 \mid \cT^\ast (1)^{I_v}) \\
       & = \Lambda_{K} \cdot \Eul_v (\Frob_v^{-1}),
   \end{align*}
   as required to complete the proof of (i).\\
   As for (ii), upon taking the Matlis dual of the inclusion $H^0 (K_{\infty, w}, T^\vee (1)) \subseteq T^\vee (1)$ for every place  $w$ above $ v$, one obtains a surjective map  
   \[ \bigoplus_{w \mid v} T ( - 1) \twoheadrightarrow \bigoplus_{w \mid v} H^0 (K_{\infty, w}, T^\vee (1))^\vee = H^0 (k_v, \cT^\vee (1))^\vee.\] 
In particular, if $v$ is finitely decomposed in $k_\infty$, then the latter module is  finitely generated over $\cO$ and so  Lemma~\ref{singular primes and mu vanishing} implies that it vanishes after localising at any prime in  $\Spec^1(\Lambda_K)$ that contains $p$. This proves (ii). 
\end{proof}

\begin{rk} \label{error term at p remark} In many arithmetic examples, there is a useful replacement for Lem\-ma~\ref{error term iwasawa theory}\,(i). To be precise, we assume $T = H^i_{\text{\'et}} (X_{k^\mathrm{c}}, \cO) (1)$ for an odd integer $i$ and  a proper smooth variety $X$ that is defined over $k$ and has potentially good reduction at a given $v \in \Pi_k^p$. Then, if $k_\infty$ is the cyclotomic $\Z_p$-extension of $k$, the group $H^0 (k_v, T_{K_\infty / k}^\vee (1))$ is finite by a result of Kubo and Taguchi \cite[Th.\@ 1.1]{KuboTaguchi} (if $X$ is an elliptic curve, this is a classical result of Imai \cite{Imai}). In addition, if $X$ is an abelian variety, $i = 1$, and $K_\infty$ is a Lubin--Tate extension, then similar criteria for finiteness have  been obtained by Ozeki in \cite{Ozeki}. If $T$ is the representation attached to a cuspidal newform, then a similar result is also obtained by Kato in \cite[(12.5.1)]{Kato04}. 
\end{rk}

\subsection{The proof of Theorem \ref{new iwasawa theory euler systems main result}}\label{proof of iwasawa theory result}

\subsubsection{An initial reduction} 

We first establish a useful reduction step in the proof of Theorem \ref{new iwasawa theory euler systems main result}.

\begin{prop} \label{towards the new iwasawa theory euler systems main result}
    Theorem \ref{new iwasawa theory euler systems main result} is valid if, under its stated hypotheses, the following is true: for every system $c \in \ES^r_{\Sigma, S_0} (T,\cK)$ and all pairs $(K, \chi)$ as in the statement of the latter result, one has an inclusion
\begin{equation}
\label{reduction step inclusion}
\Fitt^r_{\Lambda_K} ( X_{S (K)} (\cT)) \cdot (c_{K_\infty})_\chi 
\subseteq \Theta_{\Sigma, b_\bullet} ( \Det_{\Lambda_K} (C_{S, \Sigma} (\cT)))_\chi.
\end{equation}
\end{prop}

\begin{proof} Throughout this argument, we assume the hypotheses of Theorem \ref{new iwasawa theory euler systems main result} to be valid. 

Then, to prove that the inclusion claimed in Theorem \ref{new iwasawa theory euler systems main result}\,(i) is a consequence of (\ref{reduction step inclusion}), it is enough for us to show the latter inclusion implies that  
\begin{multline} \label{p-component claim}
\Fitt^r_{\Lambda_K} (X_{k, S_0} (\cT))_\p \cdot 
        \big( {\prod}_{v \in S(K) \setminus S_0 (K)} \Eul_v (\Frob_v^{-1}) \big) \cdot
        (c_{K_\infty})_\chi\\ \subseteq \Theta_{\Sigma, b_\bullet} ( \Det_{\Lambda_K} (C_{S (K), \Sigma} (\cT)))_{\chi,\p}
\end{multline}
for every $\p \in \Spec^1 (\Lambda_{K,\chi})$. 
To do this, we note first that if $p \in \p$, then 
\begin{equation} \label{Fitting ideals at singular primes}
    \Fitt^r_{\Lambda_K} (X_{S (K)} (\cT))_\p = \Fitt^r_{\Lambda_K} (X_{S_0} (\cT))_\p,
\end{equation} 
and so (\ref{p-component claim}) follows directly from the $\p$-localisation of (\ref{reduction step inclusion}). Indeed, 
since we assume no place in $S(K)\setminus \Pi_k^\infty$ splits completely in $k_\infty^\circ$, the decomposition subgroups in $\cG_{K_\infty}$ of such places are infinite and so, for every character $\psi \: \nabla_{\!\!K} \to \Phi^{\mathrm{c}, \times}$, the $\mu$-invariant of $Y_{S (K) \setminus S_0} (\cT)_\psi$ as a $\Lambda_\psi$-module vanishes. Hence, by Lemma \ref{singular primes and mu vanishing}, the module $Y_{S (K) \setminus S_0} (\cT)$ vanishes upon localisation at any $\p$ containing $p$ so that the natural short exact sequence
\begin{equation} \label{useful exact sequence 1}
\begin{tikzcd}
    0 \arrow{r} & X_{S_0 } (\cT) \arrow{r} & X_{S(K)} (\cT) 
    \arrow{r} & Y_{S (K) \setminus S_0} (\cT) \arrow{r} & 0
\end{tikzcd}%
\end{equation}
(from Remark \ref{X=Y remark}) implies bijectivity of the map $X_{S_0 (K) } (\cT)_\p \to X_{S(K)} (\cT) _\p$, and therefore also the claimed equality (\ref{Fitting ideals at singular primes}). \\
We next turn to the verification of (\ref{p-component claim}) for $\p \in \Spec^1 (\Lambda_{K,\chi})$ with $p \not \in \p$.
In this case, $\Lambda_{K, \p}$ identifies with $\Lambda_{K_\psi, \wp_\psi}$ for some character $\psi \: \Delta_K \to \Phi^{\mathrm{c}, \times}$ and prime $\wp_\psi\in {\mathrm{Spec}}^1(\Lambda_{K_\psi})$ (with $p \notin \wp_\psi$) and we set $\cT_\psi \coloneqq T (\psi) \otimes_{\cO_\psi} \Lambda_\psi$. Then, for $v\in S(K) \setminus S_0$, Lemma \ref{error term iwasawa theory}\,(i) implies $\Fitt^0_{\Lambda_{K_\psi}} ( H^0 (k_v, \cT_\psi^\vee (1))^\vee)_{\wp_\psi}$ is equal to $\Lambda_{K_\psi, \wp_\psi}$ if $v\in S_0 (K_\psi)$ (as $v \in S_\ram (K_\psi / k)$ in this case) and is generated by $\Eul_v (\Frob_v^{-1})$ if $v\notin S_0 (K_\psi)$. As a consequence, one has 
\begin{align} \nonumber
    \Fitt^r_{\Lambda_{K}} ( X_{S (K)} (\cT))_{\p} & = \Fitt^r_{\Lambda_{K}} ( X_{S_0} (\cT))_{\p} \cdot \Fitt^0_{\Lambda_{K}} ( Y_{S (K) \setminus S_0} (\cT))_{\p} \\  
    \nonumber 
    & = \Fitt^r_{\Lambda_{K_\psi}} ( X_{S_0} (\cT_\psi))_{\wp_\psi} \cdot \Fitt^0_{\Lambda_{K_\psi}} ( Y_{S (K) \setminus S_0} (\cT_\psi))_{\wp_\psi} \\ 
    & = \Fitt^r_{\Lambda_{K_\psi}} ( X_{S_0} (\cT_\psi))_{\wp_\psi} \cdot \big( {\prod}_{v \in S (K) \setminus S_0 (K_\psi)} \Eul_v (\Frob_v^{-1}) \big),
    \label{calculations with Xs}
\end{align}
where the first equality follows as a consequence of the exact sequence (\ref{useful exact sequence 1}) 
and the fact that $\Lambda_{K, \p}$ is a discrete valuation ring so that $\Lambda_{K, \p}$-Fitting ideals are multiplicative across short exact sequences.
 Hence, in this case, the inclusion (\ref{p-component claim}) for $\p$ follows from the computation
    \begin{align*}
        & \phantom{=\,} \Fitt^r_{\Lambda_K} (X_{k, S_0} (\cT))_\p \cdot 
        \big( {\prod}_{v \in S(K) \setminus S_0 (K)} \Eul_v (\Frob_v^{-1}) \big) \cdot
        (c_{K_\infty})_\chi \\
        & =  \Fitt^r_{\Lambda_{K_\psi}} (X_{S_0} (\cT_\psi))_{\wp_\psi}  \cdot 
        \big( {\prod}_{v \in S(K) \setminus S_0 (K)} \Eul_v (\Frob_v^{-1}) \big) \cdot (\NN_{K_\infty / K_{\psi, \infty }}(c_{K_\infty}))_\chi \\ 
         & = \Fitt^r_{\Lambda_{K_\psi}} (X_{k, S_0} (\cT_\psi))_{\wp_\psi} \cdot \big( {\prod}_{v \in S_0 (K) \setminus S_0 (K_\psi)} \Eul_v (\Frob_v^{-1}) \big) 
        \cdot  (c_{K_{\psi, \infty}})_\chi\\
        & =  \Fitt^r_{\Lambda_{K_\psi}} (X_{S (K)} (\cT_\psi))_{\wp_\psi} 
        \cdot  (c_{K_{\psi, \infty}})_\chi \\ 
        & \subseteq  \Theta_{K_{\psi, \infty}, \Sigma, b_\bullet} ( \Det_{\Lambda_{K_\psi}} (C_{S (K), \Sigma} (\cT_\psi)))_{\chi,\wp_\psi} \\
        & =  \Theta_{\Sigma, b_\bullet} ( \Det_{\Lambda_K} (C_{S (K), \Sigma} (\cT)))_{\chi,\p}.
    \end{align*}
Here the first equality follows from the fact that the trace element $\NN_{K_\infty / K_{\psi, \infty }}$ of $\cO [\Delta_K]$ is a unit in $\Lambda_{K, \p}$ (as $p \notin \p$), the second from the Euler system distribution relations, the third from (\ref{calculations with Xs}), and the inclusion from (\ref{reduction step inclusion}) with $K$ replaced by $K_\psi$. This completes the proof that claim (i) of Theorem \ref{new iwasawa theory euler systems main result} is implied by the inclusion (\ref{reduction step inclusion}). \\
To similarly derive claim (ii) of Theorem \ref{new iwasawa theory euler systems main result}, we observe first that for any subring $R$ of $\cQ_K$ with $\Lambda_K\subseteq R$, subsequent applications of 
Lemma \ref{standard fitting props}\,(iii) and (v) imply an inclusion
\begin{align} \nonumber 
\Fitt^r_{R} ( R \otimes_{\Lambda_K} Y_{\Pi_k^\infty} (\cT)) & = \Fitt^0_R \big( R \otimes_{\Lambda_K} \ker  ( Y_{\Pi_k^\infty} (\cT) \to Y) \big) \\ \nonumber
& \subseteq \Ann_{\cQ_K} \big( \cQ_K \otimes_{\Lambda_K} \ker  ( Y_{\Pi_k^\infty} (\cT) \to Y) \big)\\  
& = \cQ_K \epsilon_K,
\label{Fitting ideal and idempotent}
\end{align}
which is an equality if $R = \cQ_K$.
Now the $\Lambda_K$-module $X_{S (K)\setminus \Pi_k^\infty} (\cT)$ is torsion because no finite place of $k$ splits completely in $k_\infty$ and hence 
\[
\cQ_K \otimes_{\Lambda_K} \Fitt^0_{\Lambda_K} (X_{S (K)\setminus \Pi_k^\infty} (\cT)) = \Fitt^0_{\cQ_K} ( \cQ_K \otimes_{\Lambda_K} X_{S (K)\setminus \Pi_k^\infty} (\cT)) = \Fitt^0_{\Lambda_K}(0) = \cQ_K.
\]
This fact combines with 
the exact sequence 
\[ 0 \to X_{S (K)\setminus \Pi_k^\infty} (\cT) \to X_{S (K)} (\cT) \to Y_{k, \Pi_k^\infty} (\cT) \to 0\] 
(from Remark \ref{X=Y remark}) and (\ref{Fitting ideal and idempotent}) to imply that
\begin{equation} \label{Fitting ideal and idempotent 2}
\cQ_K \otimes_{\Lambda_K} \Fitt^r_{\Lambda_K} ( X_{S (K)} (\cT))=  \cQ_K \epsilon_K.
\end{equation}
From the (assumed) inclusion (\ref{reduction step inclusion}) we therefore obtain
\[
\epsilon_K (c_{K_\infty})_\chi \in \cQ_K \cdot \Theta_{\Sigma, b_\bullet} ( \Det_{\Lambda_K} (C_{S, \Sigma} (\cT))) = (\cQ_K e_K \epsilon_K) \otimes_{\Lambda_K} \bidual^r_{\Lambda_K} H^1_\Sigma (\cO_{k, S(K)}, \cT),
\]
where the equality follows from the explicit definition (\ref{definition Theta map}) of the map $\Theta_{ \Sigma, b_\bullet}$.
This shows that the idempotent $e_K$ acts as the identity on $\epsilon_K (c_{K_\infty})_\chi$ and so, because the definition of $e_K$
ensures that it annihilates $\cQ_K \otimes_{\Lambda_K} H^2_\Sigma (\cO_{k, S(K)}, \cT)$,
we see that
$\cQ_K \otimes_{\Lambda_K} H^2_\Sigma (\cO_{k, S(K)}, \cT)_\chi$ can only be supported on primes $\p$ of $\cQ_K$ at which $\epsilon_K (c_{K_\infty})_\chi$ vanishes upon localisation. In addition, the first column in the diagram (\ref{new diagram}) gives an exact sequence 
\begin{equation} \label{ingredient 1 ses}
\begin{tikzcd}
    0 \arrow{r} & H^1_{\cF^\vee_{\mathrm{rel}, \Sigma}} (k, \cT^\vee (1))^\vee \arrow{r} & 
    H^2_\Sigma (\cO_{k, S (K)}, \cT) \arrow{r} & X_{S (K) \setminus \Pi_k^\infty} (\cT) \arrow{r} & 0
    \end{tikzcd}
\end{equation}%
that combines with the vanishing of $\cQ_K \otimes_{\Lambda_K} X_{S (K)\setminus \Pi_k^\infty} (\cT)$ to imply an equality 
\[ \cQ_K \otimes_{\Lambda_K} H^2_\Sigma (\cO_{k, S(K)}, \cT) = \cQ_K \otimes_{\Lambda_K} H^1_{\cF^\vee_{\mathrm{rel}, \Sigma}} (k, \cT^\vee (1))^\vee.\] 
In view of this, our argument has therefore shown that claim (ii) of Theorem \ref{new iwasawa theory euler systems main result} is also a consequence of the inclusion (\ref{reduction step inclusion}).\\
Finally, we must derive claim (iii)  of Theorem \ref{new iwasawa theory euler systems main result} from (\ref{reduction step inclusion}). To do this, we note that, for every $\p \in \mathrm{Spec}^1(\Lambda_K)$, Proposition \ref{eagon-northcott-prop}\,(ii) implies an equality
\begin{equation}\label{conversion result}
\{ f (a) \mid a \in \im (\Theta_{\Sigma, b_\bullet}), f \in \exprod^r_{\Lambda_K} H^1_\Sigma (\cO_{k, S(K)}, \cT)^\ast \}_\p = \Fitt^r_{\Lambda_K} ( H^1 (C_{S(K), \Sigma} (\cT)))_\p.
\end{equation}
We also recall that the lower row of the diagram (\ref{new diagram}) provides us with an exact sequence
\begin{equation} \label{ingredient 2 ses}
\begin{tikzcd}
    0 \arrow{r} & H^1_{\cF^\vee_{\mathrm{rel}, \Sigma}} (k, \cT^\vee (1))^\vee \arrow{r} & H^1 (C_{S (K), \Sigma} (\cT)) \arrow{r} & X_{S(K)} (\cT) \arrow{r} & 0.
\end{tikzcd}
\end{equation}%
Now, if $\p\in \mathrm{Spec}^1(\Lambda_{K,\chi})\subseteq \mathrm{Spec}^1(\Lambda_{K})$ with $p \notin \p$, then $\Lambda_{K, \p}$ is a discrete valuation ring and so $\Lambda_{K, \p}$-Fitting ideals are multiplicative on short exact sequences. Hence, in this case, for $f \in \exprod^r_{\Lambda_K} H^1_\Sigma (\cO_{k, S (K)}, \cT)^\ast$, 
the (assumed) inclusion (\ref{reduction step inclusion}) combines with (\ref{conversion result}) and (\ref{ingredient 2 ses}) to imply that
one has  
\begin{align*}
\Fitt^r_{\Lambda_K} (X_{S (K)} (\cT))_\p \cdot f (c_K)_\p & \subseteq \Fitt^r_{\Lambda_K} ( H^1 (C_{S (K), \Sigma} (\cT)))_\p\\
& =  \Fitt^r_{\Lambda_K} (X_{S (K)} (\cT))_\p \cdot \Fitt^0_{\Lambda_K} ( H^1_{\cF^\vee_{\mathrm{rel}, \Sigma}} (k, \cT^\vee (1))^\vee)_\p.
\end{align*}
Furthermore, the ideal $\Fitt^r_{\Lambda_K} (X_{S (K)} (\cT))$ spans $\cQ_K \epsilon_K$ over $\cQ_K$ by (\ref{Fitting ideal and idempotent 2}), and so the principal ideal $\Fitt^r_{\Lambda_K} (X_{S (K)} (\cT))_\p$ is generated by a nonzero divisor in $\epsilon_K \Lambda_{K, \p}$. This shows that the $\Lambda_{K,\p}$-module $\Fitt^r_{\Lambda_K} (X_{S (K)} (\cT))_\p$ is invertible in $\epsilon_K \Lambda_{K, \p}$. The above inclusion therefore implies, after cancellation, the required containment
\begin{equation}\label{req containment}
\epsilon_K f (c_K)_\p \in \epsilon_K \Fitt^0_{\Lambda_K} ( H^1_{\cF^\vee_{\mathrm{rel}, \Sigma}} (k, \cT^\vee (1))^\vee)_\p.
\end{equation}
We next prove the same result for $\p\in \mathrm{Spec}^1(\Lambda_{K,\chi})$ with $p \in \p$. For this, we note the assumption that no finite place splits completely in $k_\infty$ combines with Lemma \ref{singular primes and mu vanishing} to imply, for each such $\p$,  that
$X_{S (K)\setminus \Pi_k^\infty} (\cT)_\p=(0)$. 
Since the assumption (\ref{p=2 condition}) ensures that $Y_{\Pi_k^\infty} (\cT)$ is a projective $\Lambda_K$-module, the vanishing of $X_{S (K)\setminus \Pi_k^\infty} (\cT)_\p$ combines with the sequence (\ref{ingredient 2 ses}) to imply that $H^1 (C_{S (K), \Sigma} (\cT))_\fp$ is isomorphic to the direct sum of $(H^1_{\cF^\vee_{\mathrm{rel}, \Sigma}} (k, \cT^\vee (1))^\vee)_\p$ and $Y_{\Pi_k^\infty} (\cT)_\p$. Recalling (\ref{Fitting ideal and idempotent}), we then see that the ideal $\Fitt^r_{\Lambda_K} ( H^1 (C_{S (K), \Sigma} (\cT)))_\p$ is contained in $\epsilon_K \Fitt^0_{\Lambda_K} H^1_{\cF^\vee_{\mathrm{rel}, \Sigma}} (k, \cT^\vee (1))^\vee)_\p$.
 For each such $\p$, the required containment therefore follows directly from (\ref{reduction step inclusion}) and the equality (\ref{conversion result}).\\ 
At this stage, we have verified (\ref{req containment}) for every prime $\p \in \mathrm{Spec}^1(\Lambda_{K,\chi})$. From Lemma \ref{ryotaro-useful-lem}\,(ii), we can therefore deduce that $\epsilon_K f (c_K)_\chi\in \epsilon_K \Fitt^0_{\Lambda_K} ( H^1_{\cF^\vee_{\mathrm{rel}, \Sigma}} (k, \cT^\vee (1))^\vee)^{\ast \ast}$, as required to derive claim (iii) of Theorem \ref{new iwasawa theory euler systems main result} from (\ref{reduction step inclusion}).\\ 
This completes the proof that Theorem \ref{new iwasawa theory euler systems main result} is a consequence of the assumed validity of the inclusion (\ref{reduction step inclusion}). \end{proof} 

\subsubsection{Observations on an application of Theorem \ref{new strategy main result}} 
 Proposition \ref{towards the new iwasawa theory euler systems main result} has reduced the proof of Theorem \ref{new iwasawa theory euler systems main result} to showing that the inclusion (\ref{reduction step inclusion}) is valid under the hypotheses of the latter result. To do this, we set $\cT_\chi \coloneqq T (\chi) \otimes_\cO \Lambda_K$ and
    write
\[ \mathrm{Tw}^r_\chi \: \ES^r_{\Sigma, S_0} (T,\cK) \to \ES^r_{S_0} (\scrF_{\mathrm{rel}, \Sigma}(\cT_\chi))\] 
for the map from Proposition \ref{twisting lemma}\,(ii). We shall then aim to apply Theorem \ref{new strategy main result} to the Euler system $\mathrm{Tw}^r_\chi (c)$ and so must 
    specify the data for which we can verify all of the hypotheses that are necessary to apply Theorem \ref{new strategy main result}.\\
  In this regard, we note firstly that the field $\cK^{\Sigma, p}$ clearly satisfies Hypothesis \ref{more new hyps}\,(i).
    We next verify Hypothesis \ref{more new hyps}\,(ii) in the relevant cases by fixing $\q \in \cQ_1$ and a non-negative integer $a$ and showing that the endomorphism $\Frob_\q^{p^a} - 1$ is injective on $\cT_\chi$. To do this, we note $\q$ is unramified in $k_\infty$ and write $\sigma_\q$ for the image of $\Frob_\q$ in $\Gamma \coloneqq \gal{k_\infty}{k}$. We recall $k_\infty$ contains a $\Z_p$-extension $k_\infty^\circ = \bigcup_{n \in \N} k_n^\circ$ of $k$ in which no finite place splits completely. In particular, the image of $\sigma_\q^{p^a}$ in $\Gamma^\circ \coloneqq \gal{k_\infty^\circ}{k}$ is a topological generator of $(\Gamma^\circ)^{p^{m (a)}} = \gal{k_\infty^\circ}{k_{m (a)}^\circ}$ for some $m (a) \geq 0$ and we set $\Gamma' \coloneqq \gal{k_\infty}{k_{m(a)}^\circ}$. Then one has $\sigma_\q^{p^a} \in \Gamma'$ and so $\cT_\chi$ decomposes, as an $\cO_\chi[\![\langle \Frob_\q^{p^a} \rangle ]\!]$-module, as the direct sum of a finite number of copies of $\cT'_\chi \coloneqq T(\chi)\otimes_{\mathcal{O}}\mathcal{O}[\![\Gamma']\!]$. It is thus enough for us to prove that $\Frob_\q^{p^a} - 1$ acts injectively on $\cT'_\chi$. Hence, if necessary after replacing $k$ by $k_{m(a)}$, we can assume in the sequel that $a=0$ and the image of $\sigma_\q$ generates $\Gamma^\circ$. In this case, the projection map $\Gamma \to \Gamma^\circ$ induces an isomorphism between $\Gamma^\circ$ and the subgroup $\widetilde{\Gamma^{\circ}}$ of $\Gamma$ that is generated (topologically) by $\sigma_\q$. Writing $k_\infty'$ for the fixed field of $k_\infty$ under $\widetilde{\Gamma^\circ}$ and $\Gamma'$ for $\gal{k'_\infty}{k}$ one therefore has a natural decomposition 
    \begin{equation} \label{Gamma decomposition}
    \Gamma = \widetilde{\Gamma^{\circ}}\times \gal{k_\infty}{k^\circ_\infty} \cong \Gamma^\circ \times \Gamma'.
    \end{equation} 
This splitting in turn induces isomorphisms
    \[
    \Lambda_k = \cO \llbracket \Gamma \rrbracket \cong \cO \llbracket \Gamma^\circ \rrbracket \otimes_\cO \cO \llbracket \Gamma' \rrbracket 
    \quad \text{ and } \quad \cT_\chi = T_{k_\infty / k} (\chi) \cong T_{k_\infty^\circ / k} (\chi) \otimes_\cO \cO \llbracket \Gamma' \rrbracket 
    \]
 in such a way that $\sigma_\q$ acts trivially on the factor $\cO \llbracket \Gamma' \rrbracket$ in both tensor products. We now write $E$ for the completion of $k$ at $\q$. Then the completion of $k_\infty^\circ$ at a place above $\q$ is the unique unramified $\Z_p$-extension $E^\mathrm{unr} = \bigcup_{n \in \N} E^\mathrm{unr}_n$ of $E$ and $\sigma_\q$ induces the Frobenius automorphism in $\gal{E^{\mathrm{unr}}}{E}$. Upon applying Shapiro's lemma, one therefore obtains an  isomorphism
    \begin{align*}
    (\cT_\chi)^{\Frob_\q = 1} & = H^0_f (E, \cT_\chi) 
    = H^0_f (E, T_{k_\infty^\circ / k} (\chi)) \otimes_\cO  \cO \llbracket \Gamma' \rrbracket \\
    & = H^0 ( E^\mathrm{unr}, T(\chi)) 
    \otimes_\cO  \cO \llbracket \Gamma' \rrbracket \\
    & \cong 
     \big( {\varprojlim}_{n \in \N} H^0 (E^\mathrm{unr}_{n}, T(\chi) ) \big)
     \otimes_\cO  \cO \llbracket \Gamma' \rrbracket .
    \end{align*}
    In addition, since $T(\chi)$ is a finitely generated $\cO$-module and $E^\mathrm{unr}$ is an infinite $p$-extension of $E$, the limit $\varprojlim_{n \in \N} H^0 (E^\mathrm{unr}_{n}, T(\chi) )$ vanishes (by \@ \cite[App.~B, Lem.~3.2]{Rubin-euler}), as required to prove the  validity of Hypothesis \ref{more new hyps}\,(ii) for $\cT_\chi$. \\
    In the notation of \S\,\ref{general set up section} we now take $\cR = (\Lambda_K)_\chi$ so that $\mathbb{K} = \cO_\chi / (\varpi)$. In particular, $\overline{\cT_\chi} = \overline{T (\chi)}$ and so the validity of
    the parts of Hypotheses \ref{new strategy hyps}\,(i) and (iv) that concern $\cT_\chi$ follow from the assumed validity of Hypotheses \ref{new strategy iwasawa hyps} for $T (\chi)$. In addition, Hypothesis \ref{new strategy hyps}\,(v) is valid by Lemma \ref{torsion lemma new} and the assumption that $H^1_\Sigma (\cO_{K, S(K)}, T)$ is $\cO$-torsion free. \\
    Next we note that the field $k_\infty$ from \S\,\ref{general set up section} is a subfield of $\mathfrak{E}$, and that the fields in (\ref{minimal fields}) for $\cT_\chi$ can be taken in $\mathfrak{E}_T^\chi$ so that the field $k (\cT_\chi)_\infty$ defined in  \S\,\ref{general set up section} agrees with $\mathfrak{E}_T^\chi$. Given this, the validity of Hypotheses \ref{new strategy hyps}\,(vii) in this case is clear (cf.\@ Remark \ref{vanishing remark}\,(vi)), and that of Hypotheses \ref{new strategy hyps}\,(ii), (ii${}^\ast$), and (iii) follows from the assumed validity of Hypotheses \ref{new strategy iwasawa hyps} for the data $(T, k_\infty, F, \chi, r)$.\\
    It still remains for us to specify the data of a morphism $\varrho \: \cR \to R$ of the form required for an application of Theorem \ref{new strategy main result}. Before doing this, however,  we observe that the required inclusion (\ref{reduction step inclusion})  will follow if we can prove that  
\begin{equation}\label{sufficient inclusion}\Fitt^r_{\Lambda_K} ( X_{S (K)} (\cT)) \cdot (c_{K_\infty})_\chi \subseteq \vartheta_Y ( \Det_{\Lambda_K} (C_{S (K), \Sigma} (\cT)))_\chi.\end{equation}
Indeed, to deduce   (\ref{reduction step inclusion}) from this, one need only  recall  $\vartheta_Y$ coincides with the map 
$\Theta_{\Sigma, b_\bullet}$ (defined in (\ref{definition Theta map})) as a consequence of \cite[Prop.\@ A.11\,(ii)]{sbA}. 

In addition, since 
the $\Lambda_K$-module $\vartheta_Y ( \Det_{\Lambda_K} (C_{S (K), \Sigma} (\cT)))_\chi$ is reflexive (as $\Det_{\Lambda_K} (C_{S (K), \Sigma} (\cT))_\chi$ is invertible and hence reflexive), 
Lemma \ref{ryotaro-useful-lem}\,(ii) implies that it is enough to verify (\ref{sufficient inclusion}) after localising at each prime in ${\mathrm{Spec}}^1((\Lambda_K)_\chi)$.  

For the remainder of the proof we shall therefore fix $\p \in {\mathrm{Spec}}^1((\Lambda_K)_\chi).$ It is then convenient to separate the discussion into two cases and, in this way, the proof of Theorem~\ref{new iwasawa theory euler systems main result} will be  completed by the arguments given in \S\,\ref{proof for singular primes section} for the case $p \in \p$ and in \S\,\ref{proof for regular primes section I} and \S\,\ref{proof for regular primes section II} for the case $p \notin \p$. 

\begin{remark}\label{rank one reduction} By assumption, $k_\infty$ contains a $\Z_p$-extension $k_\infty^\circ$ of $k$ in which no finite place of $k$ splits completely. To prove (\ref{sufficient inclusion}), one can in fact further reduce to the case that $k_\infty = k_\infty^\circ$. The reason for this is that  
 $\im (\Theta_{\Sigma, b_\bullet})$ is equal to ${\varprojlim}_E  \im (\Theta_{\Sigma, b_{E, \bullet}})$, where the limit is taken over all finite extensions $E$ of $k$ in $k_\infty$, and $b_{E, \bullet}$ denotes the $\cO \llbracket \cG_{Ek^\circ_\infty} \rrbracket$-basis  of $Y \otimes_\cR \cO \llbracket \cG_{Ek^\circ_\infty} \rrbracket$ that is induced by $b_\bullet$. In particular, if $c_{Ek^\circ_\infty}$ belongs to $\im(\Theta_{\Sigma, b_{E, \bullet}})$ for all such $E$, then it follows that $c_{k_\infty} = (c_{Ek^\circ_\infty})_F$ belongs to $\im(\Theta_{\Sigma, b_\bullet})$, as required. Since the validity of Hypothesis \ref{new strategy iwasawa hyps} for $(T, k_\infty, F, \chi, r)$ also implies its validity for $(T, k^\circ_\infty, EF, \chi, r)$, we may therefore assume at the outset that $k_\infty = k^\circ_\infty$. However, we prefer to enforce this reduction step only at a later point in order to clarify that a large part of the argument does not rely on it. 
\end{remark}

\subsubsection{Verifying the $\p$-localisation of (\ref{sufficient inclusion}) if $p \in \p$} \label{proof for singular primes section}

    For each $i \in \mathbb{N}$, we write $k^\circ_i$ for the unique subextension of $k_\infty^\circ$ of degree $p^i$ over $k$ and $U^\circ_i$ for the subgroup $\gal{k_\infty}{k^\circ_i}$. We can then fix a basis $\{ U_i \}_{i \in \N}$ of open neighbourhoods of the identity of $\cG_{k_\infty} \cong \Gamma_{\!\!K}$ (where $\Gamma_{\!\!K}$ is as in (\ref{Gamma definition})) in such a way that $U_i \subseteq U^\circ_i$ for every $i$. 
    
    As in Remark \ref{ryotaro example}\,(iii), we now take $\cR \coloneqq (\Lambda_K)_\chi$ and $\a_i \coloneqq (\varpi^i, I (U_i))\cR$ with $I(U_i)$ the kernel of the natural projection $\Lambda_K \to \ZZ_p[\cG_{K_\infty}/U_i]$. In particular, for each $i$, the ring $\cR_i \coloneqq \cR/\a_i$ is naturally isomorphic to $(\cO_\chi / (\varpi^i)) [\Gamma_{\!\!K} / U_i] [\square_K]$. We write $\kappa_\chi \coloneqq \cO_\chi / (\varpi)$ for the residue field of $\cO_\chi$, and then further take $R \coloneqq \kappa_\chi \llbracket \Gamma^\circ \rrbracket$ with $\Gamma^\circ \coloneqq \cG_{k_\infty^\circ}$ 
    and $R_i \coloneqq \kappa_\chi [\cG_{k^\circ_i}] = \kappa_\chi  [\cG_{k_\infty} / U^\circ_i]$.
     The residue field of $R$ is therefore also equal to $\kappa_\chi$.
      Furthermore, each $R_i$ is a principal local ring because, as a power series ring over a field, $R$ is a principal ideal domain. 
      In view of this, Lemma~\ref{Euler characteristic if qadratically presented} implies that, for the relaxed Mazur--Rubin structure $F_i$ on $\cT_\chi \otimes_\cR \cR_i = T (\chi) / \omega^i T (\chi)$, the core-rank $\bm{\chi} (\overline{F_i}, j(i)) $ of the pair $(\overline{F_i}, j(i))$ (in the sense of (\ref{definition core rank})) is greater than or equal to  
\begin{align*}
&\dim_{\kappa_\chi}(Y_{\Pi_k^\infty} ( \Tbar (\chi))) +  \dim_{\kappa_\chi}(\sha_{\overline{F_i}^\ast, j(i)} (\Tbar^\ast (1)(\chi^{-1})))  - \dim_{\kappa_\chi} (\sha_{\overline{F_i}, j(i)} (\Tbar (\chi))) \\
& \qquad +\dim_{\kappa_\chi} \big( H^0(k,\Tbar (\chi))\big) - \dim_{\kappa_\chi} \big( H^0(k,\Tbar^\ast (1)(\chi^{-1}))\big) - {\sum}_{v \in \Sigma} \dim_{\kappa_\chi} (H^0 (k_v, \Tbar^\ast (1)(\chi^{-1}))).
\end{align*}
We now recall (from Remark \ref{invariants vanishing remark}) that Hypothesis \ref{new strategy iwasawa hyps}\,(iii) implies the group $H^0 (k, \Tbar^\ast (1)(\chi^{-1}))$ vanishes. In addition, since $\sha_{\overline{F_i}^\ast, j(i)} (\Tbar^\ast (1)(\chi^{-1}))$ is a submodule of $H^1 (\mathfrak{E}_T^\chi / k, \Tbar^\ast (1)(\chi^{-1}))$ (by Lemma \ref{general neukirch interpretation lemma}\,(i)), this module also vanishes as a consequence of Hypothesis \ref{new strategy iwasawa hyps}\,(iii). In a similar way, Lemma \ref{general neukirch interpretation lemma}\,(i) identifies $\sha_{\overline{F_i}, j(i)} (\Tbar (\chi))$ with a submodule of $H^1 (\mathfrak{E}_T^\chi/ k, \Tbar (\chi))$. In particular, at this point the explicit upper bound on the $\kappa_\chi$-dimension of $H^1 (\mathfrak{E}_T^\chi/ k, \Tbar (\chi))$ given by Hypothesis \ref{new strategy iwasawa hyps}\,(v) can be combined with the above displayed lower bound to deduce that $\bm{\chi} (\overline{F_i}, j(i))$ is strictly positive, as required to verify Hypothesis \ref{new strategy hyps}\,(vi) in this case.
\\
We next take $\varrho \: \cR \to R$ to be the natural surjective projection map, and write $\mathfrak{P}$ for the prime ideal $\ker(\varrho)$ of $\cR$. 
Then $\varrho_\mathfrak{P}$ is surjective and nonzero, as required by condition (i) in Theorem \ref{new strategy main result}. To verify condition (ii) in Theorem \ref{new strategy main result}, we first note that the vanishing of $H^0 (k, \Tbar (\chi)^\ast (1))$ implies $X_{S (K)} (\cT_\chi) = Y_{S (K)} (\cT_\chi)$ (cf.\@ Remark \ref{X=Y remark}) and hence, via the argument of Lemma \ref{limit Fitt X}\,(i), that there exists an isomorphism 
\[
X_{S (K)} (\cT_\chi) \otimes_{\Lambda_K} \cO \llbracket \Gamma^\circ \rrbracket 
= Y_{S (K)} (\cT_\chi) \otimes_{\cR} \cO \llbracket \Gamma^\circ \rrbracket 
\cong Y_{S (K)} (T (\chi) \otimes_\cO \cO \llbracket \Gamma^\circ\rrbracket).
\]
Now, the module $Y_{S (K)} (T (\chi) \otimes_\cO \cO \llbracket \Gamma^\circ\rrbracket)$ is finitely generated over $\cO_\chi$ because no finite place splits completely in $k^\circ_\infty$ so that the above isomorphism implies $X_{S (K)} (\cT_\chi)$ is finitely generated over $\cO_\chi \llbracket \Gamma' \rrbracket$ (with $\Gamma'$ as in (\ref{Gamma decomposition})). 
As a consequence of Lemma \ref{singular primes and mu vanishing}, the module $X_{S (K)} (\cT_\chi)$ is therefore annihilated by an element $x \in \cR$ that does not belong to the ideal $(\varpi) = \ker ( \cR \to R)$. It follows that  $X_{S (K)} (\cT_\chi) \otimes_{\cR} R$ is an $R$-torsion module and hence vanishes when localised at $(0)$ as an $R$-module (respectively, at $\mathfrak{P}$ as an $\cR$-module).
This shows, in particular, that $\Tor_1^\cR (X_{S (K)} (\cT_\chi), R)_\mathfrak{P}$ vanishes and hence implies that  
        \[
        \Fitt^0_R (\Tor_1^\cR (X_{S (K)} (\cT_\chi), R))_\mathfrak{P} = \Fitt^0_{R_\mathfrak{P}} ( \Tor_1^{\cR} (X_{S (K)} (\cT_\chi), R)_\mathfrak{P}) = \Fitt^0_{R_\mathfrak{P}} (\{0\}) = R_\mathfrak{P},
        \]
    as required by condition (ii) in Theorem \ref{new strategy main result}. In this case, therefore, the latter result can be applied in order to deduce that, for every Euler system $\eta$ in $\ES^r_{\Sigma, S_0} (\cT_\chi)$, the element $(\eta_k)_\mathfrak{P}$ belongs to 
    \[ \vartheta_Y (\Det_{\cR} (C_{S (K), \Sigma}  (\cT_\chi)))_\mathfrak{P} = \vartheta_Y (\Det_{\Lambda_K} (C_{S (K), \Sigma}  (\cT)))_\mathfrak{P}.\] 
Since $\p$ is contained in $\mathfrak{P}$, this in turn implies that $(\eta_k)_\p \in \vartheta_Y (\Det_{\cR} (C_{S (K), \Sigma}  (\cT)))_\p$ and hence verifies the $\p$-localisation of (\ref{sufficient inclusion}) when applied to $\eta = \mathrm{Tw}^r_\chi (c)$, as required.     %

\subsubsection{Verifying the $\p$-localisation of (\ref{sufficient inclusion}) if $p \notin \p\!: \!$ reduction to a key inequality} \label{proof for regular primes section I}
    
    To deal with the case $p \not \in \p$, we shall adapt an approach used by Mazur and Rubin in \cite[\S\,5.3]{MazurRubin04}. At the outset, we note that, for some character $\psi \: \square_K \to \Phi^{\mathrm{c}, \times}$ and prime ideal $\wp\in \Spec^{1}(\Lambda_{\chi \psi})$, the localised ring $\Lambda_{\chi \psi, \wp}$ coincides with $\Lambda_{K, \p}$. As a consequence, if we set $K_{\chi \psi} \coloneqq K^{\ker (\chi \psi)}$ and $K_{\chi \psi, \infty} \coloneqq K_{\chi \psi} \cdot k_\infty$, then one obtains an identification
    \[
    \Theta_{\Sigma, b_\bullet} ( \Det_{\Lambda_K} ( C_{S (K), \Sigma} (\cT)))_\p 
    = \Theta_{K_{\chi \psi, \infty}, \Sigma, b_{\chi \psi, \bullet}} ( \Det_{\Lambda_{\chi \psi}} ( C_{S (K), \Sigma} ( T (\chi \psi) \otimes_{\cO_\psi} \Lambda_\psi)))_\wp
    \]
in which $b_{\chi \psi, \bullet}$ is the basis of $Y \otimes_{\Lambda_K} \Lambda_{\chi \psi}$ induced by $b_\bullet$ and $\cT_{\chi \psi} \coloneqq T (\chi \psi) \otimes_{\cO_{\chi \psi}} \Lambda_{\chi \psi}$. \\
    On the other hand, in terms of the twisting map $\mathrm{Tw}^r_{\chi \psi} \: \ES^r_\Sigma (T) \to \ES^r_\Sigma (T (\chi \psi), \cK^\Sigma)$ from Proposition~\ref{twisting lemma}, the result of \cite[Ch.~II, Lem.\@ 4.3]{Rubin-euler} implies that  
\[ \mathrm{Tw}^r_{\chi \psi }(c)_{k_\infty} = \sum_{\sigma \in \Delta_K} (\chi \psi) (\sigma) \sigma (c_{K_\infty}) \in 
    \big( \bidual^r_{\Lambda_K} H^1_\Sigma (\cO_{k, S(K)} , \cT) \big)_\p 
    = \big( \bidual^r_{\Lambda_\psi} H^1_\Sigma (\cO_{k, S(K)} , \cT_{\chi \psi}) \big)_\wp.
    \]
Hence, since the element $\sum_{\sigma \in \Delta_K} (\chi \psi) (\sigma) \sigma$ acts as multiplication by the unit $|\Delta_K|$ of $\Lambda_{K, \p}$, and  $\Fitt^r_{\Lambda_K} ( X_{S (K)} (\cT))_\p$ identifies with $\Fitt^r_{\Lambda_{\chi \psi}} ( X_{S (K)} (\cT_{\chi \psi}))_\wp$, the claimed result will follow if we can show that 
    \[
    \Fitt^r_{\Lambda_{\chi \psi}} ( X_{S (K)} (\cT_{\chi \psi})) \cdot 
    \mathrm{Tw}^r_{\chi \psi} (c)_{k_\infty} \in \Theta_{K_{\chi \psi, \infty}, \Sigma, b_{\chi \psi, \bullet}} ( \Det_{\Lambda_{\chi \psi}} ( C_{S (K), \Sigma} ( \cT_{\chi \psi})))_\wp.
    \]
    By replacing $c$ by $\mathrm{Tw}^r_{\chi \psi} (c)$ and $T$ by $T (\chi \psi)$, we may thereby reduce to the case $K = k$ and $\cR  = \cO_{\chi \psi} \llbracket \Gamma \rrbracket$. 
    Indeed, since $\psi$ is a character of $p$-power order,
    the ring $\cO_{\chi \psi}$ is a totally ramified extension of $\cO_\chi$ with residue field equal to the residue field $\kappa_\chi$ of $\cO_\chi$. In addition, writing $\varpi_{\psi \chi}$ for a uniformiser of $\cO_{\chi \psi}$, for every $\sigma \in G_k$ 
    one has a congruence $\psi (\sigma) \equiv 1\,\,(\mathrm{mod}\,\, \varpi_{\chi \psi})$  and so the residual representations 
    $\Tbar (\chi)$ and $\overline{T (\chi \psi)}$ coincide. This shows that the assumed validity of
    Hypotheses \ref{new strategy iwasawa hyps} for 
    $(T, k_\infty, F, \chi, r)$ implies its validity also for $(T (\chi \psi), k_\infty, F, \bm{1}, r)$. By replacing $\cO$ by $\cO_{\chi \psi}$ if necessary, we may then also assume that $\cO = \cO_{\chi \psi}$.\\
    Having made these reduction steps,
we shall next prove that the claimed inclusion 
\[
\Fitt^r_\cR (X_{S(K)} (\cT))_\p \cdot c_{k_\infty} \subseteq \Theta_{k_\infty, \Sigma, b_\bullet} ( \Det_\cR (C_{S(K), \Sigma} (\cT)))_\p
\]
is implied by the inclusion
\begin{equation} \label{we want to show this for dvrs}
\Fitt^r_\cR (X_{S(K)} (\cT))_\p \cdot \im (c_{k_\infty}) \subseteq \Fitt^r_\cR ( H^1 ( C_{S (K), \Sigma} (\cT)))_\p. 
\end{equation}
To justify this, we must show that, for any   element $x$ of $\Fitt^r_\cR (X_{S (K)} (\cT))$ and $\cR$-basis $\fz$ of $\Det_\cR (C_{S(K), \Sigma} (\cT))$, the latter inclusion implies  
$x\cdot c_{k_\infty} \in \cR_\p\cdot \Theta_{k_\infty, \Sigma, b_\bullet} (\fz)$. 
To do this, we can also assume that the right hand side of (\ref{we want to show this for dvrs}) is nonzero since otherwise this inclusion implies the module $\cR_\p(x\cdot c_{k_\infty})$  vanishes  and so the claimed inclusion is obviously satisfied. 

Now, by Lemma \ref{finite level reps}, the complex $C_{S(K), \Sigma} (\cT)$ is isomorphic to a complex of the form $P_0 \to P_1$ with finitely-generated free $\cR$-modules $P_0$ and $P_1$ that are placed in degree zero and one, respectively, and such that $P_1 \cong \ker (\pi) \oplus Y$ with $\pi \: P_1 \to H^1 ( C_{S(K), \Sigma} (\cT))$ the natural projection map. By applying Proposition \ref{eagon-northcott-prop}\,(ii)
to the complex $P_0 \to \ker (\pi)$ we then obtain an equality
\begin{equation} \label{eagon northcott prop consequence that we need here}
    \im ( \Theta_{k_\infty, \Sigma, b_\bullet} (\fz))_\p = \Fitt^r_\cR ( H^1 ( C_{S (K), \Sigma} (\cT)))_\p.
\end{equation}
This equality implies, in particular, that the idempotent $e_k \varepsilon_k$ appearing in the definition (\ref{definition Theta map})  of the map 
$\Theta_{k_\infty, \Sigma, b_\bullet}$ must act as the identity on the nonzero ideal $\Fitt^r_\cR ( H^1 ( C_{S (K), \Sigma} (\cT)))$. Since $\cR_\p$ is an integral domain, one therefore has $e_k \varepsilon_k = 1$ and so the definition of $\Theta_{k_\infty, \Sigma, b_\bullet}$ directly implies that its image spans 
 $\bidual^r_\cR H^1_\Sigma (\cO_{k, S(K)}, \cT)$ over the fraction field $\cQ (\cR)$ of $\cR$. We can therefore deduce the existence of an element $q$ of $\cQ (\cR)$ such that  $c_{k_\infty} = q\cdot \Theta_{k_\infty, \Sigma, b_\bullet} (\fz)$. At this point, the (assumed) inclusion (\ref{we want to show this for dvrs}) can be combined with (\ref{eagon northcott prop consequence that we need here}) to deduce that 
\begin{align*}
    q \cdot x \cdot \Fitt^r_\cR ( H^1 ( C_{S (K), \Sigma} (\cT)))_\p & = q \cdot x \cdot \im ( \Theta_{k_\infty, \Sigma, b_\bullet} (\fz)) \\
    & = x \cdot \im (c_{k_\infty}) \\
    & \subseteq \Fitt^r_\cR ( H^1 ( C_{S (K), \Sigma} (\cT)))_\p.
\end{align*}
Then, since $\cR_\p$ is a discrete valuation ring and $\Fitt^r_\cR ( H^1 ( C_{S (K), \Sigma} (\cT)))_\p$ is a nonzero ideal, cancellation now proves that $q \cdot x \in  \cR_\p$, and hence also the required containment via 
\[ x \cdot c_{k_\infty} = (q\cdot x)\cdot \Theta_{k_\infty, \Sigma, b_\bullet} (\fz) \in \cR_\p\cdot \Theta_{k_\infty, \Sigma, b_\bullet} (\fz),\] 
We are therefore now reduced to proving the inclusion (\ref{we want to show this for dvrs}). Further, since $\cR_\p$ is a discrete valuation ring, the short exact sequence (\ref{ingredient 2 ses}) implies that 
\[
\Fitt^r_\cR ( H^1 ( C_{S (K), \Sigma} (\cT)))_\p = \Fitt^r_\cR (X_{S(K)} (\cT))_\p \cdot \Fitt^0_\cR ( H^1_{\cF^\vee_{\mathrm{rel}, \Sigma}} ( k, \cT^\vee (1))^\vee)_\p,
\]
and so it is actually enough for us to show that, if $\Fitt^r_\cR (X_{S(K)} (\cT))_\p$ is nonzero, then 
\[
\im (c_{k_\infty}) \subseteq \Fitt^0_\cR ( H^1_{\cF^\vee_{\mathrm{rel}, \Sigma}} ( k, \cT^\vee (1))^\vee)_\p.
\]
If the image of  $c_{k_\infty}$ in  $H^1_\Sigma (\cO_{k, S(K)}, \cT)_\p$ vanishes, then there is nothing to show and so we may assume that this image is nonzero. In this case, then, Lemma \ref{lemma-image-Fitt}\,(ii) implies that
\begin{align*}
\im (c_{k_\infty})_\p & = \Fitt^0_{\cR} \big( \Ext^1_\cR \big ( \bidual^r_\cR H^1_\Sigma (\cO_{k, S(K)}, \cT) / \cR c_{k_\infty}, \cR\big ) \big)_\p \\
& = \Fitt^0_{\cR} \big( \big (\bidual^r_\cR H^1_\Sigma (\cO_{k, S(K)}, \cT) / \cR c_{k_\infty} \big)_\tor \big)_\p,
\end{align*}
where the second equality follows again from the fact that $\cR_\p$ is a discrete valuation ring (cf.\@ the argument of \cite[Prop.~5.5.13]{NSW} for details).\\
Finally, at this point it is convenient to assume $k_\infty = k_\infty^\circ$ and to complete the proof of Theorem \ref{new iwasawa theory euler systems main result} in this case. (We recall that, as observed in Remark \ref{rank one reduction}, the validity of Theorem \ref{new iwasawa theory euler systems main result} in this special case implies its validity in general.) 

Having made this reduction, the ring $\cR = \cO \llbracket \Gamma \rrbracket$ is then isomorphic to the power series ring $\cO \llbracket X \rrbracket$ in a single variable $X$.
We write $f$ for the irreducible Weierstra{\ss} polynomial in $\cO [ X ]$ that generates $\p$ when regarded as an element of $\cR$ via this isomorphism. 
We also fix generators $z_\mathrm{Sel}$ and $z_c$ of $\mathrm{char}_\cR ( H^1_{\cF^\vee_{\mathrm{rel}, \Sigma}} ( k, \cT^\vee (1))^\vee)$ and $\mathrm{char}_\cR ( (\bidual^r_\cR H^1_\Sigma (\cO_{k, S(K)}, \cT) / (\cR\cdot c_{k_\infty}))_\tor)$, respectively. 
At this stage our verification of the $\p$-localisation of (\ref{sufficient inclusion}) has therefore been reduced to the case $k_\infty = k_\infty^\circ$ and to showing that 
\begin{equation}\label{wanted inequality} \ord_f(z_c) \geq \ord_f(z_\mathrm{Sel}).\end{equation}

\subsubsection{Verifying the $\p$-localisation of (\ref{sufficient inclusion}) if $p \notin \p$: the key inequality} \label{proof for regular primes section II}

Following the above reduction, we assume in the remainder of the argument that $k_\infty = k_\infty^\circ$. In this case, in order to verify the required inequality (\ref{wanted inequality}) we can further assume that $f$ divides $z_\mathrm{Sel}$ since otherwise the inequality is true trivially. By enlarging $\cO$ if necessary (by passing to a faithfully flat extension of $\cO$), we may moreover assume that $f$ is a linear polynomial in $\cO [X]$. Recalling $\varpi$ denotes a uniformiser of $\cO$, we define
\[
f_n \coloneqq f + \varpi^n \in \cO [X]
\quad \text{ and } \quad \cR_n \coloneqq \cR / f_n \cR
\]
for every $n \in \N$. We then also obtain a Galois representation
\[
T_n \coloneqq \cT \otimes_\cR \cR_n \cong T (\psi_n).
\]
Here, if $\gamma$ is the topological generator of $\cG_{k_\infty}$ that induces the fixed isomorphism $\cR \cong \cO \llbracket X \rrbracket$ (via which we view $f_n$ as an element of $\cR$), then $\psi_n$ denotes the character 
\[
\cG_{k_\infty} \to \cO^\times, \,\,\, \gamma \mapsto 1 - (f (0) + \varpi^n).
\]

    \begin{lem} \label{control lemma}
    For all sufficiently large integers $n$, the following claims are valid.
    \begin{itemize}
        \item[(i)] The map $\mathrm{Tw}^r_{k, \psi_n}$ constructed in Proposition \ref{twisting lemma}\,(i) is injective and the order of its cokernel is bounded independently of $n$.
        \item[(ii)] The component $\mathrm{Tw}^r_{k, \psi_n} (c_{k_\infty})$ of $\mathrm{Tw}^r_{\psi_n} (c)_k$ at $k$ is nonzero and, in $\cR_n$, one has
        \[
        \im ( \mathrm{Tw}^r_{k, \psi_n} (c_{k_\infty})) \subseteq \Fitt^0_{\cR_n} ( H^1_{\cF_{\mathrm{rel}, \Sigma}} (k, T_n^\vee (1))^\vee).
        \]
        In particular, the group $H^1_{\cF_{\mathrm{rel}, \Sigma}} (\cO_{k, S(K)}, T_n^\vee (1))^\vee$ is finite. 
        \item[(iii)] The kernel and cokernel of the natural map 
    \[
    s_n \: H^1_{\cF^\vee_{\mathrm{rel}, \Sigma}} (k, \cT^\vee (1))^\vee \otimes_{\cR} \cR_n \to H^1_{\cF^\vee_{\mathrm{rel}, \Sigma}} (k, T_n^\vee (1))^\vee
    \]
    are finite and of order bounded independently of $n$.
    \item[(iv)] The polynomial $f_n$ does not divide the characteristic polynomial of the $\cR$-torsion submodule of $H^1_{\cF_{\mathrm{rel}, \Sigma}} (k, \cT^\vee (1))^\vee$. 
    \end{itemize}
\end{lem}

\begin{proof}
To prove (i), we take $ E = k$ in the spectral sequence (\ref{spec seq eq})
and observe that it degenerates on its second page (since the $\cR$-module $\cR_n$ has projective dimension one) to yield a short exact sequence 
    \[ 0 \to \left(\bidual_{\cR}^{r} H^1_\Sigma (\cO_{k,S(k)}, \cT)\right) \otimes_{\cR} \cR_n \to \bidual_{\cR_n}^{r} H^1(\cO_{k,S(k)}, T_n)
        \to H^1(C^{r}(\cT))[f_n] \to 0,
    \]
    thereby proving the claimed injectivity of  $\mathrm{Tw}^r_{k, \psi_n}$. Moreover,
this exact sequence also proves that the cokernel of $\mathrm{Tw}^r_{k, \psi_n}$ is isomorphic to $H^1(C^{r}(\cT))[f_n]$. In particular, if $n$ is chosen large enough such that the prime ideal $\cR f_n$ is not contained in the support of the $\cR$-torsion submodule of $H^1(C^{r}(\cT))$, then $H^1(C^{r}(\cT))[f_n]$ is finite of cardinality bounded by the order of the maximal finite $\cR$-submodule of $H^1(C^{r}(\cT))$. This proves (i). \\
    Next we note that, since the element $c_{k_\infty}$ is (by assumption) nonzero, and a nonzero element of a unique factorisation domain can only have finitely many irreducible factors, neither of the principal ideals that are given by $\im(c_{k_\infty})^{\ast \ast}$ and the characteristic ideal of the $\cR$-torsion submodule of $H^1_{\cF^\vee_{\mathrm{rel}, \Sigma}}(k,   \cT^\lor(1))^\lor$ can be contained in infinitely many of the ideals $\cR f_n$. In particular, for any large enough $n$,  the polynomial $f_n$ does not divide the characteristic polynomial of the $\cR$-torsion submodule of $H^1_{\cF^\vee_{\mathrm{rel}, \Sigma}}(k,  \cT^\lor(1))^\lor$ and, in addition, the element $\mathrm{Tw}^r_{k, \psi_n} (c_{k_\infty})$ is non-zero. The first of these properties immediately implies the property in (iv), whilst the second property implies that the image $\mathrm{Tw}^r_{\psi_n} (c)$ of $c$ inside $\ES^{r}(T_n,\cK^{\Sigma, p})$ is non-zero. \\
    We next claim that the representation $T_n$ satisfies Hypothesis \ref{new strategy hyps} (where, in the notation of Hypothesis \ref{new strategy hyps}, we are taking $\cT = T = T_n$) with the fields $k (T_{n, i})$ in (\ref{minimal fields}) taken to be the minimal extensions of $F$ such that $G_{k (T_{n, i})}$ acts trivially on both $T_i \coloneqq T / \varpi^i T$ and $T_{n, i} \coloneqq T_n / \varpi^i T_n$. To see this, we observe that the explicit definition of $\psi_n$ shows that $\psi_n \equiv 1 \, (\mathrm{mod}\, \varpi)$, and hence that $\Tbar$ and $\overline{T_n}$ coincide. Given this, and the fact that the field $k (T_n)_\infty$ that occurs in Hypothesis \ref{new strategy hyps}\,(iii) coincides with $\mathfrak{E}_T^{\bm{1}}$, one finds that the validity of Hypothesis \ref{new strategy hyps}\,(i), (ii), (iii) and (iv) follows from the assumed validity of Hypothesis \ref{new strategy iwasawa hyps}\,(i), (ii), (iii) and (iv) for the data $(T, k_\infty, F, \bm{1}, r)$. Next we recall the $\Sigma$-modified relaxed Nekov\'a\v{r}--Selmer structure $\mathscr{F}_{\mathrm{rel}, \Sigma} (T_{n, i})$ satisfies Hypothesis \ref{fF hyp} because $S (\mathscr{F}) \setminus \Pi_k^\infty$ is nonempty (as it contains $\Pi_k^p$) and $H^0 (k, \Tbar)$ vanishes (as observed in Remark \ref{vanishing remark}). This shows that Hypothesis \ref{new strategy hyps}\,(v) is also satisfied for $T_n$. Finally, since we are in the situation of Remark \ref{Euler characteristic if qadratically presented rk}\,(ii), we may use Lemma \ref{Euler characteristic if qadratically presented}\,(ii) to calculate the core rank $\bm{\chi} ( \overline{F_i}, j(i)) $ of the pair $( \overline{F_i}, j(i)) $ to be at least
    \begin{align*}
    {\sum}_{v \in \Pi_k^\infty} \dim_\kappa ( H^0 ( k_v, \Tbar^\ast (1))) - \dim_\kappa ( \sha_{\overline{F}, j} (\Tbar)) - {\sum}_{v \in \Sigma} ( H^0 (k_v, \Tbar^\ast (1))) \\
    \geq r - \dim_\kappa ( H^1 (\mathfrak{E}_T^{\bm{1}} / k, \Tbar )) - {\sum}_{v \in \Sigma} ( H^0 (k_v, \Tbar^\ast (1)))
    .
    \end{align*}
Here the inequality follows from the observations that the $r$-dimensional $\kappa$-space $Y \otimes_\cR \kappa$ is a quotient of $Y_{\Pi_k^\infty} ( \Tbar)$, and that  Lemma \ref{general neukirch interpretation lemma} implies $\sha_{\overline{F}, j} (\Tbar)$ identifies with a submodule of $H^1 (\mathfrak{E}_T^{\bm{1}} / k, \Tbar )$. Given this, the assumed validity of Hypothesis \ref{new strategy iwasawa hyps}\,(v) implies $\bm{\chi} ( \overline{F_i}, j(i)) > 0$, and hence that Hypothesis \ref{new strategy hyps}\,(vi) is valid. The validity of Hypothesis \ref{new strategy hyps}\,(vii) being clear, we have therefore verified all parts of Hypothesis \ref{new strategy hyps}. \\
Moreover, Proposition \ref{twisting lemma}\,(ii)\,(b) combines with Remark \ref{rk on avoiding more new hyps} to imply that we may apply Theorem \ref{new strategy main result} to the system $\mathrm{Tw}^r_{\psi_n} (c)$ without having to assume Hypothesis \ref{more new hyps} for $T_n$ (since, as already observed earlier, $\cT$ satisfies Hypothesis \ref{more new hyps}).
    The conclusion of Theorem \ref{new strategy main result} then implies the inclusion claimed in (ii) via Proposition \ref{eagon-northcott-prop}\,(ii). \\
 The key observation needed for the proof of (iii) is the existence of the natural exact commutative diagram
    \begin{cdiagram}[column sep=tiny, row sep=small]
        X_{S(K)} (\cT) [f_n] \arrow{r} & H^1_{\cF^\vee_{\mathrm{rel}, \Sigma}} (k, \cT^\vee (1))^\vee \otimes_{\cR}  \cR_n \arrow{r} \arrow{d} & H^2_\Sigma (\cO_{k, S(K)}, \cT) \otimes_{\cR} \cR_n \arrow[twoheadrightarrow]{r} \arrow{d}{\simeq} & X_{S(K)} (\cT) \otimes_{\cR} \cR_n  \arrow{d} \\
        0 \arrow{r} & H^1_{\cF^\vee_{\mathrm{rel}, \Sigma}} (k, T_n^\vee (1))^\vee \arrow{r} & H^2_\Sigma (\cO_{k, S(K)}, T_n) \arrow[twoheadrightarrow]{r} & X_{S(K)} (T_n).
    \end{cdiagram}%
    In particular, if $n$ is large enough such that $f_n$ is not in the support of $X_{S(K)} (\cT)$, then both of the modules $X_{S(K)} (\cT) [f_n]$ and $X_{S(K)} (\cT) \otimes_{\cR} \cR_n \cong  X_{S(K)} (\cT) / f_n X_{S(K)} (\cT)$ are finite and of order bounded (as $n$ varies) by the cardinality of the maximal finite submodule of the finitely generated $\cR$-module $X_{S(K)} (\cT)$. Given this, (iii) follows by applying the Snake Lemma to the above diagram. This completes the proof of all claims. 
\end{proof}

To prove the required inequality (\ref{wanted inequality}), it is convenient to first show that, if neither $c_{k_\infty}$ nor the ideal 
$\Fitt^r_\cR (X_{S(K)} (\cT))$ vanishes, then the quotient $\cR$-module  $\bigl(\bidual^r_\cR H^1_\Sigma (\cO_{k, S(K)}, \cT)\bigr)/(\cR\cdot  c_{k_\infty})$ 
is torsion and hence that the element $z_c$ occurring in (\ref{wanted inequality}) is (by its very definition) such that 
\begin{equation}\label{z_c interpretation}\cR\cdot z_c = \cchar_\cR \big( \big(\bidual^r_\cR H^1_\Sigma (\cO_{k, S(K)}, \cT)\big) / (\cR\cdot c_{k_\infty}) \big ).
\end{equation} 
To show this, we fix a natural number $n$ large enough to ensure that  all of the claims in Lemma~\ref{control lemma} are valid with respect to the polynomial $f_n$. Then, by Lemma~\ref{control lemma}\,(ii), the Fitting ideal $\Fitt_{\cR_n}^0(H^1_{\cF^\vee_{\mathrm{rel}, \Sigma}}(k,  T_n^\lor(1))^\lor)$ contains a non-zero element and so the (finitely generated) $\cR_n$-module $H^1_{\cF^\vee_{\mathrm{rel}, \Sigma}}(k,  T_n^\lor(1))^\lor)$ is finite. From the result of 
Lemma \ref{control lemma}\,(iii), it then follows that    $H^1_{\cF^\vee_{\mathrm{rel}, \Sigma}}(k, \cT^\lor(1))^\lor \otimes_\cR \cR_n$ is finite and hence, as a consequence of the structure theorem for finitely generated $\cR$-modules, that the $\cR$-module  $H^1_{\cF^\vee_{\mathrm{rel}, \Sigma}}(k,  \cT^\lor(1))^\lor$ is torsion.\\
Now the assumption $\Fitt^r_\cR (X_{S(K)} (\cT))$ is nonzero implies the $\cQ (\cR)$-module $\cQ (\cR) \otimes_\cR X_{S(K)} (\cT)$ is free  of rank $r$.  After recalling the exact sequence (\ref{ingredient 2 ses}), and noting that we have just shown its first term to be a torsion $\cR$-module, 
we can therefore deduce that $\cQ (\cR) \otimes_\cR H^1 (C_{S(K), \Sigma}(\cT))$ is also a free $\cQ (\cR)$-module of rank $r$. From the vanishing of the Euler characteristic of $C_{S (K), \Sigma}(\cT)$ we can then finally deduce that 
$\cQ (\cR) \otimes_\cR H^1_\Sigma (\cO_{k, S(K)},\cT)$ is a free $\cQ (\cR)$-module of rank $r$. In addition, since we are assuming $c_{k_\infty}$ is nonzero, it generates a free rank-one $\cQ (\cR)$-module and so the above observation implies that the module  
\[
\cQ (\cR) \otimes_\cR \big( \big(\bidual^r_\cR H^1_\Sigma (\cO_{k, S(K)}, \cT)\big)/ (\cR\cdot c_{k_\infty}) \big) 
\cong \exprod^r_{\cQ (\cR)} (\cQ (\cR) \otimes_\cR H^1_\Sigma (\cO_{k, S(K)}, \cT)) / (\cQ (\cR)\cdot c_{k_\infty})
 \]
must vanish, thereby showing the $\cR$-module $(\bidual^r_\cR H^1_\Sigma (\cO_{k, S(K)}, \cT))/ (\cR\cdot c_{k_\infty})$ to be torsion, as suffices to prove (\ref{z_c interpretation}). 

Turning now to the verification of (\ref{wanted inequality}), we continue to let $n$ be a natural number for which all the claims of Lemma~\ref{control lemma} are valid with respect to the fixed choice of $f$. We then choose a pseudo-isomorphism of $\cR$-modules of the form 
    \begin{equation}\label{fixed quasi}
        H^1_{\cF^\vee_{\mathrm{rel}, \Sigma}}(k,  \cT^\lor(1))^\lor \to \bigoplus_{i=1}^{i=t} \cR/(\cR f^{m_i}) \oplus \bigoplus_{j=1}^{j=t'} \cR/(\cR g_{j})
    \end{equation}
  in which $t$ and $t'$ are non-negative integers, each $m_i$ a natural number and each $g_{j}$ is a (possibly reducible) distinguished polynomial that is independent of $n$ and also, since $n$ validates Lemma~\ref{control lemma}\,(iv), coprime to $f_n$. It is now convenient to recall the following general fact.

\begin{lemma}\label{iwasawa-orders-lemma}
    Let $\cM$ be a finitely generated torsion $\cR$-module that is pseudo-isomorphic to 
    \begin{align*}
        E_\cM \coloneqq \bigoplus_{j=1}^{j=t'} \cR/(\cR g^{m_j}_j),
    \end{align*}
    where each $g_j$ is either $\varpi$ or an irreducible distinguished polynomial and each $m_j$ a natural number. Then, for any irreducible element $f$ of $\cR$ that does not divide $\mathrm{char}_\cR(\cM)$, one has 
    \begin{align*}
        \ord_p\left(\left| \cM/(f\cM)\right|\right) - \ord_p(|\cM_\mathrm{fin}|) \leq \ord_p\left(\left|E_\cM/(fE_\cM)\right|\right) \leq \ord_p\left(\left|\cM/(f\cM)\right|\right),
    \end{align*}
    where $\cM_\mathrm{fin}$ denotes the maximal finite $\cR$-submodule of $\cM$.
\end{lemma}

\begin{proof} This is well-known but, for lack of a better reference, we give an argument. The existence of a map of $\cR$-modules from $\cM$ to $E_\cM$ (or in the reverse direction) that has finite kernel and cokernel combines with a calculation of Herbrand quotients to show that \[|E_\cM[f]|/|E_\cM/(fE_\cM)| = |\cM[f]|/|\cM/(f\cM)|.\] 
The claimed inequalities follow from this equality, the obvious inequality $|\cM[f]|\le |\cM_\mathrm{fin}|$ and the fact that $E_\cM[f]$ vanishes as $f$ is coprime to each $g_j$. \end{proof}
    
We shall apply this result to the pseudo-isomorphism (\ref{fixed quasi}). To do this, we note that, for each index $j$ in (\ref{fixed quasi}) and every such $n$, the quotient module $\cR/(\cR g_{j} + \cR f_n)$ is finite and we claim that its cardinality is bounded independently of $n$. To show this, we note that, because $f_n$ is an irreducible distinguished polynomial, the Weierstra{\ss} Preparation Theorem gives an isomorphism $\cR/(\cR g_{j} +  \cR f_n) \cong \cR_n/(\cR_n g_{j} (\alpha_n))$ with $\alpha_n \coloneqq - (\varpi^n + f (0))$ the root of $f_n$. It is then enough to note that the $\cR_n$-valuation of $g_{j} (\alpha_n)$ is bounded since, for all large enough $n$, the strong triangle inequality implies that 
 $\ord_{\cR_n} ( g_{j} (\alpha_n)) 
    = 
    \ord_\cR ( g_{j} (f(0))).$ 
    
This observation implies the existence of constants $\kappa_1$ and $\kappa_2$ that are independent of $n$ and such that 
    \begin{align}\label{first sequence}
        \ord_p \big ( \big |H^1_{\cF^\vee_{\mathrm{rel}, \Sigma}}(k,  T_n^\lor(1))^\lor \big| \big) &= \ord_p\big(\big|H^1_{\cF^\vee_{\mathrm{rel}, \Sigma}}(k,  \cT^\lor(1))^\lor \otimes_\cR \cR_n\big|\big) + \kappa_1\\
        &\geq \ord_p\bigl({\prod}_{i=1}^{i=t} \left|\cR/(\cR f^{m_i} +  \cR f_n)\right|\bigr) + \kappa_2\notag\\
        &= {\sum}_{i=1}^{i=t}\ord_p\big(\left|\cR_n/(\cR_n f(\alpha_n)^{m_i})\right|\big) + \kappa_2\notag\\
        &= n\cdot {\sum}_{i=1}^{i=t} m_i + \kappa_2\notag\\
        &= n\cdot \ord_f(z_\mathrm{Sel}) + \kappa_2,\notag
    \end{align}
   where the first equality follows from Lemma \ref{control lemma}\,(iii), the inequality from the second inequality of Lemma \ref{iwasawa-orders-lemma} and the final equality from (\ref{fixed quasi}) and the choice of element  $z_\mathrm{Sel}$. \\
    In a very similar way, after fixing a pseudo-isomorphism of $\cR$-modules of the form 
    \begin{equation}\label{quasi 2}
        Q_c \coloneqq \big(\bidual^r_\cR H^1 (\cO_{k, S(K)}, \cT)\big) / (\cR\cdot c_{k_\infty}) \to \bigoplus_{i=1}^{i=t_1} \cR / (\cR f^{l_i}) \oplus \bigoplus_{j=1}^{j = t_1'} \cR / (\cR h_j),
    \end{equation}
 in which $t_1$ and $t_1'$ are non-negative integers, each $l_i$ a natural number, and each $h_j$ a (possibly reducible) polynomial that is both independent of $n$ and coprime to $f_n$, one finds that there exist constants $\kappa_3$ and $\kappa_4$ that are independent of $n$ and such that 
    \begin{align}\label{second sequence}
        & \phantom{=}\; \ord_p\Big(\big|\bigl(\bidual_{\cR_n}^{r} H^1(\cO_{k,S(K)}, T_n)\bigr)/(\cR_n\cdot \mathrm{Tw}^r_{k,\psi_n}(c_{k_\infty}))\big|\Big) \\
        &= \ord_p\big(\big|Q_c \otimes_\cR \cR_n\big|\big) + \kappa_3\notag\\
        &\leq \ord_p\big({\sum}_{i=1}^{i=t_1} \big|\cR/(\cR f^{l_i} +  \cR f_n)\big|\big) + \kappa_4\notag\\
        &= n\cdot{\sum}_{i=1}^{i=t_1}l_i + \kappa_4\notag\\
        &= n\cdot \ord_f(z_c) + \kappa_4.\notag
    \end{align}
    Here the first equality follows from Lemma \ref{control lemma}\,(i), the inequality from the first inequality of Lemma \ref{iwasawa-orders-lemma}, and the final equality from the  pseudo-isomorphism (\ref{quasi 2}) and equality (\ref{z_c interpretation}).\\
 Next we note that Lemma \ref{control lemma}\,(ii) combines with Lemma \ref{lemma-image-Fitt}\,(ii) to imply an inclusion
 \[
 \Fitt^0_{\cR_n} \Big ( \big(\bidual_{\cR_n}^{r} H^1(\cO_{k,S(K)}, T_n)\big)/(\cR_n\cdot\mathrm{Tw}^r_{k,\psi_n}(c_{k_\infty})) \Big) 
\subseteq \Fitt^0_{\cR_n} ( H^1_{\cF^\vee_{\mathrm{rel}, \Sigma}}(k,  T_n^\lor(1))^\lor).
 \]
Since $\cR_n$ is a discrete valuation ring with finite residue field, this inclusion is therefore equivalent to an inequality
    \[
         \left|\Bigl( \bidual_{\cR_n}^{a} H^1(\cO_{k,S(K)}, \cT(\psi_n))\Bigr)/\big (\cR_n \cdot j_{k,n}(\eta_{k_\infty})\big )\right| \geq \left|H^1_{\cF^\vee_{\mathrm{rel}, \Sigma}}(k, \cT(\psi_n)^\lor(1))^\lor\right|.
    \]

Upon combining this inequality with those of (\ref{first sequence}) and (\ref{second sequence}), we finally derive an inequality 
    \begin{align*}
        n\cdot \ord_f(z_c) + \kappa_4 \geq n\cdot \ord_f(z_\mathrm{Sel}) + \kappa_2.
    \end{align*}
By taking $n$ sufficiently large, this inequality implies the required inequality (\ref{wanted inequality}), and thereby completes the proof of Theorem \ref{new iwasawa theory euler systems main result}.
\qed 

\section{Elliptic curves}\label{elliptic section}

In this section we let $E$ be an elliptic curve defined over $\Q$ and of conductor $N \coloneqq N_E$. We then consider the integral and rational $p$-adic Tate modules of $E$ defined as
\[
T_p E \coloneqq {\varprojlim}_{n \in \N} E [p^n] 
\quad \text{ and } \quad 
V_p E \coloneqq \Q_p \otimes_{\Z_p} T_pE.
\]
We fix a $\Z_p$-basis of $T_pE$ (the precise choice of which does not matter to subsequent arguments) and thereby identify $\mathrm{Aut}_{\Z_p} (T_p E)$ with $\mathrm{GL}_2 (\Z_p)$. In this way, the natural action of $G_\Q$ on $T_p E$ gives rise to a homomorphism 
\[ \rho_{E, p} \: G_\Q \to \mathrm{GL}_2 (\Z_p).\] 

For any subfield $\cK$ of $\Q^\mathrm{c}$ we write $\Omega (\cK)$ for the set of finite abelian extensions of $\Q$ in $\cK$. For $K$ in $\Omega (\Q^\mathrm{c})$, we set 
$\cG_K \coloneqq \gal{K}{\Q}$ and $\widehat{\cG_K} \coloneqq \Hom_\Z (\cG_K, \Q^{\mathrm{c},\times})$ and consider the motive
\[M_{E/K}\coloneqq h^1 (E / K) (1),\]
regarded (unless explicitly indicated otherwise) as defined over $\Q$ and with coefficients $\QQ[\cG_K]$. 

\subsection{Kato's Euler system of zeta elements} \label{kato euler system section}

We review the Euler system constructed by Kato in \cite{Kato04}. To do this, we fix a minimal Weierstra{\ss} model of $E$ over $\Z$ and write $\omega \in H^0_\dR (E, \Omega^1_{E / \Q})$ for the corresponding N\'eron differential.
We then write $c_\infty \in \{ 1, 2\}$ for the number of connected components of $E (\R)$, and define periods of $E$ by setting
\[
\Omega^+ = \Omega_{\omega, \gamma^+} \coloneqq  \int_{E (\R)} \mid \omega \mid = c_\infty \cdot \int_{\gamma^+} \omega 
\quad \text{ and } \quad 
\Omega^- = \Omega_{\omega, \gamma^-} \coloneqq \int_{\gamma^-} \omega. 
\]
Here $\gamma^+$ and $\gamma^-$ denote generators of the respective subgroups $H_1 (E (\C), \Z)^+$ and $H_1 (E (\C), \Z)^-$ of $H_1 (E (\C), \Z)$ upon which complex conjugation acts by $+1$ and $-1$ that are chosen in such a way that the real numbers $\Omega^+$ and $(-i)\Omega^-$ are positive. \\
We recall that, for each $K \in \Omega (\Q^\mathrm{c})$ and place $v \in \Pi_K^p$, Kato \cite[Ch.\@ II, \S\,1.2.4]{Kato93a} has defined a canonical `dual exponential map' 
\begin{equation}\label{e: dual exp map Kato}
    \exp^\ast_{K_v} \: H^1 (K_v, V_p E) \to \mathrm{Fil}^0 D_{\dR, K_v} (V_p E).  
\end{equation} 

We fix an embedding $\iota \: \Q^\mathrm{c} \hookrightarrow \C$ and, given $K \in \Omega (\Q^\mathrm{c})$, write $w_\iota \coloneqq w_{\iota, K}$ for the place of $K$ that corresponds to the restriction $\iota_K$ of $\iota$ to $K$. We finally set 
\[ S_0 \coloneqq \{ p, w_{\iota, \Q} \}\quad\text{and}\quad S_1 \coloneqq S \cup \{ \ell \mid N\}\]
and then, for each $K \in \Omega (\Q^\mathrm{c})$, also 
\[ S_0 (K) \coloneqq S_0 \cup S_\mathrm{ram} (K / \Q)\quad \text{and}\quad S (K) \coloneqq S_1 \cup S_\mathrm{ram} (K / \Q).\]

\begin{thm}[Kato] \label{kato euler system}
    There exists a collection of elements
\[
z^\mathrm{Kato} \coloneqq (z^\mathrm{Kato}_K)_{K} \in \prod_{K \in \Omega (\Q^\mathrm{c})}
H^1 ( \cO_{K, S (K)}, V_p E)
\]
with the following properties.
\begin{romanliste}
\item For all $K$ and $L$ in $\Omega (\Q^\mathrm{c})$ with $K \subseteq L$, one has 
\[
\mathrm{Cores}_{L / K} ( z^\mathrm{Kato}_L) = \big( {\prod}_{\ell \in S (L) \setminus S (K)} \Eul_\ell (\Frob_\ell^{-1}) \big) \cdot z^\mathrm{Kato}_K,
\]
where $\mathrm{Cores}_{L / K}$ denotes the corestriction map $H^1 (L, V_pE) \to H^1 (K, V_p E)$. 
\item For $K \in \Omega (\Q^\mathrm{c})$ set
\[
y_K^\mathrm{Kato} \coloneqq \big( {\prod}_{\ell \in S (K) \setminus S_0(K)} \Eul_\ell ( \Frob_\ell^{-1}) \big)^{-1} z_K^\mathrm{Kato}.
\]
Then, if both $E [p]$ is irreducible as an $\mathbb{F}_p [G_\Q]$-module and $E (K) [p]$ vanishes, the element $c_\infty y^\mathrm{Kato}_K$, and hence also $c_\infty z_K^\mathrm{Kato}$, belongs to $H^1 (\cO_{K, S (K)}, T_p E)$. In particular, one has 
\[ z^\mathrm{Kato}\in \ES^1_{\emptyset, S} (T_p E)\quad\text{and}\quad y^\mathrm{Kato} \coloneqq (y^\mathrm{Kato}_K)_{K \in \Omega(\Q^\mathrm{c})}\in \ES^1_{\emptyset, S_1} (T_p E).\]
\item For every $K \in \Omega (\Q^\mathrm{c})$, there is an identification of spaces 
\[ 
 {\bigoplus}_{v \in \Pi_K^p} \mathrm{Fil}^0_{\dR, K_v} (V_p E) \cong \Q_p \otimes_\Q H^0_\dR ( E, \Omega^1_{E / K}) = (\Q_p \otimes_\Q K) \otimes_\Q H^0_\dR (E, \Omega^1_{E / \Q}),
 \] 
 with respect to which one has 
\[
\bigl(\exp^\ast_{K_v} ( z^\mathrm{Kato}_K)\bigr)_{v \in \Pi_K^p} = \big( {\sum}_{\chi \in \widehat{G_K}} \frac{L_{S (K)} (E, \chi^{-1}, 1)}{\Omega^{\mathrm{sgn}(\chi)}} e_\chi \big) \otimes \omega.
\]
Here the symbol $\mathrm{sgn}(\chi) \in \{ - , +\}$ is specified by the equality $\chi(\tau_\iota) = \mathrm{sgn}(\chi)1$, where $\tau_\iota$ denotes the automorphism of $\Q^\mathrm{c}$ obtained by restricting complex conjugation through $\iota$.  
\end{romanliste}
\end{thm}

\begin{proof} 
 In brief, the elements $z^\mathrm{Kato}_K$ and $y^\mathrm{Kato}_K$ are  defined by slightly modifying the corresponding elements constructed by Kato in \cite[(8.1.3)]{Kato04}, with the integrality property in (ii) proved by the argument of \cite[\S\,12.6]{Kato04}, and the `explicit reciprocity law' in (iii) a consequence of \cite[Th.\@ 9.7 and 12.5]{Kato04}.
 
If $p$ is an odd prime of good reduction, then these arguments are worked out in detail by Kataoka in \cite[Th.\@ 6.1]{Kataoka1}. In addition, upon replacing Kataoka's application of \cite[Prop.\@ 8]{Wuthrich14}  by the representation-theoretic observation made in \cite[Rem.\@ 6.9]{bss2}, the argument of loc.\@ cit.\@ extends directly to include any prime $p$ for which $E [p]$ is irreducible as in \cite[\S\,6.3]{bss2}. 
\end{proof}

\begin{rk} 
\begin{liste}
    \item Regarding Theorem \ref{kato euler system}\,(ii), it can be shown that, if $\mathrm{SL}_2 (\Z_p)\subseteq \im(\rho_{E,p})$, then  $E [p]$ is an irreducible $\mathbb{F}_p [G_\Q]$-module and also $E (K)[p] = (0)$ for every $K \in \Omega (\Q^\mathrm{c})$.  
    \item An analogue of the integrality property in Theorem \ref{kato euler system}\,(ii) also holds when $E [p]$ is a reducible $\mathbb{F}_p [G_\Q]$-representation. However, such an extension of Theorem \ref{kato euler system} requires a detailed discussion of Kato's argument in \cite[\S\,12.6]{Kato04} and so, since it is outside the scope of the Euler system theory that we develop here, this will be discussed elsewhere.
\end{liste}
\end{rk}

\subsection{Kato's Iwasawa main conjecture}

Given a finite abelian extension $K$ of $\Q$, we write $K_\infty \coloneqq \bigcup_{n \in \N} K_n$ for its cyclotomic $\Z_p$-extension. We let $\Lambda_K \coloneqq \Z_p \llbracket \gal{K_\infty}{\Q} \rrbracket$ denote the associated Iwasawa algebra, and $\cQ_K$ its total ring of fractions. \\
Define the representation $\cT \coloneqq (T_p E) \otimes_{\Z_p} \Lambda_K$ and note that $H^1 (\cO_{\Q, S(K)}, \cT)$ equals the Iwasawa cohomology group $H^1_\mathrm{Iw} (\cO_{K, S(K)}, T_p E) \coloneqq {\varprojlim}_n H^1 (\cO_{K_n, S(K)}, T_p E)$ (with the limit taken with respect to corestriction maps). The family 
\[ z_{K_\infty}^\mathrm{Kato} \coloneqq (z^\mathrm{Kato}_{K_n})_{n \in \N}\] 
therefore defines an element of $H^1_\mathrm{Iw} (\cO_{K, S(K)}, T_p E)$.\\
We write $\tau$ for the unique element of $G_\Q$ with the property that $\iota \circ \tau$ is complex conjugation, and $\tau_K$ for its projection to $\cG_K$. For $n \in \N$, we may then define a $\Z [1 / 2] [\cG_{K_n}]$-basis element 
\[
\gamma_{K_n} \coloneqq  \big( (1 + \tau_{K_n}) \otimes \gamma^+ + (1 - \tau_{K_n}) \otimes \gamma^-)/2 
\]
of the (free) module 
\[
H_1 ( (E \times \Spec K_n) (\C), \Z [1 / 2])^+
\cong \big( \Z [\cG_{K_n}] \otimes_\Z H_1 (E^\iota (\C), \Z [1 / 2]) \big)^+.
\]

For each odd prime $p$, we use the canonical comparison isomorphism 
\[ \Z_p \otimes_\Z H_1 ( (E \times \Spec K_n) (\C), \Z)^+ \cong Y_{K_n} (T_p E)\] 
to regard $\gamma_{K_n}$ as a basis of $Y_{K_n} (T_p E)$. By taking the limit over $n$, we thereby obtain a $\Lambda_K$-basis $\gamma_{K_\infty} = (\gamma_{K_n})_n$ of $Y_{\Pi_k^\infty} (\cT)$. By using this basis, the construction (\ref{definition Theta map}) gives a map
\[
\Theta_{K_\infty, S(K)} \coloneqq 
\Theta_{S(K), \gamma_{K_\infty}} \: \Det_{\Lambda_K} (C_{S(K)} (\cT)) \to 
\cQ_K \otimes_{\Lambda_K} H^1_\Iw (\cO_{K, S(K)}, T_p E). 
\]

Then, under mild technical hypotheses, the following result verifies one inclusion of the natural Iwasawa-theoretic version of Kato's Conjecture.

\begin{thm} \label{eimc result} If $p > 3$, then $z^\mathrm{Kato}_{K_\infty}\in\im (\Theta_{K_\infty, S(K)})$ if the following conditions  are satisfied: 

\begin{itemize}
\item[(i)] $\mathrm{SL}_2 (\Z_p)\subseteq \im(\rho_{E, p})$;
\item[(ii)]  $E$ has potentially good reduction at $p$, or  
    $K ( \sqrt{- 2 (A / B)})^\times$ contains no element of order $p$ for $A\in \Q^\times$ and $B \in \Q^\times$  such that $y^2 = x^3  + Ax + B$ is a Weierstra{\ss} equation for $E$.
\end{itemize}  
\end{thm}

\begin{rk} \label{big image hyp rk}
    The condition of Theorem \ref{eimc result}\,(i) is satisfied if $\rho_{E, p}$ is surjective, as proved by Serre \cite{Serre72} for all but finitely many $p$ if  $E$ does not have CM. In addition, writing $j_E$ for the $j$-invariant of $E$, Zywina \cite[Conj.~1.1]{Zywina} has conjectured $\rho_{E, p}$ is surjective for all $(p,j_E)$ outside 
\begin{align*}
 S_\mathrm{exc}  \coloneqq  (\{2, 3, 5, 11, 13 \} \times \Q) 
\cup  & \big \{ (17, -17^2 \cdot 101^3 / 2), (17, - 17 \cdot 373^3 / 2^{17}),\\
& \qquad  (37, - 7 \cdot 11^3), (37, - 7 \cdot 137^3 \cdot 2083^3) \big \}.
\end{align*}
The non-surjectivity of $\rho_{E, p}$ also implies the following further restrictions on $p$.
\begin{itemize}[label=$\circ$]
 \item If $E$ is semi-stable, then $p < 11$ by Mazur \cite[Th.\@ 4]{Mazur78}.
 \item For general $E$, one has $p \le \max \{ 37, N_E \}$ by \cite[Th.\@ 1.10 and Prop.\@ 1.8]{Zywina}. Further, if $\ell_1^{e_1} \cdot \ldots \cdot \ell_s^{e_s}$ is the factorisation of the denominator of $j_E$ into distinct prime numbers with $e_i > 0$ and $p \not \in S_\mathrm{exc}$,
  then each $\ell_i$ is congruent to $\pm 1$ (mod $\, p$) and each $e_i$ is divisible by $p$ (cf.\@ \cite[Th.\@ 1.10 and Th.\@ 1.5]{Zywina}).
\end{itemize}
\end{rk}

\begin{rk} Regarding the condition of Theorem \ref{eimc result}\,(ii), we note 
that a Weierstra{\ss} equation for $E$ of the stated form exists if $j_E \not \in \{ 0, 1728 \}$, and that the associated field $\Q ( \sqrt{- 2 ( A / B)})$ then depends only on $E$ (cf.\@ \cite[Ch.~V, Lem.~5.2\,(a)]{SilvermanII}).
Moreover, if $E / K$ has split-multiplicative reduction at every place in $\Pi^p_K$, then all such places split completely in $K(\sqrt{- 2(A/B)})$ and so one has $K(\sqrt{- 2(A/B)})^\times[p] = (0)$ if $K^\times [p] = (0)$.  
\end{rk}

\smallskip

Theorem \ref{eimc result} will be proved in \S\,\ref{proof eimc section}. First, however, we shall record a consequence towards Kato's Conjecture itself.
For this, for every integer $r$ we define an idempotent 
\[
e_{r, K} \coloneqq {\sum}_\chi e_\chi\in \Q [\cG_K],
\]
where in the sum $\chi$ runs over all characters in $\widehat{\cG_K}$ with $\ord_{s = 1} L (E, \chi, s) \leq r$ (and $e_\chi$ is the associated primitive orthogonal idempotent of $\Q^\mathrm{c}[\cG_K]$). The conjecture $\mathrm{TNC}(M_{E/K}, e_{r, K}\Z_p [\cG_K])$ can then naturally be interpreted as the `analytic-rank-at-most-$r$ component' of the conjecture $\mathrm{TNC}(M_{E/K},\Z_p [\cG_K])$. 
We note, in particular, that if $L ( E / K, 1) \neq 0$, then $e_{0, K} = 1$ and so $\mathrm{TNC}(M_{E/K}, e_{0, K}\Z_p [\cG_K])$ coincides with $\mathrm{TNC}(M_{E/K},\Z_p [\cG_K])$. \\
To give an explicit interpretation of these conjectures in terms of Kato's Euler system, we will use the
 map
\[
\Theta_{K, S(K)} \coloneqq \Theta_{K, S(K), \gamma_K} \: \Det_{\Z_p [\cG_K]} (C_{S(K)} (T_{K / \Q})) \to H^1 (\cO_{K, S(K)}, V_p E) 
\]
defined in (\ref{definition of finite level theta map}) for the induced representation $T_{K / \Q} \coloneqq T_p E \otimes_{\Z_p} \Z_p [\cG_K]$. 
We then also recall that the following equivalences are  established in \cite[\S\,6.2]{bss2}: 
\begin{itemize}[label=$\circ$]
    \item $\mathrm{TNC}(M_{E/K}, e_{0, K}\Z_p [\cG_K])$ is valid if and only if $e_{0, K}\Z_p [G] \cdot z^\mathrm{Kato}_{K} = e_{0, K}\cdot\im (\Theta_{K, S (K)})$;
 \item assuming the validity of a natural generalisation of Perrin-Riou's conjecture concerning the logarithm of $z_K^\mathrm{Kato}$, $\mathrm{TNC}(M_{E/K}, e_{1, K}\Z_p [\cG_K])$ is valid if and only if $\ZZ_p[G]\cdot z^\mathrm{Kato}_K = \im(\Theta_{K, S(K)})$.
 \end{itemize}
In particular, since the next result verifies one inclusion in these explicit conjectural equalities, it provides concrete evidence in support of $\mathrm{TNC} (M_{E/K}, \Z_p [\cG_K])$.\\
In the sequel we shall, for brevity, refer to the Birch--Swinnerton-Dyer Conjecture as the `BSD Conjecture'. 

\begin{cor} \label{etnc result 2}
    If $p$  satisfies the conditions of Theorem~\ref{eimc result}, then there is an inclusion  
\[ \Z_p [\cG_K]\cdot z^\mathrm{Kato}_{K} \subseteq \im(\Theta_{K, S(K)}).\]
Furthermore, this inclusion is an equality if the $p$-part of the BSD Conjecture for $E$ is valid over every intermediate field $F$ of $K/\QQ$ for which $[F : \Q]$ is prime-to-$p$. 
\end{cor}

\begin{proof}
At the outset we note that $K_n$ contains a $p$-th root of unity if and only if $K$ does. Given this, the displayed inclusion follows directly from Theorem \ref{eimc result} and the fact that $\Theta_{K_\infty, S(K), \Gamma_{\!\!K}}$ and $z^\mathrm{Kato}_{K_\infty}$ are the limits (over $n$) of $\Theta_{K_n, S(K), \gamma_{K_n}}$ and $z^\mathrm{Kato}_{K_n}$, respectively. \\ 
    To prove the second assertion, we recall that
 $e_{0, K_n} \coloneqq {\sum}_\chi e_\chi$,
where the sum runs over all characters $\chi \: \cG_n \to \C^\times$ with $L (E, \chi, 1) \neq 0$.
    Then, via the argument of \cite[Prop.\@ 3.6]{scarcity} (with the assumed validity of the BSD Conjecture in the stated cases playing the role of the analytic class number formula in loc.\@ cit.), Nakayama's Lemma can be used to deduce an equality $e_{0, K_n} \Z_p [\cG_n]\cdot z^\mathrm{Kato}_{K_n} = e_{0, K_n} \im (\Theta_{K_n, S(K), \gamma_{K_n}})$.\\
    We now set $\cG_n \coloneqq \cG_{K_n}$ and assume that $a_0 \coloneqq \Theta_{K, S, \gamma_K}(\alpha_0)$ does not belong to $\Z_p [G] \cdot z^\mathrm{Kato}_K$. Then, by lifting $\alpha_0$ to an element $\alpha = (\alpha_n)_{n \in \N}$ of $\Det_{\Lambda_K} (C^\bullet_{S(K)} (\cT))$, we may regard $a_0$ as the bottom class of the family $a \coloneqq (a_n)_{n \in \N} = ( \Theta_{K_n, S(K), \gamma_{K_n}} (\alpha_n))_{n \in \N}$. Moreover, the discussion above shows that, for each $n \in \N$, the idempotent $e_{0, K_n}$ annihilates the image $[a_n]$ of $a_n$ in the quotient module
    \[ Z_n \coloneqq ( \im (\Theta_{K_n, S(K), \gamma_{K_n}})) / ( \Z_p [\cG_n] \cdot z^\mathrm{Kato}_{K_n}).\] 
    Since $a$ is by definition an element of $\varprojlim_{n \in \N} \im (\Theta_{K_n, S(K), \gamma_{K_n}})$, it follows that $ ( [a_n])_{n \in \N}$ is an element of $\varprojlim_{n \in \N} Z_n$.
    Now, by Rohrlich \cite[Th.\@ 1]{Rohrlich88}, the value $L(E, \chi, 1)$ can vanish for only finitely many characters $\chi$ that are unramified outside $S(K)$. 
    In particular, there exists a subfield $F$ of $K_\infty$ that has finite degree over $\Q$ and is such that every character $\chi \: \gal{K_\infty}{\Q} \to \C^\times$ with $e_\chi e_{0,K_n} = 0$ factors through $\cG_F$. For every element $\sigma$ of $\gal{K_\infty}{F}$, one therefore has 
    \[
    (\sigma  - 1) \cdot [a_n] = (\sigma - 1) (1 -  e_{0, K_n}) \cdot [a_n] =  0 \cdot [a_n],
    \]
    and so $[a_n]$ is fixed by $\sigma$. It follows that $( [a_n])_{n \in \N}$ belongs to $( {\varprojlim}_{n \in \N} Z_n)^{\gal{K_\infty}{F}}$ and so in order to deduce the triviality of $[a_0]$ it is enough for us to prove that the limit 
    \[ {\varprojlim}_{n \in \N} Z_n = \im (\Theta_{K_\infty, S(K)}) / (\Lambda_K\cdot z^\mathrm{Kato}_{K_\infty})\] 
    has no non-trivial $\gal{K_\infty}{F}$-invariant elements. In addition, since $\gal{K_\infty}{F}$ is an open subgroup of $\gal{K_\infty}{\Q}$, this will follow (see, for example, \cite[Prop.\@ 5.3.19\,(i)]{NSW}) if the $\Lambda_K$-modules $\im (\Theta_{K_\infty, S(K)})$ and $\Lambda_K\cdot z^\mathrm{Kato}_{K_\infty}$ are both free. To justify this, we note that $\Lambda_K\cdot z^\mathrm{Kato}_{K_n}$ and $\im (\Theta_{K_n, S(K), \gamma_{K_n}})$ are cyclic $\Lambda_K$-modules and claim that their respective annihilators are contained in $\Z_p [\cG_n] [e_{0, K_n}]$. Indeed, in the first case the latter assertion follows directly from Kato's explicit reciprocity law in Theorem \ref{kato euler system}\,(iii) and in the second case from the definition of $\Theta_{K_n, S(K), \gamma_{K_n}}$ and the fact that $e_{K_n} \cdot e_{0, K_n} = e_{0, K_n}$ (cf.\@ \cite[Lem.\@ 6.1\,(iii)]{bss2}). It therefore suffices for us to prove that the limit ${\varprojlim}_{n \in \N} \Z_p [\cG_n] [e_{0, K_n}]$ vanishes. Now, just as above, \cite[Th.~1]{Rohrlich88} implies that each element of ${\varprojlim}_{n \in \N} \Z_p [\cG_n] [e_{0, K_n}]$ is fixed by every element of the open subgroup $\gal{K_\infty}{F}$ of $\cG_{K_\infty}$. Hence, since ${\varprojlim}_{n \in \N} \Z_p [\cG_n] [e_{0, K_n}]$ is a submodule of $\Lambda_K$, it must vanish, as required to complete the proof.
\end{proof}

\begin{rk}\label{curve evidence} 
    Due to important work of many authors, there is by now an impressive range of results regarding the validity of the BSD Conjecture. We restrict ourselves here to recalling the following recent result of Burungale--Castella--Skinner \cite[Cor.\@ 1.3.1]{BurungaleCastellaSkinner24}: the $p$-part of the BSD Conjecture for $E / \Q$ is valid if $p > 3$ is a prime number at which $E$ has good ordinary reduction (so $p \nmid a_p$) for which $\mathrm{SL}_2 (\Z_p)\subseteq \im(\rho_{E, p})$ and, in addition, one has $\ord_{s = 1} L (E, s) \leq 1$. 
\end{rk}

Finally, we note that Corollary \ref{etnc result 2} also has the following concrete consequence concerning the BSD Conjecture in analytic rank zero.

\begin{cor}\label{bsd evidence}
    If $p$ satisfies the conditions of Theorem~\ref{eimc result}, then one divisibility in the `$p$-part' of the BSD Conjecture for $E / K$ is valid. That is, the group $\sha_{E / K} [p^\infty]$ is finite and  
    \[
    \ord_p \Big ( \frac{L ( E / K, 1)}{|d_K|^{-1 / 2} \cdot \Omega_K} \Big) \geq \ord_p \big (  |\sha_{E / K} [p^\infty] | \cdot \mathrm{Tam}_{E / K} \big).
    \]
    Here $d_K$ is the discriminant of $K$, $\Omega_K$ is the product $(\Omega^+)^{r_1 + r_2} (\Omega^-)^{r_2}$ with $r_1$ and $r_2$ the number of real and complex places of $K$, and 
    \[ \mathrm{Tam}_{E / K} \coloneqq \prod_{v \in \Pi_K\setminus \Pi_K^\infty} ( | \omega/\omega_v |_v \cdot (E (K_v) : E_0 (K_v)))\] 
    is the product of Tamagawa numbers of $E$ over $K$ (with $\omega_v$ a N\'eron differential at $v$). Furthermore, the above inequality is an equality if the $p$-part of the BSD Conjecture for $E$ is valid over every intermediate field $F$ of $K/\QQ$ for which $[F : \Q]$ is prime-to-$p$.
\end{cor}

\begin{proof} The strategy that we use in this argument is well-known and can be summarised as follows: The non-vanishing of $L ( E / K, 1)$ implies  $e_{0, K} = 1$ and so Corollary \ref{etnc result 2} verifies one inclusion of  $\TNC(M_{E/K}, \Z_p [G])$. By Remark \ref{forgetful etnc remark}, one also knows that $\TNC(M_{E/K}, \Z_p [G])$ implies $\TNC(M_{E/K}, \Z_p)$, where in the latter case $M_{E/K}$ is regarded as defined over $K$ and with coefficients $\QQ$. Finally, one uses the fact that $\TNC(M_{E/K}, \Z_p)$ is equivalent to the $p$-part of the BSD Conjecture (as shown in each of \@ \cite{kings-bsd, venjakob, BurungaleFlach}). 

In more detail, Corollary \ref{etnc result 2} implies that there is a unique element $a$ of $\Det_{\Z_p [\cG_K]} ( C_{S(K)} (T_{K / \Q}))$ with $\Theta_{K, S(K)} (a) = z_K^\mathrm{Kato}$ and, moreover, that $a$ is a $\Z_p [\cG_K]$-basis of $\Det_{\Z_p [\cG_K]} ( C_{S(K)} (T_{K / \Q}))$ if the BSD Conjecture for $E$ holds for all subfields $F$ of $K$ with $p \nmid [F : \Q]$. In addition, the  assumption $L ( E / K, 1) \neq 0$ also implies that $H^2 (\cO_{K, S(K)}, V_p E)$ vanishes and hence there exists  a composite isomorphism of $\Q_p [\cG_K]$-modules
\begin{cdiagram}[row sep=tiny]
   j_p \: \Q_p \otimes_{\Z_p} \Det_{\Z_p [\cG_K]} ( C_{S(K)} (T_{K / \Q})) \arrow{r}[below]{\simeq}[above]{\Theta_{K, S(K)}}
   & H^1 (\cO_{K, S(K)}, V_p E) \arrow{r}[below]{\simeq} & \textstyle \bigoplus_{v \in \Pi_K^p} H^1 (K_v, V_p E) \\
   \phantom{\Det_{\Z_p [G]} ( C^\bullet_{K, S(K)}} \arrow{r}[below]{\simeq}[above]{( \exp^\ast_{K_v})_{v}}
   & \Q_p \otimes_\Q H^0 (E, \Omega^1_{E / K}) \arrow{r}{\omega_E \mapsto 1} & \Q_p \otimes_\Q K.
\end{cdiagram}%
From Kato's explicit reciprocity law in Theorem \ref{kato euler system}\,(iii) it then follows that we have
\begin{align*}
j_p (a) = {\sum}_{\chi \in \widehat{G}} \frac{L_{S (K)} (E, \chi^{-1}, 1)}{\Omega^{\epsilon (\chi)}} e_\chi. 
\end{align*}
Now, if one regards $T_{K / \Q}$ as a representation with coefficients $\Z_p$ rather than $\Z_p[\cG_K]$, then the same construction as above also gives a canonical isomorphism of $\Q_p$-vector spaces
\[
\widetilde j_p \: \Q_p \otimes_{\Z_p} \Det_{\Z_p} ( C_{S(K)} (T_{K / \Q})) \stackrel{\simeq}{\longrightarrow} \Q_p.
\]
In addition, there exists a canonical 'forgetful functor' homomorphism of $\Z_p$-modules 
\[\kappa: \Det_{\Z_p [\cG_K]} ( C_{S(K)} (T_{K / \Q})) \to \Det_{\Z_p} ( C_{S(K)} (T_{K / \Q}))\] 
(cf.\@ \cite[Lem.~3.7~(b)]{scarcity}), and the general result of \cite[Ch.~III, \S~9.6, Prop.~3]{bourbaki} (and \cite[Lem.\@ 3.7\,(b)]{scarcity})
implies that
\[
\widetilde j_p (\kappa(a)) = \NN_{(\Q_p \otimes_\Q K) / \Q_p} ( j_p (a)) = \prod_{\chi \in \widehat{\cG_K}} \frac{ L_{S (K)} (E, \chi^{-1}, 1)}{\Omega^{\epsilon (\chi)}} =  \frac{L_{S(K)} ( E / K, 1)}{\Omega_K}. 
\]
The claimed result  then follows upon combining this last equality with the fact that 
\[
\widetilde j_p ( \Det_{\Z_p} ( C_{S(K)} (T_{K / \Q})) = |d_K|^{-1 / 2} \cdot 
|\sha [p^\infty] | \cdot \mathrm{Tam}_{E / K} \cdot \big({\prod}_{v \in S(K)\setminus \Pi_k^\infty} \Eul_v (1) \big)
\cdot \Z_p,
\]
as shown in \cite[Prop.\@ 2.1]{BurungaleFlach}.
\end{proof}

\subsection{The proof of Theorem \ref{eimc result}}\label{proof eimc section} We begin by establishing a useful technical result. 

\begin{lem} \label{error terms Iwasawa theory elliptic curves}
The following claims are valid.
\begin{romanliste}
    \item The $\ZZ_p$-module $\bigoplus_{v \in \Pi_{K_\infty}^p} (E (K_{\infty, v})_\tor \otimes_\Z \Z_p)^\vee$ is finitely generated. If $E$ has potentially good reduction at $p$, then it is moreover finite. 
    \item If $\bigoplus_{v \in \Pi_{K_\infty}^p} (E (K_{\infty, v})_\tor \otimes_\Z \Z_p)^\vee$ is infinite, then
 each summand $(E (K_{\infty, v})_\tor \otimes_\Z \Z_p)^\vee$ has $\Z_p$-rank one and the induced character
 \[
 \psi \: G_{\Q_p} \to \mathrm{Aut} ( (E (K_{\infty, v})_\tor \otimes_\Z \Z_p)^\vee_\mathrm{tf}) \cong \Z_p^\times 
 \]
 is equal to $\omega \cdot \chi_\mathrm{cyc}^{-1}$ for a character $\omega \: G_{\Q_p} \to \Z_p^\times$ that factors through $\gal{\Q_p ( \sqrt{\gamma})}{\Q_p}$ with $\gamma \in \Q_p^\times$ specified as follows: 
 if $E$ is defined by a Weierstra{\ss} equation $y^2 = x^3 + Ax + B$ with $A, B \in \Q\setminus \{0\}$, then $\gamma \coloneqq - 2 (A/B)$. In particular, the character $\omega$ is unramified if $E$ has multiplicative reduction at $p$.
\end{romanliste}
\end{lem}

\begin{proof}
    To justify the first assertion of (i), we note the obvious injective homomorphism 
    \[ E (K_{\infty, v})_\tor \otimes_\Z \Z_p \hookrightarrow E (\Q_p^\mathrm{c})_\tor \otimes_\Z \Z_p \cong (T_p E)^\vee (1)\] 
induces a surjective map
    \begin{equation} \label{finitely generated}
    \bigoplus_{v \in \Pi_{K_\infty}^p} T_p E ( - 1) \twoheadrightarrow \bigoplus_{v \in \Pi_{K_\infty}^p} (E (K_{\infty, v})_\tor \otimes_\Z \Z_p)^\vee .
    \end{equation}
    This shows that the right hand module in (\ref{finitely generated}) is a finitely generated $\Z_p$-module (because $p$ is finitely decomposed in $K_\infty$). The second assertion of (i) is then a classical result of Imai \cite{Imai} (cf.\@ also Remark \ref{error term at p remark}). 
    
    To prove (ii), we may assume that $E$ does not have potentially good reduction at $p$. In this case $E$ has split-multiplicative reduction over $\Q_p ( \sqrt{\gamma})$ (see \cite[Ch.~V, Th.~5.3]{SilvermanII}). Moreover,
     $\Q_p ( \sqrt{\gamma}) / \Q_p$ is an unramified extension if $E$ has multiplicative reduction at $p$ (cf.~\cite[Ch.~V, Ex.~5.11]{SilvermanII}). 
    Fix a $p$-adic place $v$ of $K_\infty$ and set $F_\infty \coloneqq \Q_p (\mu_p, \sqrt{\gamma}) K_{\infty, v}$, the (local) cyclotomic $\Z_p$-extension of the composite $F$ of $\Q_p (\mu_p, \sqrt{\gamma})$ with the completion $K_v$ of $K$ at the restriction of $v$ to $K$. Then $E (F_\infty)$ contains $E (K_{\infty, v})$ and so $(E (K_{\infty, v})_\tor \otimes_\Z \Z_p)^\vee$ is a quotient of $(E (F_{\infty})_\tor \otimes_\Z \Z_p)^\vee$.\\
    For each intermediate field $M$ of $\QQ_p^c/\QQ_p$, we set $G_M \coloneqq \gal{\Q_p^\mathrm{c}}{M}$. We then recall that Tate's uniformisation theorem \cite[Ch.\@ V, Cor.\@ 5.4]{SilvermanII} gives a $G_L$-equivariant isomorphism
    $ E ( \Q_p^\mathrm{c}) \cong \Q_p^{\mathrm{c},\times} / q_E^\Z$ with $q_E \in \cO_L \setminus \cO_L^\times$ denoting the Tate period of $E$. Hence, upon taking $G_{F_\infty}$-invariants, we obtain an isomorphism
    \[
    E ( F_\infty)_\tor \otimes_\Z \Z_p \cong \{ \zeta u  + q_E^\Z \mid \zeta \in \mu_{p^\infty}, u^{p^s} = q_E \text{ for some } s \in \N \}^{G_{F_\infty}}.
    \]
    Suppose $\zeta u$ represents a $G_{F_\infty}$-invariant class in the right hand set. Since $F_\infty$ contains $\mu_{p^\infty}$, it then follows that $(\sigma - 1) u = (\sigma - 1) (\zeta u)$ belongs to $q_E^\Z$. On the other hand, we know that $(\sigma  - 1) u \in \mu_{p^\infty}$ because $u$ is a root of $X^{p^s} - q_E$ for some $s \in \N$, and so $(\sigma - 1) u = 1$ as $q_E^\Z \cap \mu_{p^\infty} = \{ 1 \}$. This shows that $u$ belongs to $F_\infty$.\\
    We next prove that there is a natural number $a$ such that $q_E$ is not a $p^s$-th power in $F_\infty^\times$ for any $s > a$. In combination with the discussion above, this then shows that the cokernel of the map $(\Q_p / \Z_p) (1) \hookrightarrow E ( F_\infty)_\tor \otimes_\Z \Z_p$ is finite, from which it follows by taking duals that the torsion-free quotient of
    $(E ( F_\infty)_\tor \otimes_\Z \Z_p)^\vee$ is isomorphic to $\Z_p (-1)$. This isomorphism is $G_L$-equivariant and so $\psi\cdot \chi_\mathrm{cyc}$ must factor through the group $\gal{\Q_p ( \sqrt{\gamma})}{\Q_p}$. From this we deduce that $\psi$ is of the claimed form $\omega\cdot \chi_\mathrm{cyc}^{-1}$. 
    \\ 
    Finally, to prove the existence of a suitable natural number $a$, we suppose that $q_E = u^{p^s}$ for some $s \in \N$ and $u \in F_\infty$. Then $u$ is a root of $X^{p^s} - q_E$ and so $[F (u) : F] \leq p^s$. We may therefore assume that $u$ belongs to $F (\mu_{p^s})$. In particular, $q_E$ belongs to the kernel of the restriction map 
    \[
    F^\times / (F^\times)^{p^s} = H^1 ( G_{F}, \mu_{p^s}) \to H^1 ( G_{F (\mu_{p^s})}, \mu_{p^s}) =  F (\mu_{p^s})^\times / (  F (\mu_{p^s})^\times)^{p^s}.
    \]
    By the inflation-restriction sequence, the kernel of this map identifies with the cohomology group $H^1 ( \gal{F (\mu_{p^s})}{F}, \mu_{p^s})$ and hence vanishes as $p$ is odd (see, for example, \cite[Prop.~9.1.4]{NSW}). This shows that $q_E$ is a $p^s$-th power in $F$, and therefore that $s$ is bounded, as required to complete the proof of (ii).
\end{proof}

Turning now to the proof of Theorem \ref{eimc result}, we first verify that Hypothesis \ref{new strategy iwasawa hyps} is satisfied in the case that $k = \QQ$, $T = T_p(E)$, $k_\infty$ is the cyclotomic $\ZZ_p$-extension of $\QQ$, $F$ is the maximal $p$-extension of $\Q$ inside $K$, $\chi$ is any homomorphism $G_\Q \to \C^\times$ of order prime to $p$ and $r=1$. \\
At the outset, we note that parts (ii${}^\ast$) and (iv) of Hypotheses \ref{new strategy iwasawa hyps} are vacuously satisfied in this case since we are assuming that $p > 3$. To proceed, we note that the existence of the Weil pairing implies that the full preimage of $\mathrm{SL}_2 (\Z_p)$ under $\rho_{E, p}$ is equal to $G_{\Q (\mu_{p_\infty})}$.
Since we are assuming that $\mathrm{SL}_2 (\Z_p) \subseteq \mathrm{im}(\rho_{E, p})$, it follows that $\rho_{E, p} (G_{\Q (\mu_{p_\infty})}) = \mathrm{SL}_2 (\Z_p)$. Now, $\mathrm{SL}_2 (\Z_p)$ is a perfect group (because $p > 3$) and so  also $\rho_{E, p} (G_{K (\mu_{p_\infty})}) = \mathrm{SL}_2 (\Z_p)$. 
Since the natural action of $\mathrm{SL}_2 (\mathbb{F}_p)$ on $\mathbb{F}_p^{\oplus 2}$ is irreducible, this implies that $\Tbar$ is an irreducible $\kappa [G_{K (\mu_{p^\infty})}]$-representation, 
and hence also that $\Tbar (\chi)$ is an irreducible $\kappa [G_\Q]$-representation.
This verifies Hypothesis \ref{new strategy iwasawa hyps}\,(i).
In addition, as $\mathfrak{E}$ is a subfield of $K (\mu_{p^\infty})$, Hypothesis \ref{new strategy iwasawa hyps}\,(ii) is satisfied with the element $\tau$ of $G_{K (\mu_{p_\infty})}$ taken to be any preimage under $\rho_{E,p}$ of $\begin{psmallmatrix}
    1 & 1 \\ 0 & 1
\end{psmallmatrix} \in \mathrm{SL}_2 (\Z_p)$.
To verify Hypothesis \ref{new strategy iwasawa hyps}\,(iii), we use the inflation-restriction sequence
\[
H^1 ( K (\mu_{p^\infty}) / \Q, H^0 ( K (\mu_{p^\infty}), \Tbar (\chi^{-1}))) \to 
H^1 (\mathfrak{E}^\chi_T / \Q, \Tbar (\chi^{-1}))
\to H^1 ( \mathfrak{E}^\chi_T / K (\mu_{p^\infty}), \Tbar (\chi^{-1})).
\]
In particular, since the group $H^0 ( K (\mu_{p^\infty}), \Tbar (\chi^{-1})) = H^0 ( \mathrm{SL}_2 (\mathbb{F}_p), \mathbb{F}_p^{\oplus 2})$ vanishes, the sequence implies  $H^1 (\mathfrak{E}^\chi_T / \Q, \Tbar (\chi^{-1}))$ vanishes if the group  
$H^1 ( \mathfrak{E}^\chi_T / K (\mu_{p^\infty}), \Tbar (\chi^{-1})) = H^1 ( \mathrm{SL}_2 (\mathbb{F}_p), \mathbb{F}_p^{\oplus 2})$ vanishes. Since $\mathrm{SL}_2 (\mathbb{F}_p)$ is
a normal subgroup of $\mathrm{GL}_2 (\mathbb{F}_p)$ of prime-to-$p$ index, it is therefore enough for us to prove that $H^1 (\mathrm{GL}_2 (\mathbb{F}_p), \mathbb{F}_p^{\oplus 2})$ vanishes. To do this, we consider the central subgroup $U \coloneqq \{ \begin{psmallmatrix}
    a & 0 \\ 0 & a 
\end{psmallmatrix}
\mid a \in \mathbb{F}_p^\times \} 
$ of $\mathrm{GL}_2 (\mathbb{F}_p)$. We then note that $H^0 (U, \mathbb{F}_p^{\oplus 2})$ vanishes since $p > 2$ and that $H^1 (U, \mathbb{F}_p^{\oplus 2})$ vanishes since the order of $U$ is prime-to-$p$. From the inflation-restriction sequence associated to the data $(U \subset \mathrm{GL}_2 (\mathbb{F}_p), \mathbb{F}_p^{\oplus 2})$, we can therefore deduce that $H^1 (\mathrm{GL}_2 (\mathbb{F}_p), \mathbb{F}_p^{\oplus 2})$ vanishes, as required to conclude the verification 
of Hypothesis~\ref{new strategy iwasawa hyps}~(iii). \\
Finally we note that, since $H^0 (\Q, \Tbar (\chi))$ vanishes and we are taking $\Sigma$ to be the empty set, Hypothesis~\ref{new strategy iwasawa hyps}~(v) is satisfied because $r = 1$ and, after replacing $\Tbar (\chi^{-1})$ by $\Tbar (\chi)$, the same argument as above shows that $H^1 (\mathfrak{E}_T^{\chi}/ \Q, \Tbar (\chi))$ vanishes. \\
Having verified all of the necessary hypotheses, we may therefore apply Theorem \ref{new iwasawa theory euler systems main result}\,(i) to the Euler system $y^\mathrm{Kato}$ from Theorem \ref{kato euler system}. Since 
\[ z_{K_\infty}^\mathrm{Kato} = ( {\prod}_{\ell \in S(K) \setminus S_0 (K)} \Eul_\ell (\Frob_\ell^{-1})) y_{K_\infty}^\mathrm{Kato},\]
we thereby deduce that
\begin{align*}
& \phantom{=\,} \Fitt^0_{\Lambda_K} ({\bigoplus}_{v \in \Pi^p_{K_\infty}} (E (K_{\infty, v})_\tor \otimes_\Z \Z_p )^\vee))^{\ast \ast} \cdot z_{K_\infty}^\mathrm{Kato} \\
& =
 \Fitt^0_{\Lambda_K} ({\bigoplus}_{v \in \Pi^p_{K_\infty}} (E (K_{\infty, v})_\tor \otimes_\Z \Z_p )^\vee))^{\ast \ast} \cdot \big( {\prod}_{\ell \in S(K) \setminus S_0 (K)} \Eul_\ell (\Frob_\ell^{-1})) \big) \cdot y_{K_\infty} \\
 & \in \Theta_{K_\infty, S(K)} ( C_{S(K)} (\cT)). 
\end{align*}
To analyse the Fitting ideal that occurs here we recall from Lemma \ref{error terms Iwasawa theory elliptic curves} that $(E (K_{\infty,v})_\tor \otimes_\Z \Z_p)^\vee $ is finite if $E$ has potentially good reduction at $p$, and we now further claim that this group is finite if $\mu_K [p] = (0)$. To show this, we assume that $(E (K_{\infty,v})_\tor \otimes_\Z \Z_p)^\vee$, and hence also $(E (K_{v} (\mu_{p^\infty}))_\tor \otimes_\Z \Z_p)^\vee$, is infinite. From Lemma \ref{error terms Iwasawa theory elliptic curves} we then know that as a $G_{\Q_p}$-module the divisible subgroup of $E (K_{v} (\mu_{p^\infty}))_\tor \otimes_\Z \Z_p$ is isomorphic to  $(\Q_p / \Z_p) (1) (\omega^{-1})$ 
 with a character $\omega \: G_{\Q_p} \to \Z_p^\times$ that factors through $\gal{\Q_p ( \sqrt{\gamma})}{\Q_p}$. 
By assumption, $K ( \sqrt{\gamma})$ does not contain a primitive $p$-th root of unity in this case, and so the same is true for $K_{n, v} ( \sqrt{\gamma})$ for every $n$. It follows that $\gal{K_v (\mu_{p^\infty}, \sqrt{\gamma})}{K_{v, \infty} (\sqrt{\gamma})}$ is a non-trivial cyclic group with generator $\sigma$. In particular, we have
$(\omega \cdot\chi_\mathrm{cyc}^{-1}) (\sigma) \neq 1$. From this we deduce that the maximal subgroup of $(\Q_p / \Z_p) (1) (\omega^{-1})$ fixed by $\sigma$ is finite, and hence that
also $E (K_{\infty,v})_\tor \otimes_\Z \Z_p \subseteq ( E (K_{v} (\mu_{p^\infty}, \sqrt{\gamma}))_\tor \otimes_\Z \Z_p)^{\sigma = 1}$ must be finite, as claimed. \\
In this case, therefore, the ideal 
$\Fitt^0_{\Lambda_K} ({\bigoplus}_{v \in \Pi^p_{K_\infty}} (E (K_{\infty, v})_\tor \otimes_\Z \Z_p )^\vee))^{\ast \ast}$, which is uniquely determined by its localisations at height-one primes by Lemma \ref{ryotaro-useful-lem}\,(ii), is equal to $\Lambda_K$. Given this fact, the last displayed inclusion directly implies the containment of Theorem \ref{eimc result}. \qed 

\section{Tate motives}\label{Gm section}

For a number field $F$ and integer $a$ we now consider the motive 
\[ \QQ_F(a)\coloneqq h^0 (\Spec F)(a).\]
In particular, for a finite abelian extension of number fields $K/k$, the primary aim of this section will be to  derive consequences of Theorem \ref{new iwasawa theory euler systems main result} that relate to   $\QQ_K(a)$, regarded  as defined over $k$ and with coefficients $\QQ[\cG_K]$. 

\subsection{Main conjectures of higher-rank Iwasawa theory}

With $K / k$ as above, we write $V (K) \subseteq \Pi_k^\infty$ for the subset of $\Pi_K^\infty$ comprising places that split in $K$. We then fix a subset  $V$ of $\Pi_k^\infty$, set
\[ r \coloneqq |V|\]
and write $\Omega^V$ for the collection of finite abelian extensions of $k$ for which $V (K) = V$.\\
We fix a labelling $\Pi_k^\infty \coloneqq \{ v_1, \dots, v_n\}$ and, for every $i \in \{1, \dots, n \}$, an extension $w_{k^\mathrm{c}, i}$ of $v_i$ to $k^\mathrm{c}$ and write $w_{K, i}$ for the restriction of $w_{k^\mathrm{c}, i}$ to $K$. Then $b_\bullet = b_{K, \bullet} \coloneqq ( w_{K, i} : v_i \in V(K))$ is an ordered basis of the free $\Z_p [\cG_K]$-module $Y_{V (K)} (\Z_p (1)_{K / k})$.\\
For any finite subset $S$ of $\Pi_k$ that contains $\Pi_k^\infty \cup \Pi_k^p$ and finite subset $\Sigma$ of $\Pi_k\setminus S (K)$, one can use the values at $0$ of the $r$-th derivatives of Dirichlet $L$-series over $k$ to define a `Rubin--Stark element' (depending on the basis $b_\bullet$)
\[
\varepsilon^V_{K / k, S(K), \Sigma} \in \C \otimes_\Z \exprod^{r}_{\Z [\cG_K]} \cO_{K, S(K)}^\times
\]
(for details see \cite[\S\,5.1]{BKS}). After fixing an isomorphism $\C \cong \C_p$, we can regard $\varepsilon^V_{K / k, S(K), \Sigma}$ as an element of $\C_p \otimes_{\Z_p} \exprod^{r}_{\Z_p [\cG_K]} H^1 (\cO_{K, S(K)}, \Z_p (1))$. Then the `$p$-component' of the Rubin--Stark Conjecture  predicts that if  $H^1_\Sigma (\cO_{K, S(K)}, \Z_p(1))$ is $\Z_p$-torsion free (or, as is equivalent by Lemma \ref{torsion lemma new}, the group $H^0_\Sigma (\cO_{K, S(K)}, (\Q_p/\Z_p) (1))$ vanishes), then one has 
\[
\varepsilon^V_{K / k, S(K), \Sigma} \in \bidual^{r}_{\Z_p [\cG_K]} H^1_\Sigma (\cO_{K, S(K)}, \Z_p (1)) 
\]
(cf.\@ also Remark \ref{deligne remark}). 
Assuming this conjecture to be valid, \cite[Prop.~6.1]{Rub96} then implies that one obtains a well-defined Euler system by setting 
\[
\varepsilon_k \coloneqq ( \varepsilon^V_{K / k, S(K), \Sigma})_{K \in \Omega^V} \in \ES_{\Sigma, S}^{r} (\Z_p (1)).
\]
We now fix a $\Z_p$-extension $k_\infty$ of  $k$ in which no finite place splits completely, and set $K_\infty \coloneqq K \cdot k_\infty$. Fix a splitting $\cG_{K_\infty} \cong \Delta_K \times \Gamma_{\!\!K}$ with 
$\Gamma_{\!\!K} \cong \Z_p^n$ for some $n > 0$ and $\Delta_K$ a finite abelian group,
and a direct product decomposition $\Delta_K = \nabla_{\!\!K} \times \square_K$ as in (\ref{sylow decomp}).  
Setting 
\[ L \coloneqq K_\infty^{\langle \Gamma_{\!\!K}, \square_K \rangle}\quad\text{and}\quad F \coloneqq K^{\langle \Gamma_{\!\!K}, \nabla_{\!\!K} \rangle},\] 
we then also have a decomposition $K_\infty^{\Gamma_{\!\!K}} = L \cdot F$.
Fix a character $\chi \: \nabla_{\!\!K} \to \overline{\Q_p}^\times$ and choose an unramified extension $\cO$ of $\Z_p$ that contains the values of $\chi$. We consider the representations
\[
T_\chi \coloneqq \cO (1) (\chi) 
\quad \text{ and } \quad
\cT_\chi \coloneqq \Lambda_F (1)(\chi).
\]
Writing $\Lambda_K \coloneqq \cO \llbracket \cG_{K_\infty} \rrbracket$, we then have the `projection map'
\begin{equation} \label{theta map Zp (1)}
\Theta_{S(K), b_\bullet} \: \Det_{\Lambda_K} ( C^\bullet_{k, \Sigma} ( \cT_\chi))
\to \bidual^r_{\Lambda_K} H^1_\Sigma ( \cO_{k, S(K)}, \cT_\chi)
\end{equation}
from (\ref{definition Theta map}).
Since $\varepsilon_k$ belongs to $\ES_{\Sigma, S}^r (\Z_p (1))$, we obtain an element
\begin{align*}
\varepsilon^\chi_{K_\infty, \Sigma} \coloneqq ( e_{\chi} \varepsilon^V_{E / k, S(K), \Sigma})_{E \subseteq K_\infty}
 \in&\,\,  {\varprojlim}_{E \subseteq K_\infty} \bidual^r_{\cO [\cG_E]} H^1_\Sigma (\cO_{K, S(K)}, T_\chi) \\
   \cong&\,\, \bidual^r_{\Lambda_{F}} H^1_\Sigma ( \cO_{K, S(K)}, \cT_\chi).
\end{align*}

We can now state the equivariant Iwasawa Main Conjecture for $T_\chi$.

\begin{conj} \label{eIMC G_m}
    There is a $\Lambda_F$-basis $\fz^\chi_{F_\infty, S(K), \Sigma}$ of $\Det_{ \Lambda_F} ( C_{S(K), \Sigma} ( \cT_\chi))$ such that
    \[
    \varepsilon_{K_\infty, \Sigma}^\chi = \Theta_{S(K), b_\bullet} ( \fz^\chi_{K_\infty, S(K), \Sigma}).
    \]
\end{conj}

\begin{rk} It is easily seen that the  validity of Conjecture \ref{eIMC G_m} for every character $\chi$ in $\widehat{\Delta}$ is equivalent to the validity of the `rank-$r$ component' of the conjecture \cite[Conj.~3.1]{BKS2} of Kurihara et al.
\end{rk}

For every finite subextension $E$ of $K_\infty / k$, we write $A_{E, S(K), \Sigma} \coloneqq \Cl_{E, S(K), \Sigma} \otimes_\Z \Z_p$ for the `$p$-part' of the $S(K)_E$-ray class group modulo $\Sigma_E$ of $E$, and then set 
\[ A_{K_\infty, S(K), \Sigma} \coloneqq \varprojlim_E A_{K, S(K), \Sigma},\] 
where the limit is taken with respect to norm maps.\\
To state our main result concerning  Conjecture \ref{eIMC G_m}, we write $\omega_p \: G_k \to \mu_{p - 1} \subseteq \Z_p^\times$ for the $p$-adic Teichm\"uller character of $k$.
We remark that in certain situations one can even deduce the validity of the relevant case of Kato's Conjecture (Conjecture \ref{etnc statement}) from this result, and we will discuss two such examples in \S\,\ref{imaginary quadratic fields section} and \S\,\ref{more general cases section}.

\begin{thm} \label{G_m main result}
    Assume that $p > 3$, that $H^1_\Sigma (\cO_{K, S(K)}, \Z_p (1))$ is $\Z_p$-torsion free, that the $p$-component of the rank-$r$ Rubin--Stark Conjecture holds for all finite abelian extensions of $k$, and that $\chi \not \in \{ \bm{1}, \omega_p \}$.
Then one has 
\[ \Fitt^0_{\Lambda_F} \Big( \Ext^1_{\Lambda_F} \big( \big (\bidual^r_{\Lambda_{F}} H^1_\Sigma ( \cO_{K, S(K)}, \cT_\chi) \big) / (\Lambda_F \varepsilon_{K_\infty, \Sigma}^\chi), \Lambda_F\big) \Big) 
    \subseteq \Fitt^0_{\Lambda_F} \big( A_{K_\infty, S(K), \Sigma}^\chi \big)^{\ast \ast}
    \]
and
\[\Fitt^0_{\Lambda_F} ( Y_{\Pi^p_k} (\cT_\chi))^{\ast \ast}
    \cdot \varepsilon_{K_\infty, \Sigma}^\chi \subseteq \Theta_{K_\infty / k, S(K), b_\bullet} ( \Det_{ \Lambda_F} ( C_{S(K), \Sigma} ( \cT_\chi))).\]
\end{thm}

\begin{proof}
    We first verify that Hypotheses \ref{new strategy iwasawa hyps} is valid for the data $(T, k_\infty,F,\chi, r)$. To do this, we note that the conditions of  Hypothesis \ref{new strategy iwasawa hyps}\, (i) and (ii)  are clearly satisfied because $\overline{T_\chi}$ is one-dimensional (in particular, in (ii) we may take the element $\tau$ to be trivial). Furthermore, it is proved in \cite[Lem.\@ 5.4]{bss2} that 
    the condition $\chi \not \in \{ \bm{1}, \omega_p \}$ implies
    the vanishing of both of the groups $H^1 (\mathfrak{E}_T^\chi / k, \overline{T_\chi}^\ast (1))$ and $H^1 (\mathfrak{E}_T^\chi/ k, \overline{T_\chi})$. This shows that condition (iii) in Hypothesis \ref{new strategy iwasawa hyps} is valid, and that condition (v) is always valid if $|\Sigma|\le 1$. Since Conjecture \ref{eIMC G_m} is independent of $\Sigma$ (cf.\@ the argument of \cite[Prop.\@ 3.4]{BKS}) and $H^1_\Sigma (\cO_{K, S(K)}, \Z_p (1))$ is automatically $\Z_p$-torsion free if $\Sigma\not=\emptyset$, we may assume that $|\Sigma|\le 1$ in which case the inequality of Hypothesis \ref{new strategy iwasawa hyps}\,(v) is clearly valid. Finally, the conditions of Hypothesis \ref{new strategy iwasawa hyps}\,(ii${}^\ast$) and (iv) are satisfied trivially since $p > 3$.\\
    Having thus verified Hypotheses \ref{new strategy iwasawa hyps}, we now take $S$ to be $S_0$ so that $S_\ram (\Z_p (1)) \setminus S_0$ is empty and so condition (a) of Theorem \ref{new iwasawa theory euler systems main result} is satisfied vacuously. We then finally note that the validity of condition (b) of Theorem \ref{new iwasawa theory euler systems main result} follows from the assumption  that the module $H^1_\Sigma (\cO_{K, S(K)}, \Z_p (1))$ is $\Z_p$-torsion free.
     The result of Theorem \ref{new iwasawa theory euler systems main result} can therefore be applied in this setting. \\
   We now note that the idempotent $\epsilon_K$ from \S\,\ref{statement iwasawa result section} (when taking $Y = Y_{V (K)} (\cT_\chi)$) is equal to 
    \[ \epsilon_{K, V} \coloneqq \prod_{v \in S_\infty \setminus V} (1 - e_{\cG_{K, v}}) \cdot \prod_{v \in V} e_{\cG_{K, v}},\] 
    and hence  acts trivially on $\varepsilon_{K_\infty, \Sigma}$. In light of this observation, we then deduce from Theorem~\ref{new iwasawa theory euler systems main result}\,(iii) with $T \coloneqq \Z_p (1)$ an inclusion
    \begin{align*}
    & \phantom{=\;} \Fitt^0_{\Lambda_F} \Big(  \Ext^1_{\Lambda_F} \big( \big (\bidual^r_{\Lambda_{F}} H^1_\Sigma ( \cO_{K, S(K)}, \cT_\chi) \big) / (\Lambda_F \varepsilon_{K_\infty, \Sigma}^\chi), \Lambda_F \big) \Big)  \\
    & = 
    \big \{ f ( \varepsilon_{K_\infty, \Sigma}^\chi) \mid f \in \exprod^r_{\Lambda_F} H^1 (\cO_{K, S(K)}, \cT_\chi)^\ast \big \} \\
    & \subseteq \Fitt^0_{\Lambda_F} \big ( H^1_{\cF^\vee_{\mathrm{rel}, \Sigma}} (k, \cT_\chi^\vee (1))^\vee \big)^{\ast \ast} \\
    & = \Fitt^0_{\Lambda_F} \big( A_{K_\infty, S(K), \Sigma}^\chi \big)^{\ast \ast}.
    \end{align*}
     (Here the first equality is by \cite[Lem.\@ A.10]{scarcity} and the final equality by \cite[Prop.\@ 1.6.2]{Rubin-euler}.)
     This proves the first claim in \ref{G_m main result}, and the second claim follows from Theorem \ref{new iwasawa theory euler systems main result}\,(i) upon noting that $\Fitt^r_{\Lambda_K} ( Y_{\Pi^p_k \cup \Pi^\infty_k} (\cT_\chi)) = \epsilon_{K, V} \Fitt^0_{\Lambda_K} (Y_{\Pi^p_k} (\cT_\chi))$ and $\epsilon_{K, V}$ acts as the identity on $\varepsilon_{K_\infty, V}$.
\end{proof}

\subsection{Consequences over imaginary quadratic fields}
\label{imaginary quadratic fields section}

Throughout this subsection, we assume that $k$ is imaginary quadratic and $p$ is odd. We shall first show that if $p > 3$, then Theorem \ref{G_m main result} implies the validity of Kato's Conjecture for $(\QQ_K(0),\ZZ_p[\cG_K])$ for every finite abelian extensions $K$ of $k$. We shall then explain how this result can be combined with the general approach of \S\,\ref{twisting lemma section} to derive the validity of Kato's Conjecture in other  cases. In this way, we realise the strategy discussed by Kato \cite[Ch.\@ I, \S\,3.3]{Kato93b} (where the case of motives of the form $\QQ(0)_F$ is  referred to as `the universal case') and in a more general context by Huber and Kings \cite{hk}. We recall that this sort of approach has already been used extensively in the literature (for more detailed discussion see, for example, \cite{bbs} and the references therein). 

 \subsubsection{Kato's Conjecture for Tate motives} 

The following consequence of Theorem \ref{G_m main result} extends the main result of Bley \cite{Bley06}.  

\begin{thm} \label{etnc imaginary quadratic fields}
If $p > 3$, then the conjecture $\TNC(\QQ_K(0), \Z_p [\cG_K])$ is valid for every finite abelian extension $K$ of $k$.  
\end{thm}

\begin{proof}
The conjecture $\TNC (\QQ_K(0), \Z_p [\cG_K])$ decomposes into the collection of corresponding conjectures for $(h^0 (\Spec (L_\chi F)), \cO (\chi) [P])$ with $\chi$ ranging over the characters of $\Delta$. If $\chi \in \{ \bm{1}, \omega_p\}$, then the validity of the latter conjecture was proven by Hofer and the first author in \cite[Th.~B]{BullachHofer}. (Note that the vanishing of the classical $\mu$-invariant for $F$, which the result of \cite[Th.\@ B]{BullachHofer} is conditional on, is known as a consequence of the Ferrero--Washington theorem and the fact that $[F : k]$ is a power of $p$, cf.\@ \cite[Prop.\@ 6.7\,(b)]{BullachHofer}.)
In the following we may therefore assume that $\chi \not \in \{ \bm{1}, \omega_p\}$. We similarly may assume that $p$ does not split in $k$ since otherwise the result of \cite[Th.\@ B]{BullachHofer} applies again.\\
In the remaining case, then, the second claim of Theorem \ref{G_m main result}\,(b) can be improved to give the full Conjecture \ref{eIMC G_m}. 
To explain this, we first recall that the rank-one Rubin--Stark conjecture holds for all finite abelian extensions of $k$ since the relevant Rubin--Stark elements admit a description in terms of elliptic units (cf.\@ \cite[Ch.\@ IV, Prop.\@ 3.9]{Tate}).\\
Further, taking $k_\infty$ to be the full $\Z_p^2$-extension of $k$, the module $Y_{K_\infty, \Pi^p_k}$ becomes pseudo-null over $\Lambda_K$  because, as proved in \cite[Ch.~II, Prop.~1.9]{deS87}, $p$ is finitely decomposed in $k_\infty$  (cf.\@ \cite[Lem.\@ 6.6\,(c)]{scarcity}). A standard argument via localising at height-one primes of $\Lambda_K$ therefore removes the factor $\Fitt^0_{\Lambda_K} ( Y_{\Pi^p_k} (\Lambda_K (1)))$ from the second claim in Theorem \ref{G_m main result} and the resulting inclusion must be an equality by the analytic class number formula (see \cite[Prop.\@ 6.4\,(b)\,(ii)]{scarcity} for details). Having established the equivariant Iwasawa main conjecture (Conjecture \ref{eIMC G_m}) in this way, the validity of $(h^0 (\Spec (L_\chi F)), \cO (\chi) [P])$ then follows from the descent formalism developed in \cite{BKS2} combined with \cite[Th.\@ A]{BullachHofer} (see the proof of \cite[Th.~6.9]{BullachHofer} for details).
\end{proof}

The following consequence of Theorem \ref{etnc imaginary quadratic fields} strengthens the main result of Johnson-Leung in \cite{Johnson-Leung} and itself has a variety of interesting consequences, including the verification of the Quillen--Lichtenbaum Conjecture in a new family of cases (for details of which, see Remark~\ref{lichtenbaum cor} below). 

\begin{cor} \label{etnc imaginary quadratic twist main result}
    The conjecture $\TNC(\QQ_K(1 - j), \Z_p [\cG_K])$ is valid for every finite abelian extension $K$ of $k$ and every integer $j$ with $j> 1$. 
\end{cor}

\begin{proof} Throughout this argument we fix $j>1$ and, as a first step, we explicate the conjecture $\TNC (\QQ_K(1 - j), \Z_p [\cG_K])$. To do this, we fix a finite set $S \subseteq \Pi_k$ containing $\Pi_k^\infty \cup \Pi_k^p \cup S_\ram (K / k)$ and recall that the `Chern class character' map
\[
\mathrm{ch}_{K, j} \: K_{2j - 1} ( \cO_K) \otimes_\Z \Z_p \to H^1 (\cO_{K, S}, \Z_p (1 - j))
\]
is an isomorphism if $p$ is odd. (This is a consequence of the validity of the Bloch--Kato conjecture that follows from work of Voevodsky and Rost, and completed by Weibel in \cite{Weibel-NormResidue}.) Next we fix an embedding $\iota_0 \: k \hookrightarrow \C$ and recall the `Beilinson regulator' map (as defined, for example, in \cite[\S\,10.3]{BurgosGil})
\[
\rho_{\mathrm{Bei}, K, j} \: K_{2j - 1} (\cO_K) \to \R \otimes_\Z (K \otimes_k H^0 ( (\Spec k)^{\iota_0} (\C), \Q (- j))).
\]
Fix a $\Z [\cG_K]$-basis $\eta$ of $K \otimes_k H^0 ( (\Spec k)^{\iota_0} (\C), \Z (- j)))$, which amounts to fixing an embedding $\iota \: \Q^c \hookrightarrow \C$ that extends $\iota_0$, and write $\delta (\eta)$ for the image of $\eta$ under the comparison isomorphism
\[
\Z_p \otimes (K \otimes_k H^0 ( (\Spec k)^{\iota_0} (\C), \Z_p (- j)))) \cong \cY_{K} (\Z_p (1 - j)) \coloneqq {\bigoplus}_{w\in \Pi_K^\infty} H^0 (K_w, \Z_p ( - j)).\]
Then, since $H^2 (\cO_{K, S (K)}, \Z_p (1 - j))$ is finite (by Soul\'e \cite[Th.\@ 10.3.27]{NSW}), the construction of (\ref{definition of finite level theta map}) specialises to the composite map
\begin{align*}
    \Theta_{K, S, j} \: \Det_{\Q_p [\cG_K]} (C_{S} ( \Q_p (1 - j)_{K / k})) 
    & \xrightarrow{\simeq} H^1 (\cO_{K, S}, \Q_p (1 - j)) \otimes_{\Z_p [\cG_K]} \cY_{K} (\Z_p (1 - j)) \\
    & \xrightarrow{\simeq} H^1 (\cO_{K, S}, \Q_p (1 - j)),
\end{align*}
where the first arrow is the natural `passage-to-cohomology' map and the second arrow is induced by sending $\delta (\eta) \mapsto 1$. 
Now, the assumption $j > 1$ implies that, for any finite set $\Sigma \subseteq \Pi_k$ with $\Sigma \cap S (K) = \emptyset$, one has that $\delta_{K, \Sigma} (j) \coloneqq \prod_{v \in \Sigma } ( 1 - \NN v^{1 - j} \Frob_v^{-1})$ is a nonzero divisor in $\Z_p [\cG_K]$. 
It follows that the map $C_{S, \Sigma} (\Q_p (1 - j)) \to C_{S} (\Q_p (1 - j))$ is an isomorphism in $D (\Q_p [\cG_K])$ and this allows us to regard $\Det_{\Z_p[\cG_K]} ( C_{S, \Sigma} (\Z_p (1 - j)))$
as a submodule of the domain of the map $\Theta_{K, S, j}$.
Moreover, the proof of \cite[Prop.\@ 3.4]{BKS} shows that the statement 
of $\TNC (\QQ_K(1 - j), \Z_p [\cG_K])$ is equivalent to asserting the existence of a $\Z_p [\cG_K]$-basis $\fz_{K, \Sigma} (j)$ of $\Det_{\Z_p [\cG_K]} (C_{S, \Sigma} ( \Z_p (1 - j)_{K / k}))$ with both of the following properties
\begin{equation}\label{properties}
\begin{cases} (\mathrm{ch}_{K, j} \circ \Theta_{K, S, j}) (\fz_{K, \Sigma} ( j)) \in \Q \otimes_\Z K_{1 - 2j} (\cO_K);\\
(\rho_{\mathrm{Bei}, K, j} \circ \mathrm{ch}_{K, j}^{-1} \circ \Theta_{K, S, j}) (\fz_{K, \Sigma} ( j)) = \delta_{K, \Sigma} (j) \cdot (\sum_{\chi \in \widehat{\cG_K}} L'_{S} (\chi^{-1}, 1 - j) e_{\chi}) \cdot \eta.\end{cases}\end{equation}

To construct such a basis, we denote by $\m$ the conductor of $K$, and write $w_\m$ for the number of roots of unity in $k$ that are congruent to $1 \mod \m$ (so $w_\m \mid 12$). Writing $\xi_\m (j) \in K_{1 - 2j} (\cO_K)$ for the element constructed by Johnson-Leung in \cite[Th.\@ 3.3]{Johnson-Leung}, we then define
\[
c_K (j) \coloneqq \mathrm{cores}_{K (p^n \m) / K} (w_\m^{-1} \otimes \xi_\m (j)) \quad \in \Z [1 / w_\m] \otimes_\Z K_{2j - 1} (\cO_K) 
\]
with $n$ big enough such that $K(\p^n \m) / k$ is ramified at all $\p \mid p$. 
Results of Deninger \cite{deninger90}, as adapted by Johnson-Leung in \cite[Th.~3.3]{Johnson-Leung}, then show that 
\begin{equation} \label{regulator computation}
\rho_{\mathrm{Bo}, K, j} ( c_K (j)) = (-1)^{1 + j} \cdot \frac{(2\cdot \NN \m)^{- (1 + j)}}{(-2j)!} \cdot ({\sum}_{\chi \in \widehat{\cG_K}} L'_S (\chi^{-1}, 1 - j) e_{\chi}) \cdot \eta
\end{equation}
with $S \coloneqq \Pi_k^\infty \cup \Pi_k^p \cup S_\ram (K / k)$ and $\rho_{\mathrm{Bo}, K, j}$ the Borel regulator map. \\
We now write $K_\infty$ for the cyclotomic $\Z_p$-extension of $K$ and set $\Lambda \coloneqq \Z_p \llbracket \cG_{K_\infty} \rrbracket$. Using our fixed choice of embedding $\iota \: \Q^c \hookrightarrow \C$ we can define a basis of $\Z_p (1) = {\varprojlim}_n \mu_{p^n}$ as $\zeta \coloneqq ( \iota^{-1} ( e^{2 \pi i / p^n}))_n$.\\
Choose a prime ideal $\a \nmid 6p \m$ and set $\Sigma \coloneqq \{ \a \}$.
We then obtain a map
\[
\mathrm{Tw}_j \: H^1_\Sigma (\cO_{k, S}, \Lambda (1)) \xrightarrow{\otimes \zeta^{\otimes  - j}} H^1_\Sigma (\cO_{k, S}, \Lambda (1)) \otimes_\Lambda \Z_p (-j)_{K / k} \cong H^1_\Sigma (\cO_{K, S}, \Z_p (1  - j)),
\]
where the isomorphism is induced by Proposition \ref{construction complex}\,(d).\\
Write $\varepsilon_{K_\infty, \Sigma} \coloneqq ( \varepsilon^V_{F / k, S, \Sigma})_{F \subseteq K_\infty}$ for the family of Rubin--Stark elements with $V = \Pi_k^\infty$. Then, since Rubin--Stark elements in this context admit a description in terms of elliptic units (cf.\@ \cite[Ex.~2.3~(c)]{BullachHofer}), 
results of Kings \cite{Kings01} show  that
\[
(-2j)! \cdot \delta_{K, \Sigma} (j) \cdot \mathrm{ch}_{K, j} ( c_{K} (j)) = 
\pm 2^a \cdot \NN \m^{- (1 + j)} \cdot \mathrm{Tw}_j ( \varepsilon_{K_\infty, \Sigma})
\]
for a suitable integer $a$ (cf.\@ \cite[Th.~3.6]{Johnson-Leung}). 
Combining this formula with (\ref{regulator computation}) and the equality $\rho_{\mathrm{Bo}, K, j} = 2 \cdot \rho_{\mathrm{Bei}, K, j}$ proved by Burgos Gil \cite{BurgosGil}, we can then conclude that 
\begin{equation} \label{twisted elliptic units and L function}
(\rho_{\mathrm{Bei}, K, j} \circ \mathrm{ch}_{K, j}^{-1}) ( \mathrm{Tw}_j ( \varepsilon_{K_\infty, \{ \a \}})) = \pm 2^{a'} \cdot \delta_{K, \Sigma} (j) \cdot ({\sum}_{\chi \in \widehat{\cG_K}} L'_{S} (\chi^{-1}, 1 - j) e_{\chi}) \cdot \eta
\end{equation}
for an integer $a'$ (that depends on $j$). 

Now, by Theorem \ref{etnc imaginary quadratic fields}, there exists a $\Lambda$-basis $\fz_{K, \Sigma} (0)$ of $\Det_\Lambda (C_{S, \Sigma} (\Lambda (1)))$ with
\[\Theta_{K, S, 1} (\fz_{K, \Sigma} (0)) = \varepsilon_{K_\infty, \Sigma}.\] 
In addition,  Proposition \ref{twisting lemma}\,(i)\,(a) gives a commutative diagram 
\begin{cdiagram}
    \Det_\Lambda (C_{S, \Sigma} (\Lambda (1))) \arrow[twoheadrightarrow]{d}{\mathrm{Tw}^\mathrm{det}_j} \arrow{rr}{\Theta_{K, S, 1}} & & H^1_\Sigma (\cO_{k, S}, \Lambda (1)) \arrow{d}{\mathrm{Tw_j}} \\ 
    \Det_{\Z_p [\cG_K]} ( C_{S, \Sigma} ( \Z_p (1 - j)_{K / k})) \arrow{rr}{\Theta_{K, S, j}} & &
    H^1_\Sigma (\cO_{K, S}, \Z_p (1 - j))
\end{cdiagram}%
and so the equality  (\ref{twisted elliptic units and L function}) implies that the element $\fz_{K, \Sigma} (j) \coloneqq \mp 2^{-a'} \cdot \mathrm{Tw}_j^\mathrm{det} (\fz_{K, \Sigma} (0))$ has both of the properties in (\ref{properties}).
Finally we note that,  since $p$ is odd, $\fz_{K, \Sigma} (j)$ is also a $\Z_p [\cG_K]$-basis of $\Det_{\Z_p [\cG_K]} ( C_{S, \Sigma} ( \Z_p (1 - j)_{K / k}))$ and so this concludes the proof of the claimed result. 
\end{proof}

\begin{rk}\label{lichtenbaum cor} Corollary \ref{etnc imaginary quadratic twist main result} has a variety of explicit consequences, including the following. 

\begin{itemize}
\item[(i)]  The validity of  $\TNC(\QQ_K(1 - j), \Z_p [\cG_K])$ implies that of $\TNC(\QQ_K(1 - j), \Z_p)$, where in the latter case $\QQ_K(1 - j)$ is regarded as defined over $K$ and with coefficients $\QQ$ (cf.\@ Remark \ref{forgetful etnc remark}).  Hence, upon combining  Theorem \ref{etnc imaginary quadratic twist main result} with the interpretation of $\TNC(\QQ_K(1 - j), \Z_p)$ given by \cite[Th.\@ 2.3]{BdJGRY}, one deduces  an equality 
 \[
    \ord_p \Big( \frac{\zeta^\ast_K (1 - j)}{R_{K} (1 - j)} \Big) =  \ord_p \Big ( \frac{|K_{2j - 2} (\cO_K)|}{|K_{2j - 1} (\cO_K)_\tor|} \Big).
    \]
Here $\zeta_K^\ast (1 - j)$  denotes the leading term at $s=1-j$ of the Dedekind $\zeta$-function of $K$ and  $R_K (1 - j)$ the Borel regulator of $K$ at $1 - j$ (as recalled explicitly in \cite[Th.~2.1~(iv)]{BdJGRY}). We note, in particular, that this displayed equality extends the known validity of the Quillen--Lichtenbaum Conjecture (as formulated in \cite{Lichtenbaum}) to a new family of cases. 

\item[(ii)] Corollary \ref{etnc imaginary quadratic twist main result} directly implies that the main result of El Boukhari \cite{ElBoukhari} is unconditionally valid for $p > 3$. From the results of loc.\@ cit.\@ one can therefore immediately derive several concrete consequences of Theorem \ref{etnc imaginary quadratic twist main result} (the details of which we leave to the reader).
\end{itemize}
\end{rk} 

\begin{rk} To extend Corollary \ref{etnc imaginary quadratic twist main result} to include the case $p=2$, one must precisely determine the factor $2^{a'}$ in (\ref{twisted elliptic units and L function}) and this is a delicate problem. For a treatment of related questions see \cite[\S\,2]{BdJGRY}.  \end{rk}

\subsubsection{Kato's Conjecture for elliptic curves with complex multiplication}

In this section, we shall combine Theorem \ref{etnc imaginary quadratic fields} with the general strategy described by 
Kato in \cite[Ch.\@ I, \S\,3.3]{Kato93b} and the concrete approach used by Burungale and Flach in \cite{BurungaleFlach} in order to prove an equivariant refinement of the main result of loc.\@ cit.\@ (for details see Theorem \ref{CM elliptic curves main result} and Remark \ref{bf remark} below). 

To do this we fix an elliptic curve $E$ that is defined over a number field $F$ containing $k$ and has complex multiplication by the ring of integers $\cO_k$ of $k$. We assume that $F ( E_\mathrm{tors}) / k$ is abelian. The Weil restriction
\[ B \coloneqq \mathrm{Res}_k^F (E)\] 
of $E$ is then an abelian variety of dimension $[F : k]$ defined over $k$ for which there is an algebra isomorphism 
\begin{equation}\label{av decomp}
A \coloneqq \mathrm{End}_k (B) \otimes_\Z \Q \cong L_1 \times \dots \times L_t, 
\end{equation}
where each $L_i$ is a CM field that contains $k$ and $\sum_{i = 1}^{i=t} [L_i : k] = [F : k]$ (cf.\@ \cite[Prop.~3.2]{BurungaleFlach}). 
We also fix an odd prime $p$ and define rings  
\[ \cO_p \coloneqq \Z_p \otimes_\Z \cO_k\,\,\text{ and }\,\,\mathcal{A}_p \coloneqq \mathrm{End}_k (B) \otimes_\Z \Z_p.\] 

\begin{lem} \label{CM elliptic curves lemma 1}
    The ring $\cA_p$ is isomorphic to $\cO_p [\cG_F]$ and the $G_k$-module $T_p B$ to $\mathrm{Ind}_{G_F}^{G_k} (T_p E)$. In particular, $\cA_p$ is Gorenstein and $T_p B$ is a free $\cA_p$-module of rank one.
\end{lem}

\begin{proof}
    For every $\sigma \in \cG_F$ we denote by $E^\sigma$ the $\sigma$-conjugate of $E$. Then $\sigma$ induces an isomorphism of abelian groups $E (\Q^c) \xrightarrow{\simeq} E^\sigma (\Q^c)$. From the decomposition $B (\Q^c) = \prod_{\sigma \in \cG_F} E^\sigma (\Q^c)$ we then see that $T_p B = \prod_{\sigma \in \cG_F} T_p E^\sigma \cong \mathrm{Ind}_{G_F}^{G_k} (T_p E)$, as claimed (cf.\@ also \cite[(a) on p.\@ 178]{Milne72}). \\ 
    Label the elements of $\cG_F$ as $\bar \sigma_1, \dots, \bar \sigma_n$ and, for every $i$, fix $\sigma_i \in G_k$ that restricts to $\bar \sigma_i$. These choices then define an isomorphism 
    \[
    f \: \mathrm{Ind}_{G_F}^{G_k} (T_p E) = T_p E \otimes_{\Z_p \llbracket G_F \rrbracket} \Z_p \llbracket G_k \rrbracket \to T_p E \otimes_{\Z_p} \Z_p [\cG_F], \quad a \otimes \sigma_i h \mapsto (h \cdot a) \otimes \bar \sigma_i
    \]
    of $\Z_p$-modules (here $h \in G_F$). We claim that this isomorphism is compatible with the action of $\Z_p \llbracket G_k \rrbracket$ when acting $T_p E \otimes_{\Z_p} \Z_p [\cG_F]$ via the isomorphism $\Z_p \llbracket G_k \rrbracket \cong \Z_p \llbracket G_F \rrbracket [\cG_F]$. 
    To do this, we suppose to be given an element of the form $g_j h' \in G_k$ and compute that
    \begin{align*}
        f ( \sigma_j h' \cdot a \otimes \sigma_i h) = f ( a \otimes \sigma_j \sigma_i ( \sigma_i^{-1} h' \sigma_i h))
        = (\sigma_i^{-1} h' \sigma_i h \cdot a) \otimes \sigma_j \sigma_i = (h' h \cdot a) \otimes \sigma_j \sigma_i,
    \end{align*}
    where the last equality uses that $\sigma_i$ acts trivially (by conjugation) on $\gal{F (E_\tors)}{F}$ because $F ( E_\tors) / k$ is assumed to be abelian.\\
    Write $\psi \: G_F \to \mathrm{Aut} (T_p E) \cong \cO_p^\times$ for the character induced by the action of $G_F$ on $T_p E$, then the action of $\Z_p \llbracket G_F \rrbracket [\cG_F]$ on $T_p E \otimes_{\Z_p} \Z_p [\cG_F]$ factors through the morphism $\Z_p \llbracket G_k \rrbracket \to \cO_p [\cG_F]$ induced by $\psi$. \\
    We recall next that the known validity of the Tate conjecture for abelian varieties over number fields (by Faltings \cite{Faltings}) implies the existence of a ring isomorphism 
    $\cA_p \cong \mathrm{End}_{\Z_p \llbracket G_k \rrbracket} (T_p B)$.  Writing $\cR \subseteq \cO_p$ for the $\Z_p$-order generated by the image of $\psi$, there are therefore identifications 
    \[
    \mathrm{End}_{\Z_p \llbracket G_k \rrbracket} (T_p B) = \mathrm{End}_{\cR [\cG_F]} (T_p E \otimes_{\Z_p} \Z_p [\cG_F]) = \mathrm{End}_{\cO_p [\cG_F]} (T_p E \otimes_{\Z_p} \Z_p [\cG_F]) \cong \cO_p [\cG_F].
    \]
    Here the second equality holds because $\cR$ is of finite index in $\Z_p \otimes_\Z \cO_k$ (as the cokernel of $\psi$ is finite) and $T_p B$ is $\Z_p$-torsion free, and the last isomorphism holds because $T_p E \otimes_{\Z_p} \Z_p [\cG_F]$ is a free $\cO_p [\cG_F]$-module (cf.\@ also \cite[Rk.\@ on p.\@ 502]{SerreTate}). 
\end{proof}

We now fix a finite abelian extension $K$ of $k$ and consider the motive 
\[ M_{B/K}\coloneqq h^1 (B_K) (1),\]
regarded as defined over $k$ and with coefficients $A[\cG_K]$. For the reader's convenience, we quickly review the precise formulation of the `analytic-rank-zero' component of  $\TNC(M_{B / K}, \mathcal{A}_p [\cG_K])$. 
\\
To do this, we fix (in terms of the decomposition (\ref{av decomp})) an isogeny $B \to \prod_{i = 1}^{i=t} B_i$ in which each $B_i$ is a simple abelian variety that is defined over $k$ and such that $\Q \otimes_\Z \mathrm{End}_k (B_i)$ is isomorphic to the field $L_i$. Attached to each $B_i$ is then an algebraic Hecke character (the `Serre--Tate character' of $B_i$) $
\varphi_i \: \mathbb{A}_k^\times \to L_i^\times$ of infinity type $(-1, 0)$, which gives rise, for every embedding $\tau \in \Hom (L_i, \C)$, to  a Hecke character
\[
\varphi_{i, \tau} \: \mathbb{A}_k^\times / k^\times \to \C^\times 
\]
 (see \cite[Th.\@ 10]{SerreTate} for details). Then, setting 
 \[ S \coloneqq \Pi_k^\infty \cup S_\ram (K / k) \cup S_\ram (B),\] 
the $S$-truncated $(\C \otimes_\Q A) [\cG_K]$-valued $L$-function of $M_{B/K}$ has the following explicit description 
\begin{align*}
L_S (M_{B/K}, s) =&\, \big( {\sum}_{\chi \in \widehat{\cG_K}} L_S ( \overline{\varphi_{i, \tau}} \chi^{-1}, s + 1)  e_{\chi} \big)_{i, \tau} 
\\
\in&\, \bigoplus_{i = 1}^{i=t} \bigoplus_{\tau \: L_i \hookrightarrow \C} \C [\cG_K] \cong \bigoplus_{i = 1}^{i=t} (\C \otimes_\Q L_i) [\cG_K] \cong (\C \otimes_\Q A) [\cG_K].\end{align*}
To describe the corresponding period map, it is convenient to fix an embedding $\iota \: k \hookrightarrow \C$ and hence an induced identification $(\bigoplus_{\tau \: k \hookrightarrow \C} H^1 (B^\tau (\C), \Q))^+ \cong H^1 (B^\iota (\C), \Q)$. Using the Poincar\'e duality isomorphism $H^1 (E^\iota (\C), \Q) \cong H_1 (E^\iota (\C), \Q)^\ast$, we may then explicitly describe the period map of $B_K$ as 
\[
\mathrm{per}_{B, K} \: K \otimes_k H^0 (B, \Omega^1_{B / k}) \to (\R \otimes_\Q K) \otimes_k  H^1 ( B^\iota (\C), \Q),
\quad \omega \mapsto \big \{ \gamma \mapsto \int_{\gamma} \omega \big \}. 
\]
Next we note that if $L (B / K, 1) \neq 0$, then the groups $B (K)$ and $\sha_{B / K}$ are both finite (cf.\@ \cite[Th.\@ 1.2]{BurungaleFlach}), and so the space $H^2 (\cO_{K, S}, V_p B)$ vanishes.\\
We now fix an $\End_k (B)$-basis element $\gamma$ of $ H^1 ( B^\iota (\C), \Z)$, and write $\delta (\gamma)$ for the image of $\gamma$ under the comparison isomorphism 
\[ \Z_p \otimes_\Z  H^1 ( B^\iota (\C), \Z) \cong H^1 (B, \Z_p) \cong (T_p B)^\ast.\] 
Then, setting $\cT \coloneqq \mathrm{Ind}_{G_K}^{G_k} (T_p B)$, the construction of (\ref{definition of finite level theta map}) defines a composite homomorphism
\begin{align*}
    \Theta_{K, S, \delta (\gamma)} \: 
    \Q_p \otimes_{\Z_p} \Det_{\cA_p [\cG_K]} ( C_S (\cT))
    & \xrightarrow{\simeq} H^1 (\cO_{K, S}, V_p B) \otimes_{(\Q_p \otimes_\Q A ) [\cG_K]} (V_p B)^\ast \\
    & \xrightarrow{\simeq} H^1 (\cO_{K, S}, V_p B),
\end{align*}
in which the first arrow denotes the natural `passage-to-cohomology' map and the second is induced by sending $\delta (\gamma)$ to $1\in (\Q_p \otimes_\Q A ) [\cG_K]$.
We also use the composite homomorphism 
\begin{align*}
    \lambda_{B, K, S} \: H^1 (\cO_{K, S}, V_p B) & \rightarrow {\bigoplus}_{v \in \Pi_K^p} H^1 (K_v, V_p B) \\
    & \xrightarrow{\exp^\ast} D^1_{\mathrm{dR}, \Q_p \otimes_\Q K} ( V_p B) \\
    & \xrightarrow{\simeq} (\Q_p \otimes_\Q K) \otimes_\Q H^0 (B, \Omega^1_{B / k}),
\end{align*}
 in which the first map is the natural localisation morphism, the second is the dual exponential map of $V_p B$ \cite[Ch.\@ II, \S\,1.2.4]{Kato93a}, and the third is the canonical comparison isomorphism from $p$-adic Hodge theory.\\
    We next fix a finite subset $\Sigma$ of $\Pi_k$ that is disjoint from $S$ and, for each $v \in \Sigma$, set
    \[
    \delta_{K, v} (X) \coloneqq  {\det}_{\cA_p [\cG_K]} ( 1 - \Frob_v^{-1} X \mid T_p B ) \in \cA_p [X].
    \]
   Then, since $H^0 (K_w, T_p B)$ vanishes for all $w \in \Sigma_K$, the element $\delta_{K, \Sigma} \coloneqq \prod_{v \in \Sigma} \delta_{K, v} (\Frob_v^{-1})$ is a nonzero divisor in $\cA_p [\cG_K]$
    and, as a consequence, the argument of \cite[Prop.\@ 3.4]{BKS} shows the conjecture $\TNC(M_{B/K}, \mathcal{A}_p [\cG_K])$ to be equivalent to the existence of an $\cA_p [\cG_K]$-basis element $\fz_{B / K, S, \Sigma}$ of $\Det_{\cA_p [\cG_K]} ( C_{S, \Sigma} (\cT))$ that has both of the following properties
    \begin{equation}\label{properties2}
        \begin{cases}
        (\lambda_{B, K, S} \circ \Theta_{K, S, \delta (\gamma)}) ( \fz_{B / K, S, \Sigma}) \in K \otimes_k H^0 (B, \Omega^1_{B / K});\\
        (\mathrm{per}_{B, K} \circ \lambda_{B, K, S} \circ \Theta_{K, S, \delta (\gamma)}) ( \fz_{B / K, S, \Sigma}) = \delta_{K, \Sigma} \cdot L_S (M_{B/K}, 0) \cdot \gamma. \end{cases}
    \end{equation}

We can now state the main result of this section. 

\begin{thm} \label{CM elliptic curves main result}
    Fix data $E / F$ and $B$ as above and let $K$ be a finite abelian extension of $k$ for which  $L (B / K, 1) \neq 0$. Then the conjecture  
    $\TNC(M_{B / K}, \mathcal{A}_p [\cG_K])$ is valid.
\end{thm}

\begin{proof} To deduce this result from Theorem \ref{etnc imaginary quadratic fields}, we use a special case of the general strategy described in \cite[Th.\@ 3.22]{bbs}. To do this, we consider the character
\[
\rho \: G_k \to \mathrm{Aut} ( \cT) \cong \mathcal{A}_p [\cG_K]^\times
\]
and set $L \coloneqq (k^c)^{\ker (\rho)}$ (this is an abelian extension of $k$ that contains both $K$ and a $\Z_p$-extension of $k$.) We also set $\Lambda \coloneqq \Z_p \llbracket \cG_{L} \rrbracket$,
take $S$ to be a finite subset of $\Pi_k$ that contains $\Pi_k^\infty \cup \Pi_k^p \cup S_\ram (T_p B) \cup S_\ram (K / k)$, 
and set $\Sigma \coloneqq \{ \a \}$ with $\a$ a prime ideal of $k$ that does not belong to $S$. \\
To construct the element $\fz_{B / K}$, we write 
\[
\Theta_{\mathbb{G}_m / K, S, b_\bullet} \: \Det_\Lambda ( C_{S, \Sigma} (\Lambda (1))) \to  H^1_\Sigma (\cO_{k, S}, \Lambda (1))
\]
for the map obtained as the direct sum of the maps in (\ref{theta map Zp (1)}) over characters $\chi$ of $\cG_L$ that have finite prime-to-$p$ order. We then note  that Theorem \ref{etnc imaginary quadratic fields} implies the existence of a $\Lambda$-basis $\fz_{\mathbb{G}_m / L, S, \Sigma}$ of $\Det_\Lambda ( C_{S, \Sigma} (\Lambda (1)))$ that is sent by $\Theta_{\mathbb{G}_m / K, S, b_\bullet}$ to the family  of Rubin--Stark elements
\[ \varepsilon_{L, \Sigma} \coloneqq (\varepsilon^{\{ \infty \}}_{F / k, S, \Sigma})_{F \subseteq L}.\]
Now, the isomorphism of $G_k$-representations
\[ \Lambda (1) \otimes_{\Lambda, \rho^{-1}} \mathcal{A}_p \cong \cT^\ast (1)\cong \cT\]  combines with Proposition \ref{construction complex}\,(iv) to induce an isomorphism $C_{S, \Sigma} (\Lambda (1)) \otimes_\Lambda^\mathbb{L} \mathcal{A}_p \cong C_{S, \Sigma} (\cT)$ in $D(\cA_p)$. The latter isomorphism in turn induces  a morphism of $\Lambda$-modules
\[
\mathrm{Tw}^\mathrm{det} \: \Det_\Lambda ( C_{S, \Sigma} (\Lambda (1))) \to \Det_\Lambda ( C_{S, \Sigma} (\Lambda (1))) \otimes_{\Lambda} \mathcal{A}_p \cong \Det_{\mathcal{A}_p} ( C_{S, \Sigma} (\cT))
\]
that, by Proposition \ref{twisting lemma}\,(i)\,(a), lies in  a commutative diagram
\begin{cdiagram}
    \Det_\Lambda (C_{S, \Sigma} (\Lambda (1)) \arrow{d}{\mathrm{Tw}^\mathrm{det}} \arrow{rr}{\Theta_{\mathbb{G}_m / K, S, b_\bullet}} & & H^1_\Sigma (\cO_{k, S}, \Lambda (1)) \arrow{d}{\mathrm{Tw}} \\ 
    \Det_{\mathcal{A}_p} ( C_{S, \Sigma} (\cT)) \arrow{rr}{\Theta_{K, S, \delta(\gamma)}} & &
    H^1_\Sigma (\cO_{K, S}, V_p B).
\end{cdiagram}%
Here we write $\mathrm{Tw}$ for the composite homomorphism 
\[ 
\mathrm{Tw} \:  H^1_\Sigma (\cO_{k, S}, \Lambda (1))  \xrightarrow{\gamma}  H^1_\Sigma (\cO_{k, S}, \Lambda (1)) \otimes_\Lambda (V_p B)^\ast 
 \xrightarrow{\simeq} H^1_\Sigma (\cO_{K, S}, (V_p B)^\ast (1)),\]
in which the first map is induced by sending $1 \mapsto \delta(\gamma)$ and the second is the isomorphism arising from Proposition \ref{construction complex}\,(iv).

Now, by Proposition \ref{twisting lemma}\,(i)\,(b), the element  
\[ \fz_{B / K, S, \Sigma} \coloneqq \mathrm{Tw}^\mathrm{det} (\fz_{\mathbb{G}_m / L, \Sigma})\] 
is an $\mathcal{A}_p$-basis of $\Det_{\mathcal{A}_p} ( C_{S, \Sigma} (\cT))$, and so it suffices to verify that this element also has both of the properties in (\ref{properties2}). Given the above commutative diagram, we can therefore verify the latter properties after replacing the element $\Theta_{K, S, \delta (\gamma)} ( \fz_{B / K, S, \Sigma})$ by $\mathrm{Tw} (\varepsilon_{L, \Sigma})$, and to do this we use the reciprocity law of Kato--Wiles. 

Specifically, after taking account of \cite[proof of Prop.\@ 3.3]{BurungaleFlach}, which compares the spaces $V_{L_i} (\varphi_i)$ and $S (\varphi_i)$ introduced in \cite[\S\,15.8]{Kato04} with $H^1 (B^\iota, \Q)$ and $H^0 (B, \Omega^1_{B / K})$ respectively, this reciprocity law \cite[Prop.\@ 15.9]{Kato04} asserts that %
\[ (\lambda_{B, K, S}\circ \mathrm{Tw}) (\varepsilon_{L, \Sigma})\in K \otimes_\Q H^0 (B, \Omega^1_{B / k})\] 
and is such that
\[
{\sum}_{\sigma \in \cG_K} \chi (\sigma) \cdot \mathrm{per}_B ( \sigma \cdot (\lambda_{B, K, S} \circ \mathrm{Tw}) (\varepsilon_{L, \Sigma})) 
= (\tau \otimes \chi^{-1}) (\delta_{K, \Sigma}) \cdot \bigl(L_S (\overline{\varphi}_\tau \cdot \chi, 1)\bigr)_{i\in [t], \tau\in \Upsilon_i^\iota} \otimes \gamma.
\]
Here $\Upsilon_i^\iota$ denotes the subset of $\Hom (L_i, \C)$ comprising embeddings $\tau$ that restrict to give $\iota$ on $k$, so that the element $(L_S (\overline{\varphi}_\tau \cdot \chi, 1))_{i, \tau} $ belongs to $\bigoplus_{i = 1}^{i=t} (L_i \otimes_\Q \R) \cong \R \otimes_\Q A$. The required result then follows directly from the fact that
\[
L_S (M_{B/K}, 0) =  {\sum}_{i = 1}^{i=t} {\sum}_{\tau\in \Upsilon_i^\iota} {\sum}_{\chi \in \widehat{\cG_K}} L_S (\overline{\varphi}_\tau \chi, 1) e_{\chi^{-1}} \in (\R \otimes_\Q A) [\cG_K].
\qedhere
\]
\end{proof}

\begin{rk} \label{bf remark}
\begin{romanliste}
    \item If $L (B / K, 1)=0$, then the argument of Theorem \ref{CM elliptic curves main result} still proves the `analytic-rank-zero' component of the conjecture $\TNC(M_{B/K}, \mathcal{A}_p [\cG_K])$.
    \item Lemma \ref{CM elliptic curves lemma 1} implies that the representation $T_p B$ is isomorphic to $\mathrm{Ind}_{G_F}^{G_k} (T_p E)$. Upon combining Theorem \ref{CM elliptic curves main result} (with $F = k$) with Remark \ref{forgetful etnc remark}, we can therefore deduce  the validity under the stated hypotheses of the conjecture $\TNC(h^1(E/F)(1),\cO_p)$, and hence also of the $p$-part of the BSD Conjecture for $E / F$. The latter result is the main result of Burungale and Flach in \cite{BurungaleFlach}. 
\end{romanliste}
\end{rk} 

\subsection{Kato's Conjecture in more general cases} \label{more general cases section}

The methods used to prove Theorem \ref{etnc imaginary quadratic fields} have consequences well beyond the case of abelian extensions of imaginary quadratic fields. As a concrete example, the following result is derived by combining our methods with the general approach developed by Daoud et al.\@ in \cite{scarcity}. In the sequel, we write $h_k$ for the class number of a number field $k$. 

\begin{thm} \label{RS implies etnc}
    Let $k$ be a number field and $p$ a prime number that is inert in $k$ and does not divide $6h_k$. Fix a non-empty subset $V$ of $\Pi_k^\infty$ and assume that the Rubin--Stark Conjecture is valid for all data of the form $(K / k, S^\ast(K), \Sigma_K, V)$ with $K$ a finite abelian extension of $k$, $S^\ast(K) \coloneqq \Pi^\infty_k \cup S_\mathrm{ram} (K / k)$ and $\Sigma_K$ a finite subset of $\Pi_k\setminus S^\ast(K)$. Then, for every finite abelian extension $K$ of $k$ with $\mu_K[p]=(0)$, the conjecture $\TNC(\QQ_K(0), \epsilon_{K, V} \Z_p [\cG_K])$ is valid.  
\end{thm}

\begin{proof}
 We write $\Omega^V$ for the collection of all finite abelian extensions $K$ of $k$ such that $V (K) = V$, and  $\omega_p \: G_k \to \Z_p^\times$ for the $p$-adic Teichm\"uller character of $k$. 
 
 Then the proof of \cite[Th.\@ 6.1\,(a)\,(i)]{scarcity} shows that $\TNC(\QQ_K(0), (1 - e_{\omega_p})\epsilon_{K, V} \Z_p [\cG_K])$ is valid for every finite abelian extension $K$ of $k$ if Conjecture \ref{eIMC G_m} is valid  for every $K \in \Omega^V$ and every  character $\chi \in \widehat{\nabla_{\!\!K}} \setminus \{ \omega_p \}$ with $k_\infty$ taken to be the cyclotomic $\Z_p$-extension of $k$. 
(Note that, since our formulation of Conjecture \ref{eIMC G_m} is explicitly in terms of the determinant functor and perfect complexes, we do not need to assume condition (ii) of  \cite[Th.\@ 6.1\,(a)\,(i)]{scarcity}.)
In particular, in these cases, the conjecture $\TNC(\QQ_K(0), \epsilon_{K, V} \Z_p [\cG_K])$ is valid for every finite abelian extension $K$ of $k$ with $\mu_K [p] = (0)$.
\\
To verify the validity of Conjecture \ref{eIMC G_m} in the required cases, we then proceed similarly as in Theorem \ref{etnc imaginary quadratic fields}. To be precise, we claim first it suffices to prove, for all $K \in \Omega^V$ and characters $\chi \in \widehat{\nabla_{\!\!K}} \setminus \{ \omega_p \}$, that there exists a finite set $\Sigma_K \subseteq \Pi_k \setminus (S^\ast (K) \cup \Pi_k^p)$ such that $H^1_{\Sigma_K} ( \cO_{K, S(K)}, \Z_p (1))$ is $\Z_p$-torsion free and 
that there is an inclusion 
\begin{equation} \label{let's assume this inclusion}
\Fitt^0_{\Lambda_F} ( Y_{\Pi^p_k} (\cT_\chi))^{\ast \ast}
    \cdot \varepsilon_{K_\infty, \Sigma_K}^\chi \subseteq \Theta_{K_\infty / k, S(K), b_\bullet} ( \Det_{ \Lambda_F} ( C_{S(K), \Sigma_K} ( \cT_\chi))),
\end{equation}
with $\cT_\chi \coloneqq \Lambda_F (1)(\chi)$. 

To show Conjecture \ref{eIMC G_m} is indeed implied by these inclusions, we may work locally at a fixed prime $\p$ in $\Spec^1(\Lambda_F)$. If  $p \in \p$, then Lemma \ref{singular primes and mu vanishing} implies $\Fitt^0_{\Lambda_F} ( Y_{\Pi^p_k} (\cT_\chi))_\p = \Lambda_{F, \p}$, and so the above inclusion implies that 
\begin{equation} \label{an inclusion we want}
    \varepsilon_{K_\infty, \Sigma_K}^\chi \in \Theta_{K_\infty / k, S(K), b_\bullet} ( \Det_{ \Lambda_F} ( C_{S(K), \Sigma_K} ( \cT_\chi)))_\p.
\end{equation}
To prove the same containment also holds if $p \not \in \p$, we note that the argument of Theorem \ref{new iwasawa theory euler systems main result}\,(iii) shows that (\ref{let's assume this inclusion}) implies  an inclusion 
\begin{equation} \label{let's then also assume this other inclusion}
    \im (\varepsilon_{K_\infty, \Sigma_K}^\chi)_\p \subseteq \Fitt^0_{\Lambda_F} ( A_{K_\infty, S(K), \Sigma_K}^\chi)_\p,
\end{equation}
where  $A_{K_\infty, S(K), \Sigma_K}$ is the module defined just before the statement of Theorem \ref{G_m main result}.

Next, we note that, since the given assumptions imply $|\Pi^p_k| = 1$, the `Gross--Kuz'min Conjecture' for $K$ is valid by Maksoud \cite[Th.\@ 4.3.2\,(b)]{maksoud}. As a consequence, we may combine the results of \cite[Lem.\@ 6.6\,(b) and (d)]{scarcity} with (\ref{let's then also assume this other inclusion}) to deduce that
\[
\im (\varepsilon_{K_\infty, \Sigma_K}^\chi)_\p \subseteq \Fitt^0_{\Lambda_F} ( A_{K_\infty, S(K), \Sigma_K}^\chi)_\p \cdot \Fitt^0_{\Lambda_F} ( X_{\Pi^p_k} (\cT_\chi))_\p.
\]
Since this inclusion holds for every $K \in \mathcal{X}$ that is unramified at all places in $\Sigma_K$ we may then use the argument of \cite[Lem.\@ 6.6\,(a)]{scarcity} to derive an inclusion
\[
\im (\varepsilon_{K_\infty, \Sigma_K}^\chi)_\p \subseteq \Fitt^0_{\Lambda_F} ( A_{K_\infty, S(K), \Sigma_K}^\chi)_\p \cdot \Fitt^0_{\Lambda_F} ( X_{S(K)_\mathrm{fin}} (\cT_\chi))_\p
\]
and hence also, by \cite[Prop.\@ 6.4\,(b)\,(i)]{scarcity}, the required containment (\ref{an inclusion we want}). 

Having now verified that the latter containment is valid for every prime in $\Spec^1(\Lambda_F)$, the analytic class number formula allows us to deduce the validity of Conjecture \ref{eIMC G_m} (cf.\@ the argument of \cite[Prop.\@ 6.4\,(b)\,(ii)]{scarcity}).\\
At this stage, it therefore only remains for us to justify the inclusion (\ref{let's assume this inclusion}). In addition, this inclusion follows directly from Theorem \ref{eimc result} in the case $\chi \neq \bm{1}$,   
and so it is enough for us to consider the case $\chi = \bm{1}$. To this end, we first recall the well-known fact that the $p$-adic Iwasawa $\mu$-invariant of $k$ vanishes since $|\Pi^p_k| = 1$ and $p\nmid h_k$ (see \cite[Prop.~13.22]{Washington}). By a standard argument (cf.\@ \cite[Prop.~3.15]{BKS2}), the proof of (\ref{let's assume this inclusion}) is therefore reduced to showing that, for every $\psi \in \widehat{\square_K}$, there is an inclusion of characteristic ideals
\[
\Char_{\Lambda_\psi} \big( ( \bidual^r_{\Lambda_\psi} H^1_{\Sigma_K} (\cO_{k, S(K)}, \Lambda_\psi (1) (\psi) \big) / (\Lambda_\psi\cdot \varepsilon_{K_\infty, \Sigma_K}^\psi) \big) 
\subseteq \Char_{\Lambda_\psi} ( A_{K_\infty, S(K), \Sigma_K}^\psi).
\]
If $\psi \neq \bm{1}$, then one can use Proposition \ref{bss lem 6.22} to directly deduce this inclusion from \cite[Th.~2.3.3]{Rubin-euler}. 
 On the other hand, the remaining case of $\psi = \bm{1}$ (which might not validate the hypothesis Hyp~$(K_\infty / K)$ in \cite{Rubin-euler} and so has to be considered separately) is trivially satisfied. Indeed, the assumptions that $k$ is a field with only one $p$-adic place and such that $p\nmid h_k$ implies that the group $A_{k_\infty, S(K)}$ vanishes (cf.\@ \cite[Prop.~13.22]{Washington}). 
 Since the assumption $\mu_K [p] = (0)$ moreover allows us to take $\Sigma_K = \emptyset$, this is therefore sufficient to prove the required inclusion.
\end{proof}

\begin{rk}
    Fix a non-empty subset $V$ of $\Pi_k^\infty$ and a subset $\mathcal{X}$ of $\Omega^V$ that satisfies the `closure hypothesis' of \cite[Hyp.~4.5]{scarcity}. Let $p > 3$ be a prime that is unramified in $k$ and $k_\infty$ a $\Z_p$-power extension of $k$ in which no finite place splits completely. Then a more careful analysis of the argument used to prove Theorem \ref{RS implies etnc} shows that $\TNC(\QQ_K(0), \epsilon_{K, V} \Z_p [\cG_K])$ is valid for every $K$ in $\mathcal{X}$ provided that all of the following conditions are satisfied at every $K$ in $\mathcal{X}$.   
    \begin{romanliste}
         \item The Rubin--Stark Conjecture is valid for the data  $(K / k, \Pi^\infty_k \cup S_\mathrm{ram} (K / k), \Sigma_K, V)$, with $\Sigma_K$ a finite subset of $\Pi_k\setminus S^\ast(K)$.
        \item Conjecture \ref{eIMC G_m} is valid for both $\chi = \bm{1}$ and $\chi = \omega_p$.
        \item $\im (\varepsilon^V_{K_\infty, \Sigma}) \subseteq \Fitt^0_{\Lambda_K} ( X_{\Pi_k^p} (\Lambda_K  (1)))^{\ast \ast}$.
        \item The $\Lambda_K$-modules $X_{\Pi_k^p} (\Lambda_K  (1))$ and $A_{K_\infty, S(K)}$ have disjoint support.
    \end{romanliste}
    
\end{rk}

\begin{bsp}
    Let $k$ be a complex cubic number field and write $\infty$ for its unique complex place. In this setting Bergeron--Charollois--Garc\'ia \cite{bergeron2023ellipticunitscomplexcubic} have recently provided strong evidence in support of the Rubin--Stark Conjecture for the data $(K / k, S^\ast (K), \Sigma, \{\infty\})$. Assuming the latter conjecture to be valid, there is then a positive-density set of primes $p$ for which Theorem~\ref{RS implies etnc} can be applied. Indeed, since the normal closure $\widetilde{k}$ of $k$ has Galois group $\gal{\widetilde{k}}{k} \cong S_3$, the Cebotarev density theorem provides us with a positive-density set of primes $p$ such that the alternating group $A_3$ is generated by the conjugacy class of $\Frob_\p$ for any prime of $\widetilde{k}$ lying above $\p$. In particular, any prime in this set that does not divide $h_k$ is a prime for which all of the assumptions of Theorem \ref{RS implies etnc} are satisfied.
\end{bsp}

\addcontentsline{toc}{section}{References}
\tiny
\printbibliography

\small

\textsc{Instituto de Ciencias Matem\'aticas, c/ Nicol\'as Cabrera 13-15, Campus de Cantoblanco UAM, 28049 Madrid, Spain.}\\
\textit{Email address:} \href{mailto:dominik.bullach@icmat.es}{dominik.bullach@icmat.es} 
\smallskip \\ 
\textsc{King's College London,
Department of Mathematics, Strand, 
London WC2R 2LS,
UK.} \\
\textit{Email address:} \href{mailto:david.burns@kcl.ac.uk}{david.burns@kcl.ac.uk}
\end{document}